\documentclass[12pt]{article}
\usepackage{amssymb, amstext, amsthm, amscd, amsmath,color}
\usepackage{graphicx} 
\usepackage[]{geometry}
\usepackage{tikz-cd}
\usepackage{cancel}
\usepackage{thmtools}
\usepackage{stmaryrd}
\usepackage{comment}
\usepackage{enumitem}
\usepackage{mathtools}
\usepackage{blindtext}
\usepackage{titlesec}
\usetikzlibrary{decorations.pathmorphing,shapes}
\usepackage[hidelinks]{hyperref}
\usepackage{quiver}
\usepackage{tensor}
\usepackage{tabularx}
  \newcolumntype{C}{>{\centering\arraybackslash}X}

\usepackage[textwidth=0.89in,textsize=scriptsize]{todonotes}

\hypersetup{colorlinks=false,linktoc=all}

\theoremstyle{plain}
\newtheorem{thm}{Theorem}[subsection]
\newtheorem{cor}[thm]{Corollary}
\newtheorem{prop}[thm]{Proposition}
\newtheorem{lem}[thm]{Lemma}

\newtheorem*{thm*}{Theorem} 

\theoremstyle{definition}
\newtheorem{rem}[thm]{Remark}

\newtheorem{defn}[thm]{Definition}

\newcommand{\targetsection}{\textbf{Int.Frc.(1)}}
\newcommand{\wcomp}{\textbf{Int.Frc.(2)}}
\newcommand{\ore}{\textbf{Int.Frc.(3)}}

\newcommand{\mN}{{\mathbb N}}

\newcommand{\mC}{{\mathbb C}}
\newcommand{\mD}{{\mathbb D}}

\newcommand{\mX}{{\mathbb X}}

\newcommand{\cA}{{\mathcal A}}

\newcommand{\cP}{{\mathcal P}}
\newcommand{\cC}{{\mathcal C}}

\newcommand{\cE}{{\mathcal E}}

\newcommand{\cJ}{{\mathcal J}}

\newcommand{\cod}{\mathrm{cod}}
\newcommand{\dom}{\mathrm{dom}}

\newcommand{\mCW}{\mathbb{C}[W^{-1}]}
\newcommand{\mDW}{\mathbb{D}[W^{-1}]}
\newcommand{\spn}{\text{spn}}
\newcommand{\csp}{\text{csp}}

\newcommand{\slb}{\text{sb}}
\newcommand{\Top}{\mathbf{Top}}
\newcommand{\Cat}{\mathbf{Cat}}

\newcommand{\Set}{{\mathbf{Set}}}
\newcommand{\opl}{\mbox{\scriptsize op$\ell$}}

\newcommand{\noi}{{\noindent}}

\newsavebox{\pullback}
\sbox\pullback{%
\begin{tikzpicture}%
\draw (0,0) -- (1ex,0ex);%
\draw (1ex,0ex) -- (1ex,1ex);%
\end{tikzpicture}}

\newsavebox{\pushout}
\sbox\pushout{%
\begin{tikzpicture}%
\draw (0,0) -- (0ex,1ex);%
\draw (0ex,1ex) -- (1ex,1ex);%
\end{tikzpicture}}

\newsavebox{\urpushout}
\sbox\urpushout{%
\begin{tikzpicture}%
\draw (0,0) -- (1ex,0ex);%
\draw (0ex,0ex) -- (0ex,1ex);%
\end{tikzpicture}}

\newsavebox{\vpullback}
\sbox\vpullback{%
\begin{tikzpicture}%
\draw (0ex,1ex) -- (1ex, 0ex);%
\draw (1ex, 0ex) -- (2ex,1ex);%
\end{tikzpicture}}

\newsavebox{\vpushout}
\sbox\vpushout{%
\begin{tikzpicture}%
\draw (0ex,1ex) -- (1ex, 1ex);%
\draw (1ex, 1ex) -- (2ex,1ex);%
\end{tikzpicture}}

\newcounter{sarrow}

\newsavebox{\urpullback}
\sbox\urpullback{%
\begin{tikzpicture}%
\draw (-1ex,0ex) -- (-1ex,-1ex);%
\draw (0ex,-1ex) -- (-1ex,-1ex);%
\end{tikzpicture}}

\newsavebox{\dlpullback}
\sbox\dlpullback{%
\begin{tikzpicture}%
\draw (0ex, 1ex) -- (1ex,1ex);%
\draw (1ex,1ex) -- (1ex,0ex);%
\end{tikzpicture}}

\begin{document}

\title{Pseudocolimits of Small Filtered Diagrams of Internal Categories}
\author{Deni Salja}
\date{}

\maketitle 

\begin{abstract}
Pseudocolimits are formal gluing constructions that combine objects in a category indexed by a pseudofunctor. When the objects are categories and the domain of the pseudofunctor is small and filtered it is known \cite[Expos\'e 6]{SGA4} that the pseudocolimit can be computed by taking the Grothendieck construction of the pseudofunctor and inverting the class of cartesian arrows with respect to the canonical fibration. In this thesis we present a set of conditions on an ambient category $\cE$ for defining the Grothendieck construction as an oplax colimit and another set of conditions on $\cE$ along with conditions on an internal category, $\mC$, in $\Cat(\cE)$ and a map $w : W \to \mC_1$ that allow us to translate the axioms for a category of (right) fractions, and construct an internal category of (right) fractions. We combine these results in a suitable context to compute the pseudocolimit of a small filtered diagram of internal categories.
\end{abstract}

\section*{Acknowledgements}
This is a reformatted version of a M.Sc. Thesis in Mathematics submitted to and defended at Dalhousie University in August 2022 under the supervision of Dr. Pronk. The original can be found at \url{http://hdl.handle.net/10222/81928}. Special thanks to Dr. Pronk for her invaluable feedback and support over the last two years; to Dr. Pronk, Dr. Par\'e, and Dr. Sellinger for their additional feedback and interesting discussions; and to Dr. Faridi for chairing my defence. I'd also like to acknowledge that this two-year research project was funded by the Natural Sciences and Engineering Research Council (NSERC) Canadian Graduate Scholarship - Master's program as well as the Nova Scotia Graduate Scholarship - Master's program. Finally I'd like to thank my family and friends for their patience and support over the two years this project took me to fathom and complete.  

\tableofcontents


\addcontentsline{toc}{section}{Introduction}
\section*{Introduction}

The term `Grothendieck construction' is used to describe a correspondence between pseudofunctors $\cA \to \Cat$, and fibrations over $\cA$. For a given pseudofunctor $\cA \to \Cat$ the domain of the corresponding fibration is often the {\em category of elements}. This happens to be the oplax colimit of the pseudofunctor \cite{GrothLaxColim}, whose construction can be thought of as `bundling up' the original diagram in $\Cat$ into a single category. The pseudocolimit of the diagram can be computed by formally inverting some arrows in the category of elements; more precisely, by localizing with respect to the cartesian arrows of the associated fibration. When the indexing category, $\cA$, is filtered the Grothendieck construction satisfies the Gabriel-Zisman axioms \cite{GabZis} for a (left) category of fractions and the pseudocolimit is given by a category of (left) fractions. A weaker set of axioms is given in \cite{ThreeFsforBiCats} which allows for localization with respect to a smaller class of arrows.

Having colimits in a category is important for understanding how to glue diagrams in that category together into a single representing object. The usual Grothendieck construction, as an oplax colimit, is a gluing construction for categories indexed by a pseudofunctor and this has been translated in various other settings such as ($\infty,1$)-cats \cite{GrothLaxColim}, enriched categories \cite{EnrichedGroth}, and bicategories by \cite{ThreeFsforBiCats}. The geometric realization of the Grothendieck construction for a diagram of small categories has also been studied as a homotopy colimit by Thomason in \cite{Thomason} and for a diagram of quasi categories in \cite{Sharma}. In the first part of this thesis we will translate the usual Grothendieck construction into the language of internal category theory in order to compute an oplax colimit of a small diagram of internal categories. The fibration perspective of the Grothendieck construction has already been developed for monoidal categories \cite{MonGroth}, 2-cats and bicategories \cite{Buckley}, and ($\infty,1$)-categories \cite{GrothLaxColim} as well but we focus more on the colimit perspective because it is not possible to view the indexing category as an internal category in general. 

A different setting in which this would be useful is to describe the tom Dieck fundamental group for a space equipped with a group action. This is a category enriched in topological spaces that is the oplax colimit of fundamental groups of fixed point sets of all the subgroup actions \cite{tomDieck}. This is also useful for computing atlas groupoids for orbifolds, which are pseudocolimits of categories internal to $\Top$ \cite{Sibih}. Yet another relevant setting is for double categories, which are internal categories in $\Cat$. Such a construction here would allow us to compute oplax colimits of diagrams of double categories indexed by small categories. 

To replicate these colimit constructions we need an ambient category, $\cE$, with sufficient structure. In particular, to define an internal category of elements for a pseudofunctor $D : \cA^{op} \to \Cat(\cE)$ we need that $\cE$ has pullbacks along certain source and target maps of the internal categories in the image of $D$, has disjoint coproducts of these pullbacks and of the objects of objects for the internal categories in the diagram, and that these commute with one another. This allows us to construct an oplax colimit of $D$, by Theorem~\ref{thm IntGroth is OpLaxColim} which we restate here:

\begin{thm*}[The Internal Category of Elements, $\mD$, as an oplax colimit]
Let $\cE$ admit an internal category of elements of $D : \cA \to \Cat(\cE)$, as in Definition~\ref{def E admits an internal category of elements of D}.Let $\mD$ denote the internal category of elements. Then for every internal category $\mX \in \Cat(\cE)$, the category of lax natural transformations $D \implies \Delta \mX$ and their modifications is isomorphic to the category of internal functors $\mD \to \mX$ and their internal natural transformations.
\[[D , \Delta \mX]_{\ell} \cong \Cat(\cE)(\mD , \mX) \]
\end{thm*}

The pullback and coproduct commutativity is an extensivity property and is relied on heavily to define the internal category structure and prove the required properties are satisfied and makes the construction unlikely to work for arbitrary diagrams of categories internal to non-extensive categories such as the category of vector spaces. The internal category of fractions requires a special class of epimorphisms in the ambient category to locally witness internalized versions of category of fractions axioms in the sense of admitting lifts that define local sections. Part of what makes these epimorphisms special is that the local data they witness (on their codomains) can be combined to give global definitions of structure on their codomains provided the pieces of local data satisfy a kind of compatibility/descent condition. 

In the contexts we describe we show that the constructions used and the result stated in \cite{SGA4} can be translated into the language of internal categories. The class of epimorphisms we require for our internal category of fractions are coequalizers of their kernel pairs, stable under pullback, and closed under composition. Such a class always exists in any category, namely the identity arrows, but it is not always possible to get an internal fractions construction with this class. The Internal Fractions Axioms are described in Definition~\ref{def Internal Fractions Axioms} in terms of certain lifts of these epimorphisms and a section of an induced target structure map. Asking for sections in settings where continuity is important can be a strong condition. For example, when working with an arbitrary internal category in $\Top$ asking for continuous global sections is generally too much when the axiom of choice is being used, but there are nice classes of effective epimorphisms given by open surjections or \'etale surjections which give us local sections instead of global sections. In Section \ref{S IntFrc Applied to IntGroth} we show that the internal category of elements, when it exists in a suitable context, satisfies the Internal Fractions Axioms with respect to the object representing the canonical cleavage of the cartesian arrows by global sections, meaning it only requires identities for the class of epimorphisms in Definition~\ref{defn candidate context for internal fractions and covers}. We prove this in the main theorem of this thesis which we restate here:

\begin{thm*} 
Let $D : \cA^{op} \to \Cat(\cE)$ be a pseudofunctor for which $\cE$ admits an internal category of elements \ref{def E admits an internal category of elements of D} and let $w : W \to \mD_1$ be the object of a canonical cleavage of the internal category of elements as defined in Section \ref{SS canonical cleavage of Cart arr's in Groth}. If the pair $(\mD, W)$ satisfies the Internal Fractions Axioms in Definition~\ref{def Internal Fractions Axioms}, then for any $\mX$ in $\Cat(\cE)$ there is an isomorphism

\[ [D , \Delta \mX]_{ps} \cong [\mDW , \mX]^{\cE}\]

\noi between the category of pseudonatural transformations $D \implies \Delta \mX$ (and their modifications) and the category of internal functors $\mDW \to \mX$.
\end{thm*}

We begin in Chapter \ref{Ch notation and a word on internal cats} with a few words on notation and internal categories. In Chapter \ref{Chaper Internal Grothendieck} we present a context, $\cE$, in which we can define an internal category of elements for a pseudofunctor $\cA \to \cE$. We then define said internal category of elements, $\mD$, and show it is a(n) (op)lax colimit. The context for an internal (right) category of fractions and its definition are given in Chapter \ref{Ch Internal Fractions} and an isomorphism of categories that describes the universal property of the internal localization is proven at the end in Section \ref{S UP of Internal Fractions}. Chapter \ref{Ch pseudocolims} describes a setting in which we can compute the pseudocolimit of $D$ as the localization of the internal category of elements with respect to the  canonical cleavage of the cartesian arrows object, $w : W \to \mD_1$. 

\addcontentsline{toc}{section}{Notation and A Word on Internal Categories}
\section*{Notation and A Word on Internal Categories}\label{Ch notation and a word on internal cats}

Composition is written diagrammatically, so that `$f$ followed by $g$' is written by juxtaposition as `$fg$.'  Internal categories are denoted with blackboard bold font, $\mC, \mD, \mX$, and we use $\cA$ to denote the indexing category for pseudofunctors we will consider. We assume $\cA$ is small for Chapter~\ref{Chaper Internal Grothendieck} and will assume it is also cofiltered in Chapter \ref{Ch pseudocolims}. 

The main point of this thesis is to take a technique for computing pseudocolimits of small filtered diagrams of categories, give an internal category theoretic version of it, and show that it satisfies the universal property of a pseudocolimit of a small filtered diagram of internal categories. Internal categories are defined in Section B2.3 in \cite{Johnstone} when working in an ambient category with all pullbacks. Some of the ambient categories we wish to consider in future applications of this thesis, such as the category of smooth manifolds, do not have all pullbacks however. In this thesis, we do note assume the existence of all pullbacks in an ambient category, rather we make the existence of the necessary pullbacks and structure maps part of the definition of an internal category. 
\begin{defn}
\label{def internal category when you dont have all pullbacks in the ambient category}
An internal category, $\mC$ in $\cE$, consists of the following data. 

\begin{itemize} 
\item An object of objects, $\mC_0 \in \cE_0$.
\item An object of arrows, $\mC_1 \in \cE_0$.
\item Structure maps 

\[\begin{tikzcd}
\mC_1 \rar[shift left, "s"] \rar[shift right, "t"'] & \mC_0 \rar["e"] & \mC_1
\end{tikzcd} \]

\noi in $\cE_1$ such that $e$ is a common section of $s$ and $t$. 

\item The iterated pullbacks of composable chains of arrows, $\mC_n = \mC_1 \tensor[_t]{\times}{_s} \dots \tensor[_t]{\times}{_s} \mC_1$ in $\cE_1$. 

\item A composition structure map $c : \mC_2 \to \mC_1$ such that the squares 

\[\begin{tikzcd}[column sep = large, row sep = large]
\mC_2 \rar["c"] \dar["\pi_0"'] & \mC_1 \dar["s"] \\
\mC_1 \rar["s"'] & \mC_1
\end{tikzcd}\qquad \qquad 
\begin{tikzcd}[column sep = large, row sep = large]
\mC_2 \rar["c"] \dar["\pi_1"'] & \mC_1 \dar["t"] \\
\mC_1 \rar["t"'] & \mC_1
\end{tikzcd}\]

\noi commute in $\cE$, along with the associativity diagram, 

\[ \begin{tikzcd}
\mC_3 \rar["1 \times c"] \dar["c \times 1"'] & \mC_2 \dar["c"] \\
\mC_2 \rar["c"'] & \mC_1
\end{tikzcd}, \]

\noi and the identity law diagrams

\[\begin{tikzcd}
\mC_1 \ar[dr, "1_{\mC_1}"'] \rar["(1_{\mC_1} {,} te)"] & \mC_2 \dar["c"] & \mC_1 \lar["(s e{,} 1_{\mC_1})"'] \ar[dl, "1_{\mC_1}"] \\
& \mC_1 & 
\end{tikzcd}.\]
\end{itemize}
\end{defn}

\section{Internal Category of Elements}\label{Chaper Internal Grothendieck}

In this chapter we define (the category of elements for) the (internal) Grothendieck construction of a pseudofunctor $D : \cA \to \Cat(\cE)$ where $\cE$ is an extensive category and show that it is the oplax colimit of $D$. Section \ref{S Def Int Groth} defines the internal category structure of the (internal) Grothendieck construction, $\mD$, of $D$. Section \ref{S Assoc and Id Laws Int Groth} proves the associativity and identity laws for composition and shows it is an internal category. Section \ref{S Int Groth as Oplax Colim} shows the existence of a canonical lax natural transformation from $D \implies \Delta \mD$, and then proves the universal properties that show $\mD$ is the oplax colimit of $D$. We often consider how the usual proofs and definitions look in the case $\cE = \Set$ to help our readers follow the internalized definitions and results for an arbitrary extensive category $\cE$.

\subsection{Internal Category Structure}\label{S Def Int Groth}

\phantomsection
\subsubsection*{Classical}
\addcontentsline{toc}{subsubsection}{Classical}

Let $\cA$ be a small category and let $\Cat(\Set)$ denote the 2-category of small categories, strict functors, and natural transformations. This is a fully faithful subcategory of $\Cat$ so for every pseudofunctor

\begin{center}
\begin{tikzcd}
\cA \rar["D"] & \Cat(\Set) 
\end{tikzcd},
\end{center}

\noi the \textit{Grothendieck construction} of $D$ is the strict 2-pullback of $D$ along the canonical projection, $\pi$, from the lax-pointed 2-category of categories, $\Cat_{*, \ell}$ \cite{Johnstone}. 

\begin{center}
\begin{tikzcd}
\int F \rar[] \dar["P"'] \arrow[dr, phantom, "\usebox\pullback" , very near start, color=black] 
& \Cat_{*, \ell} \dar["\pi"] \\
 \cA \rar["D"'] & \Cat
\end{tikzcd}
\end{center}

\noi The objects of $\int D$ are pairs $(A,a)$, where $A \in \cA_0$ and $a \in (DA)_0$. The morphisms are pairs $(\varphi , f) : (A,a) \to (B,b)$ where $\varphi : A \to B $ is an arrow in $\cA$ and $f : \varphi(f)(a) \to b$ is an arrow in $D(B)$. Let $\delta_A : D(1_A) \implies 1_{D(A)}$ and $\delta_{\varphi ; \psi} : D(\varphi \psi) \implies D(\varphi) D(\psi)$ denote the identity and composition natural isomorphisms associated to $A \in \cA_0$ and an arbitrary composable pair $\varphi, \psi$ in $\cA_1$ respectively. The identity morphisms in $\mD$ are the pairs $(1_A , \delta_{A,a}) : (A,a) \to (A,a)$ where $\delta_{A,a} \in $ is the $a$-indexed component the natural transformaiton $\delta_A$. Two morphisms $( \varphi,f) , (\psi.g)$ are composable when $\cod ( D( \psi )_1 (f) ) = \dom (g)$ in $D(C)$. Then

\[ f : D(\varphi)_0(a) \to b \quad , \quad g : D(\psi)_0(b) \to d \]\

\noi and the following diagram

\[ \begin{tikzcd}[]
D(\psi)_0 \left( D(\varphi_0)(a)) \right) \ar[d, "D(\psi)_1(f)"'] \ar[from = r , "\delta_{\varphi ; \psi, a}"' , "\cong "] &D( \varphi \psi)_0 (a) \dar[dotted]\\
D(\psi)_0 (b) \rar["g"'] & d
\end{tikzcd}
\]

\noi defines the composite

\[ (\varphi,f) (\psi,g) := (\varphi \psi ,\delta_{\psi, a} D(\psi)_1(f) g ). \] 


\noi The so-called category of elements, $\int D$ is an oplax colimit of $D$ in $\Cat$ \cite{GrothLaxColim}. Next we translate this into the language of internal categories. 

\phantomsection
\subsubsection*{Internal}
\addcontentsline{toc}{subsubsection}{Internal}

Let $D: \cA \to \Cat(\cE)$ be a pseudofunctor. In this section we give a suitable context, $\cE$, for defining an internal category of elements, $\mD$, in $\Cat(\cE)$ for a pseudofunctor, $\cA \to \Cat(\cE)$, that is inspired by the usual category of elements. The context given in the following definition can be thought of as a kind of `local extensivity' condition on $\cE$ with respect to the pseudofunctor $D$.

\begin{defn}\label{def E admits an internal category of elements of D} 
We say $\cE$ {\em admits an internal category of elements of} $D : \cA \to \Cat(\cE)$ if

\begin{enumerate}
 \item for every $\varphi : A \to B$ in $\cA$, the pullback 
 
 \begin{center}
\begin{tikzcd}
	D_\varphi \dar["\pi_0"'] \rar["\pi_1"] & D(B)_1 \dar["s"] \\
	D(A)_0 \rar["D(\varphi)_0"'] & D(B)_0 
\end{tikzcd}
\end{center}
 
 \noi exists in $\cE$. 
 
 \item For any composable chain of maps $\varphi_i : A_i \to A_{i+1}$ in $\cA$, where $1 \leq i \leq n$ for an arbitrary $n \in \mN$, the pullback

 \[ D_{\varphi_1 ; \dots ; \varphi_n} = D_{\varphi_1} \tensor[_{\pi_1 t}]{\times}{_{\pi_0}} \dots \tensor[_{\pi_1 t}]{\times}{_{\pi_0}} D_{\varphi_{n+1}}\]
 
 \noi exists in $\cE$.
 
 \item Let $\cA_0$ denote the objects of $\cA$ and let $\cA_n$ denote the composable paths of length $n \geq 1$ in $\cA$. The coproducts
 
 \[ \mD_0 = \coprod_{A \in \cA_0} D(A)_0\]
 \noi and 
 \[ \mD_{\coprod (n) } = \coprod_{ (\varphi_i)_{i=1}^n \in \cA_n} D_{\varphi_1 ; ... ; \varphi_n} \]

 \noi exist for all $n \geq 1$ and are disjoint, with coprojections: 
  \[\iota_{\varphi_1 ; ... ; \varphi_n} : D_{\varphi_1 ; ... ; \varphi_n} \to  \mD_{\coprod (n) } \]
  
  \item The coproducts, $\mD_{\coprod(n)}$, are stable under pullbacks of source and target in the sense that 
  
  \[ \mD_{\coprod(n)} \cong \mD_1 \tensor[_t]{\times}{_s} \mD_1 \tensor[_t]{\times}{_s} \dots \tensor[_t]{\times}{_s} \mD_1 = \mD_n\] 
  
  \noi where $s , t : \mD_1 \to \mD_0$ are uniquely induced by the source and target maps. 
  
\begin{center}
\begin{tikzcd}[column sep = large, row sep = large]
\mD_2 \arrow[dr, phantom, "\usebox\pullback" , very near start, color=black] \dar["\rho_0"'] \rar["\rho_1"] & \mD_1 \dar[dotted, "s"'] & D_\varphi \lar[tail, "\iota_\varphi"'] \ar[dl, "\pi_0 \iota_A"] \\
\mD_1 \rar[dotted, "t"] & \mD_0 \\
D_\varphi \uar[tail, "\iota_\varphi"] \ar[ur,"\pi_1 t \iota_B"'] 
\end{tikzcd}
\end{center}
  \end{enumerate}
\end{defn}

\noi These conditions allow us to define the objects and structure maps of our internal category of elements, $\mD$. The last condition above should be thought of as an extensivity condition that allows us to to define the composition structure and prove $\mD$ is an internal category. 

For the rest of this section we will assume that $\cE$ admits an internal category of elements. 

Define the object of objects to be
\[ \mD_0 := \coprod_{A \in \cA_0} D(A)_0.\]

\noi \begin{rem}When $\cE = \Set$, we can think of the elements of $\mD_0$ as elements $a \in D(A)_0$ for each $A \in \cA_0$. This implies that every element of $\mD_0$ can be represented as a pair $(A,a)$ where $A \in \cA_0$ and $a \in D(A)_0$. 
\end{rem}\

\noi For any $\varphi : A \to B$ in $\cA_1$, we have the pullback 

\begin{center}
\begin{tikzcd}
	D_\varphi \dar["\pi_0"'] \rar["\pi_1"] & D(B)_1 \dar["s"]  \\
	D(A)_0 \rar["D(\varphi)_0"'] & D(B)_0 
\end{tikzcd}
\end{center}
 
\noi which is used to define the object of arrows:

\[ \mD_1 := \coprod_{\varphi \in \cA_1} D_\varphi \]

\begin{rem} When $\cE = \Set$ an arbitrary element of $\mD_1$ is an element of $D_\varphi$ for some unique $\varphi \in \cA_1$. In this case elements of $D_\varphi$ are pairs $(x, f)$ where $x \in D(A)_0$, $f \in D(B)_1$ and $D(\varphi)_0(x) = s(f)$. In this way every element of $\mD_1$ can be represented by a pair $(\varphi, f)$ where $f : D(\varphi)(x) \to y$ in $D(B)$. 
\end{rem}\

\noi To define source and target maps for $\mD$, it suffices to define them on the components, $D_\varphi$, for each $\varphi \in \cA_1$. Let 
\[ s_\varphi , t_\varphi : D_\varphi \to \mD_0 \]

\noi be defined as the composites on the top and left in the following diagram. 
 
\begin{center}
\begin{tikzcd}
	D_\varphi \arrow[dr, phantom, "\usebox\pullback" , very near start, color=black] \dar[dd, bend right = 50, "s_\varphi "'] \rar[rrr, bend left, "t_\varphi"]\dar["\pi_0"'] \rar["\pi_1"] & D(B)_1 \dar["s"] \rar["t"] & D(B)_0 \rar[tail, "\iota_B"] & \mD_0 \\
	D(A)_0 \dar[tail, "\iota_A"] \rar["D(\varphi)_0"'] & D(B)_0 \\
	\mD_0 & &
\end{tikzcd}
\end{center}

\noi These induce the source and target maps $s, t : \mD_1 \to \mD_0$ by the universal property of the coproduct $\mD_1$. Their pullback defines the object of composable arrows, $\mD_2$. 

\begin{center}
\begin{tikzcd}[column sep = large, row sep = large]
\mD_2 \arrow[dr, phantom, "\usebox\pullback" , very near start, color=black] \dar["\rho_0"'] \rar["\rho_1"] & \mD_1 \dar[dotted, "s"'] & D_\varphi \lar[tail, "\iota_\varphi"'] \ar[dl, "s_\varphi"] \\
\mD_1 \rar[dotted, "t"] & \mD_0 \\
D_\varphi \uar[tail, "\iota_\varphi"] \ar[ur,"t_\varphi"'] 
\end{tikzcd}
\end{center}

\noi For any $\varphi, \psi \in \cA_1$ we also pull $t_\varphi$ back along $s_\psi$ and denote the object $D_{\varphi ; \psi}$. If $\varphi$ and $\psi$ are not composable in $\cA$ then $s_\psi$ and $t_\varphi$ land in different components of $\mD_0$. In that case $D_{\varphi ; \psi}$ is trivial because coproducts are disjoint in $\cE$.

Using the universal property of the coproduct $\mD_2$ we describe composition in $\mD$ by defining composition on the cofibers $D_{\varphi ; \psi}$. Suppose $\varphi \in \cA(A, B)$ and $\psi \in \cA(B,C)$. Then $\varphi \psi \in \cA(A, C)$ and the outsides of the following diagrams commute

\[\label{dgm cofiber comp 1}
\begin{tikzcd}[]
D_{\varphi ; \psi} \ar[dr, dotted, "c'_{\delta ; (\varphi ; \psi)}"] \ar[d, "p_0"'] \ar[r, "p_0"]& D_\varphi  \rar["\pi_1"] & D(B)_1 \dar["D(\psi)_1"] \\
D_\varphi \dar["\pi_0"'] & D(C)_2 \rar["q_1 "] \dar["q_0"'] & D(C)_1 \dar["s"]  \\
D(A)_0 \rar["\delta_{\varphi , \psi}"'] & D(C)_1 \rar["t"'] & D(C)_0 \tag{$\star \star$}
\end{tikzcd} \]
\noi and 
\[\label{dgm cofiber comp 2}
\begin{tikzcd}[]
D_{\varphi ; \psi} \ar[dr, dotted, "c_{\varphi ; \psi}'"] \ar[d, "p_0"'] \ar[r, "p_1"]& D_\psi  \ar[dr, bend left, "\pi_1"] & \\
D_\varphi \dar["\pi_1"'] & D(C)_2 \rar["q_1 "] \dar["q_0"'] & D(C)_1 \dar["s"]  \\
D(B)_1 \rar["D(\psi)_1"'] & D(C)_1 \rar["t"'] & D(C)_0
\end{tikzcd}. \tag{$\star \star$}
\]

\noi To see the first diagram commutes we can check directly that 

\begin{align*}
p_0 \pi_0 \delta_{\varphi , \psi} t 
&= p_0 \pi_0 D(\varphi)_0 D(\psi)_0 & (\text{Def. } \delta_{\varphi , \psi}) \\
&= p_0 \pi_1 s D(\psi)_0 & (\text{Def. } D_\varphi)\\
&= p_0 \pi_1 D(\psi)_1 s& (D(\psi)\text{ an internal functor })
\end{align*}

\noi and to see the second square commutes it suffices to show that the canonical monic $\iota_C : D(C)_0 \rightarrowtail \mD_0$ coequalizes both sides of the diagram. 

\begin{align*}
p_0 \pi_1 D(\psi)_1 t \iota_C = 
&= p_0 \pi_1 t D(\psi)_0 \iota_C \\
&= p_0 \pi_1 t \iota_B \chi_\psi \\
&= p_0 t_\varphi \chi_\psi \\
&= p_1 s_\psi \chi_\psi \\
&= p_1 \pi_0 \iota_B \chi_\psi \\
&= p_1 \pi_0 D(\psi)_0 \iota_C \\
&= p_1 \pi_1 s \iota_C \\
\end{align*}

\noi Since $\iota_C$ is monic, we can conclude that the outer squares above commute and induce the maps $c'_{\varphi ; \psi}$ and $c'_{\delta ; (\varphi ; \psi)}$ by diagrams (\ref{dgm cofiber comp 1}) and (\ref{dgm cofiber comp 2}). Let $q_{01}, q_{12} : D(C)_3 \to D(C)_2$ denote the pullback projections of $D(C)_3$. Notice that 

\begin{align*}
 c'_{\varphi ; \psi} q_0 = c'_{\delta ; (\varphi ; \psi)} q_1
\end{align*}\

\noi so there is a unique map 

\begin{center}
\begin{tikzcd}[]
& D_{\varphi ; \psi} \ar[dl, bend right, "c'_{\delta ; (\varphi ; \psi)} "' ] \ar[dr, bend left, "c'_{\varphi ; \psi}"] \dar[dotted, "c'_{\delta ; \varphi ; \psi}"] &\\
D(C)_2 & D(C)_3 \lar["q_{01}"] \rar["q_{12}"'] & D(C)_2 
\end{tikzcd}
\end{center}

\noi which we can postcompose with triple-composition in $D(C)$ (given by associativity). Notice that 

\begin{align*}
p_0 \pi_0 D(\varphi \psi)_0 
&= p_0 \pi_0 \delta_{\varphi , \psi} s &\text{(Def. } \delta_{\varphi , \psi}) \\
&= c'_{\delta ;(\varphi ;\psi)} q_0 s &\text{(Def. } c'_{\delta ; (\varphi ; \psi)})\\
&= c'_{\delta ; \varphi ;\psi} q_{01} q_0 s &\text{(Def. } c'_{\delta ; \varphi ; \psi})\\
&= c'_{\delta ; \varphi ; \psi } c s & \text{(source-composite law in $D(C)$)}
\end{align*}\

\noi so there exists a unique `cofiber-wise composition' map as shown in the following diagram.

\begin{center}
\begin{tikzcd}[]
D_{\varphi ; \psi} \ar[dr, dotted, "c_{\varphi ; \psi}"] \ar[r, "c'_{\delta ;\varphi ; \psi}"] \ar[d, "p_0"'] & D(C)_3 \ar[dr, bend left, "c"] & \\
D_\varphi \ar[dr, bend right = 30, "\pi_0"'] & D_{\varphi \psi} \rar["\pi_1"] \dar["\pi_0"']\arrow[dr, phantom, "\usebox\pullback" , very near start, color=black] & D(C)_1 \dar["s"]\\
& D(A)_0 \rar["D(\varphi \psi)_0"'] &D (C)_0 
\end{tikzcd}
\end{center}\

\noi Define composition in $\mD$ as the universal map out of the coproduct $\mD_2$ induced by the family of maps $\{ c_{\varphi ; \psi} \iota_{\varphi \psi} \}_{ (\varphi, \psi) \in \cA_2}$. 

\begin{center}
\begin{tikzcd}[column sep = large, row sep = large]
\mD_2 \rar[dotted, "c"] & \mD_1 \dar[""] \\
D_{\varphi ; \psi} \rar["c_{\varphi, \psi}"'] \uar[tail, "\iota_{\varphi ; \psi}"] & D_{\varphi \psi} \uar[tail, "\iota_{\varphi \psi}"'] 
\end{tikzcd}
\end{center}

The identity structure map 

\[\epsilon : \mD_0 \to \mD_1\]

\noi is defined as the universal map out of the coproduct $\mD_0$ induced by a family of unique maps $\epsilon_A$

\begin{center}
\begin{tikzcd}[]
\mD_0 \rar[dotted, "\epsilon"] & \mD_1 \\
D(A)_0 \uar[tail, "\iota_A"] \rar["\epsilon_A"']& D_{1_A} \uar[tail, "\iota_{1_A}"'] 
\end{tikzcd}
\end{center}

\noi where $\epsilon_A = \langle 1_{D(A)_0} , \delta_A \rangle$ is induced by the components $\delta_{A} : D(A)_0 \to D(A)_1$ of the natural isomorphism $D(1_A) \cong 1_{D(A)}$ and the identity map $1_{D(A)_0}$. This is well-defined because

\[\delta_A s = D(1_A)_0\]\

\noi by definition of $\delta_A$:

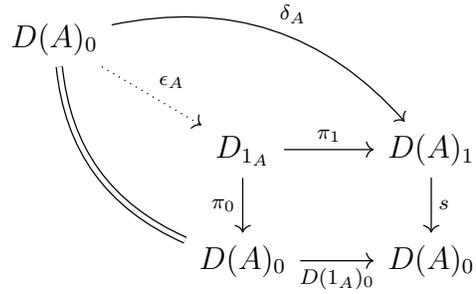
\begin{center}
\begin{tikzcd}[]
D(A)_0 \ar[ddr, bend right, equals] \ar[drr, bend left, "\delta_A"] \ar[dr, dotted, "\epsilon_A"] & & &\\
& D_{1_A} \rar["\pi_1"] \dar["\pi_0"'] & D(A)_1 \dar["s"] &\\
& D(A)_0 \rar["D(1_A)_0"'] & D(A)_0 &
\end{tikzcd}
\end{center}\

\noi 
The following commuting diagram shows that the source and target maps of $\mD$ are compatible with these identity maps. 

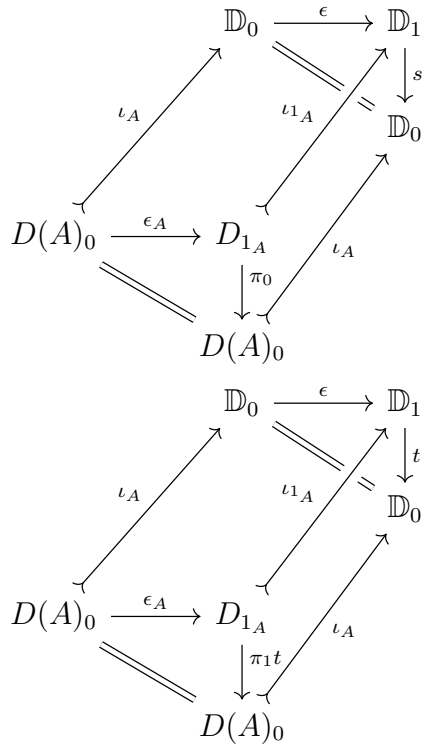
\begin{center}
\begin{tikzcd}[]
&\mD_0 \rar["\epsilon"]  & \mD_1 \dar["s"] && \\
& & \mD_0 \ar[from = ul, equals] \ar[from = ddl, crossing over, tail, "\iota_A"']&& \\
D(A)_0 \ar[uur, tail, "\iota_A"] \rar["\epsilon_A"] \ar[dr, equals] & D_{1_A} \ar[d, near start, "\pi_0"] \ar[uur, crossing over, tail, "\iota_{1_A}"] & & \\
& D(A)_0 & & & 
\end{tikzcd}\qquad \qquad 
\begin{tikzcd}[]
&\mD_0 \rar["\epsilon"]  & \mD_1 \dar["t"] && \\
& & \mD_0 \ar[from = ul, equals] \ar[from = ddl, crossing over, tail, "\iota_A"']&& \\
D(A)_0 \ar[uur, tail, "\iota_A"] \rar["\epsilon_A"] \ar[dr, equals] & D_{1_A} \ar[d, near start, "\pi_1 t"] \ar[uur, crossing over, tail, "\iota_{1_A}"] & & \\
& D(A)_0 & & & 
\end{tikzcd}
\end{center}

\noi The top squares commute by definition of $\epsilon$, the bottom squares commute trivially, and the squares on the right commute by definition of $s_{1_A}$ and $t_{1_A}$ respectively. The left front triangle commutes by definition of $\epsilon_A$ and the right front triangle commutes by definition of $\delta_A$. More precisely, on the left we have

\[\epsilon_A \pi_1 t = \delta_A t = 1_{D(A)_0}. \]

\subsection{Associativity and Identity Laws}\label{S Assoc and Id Laws Int Groth}

\phantomsection
\subsubsection*{Classical} 
\addcontentsline{toc}{subsubsection}{Classical}

When $\cE = \Set$, the identity arrows of the usual Grothendieck construction are pairs $(1_A , \delta_{A,a})$ for each object $(A,a)$ where $A \in \cA$ and $a \in D(A)_0$. The coherence law between the natural isomorphisms, $\delta_{1_A ; \varphi} , \delta_A$ and $\delta_{\varphi ; 1_B} , \delta_B$ respectively, says that pasting the 2-cells $\delta_{1_{A} ; \varphi}$ and $\delta_A$ and the 2-cells $\delta_{\varphi ; 1_B}$ and $\delta_B$ is equal to $1_{D(\varphi)}$ respectively. This means that at the level of components we have a commuting diagram 

\begin{center}
\begin{tikzcd}[]
D(\varphi)(a) \rar["\delta_{\varphi ; 1_B , a}"] \dar["\delta_{1_A ; \varphi , a} "'] \ar[dr, equals] & \left( D(\varphi) D(1_B) \right) (a) \dar["\delta_{B , D(\varphi)(a)}"] \\
\left( D(1_A) D(\varphi) \right) (a) \rar["D(\varphi) (\delta_{A,a})"'] & D(\varphi)(a) 
\end{tikzcd}
\end{center}

\noi in $D(B)$ for each $a \in D(A)_0$. The upper and lower triangles are necessary for proving the right and left identity laws for the Grothendieck construction respectively. For example, the lower coherence precisely cancels the isomorphisms we pick up in our definitions of identity and composition in order for the left identity law to hold.

\begin{align*}
(1_A , \delta_{A,a}) (\varphi , f ) = \left( 1_A \varphi , \delta_{1_A ; \varphi , a} D(\varphi)(\delta_{A,a}) f \right) = (\varphi , f)
\end{align*}\

The associativity law relies on the other coherence law for $D$ that gives the following commutating squares for every composable triple, $\varphi , \psi$, and $\gamma$, involving the components

\[ \begin{tikzcd}[column sep = large, row sep = large]
D(\varphi \psi \gamma) (a) \rar["\delta_{\varphi ; \psi \gamma ,a}"] \dar["\delta_{\varphi \psi ; \gamma , a}"'] & D(\varphi) D(\psi \gamma)(a) \dar["\delta_{\psi; \gamma , D(\varphi)(a))}"] \\
D(\varphi \psi) D(\gamma) (a) \rar["D(\gamma)(\delta_{\varphi ; \psi , a})"'] & D(\varphi)D(\psi)D(\gamma)(a) 
\end{tikzcd}. \]

\phantomsection
\subsubsection*{Internal} 
\addcontentsline{toc}{subsubsection}{Internal}

Now we give internal translations of the proofs of the associativity and identity laws for the Grothendieck construction. The following proposition states that the identity map $\epsilon : \mD_1 \to \mD_0$ of the internal category of elements satisfies the identity laws. 

\begin{prop}[Identity Laws for $\mD$] \label{prop id laws for int groth}
Given the following pullbacks, 

\begin{center}
\begin{tikzcd}[]
& & \mD_1 \times_{\mD_0} \mD_0 \arrow[dr, phantom, "\usebox\pullback" , very near start, color=black]\rar[] \dar["\rho_0"'] & \mD_0 \dar[equals] & \\
&\mD_0 \times_{\mD_0} \mD_1 \rar["\rho_1"] \dar[]  \arrow[dr, phantom, "\usebox\pullback" , very near start, color=black]& \mD_1 \dar["s"] \rar["t"'] & \mD_0 & \\
& \mD_0 \rar[equals] & \mD_0& & 
\end{tikzcd}
\end{center}

\noi let $\langle \rho_0 \epsilon {,} \rho_1 \rangle$ and $\langle \rho_0 {,} \rho_1 \epsilon \rangle$ be the universal maps induced by the pairs of pullback projections with $\epsilon$ postcomposed respectively. Then the following diagram commutes.

\begin{center}
\begin{tikzcd}[]
\mD_0 \times_{\mD_0} \mD_1 \rar["\langle \rho_0 \epsilon {,} \rho_1 \rangle"] \ar[dr, "\rho_1"'] & \mD_2 \dar["c"] & \mD_1 \times_{\mD_0} \mD_0 \lar["\langle \rho_0 {,} \rho_1 \epsilon \rangle"'] \ar[dl, "\rho_0"] \\
& \mD_1 & 
\end{tikzcd}
\end{center}

\end{prop}\

\begin{proof}
For each $A \in \cA_0$ and each $\varphi \in \cA_1$, we have the following commuting diagram 
\begin{center}
\begin{tikzcd}[column sep = small]

 & \mD_1 \times_{\mD_0} \mD_0 \arrow[dr, phantom, "\usebox\pullback" , very near start, color=black]\rar[] \dar["\rho_0"'] 
 & \mD_0 \dar[equals] \rar["\epsilon"]
 & \mD_1 
 & 
 &
 \\

\mD_0 \times_{\mD_0} \mD_1 \rar["\rho_1"] \dar[]  \arrow[dr, phantom, "\usebox\pullback" , very near start, color=black]
& \mD_1 \dar["s"] \rar[near start, "t"'] & \mD_0 &&&& \\

 \mD_0 \rar[equals] \dar["\epsilon"']  & \mD_0 & & &&\\

\mD_1 & & D_\varphi \tensor[_{\pi_1 t} ]{\times}{_1} D(B)_0 \ar[uuul, tail, dotted, crossing over] \arrow[dr, phantom, "\usebox\pullback" , very near start, color=black]\rar[] \dar["p_0"'] & D(B)_0 \dar[equals] \ar[uuul, tail, crossing over, dotted] \rar[rr, "\epsilon_A"] &&D_{1_B} \ar[uuull, tail, dotted] \\

&D(A)_0 \tensor[_1]{\times}{_{\pi_0}} D_\varphi  \ar[uuul, tail, dotted, crossing over] \dar[]  \arrow[dr, phantom, "\usebox\pullback" , very near start, color=black] \ar[r, crossing over, "p_1"]& D_\varphi \ar[uuul, tail, dotted] \dar["\pi_0"] \rar["\pi_1 t"'] & D(B)_0 \ar[uuul, tail, dotted] &\\

& D(A)_0 \dar["\epsilon_A"] \ar[uuul, tail, dotted] \rar[equals] & D(A)_0 \ar[uuul, tail, dotted]& &&\\

& D_{1_A} \ar[uuul, tail, dotted] & & & & 
\end{tikzcd}
\end{center}\

\noi where the dotted arrows are all coproduct monos by stability of coproducts under pullback. Let $\langle p_0 \epsilon_A , p_1 \rangle $ and $\langle p_1 , p_0 \epsilon_A \rangle$ be universal maps out of $D_{1_A ; \varphi}$ and $D_{\varphi ; 1_B}$ induced by the pairs $p_0 \epsilon_A , p_1$ and $p_0 , p_1 \epsilon_B$ respectively. We have a similar diagrams for each of the triangles in the proposition.

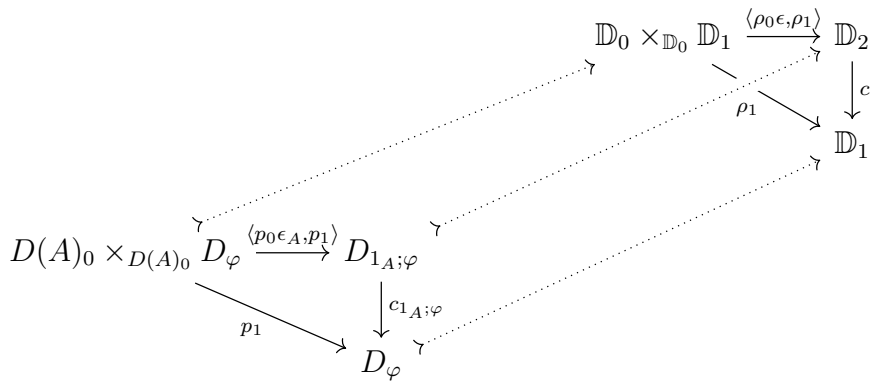
\begin{center}
\begin{tikzcd}[]
& & & \mD_0 \times_{\mD_0} \mD_1 \rar["\langle \rho_0 \epsilon {,} \rho_1 \rangle"] \ar[dr, "\rho_1"'] 
& \mD_2 \dar["c"] 
 & & \\
& & & & \mD_1 \\
 D(A)_0 \times_{D(A)_0} D_\varphi \rar["\langle p_0 \epsilon_A {,} p_1 \rangle"] \ar[dr, "p_1"'] 
 \ar[uurrr, dotted, tail, crossing over] 
& D_{1_A ; \varphi} \dar[ "c_{1_A ; \varphi}"]  \ar[uurrr, dotted, tail, crossing over] &&& \\
& D_\varphi  \ar[uurrr, dotted, tail, crossing over] & & & 
\end{tikzcd}
\end{center}

\noi and so it suffices to show the component triangles commute. Each case similarly follows by the universal property of the pullback $D_\varphi$ so we only show the proof for the diagram above. By the pullback square defining $D(A)_0 \times_{D(A)_0} D_\varphi$ above we have that 
\begin{align*}
p_0 D(\varphi)_0 = p_1 \pi_0 D(\varphi)_0 = p_1 \pi_1 s .
\end{align*}

\noi This induces the unique map in the following commuting diagram. 

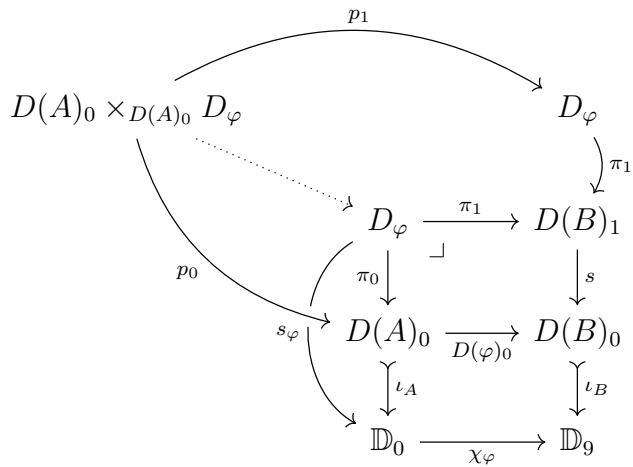
\begin{center}
\begin{tikzcd}[]
&D(A)_0 \times_{D(A)_0} D_\varphi \ar[dr, dotted, ""] \rar[rr, bend left, "p_1"] && D_\varphi \dar[ bend left, "\pi_1"] \\ 
& & D_\varphi \arrow[dr, phantom, "\usebox\pullback" , very near start, color=black] \ar[dd, bend right = 60, "s_\varphi"'] \dar["\pi_0"'] \rar["\pi_1"] & D(B)_1 \dar["s"] \\
& & D(A)_0 \arrow[from = luu,  crossing over, bend right, "p_0"'] \rar["D(\varphi)_0"'] \dar[tail, "\iota_A"] & D(B)_0 \dar[tail, "\iota_B"] \\
& & \mD_0 \rar["\chi_\varphi"'] & \mD_9 \\
\end{tikzcd}
\end{center}

\noi It suffices to check that plugging in $\langle p_0 \epsilon_A , p_1 \rangle c_{1_A ; \varphi} $ and $p_1$ as the dotted arrow both make the triangles above commute. By definition of $D(A)_0 \times_{D(A)_0} D_\varphi $ we have that 

\[ p_1 \pi_0 = p_0 \]

\noi and we tautologically know $p_1 \pi_1 = p_1 \pi_1$ so we only need to check what happens when postcomposing $\langle p_0 \epsilon_A , p_1 \rangle c_{1_A ; \varphi} $ with $\pi_0$ and $\pi_1$. First notice that 

\begin{align*}
\langle \epsilon_A , 1_{D_\varphi} \rangle c_{1_A ; \varphi} \pi_0 \iota_A 
&=\langle \epsilon_A , 1_{D_\varphi} \rangle c_{1_A ; \varphi} s_\varphi & \text{(Def : }s_\varphi ) \\
&=\langle \epsilon_A , 1_{D_\varphi} \rangle p_0 s_{1_A} &\text{(Lemma } \ref{Lem source and target of component composite} ) \\
&= p_0 \epsilon_A s_{1_A} & \text{(Def : }\langle \epsilon_A , 1_{D_\varphi} \rangle) \\
&= p_0 \epsilon_A \pi_0 \iota_A&\text{(Def : }s_\varphi ) \\
&= p_0 \iota_A &\text{(Def : } \epsilon_A )
\end{align*}

\noi implies that 

\[ \langle p_0 \epsilon_A , p_1 \rangle c_{1_A ; \varphi} \pi_0 = p_0\]

\noi since $\iota_A$ is monic. Now by definition of $e_B$ 

\[ s e_B t = s \]

\noi and this induces a unique map $\langle s e_B , 1 \rangle D(B)_1 \to D(B)_2$ which factors uniquely as 

\[ \langle s e_B , 1 \rangle = \langle s , 1 \rangle \langle e_B , 1 \rangle . \]

\noi Recall that composition in $\mD$ is defined cofiber-wise and is in part induced by composition in the internal category associated to the codomain of a composable pair in $\cA$. It's reasonable to think that the family of composable pairs of arrows in $D(B)$ indexed by the composite $\langle p_0 \epsilon_A , p_1 \rangle c'_{1_A ; \varphi }$ should simplify in terms of identity maps, $e_B$, at the internal level. More precisely, we claim that

\[\langle p_0 \epsilon_A , p_1 \rangle c'_{1_A ; \varphi } = p_1 \pi_1 \langle s e_B , 1\rangle \]

\noi and for this we use the universal property of the pullback $D(B)_2$. By pasting commuting squares we have that 

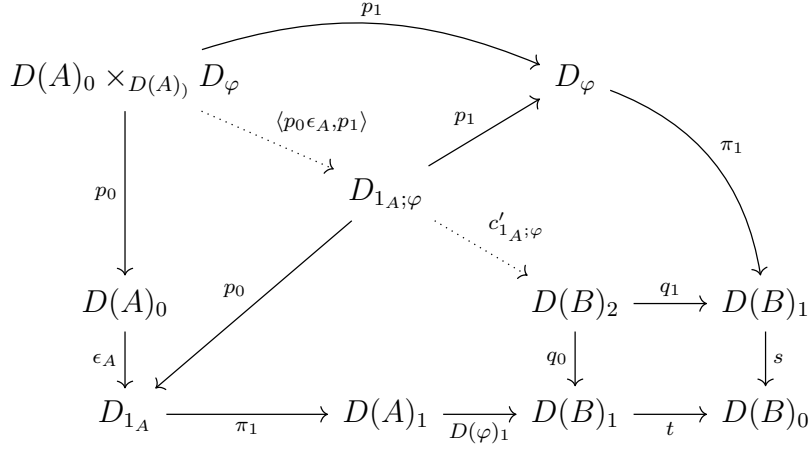
\begin{center}
\begin{tikzcd}[]
D(A)_0 \times_{D(A)_)} D_\varphi \dar[dd, "p_0"'] \ar[dr, dotted, "\langle p_0 \epsilon_A {,} p_1 \rangle "] \ar[rr, bend left = 20, "p_1"] & & D_\varphi \ar[ddr, bend left, "\pi_1"] & & \\
& D_{1_A ; \varphi} \ar[ddl, "p_0"'] \ar[ur, "p_1"] \ar[dr, dotted, "c'_{1_A ; \varphi}"] && & \\
D(A)_0 \ar[d, "\epsilon_A"'] & & D(B)_2 \rar["q_1 "] \dar["q_0 "'] & D(B)_1 \dar["s"] & \\
D_{1_A} \rar["\pi_1"'] &D(A)_1 \rar["D(\varphi)_1"'] & D(B)_1 \rar["t"'] & D(B)_0 & \\
\end{tikzcd}
\end{center}

\noi commutes. Now 

\[ p_1 \pi_1 \langle s e_B , 1\rangle q_1 = p_1 \pi_1 \]

\noi and 

\begin{align*}
p_1 \pi_1 \langle s e_B , 1\rangle q_0 
&= p_1 \pi_1 s e_B & (\text{Def } \langle s e_B , 1 \rangle )\\\
&= p_1 \pi_0 D(\varphi)_0 e_B &(\text{Def } D_\varphi )\\
&= p_0 D(\varphi)_0 e_B &(\text{Def } D(A)_0 \times_{D(A)_0} D_\varphi )\\
&= p_0 e_A D(\varphi)_1 &(\text{(internal) functoriality})\\
&= p_0 \epsilon_A \pi_1 D(\varphi)_1 & (\text{Def } \epsilon_A)\\
\end{align*}

\noi so by uniqueness we have that 

\[ \langle p_0 \epsilon_A , p_1 \rangle c'_{1_A ; \varphi } = p_1 \pi_1 \langle s e_B , 1\rangle. \]\

\noi This allows us to consider the following cone 

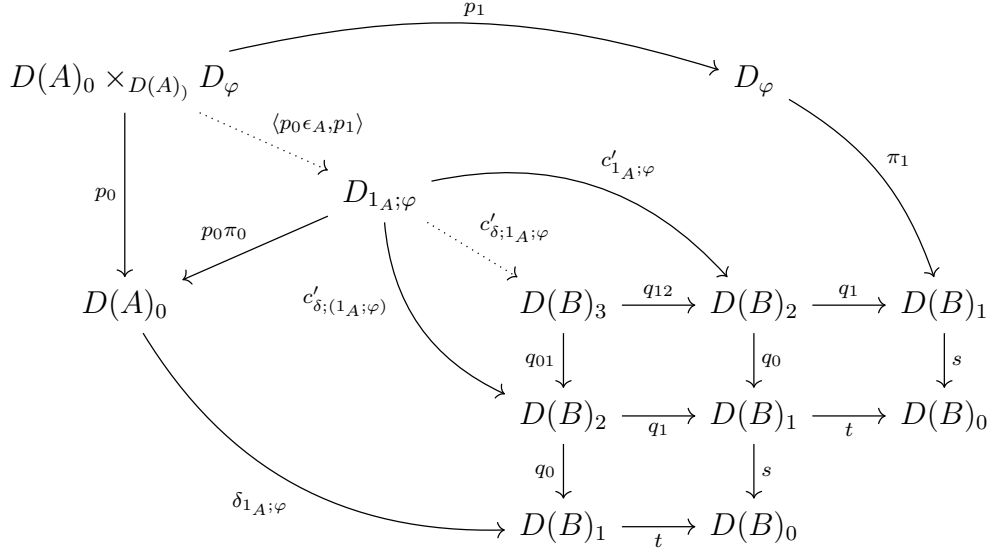
\begin{center}
\begin{tikzcd}[]
D(A)_0 \times_{D(A)_)} D_\varphi \dar[dd, "p_0"'] \ar[dr, dotted, "\langle p_0 \epsilon_A {,} p_1 \rangle "] \ar[rrr, bend left = 15, "p_1"] & & & D_\varphi \ar[ddr, bend left=20, "\pi_1"] & \\
& D_{1_A ; \varphi} \ar[dl, "p_0 \pi_0"'] \ar[ddr, bend right, near start, "c'_{\delta ; (1_A ; \varphi)} "'] \ar[drr, bend left, "c'_{1_A ; \varphi} "] \ar[dr, dotted, "c'_{\delta ; 1_A ; \varphi}"] && & \\
D(A)_0 \ar[ddrr, bend right, "\delta_{1_A ; \varphi} "'] & & D(B)_3 \rar["q_{12} "] \dar["q_{01} "'] & D(B)_2 \dar["q_0"] \rar["q_1"] & D(B)_1 \dar["s"] \\
 && D(B)_2\rar["q_1"'] \dar["q_0"'] & D(B)_1 \dar["s"] \rar["t"'] & D(B)_0 \\
&& D(B)_1 \rar["t"'] &D(B)_0 & \\

\end{tikzcd}
\end{center}

\noi which is construcucted by pasting commuting squares and triangles. Notice that

\[\langle p_0 \epsilon_A , p_1 \rangle c'_{\delta_A ; 1_A ; \varphi } = \langle p_0 \delta_{1_A;\varphi} , p_0 \delta_A D(\varphi)_1 , p_1 \pi_1 \rangle \]

\noi where the left and right components can be seen in the commuting diagram above and the middle component is verified by checking that

\begin{align*}
\langle p_0 \epsilon_A , p_1 \rangle c'_{\delta ; (1_A ; \varphi)} q_1 
= \langle p_0 \epsilon_A , p_1 \rangle p_0 \pi_1 D(\varphi)_1 
= \langle p_0 \epsilon_A , p_1 \rangle c'_{1_A ; \varphi} q_0.
\end{align*}

\noi In fact

\[ \langle p_0 \epsilon_A , p_1 \rangle p_0 \pi_1 D(\varphi)_1= p_0 \epsilon_A \pi_1 D(\varphi)_1= p_0 \delta_A D(\varphi)_1\]\

\noi shows what the middle component must be in the composable triple. After forming the composite of this triple in $D(B)$ we should have the coherence isomorphisms canceling by the coherence law for $\delta_{1_A ; \varphi}$ and $\delta_A$, and we formalize this internally by using associativity in $D(B)$ first along with the coherence law for the structure isomorphisms of $D$ that say 

\[\langle \delta_{1_A;\varphi} , \delta_A D(\varphi)_1 \rangle c = e_A D(\varphi)_1 \]
\noi and 
\begin{align*}
\langle p_0 \delta_{1_A;\varphi} , \delta_A D(\varphi)_1 , p_1 \pi_1 \rangle c 
&= \langle \langle p_0 \delta_{1_A;\varphi} , p_0 \delta_A D(\varphi)_1 \rangle c , p_1 \pi_1 \rangle c \\
&= \langle p_0 \langle \delta_{1_A;\varphi} , \delta_A D(\varphi)_1 \rangle c , p_1 \pi_1 \rangle c \\
&= \langle p_0 e_A D(\varphi)_1 , p_1 \pi_1 \rangle c \\
&= \langle p_0 D(\varphi)_0 e_B , p_1 \pi_1 \rangle c \\
&= \langle p_1 \pi_0 D(\varphi)_0 e_B , p_1 \pi_1 \rangle c & \text{(Def. } D(A)_0 \times_{D(A)_0} D_\varphi ) \\
&= \langle p_1 \pi_1 s e_B , p_1 \pi_1 \rangle c \\ 
&= p_1 \pi_1 \langle s e_B , 1 \rangle c \\
&= p_1 \pi_1 \langle s , 1 \rangle \langle e_B , 1 \rangle c \\
&= p_1 \pi_1 \langle s , 1 \rangle \\
&= p_1 \pi_1 .
\end{align*}

\noi and now we can put our calculations above together to see that 

\begin{align*}
\langle p_0 \epsilon_A , p_1 \rangle c_{1_A ; \varphi} \pi_1 
&= \langle p_0 \epsilon_A , p_1 \rangle c'_{\delta ; 1_A ; \varphi } c & \text{(Def: } c_{1_A ; \varphi}) \\
&= \langle p_0 \delta_{1_A ; \varphi } , p_1 \pi_1 e_B , p_1 \pi_1 \rangle c & (\text{above}) \\
&= p_1 \pi_1 & (\text{above})
\end{align*}

\noi and by the universal property of $D_\varphi$ we can conclude

\[ \langle p_0 \epsilon_A , p_1 \rangle c_{1_A ; \varphi} = p_1. \]

\end{proof}

To see that $\mD$ is an internal category in $\cE$ with the structure defined above it only remains to show that composition is associative. This proof is long and technical, follows by a similar pattern to the proof for the identity laws, and ultimately relies on proving associativity on the cofibers of the composable triples coproduct and using the universal property of coproducts. 

\begin{prop}\label{Prop composition is associative}
Composition in $\mD$ is associative. 
\end{prop}
\begin{proof}
By extensivity of $\cE$ it suffices to show that cofiber composition is associative and this is shown using several lemmas along with the universal property of each cofiber of $\mD_1$. A complete proof can be seen in the appendix, precisely in Proposition~\ref{Prop composition is associative - appendix}. 
\end{proof}

\noi This brings us to the main theorem of this section. 

\begin{thm}[The Internal Category $\mD$] \label{thm int groth is an internal cat}
The objects, $(\mD_0, \mD_1)$, along with the structure maps $s , t : \mD_1 \to \mD_0$ and $c : \mD_2 \to \mD_1$ defined above form an internal category in $\cE$. 
\end{thm}
\begin{proof}
The required objects, structure maps, and pullbacks exist when $\cE$ admits an internal category of elements. The associativity and identity laws follow from Propositions \ref{Prop composition is associative} and \ref{prop id laws for int groth}. 
\end{proof}

\section{Internal Category of Elements as an Oplax Colimit}\label{S Int Groth as Oplax Colim}

In Sections \ref{SS canonical transformation 1-cells} and \ref{SS 2-cells of l} we define the 1-cells and 2-cells of a canonical lax natural transformation from a small diagram of internal categories that admits an internal category of elements into the constant functor on said internal category of elements
\[ \ell : D \implies \Delta \mD .\]

\noi In Section \ref{SS UP for 1-cells of laxtransfm} we prove that a lax transformation $D \implies \Delta X$ corresponds uniquely to an internal functor $\mD \to X$. Section \ref{SS UP of 2-cells for laxtransfm} shows modifications of lax transformations $D \implies \Delta X$ correspond uniquely to internal natural transformations of internal functors $\mD \to X$. Section \ref{SS IntCatofEls is OpLaxColim} combines these results with functoriality to give an equivalence of categories that establishes $\mD$ as the oplax colimit of $D$.

\subsection{Canonical Transformation 1-cells}\label{SS canonical transformation 1-cells}

\phantomsection
\subsubsection*{Classical}
\addcontentsline{toc}{subsubsection}{Classical}

For a diagram of small categories $D : \cA \to \Cat(\Set)$, for each $A \in \cA$ there is a functor $\ell_A: D(A) \to \mD$. On an arbitrary object $a \in D(A)_0$, it is defined as

\[ \ell_A(a) = (A,a). \]\

\noi For an arrow $f \in D(A)(a,b)$, it is defined as

\[ \ell_A(f) = (1_A, \delta_{A, a} f )\]\

\noi because $s(\delta_{A,a} f) = s(\delta_{A,a}) = D(1_A)(a)$. 

\begin{center}
\begin{tikzcd}[]
D(1_A)(a) \ar[dr, ""'] \rar["\delta_{A,a}", "\cong"'] & a \dar["f"] \\
& b
\end{tikzcd}.
\end{center}\

\noi Identities are preserved by the identity law in $\cA$ in the left component along with coherence in the right component, for any $\varphi : A \to B$ in $\cA$ and any $f: D(\varphi)(a) \to b$ in $D(B)_1$

\begin{align*}
\ell_A(1_a) (\varphi , f)
&=(1_A , \delta_{A,a})(\varphi , f) \\
&= (1_A \varphi , \delta_{1_A ; \varphi , a } D(\varphi) (\delta_{A,a}) f ) & \text{Def.} \\
&= (\varphi , f) & \text{Coherence}\\
&= (\varphi 1_B , \delta_{\varphi ; 1_B , a} \delta_{B, D(\varphi)(a)} f &\text{Coherence} \\ 
&= (\varphi 1_B , \delta_{\varphi ; 1_B, a} D(1_B)(f) \delta_{B,b}) & \text{Naturality}\\
&= (\varphi , f) (1_B , \delta_{B,b}) & \text{Def.} \\
&= (\varphi , f) \ell_B(1_b) . 
\end{align*} 

\noi Similarly, composition is preserved in the left component because it is defined as composition in $\cA$, and in the right component we only need naturality of the identity coherence isomorphism. For any $f : a \to b$ and $g : b \to c$ in $D(A)_1$, we have that 

\begin{align*}
\ell_A (f) \ell_A (g)
&=(1_A , \delta_{A,a} f) ( 1_A , \delta_{A,b} g) &\text{Def.}\\
&=(1_A 1_A , \delta_{1_A ; 1_A , a} D(1_A)(\delta_{A,a} f) \delta_{A,b} g) & \text{Def.} \\
&= \left( 1_A , \delta_{1_A ; 1_A , a} D(1_A)(\delta_A) D(1_A) (f) \delta_{B,b} g \right) & \text{Functoriality} \\
&= \left( 1_A , \delta_{1_A ; 1_A , a} D(1_A)(\delta_A) \delta_{A,a} f g \right) &\text{Naturality }\delta_A \\
&= \left( 1_A , \delta_{1_A ; 1_A , a} \delta_{A, D(1_A)(a)} \delta_{A,a} f g \right) &\text{Naturality} \delta_A \\
&= \left( 1_A , \delta_{A,a} f g \right) &\text{Coherence} \\
&= (\ell_A)(fg) & \text{Def.}
\end{align*}\

\phantomsection
\subsubsection*{Internal}
\addcontentsline{toc}{subsubsection}{Internal}

\noi The definitions and proofs above can be internalized within an arbitrary extensive category $\cE$ as follows. For each $A \in \cA_0$, notice that 

\[ s \delta_A t = s 1_{D(A)_0} = s = 1_{D(A)_1} s \]

\noi so there exists a unique map $\langle s \delta_A , 1_{D(A)_1} \rangle : D(A)_1 \to D(A)_2$ in $\cE$. Now 

\[ \langle s \delta_A , 1_{D(A)_1} \rangle c s = \langle s \delta_A , 1_{D(A)_1} \rangle q_0 s = s \delta_A s = s D(1_A)_0 \]\

\noi induces a unique map $(\ell_A)'_1 := \langle s , \langle s \delta_A , 1_{D(A)_1} \rangle c \rangle$ which we can use to define $\ell_A = \left( (\ell_A)_0 , (\ell_A)_1 \right)$: 

\[ \begin{tikzcd}[]
D(A)_0 \rar[rr,"(\ell_A)_0 {:}= \iota_A "] && \mD_0 & & D(A)_1 \ar[drr, "(\ell_A)_1"'] \rar[rr,"(\ell_A)'_1 " ] & & D_{1_A} \dar[tail, "\iota_A"] \\
& && &&& \mD_1 
\end{tikzcd}. \]

\begin{lem}\label{lem ell_A preserved ids}
Identities are preserved by $\ell_A$. That is, the diagram

\begin{center}
\begin{tikzcd}[]
D(A)_0 \dar["(\ell_A)_0"'] \rar["e_A"] & D(A)_1 \dar["(\ell_A)_1"] \\
\mD_0 \rar["\epsilon"'] & \mD_1 
\end{tikzcd}
\end{center}

\noi commutes in $\cE$. 
\end{lem}
\begin{proof}
First compute 

\begin{align*}
e_A(\ell_A)'_1 \pi_0 
= e_A s 
= 1_{D(A)_0}
\end{align*}

\noi and 

\begin{align*}
e_A (\ell_A)'_1 \pi_1 
&= e_A \langle s \delta_A , 1_{D(A)_1} \rangle c\\
&= \langle e_A s \delta_A , e_A 1_{D(A)_1} \rangle c\\
&= \langle \delta_A , e_A \rangle c \\
&= \delta_A 
\end{align*}\

\noi by the identity law in D(A) and then see 

\[ e_A (\ell_A)'_1 = \langle 1_{D(A)_0} , \delta_A \rangle = \epsilon_A,\]\

\noi by definition of $\epsilon_A$. Now post-composing with $\iota_{1_A}$ and using the equality above along with the definitions of $\epsilon_A$ and $\epsilon$ gives 

\begin{align*}
e_A (\ell_A)_1 
= e_A(\ell_A)'_1 \iota_{1_A} 
= \epsilon_A \iota_{1_A} 
= \iota_A \epsilon 
\end{align*}\

\noi as required. 
\end{proof}

\noi Due to extensivity and our definition of $\mD$ involving coproducts, in order to prove composition is preserved by $\ell_A$ we need to prove that composition is preserved by $\ell_A'$ at the level of cofibers. This is done in Lemma~\ref{Lem 1-cells preserve composition at cofiber level} in the appendix. 

\begin{lem}\label{lem ell_A preserves comp}
For each $A \in \cA_0$, composition is preserved by $\ell_A$. That is, the diagram 

\begin{center}
\begin{tikzcd}[]
D(A)_2 \rar["c"] \dar["\langle q_0 (\ell_A)_1 {,} q_1 (\ell_A)_1 \rangle "'] & D(A)_1 \dar["(\ell_A)_1"] & \\
\mD_2 \rar["c"'] & \mD_1
\end{tikzcd}
\end{center}

\noi commutes in $\cE$. 
\end{lem}
\begin{proof}
First notice that 

\begin{align*}
\langle q_0 (\ell_A)'_1 {,} q_1 (\ell_A)'_1 \rangle \iota_{1_A; 1_A} \rho_0 
&=\langle q_0 (\ell_A)'_1 {,} q_1 (\ell_A)'_1 \rangle p_0 \iota_{1_A} & \text{Def. } D_{1_A ; 1_A} \\
&= q_0 (\ell_A)'_1 \iota_{1_A} & \\
&= q_0 (\ell_A)_1 & \text{Def. } 
\end{align*}

and 

\begin{align*}
\langle q_0 (\ell_A)'_1 {,} q_1 (\ell_A)'_1 \rangle \iota_{1_A; 1_A} \rho_1
&= 
\langle q_0 (\ell_A)'_1 {,} q_1 (\ell_A)'_1 \rangle p_1 \iota_{1_A} \\
&= q_1 (\ell_A)'_1 \iota_{1_A} \\
&= q_1 (\ell_A)_1 
\end{align*}\

\noi so by the universal property of $\mD_2$,

\[\langle q_0 (\ell_A)_1 {,} q_1 (\ell_A)_1 \rangle = \langle q_0 (\ell_A)'_1 {,} q_1 (\ell_A)'_1 \rangle \iota_{1_A; 1_A}. \]\

\noi Use the equation above along with Lemma~\ref{Lem 1-cells preserve composition at cofiber level},

\begin{align*}
\langle q_0 (\ell_A)_1 {,} q_1 (\ell_A)_1 \rangle c 
&= \langle q_0 (\ell_A)'_1 {,} q_1 (\ell_A)'_1 \rangle \iota_{1_A; 1_A} c & \\
&= \langle q_0 (\ell_A)'_1 {,} q_1 (\ell_A)'_1 \rangle c_{1_A; 1_A} \iota_{1_A}& \text{Def. } c \\
&= c (\ell_A)'_1 \iota_{1_A}& \text{Lemma }\ref{Lem 1-cells preserve composition at cofiber level} \\
&= c (\ell_A)_1 & \text{Def. } (\ell_A)_1 
\end{align*}\

\noi to see the square in question commutes. 
\end{proof}

\noi The following proposition is the main result of this subsection. 

\begin{prop}\label{prop 1-cells ell_A are internal functors}
For each $A \in \cA_0$, $\ell_A : D(A) \to \mD$ is an internal functor.
\end{prop}
\begin{proof}
It preserves identities by Lemma~\ref{lem ell_A preserved ids} and it preserves composition by Lemma~\ref{lem ell_A preserves comp}. 
\end{proof}

\subsection{2-cells of Canonical Lax Transformation} 

\phantomsection
\subsubsection*{Classical}
\addcontentsline{toc}{subsubsection}{Classical}

When $\cE = \Set$, for each $\varphi \in \cA(A,B)$ the natural transformation $\ell_\varphi$ is defined with components 

\[ \ell_{\varphi , a} := (\varphi , 1_{D(\varphi)(a)})  \]\

\noi such that for any $f : a \to b$ in $D(A)$, the square 

\begin{center}
\begin{tikzcd}[column sep = large]
a \dar["(1_A{,} \delta_{A,a} f) "'] \rar["(\varphi {,} 1_{D(\varphi)(a)}) "] & D(\varphi) (a) \dar["(1_B{ ,} \delta_{B, D(\varphi)(a)} D(\varphi)(f))"] \\
b \rar["(\varphi {,} 1_{D(\varphi)(b)})"'] & D(\varphi)(b) 
\end{tikzcd}
\end{center}

\noi commutes. This calculation looks like 

\begin{align*}
& \ \ \ \ (1_A,\delta_{A,a} f) (\varphi {,} 1_{D(\varphi)(b)}) \\
&=\left( \varphi {,} \delta_{1_A ; \varphi, a} D(\varphi)(\delta_{A,a} f) 1_{D(\varphi)(b)} \right) &\\
&=\left( \varphi {,} \delta_{1_A ; \varphi, a} D(\varphi)(\delta_{A,a}) D(\varphi)(f) \right) & \text{Functoriality} \\
&=\left( \varphi 1_B {,} \delta_{\varphi ; 1_B , a} \delta_{B,D(\varphi)(a)} D(\varphi)(f) \right) & \text{Coherence} \\
&=\left( \varphi 1_B {,} \delta_{\varphi ; 1_B , a} D(1_B) (1_{D(\varphi)(a)})\delta_{B,D(\varphi)(a)} D(\varphi)(f) \right) & \text{Functoriality} \\
&= (\varphi , 1_{D(\varphi)(a)})(1_B, \delta_{B, D(\varphi)(a)} D(\varphi)(f)) & \text{Def.} 
\end{align*}\

\noi and can all be internalized to an arbitrary extensive category $\cE$. Note that the class of cartesian arrows in the usual Grothendieck contruction are pairs $(\varphi , f) : (A,a) \to (B, b)$ such that $f$ is an isomorphism. The components of $\ell_\varphi$ are a special subclass of these which are actually a set when $\cA$ is small and these are typically called the canonical cleavage of the cartesian arrows. 

\phantomsection
\subsubsection*{Internal} 
\addcontentsline{toc}{subsubsection}{Internal}

Define the internal natural transformation $\ell_\varphi : \ell_A \implies D(\varphi) \ell_B$ as the composite 
 
\begin{center}
\begin{tikzcd}[column sep = large] 
D(A)_0 \rar[rr, "\langle 1_{D(A)_0} {,} D(\varphi)_0 e_B \rangle"] \ar[drr, "\ell_\varphi"'] && D_\varphi \dar[ tail, "\iota_\varphi"]\\
&& \mD_1 
\end{tikzcd}
\end{center}

\noi and we can immediately check 

\begin{align*}
\ell_\varphi s 
&=\langle 1_{D(A)_0} {,} D(\varphi)_0 e_B \rangle \iota_\varphi s \\
&=\langle 1_{D(A)_0} {,} D(\varphi)_0 e_B \rangle \pi_0 \iota_A & \text{Def. } s_\varphi \\
&=1_{D(A)_0} \iota_A \\
&=(\ell_A)_0 
\end{align*}

\noi and

\begin{align*}
\ell_\varphi t
&=\langle 1_{D(A)_0} {,} D(\varphi)_0 e_B \rangle \iota_\varphi t \\
&=\langle 1_{D(A)_0} {,} D(\varphi)_0 e_B \rangle \pi_1 t \iota_B & \text{Def. } t_\varphi \\
&=D(\varphi)_0 e_B t \iota_B \\
&= D(\varphi)_0 1_{D(B)_0} \iota_B] & \text{Def. } e_B\\
&=D(\varphi)_0(\ell_B)_0 &\text{Def. } \ell_B\\
&= \left( D(\varphi) \ell_B \right)_0 &\text{Functoriality}.
\end{align*}\

\noi This shows us that $\ell_\varphi$ is well-defined in terms of its source and target. Now we need to check that it satisfies the naturality square. This is done in the proof of the following proposition as a big calculation that involves manipulating pairing maps of pullbacks. References to a few side calculations appearing as lemmas in Section \ref{SS 2-cells of l} of the appendix are included on the side along with references to definitions, internal category structure laws, functoriality, and coherences. 

\begin{prop}[(Internal) Naturality of $\ell_\varphi$]
For each $\varphi : A \to B$, the map $\ell_\varphi : D(A)_1 \to \mD_1$ defines an internal natural transformation, $\ell_A \implies D(\varphi) \ell_B$ in the sense that the diagram, 
\[\begin{tikzcd}[column sep = huge]
D(A)_1 \rar["\langle s \ell_\varphi {,} D(\varphi)_1 (\ell_B)_1 \rangle"] \dar["\langle (\ell_A)_1 {,} t \ell_\varphi \rangle"'] & \mD_2 \dar["c"] \\
\mD_2 \rar["c"'] & \mD_1
\end{tikzcd}\]
\noi commutes in $\cE$. 

\end{prop}
\begin{proof}
\begin{align*}
& \ \ \ \ \langle (\ell_A)_1 , t \ell_\varphi \rangle c \\
&= \langle (\ell_A)'_1 \iota_{1_A} 
\ , \ 
 t \langle 1_{D(A)_0} , D(\varphi)_0 e_B \rangle \iota_\varphi\rangle c \\
&= \langle (\ell_A)'_1 
\ ,\
 t \langle 1_{D(A)_0} , D(\varphi)_0 e_B \rangle \rangle \iota_{1_A ; \varphi} c & \text{Lemma } \ref{Lem sidecalc 1} \\
&= \langle (\ell_A)'_1 
\ , \
 t \langle 1_{D(A)_0} , D(\varphi)_0 e_B \rangle c_{1_A ; \varphi} \iota_\varphi & \text{Def. } c\\
&=\langle s 
\ , \
\langle s \delta_{1_A ; \varphi} , s \delta_A D(\varphi)_1 , D(\varphi)_1 \rangle c \rangle \iota_\varphi & \text{Lemma } \ref{Lem sidecalc 3} \\
&= \langle s 
\ , \
  \langle s \langle \delta_{1_A ; \varphi} , \delta_A D(\varphi)_1 \rangle c , D(\varphi)_1 \rangle c \rangle \iota_\varphi & \text{Assoc. } \\
&= \langle s 
\ , \
  \langle s \langle \delta_{\varphi ; 1_B} , D(\varphi)_0 \delta_B \rangle c , D(\varphi)_1 \rangle c \rangle \iota_\varphi & \text{Coherence. } \\
&= \langle s 
\ , \
  \langle s \delta_{\varphi ; 1_B} , s D(\varphi)_0 \delta_B , D(\varphi)_1 \rangle c \rangle \iota_\varphi & \text{Assoc.} \\
&= \langle s 
\ , \
  \langle s \delta_{\varphi ; 1_B} , s D(\varphi)_0 \langle D(1_B)_0 , \delta_B \rangle p_1 , D(\varphi)_1 \rangle c \rangle \iota_\varphi & \text{} \\
&= \langle s 
\ , \
  \langle s \delta_{\varphi ; 1_B} , s D(\varphi)_0 \langle D(1_B)_0 , \delta_B \rangle \langle p_0 e_B , p_1 \rangle c , D(\varphi)_1 \rangle c \rangle \iota_\varphi & \text{Id.-Law} \\
&= \langle s 
\ , \
  \langle s \delta_{\varphi ; 1_B} , s D(\varphi)_0 \langle D(1_B)_0 e_B , \delta_B\rangle c , D(\varphi)_1 \rangle c \rangle \iota_\varphi & \text{} \\ 
&= \langle s 
\ , \
  \langle s \delta_{\varphi ; 1_B} , s D(\varphi)_0 \langle e_B D(1_B)_1 , \delta_B\rangle c , D(\varphi)_1 \rangle c \rangle \iota_\varphi & \text{Func'y } D(1_B) \\ 
&= \langle s 
\ , \
  \langle s \delta_{\varphi ; 1_B} , s D(\varphi)_0e_B D(1_B)_1 , \langle s D(\varphi)_0 \delta_B , D(\varphi)_1 \rangle c \rangle c \rangle \iota_\varphi & \text{Assoc. } \\ 
&= \langle s 
\ , \
  \langle s \delta_{\varphi ; 1_B} , s D(\varphi)_0e_B D(1_B)_1 , \langle D(\varphi)_1 s \delta_B , D(\varphi)_1 \rangle c \rangle c \rangle \iota_\varphi & \text{Def. } D(\varphi) \\ 
&= \langle s 
\ , \
  \langle s \delta_{\varphi ; 1_B} , s D(\varphi)_0e_B D(1_B)_1 , D(\varphi)_1 \langle s \delta_B ,  1_{D(B)_1} \rangle c \rangle c \rangle \iota_\varphi & \text{Factor} \\ 
&= \langle s 
\ , \
  \langle s \delta_{\varphi ; 1_B} , s D(\varphi)_0 e_B D(1_B)_1 , D(\varphi)_1 (\ell_B)'_1 \pi_1 \rangle c \rangle \iota_\varphi & \text{Def. } (\ell_B)'_1  \\ 
&= \langle 
  s \langle 1_{D(A)_0} , e_A D(\varphi)_1 \rangle 
  \ , \ 
  \langle s D(\varphi)_0 , D(\varphi)_1 (\ell_B)'_1 \pi_1 \rangle \rangle c_{\varphi ; 1_B} \iota_\varphi & \text{Lemma } \ref{Lem sidecalc 5} \\ 
&= \langle 
  s \langle 1_{D(A)_0} , e_A D(\varphi)_1 \rangle 
  \ , \ 
  \langle D(\varphi)_1 s , D(\varphi)_1 (\ell_B)'_1 \pi_1 \rangle \rangle c_{\varphi ; 1_B} \iota_\varphi & \text{Def. } D(\varphi) \\ 
&= \langle 
  s \langle 1_{D(A)_0} , e_A D(\varphi)_1 \rangle 
  \ , \ 
 D(\varphi)_1 \langle s , (\ell_B)'_1 \pi_1 \rangle \rangle c_{\varphi ; 1_B} \iota_\varphi & \text{Factor} \\ 
&= \langle 
  s \langle 1_{D(A)_0} , e_A D(\varphi)_1 \rangle 
  \ , \ 
 D(\varphi)_1 (\ell_B)'_1 \rangle c_{\varphi ; 1_B} \iota_\varphi & \text{Uniqueness} \\ 
&= \langle 
  s \langle 1_{D(A)_0} , e_A D(\varphi)_1 \rangle 
  \ , \ 
 D(\varphi)_1 (\ell_B)'_1 \rangle \iota_{\varphi ; 1_B} c & \text{Def. } c \\ 
&= \langle 
  s \langle 1_{D(A)_0} , e_A D(\varphi)_1 \rangle \iota_\varphi 
  \ , \ 
 D(\varphi)_1 (\ell_B)'_1 \iota_{1_B} \rangle c & \text{Lemma} \ref{Lem sidecalc 6} \\ 
&= \langle s \ell_\varphi , D(\varphi)_1 (\ell_B)_1 \rangle & \text{Def. } \ell_\varphi , (\ell_B)_1
\end{align*}
\end{proof}

\subsection{Universal Property for 1-cells}\label{SS UP for 1-cells of laxtransfm}

\phantomsection
\subsubsection*{Classical}
\addcontentsline{toc}{subsubsection}{Classical}
Suppose $\cE = \Set$ and $x : D \implies \Delta X$ is a lax natural transformation. That is, there are functors 

\[ x_A : D(A) \to X \]\

\noi for each $A \in \cA_0$ and natural transformations 

\[ x_\varphi : x_A \to D(\varphi) x_B \]\

\noi for each $\varphi : A \to B$ in $\cA$ that are coherent with respect to the pseudofunctor's structure isomorphisms. More concretely, for each $a \in D(A)_0$ and each composable $\varphi, \psi \in \cA$ where $A = \dom(\varphi)$ we have

\[ x_{1_A, a } = x_A ( \delta^{-1}_{A,a}) \quad , \quad x_{\varphi \psi, a} x_C(\delta_{\varphi ; \psi , a}) = x_{\varphi ,a } x_{\psi , D(\varphi)(a)} . \]\

\noi Then we can define a functor $\theta : \mD \to X$ on an object $(A,a)$ in $\mD$ as

\[ \theta \left( (A,a) \right) = x_A \left( (A,a) \right) \]\

\noi and on a morphism $(\varphi , f ) : (A,a) \to (B,b)$ in $\mD$ as 

\[ \theta \left( (\varphi , f ) \right) = x_{\varphi , a} x_B(f). \]\

\noi Identities are preserved by this assignment because the diagram 

\begin{center}
\begin{tikzcd}[]
x_A (a) \ar[rr, bend right=50, "\theta((1_A {,}\delta_{A,a}))"'] \ar[rr, bend left =50, "x_A (\delta_{A,a}^{-1} \delta_{A,a}) = x_A (1_a) = 1_{x_A(a)}"] \rar["x_A (\delta_{A,a}^{-1}) "]
& x_A (D(1_A)(a)) \rar["x_A (\delta_{A,a})"] 
& x_A(a) 
\end{tikzcd}
\end{center}

\noi commutes in $X$ and composition is preserved by the following commuting diagram in $X$.

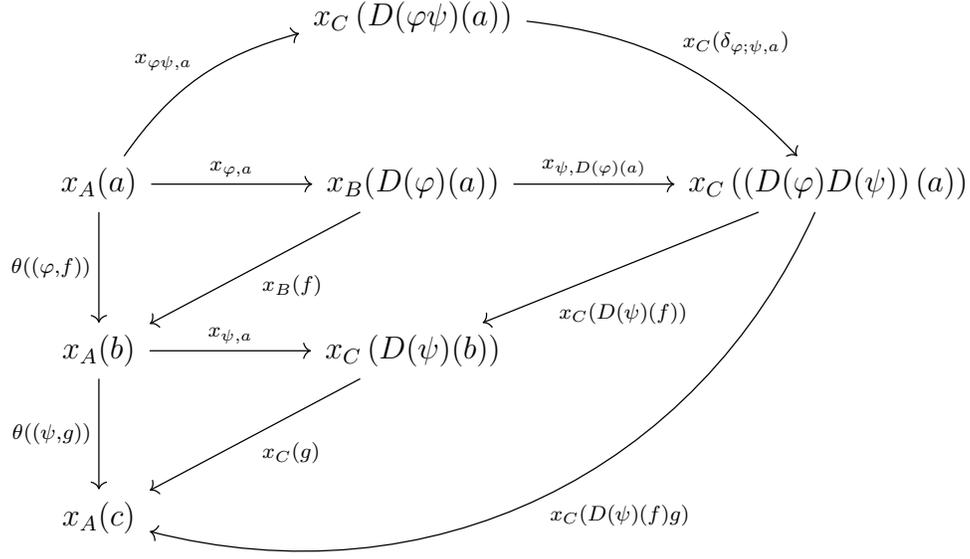
\begin{center}
\begin{tikzcd}[]
&&x_C \left( D(\varphi \psi) (a) \right) \ar[ddrr, bend left = 20,"x_C (\delta_{\varphi ; \psi , a})"]&&\\
&&&&\\
x_A(a) \ar[uurr, bend left=20, "x_{\varphi \psi , a}"] \dar[dd,"\theta( (\varphi {,} f) ) "'] \rar[rr, "x_{\varphi , a}"] && x_B (D(\varphi)(a)) \rar[rr, "x_{\psi , D(\varphi)(a)}"] \ar[ddll, "x_B(f)"]&& x_C\left( \left( D(\varphi)D(\psi) \right) (a) \right) \ar[ddll , near end, "x_C ( D(\psi)(f) )"] \ar[ddddllll, bend left = 40, "x_C(D(\psi)(f) g )"] \\
&&&&\\
x_A(b) \dar[dd,"\theta( (\psi {,} g )) "'] \rar[rr,"x_{\psi , a}"] && x_C \left( D(\psi)(b) \right) \ar[ddll, "x_C(g)"] && \\
&&&&\\
x_A(c) &&& & 
\end{tikzcd}
\end{center}\

\noi The top square commutes by coherence, the triangles on the left commute by definition of $\theta$, the middle square commutes by naturality of $x_\psi$ and by functoriality of $x_C$ we know the bottom right triangle commutes so we can see 

\begin{align*}
\theta((\varphi , f)) \theta((\psi , g))
&=x_{\varphi \psi , a} x_C(\delta_{\varphi ; \psi , a} ) x_C (D(\psi) (f) g) & \text{Def.} \theta \\
&= x_{\varphi \psi , a} x_C(\delta_{\varphi ; \psi , a} D(\psi) (f) g) & \text{Functoriality } x_C \\ 
&= \theta( (\varphi \psi , \delta_{\varphi ; \psi , a} D(\psi) (f) g )) & \text{Def. } \theta \\
&= \theta ( (\varphi ,f) (\psi , g) ). &
\end{align*}

\noi Notice that 
\[ \theta(\ell_A(a)) = \theta( (A,a)) = x_A(a)  \]

\noi and 
\begin{align*}
\theta(\ell_A(f)) 
&= \theta( (1_A , \delta_{A,a} f) ) \\
&= x_{1_A,a} x_A(\delta_{A,a} f) \\
&= x_A(\delta_{A,a}^{-1}) x_A(\delta_{A,a} f) & \text{Coherence} \\
&= x_A(\delta_{A,a}^{-1}\delta_{A,a} f) & \text{Functoriality} \\
&= x_A(f) & \text{Functoriality} \\
\end{align*}\

\noi so $\ell_A \theta = x_A$ for each $A$ in $\cA$. Moreover, for any functor $\omega : \mD \to X$ and any $\varphi : A \to B$ in $\cA$ one can get a natural transformation 
\[ \omega_{\varphi} : \ell_A \omega \implies D(\varphi) \ell_B \omega \]

\noi by whiskering. Componentwise this amounts to defining

\[ \omega_{\varphi , a } := \omega ( \ell_{\varphi,a}). \]

\noi Functoriality of $\omega$ makes $\omega_\varphi$ coherent with respect to composition in $\cA$, 

\begin{align*}
\omega_{\varphi \psi ,a} 
&= \omega (\ell_{\varphi \psi , a} ) \\
&= \omega ( \ell_{\varphi , a} \ell_{\psi , D(\varphi)(a) } \ell_C ( \delta^{-1}_{\varphi ; \psi , a} ) )\\
&= \omega ( \ell_{\varphi , a}) \omega (\ell_{\psi , D(\varphi)(a) } ) \omega (\ell_C ( \delta^{_1}_{\varphi ; \psi , a} ) )\\
&= \omega_{\varphi , a} \omega_{\psi , D(\varphi)(a)} \left( (\ell_C \omega) (\delta_{\varphi ; \psi , a} ) \right)^{-1},
\end{align*}

\noi and with respect to identities in $\cA$ 

\begin{align*}
\omega_{1_A,a} 
= \omega( \ell_{1_A , a}) 
= \omega ( \ell_A (\delta_{A,a}^{-1}) ) 
= (\ell_A \omega) (\delta_{A,a}^{-1}). 
\end{align*}

\noi The assignments above are inverses to one another. On one hand if we start with a family of natural transformations $\{x_\varphi : x_A \implies D(\varphi) x_B \}_{\varphi \in \cA(A,B)}$, consider the induced functor $\theta$ and its induced lax natural transformation. For each $\varphi : A \to B$ in $\cA$, the natural transformation 

\[ \theta_\varphi : \ell_A \omega \implies D(\varphi) \ell_B \omega \]

\noi has precisely the same components as the natural transformation $x_\varphi$ from the family we started with since $\ell_A \omega = x_A$ for each $A $ in $\cA$. 

\[ \theta_{\varphi , a} = \theta (\ell_{\varphi , a} ) = \theta (( \varphi , 1_{D(\varphi)(a)} )) = x_{\varphi , a } x_B (1_{ D(\varphi)(a)}) = x_{\varphi, a} \]\

\noi Now if we start with a functor $\omega : \mD \to X$, get the natural transformations $\omega_\varphi : \ell_A \omega \implies D(\varphi) \ell_B \omega$, and then let $\theta_{\omega} : \mD \to X$ be the induced functor, we have that for each $(A,a) \in D(A)$, 

\begin{align*}
\theta_{\omega } ((A,a))
 :=( \ell_A \omega) ((A,a)) 
 = ( \omega ( \ell_A (a) ) 
 = \omega ((A,a)) 
\end{align*}\

\noi as well as 

\begin{align*}
\theta_{\omega} ((\varphi,f))
&= \omega_{\varphi , a} \left( (\ell_B \omega)(f) \right) \\
&= \omega ( \ell_{ \varphi , a} ) \left( (\ell_B \omega)(f) \right) \\
&= \omega \left( \ell_{\varphi , a} \ell_B (f) \right) \\
&= \omega \left( (\varphi , 1_{D(\varphi)(a)} ) (1_B , \delta_{B,D(\varphi)(a)} f)\right) \\
&= \omega \left( (\varphi 1_B , \delta_{\varphi ; 1_B , a} D(1_B) (1_{D(\varphi)(a)}) \delta_{B,D(\varphi)(a)} f)\right) \\
&= \omega \left( (\varphi , \delta_{\varphi ; 1_B , a} \delta_{B,D(\varphi)(a)} f)\right) \\
&= \omega \left( (\varphi  f)\right) \\
\end{align*}

\noi where the last equality is by coherence. This shows that $\theta_\omega = \omega$ and it follows that the assignments 

\[ \{ x_\varphi : x_A \implies D(\varphi) x_B \}_{\varphi \in \cA(A,B)} \mapsto (\theta : \mD \to X) \]\

\noi and 

\[ (\omega: \mD \to X) \mapsto \{ \omega_\varphi : \ell_A \omega \implies D(\varphi) \ell_B \omega \}_{\varphi \in \cA(A,B)}\]\

\noi are inverses of one another. In particular, every functor $\mD \to X$ corresponds uniquely to a lax natural transformation $D \implies \Delta X$. 



\phantomsection
\subsubsection*{Internal}
\addcontentsline{toc}{subsubsection}{Internal}

Let $\cE$ be an extensive category and let $\mX$ be an arbitary internal category in $\cE$ with $\{ x_A : D(A) \to \mD\}$ an $\cA_0$-indexed family of internal functors and $\{ x_\varphi : x_A \implies D(\varphi) x_B\}$ an $\cA_1$-indexed family of internal natural transformations that satisfy the following (internalized) coherences with respect to the pseudofunctor isomorphisms

\[ \langle x_{1_A} , \delta_A (x_A)_1 \rangle c = e_A (x_A)_1 \qquad , \qquad \langle x_{\varphi \psi} , \delta_{\varphi ; \psi} x_C \rangle c = \langle x_\varphi , D(\varphi)_0 x_\psi \rangle c .\]\

\noi Define $\theta_0$ to be uniquely induced by the maps $\{ x_A \}_{A \in \cA_0}$ 

\begin{center}
\begin{tikzcd}[]
\mD_0 \rar[rr, dotted, "\theta_0"] && \mX_0 \\
&D(A)_0 \ar[ul, "\iota_A"] \ar[ur, "(x_A)_0"] &
\end{tikzcd}
\end{center}\

\noi and define $\theta_1$ to be uniquely induced by the family of composites

\begin{center}
\begin{tikzcd}
\mD_1 \rar[rr,dotted, "\theta_1"] && \mX_1 \\
&&\\
D_\varphi \uar[uu,"\iota_\varphi"] \rar[rr, "\langle \pi_0 x_\varphi {,} \pi_1 (x_B)_1 \rangle "'] && \mX_2 \uar[uu, "c"'] 
\end{tikzcd}.
\end{center}

\noi on each component. The following lemma shows $\theta = (\theta_0, \theta_1)$ preserves identities and will be used later to conclude $\theta$ is an internal functor. 

\begin{lem}\label{Lem universal 1-cells preserves identities} 
The assignment $ (\theta_0 , \theta_1) : \mD \to \mX$ preserves identities. 
\end{lem} 
\begin{proof}
By the universal property of the coproduct, $\mD_1$, it suffices to see the following diagram commutes. 

\begin{center}
\begin{tikzcd}[row sep = huge, column sep = huge]
\mD_0 \rar[dotted, "\theta_0"] \rar[rrrr, bend left=20, dotted, ""] & \mX_0 \rar[rr, bend left= 15, equals] \rar["e"] & \mX_1 \rar["s"] & \mX_0 \rar["e"] & \mX_1 \\
D(A)_0  \ar[rr, bend right = 10, equals] \uar["\iota_A"] \ar[ur, "(x_A)_0"] \ar[r, "e_A"] \dar["\iota_A"'] &D(A)_1 \rar["s"]  \ar[ur, "(x_A)_1"] & D(A)_0 \ar[ur, "(x_A)_0"] \rar[" \langle {1_{D(A)_0} , \delta_A} \rangle"']\ar[rr, bend left = 10 ,"\langle x_{1_A} {,} \delta_A (x_A)_1 \rangle "] & D_{1_A} \ar[dl, "\iota_{1_A}"] \rar["\langle \pi_0 x_{1_A} {,} \pi_1 (x_A)_1 \rangle "'] & \mX_2 \ar[d, "c"] \uar["c"] \\
\mD_0 \rar[rr, dotted, "\epsilon"] \rar[rrrr, bend right = 20, dotted, ""] && \mD_1 \rar[rr, "\theta_1"] && \mX_1
\end{tikzcd}
\end{center}\

\noi The top left two squares commute by definition of $\theta_0$ and $e$. They show that the composite $\theta_0 e$ is uniquely induced by the family of maps $\{ e_A (x_A)_1 \}_{A \in \cA_0}$. The other two squares on the top commute by functoriality of $x_A$ and coherence. In the middle on the left we have the source-identity coherence in $D(A)$, and a short calculation on the right using the universal property of $\mX_2$ shows that 

\[ \langle x_{1_A} , \delta_A (x_A)_1 \rangle = \langle 1_{D(A)_0} , \delta_A \rangle \langle \pi_0 x_{1_A} , \pi_1 (x_A)_1 \rangle .\]

\noi Recall that $\epsilon_A := \langle 1_D(A)_0 , \delta_A \rangle$ to see that the bottom left square commutes by definition of $\epsilon$ and the bottom right square commutes by definition of $\theta_1$. Together they show that $\epsilon \theta_1$ is uniquely induced by $\{ \epsilon_A \langle \pi_0 \iota_{1_A} , \pi_A (x_A)_1 \rangle c\}_{A \in \cA_0}$. Commutativity above shows 

\[ \epsilon_A \langle \pi_0 \iota_{1_A} , \pi_A (x_A)_1 \rangle c = e_A s \langle x_{1_A} , \delta_A (x_A)_1 \rangle c = e_A (x_A)_1 \]\

\noi and therefore 

\[ \epsilon \theta_1 = \theta_0 e\]

\noi by uniqueness. 
\end{proof} 

\begin{lem}\label{Lem universal 1-cells preserves composition}
The assignment $(\theta_0 , \theta_1) : \mD \to \mX$ preserves composition. 
\end{lem}
\begin{proof}
By definition of composition in $\mD$ and the map $\theta_1$, the composite $c \theta_1 : \mD_2 \to \mX_2$ is uniquely induced by the family of maps $ c_{\varphi ; \psi} \langle \pi_2 x_{\varphi \psi} , \pi_1 (x_C)_1 \rangle c$, where $\varphi : A \to B$ and $\psi: B \to C$ are composable morphisms in $\cA$.

\noi By the universal property of the coproduct $\mD$ it suffices to show that 

\[c_{\varphi ; \psi} \langle \pi_2 x_{\varphi \psi} , \pi_1 (x_C)_1 \rangle c = \langle p_0 \langle \pi_0 x_{\varphi} , \pi_1 (x_B)_1 \rangle \ , \ p_1 \langle \pi_0 x_\psi, \pi_1 (x_C)_1 \rangle \rangle c \]\

\noi so that the middle pentagon (disguised as a triangle) in the diagram 

\begin{center}
\begin{tikzcd}[column sep = huge, row sep = huge]
\mD_2 \rar[rr,rr, dotted, bend left=20] \rar[rr,dotted, "c"] && \mD_1 \rar[rr,dotted, "\theta_1"] && \mX_1 \\
D_{\varphi ; \psi } \ar[drr, near end, "\langle p_0 \langle \pi_0 x_{\varphi} {,} \pi_1 (x_B)_1 \rangle \ {,} \ p_1 \langle \pi_0 x_\psi {,} \pi_1 (x_C)_1 \rangle \rangle "] 
\rar[rr,"c_{\varphi ; \psi}"'] 
\uar[tail, "\iota_{\varphi ; \psi}"] 
\dar[dd, tail, "\iota_{\varphi ; \psi}"'] &&
 D_{\varphi \psi} \uar["\iota_{\varphi \psi}"'] 
\rar[rr, "\langle \pi_0 x_{\varphi \psi} {,} \pi_1 (x_C)_1 \rangle"'] & & \mX_2 \uar["c"'] \dar[dd,"c"] \\
& & \mX_2 \ar[drr, "c"] \dar[equals] && \\
\mD_2 \rar[rr, dotted, "\langle \rho_0 \theta_1 {,} \rho_1 \theta_1 \rangle "'] && \mX_2 \rar[rr,"c"']& & \mX_1 
\end{tikzcd}
\end{center}\

\noi commutes. The remaining squares and triangle in the diagram above commute by definition or by the identity law for composition in $\cE$. Let $\delta_{\varphi ; \psi}^{(-1)} : D(A)_0 \to D(C)_1$ denote the `family of inverse coherence isomorphisms' associated to $\varphi , \psi$. In particular, 

\[ \langle \delta_{\varphi ; \psi}^{(-1)} , \delta_{\varphi ; \psi} \rangle c = D(\varphi)_0 D(\psi)_0 e_C \]

\begin{align*}
& \ \ \ c_{\varphi ; \psi} \langle \pi_0 x_{\varphi ; \psi} , \pi_1 (x_C)_1 \rangle c \\
&= \langle p_0 \pi_0 x_\varphi ,
\ p_0 \pi_0 D(\varphi)_0 x_\psi , \ 
p_0 \pi_0 \delta_{\varphi ; \psi}^{(-1)} (x_C)_1 , \
p_0 \pi_0 \delta_{\varphi ; \psi} (x_C)_1 , & \\
& \ \ \ \ 
p_0 \pi_1 D(\psi)_1 (x_C)_1 , \ 
p_1 \pi_1 (x_C)_1 \rangle c \\
&= \langle p_0 \pi_0 x_\varphi ,
\ p_0 \pi_0 D(\varphi)_0 x_\psi ,
\ p_0 \pi_0D(\varphi)_0 D(\psi)_0 (x_C)_0 e , &\\
& \ \ \ \ 
p_0 \pi_1 D(\psi)_1 (x_C)_1 , \
p_1 \pi_1 (x_C)_1 \rangle c \\
&= \langle p_0 \pi_0 x_\varphi , \ 
p_0 \pi_0 D(\varphi)_0 x_\psi , \ 
p_0 \pi_1 D(\psi)_1 (x_C)_1 , \ 
p_1 \pi_1 (x_C)_1 \rangle c \\
&= \langle p_0 \pi_0 x_\varphi , \
p_0 \pi_1 s x_\psi \ , \ p_0 \pi_1 D(\psi)_1 (x_C)_1 , \
p_1 \pi_1 (x_C)_1 \rangle c & \text{Id. Law} \\
&= \langle p_0 \pi_0 x_\varphi , \ 
p_0 \pi_1 (x_B)_1 \ , \ p_0 \pi_1 t x_\psi , \ 
p_1 \pi_1 (x_C)_1 \rangle c & \text{Naturality } x_\psi \\
&= \langle p_0 \pi_0 x_\varphi \ , \ p_0 \pi_1 (x_B)_1 , \
p_1 \pi_0 x_\psi , \ 
p_1 \pi_1 (x_C)_1 \rangle c & \text{Def. } D_{\varphi ; \psi}\\
&= \langle \langle p_0 \langle \pi_0 x_\varphi , \ 
\pi_1 (x_B)_1\rangle c ,
\ p_1 \langle \pi_0 x_\psi , \
\pi_1 (x_C)_1\rangle c \rangle c & \text{Assoc. \& Factor}
\end{align*}

\end{proof}

\begin{prop}
The assignment $\theta = (\theta_0 , \theta_1 ) : \mD \to \mX$ is an internal functor.
\end{prop}
\begin{proof}
Immediate from Lemmas \ref{Lem universal 1-cells preserves identities} and \ref{Lem universal 1-cells preserves composition}.
\end{proof}\

On the other hand, given an internal functor $\omega : \mD \to \mX$, for each $\varphi : A \to B$ in $\cA$ define 

\[ \omega_A = \ell_A \omega \quad \text{ and }\quad \omega_\varphi := \ell_\varphi \omega_1. \]

\noi Notice that $\omega_A$ is an internal functor (by definition of internal functor composition) and we have that the source of $\omega_\varphi$ is $\omega_A$

\begin{align*}
\omega_\varphi s = \ell_\varphi \omega_1 = \ell_\varphi s \omega_0 = (\ell_A)_0 \omega_0 = (\ell_A \omega)_0 = (\omega_A)_0,
\end{align*}

\noi its target is $D(\varphi) \omega_B$ 

\begin{align*}
\omega_\varphi t = \ell_\varphi \omega_1 t = \ell_\varphi t \omega_0 = (D(\varphi)\ell_B)_0 \omega_0 = D(\varphi) (\ell_B \omega)_0 = D(\varphi) (\omega_B)_0 ,
\end{align*}\

\noi and the (family of) naturality square(s) 

\begin{align*}
\langle s \omega_\varphi , (D(\varphi)\omega_B)_1 \rangle c 
&=\langle s \ell_\varphi \omega_1, (D(\varphi)\ell_B \omega)_1 \rangle c \\
&= \langle s \ell_\varphi \omega_1, (D(\varphi)\ell_B)_1 \omega_1 \rangle c \\
&=\langle s \ell_\varphi, (D(\varphi)\ell_B)_1 \rangle c \omega_1 & \text{Functoriality } \omega\\
&= \langle (\ell_A)_1 , t \ell_\varphi \rangle c \omega_1 \\
&= \langle (\ell_A)_1 \omega_1 , t \ell_\varphi \omega_1 \rangle c \\
&= \langle \omega_A , t \omega_\varphi \rangle c 
\end{align*}

\noi commutes so $\omega_\varphi : \omega_A \implies D(\varphi) \omega_B$ is an internal natural transformation. Putting these families of $\cA_0$-indexed internal functors and $\cA_1$-indexed internal natural transformations together precisely defines an internal lax transformation 

\[ \omega^* : D \implies \Delta \mX \]

\noi where $\Delta \mX : \cA \to \Cat(\cE)$ is the constant functor on $\mX$. 

\begin{prop}\label{Prop 1-cell correspondence}
For each $\mX \in \Cat(\cE)$, lax transformations $D \implies \Delta \mX$ correspond uniquely to internal functors $\mD \to \mX$. 
\end{prop}
\begin{proof}
Suppose $x : D \implies \Delta \mX$ is a lax transformation. Then for each object $A$ in $\cA$ there exists an internal functor $x_A : D(A) \to \mX$ and for each morphism $\varphi : A \to B$ in $ \cA$ there exists an internal natural transformation $x_\varphi : x_A \implies D(\varphi) x_B$ that is coherent with respect to the composition and identity isomorphisms of $D$. Let $\theta : \mD \to \mX$ be the internal functor constructed above, and then consider the induced lax transformation $\theta^*$. As seen above we have that

\[ \theta_A = \ell_A \theta = x_A \] 

\noi for each $A \in \cA_0$ and for each $\varphi \in \cA_1$ we have 

\[ \theta_\varphi := \ell_\varphi \theta_1 = x_\varphi \]

\noi so $\theta^* = x$. \\
On the other hand, let $\omega^* : D \implies \Delta \mX$ denote the lax transformation constructed as above from an arbitrary internal functor $\omega : \mD \to \mX$. Let $\theta^*$ be the induced internal functor as above, then $\theta^*_0$ is uniquely induced by the $\cA_0$-indexed family of internal functors of $\omega^*$. These are precisely the maps $(\omega_A)_0$ and so $\theta^*_0 = \omega_0$ by uniqueness. Using functoriality of $\omega$ and the definition of $\omega^*$ we can see that $\theta^*_1$ is uniquely induced by the family of maps,$\iota_\varphi \omega_1$. We break the calculation up with terms on separate lines due to their length. We start with the functoriality of $\omega$ giving us

\[ \langle \pi_0 \ell_\varphi \omega_1 , \pi_1 (\ell_B \omega)_1 \rangle c 
= \langle \pi_0 \ell_\varphi , \
\pi_1 (\ell_B)_1 \rangle c \omega_1 \]

\noi and then by definition of $\ell_{\varphi, \ell_B}$ the right-hand side is equal to 
\[ \langle \pi_0 \langle 1_{D(A)_0} , \ 
D(\varphi)_0 e_B \rangle \iota_\varphi , 
\pi_1 \langle s , \ 
\langle s \delta_B , \ 
1_{D(B)_1} \rangle c \rangle \iota_B \rangle 
c \omega_1 .\]

\noi The definition of $\iota_{\varphi ; 1_B}$ says the last term is equal to 
\[ \langle \pi_0 \langle 1_{D(A)_0} , \ 
D(\varphi)_0 e_B \rangle , 
\pi_1 \langle s , \ 
\langle s \delta_B , \ 
1_{D(B)_1} \rangle c \rangle \rangle 
\iota_{\varphi ; 1_B} c \omega_1 \]

\noi which is equal to 
\[ \langle \pi_0 \langle 1_{D(A)_0} , \ 
D(\varphi)_0 e_B \rangle ,
\pi_1 \langle s , \ 
\langle s \delta_B , \ 
1_{D(B)_1} \rangle c \rangle \rangle c_{\varphi ; 1_B} \iota_\varphi \omega_1 \]

\noi by definition of $c_{\varphi ; 1_B} $. The same definition gives that this is equal to

\[ \langle \pi_0 , 
\langle \pi_0 \delta_{\varphi ; 1_B} , \ \pi_0 D(\varphi)_0 D(1_B)_0 e_B , 
\pi_1 \langle s \delta_B, 1_{D(B)_1} \rangle c \rangle c \rangle \iota_\varphi \omega_1 \]

\noi which is equal to 

\[\langle \pi_0 ,
\langle \pi_0 \delta_{\varphi ; 1_B} , 
\pi_1 \langle s \delta_B, 1_{D(B)_1} \rangle c \rangle c \rangle \iota_\varphi \omega_1 \]

\noi by the identity law in $D(B)$. Associativity of internal composition gives that this is equal to

\[ \langle \pi_0 , 
\ \langle \langle \pi_0 \delta_{\varphi ; 1_B} , \pi_1 s \delta_B\rangle c , \pi_1 \rangle c \rangle \iota_\varphi \omega_1 \]

\noi which becomes 
\[ \langle \pi_0 \ , \ \langle \langle \pi_0 \delta_{\varphi ; 1_B} , \pi_0 D(\varphi)_0 \delta_B \rangle c , \pi_1 \rangle c \rangle \iota_\varphi \omega_1 \]

\noi by definition of $D_\varphi$. Factoring pairing maps makes the last term equal to 

\[ \langle \pi_0 \ , \ \langle \pi_0 \langle \delta_{\varphi ; 1_B} , D(\varphi)_0 \delta_B \rangle c , \pi_1 \rangle c \rangle \iota_\varphi \omega_1 \]

\noi which, by coherence of the structure isomorphisms for the pseudofunctor $D$, is equal to 

\[ \langle \pi_0 \ , \ \langle \pi_0 D(\varphi)_0 e_B , \pi_1 \rangle c \rangle \iota_\varphi \omega_1 .\]

\noi The definition of $D_\varphi$ and the identity law in $D(B)$ allows us to see the term above is really the left-hand side of the final equality:
\[ \langle \pi_0 , \ \pi_1 \rangle \iota_\varphi \omega_1 = \iota_\varphi \omega_1\]

\noi By the universal property of the coproduct $\mD_1$ we have
\[ \theta^*_1 = \omega_1 \]
\noi and it follows that $\theta^* = \omega$. 
\end{proof}

\subsection{Universal Property for 2-cells}\label{SS UP of 2-cells for laxtransfm}

\subsubsection*{Classical}
\phantomsection
\addcontentsline{toc}{subsubsection}{Classical}
When $\cE = \Set$ and let $X$ be a small category. Then any natural transformation $\alpha : \theta \implies \omega$ where $\theta, \omega : \mD \to X$ induces a modification 

\[ \tilde{\alpha} : x \Rrightarrow y \]

\noi where $x,y : D \implies \Delta \mX$ are the lax natural transformations corresponding uniquely by Proposition~\ref{Prop 1-cell correspondence} to $\theta$ and $\omega$ respectively. For each $A \in \cA_0$ and $a \in D(A)_0$ we have 

\[ \tilde{\alpha}_{A,a} : x_A(a) \to y_A(a) \]
\noi defined as the component 
\[ \alpha_{\ell_A(a)} : \theta (\ell_A(a)) \to \omega(\ell_A(a)). \]\

\noi For any $g: a \to a'$ in $D(A)_0$, the diagram

\begin{center}
\begin{tikzcd}[]
x_A(a) \rar[equals] \ar[rrr, bend left=20, "\tilde{\alpha}_{A,a}"] \dar["x_A(g)"']& \theta(\ell_A(a)) \dar["\theta(\ell_A(g))"'] \rar["\alpha_{\ell_A(a)}"] & \omega( \ell_A(a))\dar["\omega(\ell_A(g))"] \rar[equals] & y_A(a) \dar["y_A(g)"] \\
x_A(a') \ar[rrr, bend right=20, "\tilde{\alpha}_{A,a'}"'] \rar[equals]& \theta(\ell_A(a')) \rar["\alpha_{\ell_A(a')}"'] & \omega(\ell_A(a')) \rar[equals] & y_A(a') 
\end{tikzcd}
\end{center}\

\noi commutes by definition and naturality of $\alpha$. Similarly, for any $\varphi : A \to B$ in $\cA$ and any $a \in D(A)_0$ we have that the diagram

\begin{center}
\begin{tikzcd}[column sep = large, row sep = large]
x_A(a) \rar[equals] \ar[rrr, bend left=20, "\tilde{\alpha}_{A,a}"] \dar["x_{\varphi, a}"']& (\ell_A \theta)(a) \dar["\theta (\ell_{\varphi , a})"'] \rar["\alpha_{\ell_A(a)}"] & ( \ell_A \omega)(a)\dar["\omega(\ell_{\varphi , a})"] \rar[equals] & y_A(a) \dar["y_{\varphi , a}"] \\
 D(\varphi) x_B (a) \ar[rrr, bend right=20, "\tilde{\alpha}_{A,a}"'] \rar[equals]& D(\varphi) \ell_B \theta(a) \rar["\alpha_{\left(D(\varphi)\ell_B \right)(a)}"'] & D(\varphi)\ell_B \omega(a) \rar[equals] &  D(\varphi)y_B (a) 
\end{tikzcd}
\end{center}\

\noi commutes. It follows that $\tilde{\alpha}$ is a modification $x \Rrightarrow y$. 


Alternatively, given a modification $\gamma: x \Rrightarrow y$ between two lax natural transformations $x , y : D \implies \Delta \mX$, let $\theta ,\omega : \mD \to X$ be the functors uniquely determined by $x$ and $y$ respectively. Then the middle two squares in the following diagram commute by definition of $\gamma$ 

\begin{center}
\begin{tikzcd}[]
\theta((A,a)) \dar[dd,"\theta((\varphi {,} f))"'] \rar[equals]
&x_A(a) \dar["x_{\varphi , a} "'] \rar["\gamma_{A,a}"]
& y_A(a) \dar["y_{\varphi , a}"]
& \omega((A,a))\dar[dd, "\omega((\varphi {,} f) )"] \lar [equals] \\
&x_B (D(\varphi)(a)) \dar["x_B(f)"'] \rar["\gamma_{B, D(\varphi)(a)}"] 
& 
y_B(D(\varphi)(a)) \dar["y_B(f)"] 
& \\
\theta((B,b)) \rar[equals] 
&x_B(b) 
 \rar["\gamma_{B,b}"']
& y_B(b)
&\omega((B,b)) 
\lar [equals]  \\
\end{tikzcd}
\end{center} 

\noi and the left and right squares commute by definition of $\theta$ and $\omega$ respectively. This means 

\[ \overline{\gamma} := \{ \gamma_{A,a} : (A,a) \in \mD \} \]\

\noi is a natural transformation $\theta \implies \omega$. Notice that the lax natural transformation $\overline{\tilde{\alpha}}$ has components 

\[ \overline{\tilde{\alpha}}_{(A,a)} = \tilde{\alpha}_{A,a} = \alpha_{\ell_A(a)} = \alpha_{(A,a)} \]\

\noi so that $\overline{\tilde{\alpha}} = \alpha$. The modification $\tilde{\overline{\gamma}}$ has components 

\[ \tilde{\overline{\gamma}}_{A,a} = \overline{\gamma}_{\ell_A(a)} = \overline{\gamma}_{(A,a)} = \gamma_{A,a} \]

\noi and so $\tilde{\overline{\gamma}} = \gamma$ by definition. The bijection follows and by uniqueness it suffices to see functoriality in one direction. Let $\gamma : x \Rrightarrow y$ and let $\eta : y \Rrightarrow z$ be modifications, then their composite $\gamma \eta$ has components

\begin{center}
\begin{tikzcd}[]
x_A(a)\ar[dr, "(\gamma \eta)_{A,a}"'] \rar["\gamma_{A,a}"] & y_A(a) \dar["\eta_{A,a}"] \\
& z_A(a) 
\end{tikzcd}
\end{center}

\noi and so 

\[ \overline{\gamma \eta} := \{ (\gamma \eta)_{A,a} : (A,a) \in \mD \} = \{ \gamma_{A,a} \eta_{A,a} : (A,a) \in \mD \} =: (\overline{\gamma})( \overline{\eta}) \]\

\noi where the right-hand side is the composite of lax natural transformations. 

\[ \begin{tikzcd}
\theta \rar[Rightarrow, "\overline{\gamma}"] & \omega \rar[Rightarrow, "\overline{\eta}"] & \sigma
\end{tikzcd} \]

\noi and $\theta, \omega, \sigma : \mD \to \mX$ are the internal functors uniquely determined by $x, y, z$ respectively. It follows that composition of modifications is preserved by the bijection above and identities are trivially preserved because 

\[ (1_x)_{A,a} = 1_{x_A(a)} \]

\noi for each $A \in \cA_0$ and each $a \in D(A)_0$ by definition of the identity modification $1_x : x \Rrightarrow x$. 

\phantomsection
\subsubsection*{Internal}
\addcontentsline{toc}{subsubsection}{Internal}

Let $\cE$ be a category that admits an internal category of elements of $D : \cA \to \Cat(\cE)$ and let $\mX$ be an arbitrary internal category in $\cE$. Let $\alpha : \theta \implies \omega$ be an internal natural transformation where $\theta, \omega : \mD \to \mX$ are internal functors and further let $x , y : D \implies \Delta \mX$ denote the unique lax natural transformations induced by $\theta$ and $\omega$ respectively. For each $A$ in $\cA_0$ define 

\begin{center}
\begin{tikzcd}[]
D(A)_0 \rar["(\ell_A)_0"]\ar[dr, "\tilde{\alpha}_A"'] & \mD_0 \dar["\alpha"] \\
& \mX_1 
\end{tikzcd}.
\end{center}

\begin{prop}\label{prop internal internal nat transfms induce a modifiction}
The $\cA_0$-indexed family of maps $ \tilde{\alpha}_A : D(A)_0 \to \mX_1$ defines an internal modification 

\[ \tilde{\alpha} : x \Rrightarrow y.\]
\end{prop}
\begin{proof}
First notice that 
\[ \tilde{\alpha}_A s = (\ell_A)_0 \alpha s = (\ell_A)_0 \theta_0 = (\ell_A \theta)_0 = (x_A)_0 \]

\noi and 

\[ \tilde{\alpha}_A t = (\ell_A)_0 \alpha t= (\ell_A)_0 \omega_0 = (\ell_A \omega)_0 = (y_A)_0. \]

\noi Then

\[ (s \tilde{\alpha}_A) t = s (y_A)_0= (y_A)_1 s \]\

\noi and 

\[ (x_A)_1 t = t (x_A)_0 = t \tilde{\alpha}_A s \]\

\noi so there are two composable pairs given by the maps 

\[ \langle s \tilde{\alpha} , (y_A)_1 \rangle , \langle (x_A)_1 , t \tilde{\alpha}_A \rangle : D(A)_1 \to \mX_2\]

\noi which coincide after composition in $\mX$. 

\begin{align*}
\langle (x_A)_1 , t \tilde{\alpha}_A \rangle c 
&=\langle (\ell_A)_1 \theta_1 , t (\ell_A)_0 \alpha \rangle c & \text{Def. } \tilde{\alpha} , x_A\\
&=\langle (\ell_A)_1 \theta_1 , (\ell_A)_1 t \alpha \rangle c & \text{Functoriality} \ell_A\\
&= (\ell_A)_1 \langle\theta_1 , t \alpha \rangle c & \text{Factor}\\
&= (\ell_A)_1 \langle s \alpha , \omega_1 \rangle c & \text{Naturality } \alpha\\
&= \langle (\ell_A)_1 s \alpha , (\ell_A)_1\omega_1 \rangle c & \text{Factor } \\
&= \langle s (\ell_A)_0 \alpha , (\ell_A)_1\omega_1 \rangle c & \text{Functoriality } \\
&=\langle s \tilde{\alpha}_A , (y_A)_1 \rangle c & \text{Def. } \tilde{\alpha} , y_A
\end{align*}\

\noi This shows $\tilde{\alpha}_A : x_A \implies y_A$ is an internal natural transformation for each $A $ in $ \cA_0$. Now since $\ell_\varphi s = (\ell_A)_0$ and $\ell_\varphi t = (D(\varphi) \ell_B)_0$ we have that

\[ \tilde{\alpha} t = (\ell_A \omega)_0 = y_\varphi s \quad , \quad D(\varphi)_0 \tilde{\alpha} s = ( D(\varphi) \ell_A \theta)_0 = x_\varphi t\]\

\noi by definitions of the induced internal natural transformations $x , y : D \implies \Delta \mX$. These give us the other two composable pairs which are equal in $\mX$ after composition. 

\begin{align*}
\langle \tilde{\alpha}_A , y_\varphi \rangle c 
&= \langle (\ell_A)_0 \alpha , \ell_\varphi \omega_1 \rangle c &\text{Def.} \\
&= \langle \ell_\varphi s \alpha , \ell_\varphi \omega_1 \rangle c &\text{Def.} \ell_\varphi \\
&= \ell_\varphi \langle s \alpha , \omega_1 \rangle c &\text{Factor} \\
&= \ell_\varphi \langle \theta_1, t \alpha  \rangle c &\text{Naturality } \alpha \\
&= \langle \ell_\varphi \theta_1,  \ell_\varphi t \alpha  \rangle c &\text{Factor } \\
&= \langle \ell_\varphi \theta_1,  D(\varphi)_0 (\ell_B)_0 \alpha  \rangle c &\text{Def. } \ell_\varphi \\
&= \langle x_\varphi ,  D(\varphi)_0 \tilde{\alpha}_B  \rangle c &\text{Def. } 
\end{align*}\

\noi This last equality shows that the indexing of $\tilde{\alpha}$ is naturally compatible with the components of the natural transformations $x_\varphi$ and $y_\varphi$, for each $\varphi : A \to B$ in $\cA$. It follows that $\tilde{\alpha} : x \Rrightarrow y$ is a modification. 
\end{proof}\

\noi On the other hand, suppose $\gamma : x \Rrightarrow y$ is a modification between two lax natural transformations $x,y : D \implies \Delta \mX$. Let $\theta , \omega : \mD \to \mX$ be the unique internal functors corresponding to $x$ and $y$ respectively. Let $\overline{\gamma}$ be the map uniquely induced by the natural transformations of $\gamma$ indexed by $\cA_0$. 

\begin{center}
\begin{tikzcd}[]
\mD_0 \rar[dotted, "\overline{\gamma}" ] & \mX_1 \\
D(A)_0 \uar[tail, "\iota_A"] \ar[ur, "\gamma_A"'] & \ 
\end{tikzcd}
\end{center}

\begin{prop}\label{prop modifications induce internal nat transfms}
The map $\overline{\gamma}$ is a natural transformation $\theta \implies \omega$. 
\end{prop}
\begin{proof}
For each $A$ in $\cA_0$, the following diagrams commute by definition of $\theta$ and $\gamma$. 

\begin{center}
\begin{tikzcd}[]
\mD_0 \rar[rr, bend left, dotted, "\theta_0"] \rar[dotted, "\overline{\gamma}"] & \mX_1 \rar["s"] & \mX_0\\
D(A)_0 \uar[tail, "\iota_A"] \ar[ur, "\gamma_A"'] \ar[urr, bend right, "(x_A)_0"']
\end{tikzcd}\qquad \qquad \qquad 
\begin{tikzcd}[]
\mD_0 \rar[rr, bend left, dotted, "\omega"] \rar[dotted, "\overline{\gamma}"] & \mX_1 \rar["t"] & \mX_0\\
D(A)_0 \uar[tail, "\iota_A"] \ar[ur, "\gamma_A"'] \ar[urr, bend right, "(y_A)_0"']
\end{tikzcd}
\end{center}

\noi and so by the universal property of the coproduct $\mD_0$ we have that 

\[ \theta_0 = \overline{\gamma} s \qquad , \qquad \omega_0 = \overline{\gamma} t. \]

\noi Naturality is all that remains to show and this is done using the universal property of the coproduct $\mD_1$. For each $\varphi : A \to B$ in $\cA_1$ we can see 

\[ \iota_\varphi t \overline{\gamma} = \pi_1 t \iota_B \overline{\gamma} = \pi_1 t \gamma_B \]

\noi and 

\[ \iota_\varphi s \overline{\gamma} = \pi_0 \iota_A \overline{\gamma} = \pi_0 \gamma_A\]

\noi and therefore 

\begin{align*}
\iota_\varphi \langle \theta_1 , \ t \overline{\gamma} \rangle c 
&= \langle \iota_\varphi \theta_1 ,\ \iota_\varphi t \overline{\gamma} \rangle c \\
&= \langle \langle \pi_0 x_\varphi , \pi_1 (x_B)_1 \rangle c , \ \pi_1 t \gamma_B \rangle c \\
&= \langle \pi_0 x_\varphi , \ \pi_1 \langle (x_B)_1 ,  t \gamma_B \rangle c \rangle c & \text{Assoc. } \\
&= \langle \pi_0 x_\varphi , \ \pi_1 \langle s \gamma_B , (y_B)_1 \rangle c \rangle c & \text{Nat. } \gamma_B \\
&= \langle \langle \pi_0 x_\varphi ,  \pi_1 s \gamma_B \rangle c , \ \pi_1 (y_B)_1 \rangle c & \text{Assoc. } \\
&= \langle \langle \pi_0 x_\varphi ,  \pi_0 D(\varphi)_0 \gamma_B \rangle c , \ \pi_1 (y_B)_1 \rangle c & \text{Def. } D_\varphi \\
&= \langle \pi_0 \langle x_\varphi , D(\varphi)_0 \gamma_B \rangle c , \ \pi_1 (y_B)_1 \rangle c & \text{Factor.} \\
&= \langle \pi_0 \langle \gamma_A , y_\varphi \rangle c , \ \pi_1 (y_B)_1 \rangle c & \text{Def.} \gamma \\
&= \langle \pi_0 \gamma_A , \ \langle \pi_0 y_\varphi  , \pi_1 (y_B)_1 \rangle c \rangle c & \text{Assoc.} \\
&= \langle \iota_\varphi s \overline{\gamma} , \ \iota_\varphi \omega \rangle c &\text{Def.}\\ 
&= \iota_\varphi \langle s \overline{\gamma} , \omega \rangle c & \text{Factor}.
\end{align*}

\noi By uniqueness we must have that 

\[ \langle \theta_1 \ , \ t \overline{\gamma} \rangle c = \langle s \overline{\gamma} , \omega \rangle c \]

\noi and it follows that $\overline{\gamma} : \theta \implies \omega$ is an internal natural transformation. 
\end{proof}

\begin{prop}\label{Prop 2-cell correspondence}
There is a one-to-one correspondence between modifications of lax natural transformations $D \implies \Delta \mX$ and internal natural transformations between the corresponding internal functors of Proposition~\ref{Prop 1-cell correspondence}. 

\end{prop}
\begin{proof}
The assignments $\overline{( - )}$ and $\tilde{( - )}$ are inverses. For any modification $\gamma : x \implies y$ we have that for each $A \in \cA_0$, 

\[ \tilde{\overline{\gamma}}_A := (\ell_A)_0 \overline{\gamma} = \iota_A \overline{\gamma} = \gamma_A \] 

\noi and so $\tilde{\overline{\gamma}} = \gamma$ by definition. On the other hand for any internal natural transformation $alpha : \theta \implies \omega$, and any $A \in \cA$ we have that 

\[ \iota_A \overline{\tilde{\alpha}} = \tilde{\alpha}_A = (\ell_A)_0 \alpha = \iota_A \alpha \]

\noi and by the universal property of the coproduct $\mD_0$, $\overline{\tilde{\alpha}} = \alpha$. 
\end{proof}

\subsection{Internal Category of Elements as an OpLax Colimit}\label{SS IntCatofEls is OpLaxColim}

In the previous two subsections we've seen that internal functors $\mD \to \mX$ and internal natural transformations between them correspond uniquely to lax natural transformations $D \implies \Delta \mX$ and modifications between them respectively. In this section we put this together as an equivalence of categories that establishes $\mD$ as the oplax colimit of $D$ in $\Cat(\cE)$. \\
\begin{thm}[$\mD$ is the oplax colimit of $D$]\label{thm IntGroth is OpLaxColim}
Let $\cE$ admit an internal category of elements of $D : \cA \to \Cat(\cE)$, as in Definition~\ref{def E admits an internal category of elements of D}. Then for every internal category $\mX \in \Cat(\cE)$, the category of lax natural transformations $D \implies \Delta \mX$ and their modifications is isomorphic to the category of internal functors $\mD \to \mX$ and their internal natural transformations.
\[[D , \Delta \mX]_{\ell} \cong \Cat(\cE)(\mD , \mX) \]
\end{thm}
\begin{proof}
The objects and morphisms are in bijection by Propositions $\ref{Prop 1-cell correspondence}$ and $\ref{Prop 2-cell correspondence}$ respectively. We only need to show composition and identities are preserved in one direction of the 2-cell correspondence. For any lax natural transformation $x : D \implies \Delta \mX$, the identity modification $1_x$ consists by the family of identity internal natural transformations $1_{x_A}$ for each $A \in \cA_0$. Let $\theta : \mD \to \mX$ and $\overline{1_\mX} : \theta \implies \theta$ be the internal functor and internal natural transformation corresponding to $x$ and $1_\mX$ respectively. Then 

\[ \iota_A \overline{1_\mX} := (1_\mX)_A = 1_{x_A} := e_A (x_A)_1 = (x_A)_0 e = \iota_A \theta_0 e \]\

\noi shows $\overline{1_\mX}$ is the identity natural transformation on $\theta$. \\
\noi Let $\gamma : x \Rrightarrow y$ and $\sigma : y \Rrightarrow z$ be modifications of lax natural transformations $x , y, z: D \implies \Delta \mX$. For each $A \in \cA_0$ let $\gamma_A : x_A \implies y_A$ and $\sigma_A : y_A \implies z_A$ be the internal natural transformations defining $\gamma$ and $\sigma$ so the composite

\[ \gamma \sigma : x \Rrightarrow z\]

\noi is defined by the composite of internal natural transformations 

\[ (\gamma \sigma)_A = \gamma_A \sigma_A\]

\noi Internally this is given by the composite 

\begin{center} 
\begin{tikzcd}
D(A)_0 \ar[dr, "\gamma_A \sigma_A"']\rar["\langle \gamma_A { , }\sigma_A \rangle "] & \mX_2 \dar["c"] \\
& \mX_1
\end{tikzcd}
\end{center}\

\noi and now we can see that 

\begin{align*}
\iota_A \overline{\gamma \sigma} 
= (\gamma \sigma)_A 
= \gamma_A \sigma_A 
= \langle \gamma_A , \sigma_A \rangle c 
= \langle \iota_A \overline{\gamma} , \iota_A \overline{\sigma} \rangle c 
= \iota_A \langle \overline{\gamma} , \overline{\sigma} \rangle c .
\end{align*}\

\noi By the universal property of $\mD_0$ we have that 

\[\overline{\gamma \sigma} = \langle \overline{\gamma} , \overline{\sigma} \rangle c \]\

\noi where the right-hand side defines the (horizontal) composition of natural transformations $\overline{\gamma} : \theta \implies \omega$ , $\overline{\sigma} \omega \implies \nu$, where $\theta, \omega, \nu : \mD \to \mX$ correspond to $x,y,$ and $z$ respectively. 
\end{proof}

\section{A Setting for an Internal Category of Fractions}\label{Ch Internal Fractions}

In this chapter we give a suitable context, $\cE$, and conditions on an internal category, $\mC$, and a map $w: W \to \mC_1$ that allow us to express a set of axioms for an internal category of (right) fractions. We show that such a pair $(\mC , w)$ satisfying our Internal Fractions Axioms allows us to define an internal category, $\mCW$, which satisfies an analogous universal property expressed by Theorem~\ref{thm 2-dimensional universal property of localization}. We write $(\mC, W)$ for the pair from now on, as the map $w : W \to \mC_1$ will be fixed and implied by $W$.
The contextual conditions on $(\mC, W)$ allow us to build the objects of diagrams in $\mC$ we need in the Internal Fractions Axioms and also represent the arrows and paths of composable arrows in our internal category of fractions as equivalence classes of spans and paths of composable spans respectively. The contextual conditions on $\cE$ then allow us formulate the Internal Fractions Axioms in terms of lifts of local witnesses to the axioms. This local data can be glued together to give globally defined structure maps provided a gluing condition is satisfied and we use this to define the internal category of fractions, $\mCW$, along with its structure maps and to prove it is an internal category. We will often give representations of our definitions and constructions as they would appear in $\cE = \Set$ to help our readers and we will overload the symbols for structure maps of internal categories, namely $s,t,c,$ and $e$. We will also abuse some notation and language by referring to arrows $W \to \mC_1$ as representing `arrows in $W$' when in general we mean it represents a family of arrows in an internal category $\mC$ indexed by $W$. The symbols, $\pi_i$, will be overloaded and used for the $i$'th (pullback) projection of all pullbacks. Here $i$ stands for `number of components to the right of the left-most component,' naturally. 

In Section \ref{S context for fractions axioms} we describe the conditions we require for the pair $(\mC , W)$ in order to state the Internal Fractions Axioms that culminate to Definition~\ref{def candidate for internal fractions}. In Section \ref{S the axioms (of internal fractions)} we define the structure we need on the ambient category $\cE$ in Definition~\ref{defn candidate context for internal fractions and covers} and present the Internal Fractions Axioms as part of Definition~\ref{def Internal Fractions Axioms}. We define the objects and structure maps for the internal category $\mCW$ in Section \ref{S Defining mCW, the internal cat of fracs}. We show that internal composition is assocaitive and satisfies the identity laws in $\mCW$, making $\mCW$ an internal category, in Section \ref{S internal cat of fracs assoc and id laws}. In Section \ref{S internal localization functor} we define the associated internal localization functor and prove that it inverts $w : W \to \mC_1$ in a suitable sense. In the last section, Section \ref{S UP of Internal Fractions}, we prove the universal property of the internal category of fractions as Theorem~\ref{thm 2-dimensional universal property of localization}.

\subsection{A Suitable Context}\label{S context for fractions axioms} 

This definition of an internal category of fractions requires working in suitably structured ambient category $\cE$ and with suitable internal categories $\mC$ in $\cE$. The conditions on $\cE$ will allow us to define the required objects and structure maps for the internal category of fractions, $\mCW$. The conditions on the internal categories, $\mC$, being considered allow allow us to describe the Internal Fractions Axioms and describe reflexive internal graphs of fractions. The following definitions will be important for defining the context in this section and the Internal Fractions Axioms in the next section.

\begin{defn}\label{def effective epi}
An {\em effective epimorphism} in a category $\cE$ is the coequalizer of its kernel pair. 
\end{defn}

\begin{defn}\label{def universal effective epis}
A class of effective epimorphisms in $\cE$ is called {\em universal} if they are stable under pullback and composition. 
\end{defn}

\noi Universal effective epimorphisms appear in each of the Internal Fractions Axioms in Definition~\ref{def Internal Fractions Axioms} and are used to define the composition, source, and target structure maps for the internal category of fractions. For the rest of this chapter we assume $\cE$ has a class of universal effective epimorphisms, $\cJ$, and we call these epimorphisms {\em covers}. We will see these in the next section when we give the Internal Fractions Axioms in Section \ref{S the axioms (of internal fractions)}.

For the rest of this section we focus on the conditions we will impose on an internal category $\mC$ and an arrow $w: W \to \mC_1$ in $\cE$ in order to construct the building blocks of our internal category of fractions, $\mCW$, as well as the objects of diagrams in $\mC$ that we use to internalize the axioms for a category of fractions. For example, let $\spn$ denote the object of spans in $\mC$ whose left leg is in $W$, 

\begin{center}
\begin{tikzcd}[]
\cdot & \cdot \lar["\circ" marking] \rar & \cdot 
\end{tikzcd},
\end{center}

\noi let $\csp$ denote the object of cospans in $\mC$ whose right leg is in $W$, 

\begin{center}
\begin{tikzcd}[]
\cdot  \rar & \cdot & \cdot \lar["\circ" marking]
\end{tikzcd},
\end{center}

\noi let $W_\triangle$ denote the object of pairs of arrows whose terminal arrow and composite is in $W$, 

\begin{center}
\begin{tikzcd}[]
&\cdot \dar \ar[dl, "\circ" marking] \\
\cdot & \cdot \lar["\circ" marking] 
\end{tikzcd},
\end{center}
and let $\slb$ denote the object of the following commuting {\em sailboat} diagrams (in $\mC$) 

\begin{center}
\begin{tikzcd}[]
&\cdot \dar \ar[dl, "\circ" marking] & \\
\cdot & \cdot \lar["\circ" marking] \rar & \cdot 
\end{tikzcd}
\end{center}

\noi These are all obtained by the following pullbacks in $\cE$ respectively:

\begin{center}
\begin{tikzcd}[]
\spn \rar["\pi_1"] \dar["\pi_0"'] 
 \arrow[dr, phantom, "\usebox\pullback" , very near start, color=black]
 & \mC_1 \dar["s"] \\
W \rar["ws"'] 
& \mC_0 
\end{tikzcd}
\qquad 
\begin{tikzcd}[]
\csp \rar["\pi_1"] \dar["\pi_0"'] 
 \arrow[dr, phantom, "\usebox\pullback" , very near start, color=black]
 & W \dar["wt"] \\
\mC_1 \rar["t"'] 
& \mC_0 
\end{tikzcd}\qquad 
\begin{tikzcd}[]
W_\triangle \dar["\pi_0"'] \rar["\pi_1"] \arrow[dr, phantom, "\usebox\pullback" , very near start, color=black]& W \dar["w"] \\
\mC_1 \tensor[_t]{\times}{_s} W \rar["c"'] & \mC_1
\end{tikzcd}\qquad
\begin{tikzcd}[]
\slb \dar["\pi_0"'] \rar["\pi_1"] \arrow[dr, phantom, "\usebox\pullback" , very near start, color=black] & \mC_1 \dar["s"] \\
W_\triangle \rar["\pi_0 \pi_1 w s"'] & \mC_0 
\end{tikzcd}
\end{center}\

\noi The pullback,

\begin{center}
\begin{tikzcd}[]
W_\circ \dar["\pi_0"'] \rar["\pi_1"] 
 \arrow[dr, phantom, "\usebox\pullback" , very near start, color=black]
 & W \dar["w"]\\ 
\mC_1 \tensor[_t]{\times}{_{ws}} W \tensor[_{wt}]{\times}{_{ws}} W \rar["c"'] 
&\mC_1
\end{tikzcd},
\end{center}

\noi is the object of composable paths of length three where the last two arrows are in $W$ and their composite is again in $W$. In $\cE = \Set$, the elements of this set would be composable pairs of arrows in the image of $w : W \to \mC_1$ along with a pre-composable arrow in $\mC_1$ such that their composition (in $\mC)$ gives an element in the image of $W$. We use this object to express a weak composition axiom for internal fractions. 

The pullback 

\begin{center}
\begin{tikzcd}[]
W_\square \dar["\pi_0"'] \rar["\pi_1"] 
 \arrow[dr, phantom, "\usebox\pullback" , very near start, color=black]
 & \mC_1 \tensor[_t]{\times}{_s} W \dar["c"]\\ 
 W \tensor[_t]{\times}{_s} \mC_1 \rar["c"'] 
&\mC_1
\end{tikzcd}
\end{center}

\noi represents commuting diagrams in $\mC$ that commonly known as {\em Ore squares}: 

\begin{center}
\begin{tikzcd}[]
\cdot \dar["\circ" marking ] \rar[] & \cdot \dar["\circ" marking] \\
\cdot \rar[] & \cdot 
\end{tikzcd}
\end{center}

\noi The arrows marked with $\circ$ denote arrows in the image of $w : W \to \mC_1$. This object is used to express the internal (right) Ore condition. Let $P(\mC)$ denote the object of parallel arrows in $\cC$ given by the pullback of pairing of source and target maps:

\begin{center}
\begin{tikzcd}[]
P(\mC) \rar["\pi_1"] \dar["\pi_0"'] \arrow[dr, phantom, "\usebox\pullback" , very near start, color=black]& \mC_1 \dar["(s{,}t)"] \\
\mC_1 \rar["(s{,}t)"'] & \mC_0 \times \mC_0 
\end{tikzcd}
\end{center}

\noi Let $\cP_{eq}(\mC)$ and $\cP_{cq}(\mC)$ be the objects of equalized and coequalized parallel arrows in $\cC$ (that don't satisfy any kind of internal universal property) given by the equalizers 

\begin{center}
\begin{tikzcd}[column sep = huge]
\cP_{eq}(\mC) \rar[tail, "\iota_{eq}"] & W \tensor[_{wt}]{\times}{_s} P(\mC) \rar[shift left, "(\pi_0 w{,} \pi_1 \pi_0)c"] \rar[shift right, "(\pi_0 w{,} \pi_1 \pi_1)c"'] & \mC_1 
\end{tikzcd}
\end{center}
\noi and 
\begin{center}
\begin{tikzcd}[column sep = huge]
\cP_{cq}(\mC) \rar[tail, "\iota_{eq}"] & P(\mC) \tensor[_t]{\times}{_{ws}} W \rar[shift left, "(\pi_0 \pi_0{,}\pi_1 w)c"] \rar[shift right, "(\pi_0 \pi_1{,}\pi_1 w)c"'] & \mC_1 
\end{tikzcd}
\end{center}

\noi in $\cE$ (that satisfy the usual universal property). Let $\cP(\mC)$ denote the following pullback

\begin{center}
\begin{tikzcd}[]
\cP(\mC) \rar["\pi_1"] \dar["\pi_0"'] \arrow[dr, phantom, "\usebox\pullback" , very near start, color=black] & \cP_{cq}(\mC) \dar["\pi_0"] \\ 
\cP_{eq}(\mC) \rar["\pi_1"'] & P(\mC)
\end{tikzcd}
\end{center}

\noi representing diagrams of the form: 

\begin{center}
\begin{tikzcd}[]
\cdot \rar["\circ" marking]& \cdot \rar[shift left] \rar[shift right] & \cdot \rar["\circ" marking] & \cdot
\end{tikzcd}
\end{center}

\noi where the $\circ$ marked arrows represent arrows indexed by $W$. This will be used for internalizing the so-called `right-cancellability' or `lifting' condition for internal fractions. In this thesis we refer to this as `zippering' in order to avoid confusion with the lifts in the Internal Fractions Axioms and because of the way it applies in proofs related to span composition. Note that these equalizers can be given as pullbacks of pairing maps 

\[ (1 , (\pi_0 \pi_0 , \pi_1 w ) c) , (1, (\pi_0 \pi_1 , \pi_1 w )c) : ( P(\mC) \tensor[_t]{\times}{_{ws}} W) \to (P(\mC) \tensor[_t]{\times}{_{ws}} W) \times \mC_1\]

\[ (1 , (\pi_0 w , \pi_1 \pi_0) c) , (1,(\pi_0 w , \pi_1 \pi_1) c) : ( W \tensor[_{wt}]{\times}{_s} P(\mC) ) \to (W \tensor[_{wt}]{\times}{_s} P(\mC) ) \times \mC_1\]

\noi when they all exist in $\cE$ with the pullback projections being made equal by the identity maps in the left-hand components of the pairing maps above and the equalizer condition being forced by the right-hand components respectively. The constructions above allow us to formalize the axioms for a category of (right) fractions and we now give a name to the collection of internal categories, $\mC$, of $\cE$ and maps, $w : W \to \mC_1$, in $\cE$ for which this happens. The next definition describes a setting in which we can state the Internal Fractions Axioms. We'll be overloading notation for the structure maps of an internal category, and suppress $w : W \to \mC_1$ when describing internal composition with arrows in $\mC$ indexed by $W$.

\begin{defn}\label{def (mC,W) is a pre-candidate for internal fractions}
Let $\mC$ be an internal category in $\cE$ and let $w : W \to \mC_1$ be an arrow in $\cE$. We say the pair $(\mC, W)$ {\em is a pre-candidate for internal fractions} if the following pullbacks

\noindent\begin{tabularx}{\linewidth}{ C C }
\begin{tikzcd}[ampersand replacement=\&]
\csp \rar["\pi_1"] \dar["\pi_0"'] 
 \arrow[dr, phantom, "\usebox\pullback" , very near start, color=black]
 \& W \dar["wt"] \\
\mC_1 \rar["t"'] 
\& \mC_0 
\end{tikzcd}
&
\begin{tikzcd}[ampersand replacement=\&]
\spn \rar["\pi_1"] \dar["\pi_0"'] 
 \arrow[dr, phantom, "\usebox\pullback" , very near start, color=black]
 \& \mC_1 \dar["s"] \\
W \rar["ws"'] 
\& \mC_0 
\end{tikzcd} \\
\begin{tikzcd}[ampersand replacement=\&]
W \tensor[_{wt}]{\times}{_s} \mC_1
\rar["\pi_1"] \dar["\pi_0"'] 
\arrow[dr, phantom, "\usebox\pullback" , very near start, color=black]
\& W \dar["ws"] \\
\mC_1 \rar["t"'] 
\& \mC_0 
\end{tikzcd}
& 
\begin{tikzcd}[ampersand replacement=\&]
\mC_1 \tensor[_t]{\times}{_{ws}} W
\rar["\pi_1"] \dar["\pi_0"'] 
\arrow[dr, phantom, "\usebox\pullback" , very near start, color=black]
\& \mC_1 \dar["s"] \\
W \rar["wt"'] 
\& \mC_0 
\end{tikzcd} \\
\begin{tikzcd}[ampersand replacement=\&]
W_\square \dar["\pi_0"'] \rar["\pi_1"] \arrow[dr, phantom, "\usebox\pullback" , very near start, color=black] \& W \tensor[_{wt}]{\times}{_s} \mC_1 \dar["c"] \\
\mC_1 \tensor[_{t}]{\times}{_{ws}} W \rar["c"'] \& \mC_1 
\end{tikzcd}
&
\begin{tikzcd}[ampersand replacement=\&]
 W \tensor[_t]{\times}{_s} W 
 \arrow[dr, phantom, "\usebox\pullback" , very near start, color=black]
 \rar["\pi_1"] \dar["\pi_0"'] 
\& W \dar["ws"] \\
W \rar["wt"'] \& \mC_0
\end{tikzcd}\\
\begin{tikzcd}[ampersand replacement=\&]
W_\blacktriangle
 \arrow[dr, phantom, "\usebox\pullback" , very near start, color=black]
 \rar["\pi_1"] 
\dar["\pi_0"'] 
\& W \tensor[_t]{\times}{_s} W 
\dar["\pi_0ws"] \\
\mC_1 \rar["t"'] \& \mC_0
\end{tikzcd}
&
\begin{tikzcd}[ampersand replacement=\&]
W_\circ \dar["\pi_0"'] \rar["\pi_1"] \arrow[dr, phantom, "\usebox\pullback" , very near start, color=black]\& W \dar["w"] \\
W_\blacktriangle \rar["c"'] \& \mC_1
\end{tikzcd}\\
\begin{tikzcd}[ampersand replacement=\&]
W_\triangle \dar["\pi_0"'] \rar["\pi_1"] \arrow[dr, phantom, "\usebox\pullback" , very near start, color=black]\& W \dar["w"] \\
\mC_1 \tensor[_t]{\times}{_s} W \rar["c"'] \& \mC_1
\end{tikzcd}
&
\begin{tikzcd}[ampersand replacement=\&]
\slb \dar["\pi_0"'] \rar["\pi_1"] \arrow[dr, phantom, "\usebox\pullback" , very near start, color=black] \& \mC_1 \dar["s"] \\
W_\triangle \rar["\pi_0 \pi_1 w s"'] \& \mC_0 
\end{tikzcd} 
\\
\begin{tikzcd}[ampersand replacement=\&]
P(\mC) \rar["\pi_1"] \dar["\pi_0"'] \arrow[dr, phantom, "\usebox\pullback" , very near start, color=black]
\& \mC_1 \dar["(s{,}t)"] \\
\mC_1 \rar["(s{,}t)"'] 
\& \mC_0 \times \mC_0 
\end{tikzcd} 
&
\begin{tikzcd}[ampersand replacement=\&]
W \tensor[_{wt}]{\times}{_{\pi_0s}} P(\mC) \rar["\pi_1"] \dar["\pi_0"'] \arrow[dr, phantom, "\usebox\pullback" , very near start, color=black]
\& P(\mC) \dar["\pi_0 s"] \\
W \rar["wt"'] 
\& \mC_0
\end{tikzcd} \\
\begin{tikzcd}[ampersand replacement=\&]
 P(\mC) \tensor[_{\pi_0t}]{\times}{_{ws}} W \rar["\pi_1"] \dar["\pi_0"'] \arrow[dr, phantom, "\usebox\pullback" , very near start, color=black]
 \& W \dar["ws"] \\
 P(\mC) \rar["\pi_0 t"'] 
 \& \mC_0
\end{tikzcd} 
&
\begin{tikzcd}[ampersand replacement=\&]
\cP(\mC) \rar["\pi_1"] \dar["\pi_0"'] \arrow[dr, phantom, "\usebox\pullback" , very near start, color=black] 
\& \cP_{cq}(\mC) \dar["\pi_0"] \\ 
\cP_{eq}(\mC) \rar["\pi_1"'] 
\& P(\mC)
\end{tikzcd}
\end{tabularx}

\noi exist in $\cE$ along with the equalizers,

\begin{center}
\begin{tikzcd}[column sep = huge]
\cP_{eq}(\mC) \rar[tail, "\iota_{eq}"] & W \tensor[_{wt}]{\times}{_s} P(\mC) \rar[shift left, "(\pi_0 w{,} \pi_1 \pi_0)c"] \rar[shift right, "(\pi_0 w{,} \pi_1 \pi_1)c"'] & \mC_1 
\end{tikzcd}
\end{center}
\noi and 
\begin{center}
\begin{tikzcd}[column sep = huge]
\cP_{cq}(\mC) \rar[tail, "\iota_{eq}"] & P(\mC) \tensor[_t]{\times}{_{ws}} W \rar[shift left, "(\pi_0 \pi_0{,}\pi_1 w)c"] \rar[shift right, "(\pi_0 \pi_1{,}\pi_1 w)c"'] & \mC_1 
\end{tikzcd}.
\end{center}

\end{defn}\

\noi These pullbacks and equalizers give us the building blocks we need to state the internal fractions axioms and construct an internal category of fractions. For the rest of this chapter we assume that $(\mC, W)$ is a pre-candidate for internal fractions. Our construction requires a bit of `scaffolding' however, in the form of a family of reflexive internal graphs encoding the equivalence relation on spans and paths of composable spans. At this point we can only define the first one given by the two maps $p_0, p_1 : \slb \to \spn$, defined explicitly as the pairing maps

\[ p_0 = (\pi_0 \pi_0 \pi_1 {,} \pi_1) , \quad p_1 = (\pi_0\pi_1 {,} ( \pi_0 \pi_0 \pi_0 {, }\pi_1) c ) \]

\noi by the universal property of the pullback $\spn$. These maps represent projecting two different spans out of a commuting diagrams (in $\mC$) which we call a sailboat. The idea is that coequalizing these will produce equivalence classes of spans that are related by being part of a common sailboat. For example, when $\cE = \Set$, the maps $p_0$ and $p_1$ can be seen to project sailboats in $\slb$ to spans in $\spn$ like this:

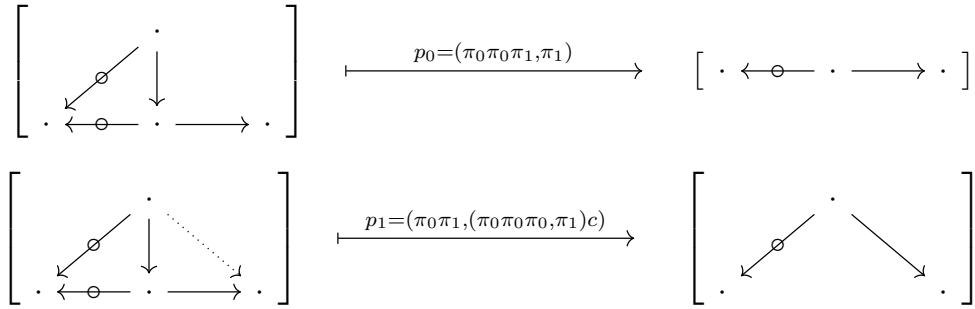
\begin{center}
$\left[\begin{tikzcd}[]
&\cdot \dar \ar[dl, "\circ" marking] & \\
\cdot & \cdot \lar["\circ" marking] \rar & \cdot 
\end{tikzcd} \right]$
\begin{tikzcd}[column sep = huge]
\rar[rr, mapsto, "p_0 = (\pi_0 \pi_0 \pi_1 {,} \pi_1) "] && \
\end{tikzcd}
$\left[\begin{tikzcd}[]
\cdot & \cdot \lar["\circ" marking] \rar & \cdot 
\end{tikzcd}\right]$
\end{center}

\begin{center}
$\left[\begin{tikzcd}[]
&\cdot \dar \ar[dl, "\circ" marking] \ar[dr, dotted] & \\
\cdot & \cdot \lar["\circ" marking] \rar & \cdot 
\end{tikzcd} \right]$
\begin{tikzcd}[column sep = huge]
\rar[rr, mapsto, "p_1 = (\pi_0\pi_1 {,} ( \pi_0 \pi_0 \pi_0 {, }\pi_1) c ) "] && \
\end{tikzcd}
$\left[\begin{tikzcd}[]
&\cdot \ar[dr] \ar[dl, "\circ" marking] & \\
\cdot & & \cdot 
\end{tikzcd}\right]$
\end{center}

\noi The dotted arrow on the left-hand side is just emphasizing that the pair is composable (in $\mC$) in order to make the mapping more clear. The arrows in a category of fractions are equivalence classes of spans, where two distinct spans represent the same equivalence class whenever there exists an intermediate span and two sail-boats such that the intermediate span forms the $p_1$ span projection of two different sailboats whose $p_0$ span projections are the original two spans. For $\cE = \Set$ the coequalizer of $p_0$ and $p_1$ describes precisely this set of equivalence classes of spans. The following lemma shows $p_0$ and $p_1$ form a reflexive pair in general. 

\begin{lem}\label{lem p_0 and p_1 are a reflexive pair (base case for paths)}
The parallel pair 

\[ \begin{tikzcd}
\slb \rar[shift left, "p_0"] \rar[shift right, "p_1"'] & \spn 
\end{tikzcd}\]
\noi is reflexive.
\end{lem}
\begin{proof}
Define a map,
\[ \varphi_s : \spn \to \slb,\]
\noi by the pairing map 

\[ \varphi_s = \big( (( \pi_0 w s e , \pi_0 ) , \pi_0 ), \pi_1 \big). \]

\noi The component 

\[ \varphi_s \pi_0 = (( \pi_0 w s e , \pi_0 ) , \pi_0 ) : \spn \to W_\triangle \] 

\noi is well-defined by the identity law for composition in $\mC$: 

\[ ( \pi_0 w s e , \pi_0w  )c = \pi_0 w (se,1) c = \pi_0 w \]

\noi The other component is well-defined because 

\[ \varphi_s \pi_1 s = \pi_1 s = \pi_0 w s = \varphi_s \pi_0 s. \]

\noi To see that $\varphi_s$ is a common section of $p_0$ and $p_1$ we get

\begin{align*}
  \varphi_s p_0 \pi_0 
  &= \varphi_s \pi_0 \pi_0 \pi_1 
  &\varphi_s p_0 \pi_1 
  &= \varphi_s \pi_1 \\
  &= \pi_0
  &
  &= \pi_1
\end{align*}

\noi by definition and by the identity law in $\mC$ we also get
\begin{align*}
  \varphi_s p_1 \pi_0
  &= \varphi_s \pi_0 \pi_1 
  &\varphi_s p_1 \pi_1 
  &= \varphi_s (\pi_0 \pi_0 \pi_0 , \pi_1) c \\
  &= \pi_0
  &
  &= (\pi_0 w s e , \pi_1) c \\
  &
  &
  &= (\pi_1 s e , \pi_1)c \\
  &
  &
  &= \pi_1 (se, 1) c \\
  &
  &
  &= \pi_1. 
\end{align*}

\noi By the universal property of the pullbacks $\slb$ and $\spn$

\[ \varphi_s p_0 = 1_{\spn} = \varphi_s p_1.\]
\end{proof}

\noi To define the source, target, and composition structure maps, and prove associativity and identity laws we need to reason about paths (or zig-zags more accurately) of composable spans and sailboats. In the following definition we give a sufficient condition for obtaining these as reflexive internal graphs. 

\begin{defn}\label{def (mC, W) admits reflexive graph of fractions}
We say that a pre-candidate for internal fractions, $(\mC , W)$, in $\cE$, {\em admits reflexive graphs of fractions} if the source and target maps on $\spn$ and $\slb$,

\[ \begin{tikzcd}
\spn \ar[dr,"\pi_0 w t "'] 
&& \spn \ar[dl,"\pi_1 t "] \\
&\mC_0 
\end{tikzcd} 
\qquad \qquad 
\begin{tikzcd}
\slb \ar[dr,"\pi_0 \pi_0 \pi_1 w t "'] 
&& \slb \ar[dl,"\pi_1 t "] \\
&\mC_0 
\end{tikzcd}\]

\noi admit pullbacks along one another. 

\end{defn}

\noi The following lemma shows precisely which reflexive graphs are being referred to in Definition~\ref{def (mC, W) admits reflexive graph of fractions}. 

\begin{lem}\label{lem p_0^n and p_1^n are reflexive pairs for all n}
The pairs 

\[ \begin{tikzcd}
\slb \tensor[_t]{\times}{_s} \dots \tensor[_t]{\times}{_s} \slb \rar[shift left, "p_0^n"] \rar[shift right, "p_1^n"'] & \spn \tensor[_t]{\times}{_s} \dots \tensor[_t]{\times}{_s} \spn 
\end{tikzcd}\]

\noi are reflexive for each $n \geq 1$ where $n = 1$ is the case $p_0 , p_1 : \slb \to \spn$ and $p_i^n = (\pi_0 p_i , \pi_1 p_i , ..., \pi_{n-1} p_i)$ is the unique map determined by the iterated pullback projections and the map $p_i$ for $i = 0,1$. 
\end{lem}
\begin{proof}
The proof is by induction on the number of pullbacks. The base case, $n = 1$, follows from Lemma~\ref{lem p_0 and p_1 are a reflexive pair (base case for paths)}. Assume $p_0^n$ and $p_1^n$ have a common section $\varphi_s^n$. Let $\slb^n$ denote the $n$-fold pullback of $t,s : \slb \to \mC_0$, similarly for $\spn^n$, for each natural number $n$. The induction step follows from the following commuting diagram:

\[\begin{tikzcd}[column sep = large, row sep = large]
 \spn^{n+1} \rar[dotted, ""] \dar["\cong"'] \rar[rr, bend left, equals]
& \slb^{n+1} \dar["\cong"] \rar[shift left, "p_0^{n+1}"] \rar[shift right, "p_1^{n+1}"'] 
& \spn^{n+1} \dar["\cong"] \\
\spn^n \tensor[_t]{\times}{_s} \spn
\rar["\varphi_s^n \times \varphi_s"'] 
\rar[rr, bend right, equals]
& \slb^n \tensor[_t]{\times}{_s} \slb 
\rar[shift left, "p_0^n \times p_0 "] \rar[shift right, "p_1^n \times p_1"'] 
& \spn^n \tensor[_t]{\times}{_s}\spn 
\end{tikzcd}\]

\noi The bottom commutes by the universal property of the pullbacks in the bottom row above:

\begin{align*}
(\varphi_s^n \times \varphi_s)(p_0^n \times p_0) 
&= (\pi_0 \varphi_s^n , \pi_1 \varphi_s ) ( \pi_0 p_0^n , \pi_1 p_0)\\
&= (\pi_0 \varphi_s^n p_0^n , p_1 \varphi_s p_0) \\
&= (\pi_0 , \pi_1) \\
&= 1\\
&= (\pi_0 , \pi_1) \\
&= (\pi_0 \varphi_s^n p_1^n , p_1 \varphi_s p_1) \\
&= (\pi_0 \varphi_s^n , \pi_1 \varphi_s ) ( \pi_0 p_1^n , \pi_1 p_1)\\
&= (\varphi_s^n \times \varphi_s)(p_1^n \times p_1).
\end{align*} 

\noi This implies the top commutes and we get a reflexive graph:

\[ \begin{tikzcd}
\slb \tensor[_t]{\times}{_s} \dots \tensor[_t]{\times}{_s} \slb \rar[shift left = 5, bend left, "p_0^{n+1}"] \rar[shift right = 5, bend right, "p_1^{n+1}"'] & \spn \tensor[_t]{\times}{_s} \dots \tensor[_t]{\times}{_s} \spn \lar["\varphi_s^n"'] 
\end{tikzcd} \]

\noi The result follows by induction.
\end{proof}\

\noi The arrows and composable paths in the internal category of fractions should be the coequalizers of the internal reflexive graphs in Lemmas \ref{lem p_0 and p_1 are a reflexive pair (base case for paths)} and \ref{lem p_0^n and p_1^n are reflexive pairs for all n} respectively. For this we need to require the existence of these coequalizers and the pullbacks of the induced source and target maps on the coequalizer of $p_0$ and $p_1$. In order for this to define an internal category we need the coequalizers of the higher order reflexive graphs to coincide with the iterated pullbacks of the first coequalizer. The following definition is used to restrict our focus to internal categories for which these coequalizers exist. 

\begin{defn}\label{def (mC, W) admits internal quotient graphs of fractions}
We say $(\mC, W)$ {\em admits internal quotient graphs of fractions} if the coequalizer, 

\[\begin{tikzcd}[column sep = large, row sep = large]
\slb \tensor[_t]{\times}{_s} \dots \tensor[_t]{\times}{_s} \slb \rar[shift left, "p_0^n"] \rar[shift right, "p_1^n"'] & \spn \tensor[_t]{\times}{_s} \dots \tensor[_t]{\times}{_s} \spn \rar[two heads, "q_n"] & \mCW_n
\end{tikzcd},\]

\noi exists in $\cE$ for each $n \geq 1$. 
\end{defn}

\noi The coequalizers in Definition~\ref{def (mC, W) admits internal quotient graphs of fractions} are named suggestively. In particular $\mCW_1$ is how we will define the object of arrows for the internal category of fractions. Using its universal property we can define source and target structure maps by the following lemma. 

\begin{lem}[Source and Target Structure for $\mCW$] \label{lem defining source and target of mCW} The
\textbf{source} and \textbf{target} maps for $\mCW$ are determined by the universal property of the coequalizer $\mCW_1$ and more precisely induced by $s' = \pi_0 w t$ and $t' = \pi_1 t$.

\[
\begin{tikzcd}[]
& \mC_1 \rar["t"] & \mC_0 \\
\slb \rar[r,shift left, "p_0"] \rar[r,shift right, "p_1"'] & \spn \rar[r, two heads, "q"] \ar[d, "\pi_0"'] \ar[dr, "s'"'] \uar["\pi_1"] \ar[ur, "t'"] & \mCW_1 \dar[dotted, "s"] \uar[dotted, "t"']\\
& W \rar["wt"'] & \mC_0 
\end{tikzcd}
\]
\end{lem} 
\begin{proof}
\noi This is well-defined by the following calculations. 

\begin{align*}
 p_0 s'
 &=(\pi_0 \pi_0 \pi_1 {,} \pi_1) s'
 & p_0t
 &=(\pi_0 \pi_0 \pi_1 , \pi_1) t' \\
 &= (\pi_0 \pi_0 \pi_1 {,} \pi_1) \pi_0 w t 
 & 
 &= (\pi_0 \pi_0 \pi_1 , \pi_1) \pi_1 t \\
 &= \pi_0 \pi_0 \pi_1 w t 
 & 
 &= \pi_1 t \\
 &= \pi_0 \pi_0 c t & &= ( \pi_0 \pi_0 \pi_0, \pi_1 ) c t \\
 &= \pi_0 \pi_1 w t &&= (\pi_0\pi_1 , ( \pi_0 \pi_0 \pi_0 ,\pi_1) c )\pi_1 t \\
 &= (\pi_0\pi_1 , ( \pi_0 \pi_0 \pi_0 , \pi_1) c )\pi_0 w t && = (\pi_0\pi_1 , ( \pi_0 \pi_0 \pi_0 , \pi_1) c )t' \\
 &= (\pi_0\pi_1 , ( \pi_0 \pi_0 \pi_0 , \pi_1) c )s' &&=p_1 t' \\
 &= p_1 s' &&
\end{align*}
\end{proof}

\noi Now we define the pairs, $(\mC, W)$, that admit internal quotient graphs of fractions in $\cE$ for which the pullbacks of the induced source and target maps on the coequalizers $\mCW_1$ exist. Notice this only requires the coequalizer of the pair $p_0 , p_1 : \slb \to \spn$ so these could exist without the other reflexive graphs. Being able to construct proofs for coherences for associativity of composition for longer paths of arrows becomes unclear without the universal property of the coequalizers of the other internal reflexive graphs, $p_0^n $ and $ p_1^n$. 

\begin{defn}\label{def (mC,W) admits paths of fractions}
We say the pair $(\mC, W)$ {\em admits paths of fractions} if the coequalizer $\mCW_1$ exists and $\cE$ admits pullbacks of the induced source or target maps, $s, t : \mCW_1 \to \mC_0$, along one another as well as the source and target maps of $\spn$ and $\slb$. 
\end{defn}

\noi The next definition is the last one in this section and describes all the structure we need for the internal categories we consider. 

\begin{defn}\label{def candidate for internal fractions}
We say the pair $(\mC, W)$ is a {\em candidate for internal fractions} if it satisfies Definitions \ref{def (mC,W) is a pre-candidate for internal fractions}, \ref{def (mC, W) admits internal quotient graphs of fractions}, and \ref{def (mC,W) admits paths of fractions} and the induced left and right product functors on the slice category for each of the source and target maps for $\slb, \spn,$ and $\mCW_1$ 

\[ (-) \times s : , \quad t \times (-) : \cE / \mC_0 \to \cE / \mC_0 \] 

\noi preserve reflexive coequalizers.

\end{defn}

For the rest of this thesis we will assume the pair $(\mC , W)$ is a candidate for internal fractions in $\cE$. The remainder of this chapter consists of general lemmas which are combined to state that the coequalizers

\begin{center}
\begin{tikzcd}[column sep = large ]
\slb \tensor[_t]{\times}{_s} \dots \tensor[_t]{\times}{_s} \slb \rar[shift left, "p_0^n"] \rar[shift right, "p_1^n"'] & \spn \tensor[_t]{\times}{_s} \dots \tensor[_t]{\times}{_s} \spn \rar[r, two heads, "q_n"] & \mCW_n
\end{tikzcd}
\end{center}

\noi exist in $\cE$ for each $n \geq 1$. The following lemmas combine in Proposition~\ref{prop path pb is coequalizer} to show how the coequalizer, $\mCW_2$, of the reflexive pair, $p_0^2$ and $p_1^2$, above coincides with the pullback, $\mCW_1 \tensor[_t]{\times}{_s} \mCW_1$, of the induced source and target maps in Lemma~\ref{lem defining source and target of mCW}.

\begin{lem}\label{lem underlying-structure functor from slice cat reflext coequalizers}
For any category $\cE$, the forgetful functor, $U : \cE / \mC_0 \to \cE$, which maps objects $A \to \mC_0$ in $\cE / \mC_0$ to objects $A$ in $\cE$ and is defined similarly on arrows, reflects coequalizers. 
\end{lem}
\begin{proof}
Suppose we have a commuting diagram 

\[\begin{tikzcd}[column sep = large]
A \rar[shift left, "f"] \rar[shift right, "g"'] \ar[dr, "a"'] & B \dar["b"] \rar["h"] & C \ar[dl, "c"] \\
& \mC_0 & 
\end{tikzcd}\]

\noi such that 

\[\begin{tikzcd}[column sep = large]
A \rar[shift left, "f"] \rar[shift right, "g"'] & B \rar["h"] & C 
\end{tikzcd}\]

\noi is a coequalizer in $\cE$. Now suppose there exists an arrow $x : X \to \mC_0$ and another arrow $\varphi : B \to X$ such that the diagram 

\[\begin{tikzcd}[column sep = large]
A \rar[shift left, "f"] \rar[shift right, "g"'] \ar[dr, "a"'] & B \dar["b"] \rar["\varphi"] & X \ar[dl, "x"] \\
& \mC_0 & 
\end{tikzcd}\]

\noi commutes in $\cE$. In $\cE$ we get a unique map $\theta : C \to X$ such that $\theta x = c$ by the universal property of the coequalizer:

\[\begin{tikzcd}[column sep = large, row sep = large]
A \rar[shift left, "f"] \rar[shift right, "g"'] \ar[ddrr, bend right, "a"'] & B \ar[ddr,"b"] \rar["h"] \ar[dr, "\varphi"'] & C \dar[dotted, "\theta"] \dar[dd, bend left, "c"] \\
& & X \dar["x"] \\
& & \mC_0 
\end{tikzcd}\]

\noi This implies that for any $\varphi : b \to x$ in $\cE / \mC_0$, there exists a unique $\theta : c \to x$ such that the diagram

\[\begin{tikzcd}[column sep = large]
a \rar[shift left, "f"] \rar[shift right, "g"'] 
& b \rar["h"] \ar[dr, "\varphi"'] 
& c \dar[dotted, "\theta"] \\
& & x
\end{tikzcd}\]

\noi commutes. It follows that $U$ reflects coequalizers. 
\end{proof}

\noi Notice that the proof above holds for reflexive coequalizers as well, since a section for a reflexive pair in $\cE / \mC_0$ is a section of the underlying reflexive pair in $\cE$. Next we restate and prove Lemma 4.7 from \cite{BarrWells}. It will be used in the proofs of Lemma~\ref{lem composable pair pb is coequalizer} and Proposition~\ref{prop path pb is coequalizer} immediately after.

\begin{lem}\label{lem diagonal reflexive coequalizer lemma}
In any category, if the top row and right column are reflexive coequalizers and the middle column is a reflexive parallel pair, then the diagonal is a coequalizer.

\[\begin{tikzcd}[column sep = huge, row sep = huge]
A \rar[shift left, "f"] \rar[shift right,"g"'] 
& B \dar[shift left, "f'"] \dar[shift right,"g'"'] \rar["h"] \lar[]
& C \dar[shift left, "f''"] \dar[shift right, "g''"' ] \\
& B' \rar["h'"'] \uar[]
& C' \dar["h''"] \uar[] \\
&& C''
\end{tikzcd}\]
\end{lem}
\begin{proof}
Let $x : B' \to X$ be any arrow in the category such that 
\[ f f' x = g g' x.\]
\noi We claim the following diagram commutes where the notation for the sections, $s$, is suppressed: 

\[\begin{tikzcd}[column sep = huge, row sep = huge]
A \rar[shift left, "f"] \rar[shift right,"g"'] 
& B \dar[shift left, "f'"] \dar[shift right,"g'"'] \lar[] \rar["h"] 
& C \dar[shift left, "f''"] \dar[shift right,"g''"']  \\
& B' \dar["x"'] \uar[] \rar["h'"' near end]
& \ar[dl,dotted, "s \theta"] C' \dar["h''"] \uar[] \\
&X \ar[from = uur, dotted, "\theta"' near start, crossing over]& C'' \lar[dotted, "\gamma"] 
\end{tikzcd}\] 

\noi Pre-composing the common section of $f$ and $g$ gives 
\[f' x = g'x \]
\noi and induces the unique map $\theta : C \to X$ such that 
\[h \theta = f'x = g'x \]
\noi by the universal property of the coequalizer $C$. Then 
\[h g'' s \theta = g' h' s \theta = g' s h \theta = g' s g' x = g' x \]
\noi and similarly 
\[ h f'' s \theta = f' x \]
\noi which implies 
\[ h g'' s \theta = h f'' s \theta.\]
\noi By the universal property of the coequalizer $C$, we have that 
\[ g'' s \theta = f'' s \theta\] 
\noi which induces the unique map $\gamma : C'' \to X$ such that 
\[ h'' \gamma = s \theta.\] 
\noi Now we can also see that 
\[ h' h'' \gamma = h' s \theta = s h \theta = s f' x = x \]
\noi and it is unique by the universal property of the coequalizer $C''$.  It follows that the diagonal is a coequalizer and it is reflexive with the section given by composing the common section of $f'$ and $g'$ and the common section of $f$ and $g$. 

\end{proof}

\noi We now apply Lemmas \ref{lem underlying-structure functor from slice cat reflext coequalizers} and \ref{lem diagonal reflexive coequalizer lemma} to get the coequalizers we need to from the internal category of fractions. 
\begin{lem}\label{lem composable pair pb is coequalizer}
The pullback 

\[\begin{tikzcd}[column sep = large, row sep = large]
\mCW_1 \tensor[_t]{\times}{_s} \mCW_1 \rar["\pi_1"] \dar["\pi_0"'] & \mCW_1 \dar["s"] \\
\mCW_1 \rar["t"'] & \mC_0
\end{tikzcd} \]

\noi is also a coequalizer 

\[\begin{tikzcd}[column sep = large, row sep = large]
\slb \tensor[_t]{\times}{_s}\slb \rar[shift left, "p_0^2"] \rar[shift right, "p_1^2"'] & \spn \tensor[_t]{\times}{_s} \spn \rar[two heads, "q \times q"] & \mCW_1 \tensor[_t]{\times}{_s} \mCW_1
\end{tikzcd}\]

\noi in $\cE$. 
\end{lem}
\begin{proof}

Since $(\mC , W)$ admits internal quotient graphs of fractions we know that the object $\mCW_1$ is a reflexive coequalizer of $p_0$ and $p_1$. By Lemma~\ref{lem underlying-structure functor from slice cat reflext coequalizers}, the diagrams 

\[ \begin{tikzcd}[column sep = large, row sep = large]
\slb \rar[shift left, "p_0"] \rar[shift right, "p_1"'] 
\ar[dr, "s"'] 
& \spn \rar[two heads, "q"] 
\dar["s"] 
& \mCW_1 \ar[dl, "s"] \\
& \mC_0 & 
\end{tikzcd}\] 
\noi and 
\[ \begin{tikzcd}[column sep = large, row sep = large]
\slb \rar[shift left, "p_0"] \rar[shift right, "p_1"'] 
\ar[dr, "t"'] 
& \spn \rar[two heads, "q"] 
\dar["t"] 
& \mCW_1 \ar[dl, "t"] \\
& \mC_0 & 
\end{tikzcd}\] 
\noi are coequalizers in $\cE / \mC_0$. These coequalizers are preserved by the left and right product functors on $\cE/\mC_0$ induced by the source and target maps for $\slb, \spn,$ and $\mCW_1$ in $\cE / \mC_0$. This means the top row and right column in the following diagram are reflexive coequalizers in $\cE / \mC_0$ 
\[\begin{tikzcd}[column sep = huge, row sep = huge]
\slb \tensor[_t]{\times}{_s} \slb \rar[shift left, "1 
\times p_0"] \rar[shift right,"1 \times p_1"'] 
& \slb \tensor[_t]{\times}{_s} \spn \dar[shift left, "p_0 \times 1"] \dar[shift right,"p_1 \times 1"'] \rar["1 \times q"] \lar[]
& \slb \tensor[_t]{\times}{_s} \mCW_1 \dar[shift left, "p_0 \times 1"] \dar[shift right, "p_1 \times 1"' ] \\
& \spn \tensor[_t]{\times}{_s} \spn \rar["1 \times q"'] \uar[]
& \spn \tensor[_t]{\times}{_s} \mCW_1 \dar["q \times 1 "] \uar[] \\
&& \mCW_1 \tensor[_t]{\times}{_s} \mCW_1
\end{tikzcd}\]

\noi where we suppress the arrows into $\mC_0$ given by the commuting pullback squares. The middle row is a reflexive pair whose coequalizer, $q \times 1$, is just not drawn in the diagram. Lemma~\ref{lem diagonal reflexive coequalizer lemma} says the diagonal is a coequalizer in $\cE / \mC_0$: 

\[\begin{tikzcd}[column sep = large, row sep = large]
\slb \tensor[_t]{\times}{_s}\slb \rar[shift left, "p_0^2"] \rar[shift right, "p_1^2"'] \ar[dr, "\pi_0 t"'] 
& \spn \tensor[_t]{\times}{_s} \spn \rar[two heads, "q \times q"] \dar["\pi_0 t"] 
& \mCW_1 \tensor[_t]{\times}{_s} \mCW_1 \ar[dl, "\pi_0 t"] \\
& \mC_0 & 
\end{tikzcd}\]

\noi Let $q_2 : \spn \tensor[_t]{\times}{_s} \spn \to \mCW_2$ denote the coequalizer of $p_0^2$ and $p_1^2$ in $\cE$. Notice the following diagram commutes in $\cE$ 

\[\begin{tikzcd}[column sep = huge, row sep = huge]
\slb \tensor[_t]{\times}{_s}\slb \rar[shift left, "p_0^2"] \rar[shift right, "p_1^2"'] \ar[d, "\pi_i "'] 
& \spn \tensor[_t]{\times}{_s} \spn \rar[two heads, "q_2"] \dar["\pi_i"] 
\ar[dr, "q \times q"' ]
& \mCW_2 \ar[d,dotted, "\overline{\pi_i q}"] \\
\slb \rar[shift left, "p_0"] \rar[shift right, "p_1"'] & \spn \rar["q"'] & \mCW_1
\end{tikzcd}\]

\noi by the universal property of the coequalizer, $\mCW_2$, in $\cE$. The same universal property induces the following unique map between the coequalizer and the pullback in the following commuting diagram: 

\[\begin{tikzcd}[column sep = huge, row sep = huge]
\slb \tensor[_t]{\times}{_s}\slb \rar[shift left, "p_0^2"] \rar[shift right, "p_1^2"'] \ar[dr, "\pi_0 t"'] 
& \spn \tensor[_t]{\times}{_s} \spn \rar[two heads, "q_2"] \dar["\pi_0 t"] 
\ar[dr, "q \times q"' ]
& \mCW_2 \ar[d,dotted, "\overline{\pi_0 q} \times \overline{\pi_1 q}"] \\
& \mC_0 & \mCW_1 \tensor[_t]{\times}{_s} \mCW_1 \lar["\pi_0 t "]
\end{tikzcd}\]

\noi More precisely, since $\pi_0t$ coequalizes $p_0^2$ and $p_1^2$ above, the universal property of $\mCW_2$ says there is a unique $\overline{\pi_0 t} : \mCW_2 \to \mC_0$ such that $q_2 \overline{\pi_0 t
} = \pi_0 t$. In particular, because the forgetful functor $\cE/\mC_0 \to \cE$ preserves commuting triangles, \[\overline{\pi_0 t} = (\overline{\pi_0 q} \times \overline{\pi_1 q}) \pi_0 t = \overline{\pi_0 q} t \]
Now the diagram

\[
\begin{tikzcd}[column sep = huge, row sep = huge]
\slb \tensor[_t]{\times}{_s}\slb \rar[shift left, "p_0^2"] \rar[shift right, "p_1^2"'] \ar[dr, "\pi_0 t"'] 
& \spn \tensor[_t]{\times}{_s} \spn \rar[two heads, "q \times q"] \dar["\pi_0 t"] 
& \mCW_1 \tensor[_t]{\times}{_s} \mCW_1 \ar[dl, "\pi_0 t" near start] \dar[dd, dotted,"\gamma"] \\
& \mC_0 & \\
& &\mCW_2 \ar[ul, "\overline{\pi_0 t }"] \ar[from = uul, crossing over, "q_2"]\\
\end{tikzcd}\]

\noi commutes in $\cE$ and induces the map $\gamma$ on the right by the universal property of the coequalizer, $\pi_0 t : \mCW_1 \tensor[_t]{\times}{_s} \mCW_1 \to \mC_0$, in $\cE / \mC_0$. In particular we have that $q_2 = (q \times q) \gamma$. Finally we can see

\[ q_2 (\overline{\pi_0 q} \times \overline{\pi_1 q}) \gamma = (q \times q) \gamma = q_2\]
\noi and 
\[ \gamma (\overline{\pi_0 q} \times \overline{\pi_1 q}) \pi_0 t = \gamma \overline{\pi_0 q} t = \pi_0 t .\]

\noi By the universal property of $\mCW_2$ we have that 
\[ (\overline{\pi_0 q} {,} \overline{\pi_1 q}) \gamma = 1_{\mCW_2} \]
\noi and by the universal property of the coequalizer, $\mCW_1 \tensor[_t]{\times}{_s} \mCW_1$, in $\cE / \mC_0$
\[ \gamma (\overline{\pi_0 q} {,} \overline{\pi_1 q}) = 1_{\mCW_1 \tensor[_t]{\times}{_s} \mCW_1}. \]

\noi It follows that

\[ \mCW_2 \cong \mCW_1 \tensor[_t]{\times}{_s} \mCW_1.\]
\end{proof}

\begin{prop}\label{prop path pb is coequalizer}
The paths of composable arrows of length $n$ in $\mCW$ given by pullbacks 

\[ \mCW_1 \tensor[_t]{\times}{_s} ... \tensor[_t]{\times}{_s} \mCW_1 \] 
\noi of $n$ copies of $\mCW_1$ are coequalizers of the parallel pairs

\[\begin{tikzcd}[column sep = large]
\slb \tensor[_t]{\times}{_s} \dots \tensor[_t]{\times}{_s} \slb \rar[shift left, "p_0^n"] \rar[shift right, "p_1^n"'] & \spn \tensor[_t]{\times}{_s} \dots \tensor[_t]{\times}{_s} \spn 
\end{tikzcd},\]

\noi for every $n \geq 2$. 
\end{prop}
\begin{proof}
This proof follows by induction on the length of path of composable arrows. Use Lemma~\ref{lem composable pair pb is coequalizer} as the base case. Assume the result holds for paths of length $n$. Then the following diagram is a reflexive coequalizer, 

\[\begin{tikzcd}[column sep = large, row sep = large]
 \slb^n \rar[shift left, "p_0^n"] \rar[shift right,"p_1^n"'] 
& \spn^n \rar["q \times ... \times q"] \lar[]
& \mCW_1 \tensor[_t]{\times}{_s} ... \tensor[_t]{\times}{_s} \mCW_1
\end{tikzcd},\]

\noi where $\slb^n$ and $\spn^n$ are pullbacks defining paths of composable sailboats and spans of length $n$ respectively. On the right we have iterated pullbacks of $n$ copies of $\mCW_1$. By Lemma~\ref{lem underlying-structure functor from slice cat reflext coequalizers}, we can view these as reflexive coequalizers in $\mC_0$ using the induces source and target maps given by taking the left-most or right-most pullback projections and applying the source or target maps on $\slb, \spn$, and $\mCW_1$ respectively. Since $(\mC, W)$ is a candidate for internal fractions, the top row and right column in the following diagram are reflexive coequalizers in $\cE / \mC_0$,

\[\begin{tikzcd}[column sep = huge, row sep = huge]
\slb \tensor[_t]{\times}{_s} \slb^n \rar[shift left, "1 
\times p_0"] \rar[shift right,"1 \times p_1"'] 
& \slb \tensor[_t]{\times}{_s} \spn^n \dar[shift left, "p_0 \times 1"] \dar[shift right,"p_1 \times 1"'] \rar["1 \times q"] \lar[]
& \slb \tensor[_t]{\times}{_s} \mCW_n \dar[shift left, "p_0 \times 1"] \dar[shift right, "p_1 \times 1"' ] \\
& \spn \tensor[_t]{\times}{_s} \spn^n \rar["1 \times q"'] \uar[]
& \spn \tensor[_t]{\times}{_s} \mCW_n \dar["q \times 1 "] \uar[] \\
&& \mCW_1 \tensor[_t]{\times}{_s} \mCW_n
\end{tikzcd},\]

\noi and the middle column is a reflexive pair with common section $\varphi_s \times 1 : \spn \tensor[_t]{\times}{_s} \spn^n \to \slb \tensor[_t]{\times}{_s} \spn^n$. By Lemma~\ref{lem diagonal reflexive coequalizer lemma}, the diagonal is a (reflexive) coequalizer. Similarly to the proof of Lemma \ref{lem composable pair pb is coequalizer}

\[\mCW_{n+1} \cong \mCW_1 \tensor[_t]{\times}{_s} \mCW_n \cong\mCW_1 \tensor[_t]{\times}{_s} ( \mCW_1 \tensor[_t]{\times}{_s} ... \tensor[_t]{\times}{_s} \mCW_1) \]

\noi and we can drop the brackets on the right-hand side because taking pullbacks is associative up to canonical isomorphism.
\end{proof}

\subsection{Internal Fractions Axioms}\label{S the axioms (of internal fractions)}

Here we give an internal description of a weakened version of the axioms in \cite{GabZis} that allow us to define a category of fractions. The conditions for the class of arrows, $W$, which we intend to invert are weakened by not assuming that $W$ contains identities nor that it is closed under composition. Instead we assume that every object in $\mC_0$ is the target of some map in $W$, and that every composable pair in $W$ can be pre-composed by some arrow in $\mC_1$ to give a composite in $W$. The purpose of this, as shown in \cite{ThreeFsforBiCatsII}, is to allow for a smaller class of arrows to be inverted when constructing the category of fractions. In particular, when applied to the Grothendieck construction, this allows us to invert a canonical cleavage of the cartesian arrows in the category of elements rather than all of the cartesian arrows. In Section \ref{S IntFrc Applied to IntGroth} we see how the canonical cleavage of the cartesian arrows is easier to describe internally than all of the cartesian arrows. 

These axioms generally sound like, `for any diagram of a certain shape, there exist some filler arrows that make a larger diagram commute.' Internalizing these statements in a category of spaces like $\Top$ becomes an issue because, while we can form the objects representing such diagrams in $\Set$ and give them topologies, picking out the arrows to fill in the larger diagrams can rarely be done globally and continuously. For topological spaces one might work with effective descent covers to witness local information on a space that, when it satisfies a certain gluing condition, can be pasted together to give global information. On that note we ask that our category $\cE$ has a class of effective epimorphisms, $\cJ$, that are stable under pullback and composition. These allow us to witness the fractions axioms in $\mC$ locally and then construct global maps with them provided their coequalizer condition is satisfied. The coequalizer condition for these amounts to saying the structures we wish to define, like composition of spans for example, are well-defined. Stability under pullback and composition is required in order to witness multiple applications of the Internal Fractions Axioms. 

\begin{defn}\label{defn candidate context for internal fractions and covers}
 We say $(\cE, \cJ)$ is {\em a candidate context for internal fractions} if $\cJ$ is a class of effective epimorphisms that are stable under pullback and composition. We will refer to the elements of $\cJ$ as {\em covers}. 
\end{defn} 

\noi With a candidate context for internal fractions, $(\cE, \cJ)$, and a candidate for internal fractions, $(\mC, W)$, we can state the Internal Fractions Axioms below and ask whether $(\cE, \cJ)$ is a context for internal fractions for a given candidate for internal fractions, $(\mC, W)$, as defined in Definition~\ref{def candidate for internal fractions}.

\begin{defn}[Internal Fractions Axioms]\label{def Internal Fractions Axioms}
Let $(\cE, \cJ)$ be a candidate context for internal fractions, as in Definition~\ref{defn candidate context for internal fractions and covers}, and let $(\mC, W)$ be a candidate for internal fractions in $\cE$, as in Definition~\ref{def candidate for internal fractions}. We say $(\mC, W)$ {\em satisfies the internal (right) fractions axioms} or {\em admits an internal category of fractions} (with respect to $w : W \to \mC_1$) if the following conditions hold. 

\begin{enumerate}[leftmargin = *, label = \textbf{In.Frc(\arabic*)}]
\item The identity map $1_{\mC_0} : \mC_0 \to \mC_0$, admits a lift along $w t : W \to \mC_0$. 
\begin{center}
\begin{tikzcd}[]
& W \dar["wt"] \\
\mC_0 \ar[ur, dotted, "\tau"] \rar[equals] & \mC_0
\end{tikzcd}
\end{center}\

\item There exists a cover \begin{tikzcd} U \rar["/" marking, "u" near start] & W \tensor[_t]{\times}{_s} W \end{tikzcd} that admits a lift, $\omega : U \to W_\circ$, along $\pi_0 \pi_{12} : W_\circ \to W \tensor[_{wt}]{\times}{_{ws}} W$.
\begin{center}
\begin{tikzcd}[]
& W_\circ \dar["\pi_0 \pi_{12}"] \\
U \ar[ur, dotted, "\omega"] \rar["/" marking, "u"' near end ] & W \tensor[_{wt}]{\times}{_{ws}} W
\end{tikzcd}
\end{center}\

\item 
There exists a cover \begin{tikzcd} U \rar["/" marking, "u" near start] & \mC_1 \tensor[_t]{\times}{_{wt}} W\end{tikzcd} that admits a lift, $\theta : U \to W_\square$, along $(\pi_0 \pi_1 , \pi_1 \pi_1) : W_\square \to \mC_1 \tensor[_{t}]{\times}{_{wt}} W$.
\begin{center}
\begin{tikzcd}[]
& W_\square \dar["(\pi_0 \pi_1{ , }\pi_1 \pi_1)"] \\
U \ar[ur, dotted, "\theta"] \rar["/" marking, "u"' near end] & \mC_1 \tensor[_{t}]{\times}{_{wt}} W
\end{tikzcd}
\end{center}\

\item There exists a cover \begin{tikzcd} U \rar["/" marking, "u" near start] & \cP_{cq}\end{tikzcd} that admits a lift, $\delta : U \to \cP(\mC)$, along $\pi_1 : \cP(\mC) \to \cP_{cq}(\mC)$. 
\begin{center}
\begin{tikzcd}[]
& \cP(\mC) \dar["\pi_1"] \\
U \ar[ur, dotted, "\delta"] \rar["/" marking, "u"' near end] & \cP_{cq}(\mC)
\end{tikzcd}
\end{center}\
\end{enumerate}
\end{defn}\

The lifts in the axioms above represent the existence of fillers for diagrams in $\mC$ represented by codomains of the covers. In order to prove our composition is well-defined, associative, and satisfies the identity laws, we want to have a notion of base change. The following lemma shows how this works with stability of covers under pullback. 

\begin{lem}\label{lem extending Int.Frc. covers}
If $u : U \nrightarrow B$ is a cover that admits a lift along $f : A \to B$, then for any map $ g : X \to B$, there exists a cover, $u' : U' \nrightarrow X$, such that the diagram

\begin{center}
  \begin{tikzcd}
  & & A \dar["f"] \\
  U' \ar[urr, bend left, dotted, "\ell"] \rar["/" marking, "u'"' near end] & X \rar["g"'] & B
  \end{tikzcd}
\end{center}

\noi commutes in $\cE$. 
\end{lem}
\begin{proof}
Since covers are stable under pullback, taking the pullback of the cover $u : U \nrightarrow B$ along the map $g : X \to B$ gives a cover $u' : U' \nrightarrow X$. Then the desired lift $\ell : U' \to A_i$ is given by post-composing the pullback projection with the lift $\ell$. This is seen in the following commuting diagram:

\begin{center}
   \begin{tikzcd}
  & & A_i \dar["f_i"] \\
  & U \ar[ur, bend left, dotted, ""] \rar["/" marking, "u" near start] & B_i \\
  U' \rar["/" marking, "u'"' near end] \ar[ur, "\pi"] \ar[uurr, bend left, dotted, "\ell"] & X \ar[ur, "g"'] & 
  \end{tikzcd}
\end{center}
\end{proof}

\noi Lemma~\ref{lem extending Int.Frc. covers} allows us to apply the axioms \textbf{In.Frc(1)} - \textbf{In.Frc(4)}, in Definition~\ref{def Internal Fractions Axioms} a little more broadly. To simplify notation in later proofs we will typically suppress the pullbacks in Lemma~\ref{lem extending Int.Frc. covers} and just write $u : U \to X$ for the cover $u': U' \to X$ with lift $\ell$.

\section{Defining the Internal Category of Fractions}\label{S Defining mCW, the internal cat of fracs}

In this section we define structure for an internal category of fractions, $\mCW$, for a pair $(\mC, W)$ that admits an internal category in a context for internal fractions, $(\cE, \cJ)$, as in Definitions \ref{def candidate for internal fractions} and \ref{def Internal Fractions Axioms}. Before we begin we should mention that the proofs in this section and Section \ref{S Assoc and Id Laws Int Groth}, using axioms \textbf{In.Frc(1)} - \textbf{In.Frc(4)}, can be difficult to follow so we have labeled and colour coded diagrams in a particular way. The diagrams labeled with capital letters, $(A), (B), (C), ...$, are representing diagrams in $\mC$ which contain the data of the usual proofs for the case $\cE = \Set$. The `cover diagrams' labeled with stars, $(\star), (\star \star), ...$, describe the corresponding applications of \textbf{In.Frc(1)} - \textbf{In.Frc(4)} whose covers and lifts allow us to witness the arrows represented by the diagrams $(A), (B), (C), ...$. Any reference to these diagrams or equations should be interpreted `locally' within the scope of the proof in which the reference occurs. We use covers to define the composition structure for the internal category of fractions and we also use the fact that they are epimorphisms to show other maps out of $\spn \tensor[_t]{\times}{_s} \spn $ are equal by showing the can be equalized by covers or their composites with other epimorphisms (such as coequalizer maps). 

The composition, source, target, and identity notation for internal categories is being overloaded, as well as notation for pullback and product projections. We have included colours in both kinds of diagrams mentioned above as well as the corresponding equations for the maps in $\cE$ of the star-labeled cover diagrams. The reference scope between these diagrams is contained within respective lemmas and propositions so there should be no issue with re-using labeling and colour patterns for diagrams in different lemma and proposition representing these two things similarly.

The object of objects is $\mC_0$, and the object of arrows is the coequalizer from Lemma~\ref{lem composable pair pb is coequalizer}: 

\[\mCW_0 = \mC_0 ,\qquad \qquad \left( \begin{tikzcd}[]
\slb \rar[shift left, "p_0"] \rar[shift right, "p_1"'] & \spn \rar[two heads, "q"'] & \mCW_1
\end{tikzcd} \right) \]

\noi The source and target maps $s,t : \mCW_1 \to \mCW_0$ are defined by the universal property of $\mCW_1$ as seen in Lemma~\ref{lem defining source and target of mCW}: 

\[
\begin{tikzcd}[]
& \mC_1 \rar["t"] & \mC_0 \\
\slb \rar[r,shift left, "p_0"] \rar[r,shift right, "p_1"'] & \spn \rar[r, two heads, "q"] \ar[d, "\pi_0"'] \uar["\pi_1"] & \mCW_1 \dar[dotted, "s"] \uar[dotted, "t"']\\
& W \rar["wt"'] & \mC_0 
\end{tikzcd}
\]

To define the identity map, $e : \mCW_0 \to \mCW_1 $, it helps to think about the case when $\cE = \Set$ for a moment. In this case, the identity for an object, $a$, in a category of fractions is represented by any span with two of the same legs in $W$: 

\begin{center}
\begin{tikzcd}[]
a & b \lar["\circ" marking, "u"'] \rar["\circ" marking, "u"] & a
\end{tikzcd}
\end{center}

\noi By \textbf{In.Frc(1)}, we have a section, $\alpha$, of the target map $w t : W \to \mC_0$ we can use to define the identity structure map. Take the unique span $\sigma_\alpha = (\alpha, \alpha w) : \mC_0 \to \spn$ induced by $\alpha$ and $\alpha w$ and post-compose it with the coequalizer map to define the identity map for $\mCW$. 

\begin{center}
\begin{tikzcd}[]
\mC_0 \dar[equals] \rar["\sigma_\alpha"]& \spn \dar[two heads, "q"]\\
\mCW_0 \rar[dotted, "e"']& \mCW_1
\end{tikzcd}
\end{center}\

\noi Now we will prove that this definition does not depend on the choice of section, $\alpha : \mC_0 \to W$, of $wt : W \to \mC_0$. 

\begin{prop}\label{Prop identity def doesnt depend on choice of section of wt}
The identity map, $e : \mCW_0 \to \mCW_1 $ does not depend on the section, $\alpha : \mC_0 \to W$. 
\end{prop}
\begin{proof}
Let $\alpha$ and $\beta$ be two sections of $w t$, and let 

\[ \sigma_\alpha = ( \alpha, \alpha w) \qquad , \qquad \sigma_\beta = (\beta , \beta w) \]

\noi be two spans $\mC_0 \to \spn$. We will show that $\sigma_\alpha q = \sigma_\beta q$ by finding a cover, $u : U \to \mC_0$, to witness a family of intermediate spans $\sigma_{\alpha \beta} : U \to \spn$ for which 

\[ u \sigma_\alpha q = \sigma_{\alpha \beta} q = u \sigma_\beta q.\]

\noi Since $u$ is an epimorphism, this will imply $\sigma_\alpha q = \sigma_\beta q$. Notice that $\alpha (w t) = 1_{\mC_0} = \beta (w t)$ so there is an induced pairing map $(\alpha w , \beta) : \mC_0 \to \csp$ which is a section of both $\pi_0 t$ and $\pi_1 w t$. By \textbf{In.Frc(3)}, there exists a cover, $u_1 : U_0 \to \mC_0$, and a lift, $\theta_{\alpha \beta}$ of $u_1 (\alpha w , \beta) \mC_0 \to \csp$, along the cospan projection, $W_\square \to \csp$, in the bottom right of the following diagram: 

\[\label{dgm covers for independence of id sections lemma}
\begin{tikzcd}[column sep = large, row sep = large]
& W_\circ \rar["(\pi_0 \pi_1 {,} \pi_0 \pi_2)"]
& W \tensor[_t]{\times}{_s} W 
& 
\\
& U
\rar["/" marking, "u_0" near start]
\uar[dotted, "\omega_{\alpha \beta}"] 
& 
U_0
\rar["/" marking, "u_1" near start]
\dar[dotted, "\theta_{\alpha\beta}"] 
\uar["(\theta_{\alpha \beta} \pi_0 \pi_0 {,} u_1 \alpha)"']
& \mC_0 \dar["(\alpha w {,} \beta)"]
\\
& 
& W_\square \rar["(\pi_0 \pi_1 w {,} \pi_1 \pi_1)"']
& \csp
\end{tikzcd} \tag{$\star$}
\]

\noi By definition of $W_\square$, we have that
\[ \theta_{\alpha \beta} \pi_0 \pi_0 w t = \theta_{\alpha \beta} \pi_0 \pi_1 s = u_1 (\alpha w , \beta) \pi_0 s = u_1 \alpha w s \]

\noi inducing the map $U_0 \to W \tensor[_t]{\times}{_s} W $ in the diagram abovel. By \textbf{In.Frc(2)}, there exists a cover, $u_0 : U \to U_0$, and a lift, $\omega_{\alpha \beta} : U \to W_\circ$ to make the square in the upper left of the diagram above commute. When $\cE = \Set$ the process can be represented by the following picture with labels of arrows corresponding to the arrows in the diagram above that would be witnessing those below internally to $\mC$. 

\[\label{Fig id section independence Ore + weak comp}
\begin{tikzcd}[row sep = large, column sep = large]
& & 
\cdot \dar[dotted, "\omega_{\alpha \beta} \pi_0 \pi_0 \pi_0"] \ar[dddll, bend right = 40, "\circ" marking , "\omega_{\alpha \beta} \pi_1"' ] \ar[dddrr, bend left = 40 ] 
& & \\
& &
\cdot \ar[dr, dotted, "u_0 \theta_{\alpha \beta}\pi_1 \pi_0"] \ar[ "\circ" marking, dl, dotted,"u_0 \theta_{\alpha \beta}\pi_0 \pi_0 "'] 
& & \\ 
&\cdot \ar[dr, dotted, "u \alpha w"] \ar[dl, dotted, "\circ" marking, "u \alpha"'] 
&&\cdot \ar[dr, dotted, "u \beta w"] \ar[dl, dotted, "\circ" marking, "u \beta"'] 
& \\
\cdot & & 
\cdot  & & 
\cdot 
\end{tikzcd} \tag{A}
\]

\noi where, since covers are stable under composition, we let $u = u_0 u_1$ denote the composite cover of $u_0$ and $u_1$. Note that by definition of $W_\square$ (ie. the Ore condition)

\begin{align*} 
\omega_{\alpha \beta} \pi_1 
&=(\omega_{\alpha \beta} \pi_0 \pi_0 \pi_0, \ 
u_0 \theta_{\alpha \beta} \pi_0 \pi_0 w,\ 
u \alpha w ) c \\
&=\big( \omega_{\alpha \beta} \pi_0 \pi_0 \pi_0 ,\ 
(u_0 \theta_{\alpha \beta} \pi_0 \pi_0w , \
u \alpha w)c \big) c \\
&= 
\big( \omega_{\alpha \beta} \pi_0 \pi_0 \pi_0 ,\ 
(u_0 \theta_{\alpha \beta} \pi_0\pi_0 w, \
u_0 \theta_{\alpha \beta} \pi_0\pi_1 )c \big) c \\
&= 
\big( \omega_{\alpha \beta} \pi_0 \pi_0 \pi_0 ,\ 
(u_0 \theta_{\alpha \beta} \pi_1\pi_0 , \
u_0 \theta_{\alpha \beta} \pi_1\pi_1 w )c \big) c \\
&= \big( \omega_{\alpha \beta} \pi_0 \pi_0 \pi_0 ,\ 
(u_0 \theta_{\alpha \beta} \pi_1\pi_0 , \
 u \beta \pi_1\pi_1 w )c \big) c \\
 &= ( \omega_{\alpha \beta} \pi_0 \pi_0 \pi_0 ,\ 
u_0 \theta_{\alpha \beta} \pi_1\pi_0 , \
 u \beta \pi_1\pi_1 w )c.
\end{align*}

\noi So the outer span in Diagram (\ref{Fig id section independence Ore + weak comp}) can be represented by the map, $\sigma_{\alpha \beta} : U \to \spn$, defined by the pairing map

\[ \sigma_{\alpha \beta} = \big( \omega_{\alpha \beta} \pi_1 , \ ( \omega_{\alpha \beta} \pi_0 \pi_0 \pi_0 , \
u_0 \theta_{\alpha \beta} \pi_1\pi_0 , \
 u \beta \pi_1\pi_1 w )c \big) \] 

\noi whose right component is the composite 

\[ \begin{tikzcd}[column sep = huge]
U \ar[drrr, "\sigma_{\alpha \beta} \pi_1"'] \rar[rrr,"( \omega_{\alpha \beta} \pi_0 \pi_0 \pi_0 {,} \
u_0 \theta_{\alpha \beta} \pi_1\pi_0 {,} \
 u \beta \pi_1\pi_1 w )c"] 
 &&& \mC_3 \dar["c"] \\
 &&& \mC_1
\end{tikzcd}.\]

\noi This composite represents the internal triple composition in $\mC$ of the arrows represented on the right side of Diagram (\ref{Fig id section independence Ore + weak comp}) above. Now let

\[ \mu_{\alpha} = (\omega_{\alpha \beta} \pi_0 \pi_0 \pi_0, \ 
u_0 \theta_{\alpha \beta} \pi_0 \pi_0)c \qquad , \qquad \mu_\beta = ( \omega_{\alpha \beta} \pi_0 \pi_0 \pi_0 ,\ u_0 \theta_{\alpha \beta} \pi_1\pi_0)c \]

\noi and notice the map $\varphi_\alpha: U \to \slb$ given by 

\begin{align*}
\varphi_\alpha = \big( (( \mu_\alpha , \ u \alpha) , \ \omega_{\alpha \beta} \pi_1) , \ u \alpha w ) 
\end{align*}

\noi is well-defined by associativity of composition and the definitions above. Similarly we can see 

\begin{align*}
   \varphi_\alpha p_0
   &= \varphi_\alpha (\pi_0 \pi_0 \pi_1 , \pi_1) = (u \alpha , u \alpha w) \\
   & = u \sigma_\alpha
\end{align*} 
\noi and 
\begin{align*}
  \varphi_\alpha p_1
  &= \varphi_\alpha (\pi_0 \pi_1 , \ (\pi_0 \pi_0 \pi_0 , \pi_1)c) \\
  &= \big( \omega_{\alpha \beta} \pi_1 , \ (\mu_\alpha , u \alpha w )c \big) \\
  &= (\sigma_{\alpha \beta} \pi_0 , \sigma_{\alpha \beta} \pi_1) \\
  &= \sigma_{\alpha \beta}.
\end{align*} 

\noi On the other hand we have another map $\varphi_\beta: U \to \slb$ given by 

\[ \varphi_\beta = \big( ( ( \mu_\beta , u \beta) , \ \omega_{\alpha \beta} \pi_1 ) , \ u \beta w)\]

\noi for which 

\[ \varphi_\beta p_0 = \varphi_\beta (\pi_0 \pi_0 \pi_1 , \pi_1) = (u \beta , u \beta w) = u \sigma_\beta \] 
\noi and 
\[ \varphi_\beta p_1 = \varphi_\beta (\pi_0 \pi_1 , \ (\pi_0 \pi_0 \pi_0 , \pi_1)c) 
= \big( \omega_{\alpha \beta} \pi_1 , \ (\mu_\beta , u \beta w )c \big) = (\sigma_{\alpha \beta} \pi_0 , \sigma_{\alpha \beta} \pi_1) = \sigma_{\alpha \beta}. \]

\noi From here we can conclude that 

\begin{align*}
u \sigma_\alpha q = \varphi_\alpha p_0 	q = \varphi_\alpha p_1 q = \sigma_{\alpha \beta} q = \varphi_\beta p_1 q = \varphi_\beta p_0 q = u \sigma_\beta q
\end{align*}

\noi and since $u$ is epic 

\[ \sigma_\alpha q = \sigma_\beta q . \] 
\end{proof}

\noi We can immediately see the identity structure map, $\sigma_\alpha q : \mC_0 \to \mCW_1$, is a section of both the source and target maps:

\begin{align*}
\sigma_\alpha q s 
 &= (\alpha , \alpha w) q s 
 &\sigma_\alpha q t' 
 &=(\alpha , \alpha w) q t \\
 &= (\alpha , \alpha w) \hat{s} 
 & 
 &=(\alpha , \alpha w) \hat{t} \\
 &= (\alpha , \alpha w) \pi_0 w t 
 & 
 &=(\alpha , \alpha w) \pi_1 t \\
 &= \alpha w t & &=\alpha w t \\
 &= 1_{\mC_0} & &= 1_{\mC_0} 
 \end{align*}\

\noi The composition structure map needs to be defined out of the following pullback,

\begin{center}
\begin{tikzcd}[]
\mCW_1 \tensor[_t]{\times}{_s} \mCW_1 
\arrow[dr, phantom, "\usebox\pullback" , very near start, color=black] 
\rar["\pi_1"] 
\dar["\pi_0"']
& \mCW_1 \dar["s"] \\
\mCW_1 \rar["t"'] & \mC_0 
\end{tikzcd},
\end{center} \

\noi in order for $\mCW$ to be an internal category. Since $(\mC , W)$ is a candidate for internal fractions, by Lemma~\ref{lem composable pair pb is coequalizer}, this pullback is also the coequalizer of the parallel pair $p_0^2, p_1^2 : \slb \tensor[_t]{\times}{_s} \slb \to \spn \tensor[_t]{\times}{_s} \spn$. When $\cE = \Set$ we can see how $p_0^2$ maps a pair of composable sailboats

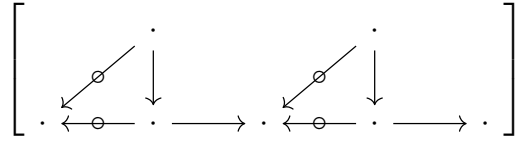
\begin{center}
$\left[
\begin{tikzcd}[]
& \cdot \dar \ar[dl, "\circ" marking] & & \cdot \dar \ar[dl, "\circ" marking] & \\
\cdot & \cdot \lar["\circ" marking] \rar & \cdot & \cdot \lar["\circ" marking] \rar & \cdot 
\end{tikzcd} 
\right]$
\end{center}
\noi to the pair of composable spans along the bottom 
\begin{center}
$\left[\begin{tikzcd}[]
\cdot & \cdot \lar["\circ" marking] \rar & \cdot & \cdot \lar["\circ" marking] \rar & \cdot 
\end{tikzcd}\right]$,
\end{center}\

\noi and how $p_1^2$ maps a pair of composable sailboats

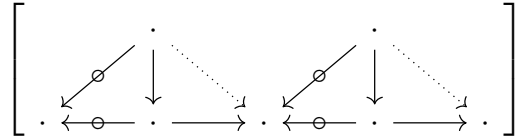
\begin{center}
$\left[\begin{tikzcd}[]
&\cdot \dar \ar[dl, "\circ" marking] \ar[dr, dotted] & &\cdot \dar \ar[dl, "\circ" marking] \ar[dr, dotted] & \\
\cdot & \cdot \lar["\circ" marking] \rar & \cdot & \cdot \lar["\circ" marking] \rar & \cdot 
\end{tikzcd} \right]$
\end{center}
\noi to the pair of composable spans 
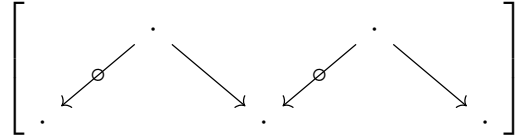
\begin{center}
$\left[\begin{tikzcd}[]
&\cdot \ar[dr] \ar[dl, "\circ" marking] &&\cdot \ar[dr] \ar[dl, "\circ" marking] & \\
\cdot & & \cdot  & & \cdot 
\end{tikzcd}\right]$
\end{center}

\noi along the top. 

The first thing to do is to use the Internal Fractions Axioms to obtain a cover, $u : U \to \spn \tensor[_t]{\times}{_s} \spn$, of composable spans which witnesses the span composition operation in the form of a map

\begin{center}
 \begin{tikzcd}[]
U \rar["\sigma_\circ "] & \spn 
\end{tikzcd}.
\end{center}

\noi When $\cE = \Set$, span composition for fractions is defined by applying the right Ore condition followed by the weak-composition axiom to get a span whose left leg is in $W$, as shown in the following figure. 

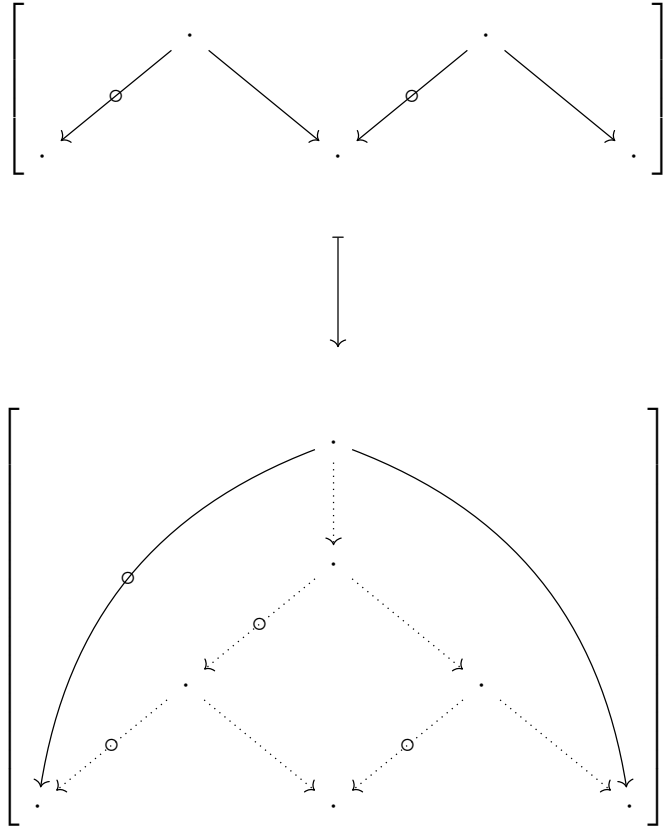
\begin{center}
$\left[\begin{tikzcd}[column sep = large, row sep = large]
&\cdot \ar[dr] \ar[dl, "\circ" marking] &&\cdot \ar[dr] \ar[dl, "\circ" marking] & \\
\cdot & & \cdot  & & \cdot 
\end{tikzcd}\right]$
\end{center}
\begin{center}
\begin{tikzcd}[]
\ar[dd, mapsto, ""]  \\
 \\
\
\end{tikzcd}
\end{center}
\begin{center}
$\left[\begin{tikzcd}[column sep = large, row sep = large]
& & \cdot \dar[dotted] \ar[dddll, bend right, "\circ" marking] \ar[dddrr, bend left] & & \\
& & \cdot \ar[dr, dotted] \ar[dl, dotted, "\circ" marking] & & \\ 
&\cdot \ar[dr, dotted ] \ar[dl, dotted, "\circ" marking] &&\cdot \ar[dr, dotted] \ar[dl, dotted, "\circ" marking] & \\
\cdot & & \cdot  & & \cdot 
\end{tikzcd}\right]$.
\end{center}

\noi To internalize this we construct a diagram of covers below, starting with a map, $\spn \tensor[_t]{\times}{_s} \spn \to \csp$, picking out a cospan whose right leg is in $W$ from a pair of composable spans and apply \ore \ along with Lemma~\ref{lem extending Int.Frc. covers} to get the cover $u_1 : U_0 \to \spn \tensor[_{t'}]{\times}{_{s'}} \spn$ that makes the bottom right square below commute. Next consider the map which picks out the composable pair in $W$ from the Ore-square filler and the left leg of the first span in the original composable pair and apply \wcomp\ along with Lemma~\ref{lem extending Int.Frc. covers} to get the cover $u_0 : U \to U_0$ that makes the top left square below commute. 

\begin{center}
\begin{tikzcd}[column sep = huge ]
 W_\circ \rar[rr, "(\pi_0 \pi_1 {,} \pi_0 \pi_2)"] 
&& W \times_{\mC_0} W 
& \\
U 
\dar["\sigma_\circ"]
\uar["\omega"]
\rar[rr, "/" marking , "u_0" near start]
 &
 & U_0 
 \dar["\theta"']
 \uar["(\theta \pi_0 \pi_0 {, } u_1 \pi_0 \pi_0) "'] 
 \rar["/" marking , "u_1" near start] 
 & \spn \tensor[_t]{\times}{_s} \spn
 \dar["( \pi_0 \pi_1 {, } \pi_1 \pi_0)"] 
 \\
 \spn 
 &
 & W_\square 
 \rar["(\pi_0 \pi_1 {,} \pi_1 \pi_1)"'] 
 & \csp
\end{tikzcd}
\end{center}\

\noi Since covers are stable under composition we can take $u = u_0 u_1 : U \to \spn \tensor[_t]{\times}{_s} \spn$ as our cover, and define $\sigma_\circ : U \to \spn$
 by the pairing map 
\[ \sigma_\circ = \big( \omega \pi_1 , \ (\omega \pi_0 \pi_0 \pi_0 , \ u_0 \theta \pi_1 \pi_0 , \ u \pi_1 \pi_1)c \big). \]\

\noi We claim the construction represented by $\sigma_\circ$ is well-defined on equivalence classes in the sense that for any two choices of fillers for the Ore-square and weak-composition conditions above, there exists a sailboat relating them. Internally this is translated as independence of the choice of filler-arrows in the lifts, $\theta$ and $\sigma$, and is proven in Lemma~\ref{lem defining c'} by finding a cover $\tilde{u} : \tilde{U} \to \ker(u)$ and two families of sailboats, 

\[ \varphi_0 : \tilde{U} \to \slb \qquad , \qquad \varphi_1 : \tilde{U} \to \slb,\]

 \noi which witness commutativity of the square

\begin{center}
\begin{tikzcd}[]
\ker{u} \dar["\pi_0"'] \rar["\pi_1"] & U \dar["\sigma_\circ q"] \\
U \rar["\sigma_\circ q"'] & \mCW_1
\end{tikzcd}
\end{center}

\noi in $\cE$. The proof is rather long and technical but full of colourful pictures. The cover $u$ is an effective epimorphism so it is the coequalizer of its kernel pair and in Lemma~\ref{lem actual definition of c'} we use this universal property to induce a composition map on spans

\[ c' : \spn \tensor[_t]{\times}{_s} \spn \to \mCW_1\]

\noi such that the square

\begin{center}
\begin{tikzcd}[]
U \dar["/" marking, "u"' near end] \ar[r, "\sigma_\circ"] & \spn \dar[two heads, "q"] \\
\spn \tensor[_t]{\times}{_s} \spn \rar[dotted, "c'"'] & \mCW_1
\end{tikzcd}
\end{center}

\noi commutes in $\cE$. Finally, in the proof of Proposition~\ref{prop composition of spans is well-defined} we show how to find an even finer cover $\hat{u} : \hat{U} \to \slb \tensor[_t]{\times}{_s} \slb$ witnessing that the map $c'$ respects the sailboat relation. More precisely, the proof of Proposition~\ref{prop composition of spans is well-defined}, shows how to construct sailboats 

\[ \varphi_i : \hat{U} \to \slb \] 

\noi witnessing equivalences between the spans

\[ \sigma_j : \hat{U} \to \spn \]

\noi so that

\[ \hat{u} p_0 c' = \hat{\pi}_0 \sigma_0 q = \varphi_0 p_0 q = \varphi_0 p_1 q = ... = \varphi_4 p_0 = \hat{\pi}_1 \sigma_0 q = \hat{u} p_1 c' .\]\

\noi Then since $\hat{u}$ is an epimorphism, we can conlude that $p_0 c' = p_1 c'$ and induce the composition map, $c : \mCW_2 \to \mCW_1$. For the rest of this section we prove the lemmas and propositions we required to define composition. 

\begin{lem}\label{lem defining c'}
There is a cover $\tilde{u} : \tilde{U} \to \ker(u)$, together with two maps 

\[ \varphi_0 : \tilde{U} \to \slb \qquad , \qquad \varphi_1 : \tilde{U} \to \slb\]

 \noi, which witness that the composite $\sigma_\circ q$ coequalizes the kernel pair of $u : U \to \spn \tensor[_t
]{\times}{_s} \spn$. That is, the diagram

\begin{center}
\begin{tikzcd}[]
\ker{u} \dar["\pi_0"'] \rar["\pi_1"] & U \dar["\sigma_\circ q"] \\
U \rar["\sigma_\circ q"'] & \mCW_1
\end{tikzcd}
\end{center}

\noi commutes. 
\end{lem}

\begin{proof}
\

We are essentially showing that any two choices of fillers above represent equivalent spans. Classically this can be done with the data in the following sketch.

\begin{itemize}
  \item Take two composites (pictured in \textcolor{orange}{orange} and \textcolor{teal}{teal} below) of a single pair of composable spans
  \item \textcolor{olive}{Apply the right Ore condition} (corresponding to \textbf{In.Frc(3)}) on the cospan determined by the left legs of the composites
  \item \textcolor{cyan}{Apply the zippering axiom} (corresponding to \textbf{In.Frc(4)}) to the parallel pair which, after post-composing with the left leg of the first span in the original composable pair, gives the two sides of the commuting Ore-square 
   \item \textcolor{cyan}{Apply zippering} (corresponding to \textbf{In.Frc(4)}) to the parallel pair which is coequalized after post-composing with the left leg of the second span in the original composable pair 
   \item \textcolor{violet}{Apply weak-composition} (corresponding to \textbf{In.Frc(2)}) three times to obtain a span whose left leg is in $W$. 
\end{itemize}

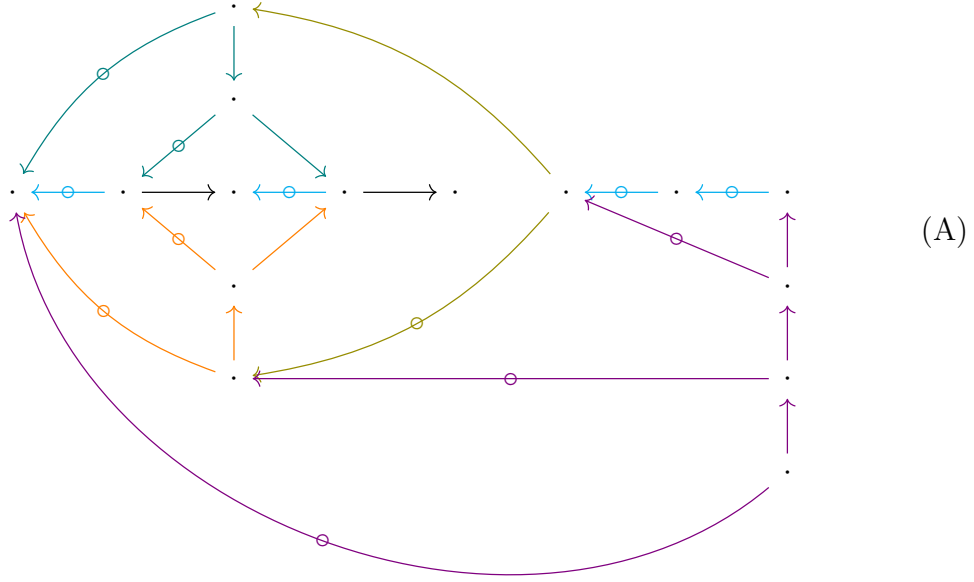
\[\label{dgm witnessing arrows for span comp}
\begin{tikzcd}
	&& \cdot \\
	&& \cdot \\
	\cdot & \cdot & \cdot & \cdot & \cdot & \cdot & \cdot & \cdot \\
	&& \cdot &&&&& \cdot \\
	&& \cdot &&&&& \cdot \\
	&&&&&&& \cdot
	\arrow[from=3-2, to=3-1, "\circ" marking, cyan]
	\arrow[from=3-2, to=3-3]
	\arrow[from=3-4, to=3-3, "\circ" marking, cyan]
	\arrow[from=3-4, to=3-5]
	\arrow[from=2-3, to=3-2, "\circ" marking, color= teal]
	\arrow[from=2-3, to=3-4, color= teal]
	\arrow[from=4-3, to=3-4, color= orange]
	\arrow[from=4-3, to=3-2, "\circ" marking, color= orange]
	\arrow[from=5-3, to=4-3, color= orange]
	\arrow[from=5-3, to=3-1, bend left = 20, "\circ" marking, color= orange]
	\arrow[from=1-3, to=2-3, color= teal]
	\arrow[from=1-3, to=3-1, bend right = 20, "\circ" marking, color= teal]
	\arrow[from=3-6, to=5-3, bend left = 20, "\circ" marking, color = olive]
	\arrow[from=3-6, to=1-3, bend right = 20, color= olive]
	\arrow[from=3-7, to=3-6, "\circ" marking, cyan]
	\arrow[from=3-8, to=3-7, "\circ" marking,cyan]
	\arrow[from=4-8, to=3-8, violet]
	\arrow[from=4-8, to=3-6, "\circ" marking, violet]
	\arrow[from=5-8, to=4-8, violet]
	\arrow[from=6-8, to=5-8, violet]
	\arrow[from=5-8, to=5-3, "\circ" marking, violet]
	\arrow[from=6-8, to=3-1, bend left = 60, "\circ" marking, violet]
\end{tikzcd} \tag{A}\]

\noi To translate this internally to $\cE$, first note that the definition of $\sigma_\circ$ implies $\pi_0 \sigma_\circ$ and $\pi_1 \sigma_\circ$ have the same source.

\[\pi_0 \sigma_\circ s' = \pi_0 u \pi_0 \pi_0 w t = \pi_1 u \pi_0 \pi_0 w t = \pi_1 \sigma_\circ s' \]

\noi Now take covers to witness the application of the axioms above in that order as follows. First take the Ore-square and zippering lifts given by \textbf{In.Frc(3)} and \textbf{In.Frc(4)} respectively: 

\[\label{dgm span composition covers (Ore + zippering)}
\begin{tikzcd}[]
&\cP(\mC)\rar["\pi_1"]
& \cP_{cq}(\mC)
& 
\\
\tilde{U}_2 
\rar["/" marking, "\tilde{u}_3" near start]
\dar[cyan, "\tilde{\delta}_1"'] 
&
\tilde{U}_3 
\rar["/" marking, "\tilde{u}_4" near start]
\uar[cyan, "\tilde{\delta}_0"] 
\dar["\delta_1"] 
& 
\tilde{U}_4 
\rar["/" marking, "\tilde{u}_5" near start]
\uar["\delta_0"'] 
\dar[olive, "\tilde{\theta}"']
& \ker u \dar["(\textcolor{orange}{\pi_0 \sigma_\circ} \pi_0 w {, } \textcolor{teal}{ \pi_1 \sigma_\circ} \pi_0)"]
\\
\cP(\mC)
\rar["\pi_1"']
& \cP_{cq}(\mC)
& W_\square \rar["(\pi_0 \pi_1 {,} \pi_1 \pi_1)"'] 
& \csp
\end{tikzcd}, \tag{$\star$}
\]

\noi and then take the weak-composition lifts. 

\[\label{dgm span composition covers (weak comp part)}
\begin{tikzcd}[]
 W_\circ \rar["\pi_0 \pi_{12}"] 
& W\tensor[_{wt}]{\times}{_{ws}}W 
& W_\circ \rar["\pi_0 \pi_{12}"] 
& W\tensor[_{wt}]{\times}{_{ws}}W \\
\tilde{U}
\rar["/" marking, "\tilde{u}_0" near start]
\uar[violet,"\tilde{\omega}_0"]
& 
\tilde{U}_0 \rar["/" marking, "\tilde{u}_1" near start]
\uar["\omega_0"'] 
\dar[violet,"\tilde{\omega}_1"'] 
& 
\tilde{U}_1 \rar["/" marking, "\tilde{u}_2" near start]
\uar[violet,"\tilde{\omega}_2"] 
\dar["\omega_1"] 
& 
\tilde{U}_2 
\uar["\omega_2"'] 
\\
& W_\circ \rar["\pi_0 \pi_{12}"'] 
& W\tensor[_{wt}]{\times}{_{ws}}W
&
\end{tikzcd} \tag{$\star \star$} 
\]

\noi The first vertical map representing cospans with right legs in $W$, seen on the right-hand side of Diagram (\ref{dgm span composition covers (Ore + zippering)}), is witnessing the following cospan of Diagram (\ref{dgm witnessing arrows for span comp}). 

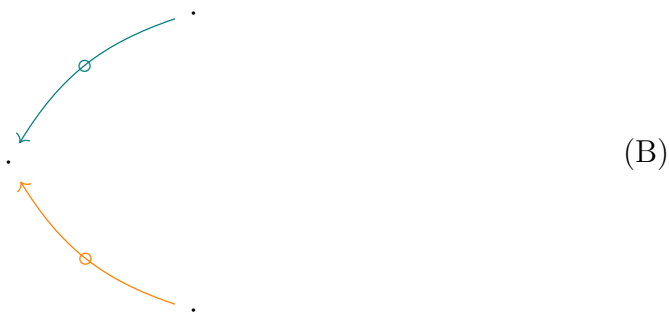
\[\label{dgm cospan arrows for witnessing Ore-square in span comp}
\begin{tikzcd}
	&& \cdot \\
	&& \\
	\cdot & & \\
	&& \\
	&& \cdot 
	\arrow[from=5-3, to=3-1, bend left = 20, "\circ" marking, color= orange]
	\arrow[from=1-3, to=3-1, bend right = 20, "\circ" marking, color= teal]
\end{tikzcd} \tag{B}\]

\noi Axiom \textbf{In.Frc(3)} along with Lemma~\ref{lem extending Int.Frc. covers} then give the cover and lift 

\[ \begin{tikzcd}
\tilde{U}_4 \dar[olive, "\tilde{\theta}"'] \rar["\scriptstyle{/}" marking, "\tilde{u}_5" near start] & \ker u\\
W_\square &
\end{tikzcd}\]

\noi that make the bottom right square in Diagram (\ref{dgm span composition covers (Ore + zippering)}) commute. The map $\delta_0$ is induced by a map, $\delta_0' : \tilde{U} \to P(\mC) \tensor[_t]{\times}{_{ws}} W$, which can be found by expanding both sides of the commuting Ore square equation 

\[ \begin{tikzcd}[column sep = huge]
\tilde{U}_4 \rar[rrrrr, "
(\textcolor{olive}{\tilde{\theta}} \pi_0 \pi_0 w {,} \ 
\tilde{u}_5 \textcolor{orange}{\pi_0 \sigma_\circ} \pi_0 w ) c = (\textcolor{olive}{\tilde{\theta}} \pi_1 \pi_0 {,} \ 
\tilde{u}_5 \textcolor{teal}{\pi_1 \sigma_\circ} \pi_0 w ) c "] &&&&&
\mC_1 
\end{tikzcd}\]

\noi The arrows involved in this calculation are pictured: 

\[
\label{dgm witnessing Ore-square arrows for defining span comp} \begin{tikzcd}
	&& \cdot \\
	&& \cdot \\
	\cdot & \cdot & & & & \cdot  \\
	&& \cdot &&& \\
	&& \cdot &&& \\
	&&&&&
	\arrow[from=3-2, to=3-1, "\circ" marking, cyan]
	\arrow[from=2-3, to=3-2, "\circ" marking, color= teal]
	\arrow[from=4-3, to=3-2, "\circ" marking, color= orange]
	\arrow[from=5-3, to=4-3, color= orange]
	\arrow[from=5-3, to=3-1, bend left = 20, "\circ" marking, color= orange]
	\arrow[from=1-3, to=2-3, color= teal]
	\arrow[from=1-3, to=3-1, bend right = 20, "\circ" marking, color= teal]
	\arrow[from=3-6, to=5-3, bend left = 20, "\circ" marking, color = olive]
	\arrow[from=3-6, to=1-3, bend right = 20, color= olive]
\end{tikzcd} \tag{C}\]

\noi On the bottom we have 
\begin{align}
  \begin{split}
& \ \ \ \ (\textcolor{olive}{\tilde{\theta}} \pi_0 \pi_0 w , \ 
\tilde{u}_5 \textcolor{orange}{\pi_0 \sigma_\circ} \pi_0 w ) c\\
  &= 
(\textcolor{olive}{\tilde{\theta}} \pi_0 \pi_0 w , \ 
\tilde{u}_5 \textcolor{orange}{\pi_0 \omega} \pi_1 w 
)c\\
  &= 
\big(
\textcolor{olive}{\tilde{\theta}} \pi_0 \pi_0 w , \ 
\tilde{u}_5 \textcolor{orange}{\pi_0}
(\textcolor{orange}{\omega} \pi_0 \pi_0 , \ 
\textcolor{orange}{\omega} \pi_0 \pi_1 w, \ 
\textcolor{cyan}{\omega} \pi_0 \pi_2 w )c
\big) c \\
 &= 
\big(
\textcolor{olive}{\tilde{\theta}} \pi_0 \pi_0 w , \ 
\tilde{u}_5 \textcolor{orange}{\pi_0}
(\textcolor{orange}{\omega} \pi_0 \pi_0 , \ 
\textcolor{orange}{u_0 \theta} \pi_0 \pi_0 w, \ 
\textcolor{cyan}{u \pi_0 \pi_0 w} )c 
\big) c\\
  &= \big(
(\textcolor{olive}{\tilde{\theta}} \pi_0 \pi_0 w , \ 
\tilde{u}_5 \textcolor{orange}{\pi_0 \omega} \pi_0 \pi_0 , \ 
\tilde{u}_5 \textcolor{orange}{\pi_0 u_0 \theta} \pi_0 \pi_0 w 
)c , \ 
\tilde{u}_5 \textcolor{cyan}{\pi_0 u \pi_0 \pi_0 w} 
\big) c
\end{split}
\end{align}
\noi and on the top we have 
\begin{align}
  \begin{split}
& \ \ \ \ (\textcolor{olive}{\tilde{\theta}} \pi_1 \pi_0 , \ 
\tilde{u}_5 \textcolor{teal}{\pi_1 \sigma_\circ} \pi_0 w ) c \\
&= 
(\textcolor{olive}{\tilde{\theta}} \pi_1 \pi_0 w , \ 
\tilde{u}_5 \textcolor{teal}{\pi_1 \omega} \pi_1 w 
)c\\
&=
\big(
\textcolor{olive}{\tilde{\theta}} \pi_1 \pi_0 w , \ 
\tilde{u}_5 \textcolor{teal}{\pi_1}
(\textcolor{teal}{\omega} \pi_0 \pi_0 , \ 
\textcolor{teal}{\omega} \pi_0 \pi_1 w, \ 
\textcolor{cyan}{\omega} \pi_0 \pi_2 w )c
\big) c \\
&= 
\big(
\textcolor{olive}{\tilde{\theta}} \pi_1 \pi_0 w , \ 
\tilde{u}_5 \textcolor{teal}{\pi_1}
(\textcolor{teal}{\omega} \pi_0 \pi_0 , \ 
\textcolor{teal}{u_0 \theta} \pi_0 \pi_0 w, \ 
\textcolor{cyan}{u \pi_0 \pi_0 w} )c 
\big) c\\
&= 
\big(
(\textcolor{olive}{\tilde{\theta}} \pi_1 \pi_0 w , \ 
\tilde{u}_5 \textcolor{teal}{\pi_1 \omega} \pi_0 \pi_0 , \ 
\tilde{u}_5 \textcolor{teal}{\pi_1 u_0 \theta} \pi_0 \pi_0 w 
)c , \ 
\tilde{u}_5 \textcolor{cyan}{\pi_1 u \pi_0 \pi_0 w} 
\big) c .
\end{split}
\end{align}

\noi Since $\pi_0 u = \pi_1 u : \ker u \to \spn \tensor[_t]{\times}{_s} \spn$ by definition of $\ker u$, we have the equality

\[ \begin{tikzcd}[column sep = large]
\tilde{U}_4 \rar[rrrr, "\tilde{u}_5 \textcolor{cyan}{\pi_0 u \pi_0 \pi_0 w} = \tilde{u}_5 \textcolor{cyan}{\pi_1 u \pi_0 \pi_0 w}"] &&&&  \mC_1
\end{tikzcd} \]

\noi between the final components in the bottom lines of calculations (1) and (2) which says there is an arrow in $W$ coequalizing a parallel pair in $\mC$. This determines a unique map, $\delta_0' : \tilde{U}_4 \to P(\mC) \tensor[_{t}]{\times}{_{ws}} W$, by the fact that
\[\delta_0' \pi_0 \pi_1 = (\textcolor{olive}{\tilde{\theta}} \pi_1 \pi_0 w , \ 
\tilde{u}_5 \textcolor{teal}{\pi_1 \omega} \pi_0 \pi_0 , \ 
\tilde{u}_5 \textcolor{teal}{\pi_1 u_0 \theta} \pi_0 \pi_0 w 
)c,\]
\[ \delta_0' \pi_0 \pi_0 = (\textcolor{olive}{\tilde{\theta}} \pi_0 \pi_0 w , \ 
\tilde{u}_5 \textcolor{orange}{\pi_0 \omega} \pi_0 \pi_0 , \ 
\tilde{u}_5 \textcolor{orange}{\pi_0 u_0 \theta} \pi_0 \pi_0 w 
)c, \]
\noi and 
\[ \delta_0' \pi_1 = \tilde{u}_5 \textcolor{cyan}{\pi_0 u \pi_0 \pi_0 } ;\]
\noi and that the equality
\[ \delta_0' (\pi_0 \pi_0 , \pi_1 w) c = \delta_0' (\pi_0 \pi_1 , \pi_1 w) c\]

\noi holds. The map $\delta_0'$ uniquely determines the map $\delta_0 : \tilde{U}_3 \to \cP_{cq}(\mC)$ for which the equalizer diagram

\begin{center}
\begin{tikzcd}[column sep = large]
\cP_{cq}(\mC) \rar[tail, "\iota_{cq}"] & P(\mC) \tensor[_{t}]{\times}{_{ws}} W \rar[shift left, "(\pi_0 \pi_0 {,} \pi_1 w) c"] \rar[shift right, "(\pi_0 \pi_1{ , }\pi_1 w) c"'] & \mC_1 \\
\tilde{U}_4 \ar[ur, "\delta_0'"'] \uar[dotted, "\delta_0"] & & 
\end{tikzcd}
\end{center}

\noi commutes in $\cE$. By \textbf{In.Frc(4)} and Lemma~\ref{lem extending Int.Frc. covers} the cover and lift

\[ \begin{tikzcd}
\cP(\mC) & \\
\tilde{U}_3 \uar[cyan, "\tilde{\delta}_0"] \rar["\scriptstyle{/}" marking, "\tilde{u}_4" near start] & \tilde{U}_4
\end{tikzcd}\]

\noi from Diagram (\ref{dgm span composition covers (Ore + zippering)}) exist. Similarly, the map $\delta_1$ is induced by a map $\delta_1' : P(\mC) \tensor[_t]{\times}{_{ws}} W$. For readability purposes, let $\tilde{u}_{i;j} = \tilde{u}_i \tilde{u}_{i+1} ... \tilde{u}_j$ for $0 \leq i < j \leq 5$ with $\tilde{u} = \tilde{u}_{0;5}$. Applying the zippering axiom and Ore conditions as we did above gives another equation, from the definitions of $\cP$ and $W_\square$, which internally expresses the commutativity in the following picture: 

\[ 
\label{dgm witnessing first zipper for defining span comp} \begin{tikzcd}
	&& \cdot \\
	&& \cdot \\
	 & \cdot & \cdot & \cdot & & \cdot & \cdot  \\
	&& \cdot &&&& \\
	&& \cdot &&&& \\
	&&&&&&
	\arrow[from=3-2, to=3-3]
	\arrow[from=3-4, to=3-3, "\circ" marking, cyan]
	\arrow[from=2-3, to=3-2, "\circ" marking, color= teal]
	\arrow[from=2-3, to=3-4, color= teal]
	\arrow[from=4-3, to=3-4, color= orange]
	\arrow[from=4-3, to=3-2, "\circ" marking, color= orange]
	\arrow[from=5-3, to=4-3, color= orange]
	\arrow[from=1-3, to=2-3, color= teal]
	\arrow[from=3-6, to=5-3, bend left = 20, "\circ" marking, color = olive]
	\arrow[from=3-6, to=1-3, bend right = 20, color= olive]
	\arrow[from=3-7, to=3-6, "\circ" marking, cyan]
\end{tikzcd} \tag{C}\]

\begin{align}\label{eq inside paths of diagram C}
\begin{split}
& \ \ \ \ 
\big(
(
\textcolor{cyan}{\tilde{\delta}_0} \pi_0 \iota_{eq} \pi_0 ,\ 
\tilde{u}_4\textcolor{olive}{\tilde{\theta}} \pi_0 \pi_0 w , \ 
\tilde{u}_{4;5} \textcolor{orange}{\pi_0 \omega} \pi_0 \pi_0 , \ 
\tilde{u}_{4;5} \textcolor{orange}{\pi_0 u_0 \theta} \pi_1 \pi_0 
)c , \ 
\tilde{u}_{4;5} \textcolor{cyan}{\pi_0 u \pi_1 \pi_0 w} 
\big) c\\
&= 
\big(
\textcolor{cyan}{\tilde{\delta}_0} \pi_0 \iota_{eq} \pi_0 ,\ 
\tilde{u}_4\textcolor{olive}{\tilde{\theta}} \pi_0 \pi_0 w , \ 
\tilde{u}_{4;5} \textcolor{orange}{\pi_0 \omega} \pi_0 \pi_0 , \ 
(\tilde{u}_{4;5} \textcolor{orange}{\pi_0 u_0 \theta} \pi_1 \pi_0 , \ 
\tilde{u}_{4;5} \textcolor{cyan}{\pi_0 u \pi_1 \pi_0 w} 
)c
\big) c\\
&= 
\big(
\textcolor{cyan}{\tilde{\delta}_0} \pi_0 \iota_{eq} \pi_0 ,\ 
\tilde{u}_4\textcolor{olive}{\tilde{\theta}} \pi_0 \pi_0 w , \ 
\tilde{u}_{4;5} \textcolor{orange}{\pi_0 \omega} \pi_0 \pi_0 , \ 
(\tilde{u}_{4;5} \textcolor{orange}{\pi_0 u_0 \theta} \pi_0 \pi_0 w , \ 
\tilde{u}_{4;5} \pi_0 u \pi_0 \pi_1 
)c
\big) c\\
&= 
\big(
(\textcolor{cyan}{\tilde{\delta}_0} \pi_0 \iota_{eq} \pi_0 ,\ 
\tilde{u}_4\textcolor{olive}{\tilde{\theta}} \pi_0 \pi_0 w , \ 
\tilde{u}_{4;5} \textcolor{orange}{\pi_0 \omega} \pi_0 \pi_0 , \ 
\tilde{u}_{4;5} \textcolor{orange}{\pi_0 u_0 \theta} \pi_0 \pi_0 w
)c , \ 
\tilde{u}_{4;5} \pi_0 u \pi_0 \pi_1
\big) c\\
&= 
\big(
(\textcolor{cyan}{\tilde{\delta}_0} \pi_0 \iota_{eq} \pi_0 ,\ 
\tilde{u}_4\textcolor{olive}{\tilde{\theta}} \pi_1 \pi_0 , \ 
\tilde{u}_{4;5} \textcolor{teal}{\pi_0 \omega} \pi_0 \pi_0 , \ 
\tilde{u}_{4;5} \textcolor{teal}{\pi_0 u_0 \theta} \pi_0 \pi_0 w
)c , \ 
\tilde{u}_{4;5} \pi_0 u \pi_0 \pi_1 
\big) c\\
&= 
\big(
\textcolor{cyan}{\tilde{\delta}_0} \pi_0 \iota_{eq} \pi_0 ,\ 
\tilde{u}_4\textcolor{olive}{\tilde{\theta}} \pi_1 \pi_0 , \ 
\tilde{u}_{4;5} \textcolor{teal}{\pi_1 \omega} \pi_0 \pi_0 , \ 
(\tilde{u}_{4;5} \textcolor{teal}{\pi_1 u_0 \theta} \pi_0 \pi_0 w , \ 
\tilde{u}_{4;5} \pi_0 u \pi_0 \pi_1 
)c
\big) c\\
&= 
\big(
\textcolor{cyan}{\tilde{\delta}_0} \pi_0 \iota_{eq} \pi_0 ,\ 
\tilde{u}_4\textcolor{olive}{\tilde{\theta}} \pi_1 \pi_0 , \ 
\tilde{u}_{4;5} \textcolor{teal}{\pi_1 \omega} \pi_0 \pi_0 , \ 
(\tilde{u}_{4;5} \textcolor{teal}{\pi_1 u_0 \theta} \pi_1 \pi_0 , \ 
\tilde{u}_{4;5} \textcolor{cyan}{\pi_0 u \pi_1 \pi_0 w } 
)c
\big) c\\
&= 
\big(
(\textcolor{cyan}{\tilde{\delta}_0} \pi_0 \iota_{eq} \pi_0 ,\ 
\tilde{u}_4\textcolor{olive}{\tilde{\theta}} \pi_1 \pi_0 , \ 
\tilde{u}_{4;5} \textcolor{teal}{\pi_1 \omega} \pi_0 \pi_0 , \ 
\tilde{u}_{4;5} \textcolor{teal}{\pi_1 u_0 \theta} \pi_1 \pi_0 
)c, \ 
\tilde{u}_{4;5} \textcolor{cyan}{\pi_0 u \pi_1 \pi_0 w} 
\big) c
\end{split}
\end{align}

\noi The first and last lines in equation~(\ref{eq inside paths of diagram C}) correspond to the concatenations of the `inside' paths in Diagram (\ref{dgm witnessing first zipper for defining span comp}). They imply the existence of a map, $\delta_1' : \tilde{U}_3 \to P(\mC) \tensor[_t]{\times}{_{ws}} W$, uniquely determined by the projections 

\begin{align*}
  \delta_1' \pi_0 \pi_0 &= (
\textcolor{cyan}{\tilde{\delta}_0} \pi_0 \iota_{eq} \pi_0 ,\ 
\tilde{u}_4\textcolor{olive}{\tilde{\theta}} \pi_0 \pi_0 w , \ 
\tilde{u}_{4;5} \textcolor{orange}{\pi_0 \omega} \pi_0 \pi_0 , \ 
\tilde{u}_{4;5} \textcolor{orange}{\pi_0 u_0 \theta} \pi_1 \pi_0 
)c, \\
\delta_1' \pi_0 \pi_1 &= (\textcolor{cyan}{\tilde{\delta}_0} \pi_0 \iota_{eq} \pi_0 ,\ 
\tilde{u}_4\textcolor{olive}{\tilde{\theta}} \pi_1 \pi_0 , \ 
\tilde{u}_{4;5} \textcolor{teal}{\pi_1 \omega} \pi_0 \pi_0 , \ 
\tilde{u}_{4;5} \textcolor{teal}{\pi_1 u_0 \theta} \pi_1 \pi_0 
)c, \\
\delta_1' \pi_1 &= \tilde{u}_{4;5} \textcolor{cyan}{\pi_0 u \pi_1 \pi_0 w }
\end{align*} 

\noi for which

\[ \delta_1' (\pi_0 \pi_0 , \pi_1 w) c = \delta_1' ( \pi_0 \pi_1 , \pi_1 w) c \]

\noi represents the inner cyan-colored arrow in Diagram \ref{dgm witnessing first zipper for defining span comp}. The map $\delta_1'$ induces the unique map $\delta_1$ that makes the following equalizer diagram

\begin{center}
\begin{tikzcd}[column sep = large]
\cP_{cq}(\mC) \rar[tail, "\iota_{cq}"] & P(\mC) \tensor[_t]{\times}{_{ws}} W \rar[shift left, "(\pi_0 \pi_0 {,} \pi_1 w) c"] \rar[shift right, "(\pi_0 \pi_1{ , }\pi_1 w) c"'] & \mC_1 \\
\tilde{U}_3 \ar[ur, "\delta_1'"'] \uar[dotted, "\delta_1"] & & 
\end{tikzcd}
\end{center}

\noi commute in $\cE$. By \textbf{In.Frc(4)} the cover lift 

\[\begin{tikzcd}
\tilde{U}_2 \dar[cyan, "\tilde{\delta}_1"'] \rar["\scriptstyle{/}" marking, "\tilde{u}_3"near start] 
& \tilde{U}_3 \\
\cP(\mC)
\end{tikzcd}\]

\noi in Diagram (\ref{dgm span composition covers (Ore + zippering)}) to make the bottom left square commute. The covers and lifts 

\[\begin{tikzcd}[]
 W_\circ 
& 
& W_\circ 
& 
\\
\tilde{U}
\rar["/" marking, "\tilde{u}_0" near start]
\uar[violet,"\tilde{\omega}_0"]
& 
\tilde{U}_0 \rar["/" marking, "\tilde{u}_1" near start]
\dar[violet,"\tilde{\omega}_1"'] 
& 
\tilde{U}_1 \rar["/" marking, "\tilde{u}_2" near start]
\uar[violet,"\tilde{\omega}_2"] 
& 
\tilde{U}_2 
\\
& W_\circ 
& 
&
\end{tikzcd}\]

\noi in Diagram (\ref{dgm span composition covers (weak comp part)}) are given by \textbf{In.Frc(2)} and Lemma~\ref{lem extending Int.Frc. covers}. It suffices to define $\omega_i : \tilde{U}_i \to W \tensor[_{wt}]{\times}{_{ws}} W$, in Diagram (\ref{dgm span composition covers (weak comp part)}) that pick out composable pairs in $W$. The relevant representative diagram in $\mC$ to keep in mind is: 

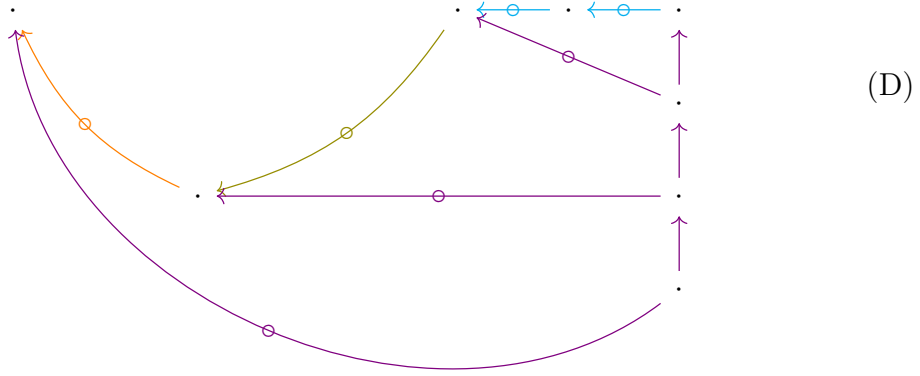
\[
\label{dgm witnessing weak comp arrows for defining span comp}\begin{tikzcd}
	&& \\
	&& \\
	\cdot & & & & & \cdot & \cdot & \cdot \\
	&& &&&&& \cdot \\
	&& \cdot &&&&& \cdot \\
	&&&&&&& \cdot
	\arrow[from=5-3, to=3-1, bend left = 20, "\circ" marking, color= orange]
	\arrow[from=3-6, to=5-3, bend left = 20, "\circ" marking, color = olive]
	\arrow[from=3-7, to=3-6, "\circ" marking, cyan]
	\arrow[from=3-8, to=3-7, "\circ" marking,cyan]
	\arrow[from=4-8, to=3-8, violet]
	\arrow[from=4-8, to=3-6, "\circ" marking, violet]
	\arrow[from=5-8, to=4-8, violet]
	\arrow[from=6-8, to=5-8, violet]
	\arrow[from=5-8, to=5-3, "\circ" marking, violet]
	\arrow[from=6-8, to=3-1, bend left = 60, "\circ" marking, violet]
\end{tikzcd} \tag{D} \]

\noi The maps $\omega_i : \tilde{U}_i \to W \tensor[_{wt}]{\times}{_{ws}}$ are defined in sequence as follows. First, the pair of arrows obtained from the two diagram-extension conditions (colored in cyan in Diagram (\ref{dgm witnessing weak comp arrows for defining span comp})) are composable by definition of $\cP(\mC)$: 
\begin{align*}
\textcolor{cyan}{\tilde{\delta}_1} \pi_0 \iota_{eq} \pi_0 w t 
&= \textcolor{cyan}{\tilde{\delta}_1} \pi_0 \iota_{eq} \pi_1 \pi_0 s & \text{Def. } W \tensor[_{wt}]{\times}{_{s}} P(\mC)\\
&= \textcolor{cyan}{\tilde{\delta}_1} \pi_1 \iota_{cq} \pi_0 \pi_0 s & \text{Def. } \cP(\mC)\\
&= \tilde{u}_3 \delta_1 \iota_{cq} \pi_0 \pi_0 s & \text{Def. } \textcolor{cyan}{\tilde{\delta}_1} \\
&= \tilde{u}_3 \delta_1' \pi_0 \pi_0 s & \text{Def. } \delta_1 \\
&= \tilde{u}_3 \textcolor{cyan}{\tilde{\delta}_0} \pi_0 \iota_{eq} \pi_0 s & \text{Def. } \delta_1'
\end{align*}
\noi This uniquely determines the map
\[ \begin{tikzcd} 
\tilde{U}_2
\rar[rrrr, "\omega_2 = (\textcolor{cyan}{\tilde{\delta}_1} \pi_0 \iota_{eq} \pi_0 {,} \tilde{u}_3 \textcolor{cyan}{\tilde{\delta}_0} \pi_0 \iota_{eq} \pi_0) "] 
&&&& 
W \tensor[_{wt}]{\times}{_{ws}} W
\end{tikzcd}\]

\noi which gives the lift $\textcolor{violet}{\tilde{\omega}_2} : \tilde{U}_1 \to W_\circ$ in Diagram (\ref{dgm span composition covers (weak comp part)}). The composite (in $W$) witnessed by $\textcolor{violet}{\tilde{\omega}_2}$ can be composed with the arrow in $W$ (colored \textcolor{olive}{olive} in Diagram (\ref{dgm witnessing weak comp arrows for defining span comp}) ) that filled the Ore square because 

\begin{align*}
\textcolor{violet}{\tilde{\omega}_2} \pi_1 w t 
&= \textcolor{violet}{\tilde{\omega}_2} \pi_0 \pi_{12} \pi_1 w t & \text{Def. } W_\circ\\
&= \tilde{u}_2 \omega_2 \pi_1 w t & \text{Def. } \textcolor{violet}{\tilde{\omega}_2} \\
&= \tilde{u}_{2;3} \textcolor{cyan}{\tilde{\delta}_0} \pi_0 \iota_{eq} \pi_0 w t & \text{Def. } \omega_2 \\
&= \tilde{u}_{2;3} \textcolor{cyan}{\tilde{\delta}_0} \pi_0 \iota_{eq} \pi_1 \pi_0 s & \text{Def. } W \tensor[_{wt}]{\times}{_{s}} P(\mC) \\
&= \tilde{u}_{2;3} \textcolor{cyan}{\tilde{\delta}_0} \pi_1 \iota_{cq} \pi_0 \pi_0 s & \text{Def. } \cP(\mC)\\
&= \tilde{u}_{2;4} \delta_0 \iota_{cq} \pi_0 \pi_0 s & \text{Def. } \\
&= \tilde{u}_{2;4} \delta_0' \pi_0 \pi_0 s & \text{Def. } \delta_0 \\
&= \tilde{u}_{2;4} \textcolor{olive}{\tilde{\theta}} \pi_0 \pi_0 w s & \text{Def. } \delta_0'\\
\end{align*}

\noi and it induces the pairing map

\[ \begin{tikzcd} 
\tilde{U}_2
\rar[rrrr, "\omega_1 = (\textcolor{violet}{\tilde{\omega}_2} \pi_1 {,} \tilde{u}_{2;4} \textcolor{olive}{\tilde{\theta}} \pi_0 \pi_0 ) "] 
&&&& 
W \tensor[_{wt}]{\times}{_{ws}} W
\end{tikzcd}. \]

\noi This gives the cover $\tilde{u}_1 : \tilde{U}_0 \to \tilde{U})1$ and lift $\textcolor{violet}{\tilde{\omega}_1} : \tilde{U}_0 \to W_\circ$ in Diagram (\ref{dgm span composition covers (weak comp part)}). Finally, the composite witnessed by $\textcolor{violet}{\tilde{\omega}_1}$ can be composed with the left leg of the original span (colored in orange in Diagram \ref{dgm witnessing weak comp arrows for defining span comp} and) witnessed by $\tilde{u}_{1;5} \textcolor{orange}{\pi_0 \sigma_\circ} : \tilde{U}_1 \to \spn $ because 

\begin{align*}
\textcolor{violet}{\tilde{\omega}_1} \pi_1 w t 
&= \textcolor{violet}{\tilde{\omega}_1} \pi_0 \pi_{12} \pi_1 w t & \text{Def. } W_\circ \\
&= \tilde{u}_1 \omega_1 \pi_1 w t & \text{Def. } \textcolor{violet}{\tilde{\omega}_1} \\
&= \tilde{u}_{1;4} \textcolor{olive}{\tilde{\theta}} \pi_0 \pi_0 w t & \text{Def. } \Omega_1 \\
&= \tilde{u}_{1;5} \textcolor{orange}{\pi_0 \sigma_\circ} \pi_0 w s & \text{Def. } \textcolor{olive}{\theta}.
\end{align*}

\noi This gives the unique pairing map

\[ \begin{tikzcd} 
\tilde{U}_2
\rar[rrrr, "\omega_0 = (\textcolor{violet}{\tilde{\omega}_1} \pi_1 {,} \tilde{u}_{1;5} \textcolor{orange}{\pi_0 \sigma_\circ} \pi_0 ) "] 
&&&& 
W \tensor[_{wt}]{\times}{_{ws}} W
\end{tikzcd} \]

\noi which induces the cover $u_0 : \tilde{U} \to \tilde{U}_0$ and lift $ \hat{\omega}_0 : \tilde{U} \to W_\circ$ in Diagram (\ref{dgm span composition covers (weak comp part)}). Now let

\[ \begin{tikzcd}[column sep = huge]
\tilde{U} \rar[rrrrr, "\omega = (\textcolor{violet}{\tilde{\omega}_0} \pi_0 \pi_0 {,} \ \tilde{u}_0 \textcolor{violet}{\tilde{\omega}_1} \pi_0 \pi_0 {,} \tilde{u}_{0;1} \textcolor{violet}{\tilde{\omega}_2} \pi_1 w) c "] 
&&&&& \mC_1
\end{tikzcd}\] 

\noi witness the composite(s) of the three vertical violet-colored arrows and the two horizontal cyan-colored arrows in Diagram (\ref{dgm witnessing weak comp arrows for defining span comp}):

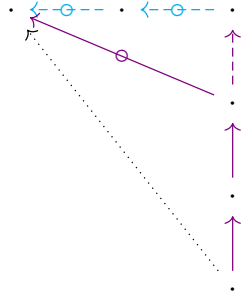
\[\begin{tikzcd}
   \cdot & \cdot & \cdot \\
	&& \cdot \\
	&& \cdot \\
	&& \cdot
	\arrow[from=1-2, to=1-1, dashed, "\circ" marking, cyan]
	\arrow[from=1-3, to=1-2, dashed, "\circ" marking,cyan]
	\arrow[from=2-3, to=1-3, dashed, violet]
	\arrow[from=2-3, to=1-1, "\circ" marking, violet]
	\arrow[from=3-3, to=2-3, violet]
	\arrow[from=4-3, to=3-3, violet]
	\arrow[from= 4-3, to=1-1, dotted] 
\end{tikzcd} \]

By zippering we get commutativity of the following piece of Diagram (\ref{dgm witnessing arrows for span comp})

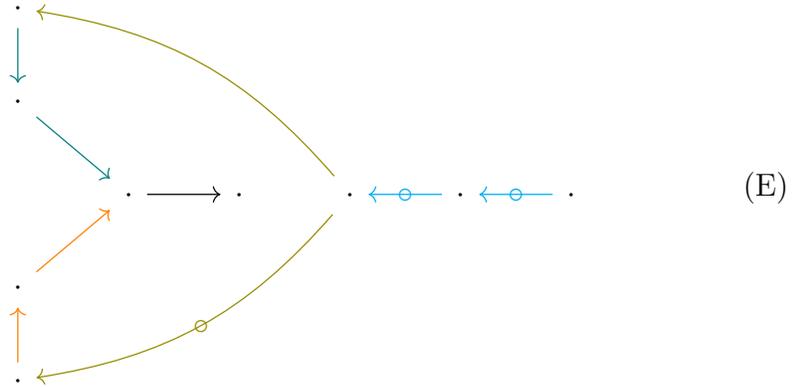
\[\label{dgm witnessing zipper commutativity arrows for span comp}
\begin{tikzcd}
	&& \cdot \\
	&& \cdot \\
	& & & \cdot & \cdot & \cdot & \cdot & \cdot \\
	&& \cdot &&&&& \\
	&& \cdot &&&&& 
	\arrow[from=3-4, to=3-5]
	\arrow[from=2-3, to=3-4, color= teal]
	\arrow[from=4-3, to=3-4, color= orange]
	\arrow[from=5-3, to=4-3, color= orange]
	\arrow[from=1-3, to=2-3, color= teal]
	\arrow[from=3-6, to=5-3, bend left = 20, "\circ" marking, color = olive]
	\arrow[from=3-6, to=1-3, bend right = 20, color= olive]
	\arrow[from=3-7, to=3-6, "\circ" marking, cyan]
	\arrow[from=3-8, to=3-7, "\circ" marking,cyan]
\end{tikzcd} \tag{E}\]

\noi Internally we can use associativity of composition, definitions of the pairing maps involved, and the definition of $\cP(\mC)$ to write this commutativity by the equation: 

\begin{align}\label{eq sigma01 descriptions}
(
\omega, \
\tilde{u}_{0;4} \textcolor{olive}{\tilde{\theta}} \pi_0 \pi_0 w ,\
\tilde{u} \textcolor{orange}{\pi_0 \sigma_\circ} \pi_1 )c 
= 
(
\omega, \
\tilde{u}_{0;4} \textcolor{olive}{\tilde{\theta}} \pi_1 \pi_0 ,\
\tilde{u} \textcolor{teal}{\pi_1 \sigma_\circ} \pi_1 )c .
\end{align} 

\noi or the commuting diagram 

\[\begin{tikzcd}[column sep = huge]
\tilde{U} \dar[dd,"(
\omega {,} \
\tilde{u}_{0;4} \textcolor{olive}{\tilde{\theta}} \pi_0 \pi_0 w {,} \
\tilde{u} \textcolor{orange}{\pi_0 \sigma_\circ} \pi_1 )"'] 
\rar[rr, "(
\omega{,} \
\tilde{u}_{0;4} \textcolor{olive}{\tilde{\theta}} \pi_1 \pi_0 {,}\
\tilde{u} \textcolor{teal}{\pi_1 \sigma_\circ} \pi_1 )"] 
&& \mC_3 \dar[dd,"c"] \\
&&\\
\mC_3 \rar[rr,"c"'] && \mC_1 
\end{tikzcd}. \qquad \qquad \qquad \qquad \qquad \]

\noi For readability we define $\mu_0$ and $\mu_1$ by composition in $\mC$

\[ \begin{tikzcd}
\tilde{U} \rar[rrr, "(\omega{,} \
\tilde{u}_{0;4} \textcolor{olive}{\tilde{\theta}} \pi_0 \pi_0 w)"] 
\ar[drrr, "\mu_0"'] 
&&& \mC_2 \dar["c"] \\
&&& \mC_1 
\end{tikzcd} 
\begin{tikzcd}
\tilde{U} \rar[rrr, "(\omega{,} \
\tilde{u}_{0;4} \textcolor{olive}{\tilde{\theta}} \pi_1 \pi_0 w)"] \ar[drrr, "\mu_1"'] 
&&& \mC_2 \dar["c"] \\
&&& \mC_1 
\end{tikzcd}\]

\noi to internally represent the composites in the following piece of Diagram (\ref{dgm witnessing arrows for span comp}): 

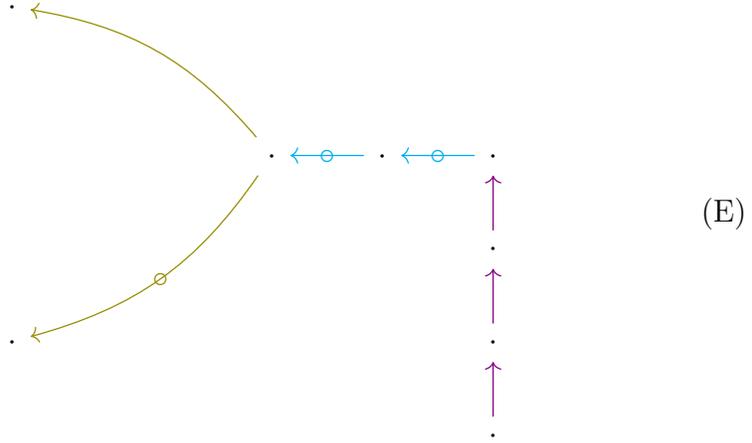
\[\label{dgm witnessing masts of sailboats for defining span comp}
\begin{tikzcd}
	&& \cdot \\
	&& \\
	& & & & & \cdot & \cdot & \cdot \\
	&& &&&&& \cdot \\
	&& \cdot &&&&& \cdot \\
	&&&&&&&\cdot
	\arrow[from=3-6, to=5-3, bend left = 20, "\circ" marking, color = olive]
	\arrow[from=3-6, to=1-3, bend right = 20, color= olive]
	\arrow[from=3-7, to=3-6, "\circ" marking, cyan]
	\arrow[from=3-8, to=3-7, "\circ" marking,cyan]
	\arrow[from=4-8, to=3-8, violet]
	\arrow[from=5-8, to=4-8, violet]
	\arrow[from=6-8, to=5-8, violet]
\end{tikzcd} \tag{E}\]

\noi Note that equation~(\ref{eq sigma01 descriptions}) above gives two descriptions of the right leg of an intermediate span $ \textcolor{violet}{\sigma_{01}} : \tilde{U} \to \spn$ given by the pairing

\[ \textcolor{violet}{\sigma_{01}} = (\textcolor{violet}{\tilde{\omega}_0} \pi_1 , \textcolor{violet}{\sigma_{01}} \pi_1), \]

\noi where the right-hand component can be rewritten as either one of the terms in the following equation:

\[ (\mu_0 , \tilde{u} \textcolor{orange}{\pi_0 \sigma_\circ} \pi_1 )c = \textcolor{violet}{\sigma_{01}} \pi_1 = (\mu_1 , \tilde{u} \textcolor{teal}{\pi_1 \sigma_\circ} \pi_1 )c \]

\noi The different representations of the right leg of this intermediate span can be seen by the two paths in Diagram (\ref{dgm witnessing arrows for span comp}) given by combining Diagrams (\ref{dgm witnessing first zipper for defining span comp}) and (\ref{dgm witnessing masts of sailboats for defining span comp}). Now by expanding internal composition in terms of pairing maps; by associativity of composition in $\mC$; and by the definitions of $W_\circ, \cP(\mC), $ and $W_\square$ we can represent the left leg of the intermediate span $\textcolor{violet}{\sigma_{01}}$ by: 

\[ ( \mu_0 , \tilde{u} \textcolor{orange}{\pi_0 \sigma_\circ} \pi_0 ) c = \textcolor{violet}{\tilde{\omega}_0} \pi_1 = ( \mu_1 , \tilde{u} \textcolor{teal}{\pi_1 \sigma_\circ} \pi_0 ) c .\]

\noi Note that the sources of the composites (in $\mC$) in the previous two equations are the sources of the maps $\mu_0, \mu_1 : \tilde{U} \to \mC_1$, which have a common source in $\omega s : \tilde{U} \to \mC_0$. We can now give well-defined explicit descriptions of the two sailboats, $\tilde{U} \to \slb$, using the universal property of $\slb$. The first sailboat represents picking out the following piece of Diagram (\ref{dgm witnessing arrows for span comp}):

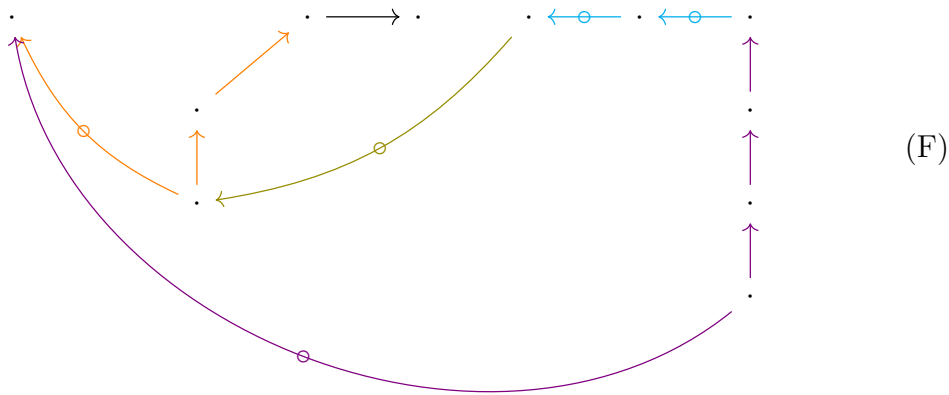
\[\label{dgm witnessing sailboat-0 for span comp}
\begin{tikzcd}
	\cdot & & & \cdot & \cdot & \cdot & \cdot & \cdot \\
	&& \cdot &&&&& \cdot \\
	&& \cdot &&&&& \cdot \\
	&&&&&&& \cdot
	\arrow[from=1-4, to=1-5]
	\arrow[from=2-3, to=1-4, color= orange]
	\arrow[from=3-3, to=2-3, color= orange]
	\arrow[from=3-3, to=1-1, bend left = 20, "\circ" marking, color= orange]
	\arrow[from=1-6, to=3-3, bend left = 20, "\circ" marking, color = olive]
	\arrow[from=1-7, to=1-6, "\circ" marking, cyan]
	\arrow[from=1-8, to=1-7, "\circ" marking,cyan]
	\arrow[from=2-8, to=1-8, violet]
	\arrow[from=3-8, to=2-8, violet]
	\arrow[from=4-8, to=3-8, violet]
	\arrow[from=4-8, to=1-1, bend left = 60, "\circ" marking, violet]
\end{tikzcd} \tag{F}\]

\noi This is determined uniquely by the components in the pairing map:
\[ \varphi_0 = \big( ( ( 
\mu_0 , \ 
\tilde{u} \textcolor{orange}{\pi_0 \sigma_\circ} \pi_0) ,\ 
\textcolor{violet}{\tilde{\omega}_0} \pi_1 ),\ 
\tilde{u} \textcolor{orange}{\pi_0 \sigma_\circ} \pi_1 
\big) .
\]

\noi By definition of $\varphi_0$ we can compute

\[ \varphi_0 p_0 = \varphi_0 (\pi_0 \pi_0 \pi_1 , \pi_1) = (\tilde{u} \textcolor{orange}{\pi_0 \sigma_\circ} \pi_0 , \ \tilde{u} \textcolor{orange}{\pi_0 \sigma_\circ} \pi_1 ) = \tilde{u} \textcolor{orange}{\pi_0 \sigma_\circ} \]
\noi and additionally with the definition of \textcolor{violet}{$\sigma_{01}$} we can see 

\begin{align*}
   \varphi_0 p_1
   &= \varphi_0 (\pi_0 \pi_1, (\pi_0 \pi_0 \pi_0, \pi_1)c ) \\
   &= \big(\textcolor{violet}{\tilde{\omega}_0} \pi_1 , \ 
(\mu_0 , \ \tilde{u} \textcolor{orange}{\pi_0 \sigma_\circ} \pi_1 )c \big) \\
  &= (\textcolor{violet}{\tilde{\omega}_0} \pi_1 , \ \textcolor{violet}{\sigma_{01}} \pi_1 )\\
  &= \textcolor{violet}{\sigma_{01}}.
\end{align*} 

\noi The second sailboat represents picking out the following piece of Diagram (\ref{dgm witnessing arrows for span comp}): 

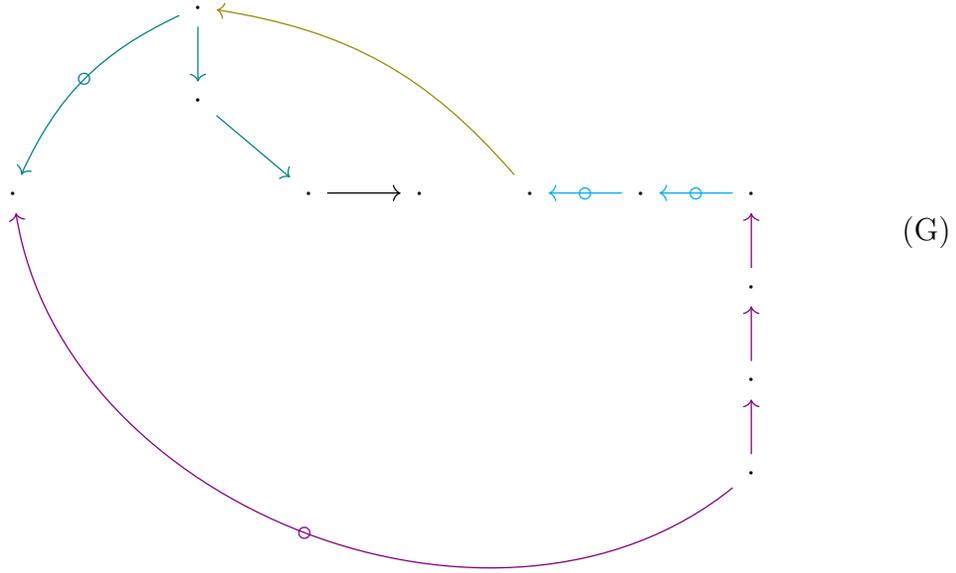
\[\label{dgm witnessing sailboat-1 for span comp}
\begin{tikzcd}
	&& \cdot \\
	&& \cdot \\
	\cdot &&& \cdot & \cdot & \cdot & \cdot & \cdot \\
	&&&&&&& \cdot \\
	&&&&&&& \cdot \\
	&&&&&&& \cdot
	\arrow[from=3-4, to=3-5]
	\arrow[from=2-3, to=3-4, color= teal]
	\arrow[from=1-3, to=2-3, color= teal]
	\arrow[from=1-3, to=3-1, bend right = 20, "\circ" marking, color= teal]
	\arrow[from=3-6, to=1-3, bend right = 20, color= olive]
	\arrow[from=3-7, to=3-6, "\circ" marking, cyan]
	\arrow[from=3-8, to=3-7, "\circ" marking,cyan]
	\arrow[from=4-8, to=3-8, violet]
	\arrow[from=5-8, to=4-8, violet]
	\arrow[from=6-8, to=5-8, violet]
	\arrow[from=6-8, to=3-1, bend left = 60, "\circ" marking, violet]
\end{tikzcd} \tag{G}\]

\noi This one is uniquely determined by the pairing map: 
\[ \varphi_1 = \big( ( ( \mu_1 , \ \tilde{u} \textcolor{teal}{\pi_1 \sigma_\circ} \pi_0) ,\ \textcolor{violet}{\tilde{\omega}_0} \pi_1 ),\ \tilde{u} \textcolor{teal}{\pi_1 \sigma_\circ} \pi_1 \big) .
\]

\noi By definition of $\varphi_1$ we get

\begin{align*}
   \varphi_1 p_0 
   &= \varphi_1 (\pi_0 \pi_0 \pi_1 , \pi_1) \\
   &= (\tilde{u} \textcolor{teal}{\pi_1 \sigma_\circ} \pi_0 , \ \tilde{u} \textcolor{teal}{\pi_1 \sigma_\circ} \pi_1 )\\
   &= \tilde{u} \textcolor{teal}{\pi_1 \sigma_\circ}
\end{align*}  
\noi and by definition of \textcolor{violet}{$\sigma_{01}$}
\begin{align*}
   \varphi_0 p_1
   &= \varphi_0 (\pi_0 \pi_1, (\pi_0 \pi_0 \pi_0, \pi_1)c ) \\
   &= \big(\textcolor{violet}{\tilde{\omega}_0} \pi_1 , \ 
(\mu_0 , \ \tilde{u} \textcolor{orange}{\pi_0 \sigma_\circ} \pi_1 )c \big)\\
&= (\textcolor{violet}{\tilde{\omega}_0} \pi_1 , \ \textcolor{violet}{\sigma_{01}} \pi_1 ) \\
&= \textcolor{violet}{\sigma_{01}}.
\end{align*}

\noi Putting the previous few computations together we can see 
\begin{align*}
\tilde{u} \textcolor{orange}{\pi_0 \sigma_\circ} q 
&= \varphi_0 p_0 q & \text{Def. } \varphi_0 \\
&= \varphi_0 p_1 q & \text{Def. } q\\
&= \textcolor{violet}{\sigma_{01}} q & \text{Def. } \varphi_0 \\
&= \varphi_1 p_1 q & \text{Def. } \varphi_1 \\
&= \varphi_1 p_0 q & \text{Def. } q \\
& = \tilde{u} \textcolor{teal}{\pi_1 \sigma_\circ} q & \text{Def. } \varphi_1
\end{align*}
\noi and since $\tilde{u}$ is epic: 
\[ \pi_0 \sigma_\circ q = \pi_1 \sigma_\circ q \]

\noi That is, the diagram

\begin{center}
\begin{tikzcd}[]
\ker{u} \dar["\pi_0"'] \rar["\pi_1"] & U \dar["\sigma_\circ q"] \\
U \rar["\sigma_\circ q"'] & \mCW_1
\end{tikzcd}
\end{center}

\noi commutes. 
\end{proof}

\begin{lem}\label{lem actual definition of c'}
There exists a unique `composition on representatives' map $c' : \spn \tensor[_t]{\times}{_{s}} \spn \to \mCW_1$ such that the diagram 

\[ \begin{tikzcd}
U \rar["\scriptstyle{/}" marking, "u" near start] \dar["\sigma_\circ"'] & \spn \tensor[_t]{\times}{_{s}} \spn \dar[dotted, "c'"]\\
\spn \rar[two heads, "q"'] & \mCW_1
\end{tikzcd}\]

\noi commutes in $\cE$. 
\end{lem}
\begin{proof}
This follows by the universal property of $u$ being the coequalizer of its kernel pair and Lemma~\ref{lem defining c'} showing that $\sigma_\circ q : \spn \tensor[_t]{\times}{_{s}} \spn \to \mCW_1$ also coequalizes the kernel pair of $u$. 
\end{proof}

\noi Having defined composition on representative spans by a map $c' : \spn \tensor[_t]{\times}{_{s}} \spn \to \mCW_1$, the next thing to do is to check it is well-defined. This is translated internally by the following proposition. 

\begin{prop}\label{prop composition of spans is well-defined}
The composition operation on spans,
\[ c' : \spn \tensor[_t]{\times}{_{s}} \spn \to \mCW_1, \]

\noi is well-defined on equivalence classes in the sense that the square

\begin{center}
\begin{tikzcd}[]
\slb \tensor[_{t}]{\times}{_{s}} \slb \rar["p_1^2"] \dar["p_0^2"'] & \spn \tensor[_{t}]{\times}{_{s}} \spn \dar["c' "] \\
\spn \tensor[_{t}]{\times}{_{s}} \spn \rar["c' "'] & \mCW_1
\end{tikzcd}
\end{center}

\noi commutes in $\cE$. 
\end{prop}
\begin{proof}
By Lemma~\ref{lem witnessing c' well-defined} 

\[ \hat{u} p_0^2 c' = \varphi_0 p_0 q = \varphi_3 p_0 q = \hat{u} p_1^2 c'\]

\noi and since $\hat{u}$ is epic 

\[p_0^2 c' = p_1^2 c' \]
\end{proof}

\noi A direct consequence of Proposition~\ref{prop composition of spans is well-defined} is it induces a unique composition map $c : \mCW_2 \to \mCW_1$ such that the diagram

\[ \begin{tikzcd}
\spn \tensor[_t]{\times}{_{s}} \spn \rar[two heads, "q_2"] \ar[dr, "c'"'] & \mCW_2 \dar[dotted, "c"] \\
& \mCW_1
\end{tikzcd}\]

\noi commutes in $\cE$, by the universal property of the coequalizer $\mCW_2$. Lemma~\ref{lem witnessing c' well-defined} is doing all the heavy lifting for showing that $c'$ is well-defined and subsequently defining the composition map $c : \mCW_2 \to \mCW_1$. We now prove this lengthy and technical lemma. 

\begin{lem}\label{lem witnessing c' well-defined}
There exists a cover $\hat{U} \to \slb \tensor[_{t}]{\times}{_s} \slb$, and four families of sailboats, $\varphi_i: \hat{U} \to \slb$ for $0 \leq i \leq 3$, such that the diagram

\begin{center}
\begin{tikzcd}[]
\slb^2 \ar[r, shift left, "p_0^2"] \ar[r, shift right, "p_1^2"'] 
& \spn^2 \ar[ddr, "c'"] 
& \\
\hat{U} 
\uar["/" marking, "\hat{u}" near start] 
\dar[shift left, "\varphi_0"] 
\dar[shift right, "\varphi_3"'] 
 \ar[r, shift left, "\hat{\pi}_0"] \ar[r, shift right, "\hat{\pi}_1"'] 
 & U 
 \uar["/" marking, "u"'near start] 
 \dar["\sigma_\circ"] & \\
 \slb \ar[r, "p_0"] 
 & \spn 
 \rar[two heads, "q"'] 
 & \mCW 
\end{tikzcd}
\end{center}

\noi commutes in the sense that 
\[ \varphi_0 p_0 q = \hat{\pi}_0 \sigma_\circ q = \hat{\pi_0} u c' = \hat{u} p_0^2 c' , \]
\[ \varphi_4 p_0 q = \hat{\pi}_1 \sigma_\circ q = \hat{\pi_1} u c' = \hat{u} p_1^2 c', \]
\noi and the sailboats glue together along comparison spans 
\[ \varphi_0 p_0 q = \varphi_0 p_1 q = \varphi_1 p_1 q = \varphi_1 p_0 q = \varphi_2 p_0 q = \varphi_2 p_1 q = \varphi_3 p_1 q = \varphi_3 p_0 q .\]

\end{lem}
\begin{proof}
The main idea is to use the explicit definition 

\[ \sigma_\circ = \big( \omega \pi_1 , \ (\omega \pi_0 \pi_0 \pi_0 , \ u_0 \theta \pi_1 \pi_0 , \ u \pi_1 \pi_1)c \big)\]

\noi and post-compose it with the maps $p_0^2$ and $p_1^2$ to get two different spans. To show these two spans are equivalent we construct a comparison span from the data involved in each of their constructions and show they're both equivalent to the comparison span. Each of these equivalences in turn requires constructing an additional comparison span and a witnessing sailboat. This accounts for the four sailboats. 

To do this we need a common domain for the covers so take pullbacks of $u : U \to \spn \tensor[_{t}]{\times}{_s} \spn$ along $p_0^2$ and $p_1^2$ to get two covers of $\slb \tensor[_{t}]{\times}{_s}\slb$ 

\[\label{dgm slb proj pb along span comp cover}
\begin{tikzcd}[]
\bar{U}_0 \arrow[dr, phantom, "\usebox\pullback" , very near start, color=black] \dar["/" marking, "\bar{u}_0"' near end] \rar["\pi_1"] & U \dar["/" marking, "u" near start] & \bar{U}_1 \dar["/" marking, "\bar{u}_1" near start] \lar["\pi_1"'] \arrow[dl, phantom, "\usebox\urpullback" , very near start, color=black]\\
 \slb \tensor[_{t}]{\times}{_s} \slb \rar["p_0^2"'] & \spn \tensor[_{t}]{\times}{_s} \spn & \slb \tensor[_{t}]{\times}{_s} \slb \lar["p_1^2"]
\end{tikzcd} \tag{1}
\]

\noi Now take a refinement 
\[ \label{dgm slb proj pb along span comp cover refinement}
\begin{tikzcd}[]
\bar{U} \ar[d,"\pi_0"'] \ar[r, "\pi_1"] \ar[dr, "/" marking, "\bar{u}" near start] & \bar{U}_1 \dar["/" marking,"\bar{u}_1" near start] \\
\bar{U}_0 \rar["/" marking,"\bar{u}_0"' near end] & \slb \tensor[_{t}]{\times}{_s} \slb
\end{tikzcd} \tag{2}
\]

\noi by taking a pullback of $\bar{u}_0$ and $\bar{u}_1$ to get a cover of the pairs of composable sailboats. Note that 
\[ \bar{u} p_0^2 = \pi_0 \pi_0 p_0^2 = \pi_0 \pi_1 u \]
\noi projects out the composable spans represented by 
\begin{center}
$\left[
\begin{tikzcd}[]
\cdot & \cdot \lar["\circ" marking] \rar & \cdot & \cdot \lar["\circ" marking] \rar & \cdot 
\end{tikzcd} 
\right]$
\end{center}
\noi while 
\[ \bar{u} p_1^2 = \pi_1 \pi_0 p_1^2 = \pi_1 \pi_1 u \]
\noi projects out the composable spans represented by
\begin{center}
$\left[
\begin{tikzcd}[]
& \cdot \dar \ar[dl, "\circ" marking] & & \cdot \dar \ar[dl, "\circ" marking] & \\
\cdot & \cdot \rar & \cdot & \cdot \rar & \cdot 
\end{tikzcd} 
\right]$
\end{center}

\noi From this point the usual set-theoretic proof can be translated into a chain of covers and lifts. The outline is that for any pair of composable sailboats, 

\begin{center}
$\left[
\begin{tikzcd}[]
& \cdot \dar \ar[dl, "\circ" marking] & & \cdot \dar \ar[dl, "\circ" marking] & \\
\cdot & \cdot \lar["\circ" marking] \rar & \cdot & \cdot \lar["\circ" marking] \rar & \cdot 
\end{tikzcd} 
\right]$
\end{center}

\noi the composites of the spans represented by $\bar{u} p_0^2$ and $\bar{u} p_1^2$ are equivalent to the composite of a comparison pair of composable spans, 

\begin{center}
$\left[
\begin{tikzcd}[]
&& & \cdot \dar[dotted] \ar[dl, "\circ" marking] \ar[dr,] & \\
\cdot & \cdot \lar["\circ" marking] \rar & \cdot & \cdot \rar[dotted] & \cdot 
\end{tikzcd} 
\right]$
\end{center}

\noi The following figure shows the construction of three different composites being constructed. 

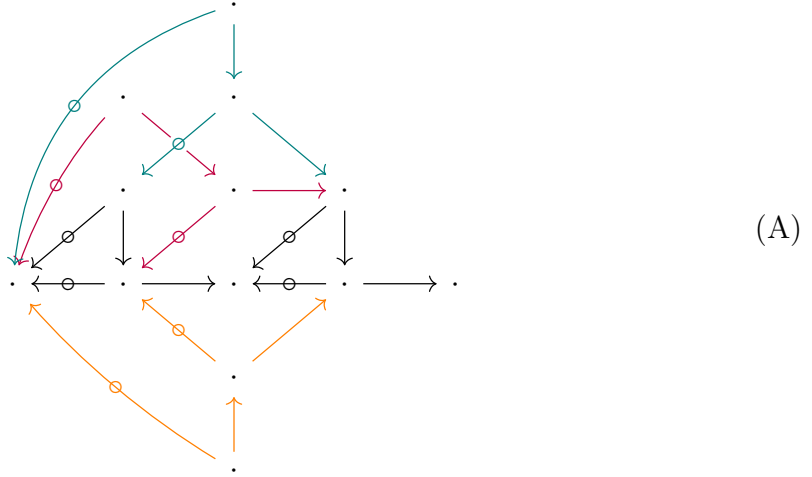
\[ \label{fig intermediate span comp well-defined} \begin{tikzcd}[]
	&& \cdot \\
	& \cdot & \cdot \\
	& \cdot & \cdot & \cdot \\
	\cdot & \cdot & \cdot & \cdot & \cdot \\
	&& \cdot \\
	&& \cdot
	\arrow[from=4-2, to=4-1, "\circ" marking ]
	\arrow[from=4-2, to=4-3]
	\arrow[from=3-4, to=4-4]
	\arrow[from=4-4, to=4-3, "\circ" marking ]
	\arrow[from=4-4, to=4-5]
	\arrow[from=3-2, to=4-2]
	\arrow[from=3-2, to=4-1, "\circ" marking ]
	\arrow[from=3-4, to=4-3, "\circ" marking ]
	\arrow[from=2-3, to=3-4, color = teal]
	\arrow[from=3-3, to=4-2, "\circ" marking, color = purple ]
	\arrow[from=3-3, to=3-4, color = purple]
	\arrow[from=5-3, to=4-2, "\circ" marking , color = orange ]
	\arrow[from=5-3, to=4-4, color = orange ]
	\arrow[from=6-3, to=5-3, color = orange ]
	\arrow[from=1-3, to=2-3, color = teal]
	\arrow[from=2-2, to=3-3, color = purple]
	\arrow[from=2-3, to=3-2, crossing over, "\circ" marking, color = teal]
	\arrow[curve={height=6pt}, from=2-2, to=4-1, "\circ" marking, color = purple]
	\arrow[curve={height=30pt}, from=1-3, to=4-1, "\circ" marking , color = teal]
	\arrow[curve={height=-6pt}, from=6-3, to=4-1, "\circ" marking , color = orange ]
\end{tikzcd} \tag{A}\]

\noi To internalize this we define the maps that pick out each of the three spans and their composites by finding a corresponding cover \begin{tikzcd} \tilde{U} \rar["/" marking, "\tilde{u}" near start] & \bar{U}\end{tikzcd}. Two of the spans can be given in terms of the composition, $\sigma_\circ$, on the cover $U$ but the comparison span needs a finer covering to witness applying the Ore and weak composition conditions to arrows from both of the first two spans. Denote the comparison pair of composable spans by $\gamma$ and define it by the universal property in the following pullback diagram.

\[ \label{dgm intermediate compblespan} 
\begin{tikzcd}[column sep = large, row sep = large]
 \bar{U} \ar[d, "/" marking, "\bar{u}"' near end] \ar[r, "/" marking, "\bar{u}" near end] \ar[dr, dotted, "\gamma"] 
& \slb \tensor[_{t}]{\times}{_s} \slb \rar["p_1^2"] 
& \spn \tensor[_{t}]{\times}{_s} \spn \dar["\pi_1"] 
\\
 \slb \tensor[_{t}]{\times}{_s} \slb \dar[ "p_0^2"'] 
& \spn \tensor[_{t}]{\times}{_s} \spn \dar["\pi_0"'] \rar["\pi_1"] 
& \spn \dar["s"] 
\\
 \spn \tensor[_{t}]{\times}{_s} \spn \rar["\pi_0"'] 
&\spn \rar["t"'] 
& \mC_0 
\end{tikzcd} \tag{3}
\]

\noi The following diagram of covers shows how the intermediate span is constructed by a similar span-composition construction for $\gamma$. Note there is another way to do this by taking a pullback of the pairing map $(\pi_0 p_0, \pi_1 p_1) : \slb \tensor[_t]{\times}{_s} \slb \to \spn \tensor[_t]{\times}{_s}$ along $u : U \to \spn \tensor[_t]{\times}{_s} \spn$ and a refinement with the previous refinement of covers of $\spn \tensor[_t]{\times}{_s} \spn$ above, and then using the span composition $\sigma_\circ : U \to \spn$ to obtain the intermediate span $\textcolor{purple}{\sigma_\gamma}$ in Diagram (\ref{dgm diagram of covers to include composite of intermediate comp'ble spans}) below. Both approaches lead to the same result. 

\[ \label{dgm diagram of covers to include composite of intermediate comp'ble spans}
\begin{tikzcd}[column sep = large ]
 W_\circ \rar[rr, "(\pi_0 \pi_1 {,} \pi_0 \pi_2)"] 
&& W \times_{\mC_0} W 
& \\
\tilde{U} 
\dar[shift left = 1.5, bend left =50, "\sigma_1", color = teal ] 
\dar[color = purple, "\sigma_\gamma"']
\dar[shift right = 1.5, bend right= 50, "\sigma_0"' , color = orange] 
\uar[color = purple, "\omega_\gamma"]
\rar[rr, "/" marking , "\tilde{u}_0" near start]
 &
 & \tilde{U}_0
 \dar[color = purple, "\theta_\gamma"']
 \uar[color = purple, "(\theta_\gamma \pi_0 \pi_0 {, } \tilde{u}_1 \gamma \pi_0 \pi_0) "'] 
 \rar["/" marking , "\tilde{u}_1" near start] 
 & \bar{U} 
 \dar[color = purple, "(\gamma \pi_0 \pi_1 {, }\gamma \pi_1 \pi_0)"] 
 \\
 \spn 
 &
 & W_\square 
 \rar["(\pi_0 \pi_1 {,} \pi_1 \pi_1)"'] 
 & \mC_1 \tensor[_{t}]{\times}{_{wt}} W 
\end{tikzcd} \tag{$\star$}
\]

\noi The left and right curved arrows, $\textcolor{orange}{\sigma_0}$ and $\textcolor{teal}{\sigma_1}$, into $\spn$ in the bottom left corner are defined by applying the composite of spans, $\sigma_\circ$, to the composable spans given by applying $p_0^2$ and $p_1^2$ to the pair of composable sailboats. Since $\sigma_\circ$ is only defined on $U$ we need to pass through the appropriate cover. The colours in the previous diagram and following equations indicate which of the three different span compositions in Figure (\ref{fig intermediate span comp well-defined}) the arrows in the following equations are witnessing. 

\begin{align}\label{def sigma_0} 
\begin{split}
 \textcolor{orange}{\sigma_0} 
 &= \textcolor{orange}{\tilde{u} \pi_0 \pi_1 \sigma_\circ} \\
 &= \textcolor{orange}{\tilde{u} \pi_0 \pi_1 \big( \omega \pi_1 , (\omega \pi_0 \pi_0 {,} u_0 \theta \pi_1 \pi_0 {,} u \pi_1 \pi_1 )c \big)} 
 \end{split}
 \end{align}
 \noi and 
 \begin{align}\label{def sigma_1}
 \begin{split}
 \textcolor{teal}{\sigma_1 }
 &= \textcolor{teal}{\tilde{u} \pi_1 \pi_1 \sigma_\circ} \\
 &= \textcolor{teal}{\tilde{u} \pi_1 \pi_1 \big( \omega \pi_1 , (\omega \pi_0 \pi_0 {,} u_0 \theta \pi_1 \pi_0 {,} u \pi_1 \pi_1 )c \big)}.
 \end{split}
 \end{align}

\noi The arrow into $\spn$ on the bottom left side of the cover diagram is the universal map
\[ \textcolor{purple}{\sigma_\gamma} = \big( \omega_\gamma \pi_1 {,} (\omega_\gamma \pi_0 \pi_0 {,} \tilde{u}' \theta_\gamma \pi_1 \pi_0{,} \tilde{u} \bar{u} p_1^2 \pi_1 \pi_1 )c \big).\]

\noi The data necessary to construct witnessing sailboats for the equivalences between the pairs of spans $\textcolor{orange}{\sigma_0}$, $\textcolor{teal}{\sigma_1}$, and $\textcolor{purple}{ \sigma_\gamma}$ can be obtained by applying the Ore condition, followed by the diagram-extension twice, and then weak composition three times. Internally this corresponds to a chain of six covers and lifts. All of this is color-coded below using \textcolor{olive}{olive} and \textcolor{brown}{brown} for the Ore condition and \textcolor{cyan}{cyan} and \textcolor{violet}{violet} for the zippering and weak composition step(s) that follow. Note that in both cases the first zippering is done to parallel pairs of composites that can be post-composed by the left leg of the bottom left span. The second zipper is done to parallel pairs of composites that can be post-composed with the left leg of the bottom right span in the pair of composale sailboats. Weak composition is then applied three times in to get comparison spans, $\textcolor{cyan}{\sigma_{0, \gamma}}$ and $\textcolor{violet}{\sigma_{1,\gamma}}$, whose left legs are in $W$.

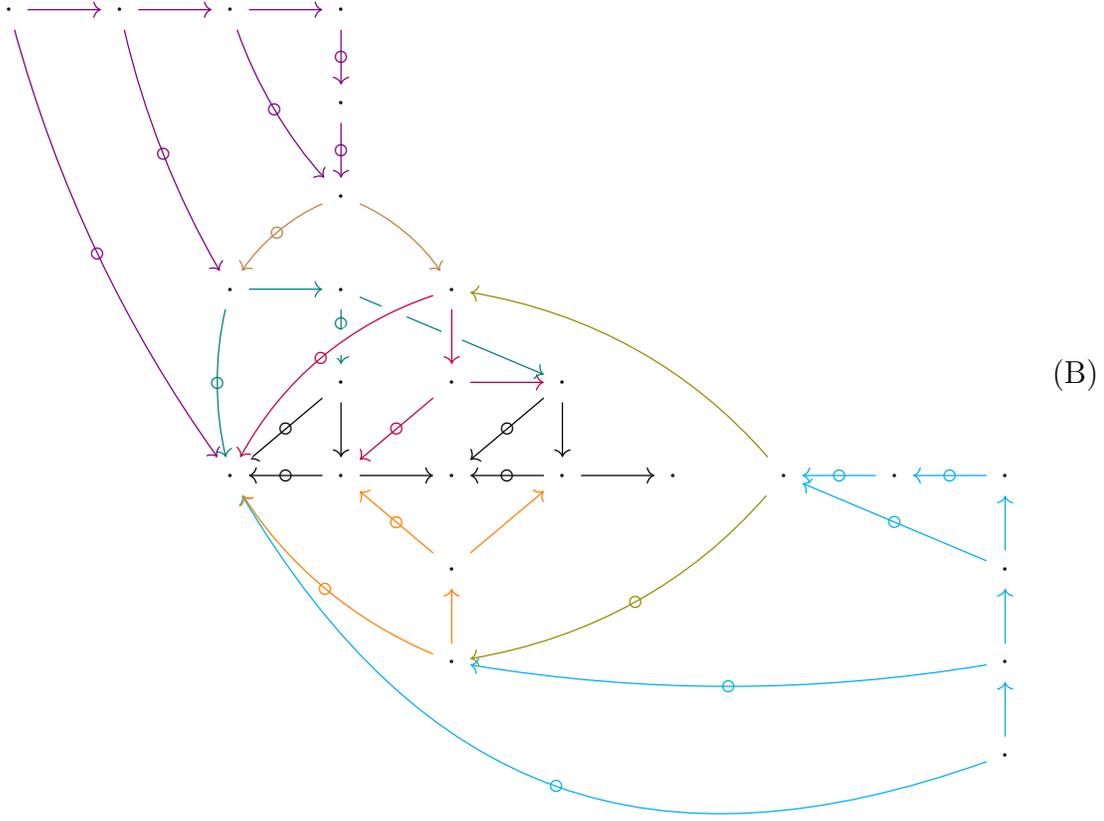
\[ \label{fig Ore + zip + w.c. showing span comp is well def wrt equiv reln}
\begin{tikzcd}
	\cdot & \cdot & \cdot & \cdot \\
	&&& \cdot \\
	&&& \cdot \\
	&& \cdot & \cdot & \cdot \\
	&&& \cdot & \cdot & \cdot \\
	{} && \cdot & \cdot & \cdot & \cdot & \cdot & \cdot & \cdot & \cdot \\
	&&&& \cdot &&&&& \cdot \\
	&&&& \cdot &&&&& \cdot \\
	&&&&&&&&& \cdot
	\arrow[from=6-4, to=6-3,"\circ" marking]
	\arrow[from=6-4, to=6-5]
	\arrow[from=6-6, to=6-5, "\circ" marking]
	\arrow[from=6-6, to=6-7]
	\arrow[from=5-6, to=6-6]
	\arrow[from=5-6, to=6-5, "\circ" marking]
	\arrow[from=5-5, to=6-4, purple, "\circ" marking]
	\arrow[from=5-5, to=5-6, purple]
	\arrow[curve={height=-90pt}, from=9-10, to=6-3, cyan, "\circ" marking]
	\arrow[from=9-10, to=8-10, cyan]
	\arrow[from=8-10, to=7-10, cyan]
	\arrow[from=7-10, to=6-10, cyan]
	\arrow[from=6-10, to=6-9, cyan, "\circ" marking]
	\arrow[from=6-9, to=6-8, cyan, "\circ" marking]
	\arrow[from=7-10, to=6-8, cyan, "\circ" marking]
	\arrow[from=7-5, to=6-4, orange, "\circ" marking]
	\arrow[from=7-5, to=6-6, orange]
	\arrow[from=8-5, to=7-5, orange]
	\arrow[curve={height=-12pt}, from=8-5, to=6-3, orange, "\circ" marking]
	\arrow[curve={height=-18pt}, from=6-8, to=8-5, olive, "\circ" marking]
	\arrow[curve={height=-12pt}, from=8-10, to=8-5, cyan, "\circ" marking]
	\arrow[from=1-4, to=2-4, violet, "\circ" marking]
	\arrow[curve={height=6pt}, from=4-3, to=6-3, teal, "\circ" marking]
	\arrow[from=4-4, to=5-6, teal]
	\arrow[from=4-5, to=5-5, purple, crossing over]
	\arrow[from=4-3, to=4-4, teal]
	\arrow[from=5-4, to=6-3, "\circ" marking]
	\arrow[from=5-4, to=6-4]
	\arrow[from=4-4, to=5-4, teal, "\circ" marking, near start]
	\arrow[curve={height=15pt}, from=4-5, to=6-3, purple, "\circ" marking, crossing over]
	\arrow[curve={height=18pt}, from=6-8, to=4-5, olive]
	\arrow[from=2-4, to=3-4, violet, "\circ" marking]
	\arrow[curve={height=6pt}, from=3-4, to=4-3, brown, "\circ" marking]
	\arrow[curve={height=-6pt}, from=3-4, to=4-5, brown]
	\arrow[curve={height=6pt}, from=1-3, to=3-4, violet, "\circ" marking]
	\arrow[curve={height=6pt}, from=1-2, to=4-3, violet, "\circ" marking]
	\arrow[curve={height=12pt}, from=1-1, to=6-3, violet, "\circ" marking]
	\arrow[from=1-1, to=1-2, violet]
	\arrow[from=1-2, to=1-3, violet]
	\arrow[from=1-3, to=1-4, violet]
\end{tikzcd} \tag{B}
\]

\noi Internally this is given by finding a cover \begin{tikzcd} \hat{U} \rar["/" marking, "\hat{u}" near start] & \tilde{U}\end{tikzcd} that witnesses application of the Ore condition, zippering, and weak composition in that order. The first three covers, $\hat{u}_5, \hat{u}_4$, and $\hat{u}_3$, witness the Ore condition being to each of the two cospans and the two applications of zippering that follow from each Ore-square.

\[\label{cover dgms for showing span comp well def : Ore + zip} 
\begin{tikzcd}[column sep = large]
& \cP(\mC) \rar["\pi_1"] 
& \cP_{cq}(\mC)
&
\\
\hat{U}_3
\rar["/" marking, "\hat{u}_3" near start] 
\dar[shift left, "\delta_{\rho_0}", color = cyan] 
\dar[shift right, "\delta_{\rho_1}"', color = violet] 
&\hat{U}_4
\rar["/" marking, "\hat{u}_4" near start]
\uar[shift right, "\delta_{\lambda_0}"', color = cyan] 
\uar[shift left, "\delta_{\lambda_1}", color = violet] 
\dar[shift left, "\rho_0", color = cyan] 
\dar[shift right, "\rho_1"', color = violet] 
&\hat{U}_5
\rar["/" marking, "\hat{u}_5" near start] 
\dar[shift left, "\theta_{\gamma_0}", color = olive]
 \dar[shift right, "\theta_{\gamma_1}"', color = brown] 
 \uar[shift right, "\lambda_0"', color = cyan] 
 \uar[shift left, "\lambda_1", color = violet] 
&\tilde{U} 
\dar[shift right, "(\sigma_1 \pi_0 w {,} \sigma_\gamma \pi_0)"', color = brown]
 \dar[shift left, "(\sigma_0 \pi_0 w {,} \sigma_\gamma \pi_0 )", color = olive] 
\\ 
\cP(\mC) \rar["\pi_1"'] 
& \cP_{cq}(\mC)
& W_\square \rar["(\pi_0 \pi_1 {,} \pi_1 \pi_1)"'] 
& \mC_1 \tensor[_{t}]{\times}{_{wt}} W 
\end{tikzcd} \tag{$\star \star$}
\]

\noi The covers, $\hat{u}_2, \hat{u}_1,$ and $\hat{u}_0$, witness three applications of weak composition in each case as seen in the following continued sequence of covers:

\[\label{cover dgms for showing span comp well def : weak comp x3} 
\begin{tikzcd}[column sep = large]
W_\circ \rar[r, "(\pi_0 \pi_1 {,} \pi_0 \pi_2)"] 
& W \times_{\mC_0} W 
&
W_\circ \rar[r, "(\pi_0 \pi_1 {,} \pi_0 \pi_2)"] 
& W \times_{\mC_0} W 
& \\
\hat{U} 
\rar["/" marking, "\hat{u}_0" near start] 
\uar[shift right, "\omega_{0,0}"', color = cyan] 
\uar[shift left, "\omega_{1,0}", color = violet] 
\dar[shift left = 1.5,""] 
\dar[shift right= 1.5 ,""']
\dar[shift left=.5] 
\dar[shift right= .5] 
& \hat{U}_1
\rar["/" marking,"\hat{u}_1" near start]
\uar[shift right, "\omega_{0,0}'"', color = cyan] 
\uar[shift left, "\omega_{1,0}'", color = violet] 
\dar[shift left, "\omega_{0,1}", color = cyan] 
\dar[shift right, "\omega_{1,1}"', color = violet] 
& \hat{U}_2
\rar["/" marking,"\hat{u}_2" near start]
\uar[shift right, ""', color = cyan, "\omega_{0, 2}"'] 
\uar[shift left, "", color = violet, "\omega_{1, 2}"] 
\dar[shift right, ""', color = violet, "\omega_{1, 1}'"'] 
\dar[shift left, "", color = cyan, "\omega_{0, 1}'"] 
& \hat{U}_3
\uar[shift right, ""', color = cyan, "\omega_{0,2}'"'] 
\uar[shift left, "", color = violet, "\omega_{1,2}'"] 
 \\ 
\slb &
W_\circ \rar[r, "(\pi_0 \pi_1 {,} \pi_0 \pi_2)"'] 
& W \times_{\mC_0} W 
& 
\end{tikzcd} \tag{$\star \star \star$}
\]

\noi The diagrams above will allow us to extract four sailboats that relate the three spans above through two new intermediate spans, $\textcolor{cyan}{\sigma_{0,\gamma}}$ and $\textcolor{violet}{\sigma_{1,\gamma}}$. All of these will be defined after we justify the maps in Diagrams (\ref{cover dgms for showing span comp well def : Ore + zip}) and (\ref{cover dgms for showing span comp well def : weak comp x3}) above. The maps, $\textcolor{cyan}{\lambda_0}$ and $\textcolor{violet}{\lambda_1}$ are induced by maps $\textcolor{cyan}{\lambda_0'}$ and $\textcolor{violet}{\lambda_1'}$ which pick out two parallel pairs of arrows in $\mC$ along with a post-composable arrow in $W$ that coequalizes them (in $\mC$). The following figure serves as a guide to defining these.

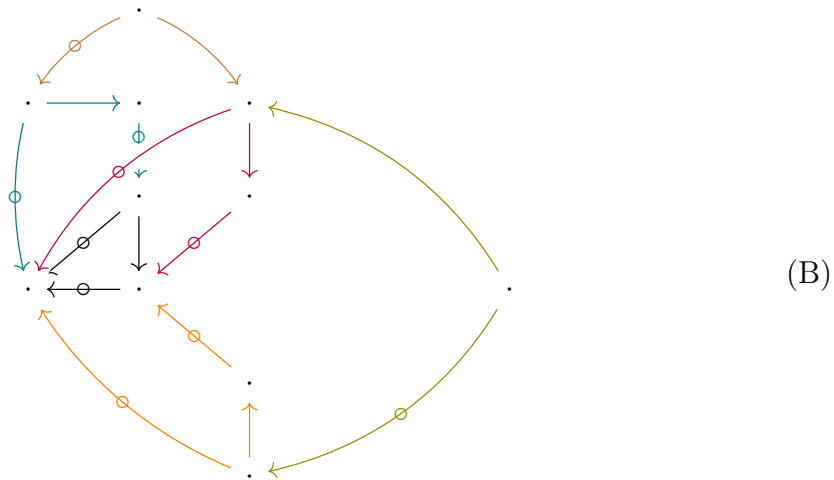
\[\label{fig Ore-squares for showing span comp well def up to equiv}
\begin{tikzcd}
	& \cdot \\
	 \cdot & \cdot & \cdot \\
	& \cdot & \cdot \\
	 \cdot & \cdot & & & & \cdot \\
	&& \cdot &&& \\
	&& \cdot &&& \\
	&&&&&
	\arrow[from=4-2, to=4-1,"\circ" marking]
	\arrow[from=3-3, to=4-2, purple, "\circ" marking]
	\arrow[from=2-3, to=3-3, purple]
	\arrow[from=5-3, to=4-2, orange, "\circ" marking]
	\arrow[from=6-3, to=5-3, orange]
	\arrow[curve={height=-12pt}, from=6-3, to=4-1, orange, "\circ" marking]
	\arrow[curve={height=-18pt}, from=4-6, to=6-3, olive, "\circ" marking]
	\arrow[curve={height=6pt}, from=2-1, to=4-1, teal, "\circ" marking]
	\arrow[from=2-1, to=2-2, teal]
	\arrow[from=3-2 , to=4-1, "\circ" marking]
	\arrow[from=3-2, to=4-2]
	\arrow[from=2-2, to=3-2, teal, "\circ" marking, near start]
	\arrow[curve={height=15pt}, from=2-3, to=4-1, purple, "\circ" marking, crossing over]
	\arrow[curve={height=18pt}, from=4-6, to=2-3, olive]
	\arrow[curve={height=6pt}, from=1-2, to=2-1, brown, "\circ" marking]
	\arrow[curve={height=-6pt}, from=1-2, to=2-3, brown]
\end{tikzcd} \tag{B} \]

\noi The explicit definitions are obtained using the universal property of the equalizer $\cP_{cq}(\mC)$. This is done similarly as in Lemma~\ref{lem defining c'}, by specifying two maps (on the left below) that also equalize the parallel pair on the right below. 

\[ \begin{tikzcd}[column sep = large]
\hat{U}_5 \rar[shift right, cyan, "\lambda_0'"'] \rar[shift left, violet, "\lambda_1'"] & P(\mC) \times_{\mC_0} W \rar[shift left, "(\pi_0\pi_0 {,} \pi_1 w )c"] \rar[shift right, "(\pi_0\pi_1 {,} \pi_1 w )c"'] & \mC_1 
\end{tikzcd}\]

\noi The maps \textcolor{cyan}{$\lambda_0 '$} and \textcolor{violet}{$\lambda_1 '$} are uniquely determined in a similar fashion to $\delta_0'$ in Lemma~\ref{lem defining c'}, namely by descending through the covers above and expanding both sides of the Ore-square equations witnessed. The calculations are lengthy and technical and can be found in Lemma~\ref{appendix lem witness ore squares of span comp well def wrt equiv result (appendix)} of Section \ref{App.S. Defining Span Comp on Reps}. The proof shows that the map $\textcolor{cyan}{\lambda_0'}$ is uniquely determined by the projections 

\begin{align*}
  \textcolor{cyan}{\lambda_0'} \pi_1 
  &= \hat{u}_5 \tilde{u} \pi_0 \pi_0 p_0^2 \pi_0 \pi_0  \\
  \textcolor{cyan}{\lambda_0'} \pi_0 \pi_0 
  &= (\textcolor{olive}{\theta_{\gamma_0}} \pi_0 \pi_0 w \
{,} \ \hat{u}_5 \tilde{u} \textcolor{orange}{\pi_0 \pi_1 \omega} \pi_0 \pi_0 \
{,} \ \hat{u}_5 \tilde{u} \textcolor{orange}{\pi_0 \pi_1 u_0 \theta} \pi_0 \pi_0 w)c \\
  \textcolor{cyan}{\lambda_0'} \pi_0 \pi_1 
  &=(\textcolor{olive}{\theta_{\gamma_0}} \pi_1 \pi_0 w \
{,} \ \hat{u}_5 \textcolor{purple}{\omega_\gamma} \pi_0 \pi_0 \
{,} \ \hat{u}_5 \tilde{u}' \textcolor{purple}{\theta_\gamma} \pi_0 \pi_0 w )c
\end{align*}

\noi and that the equalizer diagram

\[\begin{tikzcd}[column sep = huge]
\cP_{cq} \rar[tail, "\iota_{cq}"] & P(\mC) \tensor[_t]{\times}{_{ws}} W 
\rar[shift left, "(\pi_0 \pi_0 {,} \pi_1 w ) c"] 
\rar[shift right, "(\pi_0 \pi_1 {,} \pi_1 w) c"'] 
& \mC_1\\
\hat{U}_5 \ar[ur, cyan, "\lambda_0'"'] 
\uar[cyan, "\lambda_0"] 
\end{tikzcd}\]

\noi commutes in $\cE$. The map $\textcolor{violet}{\lambda_1'}$, inducing $\textcolor{violet}{\lambda_1}$ in Diagram (\ref{cover dgms for showing span comp well def : Ore + zip}), such that the equalizer diagram

\[\begin{tikzcd}[column sep = huge]
\cP_{cq} \rar[tail, "\iota_{cq}"] & P(\mC) \tensor[_t]{\times}{_{ws}} W \rar[shift left, "(\pi_0 \pi_0 {,} \pi_1 w ) c"] 
\rar[shift right, "(\pi_0 \pi_1 {,} \pi_1 w) c"'] 
& \mC_1\\
\hat{U}_5 \ar[ur, violet, "\lambda_1'"'] 
\uar[violet, "\lambda_1"] 
\end{tikzcd}\]

\noi commutes in $\cE$ can be derived by a similar computation to the one in Lemma~\ref{appendix lem witness ore squares of span comp well def wrt equiv result (appendix)} in the appendix where one replaces $\textcolor{olive}{\theta_{\gamma_0}}$ with $\textcolor{brown}{\theta_{\gamma_1}}$ and factors through the cover $\bar{U}_1$ instead of $\bar{U}_0$ to access the arrows used to construct the span $\textcolor{teal}{\sigma_1}$. The map $\textcolor{violet}{\lambda_1'}$ is uniquely determined by the following maps $\tilde{U}_4 \to \mC_1$: 

\begin{align*} 
\textcolor{violet}{\lambda_1'} \pi_1 &= \hat{u}_5 \tilde{u} \pi_0 \pi_0 p_0^2 \pi_0 \pi_0 , \\
\textcolor{violet}{\lambda_1'} \pi_0 \pi_0 &= (\textcolor{brown}{\theta_{\gamma_1}} \pi_0 \pi_0 w \
{,} \ \hat{u}_5 \tilde{u} \textcolor{teal}{\pi_1 \pi_1 \omega} \pi_0 \pi_0 \
{,} \ \hat{u}_5 \tilde{u} \textcolor{teal}{\pi_1 \pi_1 u_0 \theta} \pi_0 \pi_0 w \
{,} \ \hat{u}_5 \tilde{u} \bar{u} \pi_0 \pi_0 \pi_0 \pi_0 )c,\\
\textcolor{violet}{\lambda_1'} \pi_0 \pi_1 &=( \textcolor{brown}{\theta_{\gamma_1}} \pi_1 \pi_0 w \
{,} \ \hat{u}_5 \textcolor{purple}{\omega_\gamma} \pi_0 \pi_0 \
{,} \ \hat{u}_5 \tilde{u}' \textcolor{purple}{\theta_\gamma} \pi_0 \pi_0 w )c 
\end{align*}

\noi Applying \textbf{In.Frc(4)} along with Lemma~\ref{lem extending Int.Frc. covers} twice gives two covers and two lifts, one for each of $\textcolor{cyan}{\lambda_0}$ and $\textcolor{violet}{\lambda_1}$. Since covers are stable under pullback and composition we can take a common refinement by pulling one cover back along the other and get the cover and two lifts

\[ \begin{tikzcd}
\cP(\mC) & \\
\hat{U}_4 \uar[cyan, shift right, "\delta_{\lambda_0}"']
\uar[violet, shift left, "\delta_{\lambda_1}"]\rar["\scriptstyle{/}" marking, "\hat{u}_4" near start] & \hat{U}_5
\end{tikzcd}\]

\noi in Diagram (\ref{cover dgms for showing span comp well def : Ore + zip}).

The \textcolor{violet}{violet} and \textcolor{cyan}{cyan} arrows in the figure below are witnessed by post-composing the maps, \textcolor{violet}{$\delta_{\lambda_1}$} and \textcolor{cyan}{$\delta_{\lambda_0}$}, with the projection $ \pi_0 : \cP \to \cP_{eq}(\mC)$. There are two pairs of parallel composites that begin at each of these arrows whose codomain is that of the vertical arrow in the second of the composable sailboats. These pairs are determined by the legs of the \textcolor{brown}{brown} and \textcolor{olive}{olive} spans respectively. As a consequence of commutativity of the \textcolor{teal}{teal} and \textcolor{purple}{purple} Ore squares along with the previous diagram extension, both parallel pairs are respectively coequalized after post-composing with the left leg of the bottom span in the second of the composable sailboats. The parallel pairs $\textcolor{cyan}{\rho_0}$ and $\textcolor{violet}{\rho_1}$ can be seen beginning at each $\bullet$ and ending at $\blacksquare$ in the following figure with their common coequalizing arrow in $W$ having domain $\blacksquare$: 

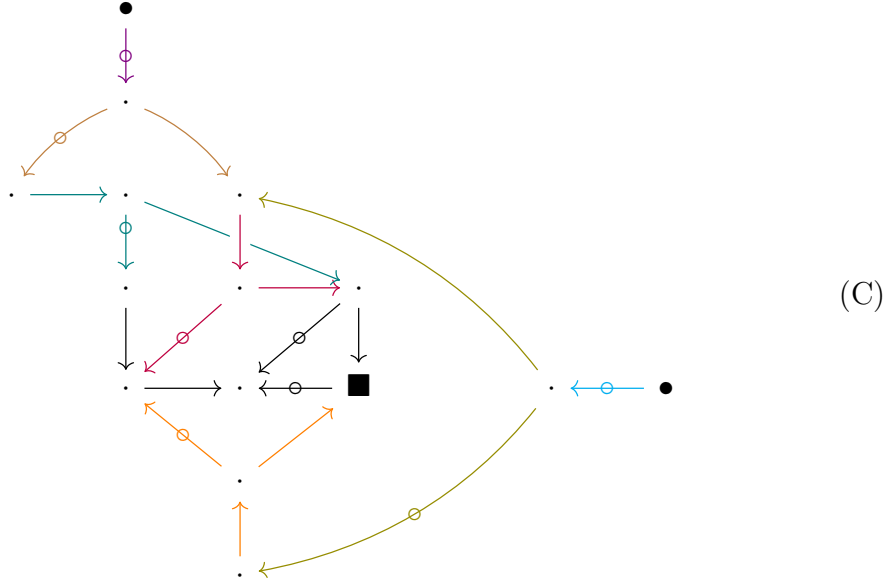
\[\label{fig first zips for span comp well def wrt equiv} \begin{tikzcd}
	 && & \\
	&&& \bullet \\
	&&& \cdot \\
	&& \cdot & \cdot & \cdot \\
	&&& \cdot & \cdot & \cdot \\
	&&& \cdot & \cdot & \blacksquare & & \cdot & \bullet & \\
	&&&& \cdot &&&&& \\
	&&&& \cdot &&&&& \\
	&&&&&&&&& 
	\arrow[from=6-4, to=6-5]
	\arrow[from=6-6, to=6-5, "\circ" marking]
	\arrow[from=5-6, to=6-6]
	\arrow[from=5-6, to=6-5, "\circ" marking]
	\arrow[from=5-5, to=6-4, purple, "\circ" marking]
	\arrow[from=5-5, to=5-6, purple]
	\arrow[from=6-9, to=6-8, cyan, "\circ" marking]
	\arrow[from=7-5, to=6-4, orange, "\circ" marking]
	\arrow[from=7-5, to=6-6, orange]
	\arrow[from=8-5, to=7-5, orange]
	\arrow[curve={height=-18pt}, from=6-8, to=8-5, olive, "\circ" marking]
	\arrow[from=4-4, to=5-6, teal]
	\arrow[from=4-5, to=5-5, purple, crossing over]
	\arrow[from=4-3, to=4-4, teal]
	\arrow[from=5-4, to=6-4]
	\arrow[from=4-4, to=5-4, teal, "\circ" marking, near start]
	\arrow[curve={height=18pt}, from=6-8, to=4-5, olive]
	\arrow[from=2-4, to=3-4, violet, "\circ" marking]
	\arrow[curve={height=6pt}, from=3-4, to=4-3, brown, "\circ" marking]
	\arrow[curve={height=-6pt}, from=3-4, to=4-5, brown]
\end{tikzcd} \tag{C}
\]

\noi Let $\hat{u}_{i;j} = \hat{u}_i \hat{u}_{i+1} ... \hat{u}_j$ for $0 \leq i < j \leq 5$ to make composition of covers a bit easier to read, where $\hat{u} = \hat{u}_{0;5}$. Internally, we use Figure (\ref{fig first zips for span comp well def wrt equiv}) as a blueprint for defining the maps $\textcolor{cyan}{\rho_0'}, \textcolor{violet}{\rho_1'} :\hat{U}_4 \to P_{cq}(\mC)$ in termds of parallel pair of arrows in $\mC$, $\hat{U}_4 \to P(\mC)$, which are coequalized (in $\mC$) by an arrow $\hat{U}_4 \to W$: 

\[ \begin{tikzcd}[column sep = large]
\hat{U}_4 \rar[shift right, cyan, "\rho_0'"'] \rar[shift left, violet, "\rho_1'"] & P(\mC) \times_{\mC_0} W \rar[shift left, "(\pi_0\pi_0 {,} \pi_1 w )c"] \rar[shift right, "(\pi_0\pi_1 {,} \pi_1 w )c"'] & \mC_1 
\end{tikzcd}\]

\noi completely determined by the following maps: 

\begin{align*}
\textcolor{cyan}{\rho_0'} \pi_1 
&= \hat{u}_{4;5} \tilde{u} \bar{u} p_0^2 \pi_1 \pi_0, \\
\textcolor{cyan}{\rho_0'} \pi_0 \pi_0 
&= (
\textcolor{cyan}{\delta_{\lambda_0}} \pi_0 \iota_{eq} \pi_0w , 
\hat{u}_4 \textcolor{olive}{\omega_{\gamma_0} } \pi_1 \pi_0 ,  
\hat{u}_{4;5} \textcolor{purple}{\omega_\gamma} \pi_0 \pi_0 , 
\hat{u}_{4;5} \tilde{u}' \textcolor{purple}{\theta_\gamma} \pi_1 \pi_0 , 
\hat{u}_{4;5} \tilde{u} \bar{u} \pi_1 \pi_0 \pi_0 \pi_0 
)c, \\
\textcolor{cyan}{\rho_0'} \pi_0 \pi_1 
&= (
\textcolor{cyan}{\delta_{\lambda_0}} \pi_0 \iota_{eq} \pi_0 w , 
\hat{u}_4 \textcolor{olive}{\omega_{\gamma_0} } \pi_0 \pi_0 , 
\hat{u}_{4;5} \tilde{u} \textcolor{orange}{\pi_0 \pi_1 \omega} \pi_0 \pi_0 ,
\hat{u}_{4;5} \tilde{u} \textcolor{orange}{\pi_0 \pi_1 u_0 \theta} \pi_1 \pi_0 )c
\end{align*} 
\noi Another lengthy but straight forward computation that comes down to the definition of $\textcolor{cyan}{\delta_{\lambda_0}} \iota_{eq} : \hat{U}_4 \to \cP_{eq}(\mC)$ and the Ore-square can be found in Lemma~\ref{appendix lem def for rho_0' for span comp well def wrt equiv result } of the appendix and shows that 

\begin{align*}
 (\textcolor{cyan} {\rho_0'} \pi_0 \pi_0 , \textcolor{cyan} {\rho_0'} \pi_1) c 
&= (\textcolor{cyan} {\rho_0'} \pi_0 \pi_1 , \textcolor{cyan} {\rho_0'} \pi_1) c
\end{align*}\

\noi implying the equalizer diagram

\begin{center}
\begin{tikzcd}[column sep = large]
\cP_{cq}(\mC) \rar[tail, "\iota_{cq}"] & P(\mC) \tensor[_t]{\times}{_{ws}} W \rar[shift left, "(\pi_0 \pi_0 {,} \pi_1 w) c"] \rar[shift right, "(\pi_0 \pi_1{ , }\pi_1 w) c"'] & \mC_1 \\
\tilde{U}_3 \ar[ur, cyan, "\rho_0'"'] 
\uar[dotted, cyan, "\rho_0"] & & 
\end{tikzcd}
\end{center}

\noi commutes in $\cE$. The map inducing \textcolor{violet}{$\rho_1$} is a map $\textcolor{violet}{\rho_1'} :\hat{U}'' \to P_{ceq}(\mC)$ similarly defined but replacing $\textcolor{cyan}{\delta_{\lambda_0}}$ with $\textcolor{violet}{\delta_{\lambda_1}}$ and $\textcolor{olive}{\theta_{\gamma_0} }$ with $\textcolor{brown}{\theta_{\gamma_1} }$. It is uniquely determined by 

\begin{align*}
   \textcolor{violet}{\rho_1'} \pi_1 
   &= \hat{u}_{4;5} \tilde{u} \bar{u} p_0^2 \pi_1 \pi_0,
\end{align*}

   \begin{align*}
   \textcolor{violet}{\rho_1'} \pi_0 \pi_0 
   &= 
   (\textcolor{violet}{\delta_{\lambda_1}} \pi_0 \iota_{eq} \pi_0w , \
  \hat{u}_4 \textcolor{brown}{\omega_{\gamma_1} } \pi_1 \pi_0 , \\
  & \ \ \ \ 
  \hat{u}_{4;5} \textcolor{purple}{\omega_\gamma} \pi_0 \pi_0 , \
  \hat{u}_{4;5} \tilde{u}' \textcolor{purple}{\theta_\gamma} \pi_1 \pi_0 ,\ 
  \hat{u}_{4;5} \tilde{u} \bar{u} \pi_1 \pi_0 \pi_0 \pi_0 
  )c, 
  \end{align*}
  \begin{align*}
  \textcolor{violet}{\rho_1'} \pi_0 \pi_1 
  &= (
\textcolor{violet}{\delta_{\lambda_1}} \pi_0 \iota_{eq} \pi_0 w , \
\hat{u}_4 \textcolor{brown}{\theta_{\gamma_1} } \pi_0 \pi_0 , \\
& \ \ \ \ 
\hat{u}_{4;5} \tilde{u} \textcolor{teal}{\pi_1 \pi_1 \omega} \pi_0 \pi_0 , \ 
\hat{u}_{4;5} \tilde{u} \textcolor{teal}{\pi_1 \pi_1 u_0 \theta} \pi_1 \pi_0 , \
\hat{u}_{4;5} \tilde{u} \bar{u} \pi_1 \pi_0 \pi_0 \pi_0 
)c
\end{align*} 

\noi and a similar lengthy but straightforward computation found in Lemma~\ref{appendix lem def for rho_1' for span comp well def wrt equiv result} of the appendix shows that
\begin{align*}
 (\textcolor{violet} {\rho_1'} \pi_0 \pi_0 , \textcolor{violet} {\rho_1'} \pi_1) c 
&= (\textcolor{violet} {\rho_1'} \pi_0 \pi_1 , \textcolor{violet} {\rho_1'} \pi_1) c .
\end{align*}\

\begin{center}
\begin{tikzcd}[column sep = large]
\cP_{cq}(\mC) \rar[tail, "\iota_{cq}"] & P(\mC) \tensor[_t]{\times}{_{ws}} W \rar[shift left, "(\pi_0 \pi_0 {,} \pi_1 w) c"] \rar[shift right, "(\pi_0 \pi_1{ , }\pi_1 w) c"'] & \mC_1 \\
\tilde{U}_3 
\ar[ur, violet, "\rho_1"'] 
\uar[dotted, violet, "\rho_1"] & & 
\end{tikzcd}
\end{center}

\noi \noi Applying \textbf{In.Frc(4)} along with Lemma~\ref{lem extending Int.Frc. covers} twice gives two covers and two lifts, one for each of $\textcolor{cyan}{\rho_0}$ and $\textcolor{violet}{\rho_1}$. A common refinement given by pulling one cover back along the other gives the cover and two lifts

\[ \begin{tikzcd}
\hat{U}_3
\dar[cyan, shift left, "\delta_{\rho_0}"]
\dar[violet, shift right, "\delta_{\rho_1}"']
\rar["\scriptstyle{/}" marking, "\hat{u}_3" near start] & \hat{U}_4\\
\cP(\mC) & 
\end{tikzcd}\]

\noi in Diagram (\ref{cover dgms for showing span comp well def : Ore + zip}). It remains to define the `weakly-composable maps' being picked out by $\omega_{i,j}'$ for $i = 0 , 1$ and $j = 0 ,1,2$ in Diagram (\ref{cover dgms for showing span comp well def : weak comp x3}). First we have maps $\hat{U}_3 \to W \tensor[_{wt}]{\times}{_{ws}} W$ given by 

\begin{align*}
\textcolor{cyan}{\omega_{0,2}'} 
&= 
(\textcolor{cyan}{\delta_{\rho_0}} \pi_0 \iota_{eq} \pi_0 , 
 \hat{u}_3 \textcolor{cyan}{\delta_{\lambda_0}} \pi_0 \iota_{eq} \pi_0 ) 
 &
\textcolor{violet}{\omega_{1,2}'} 
&= 
(\textcolor{violet}{\delta_{\rho_1}} \pi_0 \iota_{eq} \pi_0 , 
\hat{u}_3 \textcolor{violet}{\delta_{\lambda_1}} \pi_0 \iota_{eq} \pi_0 ) 
\end{align*}

\noi and by applying \textbf{In.Frc(2)} and Lemma~\ref{lem extending Int.Frc. covers} twice and taking a common refinement of covers we get the cover and lift 

\[ \begin{tikzcd}
W_\circ &\\
\hat{U}_2
\rar["/" marking,"\hat{u}_2" near start]
\uar[shift right, ""', color = cyan, "\omega_{0, 2}"'] 
\uar[shift left, "", color = violet, "\omega_{1, 2}"] 
& \hat{U}_3
\end{tikzcd}\]

\noi in Diagram (\ref{cover dgms for showing span comp well def : weak comp x3}). Next we have maps $\hat{U}_2 \to W \tensor[_{wt}]{\times}{_{ws}} W$ given by 

\begin{align*}
\textcolor{cyan}{\omega_{0,1}'} 
&= 
(\textcolor{cyan}{\omega_{0,2}} \pi_1, 
 \hat{u}_{2;4} \textcolor{olive}{\theta_{\gamma_0}} \pi_0 \pi_0) 
\\
&= 
\big( 
( \textcolor{cyan}{\omega_{0,2}} \pi_0 \pi_0 , 
\hat{u}_2 \textcolor{cyan}{\omega_{0,2}'}c )c , 
 \hat{u}_{2;4} \textcolor{olive}{\theta_{\gamma_0}} \pi_0 \pi_0 ) 
\\
&= 
\big(
(\textcolor{cyan}{\omega_{0,2}} \pi_0 \pi_0 ,
 \hat{u}_2 \textcolor{cyan}{\delta_{\rho_0}} \pi_0 \iota_{eq} \pi_0 , 
 \hat{u}_{2;3} \textcolor{cyan}{\delta_{\lambda_0}} \pi_0 \iota_{eq} \pi_0 )c , \ 
 \hat{u}_{2;4} \textcolor{olive}{\theta_{\gamma_0}} \pi_0 \pi_0 
\big) 
\end{align*}

\noi and 
\begin{align*}
\textcolor{violet}{\omega_{1,1}'} 
&= 
(\textcolor{violet}{\omega_{1,2}} \pi_1
, \hat{u}_{2;4} \textcolor{brown}{\theta_{\gamma_1}} \pi_0 \pi_0 ) 
\\
&= 
\big( 
(\textcolor{violet}{\omega_{1,2}} \pi_0 \pi_0 ,
\hat{u}_2 \textcolor{violet}{\omega_{1,2}'}c )c \ , \
\hat{u}_{2;4} \textcolor{brown}{\theta_{\gamma_1}} \pi_0 \pi_0 ) \\
&= 
\big(
(\textcolor{violet}{\omega_{1,2}} \pi_0 \pi_0 ,
\hat{u}_2\textcolor{violet}{\delta_{\rho_1}} \pi_0 \iota_{eq} \pi_0 , 
\hat{u}_{2;3} \textcolor{violet}{\delta_{\lambda_1}} \pi_0 \iota_{eq} \pi_0 )c , \
\hat{u}_{2;4} \textcolor{brown}{\theta_{\gamma_1}} \pi_0 \pi_0 
\big) 
\end{align*}

\noi that, by applying \textbf{In.Frc(2)} and Lemma~\ref{lem extending Int.Frc. covers} twice and taking a common refinement of covers, gives the cover and lifts

\[\begin{tikzcd}
\hat{U}_1
\rar["/" marking,"\hat{u}_1" near start]
\dar[shift left, "\omega_{0,1}", color = cyan] 
\dar[shift right, "\omega_{1,1}"', color = violet] 
& \hat{U}_2 \\
W_\circ & 
\end{tikzcd}\]

\noi in Diagram (\ref{cover dgms for showing span comp well def : weak comp x3}). Finally, we have maps $\hat{U}_1 \to W \tensor[_{wt}]{\times}{_{ws}} W$ given by 

\begin{align*}
\textcolor{cyan}{\omega_{0,0}'} 
&= 
(\textcolor{cyan}{\omega_{0,1}} \pi_1, 
\hat{u}_{1;5} \textcolor{orange}{\sigma_0} \pi_0 \pi_0 ) \\
&= 
\big(
( \textcolor{cyan}{\omega_{0,1}} \pi_0 \pi_0 , 
\hat{u}_1 \textcolor{cyan}{\omega_{0,1}'}c )c , \ 
\hat{u}_{1;5} \textcolor{orange}{\sigma_0} \pi_0 \pi_0 
\big) \\
&= 
\big(
( \textcolor{cyan}{\omega_{0,1}} \pi_0 \pi_0 , 
 \hat{u}_1\textcolor{cyan}{\omega_{0,2}} \pi_0 \pi_0 ,
 \hat{u}_{1;2} \textcolor{cyan}{\delta_{\rho_0}} \pi_0 \iota_{eq} \pi_0 , 
 \hat{u}_{1;3} \textcolor{cyan}{\delta_{\lambda_0}} \pi_0 \iota_{eq} \pi_0 ,
 \hat{u}_{1;4} \textcolor{olive}{\theta_{\gamma_0}} \pi_0 \pi_0 )c , \\
 & \ \ \ \ 
\hat{u}_{1;5} \textcolor{orange}{\sigma_0} \pi_0 \pi_0 
\big) 
\end{align*}
\noi and 
\begin{align*}
\textcolor{violet}{\omega_{1,0}'} 
&=
(\textcolor{violet}{\omega_{1,1}} \pi_1,
\hat{u}_{1;5} \textcolor{teal}{\sigma_1} \pi_0 ) \\
&= 
\big( 
(\textcolor{violet}{\omega_{1,1}} \pi_0 \pi_0 ,
\hat{u}_1 \textcolor{violet}{\omega_{1,1}'}c )c  , \
\hat{u}_{1;5} \textcolor{teal}{\sigma_1} \pi_0 \big) \\
&= 
\big( 
(\textcolor{violet}{\omega_{1,1}} \pi_0 \pi_0 , 
\hat{u}_1\textcolor{violet}{\omega_{1,2}} \pi_0 \pi_0 ,
\hat{u}_{1;2}\textcolor{violet}{\delta_{\rho_1}} \pi_0 \iota_{eq} \pi_0 , 
\hat{u}_{1;3} \textcolor{violet}{\delta_{\lambda_1}} \pi_0 \iota_{eq} \pi_0 ,
\hat{u}_{1;4} \textcolor{brown}{\theta_{\gamma_1}} \pi_0 \pi_0 )c , \\
& \ \ \ \ 
\hat{u}_{1;5} \textcolor{teal}{\sigma_1} \pi_0 \big) \\
\end{align*}

\noi that, by applying \textbf{In.Frc(2)} and Lemma~\ref{lem extending Int.Frc. covers} twice and taking a common refinement of covers, gives the cover and lifts

\[\begin{tikzcd}
W_\circ & \\
\hat{U} 
\rar["/" marking, "\hat{u}_0" near start] 
\uar[shift right, "\omega_{0,0}"', color = cyan] 
\uar[shift left, "\omega_{1,0}", color = violet] 
& \hat{U}_1 
\end{tikzcd}\]

\noi in Diagram (\ref{cover dgms for showing span comp well def : weak comp x3}). The object $\hat{U}$ witnesses five spans, $\hat{U} \to \spn$ related by four sailboats, $\hat{U} \to \slb$, via the covers, $\hat{u}_{0;j} : \hat{U} \to \hat{U}_{j+1}$, and lifts in Diagrams (\ref{dgm diagram of covers to include composite of intermediate comp'ble spans}), (\ref{cover dgms for showing span comp well def : Ore + zip}), and (\ref{cover dgms for showing span comp well def : weak comp x3}), and the covers and projections in Diagrams (\ref{dgm slb proj pb along span comp cover}) and \ref{dgm slb proj pb along span comp cover refinement}). The original three spans $\hat{u} \textcolor{orange}{ \sigma_0} , \hat{u} \textcolor{purple}{\sigma_\gamma} $, and $\hat{u} \textcolor{teal}{ \sigma_1}$ are immediate; and two intermediate spans, $\textcolor{cyan}{\sigma_{0, \gamma}}$ and $\textcolor{violet}{\sigma_{1,\gamma}}$, defined in technical Lemmas \ref{appendix lem def sigma_{0,gamma}} and \ref{appendix lem def sigma_{1,gamma}} in Section \ref{App.S. Defining Span Comp on Reps} of the appendix. Lemma~\ref{appendix lem defining sailboat varphi_0} in the same section shows that the sailboat $\varphi_0 : \hat{U} \to \slb$, defined by the pairing map

\begin{align*}
\varphi_0 
= 
\big( 
(
(
\mu_0 , \ 
\hat{u} \textcolor{orange}{\sigma_0} \pi_0
) , \
\textcolor{cyan}{\omega_{0,0}} \pi_1 
) , \
\hat{u} \textcolor{orange}{\sigma_0} \pi_1 
\big)
\end{align*}

\noi is well-defined, where

\[ \mu_0 =(\textcolor{cyan}{\omega_0} ,\
\hat{u}_{0;4} \textcolor{olive}{\theta_{\gamma_0}} \pi_0 \pi_0 
)c.\]

\noi The same lemma shows that $\varphi_0 : \hat{U} \to \slb$ relates the spans $\hat{u} \textcolor{orange}{\sigma_0}, \textcolor{cyan}{\sigma_{0,\gamma}} : \hat{U} \to \spn$ in the sense that

\begin{align*}
   \varphi_0 p_0 
   &=\varphi_0 (\pi_0 \pi_0 \pi_1 , \pi_1) 
   & 
   \varphi_0 p_1 
   & =\varphi_0 (\pi_0 \pi_1 , (\pi_0 \pi_0 \pi_0 , \pi_1)c) \\
   &= (\hat{u} \textcolor{orange}{\sigma_0} \pi_0 , \hat{u} \textcolor{orange}{\sigma_0} \pi_1 ) c 
   &&= ( \textcolor{cyan}{\omega_{0,0}} \pi_1 , \textcolor{cyan}{\sigma_{0,\gamma}} \pi_1 )\\
   &= \hat{u} \textcolor{orange}{\sigma_0} 
   &&= \textcolor{cyan}{\sigma_{0,\gamma}}, 
\end{align*} 

\noi and so 

\begin{align}\label{eq sigma_0 equiv sigma_0,gamma} 
  \hat{u} \textcolor{orange}{\sigma_0} q = \varphi_0 p_0 q = \varphi_0 p_1 q = \textcolor{cyan}{\sigma_{0,\gamma}} q . 
\end{align}

\noi Lemma~\ref{appendix lem defining sailboat varphi_{0,gamma}} of Section \ref{App.S. Defining Span Comp on Reps} in the appendix similarly shows that the sailboat, $\varphi_{0,\gamma} : \hat{U} \to \slb$, defined by 

\[ \varphi_{0 , \gamma} = 
\big( 
( 
( \mu_{0, \gamma} , \
\hat{u} \textcolor{purple}{\sigma_\gamma} \pi_0
) , \
\textcolor{cyan}{\omega_{0,0}} \pi_1 ) , \
\hat{u} \textcolor{purple}{\sigma_\gamma} \pi_1 
 \big)
\]
\noi where 
\[ \mu_{0, \gamma} = (\textcolor{cyan}{\omega_0} ,\
\hat{u}_{0;4} \textcolor{olive}{\theta_{\gamma_0}} \pi_1 \pi_0 )c \]
\noi is well-defined and relates $\textcolor{purple}{\sigma_\gamma}$ to $\textcolor{cyan}{\sigma_{0,\gamma}}$ in the sense that 
\begin{align*}
\varphi_{0, \gamma} p_0 &= \hat{u}\textcolor{purple}{\sigma_\gamma} 
& \varphi_{0, \gamma} p_1 &= \textcolor{cyan}{\sigma_{0,\gamma}}. 
\end{align*}
\noi This implies 
\begin{align}\label{eq sigma_0,gamma equiv sigma_gamma}
  \textcolor{cyan}{\sigma_{0,\gamma}} q
  = \varphi_{0, \gamma} p_1 q 
  = \varphi_{0, \gamma} p_0 q
  = \hat{u}\textcolor{purple}{\sigma_\gamma} q .
\end{align}

\noi By Lemma~\ref{appendix lem defining sailboat varphi_1}, the sailboat, $\varphi_1 : \hat{U} \to \slb$, defined by 

\begin{align*}
\varphi_1
= 
\big( 
(
(
\mu_1 , \ 
\hat{u} \textcolor{teal}{\sigma_1} \pi_0
) , \
\textcolor{violet}{\omega_{1,0}} \pi_1 
) , \
\hat{u} \textcolor{teal}{\sigma_1} \pi_1 
\big)
\end{align*}

\noi where 
\[ \mu_1 = (\textcolor{violet}{\omega_1} , \hat{u}_{0;4}\textcolor{brown}{\gamma_1} \pi_0 \pi_0) c \] 
\noi is well-defined and relates the spans $\textcolor{teal}{\sigma_1}, \textcolor{violet}{\sigma_{1,\gamma}} : \hat{U} \to \spn$ in the sense that 

\begin{align*}
  \varphi_1 p_0 &= \hat{u}\textcolor{teal}{\sigma_1} 
  & \varphi_1 p_1 &= \textcolor{violet}{\sigma_{1,\gamma}}
\end{align*}

\noi This implies 

\begin{align}\label{eq sigma_1 equiv sigma_1,gamma}
\hat{u}\textcolor{teal}{\sigma_1} q 
= \varphi_1 p_0 q 
= \varphi_1 p_1 q 
= \textcolor{violet}{\sigma_{1,\gamma}} q
\end{align}

\noi By Lemma~\ref{appendix lem defining sailboat varphi_{1,gamma}}, the sailboat, $\varphi_{1,\gamma} : \hat{U} \to \slb$, defined by 

\[ \varphi_{1 , \gamma} = 
\big( 
( 
( \mu_{1, \gamma} , \
\hat{u} \textcolor{purple}{\sigma_\gamma} \pi_0
) , \
\textcolor{violet}{\omega_{1,0}} \pi_1 ) , \
\hat{u} \textcolor{purple}{\sigma_\gamma} \pi_1 
 \big)
\]

\noi where

\[ \mu_{1, \gamma} = (\textcolor{violet}{\omega_1} ,\
\hat{u}_{0;4} \textcolor{brown}{\theta_{\gamma_1}} \pi_1 \pi_0 )c \]

\noi relates the spans $\textcolor{purple}{\sigma_\gamma} , \textcolor{violet}{\sigma_{1,\gamma}} : \hat{U} \to \spn$ in the sense that 

\begin{align*}
  \varphi_{1,\gamma} p_0 &= \hat{u} \textcolor{purple}{\sigma_\gamma} 
  & \varphi_{1,\gamma} p_1 &= \textcolor{violet}{\sigma_{1,\gamma}}.
\end{align*}

\noi It follows that 

\begin{align}\label{eq sigma_1,gamma equiv sigma_gamma}
  \textcolor{violet}{\sigma_{1,\gamma}} q 
  = \varphi_{1,\gamma} q 
  = \varphi_{1,\gamma} p_0 q 
  = \textcolor{purple}{\sigma_\gamma} q. 
\end{align}

\noi Equations (\ref{eq sigma_0 equiv sigma_0,gamma}), (\ref{eq sigma_0,gamma equiv sigma_gamma}), (\ref{eq sigma_1 equiv sigma_1,gamma}), and (\ref{eq sigma_1,gamma equiv sigma_gamma}) imply

\begin{align*}
\hat{u} \textcolor{orange}{\sigma_0} q 
&= \varphi_0 p_0 q \\
&= \varphi_0 p_1 q \\
&= \textcolor{cyan}{\sigma_{0, \gamma}} q \\
&= \varphi_{0, \gamma} p_1 q \\
&= \varphi_{0, \gamma} p_0 q \\
&= \hat{u} \textcolor{purple}{\sigma_\gamma} q\\
&= \varphi_{1, \gamma}p_0 q \\
&= \varphi_{1, \gamma} p_1 q \\
&= \textcolor{violet}{\sigma_{1, \gamma}} q \\
&= \varphi_1 p_1 q \\
&= \varphi_1 p_0 q \\
&= \hat{u} \textcolor{teal}{\sigma_1} q .
\end{align*}

\noi By the definitions of $c' : \spn^2 \to \mCW_1$ in Lemma (\ref{lem defining c'}) and the spans $\textcolor{teal}{\sigma_1} , \textcolor{orange}{\sigma_0} : \tilde{U} \to \spn$ in (\ref{def sigma_0})and (\ref{def sigma_1}) above along with commutativity of Diagrams (\ref{dgm slb proj pb along span comp cover}), (\ref{dgm slb proj pb along span comp cover refinement}), (\ref{dgm intermediate compblespan}), and (\ref{dgm diagram of covers to include composite of intermediate comp'ble spans}) we can see

\begin{align*}
 \hat{u} \tilde{u} \bar{u} p_1^2 c' 
&= \hat{u} \tilde{u} \pi_1 \pi_0 p_1^2 c' \\
& = \hat{u} \tilde{u} \pi_1 \pi_1 u c'\\
&= \hat{u} \tilde{u} \pi_1 \pi_1 \sigma_\circ q \\
 &= \hat{u} \textcolor{teal}{\sigma_1} q \\
 &= \hat{u} \textcolor{orange}{\sigma_0} q \\
 &= \hat{u} \tilde{u} \pi_0 \pi_1 \sigma_\circ q \\
 &= \hat{u} \tilde{u} \pi_0 \pi_1 u c' \\
 &= \hat{u} \tilde{u} \pi_0 \pi_0 p_0^2 c' \\
 &= \hat{u} \tilde{u} \bar{u} p_0^2 c' .
\end{align*}

\noi where $\hat{u} \tilde{u} \bar{u} : \hat{U} \to \slb \tensor[_t]{\times}{_s} \slb$ is a cover because covers are closed under composition. The result follows by renaming the $\varphi_i : \hat{U} \to \slb$ for $0 \leq i \leq 3$ to match with $\varphi_0, \varphi_{0,\gamma}, \varphi_{1,\gamma}$ and $\varphi_1$ (in that order). The cover $\hat{u} : \hat{U} \to \slb \tensor[_t]{\times}{_s} \slb$ in the statement of the Lemma corresponds to the composite $\hat{u} \tilde{u} \bar{u}$ mentioned above and constructed in the proof. 
\end{proof}

\subsection{Associativity and Identity Laws}\label{S internal cat of fracs assoc and id laws}

This section consists of technical proofs of associativity and identity laws for composition that are required to see that $\mCW$, as defined in Section \ref{S Defining mCW, the internal cat of fracs}, is an internal category in $\cE$. 

\subsubsection{Associativity}\label{SS internal cat of fracs assoc}

The proof for associativity of composition in the internal category of fractions is rather involving so we have given it its own subsection. Before we prove associativity we give a remark about induced projection maps for the quotient objects of the reflexive graphs of fractions and prove a Lemma to give explicit descriptions of the two possible compositions 

\[ 1 \times c , c \times 1 : \mCW_3 \to \mCW_2\]

\noi in terms of the representative composition, $c' : \spn \to \mCW_1$, and the first quotient map $q : \spn \to \mCW_1$. We use these to differentiate between which maps are being composed first for a triple composite in $\mCW$ and then prove that they are equal using the universal property of the coequalizer $\mCW_3.$

\begin{rem}
By definition, the canonical pullback projections commute with the coequalizer diagram maps: 
\begin{center}
\begin{tikzcd}
 \slb^k \ar[r,shift left, "p_0^k"] \ar[r, shift right, "p_1^k"'] \dar["\pi_{i_0, ... , i_j}"'] 
 & \spn^k\dar["\pi_{i_0, ... , i_j}"] 
 \rar[two heads, "q_k"] 
 & \mCW_k \dar["\pi_{i_0, ... , i_j}"] 
\\ 
 \slb^\ell \ar[r,shift left, "p_0^\ell"] \ar[r, shift right, "p_1^\ell"']
 & \spn^\ell
 \rar[two heads, "q_\ell"'] 
 & \mCW_\ell 
\end{tikzcd}.
\end{center}
\end{rem}

\noi Before we can prove associativity we need to define the maps that show up in the statement. We use Proposition~\ref{prop path pb is coequalizer} and the universal property of the coequalizer to do this. 

\begin{lem}\label{lem intermediate composition for composable triples }
Let $c \times 1 = (\pi_{01} c, \pi_2 )$ and $1 \times c = (\pi_0 , \pi_{12} c)$ denote the pairing maps $\mCW_3 \to \mCW_2$, and let 

\[ q \times c' = (\pi_0q {,} \pi_{12} c') \quad \text{ and } \quad c' \times q = (\pi_{01} c' {,} \pi_2 q). \]

\noi The diagram 
\begin{center}
\begin{tikzcd}[]
& \spn^3 \ar[d, two heads, "q_3"] \ar[dl, "q \times c'"'] \ar[dr, "c' \times q"] & \\
\mCW_2 & \mCW_3 \lar[dotted, "1 \times c"] \rar[dotted, "c \times 1"'] & \mCW_2
\end{tikzcd}
\end{center}

\noi commutes in $\cE$. 
\end{lem}
\begin{proof}
On the right we have 
\begin{align*}
q_3 (c \times 1) 
= (q_3 \pi_{01} c , q_3 \pi_2 )
= (\pi_{01} q_2 c , \pi_2 q) 
= (\pi_{01} c' , \pi_2 q)
= c' \times q
\end{align*}

\noi and on the left we have 

\begin{align*}
q_3 (1 \times c) 
= (q_3 \pi_0, q_3 \pi_{12} c )
= (\pi_0 q, \pi_{12} q_2 c )
= (\pi_0 q, \pi_{12} c' )
= q \times c'. 
\end{align*}
\end{proof} 

\noi Now we can state and prove the associativity law. This proof is long and technical but follows a similar pattern to the proofs in Section \ref{S Defining mCW, the internal cat of fracs}. Recall that diagrams labeled with capital letters are guides for the usual proofs when $\cE = \Set$ and represent diagrams in the internal category $\mC$, with diagrams labeled with stars giving the internal translation involving covers and lifts from the Internal Fractions Axioms.

\begin{prop}\label{prop assoc law} 
The diagram 

\begin{center}
\begin{tikzcd}[]
\mCW_3 
\dar["1 \times c"']
 \rar["c \times 1"] 
\arrow[dr, phantom, "\usebox \pullback" , very near start, color=black]
& \mCW_2 \dar["c"] \\
\mCW_2 \rar["c"'] 
& \mCW_1
\end{tikzcd}
\end{center}

\noi commutes in $\cE$. 
\end{prop}
\begin{proof}
The plan is to show there exists a cover $\hat{U} \to \spn^3$ with two sailboats, $\varphi_i : \hat{U} \to \slb$, with a common sail-projection, $\varphi_0 p_1 = \varphi_1 p_1 $, so that 

\[ \hat{u} q_3 (1 \times c) c = \varphi_0 p_0 q = \varphi_0 p_1 q = \varphi_1 p_1 q = \varphi_1 p_0 q = \hat{u} q_3 (c \times 1) c . \]

\noi The result then follows from the fact that $\hat{u}$ and $q_3$ are epic. First we find representative spans for the equivalence classes of spans being picked out by $(c\times 1) c$ and $(1 \times c) c$, then we build a comparison span and two sailboats witnessing their equivalence. 

\noi Begin by taking pullbacks of the projections, $\pi_{01}, \pi_{12} : \spn^3 \to \spn^2$, along the cover, $u : U \to \spn^2$, that witnesses the span composition construction

\[\label{dgm pb span comp cover(s) along projs}
\begin{tikzcd}[] 
\tilde{U}_0 
\dar["/"marking , "\tilde{u}_{0:0}"' near end] 
\dar[dd, bend right= 70,"/"marking , "\tilde{u}_0"' near end] 
\rar["\tilde{\pi}_{01}"] 
& U \dar["/"marking , "u_{0}" near start] 
& \tilde{U}_1 \dar["/"marking , "\tilde{u}_{0:1}" near start] 
\dar[dd, bend left =70,"/"marking , "\tilde{u}_1" near start] 
\lar["\tilde{\pi}_{12}"'] 
\\ 
\tilde{U}_{0:0} \dar["/"marking , "\tilde{u}_{1:0}"' near end] 
\rar["\tilde{\pi}_{0:01}"] 
& U_0 \dar["/"marking , "u_{1}" near start] 
& \tilde{U}_{0:1} \dar["/"marking , "\tilde{u}_{1:1}" near start] 
\lar["\tilde{\pi}_{0: 12}"']
\\ 
\spn^3 \rar["\pi_{01}"'] 
& \spn^2 
& \spn^3 \lar["\pi_{12}"] 
\end{tikzcd} \tag{1}\]

\noi and since the outer squares are pullbacks, as seen in \cite{CFTWM} we also have 

\[\label{dgm pb span comp composed cover(s) along projs}
\begin{tikzcd}[]
\tilde{U}_0 \dar["/"marking , "\tilde{u}_0"' near end] 
\rar["\tilde{\pi}_{01}"] 
& U \dar["/"marking , "u" near start] 
& \tilde{U}_1 \dar["/"marking , "\tilde{u}_1" near start] 
\lar["\tilde{\pi}_{12}"']
\\ 
\spn^3 \rar["\pi_{01}"'] 
& \spn^2 
& \spn^3 \lar["\pi_{12}"] 
\end{tikzcd}\tag{2}\]

\noi and then taking a common refinement 

\[\label{dgm assoc cover refinement}
\begin{tikzcd}[]
 \tilde{U} \dar["\pi_0"'] \rar["\pi_1"] \ar[dr, "/" marking, "\tilde{u}" near start] & \tilde{U}_1 \dar["/"marking , "\tilde{u}_1" near start] \\
 \tilde{U}_0 \rar["/"marking , "\tilde{u}_0"' near end] & \spn^3 
\end{tikzcd}. \tag{3}\]

\noi This induces two maps, $\sigma_c \times 1$ and $1 \times \sigma_c$, $\tilde{U} \to \spn^2$ defined by

\[ \sigma_c \times 1 = ( \pi_0 \tilde{\pi}_{01} \sigma_c , \tilde{u} \pi_2) \quad \text{ and } \quad 1 \times \sigma_c = (\tilde{u} \pi_0 , \pi_1 \tilde{\pi}_{12} \sigma_c). \]

\noi Taking pullbacks of these induced maps along the composites that make up $u : U \to \spn^2$ once again gives

\[ \label{dgm pb covers along composition of triples maps}
\begin{tikzcd}[] 
\bar{U}_0 
\dar["/"marking , "\bar{u}_{0:0}"' near end] 
\dar[dd, bend right= 70,"/"marking , "\bar{u}_0"' near end] 
\rar["\bar{\pi}_{01}"] 
& U \dar["/"marking , "u_{0}" near start] 
& \bar{U}_1 \dar["/"marking , "\bar{u}_{0:1}" near start] 
\dar[dd, bend left =70,"/"marking , "\bar{u}_1" near start] 
\lar["\bar{\pi}_{12}"'] 
\\ 
\bar{U}_{0:0} \dar["/"marking , "\bar{u}_{1:0}"' near end] 
\rar["\bar{\pi}_{0:01}"] 
& U_0 \dar["/"marking , "u_{1}" near start] 
& \bar{U}_{0:1} \dar["/"marking , "\bar{u}_{1:1}" near start] 
\lar["\bar{\pi}_{0: 12}"']
\\ 
 \tilde{U} \rar["\sigma_c \times 1"'] 
 & \spn^2 
 & \tilde{U} \lar["1 \times \sigma_c"] 
\end{tikzcd} \tag{4}\]

\noi and taking a common refinement of the covers on the left and right

\[ \label{dgm refinement of pb along composition of triples map}
\begin{tikzcd}[]
\bar{U} \rar["\pi_1"] \dar["\pi_0"'] \ar[dr, "/" marking, "\bar{u}" near start] & \bar{U}_1 \dar["/"marking , "\bar{u}_1" near start] \\
\bar{U}_0 \rar["/"marking , "\bar{u}_0"' near end] & \tilde{U}
\end{tikzcd}, \tag{5}\]

\noi gives a cover $\bar{\tilde{u}} = \bar{u} \tilde{u} : \bar{U} \to \spn^3$ that witnesses representatives for the two ways to compose a composable triple of spans. Let 

\[ \sigma_0 : \bar{U} \to \spn \quad \text{and} \quad \sigma_1 : \bar{U} \to \spn \]

\noi be defined by 

\[ \sigma_0 = \pi_0 \bar{\pi}_{01} \sigma_c \quad \text{and} \quad \sigma_1 = \pi_1 \bar{\pi}_{12} \sigma_c . \] 

\noi To see $\sigma_0$ represents the equivalence class of $\bar{\tilde{u}} q_3 (c \times 1) c : \bar{U} \to \mCW_1$ we use the left squares in the diagrams and definitions above along with the definitions of $c' : \spn^2 \to \mCW_1$ and $c : \mCW_2 \to \mCW_1$ to compute

\begin{align*}
\sigma_0 q 
&= \pi_0 \bar{\pi}_{01} \sigma_c q \\
&= \pi_0 \bar{\pi}_{01} u c' \\
&= \pi_0 \bar{u}_0 (\sigma_c \times 1) c' \\
&= \pi_0 \bar{u}_0 (\sigma_c \times 1) q c\\
&= \pi_0 \bar{u}_0 ( \pi_0 \tilde{\pi}_{01} \sigma_c , \tilde{u} \pi_2) q c\\
&= \pi_0 \bar{u}_0 ( \pi_0 \tilde{\pi}_{01} \sigma_c q , \tilde{u} \pi_2 q) c\\
&= \pi_0 \bar{u}_0 ( \pi_0 \tilde{\pi}_{01} u c' , \tilde{u} \pi_2 q) c\\
&= \pi_0 \bar{u}_0 ( \pi_0 \tilde{u}_{0} \pi_{01} c' , \tilde{u} \pi_2 q) c\\
&= \pi_0 \bar{u}_0 ( \pi_0 \tilde{u}_{0} \pi_{01} c' , \tilde{u} \pi_2 q) c\\
&= \bar{u} (\tilde{u} \pi_{01} c' , \tilde{u} \pi_2 q) c \\
&= \bar{u} \tilde{u} (\pi_{01} c' , \pi_2 q) c\\
&= \bar{\tilde{u}} (c' \times q) c \\
&= \bar{\tilde{u}} q_3 (c \times 1) c. 
\end{align*}

\noi A similar computation using the right squares in the diagrams above shows the spans $\sigma_1 : \bar{U} \to \spn$ represent
\[ \sigma_1 q = \bar{\tilde{u}} q_3 (1 \times c) c. \]

\noi To see these representatives are equivalent we will show there exists a cover, $\hat{u} : \hat{U} \to \bar{U}$, along with two sailboats $\varphi_i : \hat{U} \to \slb$ for $i = 0,1$ such that 
\[ \varphi_0 p_0 = \hat{u} \sigma_0 \quad , \quad \varphi_0 p_1 = \varphi_1 p_1 \quad , \quad \text{ and } \quad \varphi_1 p_0 = \hat{u} \sigma_1. \]

.\noi The following algorithm outlines the necessary steps for constructing the cover $\hat{u} : \hat{U} \to \bar{U}$

\begin{itemize}
\item Apply the \textcolor{olive}{Ore-condition} to the cospan consisting of the left legs, $\sigma_0 \pi_0 : \bar{U} \to W$ and $\sigma_1 \pi_0 : \bar{U} \to W$ 
\item Apply \textcolor{cyan}{three zippers}, one for each left leg of each span in the composable triple in order from initial to final. 
\item Apply \textcolor{violet}{weak-composition four times} to get a span whose left leg is in $W$
\end{itemize}

\noi The figure below illustrates it. 

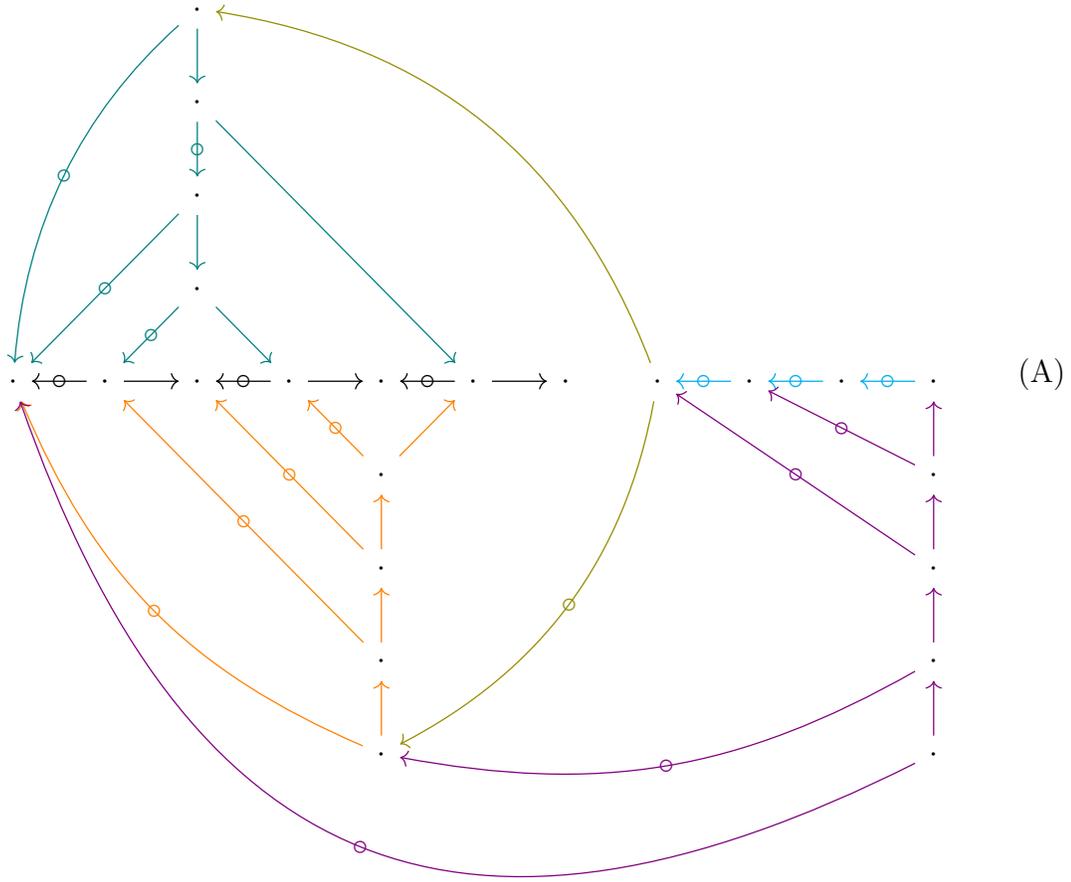
\[\label{fig assoc span comp data} 
\begin{tikzcd}[column sep = scriptsize]
	&& \cdot \\
	&& \cdot \\
	&& \cdot \\
	&& \cdot \\
	\cdot & \cdot & \cdot & \cdot & \cdot & \cdot & \cdot & \cdot & \cdot & \cdot & \cdot \\
	&&&& \cdot &&&&&& \cdot \\
	&&&& \cdot &&&&&& \cdot \\
	&&&& \cdot &&&&&& \cdot \\
	&&&& \cdot &&&&&& \cdot
	\arrow[from=5-2, to=5-1, "\circ" marking]
	\arrow[from=5-2, to=5-3]
	\arrow[from=5-4, to=5-3, "\circ" marking]
	\arrow[from=5-4, to=5-5]
	\arrow[from=5-6, to=5-5, "\circ" marking]
	\arrow[from=5-6, to=5-7]
	\arrow[from=6-5, to=5-4, "\circ" marking, orange]
	\arrow[from=6-5, to=5-6, orange]
	\arrow[from=7-5, to=6-5, orange]
	\arrow[from=4-3, to=5-2, "\circ" marking, teal]
	\arrow[from=4-3, to=5-4, teal]
	\arrow[from=3-3, to=4-3, teal]
	\arrow[from=3-3, to=5-1, "\circ" marking, teal]
	\arrow[from=8-5, to=5-2, "\circ" marking, orange]
	\arrow[from=8-5, to=7-5, orange]
	\arrow[from=7-5, to=5-3, "\circ" marking, orange]
	\arrow[from=2-3, to=3-3, "\circ" marking, teal]
	\arrow[from=2-3, to=5-6, teal]
	\arrow[from=1-3, to=2-3, teal]
	\arrow[curve={height=22 pt}, from=1-3, to=5-1, teal, "\circ" marking]
	\arrow[from=9-5, to=8-5, orange]
	\arrow[from=9-5, to=5-1, "\circ" marking, orange, bend left = 22]
	\arrow[from=5-8, to=1-3, bend right = 30, olive]
	\arrow[curve={height=-30pt}, from=5-8, to=9-5, "\circ" marking, olive]
	\arrow[from=5-9, to=5-8, "\circ" marking, cyan]
	\arrow[from=5-10, to=5-9,"\circ" marking, cyan]
	\arrow[from=7-11, to=6-11, violet]
	\arrow[from=8-11, to=7-11, violet]
	\arrow[from=9-11, to=8-11, violet]
	\arrow[from=6-11, to=5-9,"\circ" marking, violet]
	\arrow[from=7-11, to=5-8,"\circ" marking, violet]
	\arrow[from=8-11, to=9-5, bend left = 20, "\circ" marking, violet]
	\arrow[curve={height=-150pt}, from=9-11, to=5-1,"\circ" marking, violet]
	\arrow[from=5-11, to=5-10,"\circ" marking, cyan]
	\arrow[from=6-11, to=5-11, violet]
\end{tikzcd} \tag{A}\]

\noi Internally taking the Ore-square and zippering three times corresponds to applying \textbf{In.Frc(3)} followed by \textbf{In.Frc(4)} four times to get the chain of covers and lifts: 

\[ \label{dgm assoc covers for Ore + zipx3} 
\begin{tikzcd}[]
\cP(\mC) \rar["\pi_1"] 
& \cP_{cq}(\mC) 
& \cP(\mC) \rar["\pi_1"] 
& \cP_{cq}(\mC) 
& 
\\ 
\hat{U}_3 
\rar["/" marking, "\hat{u}_4" near start]
\uar[dotted, "\delta_2"] 
& \hat{U}_4
 \rar["/" marking, "\hat{u}_5" near start]
\dar[dotted, "\delta_1"'] 
\uar["\delta_2'"']
& \hat{U}_5 
\rar["/" marking, "\hat{u}_6" near start]
\uar[dotted, "\delta_0"'] 
\dar["\delta_1'"]
& \hat{U}_6 \rar["/" marking, "\hat{u}_7" near start] 
\uar["\delta_0'"'] 
\dar[dotted, "\theta_a"'] 
&\bar{U} \dar["(\sigma_0 \pi_0 w {,} \sigma_1 \pi_0)"] 
\\
&\cP(\mC) \rar["\pi_1"'] 
&\cP_{cq}(\mC)
&W_\square \rar[ "(\pi_0 \pi_1 {,} \pi_1 \pi_1)"'] 
&\csp 
\end{tikzcd} \tag{$\star$}\]

\noi where $\delta_0'$ is induced by the map $\delta_0'' : \hat{U}_6 \to P(\mC) \prescript{}{wt}{\times_{ws}} W$ that can be found by taking the equality

\[ \theta_a (\pi_0 \pi_0 , \pi_0 \pi_1) c = \theta_a (\pi_1 \pi_0 , \pi_1 \pi_1) c \]

\noi from the definition of $W_\square$; expanding the second components using the definitions of $\theta_a$, $\sigma_c \times 1$, and $1 \times \sigma_c$ with the definition 

\[\sigma_c = ( \omega \pi_1 , ( \omega \pi_0 \pi_0 , u_0 \theta \pi_1 \pi_0 , u \pi_1 \pi_1)c ) : U \to spn\] 

\noi and finding a common final arrow in $W$ in this expansion process. This arrow is the left leg of the initial span in the original composable triple of spans and can be seen in Figure (\ref{fig assoc span comp data}) above. The composite obtained from the teal upper half of the figure is longer than the other as it requires factoring through two weak-composition triangles. This is formalized by expanding 

\begin{align}
\begin{split}\label{eq assoc proof sidecalc theta_a pi_0pi_1}
\theta_a \pi_0 \pi_1 
&= \bar{u}_7 \sigma_0 \pi_0 w \\
&= \bar{u}_7 \pi_0 \bar{\pi}_{01} \sigma_c \pi_0 w\\ 
&= \bar{u}_7 \pi_0 \bar{\pi}_{01} \omega \pi_1 w\\
&= \bar{u}_7 \pi_0 \bar{\pi}_{01} (\omega \pi_0 \pi_0 , u_0 \theta \pi_0 \pi_0 w , u \pi_0 \pi_0 w) c \\
&= \bar{u}_7 \pi_0
( 
\bar{\pi}_{01} \omega \pi_0 \pi_0 , \ 
 \bar{\pi}_{01} u_0 \theta \pi_0 \pi_0 w , \ 
 \bar{\pi}_{01} u \pi_0 \pi_0 w) c \\
&= \bar{u}_7 \pi_0 
( 
\bar{\pi}_{01} \omega \pi_0 \pi_0 , \ 
\bar{u}_{0:0} \bar{\pi}_{0:01} \theta \pi_0 \pi_0 w , \ 
\bar{u}_{0} (\sigma_c \times 1) \pi_0 \pi_0 w 
) c 
\end{split}
\end{align} 

\noi and similarly 

\begin{align} \label{eq assoc proof sidecalc theta_a pi_1 pi_1}
\begin{split}
\theta_a \pi_1 \pi_1 &= \bar{u}_7 \sigma_1 \pi_0 \\
&\ \ \vdots \\
&=
\bar{u}_7 \pi_1 
\big( 
 \bar{\pi}_{12} \omega \pi_0 \pi_0 , \ 
\bar{u}_{0:1} \bar{\pi}_{0:12} \theta \pi_0 \pi_0 w , \ 
\bar{u}_{1} (1 \times \sigma_c) \pi_0 \pi_0 w \big) c. 
\end{split}
\end{align}

\noi Now recall that 
\[ \sigma_c \times 1 = ( \pi_0 \tilde{\pi}_{01} \sigma_c , \tilde{u} \pi_2) \quad \text{ and } \quad 1 \times \sigma_c = (\tilde{u} \pi_0 , \pi_1 \tilde{\pi}_{12} \sigma_c)\] 
\noi and we have 
\begin{align}\label{eq assoc proof sidecalc for (sigma_c times 1) pi_0 pi_0 w}
\begin{split}
(\sigma_c \times 1) \pi_0 \pi_0 w 
&= \pi_0 \tilde{\pi}_{01} \sigma_c \pi_0 w \\
&= \pi_0 \tilde{\pi}_{01} \omega \pi_1 w \\
&= 
\pi_0 \tilde{\pi}_{01} ( 
 \omega \pi_0 \pi_0 , 
 u_0 \theta \pi_0 \pi_0 w , 
 u \pi_0 \pi_0 w
)c 
\end{split}
\end{align}

\noi where the last map, $\tilde{U} \to \mC_1$ in this composite is: 
\begin{align}\label{eq assoc proof sidecalc (1 times sigma_c) pi_0 pi_0 w}
\begin{split}
\pi_0 \tilde{\pi}_{01} u \pi_0 \pi_0 w 
&= \pi_0 \tilde{u}_0 \pi_{01} \pi_0 \pi_0w \\
&= \tilde{u} \pi_{01} \pi_0 \pi_0 w \\
&= \tilde{u} \pi_0 \pi_0 w \\
&= (1 \times \sigma_c) \pi_0 \pi_0 w 
\end{split}
\end{align}

\noi The definition of the cover $\bar{u}$ from Diagram (\ref{dgm refinement of pb along composition of triples map}) and equation~(\ref{eq assoc proof sidecalc (1 times sigma_c) pi_0 pi_0 w}) imply the final component of the internal composition defining $\theta_a \pi_1 \pi_1$ in equation~(\ref{eq assoc proof sidecalc theta_a pi_1 pi_1}) is the final component of the internal composition defining $\theta_a \pi_0 \pi_1$ in equation~(\ref{eq assoc proof sidecalc theta_a pi_0pi_1}):

\begin{align}\label{eq assoc proof sidecalc final components of theta_a pi_0 pi_1 and theta_a pi_1 pi_1 coincide}
\begin{split}
 \hat{u}_7 \pi_0 \bar{u}_0 \pi_0 \tilde{\pi}_{01} u \pi_0 \pi_0 w 
 &= \hat{u}_7 \pi_1 \bar{u}_1 \pi_0 \tilde{\pi}_{01} u \pi_0 \pi_0 w \\
&= \hat{u}_7 \pi_1 \bar{u}_1 (1 \times \sigma_c) \pi_0 \pi_0 w.
\end{split}
\end{align}

\noi With these calculations and the commuting diagrams defining the covers above we can see there exists a map $\delta_0'' : \hat{U}_6 \to P(\mC) \prescript{}{wt}{\times_{ws}} W $, uniquely determined by the maps 

\begin{align*}
\delta_0'' \pi_1 
= \hat{u}_7 \bar{u} (1 \times \sigma_c) \pi_0 \pi_0 ,
\end{align*}\begin{align*}
\delta_0'' \pi_0 \pi_0
= \big( 
\theta_a \pi_1 \pi_0 , \ 
\bar{u}_7 \pi_1 
( 
 \bar{\pi}_{12} \omega \pi_0 \pi_0 , \ 
\bar{u}_{0:1} \bar{\pi}_{0:12} \theta \pi_0 \pi_0 w
)c 
\big)c ,
\end{align*}
\noi and 
\begin{align*}
\delta_0'' \pi_0 \pi_1 
&= \big( 
\theta_a \pi_0 \pi_0 , \\
& \ \ \ \ 
\bar{u}_7 \pi_0 
( \bar{\pi}_{01} \omega \pi_0 \pi_0 , \ 
\bar{u}_{0:0} \bar{\pi}_{0:01} \theta \pi_0 \pi_0 w , \\
& \ \ \ \ 
\bar{u}_{0} 
( \pi_0 \tilde{\pi}_{01} ( \omega \pi_0 \pi_0 , 
 u_0 \theta \pi_0 \pi_0 w
)c \
)c \
\big) c \\
&= \big( 
\theta_a \pi_0 \pi_0 , \\
& \ \ \ \ \bar{u}_7 \pi_0 
( 
\bar{\pi}_{01} \omega \pi_0 \pi_0 , \ 
\bar{u}_{0:0} \bar{\pi}_{0:01} \theta \pi_0 \pi_0 w , \ 
\bar{u}_{0} (\sigma_c \times 1) \pi_0 \pi_0 w 
) c \big).
\end{align*}

\noi Moreover, the equations five equations above imply

\begin{align}\label{eq assoc proof sidecalc delta_0'' inducing delta_0'}
\begin{split}
\delta_0'' (\pi_0 \pi_0 , \pi_1 ) c 
&= \big( \theta_a \pi_1 \pi_0 , \\
& \ \ \ \ 
\bar{u}_7 \pi_1 
( 
 \bar{\pi}_{12} \omega \pi_0 \pi_0 , \ 
\bar{u}_{0:1} \bar{\pi}_{0:12} \theta \pi_0 \pi_0 w
)c , \\
& \ \ \ \ 
\hat{u}_7 \bar{u} (1 \times \sigma_c) \pi_0 \pi_0
\big) c \\
&= ( \theta_a \pi_1 \pi_0 , \theta_a \pi_1 \pi_1 )c \\
&= ( \theta_a \pi_0 \pi_0 , \theta_a \pi_0 \pi_1 )c \\
&= \big( \theta_a \pi_0 \pi_0 , \\
& \ \ \ \ \bar{u}_7 \pi_0 
( 
\bar{\pi}_{01} \omega \pi_0 \pi_0 , \ 
\bar{u}_{0:0} \bar{\pi}_{0:01} \theta \pi_0 \pi_0 w , \ 
\bar{u}_{0} (\sigma_c \times 1) \pi_0 \pi_0 w 
) c 
\big)c \\
&= \delta_0 '' (\pi_0 \pi_1 , \pi_1) c .
\end{split}
\end{align}

\noi By the universal property of the equalizer,$\cP_{cq}(\mC)$, equation~(\ref{eq assoc proof sidecalc delta_0'' inducing delta_0'}) induces a unique map $\delta_0' : \hat{U}_6 \to \cP_{cq}(\mC)$ such that the diagram 

\begin{center}
\begin{tikzcd}[]
\cP_{cq}(\mC) \rar["\iota_{cq}", tail] &P(\mC) \prescript{}{t}{\times_{ws}} W \\
\hat{U}_6 \ar[ur, "\delta_0''"'] \uar[dotted, "\delta_0'"] 
\end{tikzcd}
\end{center}

\noi commutes in $\cE$. Next, to define the map $\delta_1' : \hat{U}_5 \to \cP_{cq}(\mC)$, we start by considering the definition of the pullback, $\cP ( \mC)$, that says 

\[ \delta_0 \pi_0 \iota_{eq} (\pi_0 , \pi_1 \pi_0 )c = \delta_0 \pi_0 \iota_{eq} (\pi_0 , \pi_1 \pi_1 )c : \hat{U}_5 \to \mC_1 \]

\noi are equal in $\cE$ and represent that same family of arrows in $\mC$. We can post-compose these internally to $\mC$ with the family of arrows

\[ \hat{u}_{6;7} \bar{u} \tilde{u} \pi_0 \pi_1 : \hat{U}_5 \to \mC_1\]

\noi and after re-associating the internal composition to find each of the two Ore-squares arising at the two different covers in the following diagram. 

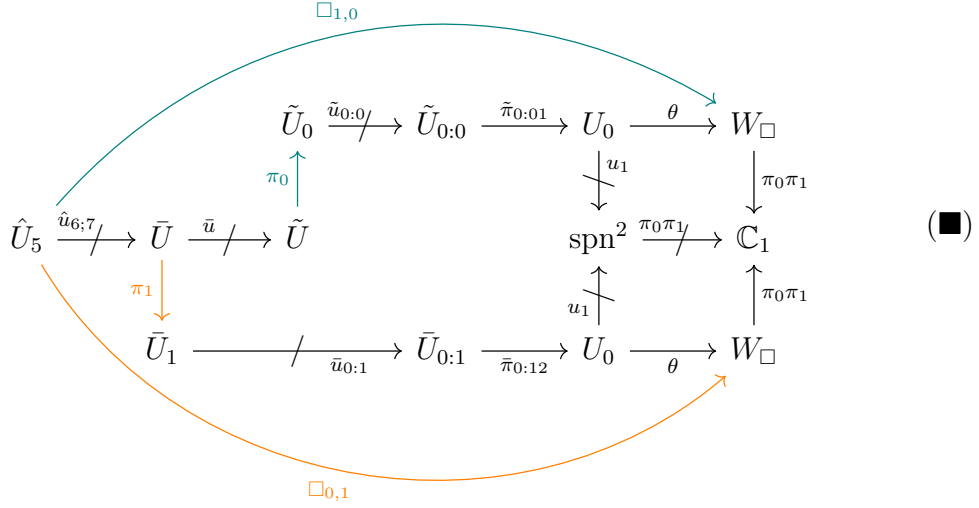
\[\label{dgm assoc picking ore squares through covers}
\begin{tikzcd}[]
& 
&\tilde{U}_0 
\rar[r, "/" marking, "\tilde{u}_{0:0}" near start] 
& \tilde{U}_{0:0} 
\rar[r,"\tilde{\pi}_{0:01}"] 
& U_0 
\dar["/" marking, "u_1" near start] 
\ar[r, "\theta"] 
& W_\square \dar["\pi_0 \pi_1"] \\
\hat{U}_5 
\ar[drrrrr, bend right = 50, orange, "\square_{0,1}"']
\ar[urrrrr, bend left = 40, teal, "\square_{1,0}"]
\rar["/" marking, "\hat{u}_{6;7}" near start] 
& \bar{U} 
\dar["\pi_1"', orange] 
\rar["/" marking, "\bar{u}" near start] 
& \tilde{U} 
\ar[u, "\pi_0", teal] 
& 
& \spn^2 \rar["/" marking, "\pi_0 \pi_1" near start] 
& \mC_1 \\
& \bar{U}_1 
\rar[rr,"/" marking, "\bar{u}_{0:1}"' near end] 
&
&\bar{U}_{0:1}
\rar["\bar{\pi}_{0:12}"'] 
& U_0 
\uar["/" marking, "u_1" near start] 
\ar[r, "\theta"'] 
& W_\square \uar["\pi_0 \pi_1"'] 
\end{tikzcd}\tag{$\blacksquare$}\]

\noi This diagram commutes in the sense that the two maps on the outside, $\hat{U}_5 \to \mC_1,$ are equal. In the figure above this corresponds to the statement that the two left-most Ore-squares on the top and bottom agree on the projection, $\pi_0 \pi_1 : W_\square \to \mC_1$. Notice $\textcolor{teal}{\square_{1,0}}$ represents one of the two Ore-squares in Diagram \ref{fig assoc span comp data} of the triple composition construction on the top (in teal) while $\textcolor{orange}{\square_{0,1}}$ represents the second of two Ore-squares in the triple composition construction on the bottom (in orange). Now compute the projection 

\begin{align*}
\textcolor{orange}{\square_{0,1}} \pi_1 \pi_1 
&= \hat{u}_{6;7} \textcolor{orange}{\pi_1} \bar{u}_{0:1} \bar{\pi}_{0:12} \theta \pi_1 \pi_1 \\ 
&= \hat{u}_{6;7} \textcolor{orange}{\pi_1} \bar{u}_{0:1} \bar{\pi}_{0:12} u_1 \pi_1 \pi_0 \\ 
&= \hat{u}_{6;7} \textcolor{orange}{\pi_1} \bar{\pi}_{12} u_{0} u_1 \pi_1 \\ 
&= \hat{u}_{6;7} \textcolor{orange}{\pi_1} \bar{\pi}_{12} u \pi_1 \pi_0\\ 
&= \hat{u}_{6;7} \textcolor{orange}{\pi_1} \bar{u}_{1} (1 \times \sigma_c) \pi_1 \pi_0\\ 
&= \hat{u}_{6;7} \bar{u} (1 \times \sigma_c) \pi_1 \pi_0 \\ 
&= \hat{u}_{6;7} \bar{u} \pi_1 \tilde{\pi}_{12} \sigma_c \pi_0 \\ 
&= \hat{u}_{6;7} \bar{u} \pi_1 \tilde{\pi}_{12} \omega \pi_1. 
\end{align*}
 
\noi Since $\omega \pi_1 = (\omega \pi_0 \pi_0 , \omega \pi_0 \pi_1 , \omega \pi_0 \pi_2) c$ with $\omega \pi_0 \pi_2 = u \pi_0 \pi_0$ by definition of $\omega : U \to W_\circ$ we can expand and notice that the final map, 

\[ \hat{u}_{6;7} \bar{u} \pi_1 \tilde{\pi}_{12} \omega \pi_0 \pi_2 : \hat{U}_5 \to W \]

\noi is equal to 

\[ \hat{u}_{6;7} \bar{u} \pi_1 \tilde{\pi}_{12} u \pi_0 \pi_0 = \textcolor{teal}{\square_{1,0}} \pi_1 \pi_1 : \hat{U}_5 \to W. \] 

\noi This induces a unique map $\delta_1'' : \hat{U}_5 \to P(\mC) \prescript{}{t}{\times_{ws}} W$ determined by the map $\hat{U}_5 \to W$ given by

\begin{align*}
\delta_1'' \pi_1 
= \textcolor{teal}{\square_{1,0}} \pi_1 \pi_1 ,
\end{align*}
\noi and the map $\delta_1'' \pi_0 : \hat{U}_5 \to P(\mC) $ whose left projection, $\delta_1'' \pi_0\pi_0$, is the pairing map
\begin{align*}
\big( 
&\delta_0 \iota_{eq} \pi_0 w ,\\
&
\hat{u}_6 \theta_a \pi_1 \pi_0 , \\
&\hat{u}_{6;7} \textcolor{orange}{\pi_1} \bar{\pi}_{12} \omega \pi_0 \pi_0 , \\
&\hat{u}_{6;7} \textcolor{orange}{\pi_1} \bar{\pi}_{12} u_0 \theta \pi_1 \pi_0 ,\\ 
&\hat{u}_{6;7} \bar{u} \textcolor{orange}{\pi_1} \tilde{\pi}_{12} \omega \pi_0 \pi_0 ,\\
&\hat{u}_{6;7} \bar{u} \textcolor{orange}{\pi_1} \tilde{\pi}_{12} u_0 \theta \pi_0 \pi_0 w
\big) c
\end{align*}
\noi and whose right projection, $\delta_1'' \pi_0 \pi_1$, is the pairing map 
\begin{align*}
\big( 
&\delta_0 \iota_{eq} \pi_0 w ,\\ 
&\hat{u}_6 \theta_a \pi_0 \pi_0 w , \\
&\hat{u}_{6;7} \textcolor{teal}{\pi}_0 \bar{\pi}_{01} \omega \pi_0 \pi_0 , \\
&\hat{u}_{6;7} \textcolor{teal}{\pi}_0 \bar{\pi}_{01} u_0 \theta \pi_0 \pi_0 w ,\\
&\hat{u}_{6;7} \bar{u} \textcolor{teal}{\pi}_1 \tilde{\pi}_{01} \omega \pi_0 \pi_0 ,\\ 
&\hat{u}_{6;7} \bar{u} \textcolor{teal}{\pi}_0 \tilde{\pi}_{01} u_0 \theta \pi_1 \pi_0
\big) c . 
\end{align*}

\noi The fact that $\delta_1''$ satisfies the equalizer condition for $\cP_{cq}(\mC)$, namely 
\[\delta_1'' (\pi_0 \pi_0 , \pi_1) c = \delta_1'' (\pi_0 \pi_1 , \pi_1) c,\]
\noi follows from the calculations above. This induces the unique map $\delta_1'$ that makes the following diagram commute:

\begin{center}
\begin{tikzcd}[]
\cP_{cq}(\mC) \rar["\iota_{cq}" ] & P(\mC) \prescript{}{t}{\times_{ws}} W \\
\hat{U}_5 \ar[ur, "\delta_1''"'] \uar[dotted, "\delta_1'"] 
\end{tikzcd}
\end{center}

\noi Finally, the map $\delta_2' : \hat{U}_4 \to \cP_{cq}(\mC)$ is similarly induced by a map $\delta_2'' : \hat{U}_4 \to P(\mC) \prescript{}{t}{\times_{ws}} W$ which can be deduced expanding the right and left-hand sides of the equation 

\[ \delta_1 \iota_{eq} (\pi_0 , \pi_1 \pi_0) c = \delta_1 \iota_{eq} (\pi_0 \pi_1 \pi_1) c \]\

\noi whose common target is the middle object of the original composable triple of spans, post-composing with the arrow

\[ \hat{u}_{5;6} \bar{u} \tilde{u} \pi_1 \pi_1 : \hat{U}_4 \to \mC_1, \]

\noi which is final map given by applying the projection $\pi_0 \pi_1 : W_\square \to \mC_1$ of the Ore-squares 

\[
\begin{tikzcd}[]
&\tilde{U}_0 
\rar[rr, "/" marking, "\tilde{u}_{0:0}" near start] 
&
& \tilde{U}_{0:0} 
\rar[r,"\tilde{\pi}_{0:01}"] 
& U_0 
\dar["/" marking, "u_1" near start] 
\ar[r, "\theta"] 
& W_\square \dar["\pi_1 \pi_1"] 
\\
\hat{U}_4 
\ar[drrrrr, bend right = 40, orange, "\square_{0,0}"']
\ar[urrrrr, bend left = 50, teal, "\square_{1,1}"]
\rar["/" marking, "\hat{u}_{4;7}" near start] 
& \bar{U} 
\ar[u, "\pi_0", teal] 
\rar["/" marking, "\bar{u}" near start] 
& \tilde{U} 
\dar["\pi_1"', orange] 
& 
& \spn^2 \rar["/" marking, "\pi_1 \pi_0" near start] 
& W 
\\
&
& \bar{U}_1 
\rar[r,"/" marking, "\bar{u}_{0:1}"' near end] 
&\bar{U}_{0:1}
\rar["\bar{\pi}_{0:12}"'] 
& U_0 
\uar["/" marking, "u_1" near start] 
\ar[r, "\theta"'] 
& W_\square \uar["\pi_1 \pi_1"'] 
\\
\end{tikzcd}.\tag{$\blacksquare \blacksquare$}\]

\noi Picking out the composable pairs in $W$ to get a cover witnessing a family of spans whose left leg is in $W$ is done in order from right to left in the diagram before. This is identical to how it was done in the proof of Lemma~\ref{lem witnessing c' well-defined} from Section \ref{S Defining mCW, the internal cat of fracs}, except this time an extra zippering step leads to an extra composable pair in $W$. The chain of covers and lifts are given by applying \textbf{In.Frc(2)} four times as seen in diagrams

\[\label{dgm assoc covers weak comp}
\begin{tikzcd}[column sep = large]
W_\circ \rar["(\pi_0 \pi_1 {,} \pi_0 \pi_2)"]
&W \prescript{}{wt}{\times_{ws}} W
&
\\
\hat{U}_1 \rar["/" marking, "\hat{u}_2" near start]
\uar[dotted, "\omega_1"] 
&
\hat{U}_2
\rar["/" marking, "\hat{u}_3" near start]
\dar[dotted, "\omega_0"'] 
\uar["(\omega_0 \pi_1 {,} \hat{u}_{3;5} \delta_0 \iota_{eq} \pi_0) "']
&
\hat{U}_3 \dar["( \delta_2 \iota_{eq} \pi_0 {,} \hat{u}_4 \delta_1 \iota_{eq} \pi_0 )"] 
\\
& W_\circ \rar["(\pi_0 \pi_1 {,} \pi_0 \pi_2)"']
&W \prescript{}{wt}{\times_{ws}} W
\end{tikzcd} \tag{$\star \star$} \]
\noi and 
\[ \label{dgm assoc covers weak comp 2}
\begin{tikzcd}
W_\circ \rar["(\pi_0 \pi_1 {,} \pi_0 \pi_2)"]
&W \prescript{}{wt}{\times_{ws}} W
& \\
\hat{U}
\uar[dotted, "\omega_3"] 
\rar["/" marking, "\hat{u}_0" near start]
&
\hat{U}_0
\rar["/" marking, "\hat{u}_1" near start]
\dar[dotted, "\omega_2"'] 
\uar["(\omega_2 \pi_1 {,} \hat{u}_{1;7} \sigma_0 \pi_0)"']
&
\hat{U}_1 
\dar["(\omega_1 \pi_1 {,} \hat{u}_{2;6} \theta_a \pi_0 \pi_0 )"]
\\
&W_\circ \rar["(\pi_0 \pi_1 {,} \pi_0 \pi_2)"']
&W \prescript{}{wt}{\times_{ws}} W
\end{tikzcd}. \tag{$\star \star \star$}\]
\noi At this point we can define two sailboats, $\varphi_0 , \varphi_1 : \hat{U} \to \slb$, whose deck-projections give the two composite representatives we care for, 
\[ \varphi_0 p_0 = \hat{u} \sigma_0 \quad , \quad \varphi_1 p_0 = \hat{u} \sigma_1\] 

\noi and whose sail-projections agree, 
\[ \varphi_0 p_1 = \varphi_1 p_1\] 
\noi by virtue of zippering. To define these explicitly we first expand both sides of the equation 

\[ \hat{u}_{0; 3} \delta_2 \iota_{eq} (\pi_0 , \pi_1 \pi_0) c = \hat{u}_{0; 3} \delta_2 \iota_{eq} (\pi_0 , \pi_1 \pi_1) c\]

\noi into composites and post-compose both sides with the map represented by

\[ \hat{u} \bar{u} \tilde{u} \pi_2 \pi_1 : \hat{U} \to \mC_1.\]

\noi This gives two equal representations of the right leg of the intermediate span, 
\[ \varphi_0 p_1 \pi_1 = \varphi_1 p_1 \pi_1, \]

\noi and by re-associating the composites in both representations we can get

\[ \varphi_0 p_1 \pi_1 = ( \mu_0 , \hat{u} \sigma_0 \pi_1) c \]
\noi and 
\[ \varphi_1 p_1 \pi_1 = (\mu_1 , \hat{u} \sigma_1 \pi_1) c\]

\noi for two maps, $\mu_0 , \mu_1 : \hat{U} \to \mC_1$, which represent the masts of the sailboats being picked out. This gives the maps $\hat{U} \to \slb$ defined by the pairing maps 

\[ \varphi_0 = \big( ( ( \mu_0 , \hat{u} \sigma_0) , \ \omega_2 \pi_1 ) , \ \hat{u} \sigma_0 \pi_1 \big) \]
\noi 
\[ \varphi_1 = \big( ( ( \mu_1 , \hat{u} \sigma_1) , \ \omega_2 \pi_1 ) , \ \hat{u} \sigma_1 \pi_1 \big) . \]

\noi It follows that 
\[ \hat{u} \sigma_0 q = \varphi_0 p_0q = \varphi_0 p_1 q = \varphi_1 p_1 q = \varphi _1 p_0 q = \hat{u} \sigma_1 q \]

\noi and since $\hat{u}$ is epic, 
\[ \sigma_0 q = \sigma_1 q.\] 
\end{proof}

\subsubsection{Identity Laws}\label{SS internal cat of fracs id laws}
The only conditions left to check in order to see that $\mCW$ is an internal category are the left and right identity laws for composition. This rest of this section is dedicated precisely to this. The identity laws are typically proven, when $\cE = \Set$, by looking at the composites

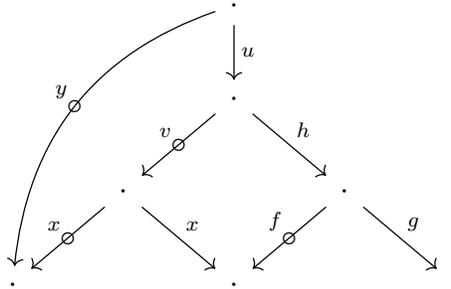
\[\begin{tikzcd}
	&& \cdot \\
	&& \cdot \\
	& \cdot && \cdot \\
	\cdot && \cdot && \cdot
	\arrow[from=3-2, to=4-1, "\circ" marking, "x"']
	\arrow[from=3-2, to=4-3, "x"]
	\arrow[from=3-4, to=4-3, "\circ" marking, "f"']
	\arrow[from=3-4, to=4-5, "g"]
	\arrow[from=2-3, to=3-2, "\circ" marking, "v"']
	\arrow[from=2-3, to=3-4, "h"]
	\arrow[from=1-3, to=2-3, "u"]
	\arrow[curve={height=30pt}, from=1-3, to=4-1,"\circ" marking, "y"']
\end{tikzcd}\]
\noi and 
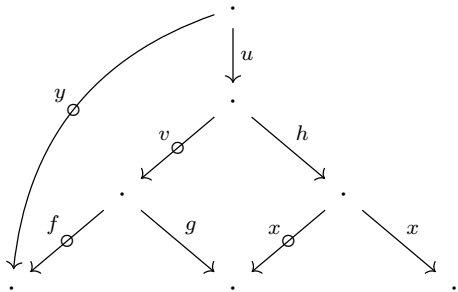
\[
\begin{tikzcd}
	&& \cdot \\
	&& \cdot \\
	& \cdot && \cdot \\
	\cdot && \cdot && \cdot
	\arrow[from=3-2, to=4-1, "\circ" marking, "f"']
	\arrow[from=3-2, to=4-3, "g"]
	\arrow[from=3-4, to=4-3, "\circ" marking, "x"']
	\arrow[from=3-4, to=4-5, "x"]
	\arrow[from=2-3, to=3-2, "\circ" marking, "v"']
	\arrow[from=2-3, to=3-4, "h"]
	\arrow[from=1-3, to=2-3, "u"]
	\arrow[curve={height=30pt}, from=1-3, to=4-1, "\circ" marking, "y"']
\end{tikzcd}\]

\noi and producing the sailboats

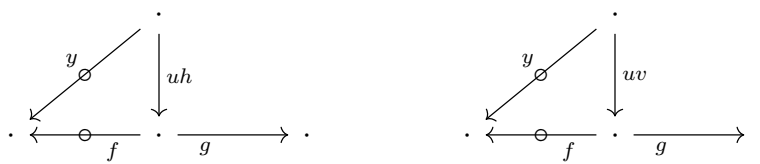
\[\begin{tikzcd}[column sep = large, row sep = large]
	& \cdot & \\
	\cdot & \cdot & \cdot 
	\arrow[from=1-2, to=2-2, "uh"] 
	\arrow[from=1-2, to=2-1, "\circ" marking, "y"']
	\arrow[from=2-2, to=2-1, "\circ" marking, "f"near start] 
	\arrow[from=2-2, to=2-3, "g"' near start] 
\end{tikzcd}\qquad \qquad 
\begin{tikzcd}[column sep = large, row sep = large]
	& \cdot & \\
	\cdot & \cdot & \cdot 
	\arrow[from=1-2, to=2-2, "uv"] 
	\arrow[from=1-2, to=2-1, "\circ" marking, "y"']
	\arrow[from=2-2, to=2-1, "\circ" marking, "f"near start] 
	\arrow[from=2-2, to=2-3, "g"' near start] 
\end{tikzcd}\] 

\noi that relate each composite to the span $(f,g)$ respectively. Since proving the identity laws requires a lot of source and target maps for different objects and we have been overloading their notation, for the rest of this section we rename the source and target maps for spans and sailboats to keep our calculations somewhat more legible. Let $s', t' : \spn \to \mC_0$ denote the source and target maps for spans, given by the pairing maps $s' = \pi_0 w t$ and $t' = \pi_1 t$ respectively. Also let $s'' , t'' : \slb \to \mC_0$ denote the source and target maps for sailboats given by the pairing maps $s'' = \pi_0 \pi_1 t = \pi_0 \pi_0 \pi_1 t$ and $t'' = \pi_1 t $. Internalizing this will require covers that witness composition of spans along with the canonical left and right identity inclusions,

\[
\begin{tikzcd}[]
\spn \rar["( s' \sigma_\alpha {,} 1)"] & \spn \tensor[_{t'}]{\times}{_{s'}} \spn
\end{tikzcd}\qquad \qquad 
\begin{tikzcd}[]
\spn \rar["( 1 {,} t' \sigma_\alpha)"] & \spn \tensor[_{t'}]{\times}{_{s'}} \spn
\end{tikzcd}
\]
\noi and 
\[
\begin{tikzcd}[]
\slb \rar["( s'' \varphi_\alpha {,} 1)"] & \slb \tensor[_{t''}]{\times}{_{s''}} \slb
\end{tikzcd}\qquad \qquad 
\begin{tikzcd}[]
\slb \rar["( 1 {,} t'' \varphi_\alpha)"] & \slb \tensor[_{t''}]{\times}{_{s''}} \slb
\end{tikzcd},
\]

\noi induced by $\sigma_\alpha = (\alpha, \alpha w)$, $\varphi_\alpha = \big( ( ( \alpha s w e , \alpha), \alpha ) , \alpha w \big)$, and the fact that $\alpha$ is a section of $w t$. The following lemma is used to define the identity inclusions $\mCW_1 \to \mCW_1 \tensor[_t]{\times}{_s} \mCW_1$ used in the identity law statement. 

\begin{lem}\label{lem id span inclusions to composable spans}
The diagrams

\begin{center}
\begin{tikzcd}[]
\slb \dar["(s'' \varphi_\alpha {,} 1)"'] \rar[shift left, "p_0"] \rar[shift right, "p_1"'] & \spn \dar["(s' \sigma_\alpha {,} 1)"] \rar[two heads, "q"] & \mCW_1 \dar[dotted, "(s e {,} 1)"] \\
\slb \tensor[_{t''}]{\times}{_{s''}}\slb \rar[shift left, "p_0^2"] \rar[shift right, "p_1^2"'] & \spn \tensor[_{t'}]{\times}{_{s'}} \spn \rar["q^2"'] & \mCW_1 \tensor[_t]{\times}{_s} \mCW_1
\end{tikzcd} 
\end{center}

\noi and 

\begin{center}
\begin{tikzcd}[]
\slb \dar["( 1 {,} t'' \varphi_\alpha)"'] \rar[shift left, "p_0"] \rar[shift right, "p_1"'] & \spn \dar["(1 {,} t' \sigma_\alpha)"] \rar[two heads, "q"] & \mCW_1 \dar[dotted, "(1 {,} t e)"] \\
\slb \tensor[_{t''}]{\times}{_{s''}}\slb \rar[shift left, "p_0^2"] \rar[shift right, "p_1^2"'] & \spn \tensor[_{t'}]{\times}{_{s'}} \spn \rar["q^2"'] & \mCW_1 \tensor[_t]{\times}{_s} \mCW_1
\end{tikzcd}
\end{center}

\noi commute in the sense that for $i = 0,1$

\[(s'' \varphi_\alpha , 1) p_i^2 = p_i (s' \sigma_\alpha , 1) \qquad , \qquad (t'' \varphi_\alpha , 1) p_i^2 = p_i (t' \sigma_\alpha , 1), \]
\noi which uniquely determines 
\[(s e , 1) \qquad \text{ and } \qquad (1,te)\]
\noi respectively. 
\end{lem}
\begin{proof}
To see the squares on the left commute first notice that

\[ p_0 s' = p_0 \pi_0 w t = \pi_0 \pi_0 \pi_1 w t = s'' = \pi_0 \pi_1 w t = p_1 \pi_0 w t = p_1 s' \]
\noi and 
\[ p_0 t' = p_0 \pi_1 t = \pi_1 t = t'' = p_1 \pi_1 t = p_1 t'.\] 

\noi Now for $i = 0,1$ we have 

\begin{align*}
(s'' \varphi_\alpha , 1)p_i^2 
&= (s'' \varphi_\alpha , 1)(\pi_0 p_i , \pi_1 p_i) \\
&= \big((s'' \varphi_\alpha , 1)\pi_0 p_i ,  \ (s'' \varphi_\alpha , 1) \pi_1 p_i \big) \\
&= \big(s'' \varphi_\alpha p_i , \ p_i \big) \\
&= \big(s'' (\alpha , \alpha w) , \ p_i \big) \\
&= \big(s'' \sigma_\alpha , \ p_i \big) \\
&= \big(p_i s' \sigma_\alpha , \ p_i \big) \\
&= p_i (s' \sigma_\alpha , 1) 
\end{align*}
\noi and similarly
\begin{align*}
(t'' \varphi_\alpha , 1)p_i^2 
&= p_i (t' \sigma_\alpha , 1) 
\end{align*}
\noi showing that the squares on the left commute as described. Then since $p_0^2 q_2 = p_1^2 q_2$ we get
\[ p_0 (s' \sigma_\alpha , 1) q^2 = p_1 (s' \sigma_\alpha , 1) q^2 \]
\noi and 
\[ p_0 ( 1, t' \sigma_\alpha ) q^2 = p_1 (1, t' \sigma_\alpha ) q^2 \]

\noi inducing the unique vertical maps on the right in the lemma's diagrams by the universal property of the coequalizer $\mCW_1$. Now we show these are precisely $(se, 1), (1, te) : \mCW_1 \to \mCW_1 \tensor[_t]{\times}{_s} \mCW_1$. Notice the outer squares of the following pullback diagrams

\begin{center}
\begin{tikzcd}[]
\spn \ar[ddr, bend right, "s' e"'] \ar[drr, bend left, "q"] \ar[dr, dotted, "(s' \sigma_\alpha {,} 1) q^2"] & & \\
& \mCW_1 \tensor[_t]{\times}{_s} \mCW_1 \arrow[dr, phantom, "\usebox\pullback" , very near start, color=black] \dar["\pi_0^2"'] \rar["\pi_1^2"] & \mCW_1 \dar["s"] \\
& \mCW_1 \rar["t"'] & \mCW_0
\end{tikzcd} \qquad \qquad
\begin{tikzcd}[]
\spn \ar[ddr, bend right, "q"'] \ar[drr, bend left, "t'e"] \ar[dr, dotted, "(1 {,} t' \sigma_\alpha ) q^2"] & & \\
& \mCW_1 \tensor[_t]{\times}{_s} \mCW_1 \arrow[dr, phantom, "\usebox\pullback" , very near start, color=black] \dar["\pi_0^2"'] \rar["\pi_1^2"] & \mCW_1 \dar["s"] \\
& \mCW_1 \rar["t"'] & \mCW_0
\end{tikzcd}
\end{center}

\noi commute because $s' = q s$ implies
\begin{align*}
 e t 
= \sigma_\alpha q t 
= \sigma_\alpha t' 
= \sigma_\alpha \pi_1 t 
= (\alpha , \alpha w) \pi_1 t
= \alpha w t 
= 1_{\mC_0}
\end{align*}
\noi and similarly $t' = qt$ implies 

\[ e s = \sigma_\alpha q s = \sigma_\alpha s' = \alpha w t = 1_{\mC_0}. \]

\noi The triangles in the left diagram commute because 
\[(s' \sigma_\alpha {,} 1) q^2 \pi_0^2 
= (s' \sigma_\alpha {,} 1) \pi_0 q 
= s' \sigma_\alpha q 
= s' e \]
\noi and 
\[ (s' \sigma_\alpha {,} 1) q^2 \pi_1^2 
= (s' \sigma_\alpha {,} 1) \pi_1 q 
=q\]
\noi and a similar calculation shows the triangles on the right commute. 
\noi Then we have 

\[(s' \sigma_\alpha , 1) q^2 = (s' e , q) = (qse , q) = q (se, 1) \]
\noi and 
\[( 1, t' \sigma_\alpha) q^2 = (q, t' e) = (q, qte) = q (1, te). \]
\noi as required.

\end{proof}\

\begin{lem}\label{lem id law on representatives}
The diagram

\begin{center}
\begin{tikzcd}[column sep = large, row sep = large]
\spn \rar["s' \sigma_\alpha {,} 1)"] \ar[dr, two heads, "q"'] & \spn \tensor[_{t'}]{\times}{_{s'}} \spn \dar["c'"] & \spn \lar["(1 {,} t' \sigma_\alpha)"'] \ar[dl, two heads,"q"] \\
& \mCW_1 &
\end{tikzcd}
\end{center}

\noi commutes in $\cE$. 
\end{lem}
\begin{proof}
We show the left triangle commutes, the argument for the right triangle is similar. Pullback $u_1$ along $(s' \sigma_\alpha , 1)$ and then pullback along $u_0$ as shown in the diagram below to obtain a cover of $\spn$ that witnesses the entire composition process of an arbitary span and a pre-composable span representing the identity in $\mCW_1$. The following diagram commutes by definition. 

\begin{center}
\begin{tikzcd}[column sep = large ]
U^* \arrow[dr, phantom, "\usebox\pullback" , very near start, color=black]\ar[dd, bend right = 50, "/" marking, "u^*"' near end] \dar["/" marking, "u_0^*" near start] \rar["\pi_1"] & 
U 
\ar[dd, bend left = 50, "/" marking, "u" near start]
\dar["/" marking, "u_0"'near end] 
\rar["\sigma_\circ"] 
& 
\spn 
\dar[dd, two heads, "q"] 
\\
U^*_0 
\arrow[dr, phantom, "\usebox\pullback" , very near start, color=black] 
\dar["/" marking, "u_1^*"near start] 
\rar["\pi_1"] 
& 
U_0 
\dar["/" marking, "u_1"'near end] 
& 
\\
\spn 
\rar["(s' \sigma_\alpha {,} 1)"'] 
& 
\spn \tensor[_{t'}]{\times}{_{s'}} \spn 
\rar["c"'] 
& 
\mCW_1
\end{tikzcd}
\end{center}\

\noi By commutativity of the outer square above and since $u^*$ is epic, it suffices to show 

\[ u^* \pi_1 \sigma_\circ q = u^* q. \] 

\noi This can be done by translating the usual proof of the left identity law for span composition and defining a sailboat $\varphi : U^* \to \slb$ such that 
\[ \varphi p_0 = u^* \quad \text{ and } \quad \varphi p_1 = \pi_1 \sigma_\circ \]
\noi to give 
\[ u^* q =\varphi p_0 q = \varphi p_1 q = \pi_1 \sigma_\circ q. \]

\noi First compute 

\begin{align*}
\pi_1 \sigma_\circ 
&= \pi_1 \big( \omega \pi_1 , \ (\omega \pi_0 \pi_1 , u_0 \theta \pi_1 \pi_0 , \ u \pi_1 \pi_1 )c \big) 
\\
&= 
\big( 
\pi_1 \omega \pi_1 , \ 
(\pi_1 \omega \pi_0 \pi_1 , 
\pi_1 u_0 \theta \pi_1 \pi_0 , \ 
\pi_1 u \pi_1 \pi_1 
)c 
\big) \\
&= 
\big( 
\pi_1 \omega \pi_1 , \ 
(\pi_1 \omega \pi_0 \pi_1 , 
 u^*_0 \pi_1 \theta \pi_1 \pi_0 , \ 
 u^* \pi_1 
)c 
\big) .
\end{align*}

\noi Now notice that 
\begin{align*}
\pi_1 u \pi_0 \pi_0 w 
&= u^* (s' \sigma_\alpha, 1) \pi_0 \pi_0 w \\
&= u^* (s' \sigma_\alpha, 1) \pi_0 \pi_0 w \\
&= u^* s' \sigma_\alpha \pi_0 w \\
&= u^* s' \alpha w \\
&= u^* s' \sigma_\alpha \pi_1 \\
&= u^*_0 u^*_1 (s' \sigma_\alpha , 1) \pi_0 \pi_1\\
&= u^*_0 \pi_1 u_1 \pi_0 \pi_1 \\
&= u^*_0 \pi_1 \theta \pi_0 \pi_1 
\end{align*}

\noi and use this in the third line of the following calculation along with the definition of $W_\square$ (the Ore-condition) in the fifth line. 

\begin{align*}
\pi_1 \omega \pi_1 
&= 
(\pi_1 \omega \pi_0 \pi_0 , \
\pi_1 \omega \pi_0 \pi_1 w, \
\pi_1 \omega \pi_0 \pi_2 w
)c \\
&= 
(\pi_1 \omega \pi_0 \pi_0 , \
\pi_1 u_0 \theta \pi_0 \pi_0 w, \
\pi_1 u \pi_0 \pi_0 w
)c \\
&=
(\pi_1 \omega \pi_0 \pi_0 , \
u^*_0 \pi_1 \theta \pi_0 \pi_0 w, \
u^*_0 \pi_1 \theta \pi_0 \pi_1 
)c \\
&=
\big(
\pi_1 \omega \pi_0 \pi_0 , \
u^*_0 \pi_1 
(\theta \pi_0 \pi_0 w, \
 \theta \pi_0 \pi_1 )c
 \big)c \\
&=
\big(
\pi_1 \omega \pi_0 \pi_0 , \
u^*_0 \pi_1 
(\theta \pi_1 \pi_0 , \
 \theta \pi_1 \pi_1 w)c
 \big)c \\
&= 
(
\pi_1 \omega \pi_0 \pi_0 , \
u^*_0 \pi_1 \theta \pi_1 \pi_0 , \
u^*_0 \pi_1 \theta \pi_1 \pi_1 w
)c \\
&= 
(
\pi_1 \omega \pi_0 \pi_0 , \
u^*_0 \pi_1 \theta \pi_1 \pi_0 , \
u^*_0 \pi_1 u_1 \pi_1 \pi_0 w
)c \\
&= 
(
\pi_1 \omega \pi_0 \pi_0 , \
u^*_0 \pi_1 \theta \pi_1 \pi_0 , \
u^* (s' \sigma_\alpha , 1) \pi_1 \pi_0 w
)c \\
&= 
(
\pi_1 \omega \pi_0 \pi_0 , \
u^*_0 \pi_1 \theta \pi_1 \pi_0 , \
u^* \pi_0 w
)c 
\end{align*}

\noi The last calculation shows that the sailboat $\varphi : U^* \to \slb$ defined by 

\[ \varphi = 
\big( 
( 
( 
(\pi_1 \omega \pi_0 \pi_0 , u_0^* \pi_1 \theta \pi_1 \pi_0 w ) c , \
u^* \pi_0 ) ,\ 
\pi_1 \omega \pi_1 ) ,\ 
u^* \pi_1 \big) 
\]
\noi is well-defined. Clearly we have
\[ \varphi p_0 = \varphi (\pi_0 \pi_0 \pi_1 , \pi_1 ) = (u^* \pi_o , u^* \pi_1) = u^*\]
\noi and the first calculation along with associativity of composition in the last equality below shows us that 
\begin{align*}
   \varphi p_1 
   &= \varphi (\pi_0 \pi_1 , \ (\pi_0 \pi_0 \pi_0 , \pi_1)c) \\
   &= \big( \pi_1 \omega \pi_1 , \ ( (\pi_1 \omega \pi_0 \pi_0 , u^*_0 \pi_1 \theta \pi_1 \pi_0) c , u^* \pi_1\big)\\
   &= \pi_1 \sigma_\circ . 
\end{align*}  
\end{proof}\

\begin{prop}[Identity Laws] \label{Prop Id Laws}
The diagram 

\begin{center}
\begin{tikzcd}[]
\mCW_1 \rar["(se{,} 1)"] \ar[dr, equals]& \mCW_1 \tensor[_t]{\times}{_s} \mCW_1 \dar["c"] & \lar["(1{,} te)"'] \mCW_1 \ar[dl, equals] \\
& \mCW_1 &
\end{tikzcd}
\end{center}

\noi commutes in $\cE$.
\end{prop}
\begin{proof}
By Lemma~\ref{lem id span inclusions to composable spans}, the diagrams

\[
\begin{tikzcd}[]
\spn \dar["(s' \sigma_\alpha {,} 1)"'] \rar[two heads, "q"] & \mCW_1 \dar[dotted, "(s e {,} 1)"] & \\
\spn \tensor[_{t'}]{\times}{_{s'}} \spn \rar["q^2"'] \ar[dr, "c'"'] & \mCW_1 \tensor[_t]{\times}{_s} \mCW_1 \dar[dotted, "c"] \\
& \mCW_1
\end{tikzcd} \] 
\noi and 
\[
\begin{tikzcd}[]
 \spn \dar["(1 {,} t' \sigma_\alpha)"'] \rar[two heads, "q"] & \mCW_1 \dar[dotted, "(1 {,} t e)"] \\
\spn \tensor[_{t'}]{\times}{_{s'}} \spn \rar["q^2"'] \ar[dr, "c'"'] & \mCW_1 \tensor[_t]{\times}{_s} \mCW_1 \dar[dotted, "c"] \\
& \mCW_1
\end{tikzcd}
\]

\noi commute and by Lemma~\ref{lem id law on representatives} the composites on the left sides are both equal to $q$. It follows that the right-hand sides are identities by uniqueness. 
\end{proof}

\subsection{The Internal Localization Funtor}\label{S internal localization functor}

In this section we define the (internal) localization functor, $L : \mC \to \mCW$, prove it is an internal functor, define what it means for an internal functor to invert an arrow $w : W \to \mC_1$, and then show that $L$ inverts $w : W \to \mC_1$.

\subsubsection{Defining the Internal Functor}
The localizing internal functor, $L : \mC \to \mCW$, is defined on objects to be the identity map, $L_0 = 1_{\mC_0} $, because $\mCW_0 = \mC_0$. On arrows we use the section $\alpha : \mC_0 \to W$ along with the source map and the identity to get a (family of) span(s) which can be mapped to $\mCW_1$ as follows. 

\begin{center}
\begin{tikzcd}[column sep = large]
\mC_1 \ar[dr, "L_1"'] \rar["\big( s \alpha {,} ( s \alpha w {,} 1) c \big)"] & \spn \dar[two heads, "q"] \\
& \mCW_1 
\end{tikzcd}
\end{center}

\noi When $\cE = \Set$ this says $L_1$ maps an arrow $f : a \to b$ in $\mC_1$ to the equivalence class of spans represented by the span

\[ \begin{tikzcd}[]
a & \lar["\circ" marking, "\alpha(a)"'] a \rar[rr, bend left, "\alpha (a) f"] \rar["\circ" marking, "\alpha(a)"'] & a \rar["f"'] & b 
\end{tikzcd}. \] 

\noi Identities are preserved since $e s = 1_{\mC_0}$ and by the identity law, $(1 , te) c = 1_{\mC_1}$, in $\mC$ 

\begin{align*}
e L_1 
&= e \big( s \alpha {,} ( s \alpha w {,} 1) c \big) q \\
&= \big( e s \alpha {,} ( e s \alpha w {,} e ) c \big) q \\
&= \big( \alpha {,} ( \alpha w {,} e ) c \big) q \\
&= \big( \alpha {,} \alpha w ( 1 , t e ) c \big) q \\
&= ( \alpha {,} \alpha w ) q \\
\end{align*}

\noi where the last line is the identity structure map, $e = ( \alpha {,} \alpha w ) q : \mCW_0 \to \mCW_1$, for the internal category $\mCW$. This shows the diagram

\begin{center}
\begin{tikzcd}[]
\mC_0 \dar["e"'] \rar[equals, "L_0"] & \mCW_0 \dar["e"] \\
\mC_1 \rar["L_1"'] & \mCW_1
\end{tikzcd}
\end{center}

\noi commutes in $\cE$ so $L = (L_0, L_1)$ preserves the identity structure. Composition is preserved in a less obvious way. We need Lemma~\ref{lem intLocfunctor preserves comp helper lemma} to see 

\begin{align*}
c L_1 
&= c \big( s \alpha {,} ( s \alpha w {,} 1) c \big) q \\
&= \big( c s \alpha {,} ( c s \alpha w {,} c ) c \big) q \\
&= \big( \pi_0 s \alpha {,} ( \pi_0 s \alpha w {,} c ) c \big) q \\
&= \big( \pi_0 ( s \alpha {,} ( s \alpha w {,} 1) c ), \pi_1 ( s \alpha {,} ( s \alpha w {,} 1) c ) \big) c' \\
&= \big( \pi_0 ( s \alpha {,} ( s \alpha w {,} 1) c ), \pi_1 ( s \alpha {,} ( s \alpha w {,} 1) c ) \big) (q \times q) c \\
&= \big( \pi_0 ( s \alpha {,} ( s \alpha w {,} 1) c ) q , \pi_1 ( s \alpha {,} ( s \alpha w {,} 1) c )q \big) c \\
&= (L_1 \times L_1) c 
\end{align*}
\noi and conclude that the diagram 

\begin{center}
\begin{tikzcd}[]
\mC_2 \dar["c"'] \rar[ "L_1 \times L_1 "] & \mCW_1^2 \dar["c"] \\
\mC_1 \rar["L_1"'] & \mCW_1
\end{tikzcd}
\end{center}

\noi commutes in $\cE$. It follows that $L = (L_0, L_1)$ is an internal functor. 

\begin{lem}\label{lem intLocfunctor preserves comp helper lemma}
The diagram 

\begin{center}
\begin{tikzcd}[column sep = huge, row sep = large]
\mC_2 
\dar["\big( \pi_0 s \alpha {,} ( \pi_0 s \alpha w {,} c ) c \big) "']
\rar[rr,"\big( \pi_0 ( s \alpha {,} ( s \alpha w {,} 1) c ){,}\pi_1 ( s \alpha {,} ( s \alpha w {,} 1) c ) \big)"] 
&& \spn^2
\dar["c'"]
\\
\spn
\rar[rr,two heads, "q"'] 
&& \mCW_1
\end{tikzcd}
\end{center}

\noi commutes in $\cE$. 

\end{lem}
\begin{proof}
Internalize the following figure 

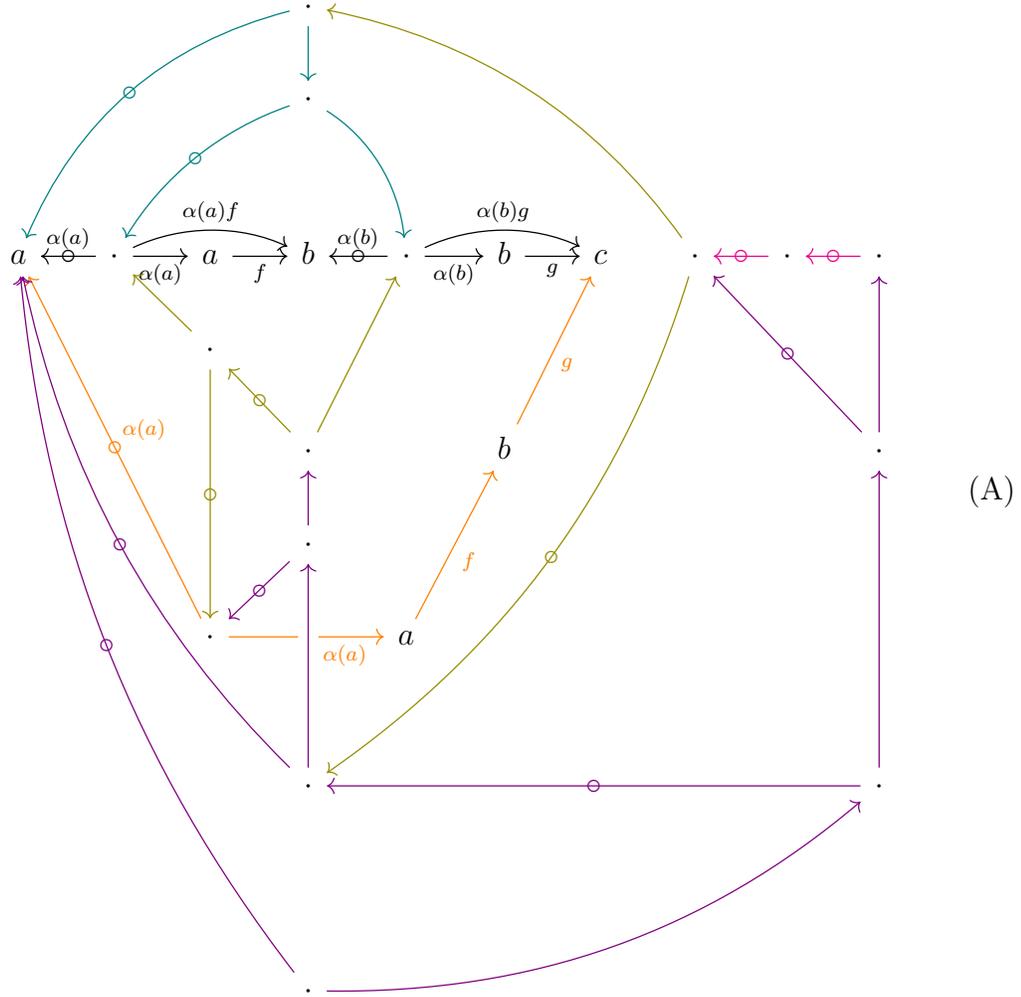
\begin{equation}\label{fig the big}
\begin{tikzcd}[column sep = scriptsize]
	&&& \cdot \\
	&&& \cdot \\
	\\
	a & \cdot & a & b & \cdot & b & c & \cdot & \cdot & \cdot \\
	&& \cdot \\
	&&& \cdot &&b & &&& \cdot \\
	&&& \cdot \\
	&& \cdot && a \\
	\\
	&&& \cdot &&&&&& \cdot \\
	\\
	\\
	&&& \cdot
	\arrow[from=4-2, to=4-1, "\circ" marking, "\alpha(a)"']
	\arrow[from=4-2, to=4-3 ,"\alpha(a)"']
	\arrow[from=4-3, to=4-4, "f"']
	\arrow[curve={height=-12pt}, from=4-2, to=4-4, "\alpha(a) f"]
	\arrow[from=4-5, to=4-4, "\circ" marking, "\alpha(b)"']
	\arrow[from=4-5, to=4-6,"\alpha(b)"']
	\arrow[from=4-6, to=4-7,"g"']
	\arrow[curve={height=-12pt}, from=4-5, to=4-7, "\alpha(b) g"]
	\arrow[curve={height=12pt}, from=2-4, to=4-2, "\circ" marking,, teal]
	\arrow[curve={height=-12pt}, from=2-4, to=4-5, teal]
	\arrow[from=1-4, to=2-4, teal]
	\arrow[curve={height=25pt}, from=1-4, to=4-1, "\circ" marking, teal]
	\arrow[from=8-3, to=4-1, "\circ" marking, orange, "\alpha(a)"']
	\arrow[from=8-3, to=8-5, "\alpha(a)"' near end, orange]
	\arrow[from=8-5, to=6-6, orange, "f"']
	\arrow[from= 6-6, to=4-7, "g"', orange]
	\arrow[from=5-3, to=4-2, olive]
	\arrow[from=5-3, to=8-3, olive, "\circ" marking]
	\arrow[from=6-4, to=5-3, olive, "\circ" marking]
	\arrow[from=6-4, to=4-5, olive]
	\arrow[from=7-4, to=6-4, violet]
	\arrow[from=7-4, to=8-3, violet, "\circ" marking]
	\arrow[from=10-4, to=7-4, violet, crossing over]
	\arrow[curve={height=-24pt}, from=10-4, to=4-1, violet, "\circ" marking]
	\arrow[curve={height=-30pt}, from=4-8, to=10-4, olive, "\circ" marking]
	\arrow[curve={height=24pt}, from=4-8, to=1-4, olive]
	\arrow[from=4-9, to=4-8, magenta, "\circ" marking]
	\arrow[from=4-10, to=4-9,magenta, "\circ" marking]
	\arrow[from=6-10, to=4-10, violet]
	\arrow[from=6-10, to=4-8, violet, "\circ" marking]
	\arrow[from=10-10, to=6-10,violet]
	\arrow[from=10-10, to=10-4, "\circ" marking, violet]
	\arrow[curve={height=30pt}, from=13-4, to=10-10, violet]
	\arrow[curve={height=-30pt}, from=13-4, to=4-1,"\circ" marking, violet]
\end{tikzcd} \tag{A} 
\end{equation}

\noi where olive coloured arrows are fillers of Ore-squares, violet coloured arrows $W$-composition fillers, and magenta coloured arrows are zippering fillers. To do this internally, start by taking the pullback

\begin{equation}\label{pb comp + representatives for L_1}
\begin{tikzcd}[column sep = huge, row sep = large]
\tilde{U} \rar[rrr, "\pi"] \dar["/" marking, "\tilde{u}"' near end] \arrow[drrr, phantom, "\usebox\pullback" , very near start, color=black] &&& U \dar["/" marking, "u" near start] \\
\mC_2 
\rar[rrr,"\big( \pi_0 ( s \alpha {,} ( s \alpha w {,} 1) c ){,}\pi_1 ( s \alpha {,} ( s \alpha w {,} 1) c ) \big)"'] 
&&& \spn^2
\end{tikzcd}\tag{0}
\end{equation}

\noi to give a cover witnessing span-composition along with the representative spans for $L_1$. In Figure $(A)$, this gives access to the teal and black coloured arrows. We begin by internalizing the `inner' part of Figure \ref{fig the big} 
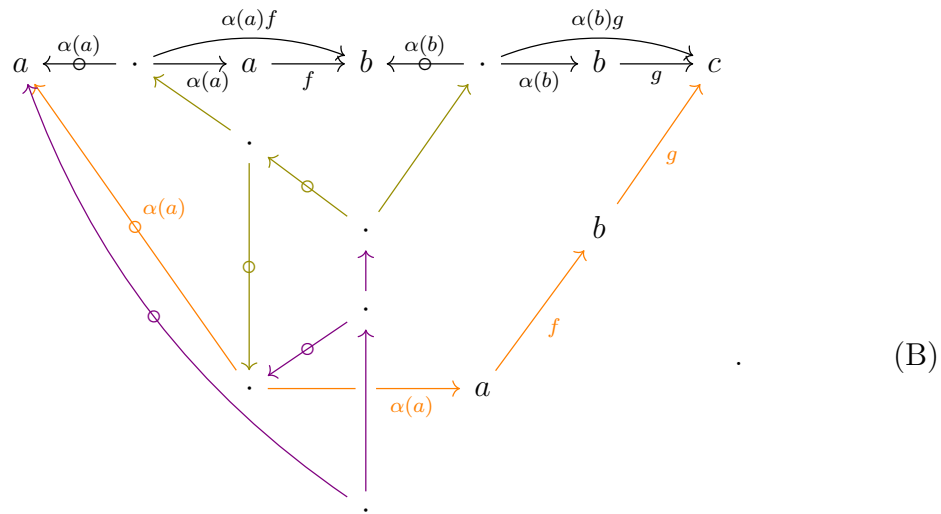
\begin{equation*}\label{fig inner part}
\begin{tikzcd}[row sep = scriptsize]
	a & \cdot & a & b & \cdot & b & c \\
	&& \cdot \\
	&&& \cdot &&b & \\
	&&& \cdot \\
	&& \cdot && a \\
	\\
	&&& \cdot &&& \\
	\\
	\\
	&&& \cdot
	\arrow[from=1-2, to=1-1, "\circ" marking, "\alpha(a)"']
	\arrow[from=1-2, to=1-3 ,"\alpha(a)"', near end]
	\arrow[from=1-3, to=1-4, "f"']
	\arrow[curve={height=-12pt}, from=1-2, to=1-4, "\alpha(a) f"]
	\arrow[from=1-5, to=1-4, "\circ" marking, "\alpha(b)"']
	\arrow[from=1-5, to=1-6,"\alpha(b)"']
	\arrow[from=1-6, to=1-7,"g"']
	\arrow[curve={height=-12pt}, from=1-5, to=1-7, "\alpha(b) g"]
	\arrow[from=5-3, to=1-1, "\circ" marking, orange, "\alpha(a)"']
	\arrow[from=5-3, to=5-5, "\alpha(a)"' near end, orange]
	\arrow[from=5-5, to=3-6, orange, "f"']
	\arrow[from= 3-6, to=1-7, "g"', orange]
	\arrow[from=2-3, to=1-2, olive]
	\arrow[from=2-3, to=5-3, olive, "\circ" marking]
	\arrow[from=3-4, to=2-3, olive, "\circ" marking]
	\arrow[from=3-4, to=1-5, olive]
	\arrow[from=4-4, to=3-4, violet]
	\arrow[from=4-4, to=5-3, violet, "\circ" marking]
	\arrow[from=7-4, to=4-4, violet, crossing over]
	\arrow[curve={height=-24pt}, from=7-4, to=1-1, violet, "\circ" marking]
\end{tikzcd}. \tag{B}
\end{equation*}

\noi We can build diagram (\ref{cover dgm - `inner' part}) below by noticing 
\[ \label{eq alpha(a)} \pi u \pi_0 \pi_0 = \tilde{u} \pi_0 s \alpha \tag{a}\]

\noi which means there is a unique map $(\tilde{u} \pi_0 s \alpha w , \pi u \pi_0 \pi_0) : \tilde{U} \to \csp$. Applying the internal Ore condition, \textbf{Int.Frc.(3)}, witnesses the first (family of) Ore-square(s), $\theta_0: \hat{U}_8 \to W_\square$, and using diagram (\ref{pb comp + representatives for L_1}) and equation $(a)$ above we can rewrite 

\[ \hat{u}_9\pi u \pi_0 \pi_1 = \hat{u}_9 \tilde{u} \pi_0 (s \alpha w , 1) c = (\hat{u}_9 \tilde{u} \pi_0 s \alpha w , \hat{u}_9 \tilde{u} \pi_0) c = (\theta_0 \pi_1 \pi_1 w , \hat{u}_9 \tilde{u} \pi_0) c \]

\noi and see its source coincides with the target of $\theta_0 \pi_1 \pi_0 : \hat{U}_8 \to W$ by definition of $W_\square$. This induces the unique map 
\[ (\theta_0 \pi_1 \pi_0 , \pi u \pi_0 \pi_1) : \hat{U}_8 \to \mC_2. \]
\noi Now the target of $\pi u \pi_0 \pi_1 : \tilde{U} \to \mC_1$ is the target of $\pi u \pi_1 \pi_0 : \tilde{U} \to W$ by definition of $\spn^2$, and this gives rise to the map
\[\big( (\theta_0 \pi_1 \pi_0 , \pi u \pi_0 \pi_1)c , \pi u \pi_1 \pi_0 \big) : \hat{U}_8 \to \csp\]

\noi in diagram (\ref{cover dgm - `inner' part}) below. Applying \textbf{Int.Frc.(3)} here witnesses the second (family of) Ore-square(s), $\theta_1 : \hat{U}_7 \to W_\square$. The map representing pairs of composable arrows in $W$, that induce the cover $\tilde{u}_f$ and the lift $\omega_0: \hat{U}_6 \to W_\circ$ by \textbf{Int.Frc.(2)} in diagram (\ref{cover dgm - `inner' part}) below, are pretty self-explanatory. The one inducing $\omega_1 : \hat{U}_5 \to W_\circ $ can be justified by chasing through the already established parts of diagram (\ref{cover dgm - `inner' part}) below. First notice that 
\[ \hat{u}_{7;8} \theta_0 \pi_0 \pi_1 = \hat{u}_{7;9} \tilde{u} \pi_0 s \alpha w \]
\noi and 
\[ \omega_0 \pi_1 = (\omega_0 \pi_0 \pi_0 , \theta_1 \pi_0 \pi_0 , \hat{u}_8 \theta_0 \pi_0 \pi_0 ) c .\]
\noi The target of this last composite is the target of the last arrow which is the source 

\[ \omega_0 \pi_1 w t = \hat{u}_{7;8} \theta_0 \pi_0 \pi_0 w t = \hat{u}_{7;8} \theta_0 \pi_0 \pi_1 s = \hat{u}_{7;9} \tilde{u} \pi_0 s \alpha w s. \]

\noi Applying \textbf{Int.Frc.(3)} induces the map $\omega_1 : \hat{u}_5 \to W_\circ$ and all together we get a commuting diagram of witnesses to the the inner part of Figure (\ref{fig the big}). 

\begin{equation}\label{cover dgm - `inner' part}
\begin{tikzcd}[column sep = large]
W_\square \rar["(\pi_0 \pi_1 {,} \pi_1 \pi_1)"] 
& \csp
&\\
\hat{U}_7
\uar["\theta_1", olive]
\rar["/" marking, "\hat{u}_8" near start] 
& \hat{U}_8
\uar["\big( ( \theta_0 \pi_1 \pi_0 {,} \hat{u}_9 \pi u \pi_0 \pi_1 )c {,} \hat{u}_9 \pi u \pi_1 \pi_0 \big) "']
\dar["\theta_0"', olive]
\rar["/" marking, "\hat{u}_9" near start] 
& \tilde{U} 
\dar["( \tilde{u} \pi_0 s \alpha w{,}\pi u \pi_0 \pi_0)"] 
\\
& W_\square \rar["(\pi_0 \pi_1 {,} \pi_1 \pi_1)"'] 
& \csp \\
& 
& 
\\
W_\circ \rar["(\pi_0 \pi_1 {,} \pi_0 \pi_2)"] 
& W \prescript{}{wt}{\times_{ws}} W
& 
\\
\hat{U}_5
\uar["\omega_1" violet] 
\rar["/" marking, "\hat{u}_6" near start] 
&\hat{U}_6
\uar["(\omega_0 \pi_1 {,} \hat{u}_{7;9} \tilde{u} \pi_0 s \alpha)"'] 
\dar["\omega_0"', violet]
\rar["/" marking, "\hat{u}_7" near start] 
& \hat{U}_7 \dar["(\theta_1 \pi_0 \pi_0 {,} \hat{u}_8 \theta_0 \pi_0 \pi_0)"']
\\
& W_\circ \rar["(\pi_0 \pi_1 {,} \pi_0 \pi_2)"'] 
& W \prescript{}{wt}{\times_{ws}} W
\end{tikzcd} \tag{1}
\end{equation}

\noi Applying \textbf{Int.Frc(3)} once followed by \textbf{Int.Frc(4)} twice gives local witnesses to the existence of the outer Ore-square and zippering arrows in magenta from Figure (\ref{fig the big}). Note that the first and second magenta arrows equalize the parallel pairs obtained by going around either side of the Ore-square and ending at the domains of $\alpha(a)$ and $\alpha(b)$ respectively. 

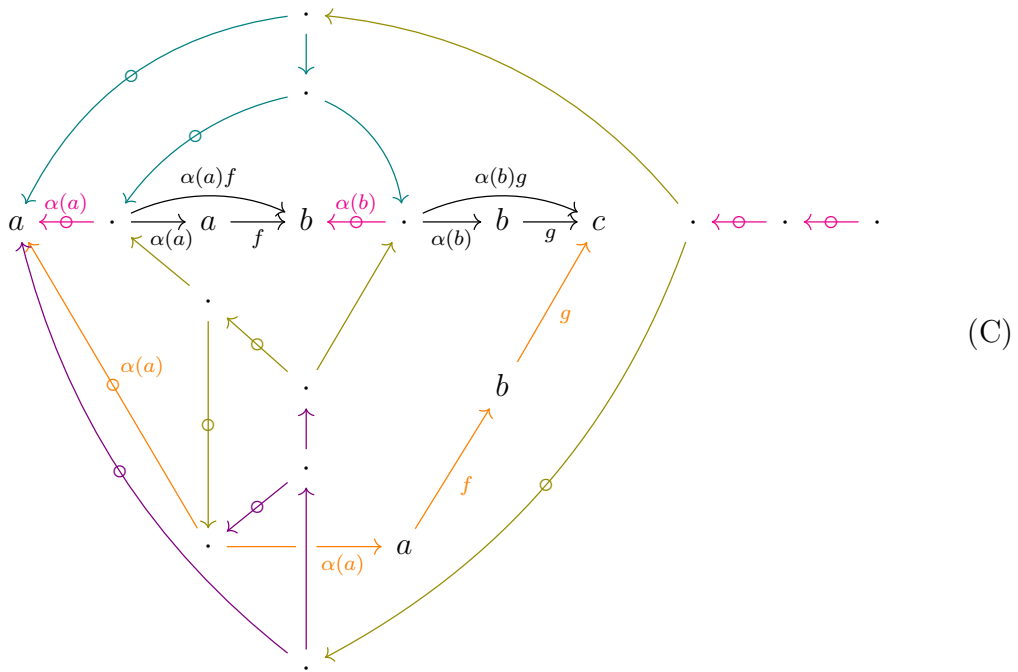
\begin{equation}\label{fig outer part 1}
\begin{tikzcd}[row sep = scriptsize, column sep = scriptsize]
	&&& \cdot \\
	&&& \cdot \\
	\\
	a & \cdot & a & b & \cdot & b & c & \cdot & \cdot & \cdot \\
	&& \cdot \\
	&&& \cdot &&b & &&& \\
	&&& \cdot \\
	&& \cdot && a \\
	\\
	&&& \cdot &&&&&& 
	\arrow[from=4-2, to=4-1, "\circ" marking, "\alpha(a)"', magenta]
	\arrow[from=4-2, to=4-3 ,"\alpha(a)"', near end]
	\arrow[from=4-3, to=4-4, "f"']
	\arrow[curve={height=-12pt}, from=4-2, to=4-4, "\alpha(a) f"]
	\arrow[from=4-5, to=4-4, "\circ" marking, "\alpha(b)"', magenta]
	\arrow[from=4-5, to=4-6,"\alpha(b)"']
	\arrow[from=4-6, to=4-7,"g"']
	\arrow[curve={height=-12pt}, from=4-5, to=4-7, "\alpha(b) g"]
	\arrow[curve={height=12pt}, from=2-4, to=4-2, "\circ" marking,, teal]
	\arrow[curve={height=-12pt}, from=2-4, to=4-5, teal]
	\arrow[from=1-4, to=2-4, teal]
	\arrow[curve={height=25pt}, from=1-4, to=4-1, "\circ" marking, teal]
	\arrow[from=8-3, to=4-1, "\circ" marking, orange, "\alpha(a)"']
	\arrow[from=8-3, to=8-5, "\alpha(a)"' near end, orange]
	\arrow[from=8-5, to=6-6, orange, "f"']
	\arrow[from= 6-6, to=4-7, "g"', orange]
	\arrow[from=5-3, to=4-2, olive]
	\arrow[from=5-3, to=8-3, olive, "\circ" marking]
	\arrow[from=6-4, to=5-3, olive, "\circ" marking]
	\arrow[from=6-4, to=4-5, olive]
	\arrow[from=7-4, to=6-4, violet]
	\arrow[from=7-4, to=8-3, violet, "\circ" marking]
	\arrow[from=10-4, to=7-4, violet, crossing over]
	\arrow[curve={height=-24pt}, from=10-4, to=4-1, violet, "\circ" marking]
	\arrow[curve={height=-30pt}, from=4-8, to=10-4, olive, "\circ" marking]
	\arrow[curve={height=24pt}, from=4-8, to=1-4, olive]
	\arrow[from=4-9, to=4-8, magenta, "\circ" marking]
	\arrow[from=4-10, to=4-9,magenta, "\circ" marking]
\end{tikzcd} \tag{C}
\end{equation}
 
\noi For the additional Ore-square added in (\ref{fig outer part 1}) recall the construction of the cover $u : u \to \spn^2$ and notice that target of the left leg of the composite coincides with the target of $\omega_1 \pi_1 : \hat{U}_5 \to W$,

\begin{align*}
   \hat{u}_{6;9} \pi \sigma_\circ \pi_0 w t
   &= \hat{u}_{6;9} \pi \omega \pi_1 w t \\
   &= \hat{u}_{6;9} \pi u \pi_0 \pi_0 w t\\
   &= \hat{u}_{6;9} s \alpha w t\\
   &= \omega_1 \pi_0 \pi_2 w t\\
   &= \omega_1 \pi_1 w t .
\end{align*}  

\noi This induces a unique map $(\omega_1 \pi_1 , \hat{u}_{6;9} \pi \sigma_\circ \pi_0 ) : \hat{U}_5 \to \csp$ which in turn gives a witnessing map $\theta_2 : \hat{U}_4 \to W_\square$ in diagram (\ref{cover dgm - `outer' part 1}) below. The map $\lambda' : \hat{U}_4 \to \cP_{cq} (\mC)$ is induced by $\lambda'' : \hat{U} \to P(\mC) \prescript{}{\pi_0 t}{\times_s} W$, which itself is induced by the universal property of the pullback $P(\mC) \prescript{}{\pi_0 t}{\times_s} W$ and can be defined explicitly as a pairing map by expanding each side of the equality determined by commutativity of the last Ore-square. Internally this is captured by the definition of $W_\square$ and the lift $\theta_2$ from \ore. On one side of the equality we have 

\begin{align}
\label{eq a side of outer Ore square}
\begin{split}
& \ \ \ \ 
( \theta_2 \pi_0 \pi_0 , \theta_2 \pi_0 \pi_1) c \\
&= (\theta_2 \pi_0 \pi_0 , \hat{u}_5 \omega_1 \pi_1)c \\
&= \big( \theta_2 \pi_0 \pi_0 , \hat{u}_5 ( \omega_1 \pi_0 \pi_0 , \hat{u}_6 \omega_0 \pi_1 , \hat{u}_{6;9} \tilde{u} \pi_0 s \alpha w )c \big)c \\
&= \big( \theta_2 \pi_0 \pi_0 , \hat{u}_5 \omega_1 \pi_0 \pi_0 , \hat{u}_{5;6} \omega_0 \pi_1 , \hat{u}_{5;9} \tilde{u} \pi_0 s \alpha w  \big)c \\
&=  \big( \theta_2 \pi_0 \pi_0 , \hat{u}_5 \omega_1 \pi_0 \pi_0 , \hat{u}_{5;6} (\omega_0 \pi_0 \pi_0 , \hat{u}_7 \theta_1 \pi_0 \pi_0, \hat{u}_{7;8} \theta_0 \pi_0 \pi_0) c ,  \hat{u}_{5;9} \tilde{u} \pi_0 s \alpha w   \big)c \\
&= \big( \theta_2 \pi_0 \pi_0 , 
\hat{u}_5 \omega_1 \pi_0 \pi_0 , 
\hat{u}_{5;6} \omega_0 \pi_0 \pi_0 , \hat{u}_{5;7 } \theta_1 \pi_0 \pi_0, \hat{u}_{5;8} ( \theta_0 \pi_0 \pi_0 w , 
 \hat{u}_{9} \tilde{u} \pi_0 s \alpha w )c   \big)c\\
 &= \big( \theta_2 \pi_0 \pi_0 , 
\hat{u}_5 \omega_1 \pi_0 \pi_0 , 
\hat{u}_{5;6} \omega_0 \pi_0 \pi_0 , \hat{u}_{5;7 } \theta_1 \pi_0 \pi_0, \hat{u}_{5;8} \theta_0 \pi_1 \pi_0 , 
 \hat{u}_{5;9} \tilde{u} \pi_0 s \alpha w   \big)c
\end{split}
\end{align}

\noi and on the other we have 

\begin{align}
\label{eq other side of outer Ore square}
\begin{split}
& \ \ \ \ ( \theta_2 \pi_1 \pi_0 , \theta_2 \pi_1 \pi_1 w ) c \\
&= (\theta_2 \pi_1 \pi_0 , \hat{u}_{5;9} \pi \sigma_\circ \pi_0 w )c \\
&= (\theta_2 \pi_1 \pi_0 , \hat{u}_{5;9} \pi \omega \pi_1 w )c \\
&= 	\big(\theta_2 \pi_1 \pi_0 ,
 		\hat{u}_{5;9} \pi ( \omega \pi_0 \pi_0 , 
 							 	u_0 \theta \pi_0 \pi_0 w , 
 							 	u \pi_0 w )c 
 		\big)c \\
&= \big(\theta_2 \pi_1 \pi_0 ,
 		\hat{u}_{5;9} \pi \omega \pi_0 \pi_0 , 
 		\hat{u}_{5;9} \pi u_0 \theta \pi_0 \pi_0 w , 
 		\hat{u}_{5;9} \pi u \pi_0 w 
 		\big)c \\
&= \big(\theta_2 \pi_1 \pi_0 ,
 		\hat{u}_{5;9} \pi \omega \pi_0 \pi_0 , 
 		\hat{u}_{5;9} \pi u_0 \theta \pi_0 \pi_0 w , 
 		\hat{u}_{5;9} \tilde{u} \pi_0 s \alpha w 
 		\big)c 
\end{split}.
\end{align}

\noi Notice that the last coordinates of the internal compositions described in the last lines of equations (\ref{eq a side of outer Ore square}) and (\ref{eq other side of outer Ore square}) coincide. Then the map $\lambda'' : \hat{U}_4 \to P(\mC) \prescript{}{t}{\times_{ws}} W$ is determined by the projections

\begin{align*}
\lambda'' \pi_1 &= \hat{u}_{5;9} \tilde{u} \pi_0 s \alpha w\\
\lambda'' \pi_0 \pi_0 &= (\theta_2 \pi_0 \pi_0 w , 
\hat{u}_5 \omega_1 \pi_0 \pi_0 , 
\hat{u}_{5;6} \omega_0 \pi_0 \pi_0 , \hat{u}_{5;7 } \theta_1 \pi_0 \pi_0, \hat{u}_{5;8} \theta_0 \pi_1 \pi_0 )c \\
 \lambda'' \pi_0 \pi_1 &= (\theta_2 \pi_1 \pi_0 ,
 		\hat{u}_{5;9} \pi \omega \pi_0 \pi_0 , 
 		\hat{u}_{5;9} \pi u_0 \theta \pi_0 \pi_0 w) c .
\end{align*}

\noi The left-hand sides of equations (\ref{eq a side of outer Ore square}) and (\ref{eq other side of outer Ore square}) are equal by definition of $W_\square$ and this induces the unique map $\lambda' : \hat{U}_4 \to \cP_{cq}(\mC)$ such that the triangle

\begin{center}
\begin{tikzcd}[ ]
\cP_{cq}(\mC) \rar[tail, "\iota_{cq}"] & P(\mC) \prescript{}{t}{\times_{ws}} W \\
\hat{U}_4 \uar[dotted, "\lambda'"] \ar[ur, "\lambda"'] 
\end{tikzcd}
\end{center}

\noi commutes by the universal property of the equalizer $\cP_{cq}(\mC)$. By definition of the pullback $\cP(\mC)$ we have 

\[ \lambda \pi_0 \iota_{eq} \pi_1 = \lambda \pi_1 \iota_{cq} \pi_0 = \lambda ' \iota_{cq} \pi_0 = \lambda '' \pi_0 ,\] 
\noi so that 
\begin{align} \label{eq pre-zip 0}
(\lambda \pi_0 \iota_{eq} \pi_0 w , \lambda'' \pi_0 \pi_0) c = (\lambda \pi_0 \iota_{eq} \pi_0 w , \lambda'' \pi_0 \pi_1) c. 
\end{align}

\noi Define 

\[ \label{def eta}\eta = ( \theta_2 \pi_0 \pi_0 w , 
\hat{u}_5 \omega_1 \pi_0 \pi_0 , 
\hat{u}_{5;6} \omega_0 \pi_0 \pi_0 , \hat{u}_{5;7 } \theta_1 \pi_1 \pi_0 )c \tag{Def. $\eta$} \] 

\noi and then by definition of the first two Ore-square maps in diagram (\ref{cover dgm - `inner' part}) we have 

\begin{align}\label{eq pre zip 1a}
\begin{split}
& \ \ \ \ \big( \lambda '' \pi_0 \pi_0 , \hat{u}_{5;9} \pi u \pi_0 \pi_1 \big) c \\
&= \big( (\theta_2 \pi_0 \pi_0 w , 
\hat{u}_5 \omega_1 \pi_0 \pi_0 , 
\hat{u}_{5;6} \omega_0 \pi_0 \pi_0 , \hat{u}_{5;7 } \theta_1 \pi_0 \pi_0w , \hat{u}_{5;8} \theta_0 \pi_1 \pi_0 )c, 
\hat{u}_{5;9} \pi u \pi_0 \pi_1 \big) c \\
&= \big( \theta_2 \pi_0 \pi_0 w , 
\hat{u}_5 \omega_1 \pi_0 \pi_0 , 
\hat{u}_{5;6} \omega_0 \pi_0 \pi_0 , \hat{u}_{5;7 } \theta_1 \pi_0 \pi_0 w, \hat{u}_{5;8}( \theta_0 \pi_1 \pi_0 , 
\hat{u} 9 \pi u \pi_0 \pi_1 ) c\big) c \\
&= \big( \theta_2 \pi_0 \pi_0 w , 
\hat{u}_5 \omega_1 \pi_0 \pi_0 , 
\hat{u}_{5;6} \omega_0 \pi_0 \pi_0 , \hat{u}_{5;7 } \theta_1 \pi_0 \pi_0 w, \hat{u}_{5;7} \theta_1 \pi_0 \pi_1 \big) c \\
&= \big( \theta_2 \pi_0 \pi_0 w , 
\hat{u}_5 \omega_1 \pi_0 \pi_0 , 
\hat{u}_{5;6} \omega_0 \pi_0 \pi_0 , \hat{u}_{5;7 } (\theta_1 \pi_0 \pi_0w , \theta_1 \pi_0 \pi_1)c \big) c \\
&= \big( \theta_2 \pi_0 \pi_0 w , 
\hat{u}_5 \omega_1 \pi_0 \pi_0 , 
\hat{u}_{5;6} \omega_0 \pi_0 \pi_0 , \hat{u}_{5;7 } (\theta_1 \pi_1 \pi_0 , \theta_1 \pi_1 \pi_1 w)c \big) c \\
&= \big( \theta_2 \pi_0 \pi_0 w , 
\hat{u}_5 \omega_1 \pi_0 \pi_0 , 
\hat{u}_{5;6} \omega_0 \pi_0 \pi_0 , \hat{u}_{5;7 } \theta_1 \pi_1 \pi_0 , \hat{u}_{5;9 } \pi u \pi_1 \pi_0 w \big) c \\
&= ( \eta , \hat{u}_{5;9 } \pi u \pi_1 \pi_0 w ) c 
\end{split}. 
\end{align}

\noi It will help to define, 

\[ \label{def nu} \nu = (\theta_2 \pi_1 \pi_0 ,
 		\hat{u}_{5;9} \pi \omega \pi_0 \pi_0 , 
 		\hat{u}_{5;9} \pi u_0 \theta \pi_1 \pi_0 w )c \tag{Def. $\nu$}\]

\noi and by definition of the Ore-square in the definition of composition on representative spans, $\sigma_\circ : U \to \spn$, we have
\begin{align}\label{eq pre zip 1b}
\begin{split}
& \ \ \ \ \big( \lambda '' \pi_0 \pi_1 , \hat{u}_{5;9} \pi u \pi_0 \pi_1 \big) c \\
&= \big( (\theta_2 \pi_1 \pi_0 ,
 		\hat{u}_{5;9} \pi \omega \pi_0 \pi_0 , 
 		\hat{u}_{5;9} \pi u_0 \theta \pi_0 \pi_0 w) c , 
\hat{u}_{5;9} \pi u \pi_0 \pi_1 \big) c \\
&= \big( \theta_2 \pi_1 \pi_0 ,
 		\hat{u}_{5;9} \pi \omega \pi_0 \pi_0 , 
 		\hat{u}_{5;9} \pi u_0 ( \theta \pi_0 \pi_0 w , 
 u_1 \pi_0 \pi_1)c \big) c \\
 &= \big( \theta_2 \pi_1 \pi_0 ,
 		\hat{u}_{5;9} \pi \omega \pi_0 \pi_0 , 
 		\hat{u}_{5;9} \pi u_0 ( \theta \pi_0 \pi_0 w , 
\theta \pi_0 \pi_1 )c \big) c \\
 &= \big( \theta_2 \pi_1 \pi_0 ,
 		\hat{u}_{5;9} \pi \omega \pi_0 \pi_0 , 
 		\hat{u}_{5;9} \pi u_0 ( \theta \pi_1 \pi_0 w , 
\theta \pi_1 \pi_1 )c \big) c \\
 &= \big( \theta_2 \pi_1 \pi_0 ,
 		\hat{u}_{5;9} \pi \omega \pi_0 \pi_0 , 
 		\hat{u}_{5;9} \pi u_0 \theta \pi_1 \pi_0 w , 
\hat{u}_{5;9} \pi u \pi_1 \pi_0 w \big) c \\
&= (\nu , \hat{u}_{5;9} \pi u \pi_1 \pi_0 w )c
\end{split}
\end{align}
 
\noi Putting equations (\ref{eq pre-zip 0}), (\ref{eq pre zip 1a}), and (\ref{eq pre zip 1b}) all together gives 

\begin{align}\label{eq rho'' equalizes} 
\begin{split}
(\lambda \pi_0 \iota_{eq} \pi_0 w , \hat{u}_4 \eta , \hat{u}_{4;9} \pi u \pi_1 \pi_0 w ) c 
&=(\lambda \pi_0 \iota_{eq} \pi_0 w , \hat{u}_4 \lambda'' \pi_0 \pi_0 , \hat{u}_{4;9} \pi u \pi_0 \pi_1 )c \\
&= (\lambda \pi_0 \iota_{eq} \pi_0 w , \hat{u}_4 \lambda'' \pi_0 \pi_1 , \hat{u}_{4;9} \pi u \pi_0 \pi_1 ) c \\
&= (\lambda \pi_0 \iota_{eq} \pi_0 w ,  \hat{u}_4 \nu , \hat{u}_{4;9} \pi u \pi_1 \pi_0 w ) c
\end{split}
\end{align}

\noi and induces the unique map $\rho'' : \hat{U}_3 \to P(\mC) \prescript{}{t}{\times_{ws}} W$ which is determined by the projections 

\begin{align*}
\rho'' \pi_1 &= \hat{u}_{4;9} \pi u \pi_1 \pi_0 w , \\
\rho'' \pi_0 \pi_0 &= (\lambda \pi_0 \iota_{eq} \pi_0 w ,  \hat{u}_4 \eta)c , \\
\rho'' \pi_0 \pi_1 &= (\lambda \pi_0 \iota_{eq} \pi_0 w , \hat{u}_4 \nu)c . 
\end{align*}

\noi equation~(\ref{eq rho'' equalizes}) induces the unique map $\rho' : \hat{U}_3 \to \cP_{cq}(\mC)$ such that the triangle 

\begin{center}
\begin{tikzcd}[ ]
\cP_{cq}(\mC) \rar[tail, "\iota_{cq}"] & P(\mC) \prescript{}{t}{\times_{ws}} W \\
\hat{U} \uar[dotted, "\rho'"] \ar[ur, "\rho"'] 
\end{tikzcd}
\end{center}

\noi commutes by the universal property of the equalizer, $\cP_{cq}(\mC)$. 

\begin{equation}\label{cover dgm - `outer' part 1}
\begin{tikzcd}[column sep = large]
&\cP(\mC) \rar["\pi_1"'] 
& \cP_{cq}(\mC)
&
\\
\hat{U}_2
\dar["\rho"', magenta]
\rar["/" marking, "\hat{u}_3" near start] 
& \hat{U}_3
\uar["\lambda", magenta]
\dar["\rho' "]
\rar["/" marking, "\hat{u}_4" near start] 
& \hat{U}_4
\uar["\lambda'"']
\dar["\theta_2"', olive]
\rar["/" marking, "\hat{u}_5" near start] 
& \hat{U}_5
\dar["(\omega_1 \pi_1 {,} \hat{u}_{6;9} \pi \sigma_\circ \pi_0 )"] 
\\
\cP(\mC) \rar["\pi_1"'] 
& \cP_{cq}(\mC)
& W_\square \rar["(\pi_0 \pi_1 {,} \pi_1 \pi_1)"'] 
& \csp 
\end{tikzcd}.\tag{2}
\end{equation}

\noi Applying \textbf{Int.Frc(2)} three times gives a cover that witnesses everything in Diagram \ref{fig the big}, and from there we can find two sailboats, $\varphi , \psi : \hat{U} \to \slb$, along with a comparison span, $\varphi p_1 = \psi p_1$, whose left leg is in $W$. 

\begin{equation}\label{cover dgm - `outer' part 2}
\begin{tikzcd}[column sep = large]
W_\circ \rar["(\pi_0 \pi_1 {,} \pi_0 \pi_2)"] 
& W \prescript{}{wt}{\times_{ws}} W
&W_\circ \rar["(\pi_0 \pi_1 {,} \pi_0 \pi_2)"] 
& W \prescript{}{wt}{\times_{ws}} W \\
\hat{U} 
\dar[shift right, "\psi"', teal] 
\dar[shift left, "\varphi", orange] 
\uar["\omega_4", violet] 
\rar["/" marking, "\hat{u}_0" near start] 
&\hat{U}_0 
\uar["(\omega_3 \pi_1 {,} \hat{u}_{2;5} \omega_1 \pi_1)"'] 
\dar["\omega_3"', violet]
\rar["/" marking, "\hat{u}_1" near start] 
&\hat{U}_1
\uar["\omega_2", violet]
\dar["(\omega_2 \pi_1 {,} \hat{u}_{2;4}\theta_2 \pi_0 \pi_0)"] 
\rar["/" marking, "\hat{u}_2" near start] 
&\hat{U}_2
\uar["(\rho \pi_0 \iota_{eq} \pi_0 {,} \hat{u}_3 \lambda \pi_0 \iota_{eq} \pi_0 )"] 
\\
\slb
& W_\circ \rar["(\pi_0 \pi_1 {,} \pi_0 \pi_2)"'] 
& W \prescript{}{wt}{\times_{ws}} W
&
\end{tikzcd} \tag{3}
\end{equation}

\noi For defining the sailboats above we should notice that commutativity of the first two Ore-squares and weak-composition triangle along with the commuting forks given by zippering imply that the composites of solid arrows in Figure \ref{fig right leg intermediate spans} below are equal.

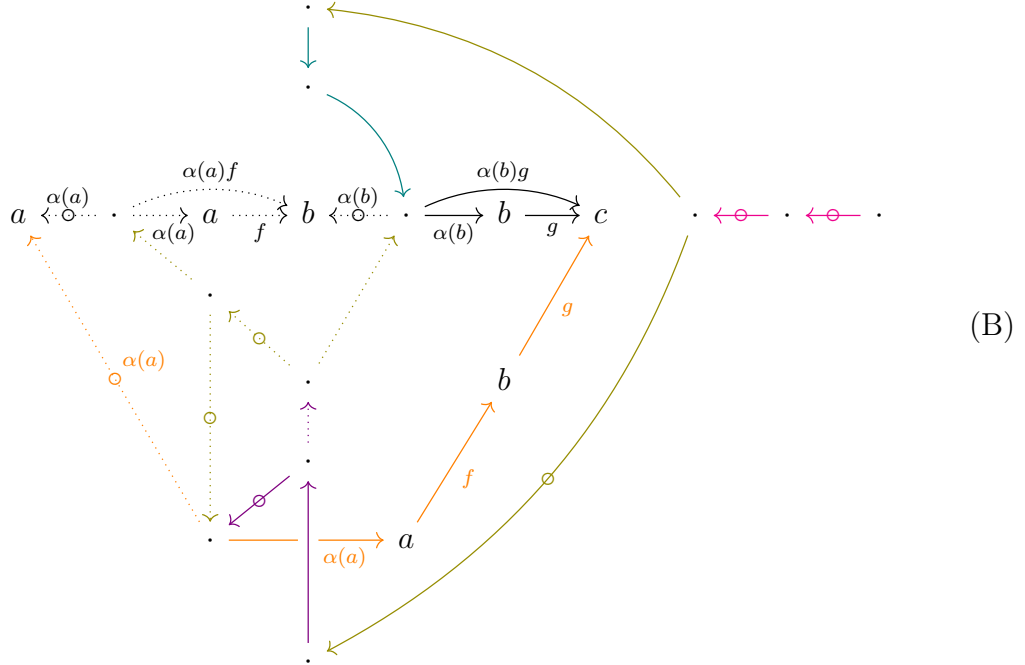
\begin{equation}\label{fig right leg intermediate spans} 
\begin{tikzcd}[row sep = scriptsize, column sep = scriptsize] 
	&&& \cdot \\
	&&& \cdot \\
	\\
	a & \cdot & a & b & \cdot & b & c & \cdot & \cdot & \cdot \\
	&& \cdot \\
	&&& \cdot &&b & &&& \\
	&&& \cdot \\
	&& \cdot && a \\
	\\
	&&& \cdot &&&&&& 
	\arrow[from=4-2, to=4-1, "\circ" marking, "\alpha(a)"', dotted]
	\arrow[from=4-2, to=4-3 ,"\alpha(a)"', dotted, near end]
	\arrow[from=4-3, to=4-4, "f"',dotted]
	\arrow[curve={height=-12pt}, from=4-2, to=4-4, "\alpha(a) f",dotted]
	\arrow[from=4-5, to=4-4, "\circ" marking, "\alpha(b)"',dotted]
	\arrow[from=4-5, to=4-6,"\alpha(b)"']
	\arrow[from=4-6, to=4-7,"g"']
	\arrow[curve={height=-12pt}, from=4-5, to=4-7, "\alpha(b) g"]
	\arrow[curve={height=-12pt}, from=2-4, to=4-5, teal]
	\arrow[from=1-4, to=2-4, teal]
	\arrow[from=8-3, to=4-1, "\circ" marking, orange, "\alpha(a)"', dotted]
	\arrow[from=8-3, to=8-5, "\alpha(a)"' near end, orange]
	\arrow[from=8-5, to=6-6, orange, "f"']
	\arrow[from= 6-6, to=4-7, "g"', orange]
	\arrow[from=5-3, to=4-2, olive, dotted]
	\arrow[from=5-3, to=8-3, olive, "\circ" marking, dotted]
	\arrow[from=6-4, to=5-3, olive, "\circ" marking, dotted]
	\arrow[from=6-4, to=4-5, olive, dotted]
	\arrow[from=7-4, to=6-4, violet, dotted]
	\arrow[from=7-4, to=8-3, violet, "\circ" marking]
	\arrow[from=10-4, to=7-4, violet, crossing over]
	\arrow[curve={height=-30pt}, from=4-8, to=10-4, olive, "\circ" marking]
	\arrow[curve={height=24pt}, from=4-8, to=1-4, olive]
	\arrow[from=4-9, to=4-8, magenta, "\circ" marking]
	\arrow[from=4-10, to=4-9,magenta, "\circ" marking]
\end{tikzcd}\tag{B}
\end{equation}

\noi This is seen internally by first taking the equation

\[(\rho \pi_0 \iota_{eq} \pi_0 w ,  \rho \pi_0 \iota_{eq} \pi_1 \pi_0) c = (\rho \pi_0 \iota_{eq} \pi_0 w ,  \rho \pi_0 \iota_{eq} \pi_1 \pi_1) c \]

\noi from the definition of the equalizer $\cP_{eq}(\mC)$, post-composing (in $\mC$) on both sides with $\hat{u}_{3;9} \pi u \pi_1 \pi_1 : \hat{U}_3 \to \mC_1$ and using associativity to get the equation

\begin{align}\label{eq final zip}
(\rho \pi_0 \iota_{eq} \pi_0 w , \rho \pi_0 \iota_{eq} \pi_1 \pi_0 , \hat{u}_{3;9} \pi u \pi_1 \pi_1 ) c 
&= (\rho \pi_0 \iota_{eq} \pi_0 w , \rho \pi_0 \iota_{eq} \pi_1 \pi_1 , \hat{u}_{3;9} \pi u \pi_1 \pi_1 ) c, 
\end{align}

\noi and then expanding the latter composites on both sides to get 

\begin{align*}
( \rho \pi_0 \iota_{eq} \pi_1 \pi_0 , \hat{u}_{3;9} \pi u \pi_1 \pi_1 ) c 
&= \big( \rho \pi_1 \iota_{cq} \pi_0 \pi_0 , (\hat{u}_{3;9} \tilde{u} \pi_1 (s \alpha w , 1) c \big)c\\ 
&= (\hat{u}_3 \rho ' \iota_{cq} \pi_0 \pi_0 ,  \hat{u}_{3;9} \tilde{u} \pi_1s \alpha w , \hat{u}_{3;9} \tilde{u} \pi_1 ) c \\
&= (\hat{u}_3 \rho '' \pi_0 \pi_0 ,  \hat{u}_{3;9} \pi u \pi_1 \pi_0 w , \hat{u}_{3;9} \tilde{u} \pi_1 ) c \\
&= \big( \hat{u}_3 (\lambda \pi_0 \iota_{eq} \pi_0 w , \hat{u}_{4} \eta)c,  \hat{u}_{3;9} \pi u \pi_1 \pi_0 w , \hat{u}_{3;9} \tilde{u} \pi_1 \big) c \\
&= ( \hat{u}_3 ( \lambda \pi_0 \iota_{eq} \pi_0 w , \hat{u}_4 \eta,  \hat{u}_{4;9} \pi u \pi_1 \pi_0 w)c , \hat{u}_{3;9} \tilde{u} \pi_1 ) c \\
\end{align*}

\noi from the left-hand side, and 
\begin{align*}
( \rho \pi_0 \iota_{eq} \pi_1 \pi_1 , \hat{u}_{3;9} \pi u \pi_1 \pi_1 ) c
&= ( \rho \pi_1 \iota_{cq} \pi_0 \pi_1 , \hat{u}_{3;9} \tilde{u} \pi_1 (s \alpha w , 1)c ) c \\
&= ( \hat{u}_3 \rho' \iota_{cq} \pi_0 \pi_1 ,  \hat{u}_{3;9} \tilde{u} \pi_1 s \alpha w , \hat{u}_{3;9} \tilde{u} \pi_1 ) c \\  
&= ( \hat{u}_3 \rho'' \pi_0 \pi_1 ,  \hat{u}_{3;9} \pi u \pi_1 \pi_0 w , \hat{u}_{3;9} \tilde{u} \pi_1 ) c \\
&= ( \hat{u}_3 ( \lambda \pi_0 \iota_{eq} \pi_0 w , \hat{u}_4 \nu)c ,  \hat{u}_{3;9} \pi u \pi_1 \pi_0 w , \hat{u}_{3;9} \tilde{u} \pi_1 ) c \\
&= ( \hat{u}_3 ( \lambda \pi_0 \iota_{eq} \pi_0 w , \hat{u}_4 \nu,  \hat{u}_{4;9} \pi u \pi_1 \pi_0 w)c , \hat{u}_{3;9} \tilde{u} \pi_1 ) c
\end{align*}

\noi from the right-hand side. These expansions are used below in equation~(\ref{eq mu_0' to mu_1'}) where we start establishing how the middle arrows for the sailboats picked out by $\varphi, \psi : \hat{U} \to \slb$ coincide. These middle arrows will be picked out by maps, $\mu_0, \mu_1 : \hat{U} \to \mC_1$, which will be internal composites, $\mu_0 = (\omega' ,\hat{u}_{0;2} \mu_0')c$ and $\mu_1 = (\omega', \hat{u}_{0;2} \mu_1')c$, for the arrow $\omega' :\hat{U} \to \mC_1$ that is defined after Figure (\ref{fig - witnessing sailboats}). The map picking out the part of the middle arrows in the sailboats determined by $\varphi : \hat{U} \to \slb$ is

\[ \mu_0' = (\rho \pi_0 \iota_{eq} \pi_0 w , \hat{u}_3 \lambda \pi_0 \iota_{eq} \pi_0 w , \hat{u}_{3;4} \theta_2 \pi_0 \pi_0 w , \hat{u}_{3;5} \omega_1 \pi_0 \pi_0 , \hat{u}_{3;6} \omega_0 \pi_1 w) c \]

\noi and by expanding the composites with the definitions of $\theta_0, \theta_1,$ and $ \omega_0$ in \ref{cover dgm - `inner' part} along with the fact that $\pi u \pi_0 \pi_1 = \tilde{u} \pi_0 (s \alpha w , 1) c$ from $\ref{pb comp + representatives for L_1}$ we can see

\[ (\rho \pi_0 \iota_{eq} \pi_0 w , \hat{u}_3 \lambda \iota_{eq} \pi_0 w , \hat{u}_{3;4} \eta , \hat{u}_{3;9} \pi u \pi_0 \pi_1 ) c 
= \big( \mu_0' , \hat{u}_{3;9} \tilde{u} c ( s \alpha w , 1) c \big)c . \]
\noi The middle of the sailboats being picked out by $\psi$ are given explicitly by
\[ \mu_1' = (\rho \pi_0 \iota_{eq} \pi_0 w , \hat{u}_3 \lambda \pi_0 \iota_{eq} \pi_0 w , \hat{u}_{3;4} \theta_2 \pi_1 \pi_0) c .\]
\noi Putting together equation~(\ref{eq final zip}) with the expansions and definitions of $\mu_0'$ and $\mu_1'$ and the definitions (\ref{def eta}) and (\ref{def nu}) and the definition of $\sigma_\circ : U \to \spn$. 

\begin{align}\label{eq mu_0' to mu_1'}
\begin{split} 
\big( \mu_0' , \hat{u}_{3;9} \tilde{u} c ( s \alpha w , 1) c \big)c 
&= \big( \rho \pi_0 \iota_{eq} \pi_0 w , \hat{u}_3 \lambda \iota_{eq} \pi_0 w , \hat{u}_{3;4} \eta , \hat{u}_{3;9} \pi u \pi_1 \pi_1 )c\\
&= \big( \rho \pi_0 \iota_{eq} \pi_0 w , \hat{u}_3 \lambda \iota_{eq} \pi_0 w , \hat{u}_{3;4} \nu , \hat{u}_{3;9} \pi u \pi_1 \pi_1 )c\\
&= \big(\mu_1' , \hat{u}_{3;9} \pi \sigma_\circ \pi_1)c
\end{split} .
\end{align}

\noi The maps $\varphi$ and $\psi$ picking out the sailboats can be seen in Figure (\ref{fig - witnessing sailboats}) below as the appropriate composites of the solid arrows.

\begin{equation}\label{fig - witnessing sailboats} 
\begin{tikzcd}[row sep = scriptsize, column sep = scriptsize] 
	&&& \cdot \\
	&&& \cdot \\
	\\
	a & \cdot & a & b & \cdot & b & c & \cdot & \cdot & \cdot \\
	&& \cdot \\
	&&& \cdot &&b & &&& \cdot \\
	&&& \cdot \\
	&& \cdot && a \\
	\\
	&&& \cdot &&&&&& \cdot \\
	\\
	\\
	&&& \cdot
	\arrow[from=4-2, to=4-1, "\circ" marking, "\alpha(a)"', dotted]
	\arrow[from=4-2, to=4-3 ,"\alpha(a)"', dotted, near end]
	\arrow[from=4-3, to=4-4, "f"', dotted]
	\arrow[curve={height=-12pt}, from=4-2, to=4-4, "\alpha(a) f", dotted]
	\arrow[from=4-5, to=4-4, "\circ" marking, "\alpha(b)"', dotted]
	\arrow[from=4-5, to=4-6,"\alpha(b)"']
	\arrow[from=4-6, to=4-7,"g"']
	\arrow[curve={height=-12pt}, from=4-5, to=4-7, "\alpha(b) g"]
	\arrow[curve={height=12pt}, from=2-4, to=4-2, "\circ" marking,teal, dotted]
	\arrow[curve={height=-12pt}, from=2-4, to=4-5, teal]
	\arrow[from=1-4, to=2-4, teal]
	\arrow[curve={height=25pt}, from=1-4, to=4-1, "\circ" marking, teal]
	\arrow[from=8-3, to=4-1, "\circ" marking, orange, "\alpha(a)"']
	\arrow[from=8-3, to=8-5, "\alpha(a)"' near end, orange]
	\arrow[from=8-5, to=6-6, orange, "f"']
	\arrow[from= 6-6, to=4-7, "g"', orange]
	\arrow[from=5-3, to=4-2, olive, dotted]
	\arrow[from=5-3, to=8-3, olive, "\circ" marking, dotted]
	\arrow[from=6-4, to=5-3, olive, "\circ" marking, dotted]
	\arrow[from=6-4, to=4-5, olive, dotted]
	\arrow[from=7-4, to=6-4, violet, dotted]
	\arrow[from=7-4, to=8-3, violet, "\circ" marking]
	\arrow[from=10-4, to=7-4, violet, crossing over]
	\arrow[curve={height=-24pt}, from=10-4, to=4-1, violet, "\circ" marking, dotted]
	\arrow[curve={height=-30pt}, from=4-8, to=10-4, olive, "\circ" marking]
	\arrow[curve={height=24pt}, from=4-8, to=1-4, olive]
	\arrow[from=4-9, to=4-8, magenta, "\circ" marking, dotted]
	\arrow[from=4-10, to=4-9,magenta, "\circ" marking, dotted]
	\arrow[from=6-10, to=4-10, violet, dotted]
	\arrow[from=6-10, to=4-8, violet, "\circ" marking]
	\arrow[from=10-10, to=6-10,violet]
	\arrow[from=10-10, to=10-4, "\circ" marking, violet, dotted]
	\arrow[curve={height=30pt}, from=13-4, to=10-10, violet]
	\arrow[curve={height=-30pt}, from=13-4, to=4-1,"\circ" marking, violet]
\end{tikzcd} \tag{C}
\end{equation}

\noi Explicitly define $\omega' : \hat{U} \to \mC_1$ to be the composite of the weak-composition arrows in (\ref{cover dgm - `outer' part 2}), 
\[ \label{def omega'} \omega' = (\omega_4 \pi_0 \pi_0 , \hat{u}_0 \omega_3 \pi_0 , \hat{u}_{0;1} \omega_2 \pi_0) c, \tag{Def. $\omega'$}\]
\noi and then $\mu_0 : \hat{U} \to \mC_1$ by 

\[ \label{def mu_0} \mu_0 = ( \omega' , \hat{u}_{0;2} \mu_0') c. \tag{Def. $\mu_0$} \]

\noi By definition of $\omega_4 : \hat{U} \to W_\circ$ we have
\[ \omega_4 \pi_1 = (\mu_0 , \hat{u} \tilde{u} \pi_0 s \alpha w ) c \]

\noi so the map $\varphi : \hat{U} \to \slb$, that picks out the sailboats in the bottom of Figure (\ref{fig - witnessing sailboats}) (consisting of orange and violet arrows and factoring through the bottom of the olive coloured Ore-square arrows), given by 

\[ \label{def varphi} \varphi = \big( ( (\mu_0, \hat{u} \tilde{u} \pi_0 s \alpha w ), \omega_4 \pi_1), \hat{u} \tilde{u} c ( s \alpha w , 1) c \big) \tag{Def. $\varphi$} \]
\noi is well-defined. Similarly define 

\[ \label{def mu_1} \mu_1 = (\omega' , \hat{u}_{0;2} \mu_1')c \tag{ Def. $\mu_1$ }.\] 

\noi By the first zippering, $\lambda : \hat{U}_3 \to \cP(\mC)$, in (\ref{cover dgm - `outer' part 1}), in particular by $\lambda \pi_0 : \hat{U}_3 \to \cP_{eq}(\mC)$ and the equalizer in its codomain we have that 

\begin{align*} 
& \ \ \ \ (\hat{u}_{0;2}\mu_1' , \hat{u}_{0 ;9} \pi \omega \pi_0 \pi_0 , \hat{u} \pi u_0 \theta \pi_0 \pi_0 w)c \\
&= (\hat{u}_{0;2} \rho \pi_0 \iota_{eq} \pi_0 w , \hat{u}_{0;3} \lambda \pi_0 \iota_{eq} \pi_0 w , \hat{u}_{0;3} \lambda \pi_0 \iota_{eq} \pi_1 \pi_1)c \\
&= (\hat{u}_{0;2}\rho \pi_0 \iota_{eq} \pi_0 w , \hat{u}_{0;3} \lambda \pi_0 \iota_{eq} \pi_0 w , \hat{u}_{0;3} \lambda \pi_0 \iota_{eq} \pi_1 \pi_0)c 
\end{align*}

\noi where the second to last line comes from the definition of $\mu_0'$ and the Ore-square picked out by $\theta_0 : \hat{U}_8 \to W_\square$ and the last line is by definition of $\omega_4 : \hat{U} \to W_\circ$. This allows us to see 

\begin{align*}
& \ \ \ \ (\mu_1 , \hat{u} \pi \sigma_\circ \pi_0 w) c\\
&= (\omega' , \hat{u}_{0;2}\mu_1' , \hat{u} \pi \omega \pi_0 \pi_0 , \hat{u} \pi u_0 \theta \pi_0 \pi_0 w, \hat{u} \pi u \pi_0 \pi_0 w) c \\
&= (\omega' , \hat{u}_{0;2} \mu_1' , \hat{u} \pi \omega \pi_0 \pi_0 \hat{u} \pi u_0 \theta \pi_0 \pi_0 w, \hat{u} \pi u \pi_0 \pi_0 w) c \\
&= (\omega' , \hat{u}_{0;2} \rho \pi_0 \iota_{eq} \pi_0 w , \hat{u}_{0;3} \lambda \pi_0 \iota_{eq} \pi_0 w , \hat{u}_{0;3} \lambda \pi_0 \iota_{eq} \pi_1 \pi_0, \hat{u} \pi u \pi_0 \pi_0 w)c
\\
&= (\omega ' , \hat{u}_{0;2} \mu_0' , \hat{u} \pi u \pi_0 \pi_0 w) \\
&= ( \mu_0, \hat{u} \tilde{u} \pi_0 s \alpha w) \\
&= \omega_4 \pi_1
\end{align*}

\noi so that the map $\psi : \hat{U} \to \slb$ determined by 

\[ \label{def psi} \psi = \big( ( ( \mu_1 , \hat{u} \pi \sigma_\circ \pi_0 ) , \omega_4 \pi_1 ), \hat{u} \pi \sigma_\circ \pi_1 \big) \tag{Def. $\psi$} \] 

\noi is well-defined. Notice the intermediate spans picked out by $\varphi$ and $\psi$ coincide due to the composites in Figure (\ref{fig right leg intermediate spans}) being equal. Formally, by (\ref{def varphi}), (\ref{def psi}), and equation~(\ref{eq mu_0' to mu_1'}), we can see

\begin{align}\label{eq varphi p_1}
\begin{split}
\varphi p_1 
&= \varphi (\pi_0 \pi_1, (\pi_0 \pi_0 \pi_0 , \pi_1)c) \\
&= \big( \omega_4 \pi_1 , (\mu_0, \hat{u}_{3;9} \tilde{u} c ( s \alpha w , 1) c \big) c \\
&= \big( \omega_4 \pi_1 , (\omega' , \hat{u}_{0;2} \mu_0' , \hat{u}_{3;9} \tilde{u} c ( s \alpha w , 1) c \big) c \\
&= \big( \omega_4 \pi_1 , (\omega' , \hat{u}_{0;2} \mu_1' , \hat{u}_{3;9} \tilde{u} \pi \sigma_\circ \pi_1) c \big) c \\ 
&= \big( \omega_4 \pi_1 , (\mu_1, \hat{u} \tilde{u} \pi \sigma_\circ \pi_1) c \big) c\\
&= \psi (\pi_0 \pi_1, (\pi_0 \pi_0 \pi_0 , \pi_1)c) \\
&= \psi p_1
\end{split}
\end{align} 

\noi Also notice that by (\ref{def varphi}), 

\begin{align}\label{eq varphi p_0} 
\begin{split}
\varphi p_0 
&= \varphi ( \pi_0 \pi_0 \pi_1 , \pi_1) \\
&= \big( \hat{u} \tilde{u} \pi_0 s \alpha w , \hat{u} \tilde{u} c ( s \alpha w , 1) c \big) \\
&= \hat{u} \tilde{u} \big( \pi_0 s \alpha w , c ( s \alpha w , 1) c \big) 
\end{split}
\end{align}

\noi and by (\ref{def psi}),

\begin{align}\label{eq psi p_0}
\begin{split}
\psi p_0 
&= \psi ( \pi_0 \pi_0 \pi_1 , \pi_1) \\
&= (\hat{u} \pi \sigma_\circ \pi_0 , \hat{u} \pi \sigma_\circ \pi_1 ) \\
&= \hat{u} \pi \sigma_\circ.
\end{split}
 \end{align}

\noi Putting equations (\ref{eq varphi p_1}), (\ref{eq psi p_0}), and (\ref{eq varphi p_0}) together gives

\begin{align*}
\hat{u} \tilde{u} \big( \pi_0 s \alpha w , c ( s \alpha w , 1) c \big) q 
&= \varphi p_0 q \\
&= \varphi p_1 q\\ 
&= \psi p_1 q \\
&= \psi p_0 q \\
&= \hat{u} \pi \sigma_\circ q \\
&= \hat{u} \pi u c'\\
&= \hat{u} \tilde{u} \big( \pi_0 ( s \alpha {,} ( s \alpha w {,} 1) c ){,}\pi_1 ( s \alpha {,} ( s \alpha w {,} 1) c ) \big) c'
\end{align*}

\noi and since $\hat{u} \tilde{u} : \hat{U} \to \mC_2$ is epic, we get 

\[ \big( \pi_0 s \alpha w , c ( s \alpha w , 1) c \big) q = \big( \pi_0 ( s \alpha {,} ( s \alpha w {,} 1) c ){,}\pi_1 ( s \alpha {,} ( s \alpha w {,} 1) c ) \big) c' \]

\noi as promised. 
\end{proof}

\subsubsection{Inverting the Canonical Cartesian Cleavage}

\noi Now that we know $L : \mC \to \mCW$ is an internal functor, we can show that it satisfies an important property. The rest of this section consists of lemmas leading to Proposition~\ref{prop int loc functor inverts W}, which shows that the localization functor, $L : \mC \to \mCW$, inverts $w : W \to \mC_1$ in the sense of the following definition. 

\begin{defn}\label{def invertibility of internal arrows}
We say a map $x : X \to \mC_1$ is {\em invertible} if there exists a map $x^{-1} : X \to \mC_1$ such that the diagrams 

\[ \begin{tikzcd}
X \ar[drr, bend left, "x^{-1}"] \ar[ddr, bend right, "x"'] \ar[dr, dotted, "(x {,} x^{-1})"] \\
& \mC_2 \dar["\pi_0"'] \rar["\pi_1"] & \mC_1 \dar["s"] \\
& \mC_1 \rar["t"'] & \mC_1 
\end{tikzcd} ,\qquad \qquad 
\begin{tikzcd}
X \ar[drr, bend left, "x"] \ar[ddr, bend right, "x^{-1}"'] \ar[dr, dotted, "(x^{-1} {,} x)"] \\
& \mC_2 \dar["\pi_0"'] \rar["\pi_1"] & \mC_1 \dar["s"] \\
& \mC_1 \rar["t"'] & \mC_1 
\end{tikzcd},\]
\[ \begin{tikzcd}
X \rar["(x{,} x^{-1})"] \dar["x"'] & \mC_2 \dar["c"] \\
\mC_1 \rar["s e "'] & \mC_1 
\end{tikzcd},\qquad \text{ and } \qquad 
\begin{tikzcd}
X \rar["(x^{-1}{,} x)"] \dar["x"'] & \mC_2 \dar["c"] \\
\mC_1 \rar["t e "'] & \mC_1 
\end{tikzcd}\]

\noi commute in $\cE$. In this case we say {\em $x^{-1}$ is an inverse for $x$ in $\mC$}. 
\end{defn}

\noi The next definition describes what it means for an internal functor to invert a class of arrows in its domain. 

\begin{defn}\label{def internal functor inverts a map}
We say an internal functor, $F : \mC \to \mD$ {\em inverts} $x : X \to \mC_1$ if there exists a map $F(x)^{-1} : X \to \mD_1$ such that $F(x)^{-1}$ is an inverse for the composite

\[ \begin{tikzcd}
X \rar["x"] \ar[dr, "F(x)"'] & \mC_1 \dar["F_1"] \\
& \mD_1
\end{tikzcd}\]

\noi in $\mD_1$. 
\end{defn}

\noi One might expect that internal functors preserve inverses, and sure enough the following lemma states and proves this: 

\begin{lem}\label{lem internal functors preserve internal inverses}
If $x : X \to \mC_1$ is an arrow in $\cE$ that has an inverse $x^{-1} : X \to \mC_1$ and $F : \mC \to \mX$ is an internal functor, then the composite 

\[\begin{tikzcd}
X \ar[dr, "F(x)"'] \rar["x"] & \mC_1 \dar["F_1"] \\
& \mD_1 
\end{tikzcd}\]
\noi has an inverse given by
\[\begin{tikzcd}
X \ar[dr, "F(x)^{-1}"'] \rar["x^{-1}"] & \mC_1 \dar["F_1"] \\
& \mD_1 
\end{tikzcd}. \]
\end{lem}
\begin{proof}
By functoriality we can compute
\begin{align*}
  (F(X) , F(X)^{-1}) c 
  &= (x F_1 , x^{-1} F_1 ) c \\
  &= (x , x^{-1}) c F_1 \\
  &= x s e F_1 \\
  &= x F_1 s e \\
  &= F(X) s e 
\end{align*}
\noi and 
\begin{align*}
  (F(X)^{-1} , F(X)) c 
  &= (x^{-1} F_1 , x F_1 ) c \\
  &= (x^{-1 , x}) c F_1 \\
  &= x t e F_1 \\
  &= x F_1 t e \\
  &= F(X) t e 
\end{align*}
\noi and the result follows from Definition~\ref{def internal functor inverts a map}. 
\end{proof}

\noi The following lemma shows how every span is equivalent to a canonical composite of spans and will be useful for proving our main result, Proposition~\ref{prop int loc functor inverts W}. The idea is that every span, represented as

\[\begin{tikzcd}
c & a \lar["\circ" marking, "v"'] \rar["f"] & b,
\end{tikzcd}\]

\noi is equivalent to a composite of the pair of composable spans, represented as

\[\begin{tikzcd}
c & a \lar["\circ" marking, "v"' near start] \rar[equals] & a & d \lar["\circ" marking, "\alpha(b)"'] \rar[rr,bend left = 40, "\alpha(a) f"] \rar[dashed, "\circ" marking, "\alpha(a)"] & a \rar[dashed, "f"] & b 
\end{tikzcd}\] 

\noi in particular. This is translated to the following statement about internal categories. 

\begin{lem}\label{lem factorization of equiv classes of spans}
The triangle

\begin{center}
\begin{tikzcd}[column sep = huge]
\spn \ar[drrr, two heads, "q"'] \rar[rrr, "\big( (\pi_0 {,} \pi_0 w s e) \ {,} \ (\pi_0 w s \alpha {,} (\pi_0 w s \alpha w {,} \pi_1 )c ) \big)"] &&& \spn^2 \dar["c'"] \\
&&& \mCW_1 
\end{tikzcd}
\end{center}

\noi commutes in $\cE$. 
\end{lem}

\begin{proof}
Let $\gamma : \spn \to \spn^2$ be the unique pairing map 
\[\gamma = \big( (\pi_0 {,} \pi_0 w s e) \ {,} \ (\pi_0 w s \alpha {,} (\pi_0 w s \alpha {,} \pi_1 )c ) \big).\] 

\noi Take the pullback of the cover $u : U \to \spn^2$, used to define $c' : \spn^2 \to \mCW_1$ in Lemmas \ref{lem defining c'} and \ref{lem actual definition of c'}, along $\gamma$ to get a cover $\tilde{u} : \tilde{U} \to \spn$ that witnesses the span composition process for the family of composable spans represented by $\gamma$.

\[ \begin{tikzcd}[]
\tilde{U} \arrow[dr, phantom, "\usebox\pullback" , very near start, color=black] \dar["/" marking , "\tilde{u}" near start] \rar["\pi"] & U \dar["/" marking , "u" near start] \\
\spn \rar["\gamma"'] & \spn^2 
\end{tikzcd}\] 

\noi It suffices to construct a (family of) sailboat(s) $\varphi : \tilde{U} \to \slb$ such that 
\[ \varphi p_0 = \tilde{u} \quad \text{ and } \varphi p_1 = \pi \sigma_\circ \]
\noi because that would give 
\[ \tilde{u} q = \varphi p_0 q = \varphi p_1 q = \pi \sigma_\circ q = \pi u c' = \tilde{u} \gamma c' \] 
\noi and since $\tilde{u}$ is epic we could conclude that 
\[ \gamma c' = q\]
\noi as desired. This family of sailboats will be constructed using the definition of the span composition, but let us take a moment to consider how this works when $\cE = \Set$. In this case, for each $f : a \to b$ in $\mC_1$ and $u : a \to c$ in $W$ we have a diagram that looks like

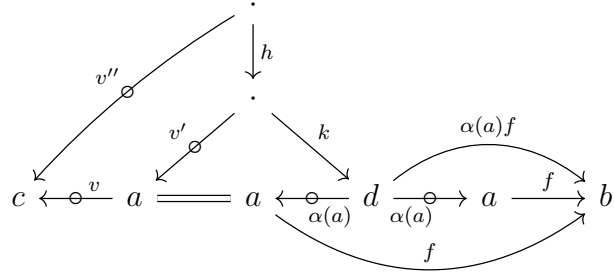
\[\begin{tikzcd}
	&& \cdot \\
	&& \cdot \\
	c & a & a & d & a & b
	\arrow[from=3-2, to=3-1, "\circ" marking, "v"' near start]
	\arrow[from=3-2, to=3-3, equals]
	\arrow[from=3-5, to=3-6, "f"]
	\arrow[from=3-4, to=3-3, "\circ" marking, "\alpha(a)"near start]
	\arrow[from=3-4, to=3-5,"\circ" marking, "\alpha(a)"' near start]
	\arrow[from=2-3, to=3-2, "\circ" marking, "v'"']
	\arrow[from=2-3, to=3-4, "k"]
	\arrow[curve={height=34pt}, from=3-3, to=3-6, "f"]
	\arrow[from=1-3, to=2-3, "h"]
	\arrow[curve={height=6pt}, from=1-3, to=3-1,"\circ" marking, "v''"']
	\arrow[curve={height=-25pt}, from=3-4, to=3-6, "\alpha(a)f"]
\end{tikzcd}\]

\noi which gives rise to the sailboat

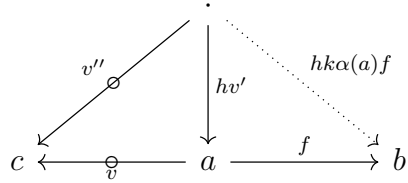
\begin{center}
\begin{tikzcd}[column sep = huge, row sep = huge]
	& \cdot \\
	c & a & b \\
	\arrow[from=1-2, to=2-1, "\circ" marking, "v''"']
	\arrow[from=1-2, to=2-3, dotted, "h k \alpha(a) f"]
	\arrow[from=2-2, to=2-1, "\circ" marking, "v"]
	\arrow[from=2-2, to=2-3, "f"]
	\arrow[from=1-2, to=2-2, "hv'"]
\end{tikzcd}
\end{center}

\noi showing that the span $(u, f)$ is equivalent to the composite $(u, 1) * (\alpha(a) , \alpha(a) f)$. The map picking out such sailboats internally, $\varphi : \tilde{U} \to \slb$, can be constructed by picking out the corresponding arrows through the composition process witnessed by the cover $u : U \to \spn^2$. Explicitly, this is given by

\[ \varphi = \big( ( (		(\pi \omega \pi_0 \pi_0 , \pi u_0 \theta \pi_0 \pi_0 w )c ,
									\tilde{u} \pi_0) , 
									\pi \sigma_\circ \pi_0 ), 
									\tilde{u} \pi_1 \big) \]

\noi and this is well-defined because $\tilde{u} \pi_0 = \tilde{u} \gamma \pi_0 \pi_0$ and the definition of $W_\circ$ shows 

\begin{align*}
\big( (\pi \omega \pi_0 \pi_0 , \pi u_0 \theta \pi_0 \pi_0 w )c , \tilde{u} \pi_0 w \big) c  
&= \big( \pi \omega \pi_0 \pi_0 , \pi u_0 \theta \pi_1 \pi_0, \tilde{u} \gamma \pi_0 \pi_0 w \big) c \\
&= \big( \pi \omega \pi_0 \pi_0 , \pi u_0 \theta \pi_0 \pi_0 w, \pi u \pi_0 \pi_0 w \big) c \\
&= \pi \omega \pi_1 w   \\
&= \pi \sigma_\circ \pi_0 w. 
\end{align*}				

\noi We can immediately see that 

\begin{align}\label{eq varphi p0 - factorization lemma for spans}
\varphi p_0 = \varphi (\pi_0 \pi_0 \pi_1 , \pi_1) = ( \tilde{u} \pi_0 , \tilde{u} \pi_1 ) = \tilde{u} 
\end{align}

\noi and by adding an identity map, $\tilde{U} \to \mC_1$, given by 

\[\tilde{u} \pi_0 w s e = \tilde{u} \gamma \pi_0 \pi_1 = \pi u \pi_0 \pi_1 \]

\noi into the following computation we can use the definition of $W_\square$, the fact that 

\[ (\tilde{u} \pi_0 w s \alpha w , \tilde{u} \pi_1 ) c = \tilde{u} \gamma \pi_1 = \pi u \pi_1,\] 

\noi and the definition of $\sigma_\circ : U \to \spn$ in Lemma~\ref{lem defining c'}.

\begin{align*}
\varphi (\pi_0 \pi_0 \pi_0 , \pi_1) c 
&= \big( (\pi \omega \pi_0 \pi_0 , \pi u_0 \theta \pi_0 \pi_0 w )c , \tilde{u} \pi_1 \big) c\\
&= (\pi \omega \pi_0 \pi_0 , \pi u_0 \theta \pi_0 \pi_0 w , \tilde{u} \pi_1 ) c\\
&= (\pi \omega \pi_0 \pi_0 , \pi u_0 \theta \pi_0 \pi_0 w, \tilde{u} \pi_0 w s e , \tilde{u} \pi_1 ) c\\
&= (\pi \omega \pi_0 \pi_0 , \pi u_0 \theta \pi_0 \pi_0 w, \pi u \pi_0 \pi_1 , \tilde{u} \pi_1 ) c\\
&= (\pi \omega \pi_0 \pi_0 , (\pi u_0 \theta \pi_0 \pi_0 w, \pi u \pi_0 \pi_1)c , \tilde{u} \pi_1 ) c\\
&= (\pi \omega \pi_0 \pi_0 , (\pi u_0 \theta \pi_1 \pi_0, \pi u \pi_1 \pi_0 w)c , \tilde{u} \pi_1 ) c\\
&= (\pi \omega \pi_0 \pi_0 , \pi u_0 \theta \pi_1 \pi_0, \tilde{u} \pi_0 w s \alpha w , \tilde{u} \pi_1 ) c\\
&= \pi (\omega \pi_0 \pi_0 , u_0 \theta \pi_1 \pi_0, u \pi_1 ) c\\
&= \pi \sigma_\circ \pi_1. 
\end{align*}

\noi Now we can easily see

\begin{align}\label{eq varphi p1 - factorization lemma for spans}
\begin{split}\varphi p_1 
&= \varphi \big(\pi_0 \pi_1 , (\pi_0 \pi_0 \pi_0 , \pi_1) c \big)\\
&= \big( \pi \sigma_\circ \pi_0 , \pi \sigma_\circ \pi_1) \\
&= \pi \sigma_\circ 
\end{split}
\end{align}

\noi The result follows from equations (\ref{eq varphi p0 - factorization lemma for spans}) and (\ref{eq varphi p1 - factorization lemma for spans}) as discussed at the beginning of this proof. 
\end{proof}

\noi The next lemma is used to give an equivalent representation of the identity spans in $\mCW$ which we use in the proof of Proposition~\ref{prop int loc functor inverts W}. 

\begin{lem} \label{lem equiv id span reps}
The diagram 

\begin{center}
\begin{tikzcd}[]
W \rar["(1 {,}w)"] \dar["wt"'] & \spn \dar[two heads, "q"] \\
\mC_0 \rar["(\alpha {,} \alpha w)q "'] & \mCW_1 
\end{tikzcd}
\end{center}

\noi commutes, where $\alpha : \mC_0 \to W$ is a section of $wt : W \to \mC_0$ from \targetsection. 
\end{lem}
\begin{proof} 
By \textbf{Int.Frc.(2)} and \textbf{Int.Frc.(3)} there exist covers, $\tilde{u}_0$ and $\tilde{u}_1$, and lifts, $\tilde{\omega}$ and $\tilde{\theta}$, that make the squares in the following diagram commute respectively: 

\begin{center}
\begin{tikzcd}[column sep = huge]
W_\circ \rar["(\pi_0 \pi_1 {,} \pi_0 \pi_2)"] 
& W \prescript{}{wt}{\times_{ws}} W 
& \\
\tilde{U} 
\uar[dotted, "\tilde{\omega}"] 
\rar["/" marking, "\tilde{u}_0" near start] 
& \tilde{U}_0
\uar[""'] 
\dar[dotted, "\tilde{\theta}"']
\rar["/" marking, "\tilde{u}_1" near start] 
& W \dar["(1{,}w t \alpha w)"] \\
&
W_\square \rar["(\pi_0 \pi_1 {,} \pi_1 \pi_1)"'] 
& \csp.
\end{tikzcd}
\end{center}

\noi The map $\tilde{\omega} \pi_1 : \tilde{U} \to W$ results in an intermediate (family of) span(s), $( \tilde{\omega} \pi_1 , \tilde{\omega} \pi_1 w) : \tilde{U} \to \spn$ and by definition of $W_\square$ and the maps in the diagram above we have that 
\[ \tilde{\omega} \pi_1 = ( \tilde{\omega} \pi_0 \pi_0 , \tilde{u}_0 \tilde{\theta} \pi_0 \pi_0 w, \tilde{u} w) . \]

\noi This gives a map $\tilde{U} \to W_\triangle$ and since $\slb = W_\triangle \tensor[_{\pi_0 \pi_1 s}]{\times}{_s} \mC_1$ we can see this determines a sailboat, $\varphi: \tilde{U} \to \slb$, given by the unique pairing map
\[ \varphi = \big( ( ( ( \tilde{\omega} \pi_0 \pi_0 , \tilde{u}_0 \tilde{\theta} \pi_0 \pi_0 w)c ,
								\tilde{u} ), 
								\tilde{\omega} \pi_1),
								\tilde{u} w\big). \]
\noi Similarly, we have 
\[ \tilde{\omega} \pi_1 = ( \tilde{\omega} \pi_0 \pi_0 , \tilde{u}_0 \tilde{\theta} \pi_1 \pi_0 , \tilde{u} w t \alpha w)\]

\noi giving another unique map $\tilde{U} \to W_\triangle$ and determining a sailboat $\psi: \tilde{U} \to \slb$, by the unique pairing map
\[ \psi = \big( ( ( ( \tilde{\omega} \pi_0 \pi_0 , \tilde{u}_0 \tilde{\theta} \pi_1 \pi_0 )c ,
								\tilde{u} w t \alpha ),
								\tilde{\omega} \pi_1),
								\tilde{u} w t \alpha w \big) .\]
								
\noi First we can use the calculations and definitions above (along with the definition of the pullback projections and how they interact with pairing maps) to see 

\begin{align*}
  \varphi p_0 
  &= \varphi (\pi_0 \pi_0 \pi_1 , \pi_1) \\
  &= \tilde{u} (1 , w)
\end{align*} 

\noi and 
\begin{align*}
  \varphi p_1 
  &= \varphi (\pi_0 \pi_1 , (\pi_0 \pi_0 \pi_0 , \pi_1)c) \\
  &= (\tilde{\omega} \pi_1 , \tilde{\omega} \pi_1 w)\\ &= \tilde{\omega} (\pi_1 , \pi_1 w)
\end{align*} 

\noi as well as 
\begin{align*}
  \psi p_0 
  &= \psi (\pi_0 \pi_0 \pi_1 , \pi_1) \\
  &= ( \tilde{u} w t \alpha , 	\tilde{u} w t \alpha w) \\
  &= ( \tilde{u} w t ( \alpha , \alpha w) 
\end{align*}
\noi and 
\begin{align*}
\psi p_1 
&= \psi (\pi_0 \pi_1 , (\pi_0 \pi_0 \pi_0 , \pi_1)c) \\
&= (\tilde{\omega} \pi_1 , \tilde{\omega} \pi_1 w) \\
&= \tilde{\omega} (\pi_1 , \pi_1 w) .
\end{align*}

\noi Putting it all together shows 
\[ \tilde{u} (1 , w)q = \varphi p_0 q = \varphi p_1q = \psi p_1 q = \psi p_0 q = \tilde{u} w t ( \alpha , \alpha w) q \]
\noi and since $\tilde{u}$ is epic we get that 

\[ (1 , w)q = w t ( \alpha , \alpha w) q  \]
\noi as desired. 
\end{proof}

\noi An immediate corollary to Lemma~\ref{lem equiv id span reps} is that the internal localization functor, $L : \mC \to \mCW$, maps the arrows from $w : W \to \mC_1$ to arrows in $\mCW$ that have left inverses.

\begin{cor}\label{cor left inverses for L(w) } 
The map $L_1 : \mC_1 \to \mCW_1$ has left inverses with respect to $w : W \to \mC_1$, in the sense that the diagram 

\begin{center}
\begin{tikzcd}[column sep = huge]
W \rar["\big( (1 {,} wse) q {,} w L_1 \big) "] \dar["(1{,} w)"'] & \mCW_2\dar["c"] \\
\spn \rar["q "'] & \mCW_1 
\end{tikzcd}
\end{center}

\noi commutes. 
\end{cor}

\begin{proof}

Consider the following diagram. 

\begin{center}
\begin{tikzcd}[column sep = large]
& & & W \dar[dddlll, bend right = 35, "(1{,} w)"'] \ar[ddd, "\big( ( 1 {,} w s e ) \ {,} \ ( w s \alpha w {,} (w s \alpha {,}w) c ) \big)"' ] \ar[dddr, "\big( (1{,} wse)q {,} w L_1 \big)"] & & \\
& & & & \\
& & & & \\
\spn \ar[drrr, two heads, "q"'] \rar[rrr, "\big( (\pi_0 {,} \pi_0 w s e) \ {,} \ (\pi_0 w s \alpha {,} (\pi_0 w s \alpha w {,} \pi_1 )c ) \big)"]  &&& \spn^2 \dar["c'"] \rar[two heads, "q \times q"] & \mCW_2 \ar[dl, "c"] \\
&&& \mCW_1 
\end{tikzcd}
\end{center}

\noi The bottom left triangle commutes by Lemma~\ref{lem factorization of equiv classes of spans}; the bottom right commutes by definition of $c$; the top left triangle commutes by the universal property of the pullback $\spn^2$; and the top right triangle commutes by the universal property of the pullback $\mCW_2$ along with the definitions of $L_1$ and $q \times q$. More precisely, post-composing the upper right triangle with the projection $\pi_1 : \mCW_2 \to \mCW_1$ gives precisely

\[ ( w s \alpha {,} (w s \alpha {,}w) c ) q = w ( s \alpha {,} ( s \alpha {,}1) c ) q = w L_1. \] 

\noi It follows that the diagram above commutes, in particular the outer square commutes. 
\end{proof}

\noi Next we prove a lemma that shows the internal localization functor, $L : \mC \to \mCW$, maps arrows coming from $w : W \to \mC_1$ to arrows that have right inverses in $\mCW$. 

\begin{lem}\label{lem right inverses for L(W)}
The diagram

\begin{center}
\begin{tikzcd}[column sep = huge]
W \rar["\big( w L_1 {,} (1 {,} wse) q \big) "] \dar["w s(\alpha {,} \alpha w)"'] & \mCW_2\dar["c"] \\
\spn \rar["q "'] & \mCW_1 
\end{tikzcd}
\end{center}

\noi commutes.
\end{lem}
\begin{proof}
Using the fact that 
\[ ( w s \alpha {,} (w s \alpha {,}w) c ) q = w ( s \alpha {,} ( s \alpha {,}1) c ) q = w L_1\] 

\noi and the universal property of the pullback $\mCW_2$ we can see the top triangle in the diagram, 

\begin{center}
\begin{tikzcd}[column sep = large]
W
\rar[rrrr, bend left = 40, "(w L_1 {,} ( 1 {,} wse) q \big) "]
 \dar["(w s \alpha {,} w s \alpha w)"']
 \rar[rrr,"\big( ( w s \alpha {,} ( w s \alpha w {,} w ) c) \ {,} \ (1 {,} w s e) \big) "] &&& \spn^2 \dar["c'"] \rar["q \times q"] & \mCW_2 \ar[dl, "c"] \\
\spn \rar[rrr,"q"'] &&& \mCW_1
\end{tikzcd},
\end{center}

\noi commutes. The right triangle commutes by definition so it suffices to show the bottom left square commutes.

\noi Let $\gamma = \big( ( w s \alpha {,} ( w s \alpha w {,} w ) c) \ {,} \ (1 {,} w s e) \big)$ and take the pullback of the cover $u : U \to \spn^2$ along $\gamma$. 

\begin{center}
\begin{tikzcd}[column sep = huge]
\tilde{U} \arrow[dr, phantom, "\usebox\pullback" , very near start, color=black] \rar[r, "\pi"] \dar["/" marking , "\tilde{u}"' near end]& U \dar["/" marking , "u" near start] \\
W \rar[r, "\gamma"']& \spn^2 
\end{tikzcd}
\end{center}

\noi The cover $\tilde{u} : \tilde{U} \to W$ witnesses the composition of the spans being picked out by $\gamma$. Thinking momentarily about the case when $\cE = \Set$ for visualization purposes, this says that for every arrow $v : a \to b$ in $W$ there exists a point in $\tilde{U}$ witnessing the commuting diagram: 

\[ \begin{tikzcd}[column sep = large] 
	&&\cdot \\
	&& \cdot \\
	a & c & a & b & a & a
	\arrow[from=3-3, to=3-4, "\circ" marking,"v" near start]
	\arrow[from=3-2, to=3-1,"\circ" marking, "\alpha(a)"]
	\arrow[from=3-2, to=3-3,"\circ" marking, "\alpha(a)" near end]
	\arrow[from=3-5, to=3-4,"\circ" marking, "v" near end]
	\arrow[from=3-5, to=3-6, equals]
	\arrow[from=2-3, to=3-2,"\circ" marking, "v'"']
	\arrow[curve={height=-6pt}, from=2-3, to=3-5, "k"]
	\arrow[from=1-3, to=2-3,"h"]
	\arrow[curve={height=18pt}, from=3-2, to=3-4, "\alpha(a) v"']
	\arrow[from=1-3, to=3-1, bend right = 10, "\circ" marking, "v''"']
\end{tikzcd}\]

\noi In this case the parallel pair of arrows, $hv' \alpha(a)$ and $hk$ being coequalized by $v : a \to b$ allows us to zipper before applying weak composition for $W$ to get a span whose left leg is in $W$, as pictured in the following diagram:

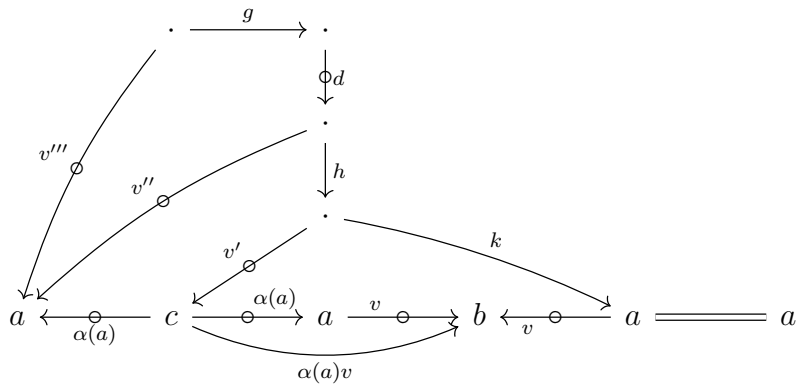
\[ \begin{tikzcd}[column sep = large] 
	& \cdot &\cdot \\
	&&\cdot \\
	&& \cdot \\
	a & c & a & b & a & a
	\arrow[from=4-3, to=4-4, "\circ" marking,"v" near start]
	\arrow[from=4-2, to=4-1,"\circ" marking, "\alpha(a)"]
	\arrow[from=4-2, to=4-3,"\circ" marking, "\alpha(a)" near end]
	\arrow[from=4-5, to=4-4,"\circ" marking, "v" near end]
	\arrow[from=4-5, to=4-6, equals]
	\arrow[from=3-3, to=4-2,"\circ" marking, "v'"']
	\arrow[curve={height=-6pt}, from=3-3, to=4-5, "k"]
	\arrow[from=2-3, to=3-3,"h"]
	\arrow[curve={height=18pt}, from=4-2, to=4-4, "\alpha(a) v"']
	\arrow[from=2-3, to=4-1, bend right = 10, "\circ" marking, "v''"']
	\arrow[from=1-3, to=2-3, "\circ" marking, "d"]
	\arrow[from=1-2, to=1-3, "g"] 
	\arrow[from=1-2, to=4-1, bend right = 10, "\circ" marking, "v'''"']
\end{tikzcd}\]

\noi The zippering axiom says there exists a map $d$ such that 
\[ d h v' \alpha(a) = d h k \]

\noi and the weak-composition axiom says there exists a map $g$ in the diagram above such that $gd v'' = v'''$ is in $W$. This data gives rise to two sailboats with a common projection, 

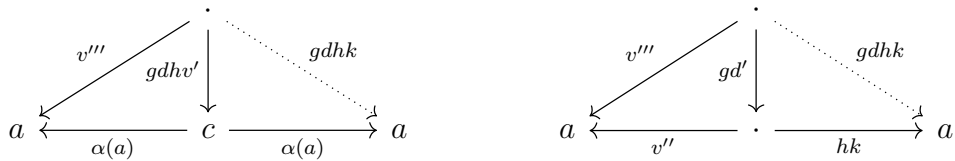
\begin{center}
\begin{tikzcd}[column sep = huge, row sep = large]
& \cdot \ar[dl, "v'''"'] \ar[dr, dotted, "gdhk"] \ar[d, "gdhv'"'] & \\
a & c \lar["\alpha(a)"] \rar["\alpha(a)"']  & a \\
\end{tikzcd}\qquad \qquad 
\begin{tikzcd}[column sep = huge, row sep = large]
& \cdot \ar[dl, "v'''"'] \ar[dr, dotted, "gdhk"] \ar[d, "gd'"'] & \\
a & \cdot \lar["v''"] \rar["hk"']  & a \\
\end{tikzcd}
\end{center}

\noi implying that the composite of spans, $(\alpha(a) , \alpha(a) v) * (v , 1_a) = (v'' , hk)$, is equivalent to the span $(\alpha(a), \alpha(a))$ by transitivity. Translating this argument to the internal setting for $\cE$ not necessarily equal to $\Set$ amounts to defining the map $\delta : \tilde{U} \to \cP_{cq}(\mC)$ in the following diagram,

\begin{center}
\begin{tikzcd}[column sep = large]
W_\circ \rar["(\pi_0 \pi_1 {,} \pi_0 \pi_2)"] 
& W \prescript{}{wt}{\times_{ws}} W 
&\\
\hat{U} 
\rar["/" marking, "\hat{u}_0" near start] 
\uar["\hat{\omega}"]
&
\hat{U}_0
\rar["/" marking, "\hat{u}" near start] 
\dar["\hat{\delta}"']
\uar["(\hat{\delta} \pi_0 \iota_{eq} \pi_0 {,} \hat{u}_1 \pi \sigma_\circ \pi_0) "']
& \tilde{U} 
\dar["\delta"] 
\\
& \cP (\mC)\rar["\pi_1"'] 
& \cP_{cq} (\mC) 
\end{tikzcd}
\end{center}

\noi applying \textbf{Int.Frc(4)} to get the cover $\hat{u}_1 : \hat{U}_0 \to \tilde{U}$ and the lift $\hat{\delta} : \hat{U}_1 \to \cP(\mC)$, and then applying by \textbf{Int.Frc(2)} to get the cover $\hat{u}_0 : \hat{U} \to \hat{U}_0$ and the lift $\hat{\omega} : \hat{U} \to W_\circ$. 

The map $\delta : \tilde{U} \to \cP_{cq}(\mC)$ is induced by the universal property of the equalizer $\cP_{cq}(\mC)$ and the map $\delta' : \tilde{U} \to P(\mC) \prescript{}{t}{\times_{ws}} W$. The map $\delta'$ is induced by the universal property of the pullback $P(\mC) \prescript{}{t}{\times_{ws}} W$ and to define it we start by using the definition of $W_\square$ to see

\[ \label{eq inducing delta'} \pi (\omega \pi_0 \pi_0 , u_0 \theta \pi_0 \pi_0 w , u \pi_0 \pi_1)c = \pi (\omega \pi_0 \pi_0, u_0 \theta \pi_1 \pi_0, u \pi_1 \pi_0w) c \tag{$\star$} .\]
\noi Since 
\[ \pi u \pi_0 \pi_1 = \tilde{u} \gamma \pi_0 \pi_1 = \tilde{u} (ws \alpha w, w)c \]
\noi the left-hand side of equation~(\ref{eq inducing delta'}) becomes 
\[ \pi (\omega \pi_0 \pi_0 ,  u_0 \theta \pi_0 \pi_0 w , u \pi_0 \pi_1) = (\pi \omega \pi_0 \pi_0 , \pi u_0 \theta \pi_0 \pi_0 w , \tilde{u} ws \alpha w, \tilde{u} w)c .\]
\noi Rewriting equation~(\ref{eq inducing delta'}) while recalling that $\pi u = \tilde{u} \gamma$ and $\gamma \pi_1 \pi_0 = 1_W$ gives 

\[ \label{eq inducing delta} (\pi \omega \pi_0 \pi_0 , \pi u_0 \theta \pi_0 \pi_0 w , \tilde{u} ws \alpha w, \tilde{u} w)c 
= (\pi\omega \pi_0 \pi_0,\pi u_0 \theta \pi_1 \pi_0, \tilde{u} w) c \tag{$\star \star$} \]

\noi and induces a unique $\delta' : \tilde{U} \to P(\mC) \prescript{}{t}{\times_{ws}} W$ such that 
\begin{align*}
\delta' \pi_1 &= \tilde{u} \\
\delta' \pi_0 \pi_0 &= (\pi \omega \pi_0 \pi_0 , \pi u_0 \theta \pi_0 \pi_0 w , \tilde{u} ws \alpha w)c \\
\delta' \pi_0 \pi_0 &= (\pi\omega \pi_0 \pi_0,\pi u_0 \theta \pi_1 \pi_0)c .
\end{align*}
\noi equation~(\ref{eq inducing delta}) can then be simplified as 
\begin{align*}
\delta' (\pi_0 \pi_0 , \pi_1)c = \delta' (\pi_0 \pi_1 , \pi_1)c
\end{align*}
\noi which induces the unique map $\delta : \tilde{U} \to \cP_{cq}(\mC)$ such that 

\begin{center}
\begin{tikzcd}[]
\cP_{cq}(\mC) \rar[tail, "\iota_{cq}"] & P(\mC) \prescript{}{t}{\times_{ws}} W \\
\tilde{U} \ar[ur, "\delta'"'] \uar[dotted, "\delta"] .
\end{tikzcd}
\end{center}

\noi The two (families of) sailboats, $\varphi , \psi : \hat{U} \to \slb$, with a common projection, $\varphi \pi_1 = \psi \pi_1$, can now be defined. By definition of 

\[ W_\circ = (\mC_1 \tensor[_t]{\times}{_{ws}} W \tensor[_{wt}]{\times}{_{ws}} W ) \tensor[_c]{\times}{_{ws}} W \] 

\noi we have 

\[ \hat{\omega} \pi_1 = ( \hat{\omega} \pi_0 \pi_0 , \hat{u}_0 \hat{\delta} \pi_0 \iota_{eq} \pi_0 w , \hat{u} \pi \sigma_\circ \pi_0 w)c \] 
\noi so let 
\[ \mu_0 = ( \hat{\omega} \pi_0 \pi_0 , \hat{u}_0 \hat{\delta} \pi_0 \iota_{eq} \pi_0 w )c \]
\noi to determine the unique pairing map
\[ \psi = \big( ( ( \mu_0, 
						\hat{u} \pi \sigma_\circ \pi_0 ), 
						\hat{\omega }\pi_1 ),
						\hat{u} \pi ( \omega \pi_0 \pi_0, 
						 u_0 \theta \pi_1 \pi_0)c \big). \]

\noi Notice the last map in the composite
\[ \hat{u} \pi u \pi_1 \pi_1 = \hat{u} \tilde{u} w s e \]
\noi is the identity structure map for $\mC$. The identity laws in $\mCW$ and $\mC$ can then both be used in the final calculation we need to determine the span projections for $\psi : \hat{U} \to \slb$. 

\begin{align*}
(\hat{u} \pi \omega \pi_0 \pi_0, \hat{u} \pi u_0 \theta \pi_1 \pi_0)c 
 &= \hat{u} \pi (\omega \pi_0 \pi_0 , u_0 \theta \pi_1 \pi_0)c \\
 &= (\hat{u} \pi \omega \pi_0 \pi_0 , \hat{u} \pi u_0 \theta \pi_1 \pi_0 )c \\
 &= (\hat{u} \pi \omega \pi_0 \pi_0 , \hat{u} \pi u_0 \theta \pi_1 \pi_0(1, t e)c  )c \\
 &= (\hat{u} \pi \omega \pi_0 \pi_0 , \hat{u} \pi u_0 \theta \pi_1 \pi_0, \hat{u} \pi u_0 \theta \pi_1 \pi_0 t e  )c \\
 &= (\hat{u} \pi \omega \pi_0 \pi_0 , \hat{u} \pi u_0 \theta \pi_1 \pi_0, \hat{u} \pi u_0 \theta \pi_1 \pi_1 w s e  )c \\
 &= (\hat{u} \pi \omega \pi_0 \pi_0 , \hat{u} \pi u_0 \theta \pi_1 \pi_0, \hat{u} \pi u \pi_1 \pi_0 w s e  )c \\
 &= (\hat{u} \pi \omega \pi_0 \pi_0 , \hat{u} \pi u_0 \theta \pi_1 \pi_0, \hat{u} \pi u \pi_1 \pi_1 s e  )c \\
 &= (\hat{u} \pi \omega \pi_0 \pi_0 , \hat{u} \pi u_0 \theta \pi_1 \pi_0, \hat{u} \tilde{ u} w s e s e  )c \\
 &= (\hat{u} \pi \omega \pi_0 \pi_0 , \hat{u} \pi u_0 \theta \pi_1 \pi_0, \hat{u} \tilde{ u} w s e  )c \\
 &= (\hat{u} \pi \omega \pi_0 \pi_0 , \hat{u} \pi u_0 \theta \pi_1 \pi_0, \hat{u} \pi u \pi_1 \pi_1 )c \\
 &= \hat{u} \pi ( \omega \pi_0 \pi_0 , u_0 \theta \pi_1 \pi_0, u \pi_1 \pi_1  )c \\
 &= \hat{u} \pi \sigma_\circ \pi_1 
\end{align*}
\noi Now we can see 
\[ \psi p_0 = \psi (\pi_0 \pi_0 \pi_1 , \pi_1) = (\hat{u} \pi \sigma_\circ \pi_0 , \hat{u} \pi \sigma_\circ \pi_1) = \hat{u} \pi \sigma_\circ \]
\noi and 
\[ \psi p_1 = \psi (\pi_0 \pi_1 , (\pi_0 \pi_0 \pi_0 , \pi_1 )c ) = \big( \hat{\omega }\pi_1 , (\mu_0, \hat{u} \pi \omega \pi_0 \pi_0, \hat{u} \pi u_0 \theta \pi_1 \pi_0)c \big). \]\

\noi For $\varphi : \hat{U} \to \slb$ let 
\[ \mu_1 = (\mu_0 , \hat{u} \pi \omega \pi_0 \pi_0 , \hat{u} \pi u_0 \pi_0 \pi_0w )c \]
\noi and notice on one hand that 
\[ (\mu_1 , \hat{u} \tilde{u} w s \alpha w )c = (\mu_1 , \hat{u} \pi u \pi_0 \pi_0 w) c = ( \mu_0 , \hat{u} \pi \sigma_\circ \pi_0 w)c = \hat{\omega} \pi_1\]
\noi and on the other hand that 
\begin{align*}
(\mu_1 , \hat{u} \tilde{u} w s \alpha w )c 
&= (\hat{\omega} \pi_0 \pi_0 , \delta \pi_0 \iota_{eq} \pi_0 , \delta \pi_0 \iota_{eq} \pi_1 \pi_0)c\\ 
&= (\hat{\omega} \pi_0 \pi_0 , \delta \pi_0 \iota_{eq} \pi_0 , \delta \pi_0 \iota_{eq} \pi_1 \pi_1)c\\ 
&= (\mu_0 , \hat{u} \pi \omega \pi_0 \pi_0, \hat{u} \pi u_0 \theta \pi_1 \pi_0)c 
\end{align*}
\noi Then define

\[ \varphi = \big( ( ( \mu_1 , 
								\hat{u} \tilde{u} w s \alpha ), 
								 \hat{\omega} \pi_1 ), 
								\hat{u} \tilde{u} w s \alpha w \big) \]
								
\noi and we get 

\begin{align*}
  \varphi p_0 
  &= (\pi_0 \pi_0 \pi_1 , \pi_1) \\
  &= (\hat{u} \tilde{u} w s \alpha, \hat{u} \tilde{u} w s \alpha w) 
\end{align*} 
\noi and 
\begin{align*}
  \varphi p_1 
  &= \varphi \big( \pi_0 \pi_1 , (\pi_0 \pi_0 \pi_0 , \pi_1)c \big) \\
  &= \big( \hat{\omega} \pi_1 , \mu_1 , \hat{u} \tilde{u} w s \alpha w )c \big) \\
  &= (\hat{\omega} \pi_1 , (\mu_0 , \hat{u} \pi \omega \pi_0 \pi_0, \hat{u} \pi u_0 \theta \pi_1 \pi_0)c \\
  &= \psi p_1.
\end{align*} 

\noi Combining our computations gives us that
\begin{align*}
  \hat{u} \tilde{u} (ws \alpha , ws \alpha w) q 
  &= \varphi p_0 q = \varphi p_1 q \\
  &= \psi p_1 q = \psi p_0 q \\
  &= \hat{u} \pi \sigma_\circ q \\
  &= \hat{u} \pi u c' \\
  &= \hat{u} \tilde{u} \gamma c' 
\end{align*} 
\noi and since the composite $\hat{u} \tilde{u}$ is epic we can conclude
\[ (ws \alpha , ws \alpha w) q = \gamma c'. \]
\end{proof}

\noi Now we prove the second main result of this section. 

\begin{prop}\label{prop int loc functor inverts W}
The localization (internal) functor, $L : \mC \to \mCW$ inverts $w : W \to \mC_1$. 
\end{prop}
\begin{proof}
Consider the composite 

\[ \begin{tikzcd}
W \rar["(1 {,} wse)"] \ar[dr] & \spn \dar[two heads, "q"] \\
& \mCW_1
\end{tikzcd}\]

\noi In the proofs of Lemma~\ref{lem right inverses for L(W)} and Corollary \ref{cor left inverses for L(w) } we have already seen that 

\[ (1,wse) q s = w L_1 t , \qquad (1,wse) q t = w L_1 s\] 

\noi so it suffices to show the last two diagrams from Definition~\ref{def internal functor inverts a map} commute in this setting. 

First note that $e = (\alpha, \alpha w) q : \mC_0 \to \mCW_1$ is the identity structure map on $\mCW$, and that the source map $s : \mCW_1 \to \mC_0$ is uniquely determined by the map $q s = \pi_0 w t : \spn \to \mC_0$. Also recall that 

\[w L_1 = (w s \alpha , (w s \alpha , w ) c ) q \]

\noi and then compute 
\begin{align*}
  w L_1 s e 
  &= w L_1 s (\alpha, \alpha w) q \\
  &= (w s \alpha , (w s \alpha , w ) c ) q s (\alpha, \alpha w) q \\
  &= (w s \alpha , (w s \alpha , w ) c ) \pi_0 w t (\alpha, \alpha w) q \\
  &= w s \alpha w t (\alpha, \alpha w) q \\
  &= w s (\alpha, \alpha w) q .
\end{align*}

\noi We can replace the left and bottom composite in the commuting square of Lemma~\ref{lem right inverses for L(W)} by the last equation to give the commuting square 
\begin{center}
\begin{tikzcd}[column sep = huge]
W \rar["\big( w L_1 {,} (1 {,} wse) q \big) "] 
\dar["w L_1"'] & \mCW_2\dar["c"] \\
\mCW_1 \rar["s e "'] & \mCW_1 
\end{tikzcd}
\end{center} 

\noi in $\cE$ and shows $(1, w s e)q : W \to \mCW_1$ satisfies half of Definition~\ref{def internal functor inverts a map}. For the rest of it we recall that $q t = \pi_1 t : \spn \to \mC_0$ uniquely determines the structure map $t : \mCW_1 \to \mC_0$ and similarly compute 

\begin{align*}
  w L_1 t e 
  &= w L_1 t (\alpha, \alpha w) q \\
  &= (w s \alpha , (w s \alpha , w ) c ) q t (\alpha, \alpha w) q \\
  &= (w s \alpha , (w s \alpha , w ) c ) \pi_1 t (\alpha, \alpha w) q \\
  &= (w s \alpha , w ) c t (\alpha, \alpha w) q \\
  &= w t (\alpha, \alpha w) q .
\end{align*}

\noi Putting this together with Lemma~\ref{lem equiv id span reps} gives

\[ w L_1 t e = w t (\alpha, \alpha w) q = (1,w)q\]

\noi and allows us to rewrite the commuting square in Corollary \ref{cor left inverses for L(w) } as

\begin{center}
\begin{tikzcd}[column sep = huge]
W \rar["\big( (1 {,} wse) q {,} w L_1 \big) "] \dar["w L_1"'] & \mCW_2\dar["c"] \\
\mCW_1 \rar["t e"'] & \mCW_1 
\end{tikzcd}.
\end{center}

\noi This means $(1 {,} wse) q $ inverts $(w L_1)$ by Definition~\ref{def internal functor inverts a map}.
\end{proof}

\subsection{Universal Property of Internal Fractions}\label{S UP of Internal Fractions}

The main result of this section is Theorem~\ref{thm 2-dimensional universal property of localization}, the universal property of internal localization. It is an isomorphism of categories between the category of internal functors, $\mC \to \mD$, that invert $w : W \to \mC_1$ and their natural transformations, and the category of internal functors $\mCW \to \mD$ and their natural transformations. In Section \ref{SS UPIntFrac 1-cells} we prove that the objects in each category uniquely correspond to one another in Proposition~\ref{prop UPIntFrac 1-cell corresondence}, and then in Lemma~\ref{lem int.loc 2-cell correspondence} we show that the 2-cells in each category uniquely correspond to one another. In Lemma~\ref{lem 2-cell correspondence functoriality} we show that the correspondence between natural transformations is functorial, and Theorem~\ref{thm 2-dimensional universal property of localization} follows immediately.

\subsubsection{Correspondence Between 1-cells}\label{SS UPIntFrac 1-cells} 

The results in this subsection come together to prove that for any internal functor $F : \mC \to \mX$ that inverts $w : W \to \mC_1$, there exists a unique internal functor $\lbrack F \rbrack : \mCW \to \mX$ such that the diagram 

\begin{center}
\begin{tikzcd}[]
\mC \ar[dr, "L"'] \rar[rr,"F"] & & \mX \\
& \mCW \ar[ur, dotted, "\lbrack F \rbrack"'] & 
\end{tikzcd}
\end{center}

\noi commutes. First we define $[F]$ and prove it is an internal functor, then we notice how every internal functor $\mCW \to \mX$ corresponds to an internal functor $\mC \to \mX$ that inverts $W$ by pre-composition with $L : \mC \to \mX$. Finally we show that these assignments are inverses to one another to prove the main result of this subsection, Proposition~\ref{prop UPIntFrac 1-cell corresondence}. 

It's clear how to define $[F]$ on objects: 

\begin{center}
\begin{tikzcd}[column sep = large]
\mCW_0 \dar[equals] \rar["\lbrack F \rbrack_0"] & \mX_0 \dar[equals] \\
\mC_0 \ar[r, "F_0"'] & \mX_0 
\end{tikzcd}.
\end{center}

\noi On arrows we use the universal property of the coequalizer $\mCW_1$. By Definition~\ref{def internal functor inverts a map}, there exists a map $F(w)^{-1} : W \to \mX_1$ that inverts $w F_1 : W \to \mX_1$ in $\mX$. That is, the diagrams 

\[ \begin{tikzcd}[]
W\dar["wF_1"'] \rar[rr,"(w F_1 {,} F(w)^{-1})"] && \mX_2 \dar["c"] \\
\mX_1 \rar[rr,"se"'] && \mX_1
\end{tikzcd}\qquad \qquad \qquad 
\begin{tikzcd}[]
W\dar["wF_1"'] \rar[rr,"( F(w)^{-1} {,} w F_1 )"] && \mX_2 \dar["c"] \\
\mX_1 \rar[rr,"te"'] && \mX_1
\end{tikzcd}\]

\noi commute in $\cE$. In particular we have that 

\[ F(w)^{-1} t = w F_1 s \]

\noi and this along with functoriality of $F$ and the definition of $\spn = W \tensor[_{ws}]{\times}{_s} \mC_1$ is enough to see that the outside of the diagram, 

\begin{center}
\begin{tikzcd}[column sep = large]
\spn \dar["\pi_0"'] \rar["\pi_1"] \ar[dr, dotted, "\lbrack F \rbrack'"] & \mC_1 \ar[dr, bend left, "F_1"] & \\
W \ar[dr, bend right, "F(w)^{-1}"'] & \mX_2 \arrow[dr, phantom, "\usebox\pullback" , very near start, color=black]
 \rar["\pi_1"] \dar["\pi_0"'] & \mX_1 \dar["s"] \\
& \mX_1 \rar["t"'] & \mX_0
\end{tikzcd},
\end{center}

\noi commutes and induces the unique map $\lbrack F \rbrack' : \spn \to \mX_2$. This map is used to define $\lbrack F \rbrack_1 : \mCW_1 \to \mX_1$ in the following lemma by the universal property of the coequalizer $\mCW_1$. 

\begin{lem}\label{lem defining Phi_1}
The coequalizer diagram, 

\begin{center}
\begin{tikzcd}[]
\slb \rar[shift left, "p_0"] \rar[shift right, "p_1"'] & \spn \dar["\lbrack F \rbrack'"'] \rar[two heads, "q"] & \mCW_1 \dar[dotted, "\lbrack F \rbrack_1"] \\
& \mX_2 \rar["c"] & \mX_1 \\
\end{tikzcd}
\end{center}

\noi commutes in $\cE$ and uniquely determines the map $\lbrack F \rbrack_1 : \mCW_1 \to \mX_1$. 
\end{lem}
\begin{proof}

\noi The main idea here is that the left legs of the two spans inhabiting a sailboat, represented by $p_0 , p_1 : \slb \to \spn$, are arrows coming from $W$. These are part of a commuting triangle represented by $\pi_0 : \slb \to W_\triangle.$ More precisely, the left leg of the $p_1$ projection factors through the left leg of the $p_0$ projection by the arrow represented by the map $\pi_0 \pi_0 \pi_0 : \slb \to \mC_1$ in $\cE$. This is shown in the following calculation: 

\begin{align*}
  p_1 \pi_0 w
  &= \pi_0 \pi_1 w\\
  &= \pi_0 ( \pi_0 \pi_0 , \pi_0 \pi_1 w) c \\
  &=( \pi_0 \pi_0 \pi_0 , \pi_0 \pi_0 \pi_1 w) c \\
  &= ( \pi_0 \pi_0 \pi_0 , p_0 \pi_0 w) c
\end{align*}

\noi Functoriality of $F$ then gives 

\begin{align*}
  p_1 \pi_0 w F_1 
  = ( \pi_0 \pi_0 \pi_0 F_1 , p_0 \pi_0 w F_1 ) c.
\end{align*}

\noi The internal functor $F$ inverts the arrows coming from $w : W \to \mC_1$ so we can internally post-compose with $ p_0 \pi_0 F(w)^{-1} : \slb \to \mX_1$ to give the following calculation. This calculation uses associativity and the identity laws for internal composition in $\mX$, along with the definitions of $F(w)^{-1}$ and $\slb$ and functoriality of $F$. 

\begin{align*}
(p_1 \pi_0 w F_1 , p_0 \pi_0 F(w)^{-1} ) c 
&= (\pi_0 \pi_0 \pi_0 F_1 
,p_0 \pi_0 w F_1
,p_0 \pi_0 F(w)^{-1} ) c\\
&= (\pi_0 \pi_0 \pi_0 F_1 
,p_0 \pi_0 
(w F_1 , F(w)^{-1} ) c )c \\
&= (\pi_0 \pi_0 \pi_0 F_1 
,p_0 \pi_0 w F_1 s e)c \\
&= (\pi_0 \pi_0 \pi_0 F_1 
, \pi_0 \pi_0 \pi_1 w F_1 s e)c \\
&= (\pi_0 \pi_0 \pi_0 F_1 
, \pi_0 \pi_0 \pi_1 w s F_0 e)c \\
&= (\pi_0 \pi_0 \pi_0 F_1 
, \pi_0 \pi_0 \pi_0 t F_0 e)c \\
&= (\pi_0 \pi_0 \pi_0 F_1 
, \pi_0 \pi_0 \pi_0 F_1 t e)c \\
&= \pi_0 \pi_0 \pi_0 F_1 ( 1 
, t e)c \\
&= \pi_0 \pi_0 \pi_0 F_1 
\end{align*}

\noi A similar internal composition involving the first and last terms in the equation above with $p_1 \pi_0 F(w)^{-1} : \slb \to \mX_1$ gives

\begin{align*}
 ( p_1 \pi_0 F(w)^{-1} , \pi_0 \pi_0 \pi_0 F_1 ) c 
  &= (p_1 \pi_0 F(w)^{-1} , p_1 \pi_0 w F_1 , p_0 \pi_0 F(w)^{-1} ) c \\
  &= (p_1 \pi_0 (F(w)^{-1} ,w F_1)c , p_0 \pi_0 F(w)^{-1} ) c \\
  &= (p_1 \pi_0 w F_1 t e , p_0 \pi_0 F(w)^{-1} ) c \\
  &= (p_1 \pi_0 w t F_0 e , p_0 \pi_0 F(w)^{-1} ) c \\
  &= (p_0 \pi_0 w t F_0 e , p_0 \pi_0 F(w)^{-1} ) c \\
  &= (p_0 \pi_0 w F_1 t e , p_0 \pi_0 F(w)^{-1} ) c \\
  &= (p_0 \pi_0 F(w)^{-1} s e , p_0 \pi_0 F(w)^{-1} ) c \\
  &= p_0 \pi_0 F(w)^{-1} (s e , 1 ) c \\
  &= p_0 \pi_0 F(w)^{-1} 
\end{align*}

\noi Now we can substitute the last equation into the following calculation to see $\lbrack F \rbrack' c$ coequalizes the pair $p_0$ and $p_1$:

\begin{align*}
  p_0 \lbrack F \rbrack' c 
  &= p_0 (\pi_0 F(w)^{-1} , \pi_1 F_1)c \\
  &= (p_0 \pi_0 F(w)^{-1} , p_0 \pi_1 F_1)c \\
  &= \big( 
  ( p_1 \pi_0 F(w)^{-1} , \pi_0 \pi_0 \pi_0 F_1 ) c , 
  p_0 \pi_1 F_1 \big) c\\
  &= ( p_1 \pi_0 F(w)^{-1} , 
  \pi_0 \pi_0 \pi_0 F_1 , 
  p_0 \pi_1 F_1 ) c\\
  &= ( p_1 \pi_0 F(w)^{-1} , 
  (\pi_0 \pi_0 \pi_0 F_1 , 
  p_0 \pi_1 F_1) c ) c\\
  &= ( p_1 \pi_0 F(w)^{-1} , 
  (\pi_0 \pi_0 \pi_0, 
  p_0 \pi_1 ) c F_1 ) c\\
  &= ( p_1 \pi_0 F(w)^{-1} , 
  p_1 \pi_1 F_1 ) c\\
  &= p_1 (\pi_0 F(w)^{-1} , 
  \pi_1 F_1 ) c\\
  &= p_1 \lbrack F \rbrack' c\\
\end{align*}

\noi The existence and uniqueness of the map $\lbrack F \rbrack_1 : \mCW_1 \to \mX_1$ such that $q \lbrack F \rbrack_1 = \lbrack F \rbrack' c$ follows from the universal property of $\mCW_1$. 
\end{proof}

\noi The next step is to show that $\lbrack F \rbrack = (\lbrack F \rbrack_0 , \lbrack F \rbrack_1)$ is an internal functor. First we show identities are preserved by proving the following lemma. 

\begin{lem}\label{lem Phi preserves id}
The diagram 

\[ \begin{tikzcd}[]
\mC_0 \rar["F_0"] \dar["(\alpha {,} \alpha w)q"'] & \mX_0 \dar["e"] \\
\mCW_1 \rar["\lbrack F \rbrack_1"'] & \mX_1
\end{tikzcd}\]

\noi commutes in $\cE$. 
\end{lem}
\begin{proof}
By the universal property of the pullback $\mX_2$, the definition of $F(w)^{-1}$, functoriality of $F$, and the fact that $\alpha$ is a section of $wt$ we have

\begin{align*}
(\alpha , \alpha w) q \lbrack F \rbrack_1 
&= (\alpha , \alpha w) \lbrack F \rbrack' c \\
&= (\alpha , \alpha w) (\pi_0 F(w)^{-1} , \pi_1 F_1) c \\
&= (\alpha F(w)^{-1} , \alpha w F_1) c \\
&= \alpha ( F(w)^{-1} , w F_1) c \\
&= \alpha w F_1 t e \\
&= \alpha w t F_0 e \\
&= F_0 e 
\end{align*}
\end{proof}

\noi The following lemma shows that $\lbrack F \rbrack$ preserves (internal) composition. When $\cE = \Set$ this is saying that for any pair of composable spans 
\[ \begin{tikzcd}[]
a & b \lar["\circ" marking, "v"] \rar["f"] & c & d \lar["\circ" marking, "v'"] \rar["f'"]& e
\end{tikzcd}\]
\noi with composite span, 
\[ \begin{tikzcd}[]
a & b' \lar["\circ" marking, "v''"] \rar["f''"] & e
\end{tikzcd},\]
\noi the diagram 

\[ \begin{tikzcd}[]
F(a) \rar["F(v)^{-1} "] \dar[dd, "f(v'')^{-1} "'] & F(b) \rar["F(f)"] & F(c) \dar["f(v')^{-1} "] \\
& & F(d) \dar["F(f')"] \\
F(b') \rar[rr, "F(f'')"'] && F(e) 
\end{tikzcd}\]

\noi commutes in $\mX$. To see this we look at the composition data 

\[\begin{tikzcd}
	&& e \\
	& {} & d \\
	a & b & c & {b'} & {a'}
	\arrow["v", from=3-2, to=3-1,"\circ" marking]
	\arrow["f"', from=3-2, to=3-3]
	\arrow["{v_0}"', from=2-3, to=3-2,"\circ" marking]
	\arrow["h", from=1-3, to=2-3]
	\arrow["{v'}", from=3-4, to=3-3,"\circ" marking]
	\arrow["{f'}"', from=3-4, to=3-5]
	\arrow["k", from=2-3, to=3-4]
	\arrow["{v''}"', curve={height=6pt}, from=1-3, to=3-1,"\circ" marking]
	\arrow["{f''}",curve={height=-6pt}, from=1-3, to=3-5]
\end{tikzcd}\]

\noi and apply the functor $F$ to the weak-composition triangle on the left to get the equation 

\[ F(v'') = F(h) F(v_0) F(v).\]

\noi Since $F$ inverts $W$, we can pre-compose both sides by $F(v'')^{-1}$ and post-compose them both by $F(v)^{-1} $ to get the equation 
\[ F(v)^{-1} = F(v'')^{-1} F(h) F(v_0). \] 
\noi Similarly, applying $F$ to the Ore-square gives the equation 
\[ F(v_0) F(f) = F(k) F(v')\]
\noi and post-composing with $F(v')^{-1}$ gives 
\[ F(v_0) F(f) F(v')^{-1} = F(k). \]
\noi Put it all together with functoriality of $F$ to see the square commutes. 

\begin{align*}
 F(v)^{-1} F(f) F(v')^{-1} F(f') 
 &= F(v'')^{-1} F(h) F(v_0) F(f) F(v')^{-1} F(f') \\
 &=F(v'')^{-1} F(h) F(k) F(f')\\
 &=F(v'')^{-1} F(f'')  .
 \end{align*}

\begin{lem}\label{lem Phi preserves composition}
The diagram,

\begin{center}
\begin{tikzcd}[column sep = large]
\mCW_2 \dar["c"'] \rar["\lbrack F \rbrack_1 \times \lbrack F \rbrack_1"] & \mX_2 \dar["c"] \\
\mCW_1 \rar["\lbrack F \rbrack_1"'] & \mX_1
\end{tikzcd}
\end{center}
\noi where $\lbrack F \rbrack_1 \times \lbrack F \rbrack_1 = (\pi_0 \lbrack F \rbrack_1 , \pi_1 \lbrack F \rbrack_1)$ is the unique pairing map, commutes in $\mX$. 
\end{lem}
\begin{proof}
Recall that $u : U \nrightarrow \spn \tensor[_t]{\times}{_s} \spn$ is the cover on which we defined composition, with $u = u_0 u_1$ in the diagram constructed by the Internal Fractions Axioms: 

\begin{center}
\begin{tikzcd}[column sep = huge ]
 W_\circ \rar[rr, "(\pi_0 \pi_1 {,} \pi_0 \pi_2)"] 
&& W \times_{\mC_0} W 
& \\
U 
\dar["\sigma_\circ"]
\uar["\omega"]
\rar[rr, "/" marking , "u_0" near start]
 &
 & U_0 
 \dar["\theta"']
 \uar["(\theta \pi_0 \pi_0 {, } u_1 \pi_0 \pi_0) "'] 
 \rar["/" marking , "u_1" near start] 
 & \spn \tensor[_t]{\times}{_s} \spn
 \dar["( \pi_0 \pi_1 {, } \pi_1 \pi_0)"] 
 \\
 \spn 
 &
 & W_\square 
 \rar["(\pi_0 \pi_1 {,} \pi_1 \pi_1)"'] 
 & \csp
\end{tikzcd}
\end{center}\

\noi We use this cover when we need to show certain maps out of $\spn \tensor[_t]{\times}{_s} \spn$ are equal. More precisely, by showing it (or its composition with other epimorphisms) equalizes two maps we are interested in proving are equal. We begin this proof with the weak-composition triangle, $\omega : U \to W_\circ$, and the equation encoding it is commutativity. 
 
 \[ \omega_0 \pi_1 = (\omega_0 \pi_0 \pi_0 , u_0 \theta \pi_0 \pi_0 w , u \pi_0 \pi_0 w ) c \]

\noi By functoriality of $F$ we have 

\[ \omega \pi_1 w F_1 = (\omega_0 \pi_0 \pi_0 F_1 , u_0 \theta \pi_0 \pi_0 w F_1 , u \pi_0 \pi_0 w F_1 ) c .\]

\noi We can internally pre-compose both sides of $\mX$ with the map $\omega \pi_1 F(w)^{-1} : U \to \mX_1$ and internally post-compose with $u \pi_0 \pi_0 F(w)^{-1} : U \to \mX_1$. Before writing down the new equation however we do the following intermediate calculation: 

\begin{align*}
(\omega \pi_1 F(w)^{-1} , \omega \pi_1w F_1)c 
&= \omega \pi_1 (F(w)^{-1} , w F_1)c\\
&= \omega \pi_1 w F_1 t e \\
&= \omega \pi_1 w t e F_1 \\
&= u \pi_0 \pi_0 w t e F_1 \\
&= u \pi_0 \pi_0 w F_1 t e\\
&= u \pi_0 \pi_0 F(w)^{-1} s e 
\end{align*}

\noi Using this along with the definitions of $W_\square$, $\theta$, and $\csp$ gives 

\begin{align*}
( u \pi_0 \pi_0 w F_1 , u \pi_0 \pi_0 F(w)^{-1} ) c 
&= u \pi_0 \pi_0 ( w F_1, F(w)^{-1})c \\
&= u \pi_0 \pi_0 w F_1 s e \\
&= u \pi_0 \pi_0 w s e F_1 \\
&= u_0 u_1 \pi_0 \pi_1 w t e F_1 \\
&= u_0 \theta \pi_0 \pi_0 w t e F_1\\
&= u_0 \theta \pi_0 \pi_0 w F_1 t e. 
\end{align*}

\noi Now the pre/post-composed equation is

\begin{align*}
  & \ \ \ \ ( u \pi_0 \pi_0 F(w)^{-1} s e , u \pi_0 \pi_0 F(w)^{-1})c\\
  & = = (\omega \pi_1 F(w)^{-1} , \omega_0 \pi_0 \pi_0 F_1 , u_0 \theta \pi_0 \pi_0 w F_1, u_0 \theta \pi_0 \pi_0 w F_1 t e )c 
\end{align*} 

\noi where the left side simplifies via the identity law in $\mX$ as 

\[ ( u \pi_0 \pi_0 F(w)^{-1} s e , u \pi_0 \pi_0 F(w)^{-1})c = u \pi_0 \pi_0 F(w)^{-1} (s e , 1 )c = u \pi_0 \pi_0 F(w)^{-1} \]

\noi and the right side simplifies similarly as 

\begin{align*}
& \ \ \ \ (\omega \pi_1 F(w)^{-1} , \omega_0 \pi_0 \pi_0 F_1 , u_0 \theta \pi_0 \pi_0 w F_1, u_0 \theta \pi_0 \pi_0 w F_1 t e)c \\
&= \big(\omega \pi_1 F(w)^{-1} , \omega_0 \pi_0 \pi_0 F_1 , 
u_0 \theta \pi_0 \pi_0 w F_1(1, t e)c \big) c \\
&= \big(\omega \pi_1 F(w)^{-1} , \omega_0 \pi_0 \pi_0 F_1 , 
u_0 \theta \pi_0 \pi_0 w F_1 \big) c
\end{align*}

\noi giving the simplified equation: 

\[ \label{eq F(wc triangle) conj'd simp'd} 
u \pi_0 \pi_0 F(w)^{-1} = \big(\omega \pi_1 F(w)^{-1} , \omega_0 \pi_0 \pi_0 F_1 , 
u_0 \theta \pi_0 \pi_0 w F_1 \big) c \tag{$\star$}. \]

\noi Now take the Ore-square, $u_0 \theta : U \to W_\square$, and the equation describing commutativity, 
\[ (u_0 \theta \pi_0 \pi_0 w , u \pi_0 \pi_1) c = (u_0 \theta \pi_1 \pi_0 , u \pi_1 \pi_0 w)c , \]
\noi and apply $F$ to get the equation 
\[ (u_0 \theta \pi_0 \pi_0 wF_1 , u \pi_0 \pi_1 F_1) c = (u_0 \theta \pi_1 \pi_0 F_1 , u \pi_1 \pi_0 w F_1 )c .\]

\noi Since $F$ inverts $w: W \to \mC_1$ we can post-compose both sides with $u \pi_1 \pi_0 F(w)^{-1} : U \to \mX_1$ to get a new equation. The following computation showing how this is done follows more or less by the definitions of $F(w)^{-1}$ and $\theta$ and functoriality of $F$:

\begin{align*}
( u \pi_1 \pi_0 w F_1 , u \pi_1 \pi_0 F(w)^{-1}) c 
&= u \pi_1 \pi_0 (w F_1 , F(w)^{-1}) c \\
&= u \pi_1 \pi_0 w F_1 s e \\
&= u \pi_1 \pi_0 w s e F_1 \\
&= u_0 \theta \pi_1 \pi_0 t e F_1 \\
&= u_0 \theta \pi_1 \pi_0 F_1 t e 
\end{align*}
\noi Adding internal composition with $u_0 \theta \pi_1 \pi_0 F_1 : U \to \mX_1$ on the right of both side of the internal compositions shown in the last equation and applying the identity law in $\mX$ gives: 

\begin{align*}
(u_0 \theta \pi_1 \pi_0 F_1 , u \pi_1 \pi_0 w F_1 , u \pi_1 \pi_0 F(w)^{-1} )c 
&= (u_0 \theta \pi_1 \pi_0 F_1 , u_0 \theta \pi_1 \pi_0 F_1 t e ) c \\
&= u_0 \theta \pi_1 \pi_0 F_1 (1 , te) c\\
&= u_0 \theta \pi_1 \pi_0 F_1.
\end{align*}

\noi Then by the definition of $W_\square$ we get the equation 

\begin{align*}
& \ \ \ \ (u_0 \theta \pi_0 \pi_0 wF_1 , u \pi_0 \pi_1 F_1 , u \pi_1 \pi_0 F(w)^{-1}) c \\
&= ((u_0 \theta \pi_0 \pi_0 w F_1 , u \pi_0 \pi_1 F_1)c , u \pi_1 \pi_0 F(w)^{-1}) c \\
&= ((u_0 \theta \pi_1 \pi_0 F_1 , u \pi_1 \pi_0 w F_1)c , u \pi_1 \pi_0 F(w)^{-1}) c \\
&= (u_0 \theta \pi_0 \pi_0 wF_1 , u \pi_0 \pi_1 F_1 , u \pi_1 \pi_0 F(w)^{-1}) c \\
&=u_0 \theta \pi_1 \pi_0 F_1. \tag{$\star \star$}
\end{align*}  

\noi which simplifies to the equality

\[ \label{eq F(ore sqr) conj'd simp'd} (u_0 \theta \pi_0 \pi_0 wF_1 , u \pi_0 \pi_1 F_1 , u \pi_1 \pi_0 F(w)^{-1}) c
= u_0 \theta \pi_1 \pi_0 F_1. \tag{$\star \star$} \]

\noi which we use in following calculation that shows $[F]$ preserves composition. By equations (\ref{eq F(ore sqr) conj'd simp'd}) and (\ref{eq F(wc triangle) conj'd simp'd}) along with functoriality of $F$ and the identity law in $\mX$ we have: 

\begin{align*}
 & \ \ \ \ u (q \times q) c \lbrack F \rbrack_1 \\
 &= u c' \lbrack F \rbrack_1 \\
 &= \sigma_c q \lbrack F \rbrack_1 \\
 &= \sigma_c \lbrack F \rbrack ' \\
 &= (\sigma_c \pi_0 F(w)^{-1} , \sigma_c \pi_1 F_1) c  \\
 &= (\omega \pi_1 F(w)^{-1} , \omega \pi_0 \pi_0 F_1 , u_0 \theta \pi_1 \pi_0 F_1, u \pi_1 \pi_1 F_1 ) c \\ 
 &= (\omega \pi_1 F(w)^{-1} , \omega \pi_0 \pi_0 F_1 , u_0 \theta \pi_0 \pi_0 wF_1 , u \pi_0 \pi_1 F_1 , u \pi_1 \pi_0 F(w)^{-1} , u \pi_1 \pi_1 F_1 ) c \\ 
 &= ((\omega \pi_1 F(w)^{-1} , \omega \pi_0 \pi_0 F_1 , u_0 \theta \pi_0 \pi_0 wF_1 )c , u \pi_0 \pi_1 F_1 , u \pi_1 \pi_0 F(w)^{-1} , u \pi_1 \pi_1 F_1 ) c \\ 
 &= (u \pi_0 \pi_0 F(w)^{-1} , u \pi_0 \pi_1 F_1 , u \pi_1 \pi_0 F(w)^{-1} , u \pi_1 \pi_1 F_1 ) c \\ 
&= \big( u \pi_0 (\pi_0 F(w)^{-1} , \pi_1 F_1)c , u \pi_1 (\pi_0 F(w)^{-1} , \pi_1 F_1 )c \big) c \\ 
&= \big( u \pi_0 \lbrack F \rbrack' c , u \pi_1 \lbrack F \rbrack' c \big) c \\ 
&= u \big( \pi_0 q \lbrack F \rbrack_1 , \pi_1 q \lbrack F \rbrack_1 \big) c \\ 
&= u (q \times q) (\lbrack F \rbrack_1 \times \lbrack F \rbrack_1) c .
\end{align*} 

\noi The composite $u (q \times q)$ is epic, so we get 
\[ c \lbrack F \rbrack_1 = (\lbrack F \rbrack_1 \times \lbrack F \rbrack_1) c \] 
\noi as desired. 
\end{proof}

\begin{prop}\label{prop IntFrc 1-cell corresp. well-defined in one direction}
The maps $\lbrack F \rbrack_0 = F_0 : \mC_0 \to \mX_0$ and $\lbrack F \rbrack_1 : \mCW_1 \to \mX_1$ determine an internal functor $\lbrack F \rbrack : \mCW \to \mX$ such that the diagram 
\[ \begin{tikzcd}[]
\mC \rar[rr, "F"] \ar[dr, "L"'] & & \mX\\
& \mCW \ar[ur, "\lbrack F \rbrack"']
\end{tikzcd}\]
\noi commutes in $\cE$. 
\end{prop}
\begin{proof}
Functoriality follows from Lemma~\ref{lem Phi preserves id} and Lemma~\ref{lem Phi preserves composition}. To see the diagram commutes we can immediately see 

\[ L_0 \lbrack F \rbrack_0 = 1_{\mC_0} F_0 = F_0 \]

\noi and then use the definitions of $L_1$, $\lbrack F \rbrack_1$, and $\lbrack F \rbrack'$, along with functoriality of $F$, the identity law, $(se , 1) c =1$, in $\mX$, and the fact that $\alpha$ is a section of $wt$ to compute

\begin{align*}
L_1 \lbrack F \rbrack_1 
&= (s \alpha , (s \alpha w , 1) c) q \lbrack F \rbrack_1 \\
&= (s \alpha , (s \alpha w , 1) c) \lbrack F \rbrack' c\\
&= (s \alpha , (s \alpha w , 1) c) (\pi_0 F(w)^{-1} , \pi_1 F_1) c\\
&=(s \alpha F(w)^{-1} , (s \alpha w , 1) c F_1) c\\
&=(s \alpha F(w)^{-1}, s \alpha wF_1 , F_1) c\\
&= (s \alpha (F(w)^{-1} , w F_1) c , F_1)c \\
&= (s \alpha w F_1 t e , F_1) c \\
&= (s \alpha w t e F_1 , F_1) c \\
&= (s e F_1 , F_1) c \\
&= (F_1 se , F_1) c\\
&= F_1 (se , 1) c\\
&= F_1. 
\end{align*} 
\end{proof}

\begin{prop}\label{prop UPIntFrac 1-cell corresondence}
Every internal functor $F : \mC \to \mX$ that inverts $w : W \to \mC_1$ corresponds uniquely to an internal functor $ [F] : \mCW \to \mX$.
\end{prop}
\begin{proof}
Lemma~\ref{prop IntFrc 1-cell corresp. well-defined in one direction} implies the forward direction. Now notice that for any internal functor $G : \mCW \to \mX$, there is an internal functor $LG : \mC \to \mX$ given by pre-composing with the localization functor $L : \mC \to \mCW$. In Proposition~\ref{prop int loc functor inverts W} we saw that $L$ inverts $w : W \to \mC_1$, with $(wL)^{-1} = ( 1 , w s e) q : W \to \mCW_1$. Functoriality of $G$ implies $(wL)^{-1} G_1 : W \to \mX_1$ is an inverse of $wLG : W \to \mX_1$ in $\mX$ so this establishes the other direction of the correspondence. 

For any $F : \mC \to \mX$ inverting $w : W \to \mC_1$, let $[F] : \mCW \to \mX$ be the corresponding internal functor. Pre-composing with $L : \mC \to \mX$ gives 

\[ L [F] = F\]

\noi so composing the assignments in one direction gives an identity. On the other hand, for any $G : \mCW \to \mX$, if we pre-compose with $L : \mC \to \mCW$ and then find the corresponding internal functor $\mCW \to \mX$ we get that 

\[ [LG] : \mCW \to \mX \]

\noi where $q [LG] = [LG]' c : \spn \to \mX_1$. Now expanding this with the explicit definition of 

\[ [LG]' = (\pi_0 (LG)(w)^{-1} , \pi_1 (LG)_1 ) \]

\noi we can use the definition 

\[ (LG)(w)^{-1} = (wL)^{-1} G_1 = (1, wse)q G_1 ,\]

\noi functoriality of $L$ and $G$, the definition 

\[ L_1 = (s\alpha, (s \alpha w, 1)c )q ,\]

\noi a bit of factoring with pairing maps, the fact that $q_2 c = c' : \spn \tensor[_t]{\times}{_s} \spn \to \spn$, and Lemma~\ref{lem factorization of equiv classes of spans} in the last line to see: 

\begin{align*}
[LG]'c 
&= (\pi_0 (LG)(w)^{-1} , \pi_1 (LG)_1 )c \\ 
&= (\pi_0 (wL)^{-1} G_1 , \pi_1 L_1 G_1 )c \\
&=(\pi_0 (wL)^{-1} , \pi_1 L_1 )c G_1 \\
&=(\pi_0 (1 , w s e)q , \pi_1 (s\alpha, (s \alpha w, 1)c )q ) c G_1 \\
&=(( \pi_0 , \pi_0 w s e ) q , (\pi_1 s\alpha, (\pi_1 s \alpha w, \pi_1)c )q ) c G_1 \\
&=(( \pi_0 , \pi_0 w s e ) , (\pi_0 s\alpha, (\pi_0 s \alpha w, \pi_1)c ) q_2 c G_1 \\
&=(( \pi_0 , \pi_0 w s e ) , (\pi_0 s\alpha, (\pi_0 s \alpha w, \pi_1)c ) c' G_1 \\
&= q G_1 
\end{align*}

\noi This implies that $[LG] = G_1$ by the universal property of the coequalizer $\mCW_1$. We have shown that the assignments in either direction are inverses to one another, so this correspondence is unique. 
\end{proof}

\subsubsection{Correspondence Between 2-Cells}

Next we show the 2-cell correspondence between internal natural transformations for the internal functors in the 1-cell correspondence of Proposition~\ref{prop UPIntFrac 1-cell corresondence}. 

In this subsection we see that internal natural transformations, $\alpha : F \implies G$, between internal functors, $F,G: \mC \to \mX$, that invert $w : W \to \mC_1$ correspond uniquely to natural transformations, $[\alpha] : [F] \implies [G]$, between the uniquely corresponding internal functors $[F] , [G] : \mCW \to \mX$ from Section \ref{SS UPIntFrac 1-cells}. The main result of Section \ref{S UP of Internal Fractions} is the isomorphism of categories established in Theorem~\ref{thm 2-dimensional universal property of localization}. 

We begin with a lemma that shows one direction of the correspondence between the aforementioned natural transformations,

\begin{lem}\label{lem induced 2-cells `foreward direction'}
Every internal natural transformation,

\[ \begin{tikzcd}[column sep = huge, row sep = huge, every label/.append style={font=\scriptsize}]
\mC 
\arrow[r, bend left=50, "F"{name=f}] 
\arrow[r, bend right=50,"G"'{name=g}] 
& \mX 
 \arrow[Rightarrow, shorten <= 10pt, shorten >=10pt, from=f, to=g, "\alpha"]
\end{tikzcd},\]

\noi induces a canonical natural transformation:

\[ \begin{tikzcd}[column sep = huge, row sep = huge, every label/.append style={font=\scriptsize}]
\mCW
\arrow[r, bend left=50, "\lbrack F \rbrack"{name=f}] 
\arrow[r, bend right=50,"\lbrack G \rbrack"'{name=g}] 
& \mX 
 \arrow[Rightarrow, shorten <= 10pt, shorten >=10pt, from=f, to=g, "\lbrack \alpha \rbrack"]
\end{tikzcd}\]
\end{lem}
\begin{proof}
Since $\mCW_0 = \mC_0$, define the components of $[\alpha]$ to be the components of $\alpha$: 

\[ \begin{tikzcd}[]
\mCW_0 \rar["\lbrack \alpha \rbrack"] & \mX_1 \\
\mC_0 \uar[equals, "L_0"] \ar[ur, "\alpha"']
\end{tikzcd}\]

\noi To see this is well-defined we need to show the (naturality) square 

\[ \begin{tikzcd}[column sep = large, row sep = large]
\mCW_1 \rar["(s [\alpha \rbrack {,} [ g \rbrack_1)"] 
\dar["( [ f \rbrack_1 {,} t[ \alpha \rbrack)"'] 
& \mX_2 \dar["c"] \\
\mX_2 \rar["c"'] & \mX_1 
\end{tikzcd}\]

\noi commutes in $\cE$. Let $F(w)^{-1}, G(w)^{-1} : W \to \mX_1$ denote the inverses of $w F_1 , w G_1: W \to \mX_1$. Naturality of $\alpha : F \implies G$ implies the diagram

\[ \begin{tikzcd}[column sep = huge, row sep = large]
W \rar["(w s \alpha {,} w G_1)"] \dar["(w F_1 {,} w t \alpha)"'] & \mX_2 \dar["c"] \\
\mX_2 \rar["c"'] & \mX_1
\end{tikzcd} \]

\noi commutes in $\cE$. Using internal composition in $\mX$ to compose with $F(w)^{-1} : W \to \mX_1$ on the left and $G(w)^{-1} : W \to \mX_1$ on the right on both sides gives a new commuting diagram, 

\[ \begin{tikzcd}[column sep = huge, row sep = large]
W \rar["(F(W)^{-1} {,}w s \alpha )"] \dar["( w t \alpha {,} g(W)^{-1})"'] & \mX_2 \dar["c"] \\
\mX_2 \rar["c"'] & \mX_1
\end{tikzcd}, \]

\noi by cancellation using the identity law in $\mX$. It will also be helpful to recall the following commuting diagrams from the definition of $\mCW$ and its universal property.

\[
\begin{tikzcd}[column sep = large]
\spn \rar[two heads, "q"] \dar["(\pi_0 F(w)^{-1} {,} \pi_1 F_1 )"'] & \mCW_1 \dar[dotted, "[F \rbrack_1"] \\
\mX_2 \rar["c"'] & \mX_1
\end{tikzcd}\qquad \qquad \]
\[
\begin{tikzcd}[column sep = large]
\spn \rar[two heads, "q"] \dar["\pi_0"'] & \mCW_1 \dar[dotted, "s"] \\
W \rar["wt"'] & \mC_0
\end{tikzcd}\qquad \qquad 
\begin{tikzcd}[column sep = large]
\spn \rar[two heads, "q"] \dar["\pi_1"'] & \mCW_1 \dar[dotted, "t"] \\
\mC_1 \rar["t"'] & \mC_0
\end{tikzcd}\]

\noi Now consider the following diagram: 

\[ \begin{tikzcd}[column sep = huge, row sep = huge]
\mCW_1 
\dar["( [F \rbrack_1 {,} t \lbrack \alpha \rbrack )"']
& \spn 
\lar[two heads, "q"'] \rar[two heads, "q"] 
\ar[dr, "(q s [\alpha\rbrack {,} q [G \rbrack_1 )"' near end] 
\ar[dl, "( q [F \rbrack_1 {,}q t [\alpha\rbrack )" near end]
& \mCW_1 \dar["(s [\alpha\rbrack {,} [G \rbrack_1 )"] \\
\mX_2 \ar[dr, "c"'] & & \mX_2 \ar[dl, "c"] \\
& \mX_1 &
\end{tikzcd}\]

\noi The inside commutes by the following calculation which uses associativity of composition along with naturality of $\alpha$ and the definitions of $[F]_1$, $[G]_1$, and the pullback $\spn = W \tensor[_{ws} ]{\times}{_s} \mC_1$:

\begin{align*}
(q [F]_1 , qt\alpha) c 
&=\big( ( \pi_0 F(w)^{-1} , \pi_1 f_1 ) c , q t \alpha \big)c & \text{Def. }[F]_1 \\
&= \big( \pi_0 F(w)^{-1} , \pi_1 f_1 , \pi_1 t \alpha \big)c & \text{Assoc.} \\
&= \big( \pi_0 F(w)^{-1} , \pi_1 (F_1 , t \alpha)c \big)c & \text{Assoc.} \\
&= \big( \pi_0 F(w)^{-1} , \pi_1 (s \alpha , G_1)c \big)c & \text{Nat. }\alpha \\
&= \big( \pi_0 F(w)^{-1} , \pi_1 s \alpha , \pi_1 G_1 \big)c & \text{Assoc.} \\
&= \big( \pi_0 F(w)^{-1} , \pi_0 w s \alpha , \pi_1 G_1 \big)c & \text{Def. }\spn \\
&= \big( \pi_0 (F(w)^{-1} , w s \alpha)c , \pi_1 G_1 \big)c & \text{Assoc.} \\
&= \big( \pi_0 (w t \alpha , G(w)^{-1} )c , \pi_1 G_1 \big)c & \text{Nat. }\alpha \\
&= \big( \pi_0 w t \alpha , \pi_0 G(w)^{-1} , \pi_1 G_1 \big)c & \text{Assoc.} \\
&= \big( q s \alpha , (\pi_0 G(w)^{-1} , \pi_1 G_1) c \big)c & \text{Assoc. }\\
&= \big( q s \alpha , q [G]_1 \big)c & \text{Def. }[g]_1 .
\end{align*}

\noi Since $q$ is an epi we can conclude that$[\alpha]$ satisfies the appropriate naturality condition: 

\[( [F \rbrack_1 {,} t \lbrack \alpha \rbrack ) c = ([\alpha\rbrack {,} [G \rbrack_1 ) c .\]

\noi It's source and target can be computed component-wise by the following commuting diagrams
\[ \begin{tikzcd}[column sep = large]
\mCW_0 \dar[equals] \rar["[\alpha\rbrack"] & \mX_1 \dar["s"] \\
\mC_0 \rar["F_0"'] \ar[ur, "\alpha"'] \dar[equals] & \mX_0\\
\mCW_0 \ar[ur, "[F\rbrack_0"'] & 
\end{tikzcd} \qquad \qquad \qquad \begin{tikzcd}[column sep = large]
\mCW_0 \dar[equals] \rar["[\alpha\rbrack"] & \mX_1 \dar["t"] \\
\mC_0 \rar["G_0"'] \ar[ur, "\alpha"'] \dar[equals] & \mX_0\\
\mCW_0 \ar[ur, "[G\rbrack_0"'] & 
\end{tikzcd}, \]

\noi It follows that $[\alpha] : [F] \implies [G]$ is an internal natural transformation. 
\end{proof}

\noi We continue with another lemma establishing the other direction of the correspondence between natural transformations. 

\begin{lem}\label{lem induced 2-cells `backward direction'}
Every internal natural tranformation,

\[ \begin{tikzcd}[column sep = huge, row sep = huge, every label/.append style={font=\scriptsize}]
\mCW
\arrow[r, bend left=50, "H"{name=f}] 
\arrow[r, bend right=50,"K"'{name=g}] 
& \mX 
 \arrow[Rightarrow, shorten <= 10pt, shorten >=10pt, from=f, to=g, " \beta "]
\end{tikzcd},\]

\noi induces a canonical natural transformation:

\[ \begin{tikzcd}[column sep = huge, row sep = huge, every label/.append style={font=\scriptsize}]
\mC 
\arrow[r, bend left=50, "LH"{name=f}] 
\arrow[r, bend right=50,"LK"'{name=g}] 
& \mX 
 \arrow[Rightarrow, shorten <= 10pt, shorten >=10pt, from=f, to=g, "L\beta"]
\end{tikzcd}\]
\end{lem}
\begin{proof}

The notation is suggestive of the fact that this these are given by composing with the internal functors $H, K : \mCW \to \mX$ and whiskering the internal natural transformation $\beta : H \implies K$ with the internal functor $L : \mC \to \mCW$. Note that the components of the whiskered transformation coincide with those of $\beta$ because $L$ is the identity on objects: 

\[ \begin{tikzcd}[]
\mC_0 \rar["L \beta"] & \mX_1 \\
\mCW_0 \uar[equals , "L_0"] \ar[ur, "\beta"'] 
\end{tikzcd}. \]
\end{proof}

\noi Lemmas \ref{lem induced 2-cells `foreward direction'} and \ref{lem induced 2-cells `backward direction'} show us the two directions of the correspondence we need to prove. Now we show the assignments described in the proofs of these lemmas are inverses to get the correspondence we need in the following lemma.

\begin{lem}\label{lem int.loc 2-cell correspondence} 
Let $F, G : \mC \to \mX$ be internal functors that invert $w : W \to \mC_1$ in $\mX$. Then the internal natural transformations 

\[ \begin{tikzcd}[column sep = huge, row sep = huge, every label/.append style={font=\scriptsize}]
\mC 
\arrow[r, bend left=50, "F"{name=f}] 
\arrow[r, bend right=50,"G"'{name=g}] 
& \mX 
 \arrow[Rightarrow, shorten <= 10pt, shorten >=10pt, from=f, to=g, "\alpha"]
\end{tikzcd}\] 

\noi (bijectively) correspond to the internal natural transformations

\[ \begin{tikzcd}[column sep = huge, row sep = huge, every label/.append style={font=\scriptsize}]
\mCW
\arrow[r, bend left=50, "[F\rbrack"{name=f}] 
\arrow[r, bend right=50,"[G\rbrack"'{name=g}] 
& \mX 
 \arrow[Rightarrow, shorten <= 10pt, shorten >=10pt, from=f, to=g, "\lbrack \alpha \rbrack"]
\end{tikzcd}. \]
\end{lem}
\begin{proof}
Let $\alpha : F \implies G$ be an internal natural transformation between internal functors $\mC \to \mX$ that invert $w : W \to \mC_1$ in $\mX$. We will show whiskering the internal natural transformation $[\alpha] : [F] \implies [G]$ with $L : \mC \to \mCW$ recovers $\alpha : F \implies G$: 

\[ \begin{tikzcd}[column sep = huge, row sep = huge]
\mC 
\arrow[r, bend left=50, "L[F\rbrack"{name=f}] 
\arrow[r, bend right=50,"L[G\rbrack"'{name=g}] 
& \mX 
 \arrow[Rightarrow, shorten <= 10pt, shorten >=10pt, from=f, to=g, "L[\alpha\rbrack"]
\end{tikzcd}\quad = \quad 
\begin{tikzcd}[column sep = huge, row sep = huge, every label/.append style={font=\scriptsize}]
\mC 
\arrow[r, bend left=50, "F"{name=f}] 
\arrow[r, bend right=50,"G"'{name=g}] 
& \mX 
 \arrow[Rightarrow, shorten <= 10pt, shorten >=10pt, from=f, to=g, "\alpha"]
\end{tikzcd}\]

\noi By definition of $[F], [G] : \mCW \to \mX$ we have 

\[ L [F] = F \qquad \qquad L[G] = G \]

\noi and the following commuting diagram in $\cE$ shows the components of $L[\alpha] : F \implies G$ are precisely those of $\alpha$: 

\[ \begin{tikzcd}[column sep = large, row sep = large]
\mC_0 \rar["L [\alpha\rbrack"] & \mX_1\\
\mCW_0 \uar[equals, "L_0"] \rar[equals] \ar[ur, "[\alpha\rbrack"'] & \mC_0 \uar["\alpha"'] 
\end{tikzcd}\]

\noi On the other hand, for any natural transformation 

\[ \begin{tikzcd}[column sep = huge, row sep = huge, every label/.append style={font=\scriptsize}]
\mCW
\arrow[r, bend left=50, "H"{name=f}] 
\arrow[r, bend right=50,"K"'{name=g}] 
& \mX 
 \arrow[Rightarrow, shorten <= 10pt, shorten >=10pt, from=f, to=g, " \beta "]
\end{tikzcd}\]

\noi we can see that 

\[ \begin{tikzcd}[column sep = huge, row sep = huge, every label/.append style={font=\scriptsize}]
\mCW
\arrow[r, bend left=50, "[Lh\rbrack"{name=f}] 
\arrow[r, bend right=50,"[LK\rbrack"'{name=g}] 
& \mX 
 \arrow[Rightarrow, shorten <= 10pt, shorten >=10pt, from=f, to=g, " [L\beta \rbrack"]
\end{tikzcd} \quad = \quad 
\begin{tikzcd}[column sep = huge, row sep = huge, every label/.append style={font=\scriptsize}]
\mCW
\arrow[r, bend left=50, "H"{name=f}] 
\arrow[r, bend right=50,"K"'{name=g}] 
& \mX 
 \arrow[Rightarrow, shorten <= 10pt, shorten >=10pt, from=f, to=g, " \beta "]
\end{tikzcd}\]

\noi by first noticing that the triangles, 

\[ \begin{tikzcd}[column sep = huge, row sep = huge]
\mX & \mC \rar["L [LK\rbrack "] \lar["L [LH\rbrack "'] \dar["L"] & \mX\\
\mC \uar["LH"] \rar["L"'] 
& \mCW 
\ar[ur, dotted, "[LK\rbrack"] 
\ar[ul, dotted,"[LH \rbrack"'] 
& \mC \uar["LK"'] \lar["L"] 
\end{tikzcd},\]
\noi commute in $\cE$ and imply that $[LH\rbrack = H $ and $ [LK\rbrack = K$ by the 1-cell universal property of the internal localization in Proposition~\ref{prop UPIntFrac 1-cell corresondence}. The following commuting diagram shows that the components for the natural transformations agree too: 

\[\begin{tikzcd}[column sep = large, row sep = large]
\mCW_0 \rar["[L\beta\rbrack"] & \mX_1 \\
\uar[equals] \mC_0 \rar[equals, "L_0"'] \ar[ur, "L\beta"] & \mCW_0 \uar["\beta"'] 
\end{tikzcd}\]
\end{proof}

\noi The only other piece we need to prove Theorem~\ref{thm 2-dimensional universal property of localization} is that the correspondence between 2-cells in Lemma~\ref{lem int.loc 2-cell correspondence} is functorial. It suffices to show functoriality in one direction. We show it in the direction of Lemma~\ref{lem induced 2-cells `foreward direction'} and leave the other direction (involving whiskering) as an exercise to the reader who wants to take a break. 

\begin{lem}\label{lem 2-cell correspondence functoriality}
The assignment of natural transformations, $\alpha \mapsto [\alpha]$, in Lemma~\ref{lem induced 2-cells `foreward direction'} is functorial. 
\end{lem}
\begin{proof}
For any internal functor $f : \mC \to \mX$, we have 

\[ \begin{tikzcd}[column sep = huge, row sep = huge, every label/.append style={font=\scriptsize}]
\mC 
\arrow[r, bend left=50, "F"{name=f}] 
\arrow[r, bend right=50,"F"'{name=g}] 
& \mX 
 \arrow[Rightarrow, shorten <= 10pt, shorten >=10pt, from=f, to=g, "1_{F}"]
\end{tikzcd} 
\begin{tikzcd}[]
\ \rar[mapsto] & \
\end{tikzcd}
\begin{tikzcd}[column sep = huge, row sep = huge, every label/.append style={font=\scriptsize}]
\mCW
\arrow[r, bend left=50, "[F\rbrack"{name=f}] 
\arrow[r, bend right=50,"[F\rbrack"'{name=g}] 
& \mX 
 \arrow[Rightarrow, shorten <= 10pt, shorten >=10pt, from=f, to=g, "\lbrack 1_{F} \rbrack"]
\end{tikzcd}\] 

\noi where the commuting diagram

\[ \begin{tikzcd}[]
\mCW_0 \rar["\lbrack 1_{F} \rbrack"] & \mX_1 \\
\mC_0 \uar[equals] \ar[ur, bend right, "1_F"'] & \\
\mCW_0 \uar[equals] \ar[uur, bend right, "1_{[F\rbrack}"'] 
\end{tikzcd}\]

\noi shows that the components of $[1_F]$ coincide with those of $1_{[f]}$. This means $[1_F] = 1_{[F]}$ are the same natural transformation and so identities are preserved. 

\noi To see composition is preserved suppose we have two vertically composable internal natural transformations: 

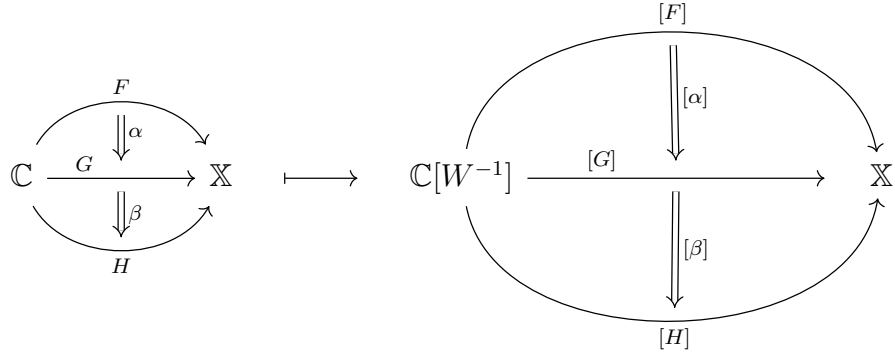
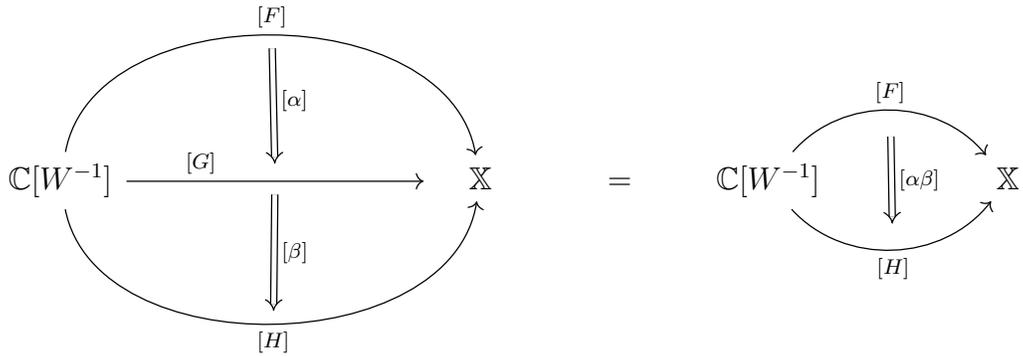
\[ \begin{tikzcd}[column sep = huge, row sep = huge, every label/.append style={font=\scriptsize}]
\mC 
\arrow[r, shift left, bend left=60, "F"{name=f}] 
\arrow[r, "G"near start, ""{name=g}] 
\arrow[r, shift right, bend right=60,"H"'{name=h}] 
& \mX 
 \arrow[Rightarrow, shorten <= 5pt, shorten >=1pt, from=f, to=g, "\alpha"]
\arrow[Rightarrow, shorten <= 5pt, shorten >=5pt, from=g, to=h, "\beta"]
\end{tikzcd}
\begin{tikzcd}[]
\ \rar[mapsto] & \
\end{tikzcd}
\begin{tikzcd}[column sep = huge, row sep = huge, every label/.append style={font=\scriptsize}]
\mCW
\arrow[rr, bend left=80, "[F\rbrack"{name=f}] 
\arrow[rr, "[G\rbrack"near start, "\ "{name=g}] 
\arrow[rr, bend right=80,"[H\rbrack"'{name=h}] 
&& \ \ \ \mX \ \ \ 
 \arrow[Rightarrow, shorten <= 5pt, shorten >=1pt, from=f, to=g, "\lbrack \alpha \rbrack"]
 \arrow[Rightarrow, shorten <= 5pt, shorten >=5pt, from=g, to=h, "\lbrack \beta \rbrack"]
\end{tikzcd}\] 

\noi We can see 

\[\begin{tikzcd}[column sep = huge, row sep = huge, every label/.append style={font=\scriptsize}]
\mCW
\arrow[rr, bend left=80, "[F\rbrack"{name=f}] 
\arrow[rr, "[G\rbrack"near start, "\ "{name=g}] 
\arrow[rr, bend right=80,"[H\rbrack"'{name=h}] 
&& \ \ \ \mX \ \ \ 
 \arrow[Rightarrow, shorten <= 5pt, shorten >=1pt, from=f, to=g, "\lbrack \alpha \rbrack"]
 \arrow[Rightarrow, shorten <= 5pt, shorten >=5pt, from=g, to=h, "\lbrack \beta \rbrack"]
\end{tikzcd} \qquad = \qquad 
\begin{tikzcd}[column sep = huge, row sep = huge, every label/.append style={font=\scriptsize}]
\mCW
\arrow[r, bend left=50, "[F\rbrack"{name=f}] 
\arrow[r, bend right=50,"[H\rbrack"'{name=g}] 
& \mX 
 \arrow[Rightarrow, shorten <= 10pt, shorten >=10pt, from=f, to=g, "\lbrack \alpha \beta \rbrack"]
\end{tikzcd}\]

\noi by noticing that their components coincide via the following commuting diagram: 

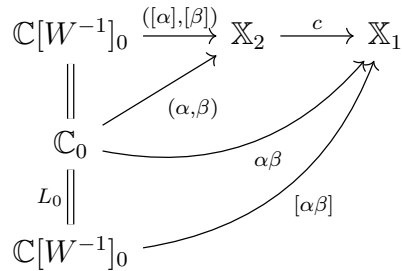
\[ \begin{tikzcd}[]
\mCW_0 \rar["([\alpha \rbrack {,} [ \beta \rbrack)"] 
& \mX_2 \rar["c"] 
&\mX_1 \\
\mC_0 \uar[equals] \ar[ur, "(\alpha {,} \beta)"'] 
\ar[urr, bend right, "\alpha \beta"']
&& \\
\mCW_0 \uar[equals, "L_0"] \ar[uurr, bend right, "[\alpha \beta\rbrack"'] 
\end{tikzcd}\]
\end{proof}

\noi The following theorem is a direct consequence of all the lemmas that came before in this section and formalizes the universal property of the internal localization, $\mCW$. 

\begin{thm}\label{thm 2-dimensional universal property of localization}
There is an isomorphism of categories

\[ [\mC , \mX]_W^\cE \cong [\mCW , \mX]^{\cE}\]

\noi between internal functors $\mC \to \mX$ in $\cE$ that invert $w : W \to \mC_1$ and their internal natural transformations, and internal functors $\mCW \to \mX$ in $\cE$ and their internal natural transformations. 
\end{thm}
\begin{proof}
The objects are in bijection by Proposition~\ref{prop UPIntFrac 1-cell corresondence}, the arrows are in bijection by Lemma~\ref{lem int.loc 2-cell correspondence}, and functoriality follows from Lemma~\ref{lem 2-cell correspondence functoriality}. 
\end{proof}

\section{Pseudocolimits of Small Filtered Diagrams of Internal Categories} \label{Ch pseudocolims}

\subsection{Internal Fractions Applied to the internal category of elements}\label{S IntFrc Applied to IntGroth}

Exercise 6.6, of Expos\'e VI in \cite{SGA4} states that the pseudocolimit of a filtered diagram $\cA^{op} \to \Cat$ can be obtained by localizing the Grothendieck construction with respect to the cartesian arrows. A current paper in progress, \cite{ThreeFsforBiCatsII}, by Bustillo-Vazquez, Pronk, and Szyld shows that with a weaker composition axiom for the category of fractions, the class of arrows one needs to invert to get the pseudocolimit can be reduced from all cartesian arrows to a convenient cleavage of them. For the rest of this chapter we consider an arbitrary but fixed filtered diagram, $D : \cA^{op} \to \Cat(\cE)$ so that every finite diagram in $\cA^{op}$ has a cone. The main theorem of this section states that, in a suitable context $\cE$, the pseudocolimit of a filtered diagram of internal categories, $\cA^{op} \to \Cat(\cE)$, can be computed by forming the internal category of (right) fractions of the internal category of elements with respect to the object representing the canonical cleavage of the cartesian arrows.

Note that the axioms we gave in Section \ref{S the axioms (of internal fractions)} are for a category of \textit{right} fractions, so we need to use the \textit{contravariant} form of the internal category of elements for a functor $D : \cA^{op} \to \Cat(\cE)$. In Section \ref{SS canonical cleavage of Cart arr's in Groth} we introduce the object representing the canonical cleavage and show that it satisfies the Internal Fractions Axioms in Definition~\ref{def Internal Fractions Axioms}. Section \ref{SS IntFrac of IntGroth is PseudoColim} is all about proving the main result of this thesis. Namely that, when it exists, the internal category of (right) fractions, $\mDW$, of the internal category of elements with respect to the canonical cleavage object, $(\mD , W)$, is the pseudocolimit of the original filtered diagram $D : \cA^{op} \to \Cat(\cE)$. 

\subsubsection{The Canonical Cleavage of the Internal Category of Elements}\label{SS canonical cleavage of Cart arr's in Groth}

The internal category of elements we need for a contravariant functor and a calculus of (right) fractions has an object of arrows defined by

\[ \mD_1 = \coprod_{\varphi \in \cA_1} D_\varphi \qquad \text{ where } \qquad \begin{tikzcd} 
D_\varphi \arrow[dr, phantom, "\usebox\pullback" , very near start, color=black]\dar["\pi_0"'] \rar["\pi_1"] 
& D(A)_1 \dar["t"] \\
D(B)_0 \rar["D(\varphi)_0"'] 
& D(A)_0 \end{tikzcd} \]

\noi whenever $\varphi : A \to B$ is an arrow in $\cA$. The subtle difference in this definition is that the vertical map on the right is a target rather than a source and that $D(\varphi) : D(B) \to D(A)$ for $\varphi : A \to B$ in $\cA$. Another subtle but important difference is the definition of cofiber composition for this version of the Grothendieck construction. For $\varphi : A \to B$ and $\psi : B \to C$ in $\cA$, the cofiber composition is given by

\[ \begin{tikzcd}[]
D_{\varphi ; \psi}
\ar[d, "\pi_1"'] 
\ar[r,"c'_{\varphi ; \psi ; \delta^{-1}}"] 
\ar[dr, dotted, "c_{\varphi ; \psi}"] 
& D(A)_3 \ar[dr, bend left, "c"] & \\
D_\psi 
\ar[dr, bend right, "\pi_0"'] 
& D_{\varphi \circ \psi} \rar["\pi_1"] 
\dar["\pi_0"']
\arrow[dr, phantom, "\usebox\pullback" , very near start, color=black] 
& D(A)_1 \dar["t"] \\
& D(C)_0 \rar["D(\varphi \circ \psi)_0"'] & D(A)_0 
\end{tikzcd}\]

\noi where $c'_{\varphi ; \psi ; \delta^{-1}}$ is the universal map 

\[\begin{tikzcd}[]
D_{\varphi ; \psi} 
\ar[ddr, bend right, "c'_{\varphi ; \psi}"'] 
\ar[drr, bend left, "c'_{(\varphi ; \psi) ; \delta^{-1}}"] 
\ar[dr, dotted, "c'_{\varphi ; \psi ; \delta^{-1}}"] 
& & \\
& D(A)_3 \rar["\pi_1"] \dar["\pi_0"'] & D(A)_2 \dar["\pi_0"] \\
& D(A)_2 \rar["\pi_1"'] & D(A)_1
\end{tikzcd}.\]

\noi given explicitly by the triple:

\[ \label{contravar cofiber composition arrow} c'_{\varphi ; \psi ; \delta^{-1}} = \big( \pi_0 \pi_1 , \pi_1 \pi_1 D(\varphi)_1 , \pi_1 \pi_0 \delta_{\varphi ; \psi}^{-1} \big) \tag{$\star$} \]

\noi in which $\delta_{\varphi ; \psi}^{-1} : D(C)_0 \to D(A)_1$ represents the inverse components for the structure isomorphism of the pseudofunctor, $D : \cA^{op} \to \Cat(\cE)$. 

When $\cE = \Set$, the arrows represented by $\mD_1$ are pairs $(\varphi, f) : ( A,a) \to (B,b)$ where $b \in D(B)_0$ and $f : a \to D(\varphi)(b)$ is in $D(A)_1$. The arrows being picked out by $w : W \to \mD_1$ should correspond to pairs $(\varphi, 1_{D(\varphi)}(b)) : (A,D(\varphi)(b)) \to (B,b)$ where $b \in D(B)_0$ and $1_{D(\varphi)(b)} : D(\varphi)(b) \to D(\varphi)(b)$ is the identity map in $D(A)_1$. The following definition describes an extra condition on $\cE$ that we need in order to work with an object of the canonical cleavage of cartesian arrows in $\mD$.

\begin{defn}\label{def admits a canonical cleavage of cartesian arrows}
Suppose $\cE$ admits an internal category of elements of $D : \cA^{op} \to \Cat(\cE)$. Then we say that $\mD$ {\em admits a canonical cleavage of cartesian arrows} if for each $\varphi : A \to B$ in $\cA$, the top pullback 

\[ 
\begin{tikzcd}[]
W_\varphi \arrow[dr, phantom, "\usebox\pullback" , very near start, color=black] \dar[tail, "w_\varphi"'] \rar["\pi_\varphi"] & D(A)_0 \dar[tail, "e"] \dar[dd, bend left = 45, equals] \\
D_\varphi \arrow[dr, phantom, "\usebox\pullback" , very near start, color=black]\dar["\pi_0"'] \rar["\pi_1"] 
& D(A)_1 \dar["t"] \\
D(B)_0 \rar["D(\varphi)_0"'] 
& D(A)_0 
\end{tikzcd}
\]

\noi exists and the coproduct 

\[ W = \coprod_{\varphi \in \cA_1} W_\varphi\]

\noi over all $\varphi \in \cA_1$ exists in $\cE$.
\end{defn}

\noi When $\mD$ admits an object of the canonical cleavage of cartesian arrows as in Definition \ref{def admits a canonical cleavage of cartesian arrows} we can use the universal property of the coproduct, $W$, to get the map $w : W \to \mD_1$ in $\cE$ as follows: 

\[\begin{tikzcd}
W \rar[dotted, "w"] & \mD_1 \\
W_\varphi \uar["\iota_\varphi"] \ar[ur, tail, bend right, "w_\varphi"'] & 
\end{tikzcd} \]

\noi This can be thought of as indexing the canonical cleavage of the cartesian arrows in the internal category $\mD$. From this point on we assume that $\mD$ admit a canonical cleavage of cartesian arrows. The first lemma we prove in this section shows that $(\mCW, W)$ satisfies \textbf{Int.Frc.1}. In the case when $\cE = \Set$, the sections of the target map are given by 
\[ (1_B , D(1_B)(b) : (B, D(1_B)(b)) \to (B, b) \]

\noi which are completely determined by $B \in \cA_0$ and the objects of $D(B)$.

\begin{lem}[\textbf{Int.Frc.1}]\label{lem IntFrc1 for internal groth}
There exists a section of the target map $w t : W \to \mD_0$. 
\end{lem}
\begin{proof}
It suffices to show that the cofibers $w_\varphi \pi_0 : W_\varphi \to D(B)_0$ have sections. For each $B \in \cA_0$, the cofiber section, $\alpha_B: D(B)_0 \to W_{1_B}$, of the target map on, $w_{1_B}\pi_0 : W_{1_B} \to D(B)_0$, is induced by the pair of maps $1_{D(B)_0}, D(1_B)_0 : D(B)_0 \to D(B)_0$. This is shown in the following commuting diagram, where the outer square clearly commutes and induces the dotted arrows on the left by the universal property of the two pullback squares on the inside.

\[ 
\begin{tikzcd}[]
D(B)_0 
\ar[dddr, bend right, equals] 
\ar[ddr, bend right, dotted] 
\ar[dr, dotted, "\alpha_B"] 
\ar[drr, bend left, "D(1_B)_0"] 
&
&
\\
& W_{1_B} 
\arrow[dr, phantom, "\usebox\pullback" , very near start, color=black] 
\dar[tail, "w_{1_B}"'] 
\rar["\pi_{1_B}"] 
& D(B)_0 
\dar[tail, "e"] 
\dar[dd, bend left = 45, equals] \\
& D_{1_B} 
\arrow[dr, phantom, "\usebox\pullback" , very near start, color=black]
\dar["\pi_0"'] 
\rar["\pi_1"] 
& D(B)_1 
\dar["t"] \\
& D(B)_0 
\rar["D(1_B)_0"'] 
& D(B)_0 
\end{tikzcd}
\]

\noi Using the universal property of coproducts, the section $\alpha : \mD_0 \to \mD_1$ is induced by the family of maps $\{\alpha_B \iota_{1_B} : B \in \cA_0\}$. Since the map $wt : W \to \mD_0$ is induced by the family of maps $\{ w_{\varphi} \pi_0 : \varphi \in \cA_1\}$ we have that 
\[ \iota_B \alpha w t = \alpha_B \iota_{1_B} w t = \alpha_B w_{1_B} \pi_0 \iota_B.\]

\noi This means the diagram
\[ 
\begin{tikzcd}[]
\mD_0 \ar[dr, equals] \rar["\alpha"] & W \dar["wt"] \\
& \mD_0 
\end{tikzcd}\]

\noi commutes by the universal property of the coproduct $\mD_0$ and it follows that $\alpha :\mD_0 \to W$ is a section of $wt : W \to \mD_0$. 
\end{proof}

\noi Before we prove the second axiom, let us consider the case when $\cE = \Set$. Here one typically shows that any composable arrows

\[ \begin{tikzcd}[]
(A,a) \rar["( \varphi {,} 1)"] & ( B,b) \rar["( \psi {,} 1)"] & (C,c)
\end{tikzcd}\]

\noi in $W \subseteq \mD_1$ can be precomposed by an arrow, 

\[ \begin{tikzcd}[column sep = large]
\big( A , D(\varphi \circ \psi)(c) \big) \rar[rrr, " \big( 1_A {,} \delta_{1_A \circ (\varphi \circ \psi) {,} c} D(1_A) (\delta_{\varphi \circ \psi {,} c}) \big)"]
 &&& (A,a)
\end{tikzcd}\]

\noi in $\mD_1$ to make the diagram 

\[ \begin{tikzcd}[row sep = large, column sep = large]
\big( A , D(\varphi \circ \psi)(c) \big) 
\dar["\big( 1_A {,} \delta_{1_A \circ (\varphi \circ \psi) , c} D(1_A) (\delta_{\varphi \circ \psi , c}) \big)"'] 
\ar[drr, bend left = 15, "(\varphi \circ \psi {,} 1 )"] & & \\
(A ,a) \rar["( \varphi {,} 1)"'] & (B,b) \rar["( \psi {,} 1)"'] & (C,c)
\end{tikzcd} \]

\noi commute in $\mD$. A convenient way to show this is to first notice that $a = D(\varphi)(b)$ and $b = D(\psi)(c)$ by definition, $D(\varphi)(1_{D(\psi)(c)}) = 1_{D(\varphi) \circ D(\psi) (c)}$ by functoriality of $D(\varphi)$, and the composite of the two arrows in $W$ is: 

\[ 
\begin{tikzcd}[column sep = large, row sep = large]
(A,a) 
\rar["(\varphi {,} 1)"] 
\ar[dr, "(\varphi \circ \psi {,} \ \delta_{\varphi \circ \psi {,} c}^{-1} ) "' ]
& ( B,b) 
\dar["(\psi {,} 1)"] \\
& (C,c)
\end{tikzcd}
\]

\noi Now computing the composite 

\[ \begin{tikzcd}[column sep = large, row sep = large]
\big( A, D(\varphi \circ \psi) (c) \big)
\rar[rrr,"\big( 1_A {,} \delta_{1_A \circ (\varphi \circ \psi) {,} c} D(1_A) (\delta_{\varphi \circ \psi {,} c}) \big)"]
 \ar[drrr, "( \varphi \circ \psi \ {,} \ 1)"']
& & & 
(A,a) \dar[ "( \varphi \circ \psi \ {,} \ \delta_{\varphi \circ \psi {,} c}^{-1}) " ] \\
& & & (C,c) 
\end{tikzcd}  \]

\noi in $\mD$ is done by noting that $1_A \circ (\varphi \circ \psi) = \varphi \circ \psi$ in $\cA^{op}$ and checking that

\[ \label{dgm usual case weak comp cancellation} \begin{tikzcd}[column sep = tiny, row sep = large]
D(\varphi \circ \psi) (c) 
\dar["\delta_{1_A \circ (\varphi \circ \psi) {,} c}"'] 
\ar[rr, equals] 
&& D(\varphi \circ \psi)(c) 
 \\
D(1_A) \circ D(\varphi \circ \psi) (c) 
\ar[dr,"D(1_A)(\delta_{\varphi \circ \psi {,} c})"'] 
\ar[rr, equals] 
&& D(1_A) \circ D(\varphi \circ \psi) (c) 
\uar["\delta_{1_A \circ (\varphi \circ \psi) {,} c}^{-1}"'] 
& \\
&
D(1_A)\circ D(\varphi) \circ D(\psi) (c) 
\ar[ur, "D(1_A) (\delta_{\varphi \circ \psi {,} c}^{-1} )"'] 
& 
\end{tikzcd} \tag{$\star$} \]

\noi commutes in the category $D(A)$. The bottom left triangle commutes by functoriality of $D(1_A)$ and then the outer triangle commutes by definition of the natural isomorphism $\delta_{1_A \circ (\varphi \circ \psi)}$. We now give an internal version of this proof.

\begin{lem}[\textbf{Int.Frc.2}]\label{lem IntFrc2 for internal groth}
There exists a cover $\begin{tikzcd}[]
U \rar["/" marking, "u" near start] & W 
\tensor[_{wt}]\times{_{ws}}W
\end{tikzcd}$ and a lift $\ell : U \to W_\circ$ such that the diagram 

\[ \begin{tikzcd}[]
& W_\circ \dar["\pi_0 \pi_{12}"] \\
U \ar[ur, bend left = 15, dotted, "\ell"] \rar["/" marking, "u"' near end] &W \prescript{}{wt}{\times_{ws}} W
\end{tikzcd}\]

\noi commutes in $\cE$. 
\end{lem}
\begin{proof}
By extensivity we have that $W \prescript{}{wt}{\times_{ws}} W \cong \coprod_{(\varphi, \psi) \in \cA_2} W_{\varphi ; \psi}$ where the cofibers are given by pullbacks 

\[ \begin{tikzcd}[column sep = large, row sep = large]
W_{\varphi ; \psi} 
\rar["\pi_1"] 
\dar["\pi_0"'] 
\arrow[dr, phantom, "\usebox\pullback" , very near start, color=black] 
& W_\psi 
\ar[dr, "\pi_\psi e"] 
\rar[tail, "w_\psi"] 
\dar["\pi_\psi"] 
& D_\psi 
\dar[ "\pi_1"] \\
W_\varphi 
\rar["t_\varphi"'] 
\dar[tail, "w_\varphi"'] 
& D(B)_0 
& D(B)_1 \lar["s"] \\
D_\varphi 
\ar[ur, bend right, "\pi_0"'] 
\end{tikzcd}\]

\noi Now using the component maps of the structure isomorphisms for the pseudofunctor $D: \cA^{op} \to \Cat(\cE)$ we represent the composable vertical maps on the left-hand side in Diagram (\ref{dgm usual case weak comp cancellation}) by the internally composable pair, $D(C)_0 \to D(A)_2$, determined by the unviersal property of the following pullback: 

\[ \begin{tikzcd}[column sep = large, row sep = large]
D(C)_0
 \ar[ddr, bend right, "\delta_{1_A \circ (\varphi \circ \psi)}"'] 
 \ar[rr, bend left, "\delta_{\varphi \circ \psi}"] 
 \ar[dr, dotted, "\tilde{\delta}^w_{\varphi \circ \psi}"] 
 & 
 & D(A)_1 
 \dar[bend left = 15,"D(1_A)_1"] \\
& D(A)_2
\arrow[dr, phantom, "\usebox\pullback" , very near start, color=black] 
\rar["\pi_1"] 
\dar["\pi_0"'] 
 & D(A)_1 \dar["s"] \\
& D(A)_1 \rar["t"'] & D(A)_0
\end{tikzcd}.\] 

\noi Specifying that such a pair comes from $W_{\varphi ; \psi}$ and composing with the composition structure map of $D(A)$ gives the internal version of one component of pre-composable map in $\mD$ that we need define the necessary lift:

\[ \begin{tikzcd}[column sep = large, row sep = large]
W_{ \varphi ; \psi} 
\rar["\pi_1 w_\psi \pi_0"] 
\ar[drr, dotted, "\delta^w_{\varphi \circ \psi}"'] 
& D(C)_0 
\rar["\tilde{\delta}^w_{\varphi \circ \psi}"] 
& D(A)_2
\dar["c"] \\
& & D(A)_1
\end{tikzcd}\] 

\noi To bring this together with the other component keeping track of the indexing, we need to map into the proper cofiber, $D_{1_A}$, which can be done using the universal property of the pullback:

\[ \begin{tikzcd}[column sep = large, row sep = large]
W_{\varphi ; \psi} 
\ar[drrr, bend left, "\delta^w_{\varphi \circ \psi}"] 
\ar[dd, bend right = 15, "\pi_1 w_\psi \pi_0"'] 
\ar[drr, dotted, "\delta^w_{1_A ; \varphi \circ \psi}"] & & & \\
& & D_{1_A} \rar["\pi_1"] \dar["\pi_0"'] 
\arrow[dr, phantom, "\usebox\pullback" , very near start, color=black] 
& D(A)_1 \dar["t"] \\
D(C)_0 \rar["D(\psi)_0"']
& D(B)_0\rar["D(\varphi)_0"']
& D(A)_0 \rar["D(1_A)_0"'] 
& D(A)_0 
\end{tikzcd}\]

\noi along with the fact that 

\begin{align*}
\delta^w_{\varphi \circ \psi} t 
&= \pi_1 w_\psi \pi_0 \tilde{\delta}^w_{\varphi \circ \psi} c t \\
&= \pi_1 w_\psi \pi_0 \delta_{\varphi \circ \psi} D(1_A)_1 t \\
&= \pi_1 w_\psi \pi_0 \delta_{\varphi \circ \psi} t D(1_A)_0 \\
&= \pi_1 w_\psi \pi_0 D(\psi)_0 D(\varphi)_0 D(1_A)_0. 
\end{align*}

\noi Now we need to compose the cofiber triple and show that it factors through $W_{\varphi \circ \psi}$ via some map $c^w_{\varphi ; \psi}$ in the following diagram. 

\[ \label{dgm weak composition triple factoring} \begin{tikzcd} 
W_{\varphi ; \psi} \dar[d, "c^w_{\varphi ; \psi}"'] \rar[rrr, "\big( \delta^w_{1_A ; \varphi \circ \psi} {,} \pi_0 w_\varphi {,}\pi_1 w_\psi \big)"] && & D_{1_A ; \varphi ; \psi} \dar["c_{1_A ; \varphi ; \psi}"] \\
W_{\varphi \circ \psi} \rar[rrr, "w_{\varphi \circ \psi}"'] & & & D_{\varphi \circ \psi} 
\end{tikzcd} \tag{$\star \star$} \] 

\noi We break this up into a couple steps using associativity of composition. First we compute the composite 

\[\begin{tikzcd}[]
W_{\varphi ;\psi} 
\ar[drrr, "(\pi_1 w_\psi \pi_0 {,} \pi_1 w_\psi \pi_0 \delta_{\varphi ; \psi}^{-1})"'] 
\rar[rrr,"(\pi_0 w_\varphi {,} \pi_1 \psi)"] &&& D_{\varphi ; \psi} \dar["c_{\varphi ; \psi} "] \\
&&& D_{\varphi \circ \psi} 
\end{tikzcd}\]

\noi by calculating 

\begin{align*}
 (\pi_0 w_\varphi {,} \pi_1 w_\psi) c'_{\varphi ; \psi ; \delta^{-1} } 
 &= (\pi_0 w_\varphi {,} \pi_1 w_\psi) (\pi_0 \pi_1 , \pi_1 \pi_1 D(\varphi)_1 , \pi_1 \pi_0 \delta_{\varphi ; \psi}^{-1} ) \\
 &= (\pi_0 w_\varphi \pi_1 , \pi_1 w_\psi \pi_1 D(\varphi)_1 , \pi_1 w_\psi \pi_0 \delta_{\varphi ; \psi}^{-1} )\\
 &= (\pi_0 \pi_\varphi e , \pi_1 \pi_\psi e D(\varphi)_1 , \pi_1 w_\psi \pi_0 \delta_{\varphi ; \psi}^{-1} )\\
 &= (\pi_0 \pi_\varphi e , \pi_0 w_\varphi \pi_0 e D(\varphi)_1 , \pi_1 w_\psi \pi_0 \delta_{\varphi ; \psi}^{-1} )\\
 &= (\pi_0 \pi_\varphi e , \pi_0 w_\varphi \pi_0 D(\varphi)_0 e , \pi_1 w_\psi \pi_0 \delta_{\varphi ; \psi}^{-1} )\\
 &= (\pi_0 \pi_\varphi e , \pi_0 \pi_\varphi e , \pi_1 w_\psi \pi_0 \delta_{\varphi ; \psi}^{-1} )
\end{align*}

\noi and then using the identity law in $D(A)$ twice in the last line below to see

\begin{align*}
(\pi_0 w_\varphi {,} \pi_1 w_\psi)c_{\varphi ; \psi} 
&= (\pi_0 w_\varphi {,} \pi_1 w_\psi) (\pi_1 \pi_0 , c'_{\varphi ; \psi ; \delta^{-1} } c) \\
&= \big( \pi_1 w_\psi \pi_0 ,  (\pi_0 w_\varphi {,} \pi_1 w_\psi) c'_{\varphi ; \psi ; \delta^{-1} } c \big) \\
&= \big( \pi_1 w_\psi \pi_0 ,  (\pi_0 \pi_\varphi e , \pi_0 \pi_\varphi e , \pi_1 w_\psi \pi_0 \delta_{\varphi ; \psi}^{-1} ) c \big) \\
&= \big( \pi_1 w_\psi \pi_0 ,  \pi_1 w_\psi \pi_0 \delta_{\varphi ; \psi}^{-1} \big). 
\end{align*}

\noi Now to see that we can pre-compose 
\[(\pi_0 w_\varphi , \pi_1 w_\psi) c_{\varphi ; \psi} : W_{\varphi ; \psi} \to D_{\varphi \circ \psi}\] 

\noi with $\delta^w_{1_A ; \varphi \circ \psi} : W_{\varphi ; \psi} \to D_{1_A}$ at the cofiber $D_{1_A ; (\varphi \circ \psi)}$, we check that

\begin{align*}
\delta^w_{1_A ; \varphi \circ \psi} \pi_0 
&= \pi_1 w_\psi \pi_0 D(\psi)_0 D(\varphi)_0 \\
&= \pi_1 w_\psi \pi_0 \delta_{\varphi ; \psi} t \\
&= \pi_1 w_\psi \pi_0 \delta_{\varphi ; \psi}^{-1} s \\
&= (\pi_1 w_\psi \pi_0 , \pi_1 w_\psi \pi_0 \delta_{\varphi ; \psi}^{-1}) \pi_1 s\\
&= (\pi_0 w_\varphi , \pi_1 w_\psi) c_{\varphi ; \psi} \pi_1 s.
\end{align*}

\noi To see this cofiber composition factors through $W_{\varphi \circ \psi}$ we need to show that the arrow given by post-composing with the first projection, $\pi_1 : D_{\varphi \circ \psi} \to D(A)_1$ is an identity. To break this up a bit we first calculate 

\begin{align*}
& \ \ \ \ \big( \delta^w_{1_A ; \varphi \circ \psi} , \ (\pi_0 w_\varphi , \pi_1 w_\psi) c_{\varphi ; \psi} \big) c'_{1_A ; (\varphi \circ \psi) ; \delta^{-1}} \\
&= \big( \delta^w_{1_A ; \varphi \circ \psi} , \ (\pi_0 w_\varphi , \pi_1 w_\psi) c_{\varphi ; \psi} \big) \big( \pi_0 \pi_1 , \ \pi_1 \pi_1 D(1_A)_1 , \ \pi_1 \pi_0 \delta_{1_A ; (\varphi \circ \psi)}^{-1} \big)\\
&= \big( \delta^w_{1_A ; \varphi \circ \psi} \pi_1 ,  \ \pi_1 w_\psi \pi_0 \delta_{\varphi ; \psi}^{-1} D(1_A)_1 , \ \pi_1 w_\psi \pi_0 \delta_{1_A ; (\varphi \circ \psi)}^{-1} \big)\\
&=\big( \delta^w_{\varphi \circ \psi} ,  \ \pi_1 w_\psi \pi_0 \delta_{\varphi ; \psi}^{-1} D(1_A)_1 , \ \pi_1 w_\psi \pi_0 \delta_{1_A ; (\varphi \circ \psi)}^{-1} \big)\\
&= \big( \pi_1 w_\psi \pi_0 \tilde{\delta}_{\varphi \circ \psi}^w c ,  \ \pi_1 w_\psi \pi_0 \delta_{\varphi ; \psi}^{-1} D(1_A)_1 , \ \pi_1 w_\psi \pi_0 \delta_{1_A ; (\varphi \circ \psi)}^{-1} \big)\\
\end{align*}

\noi and then substituting it into the following calculation along with the definition 

\[\tilde{\delta}_{\varphi \circ \psi}^w = (\delta_{1_A ; \circ(\varphi \circ \psi} , \delta_{\varphi ; \psi} D(1_A)_1)\]

\noi gives: 

\begin{align*}
& \ \ \ \ 
\big( \delta^w_{1_A ; \varphi \circ \psi} , (\pi_0 w_\varphi , \pi_1 w_\psi) c_{\varphi ; \psi} \big) c_{1_A ; (\varphi \circ \psi)} \pi_1 \\
&= \big( \delta^w_{1_A ; \varphi \circ \psi} , (\pi_0 w_\varphi , \pi_1 w_\psi) c_{\varphi ; \psi} \big) c'_{1_A ; (\varphi \circ \psi) ; \delta^{-1}} c\\
&= \big( \pi_1 w_\psi \pi_0 \tilde{\delta}_{\varphi \circ \psi}^w c ,  \ \pi_1 w_\psi \pi_0 \delta_{\varphi ; \psi}^{-1} D(1_A)_1 , \ \pi_1 w_\psi \pi_0 \delta_{1_A ; (\varphi \circ \psi)}^{-1} \big) c \\
&= \big( \pi_1 w_\psi \pi_0 \tilde{\delta}_{\varphi \circ \psi}^w c ,  \ \pi_1 w_\psi \pi_0 ( \delta_{\varphi ; \psi}^{-1} D(1_A)_1 , \delta_{1_A ; (\varphi \circ \psi)}^{-1})c  \big) c \\
&= \pi_1 w_\psi \pi_0\big( \delta_{1_A ; \circ(\varphi \circ \psi)} , \ \delta_{\varphi ; \psi} D(1_A)_1 , \ \delta_{\varphi ; \psi}^{-1} D(1_A)_1 , \ \delta_{1_A ; (\varphi \circ \psi)}^{-1} \big) c \\
&= \pi_1 w_\psi \pi_0\big( \delta_{1_A ; \circ(\varphi \circ \psi)} , \ (\delta_{\varphi ; \psi} D(1_A)_1 , \ \delta_{\varphi ; \psi}^{-1} D(1_A)_1)c , \ \delta_{1_A ; (\varphi \circ \psi)}^{-1} \big) c \\
&= \pi_1 w_\psi \pi_0\big( \delta_{1_A ; \circ(\varphi \circ \psi)} , \ (\delta_{\varphi ; \psi} , \ \delta_{\varphi ; \psi}^{-1})c D(1_A)_1 , \ \delta_{1_A ; (\varphi \circ \psi)}^{-1} \big) c \\
&= \pi_1 w_\psi \pi_0\big( \delta_{1_A ; \circ(\varphi \circ \psi)} , \ e D(\varphi\circ \psi)_1 D(1_A)_1 , \ \delta_{1_A ; (\varphi \circ \psi)}^{-1} \big) c \\
&= \pi_1 w_\psi \pi_0\big( \delta_{1_A ; \circ(\varphi \circ \psi)} , \ D(\varphi \circ \psi)_0 D(1_A)_0 e , \ \delta_{1_A ; (\varphi \circ \psi)}^{-1} \big) c \\
&= \pi_1 w_\psi \pi_0\big( \delta_{1_A ; \circ(\varphi \circ \psi)} , \delta_{1_A ; \circ(\varphi \circ \psi)} t e , \ \delta_{1_A ; (\varphi \circ \psi)}^{-1} \big) c \\
&= \pi_1 w_\psi \pi_0\big( \delta_{1_A ; \circ(\varphi \circ \psi)}, \ \delta_{1_A ; (\varphi \circ \psi)}^{-1} \big) c \\
&= \pi_1 w_\psi \pi_0 D(1_A \circ (\varphi \circ \psi) )_0 e \\
&= \pi_1 w_\psi \pi_0 D(\varphi \circ \psi)_0 e. 
\end{align*}

\noi This internalizes the commutativity of Diagram (\ref{dgm usual case weak comp cancellation}) and shows that for every composable $\varphi : A \to B $ and $ \psi : B \to C$ in $\cA$ there is a commuting diagram 

\[ 
\begin{tikzcd}[column sep = large, row sep = large]
W_{\varphi ; \psi} \ar[dr, dotted, "c^w_{\varphi ; \psi}"] \ar[dd, bend right, "\big( \delta^w_{1_A ; \varphi \circ \psi} {, }(\pi_0 w_\varphi {, }\pi_1 w_\psi) c_{\varphi ; \psi} \big)"'] \ar[drr, bend left, "\pi_1 w_\psi \pi_0 D(\varphi \circ \psi)_0"] & & \\
&W_{\varphi \circ \psi}  
\arrow[dr, phantom, "\usebox\pullback" , very near start, color=black] 
\dar[tail, "w_{\varphi \circ \psi} "'] 
\rar["\pi_{\varphi \circ \psi}"] 
& D(A)_0 
\dar[tail, "e"] \\
D_{\varphi ; \psi} \rar["c_{1_A ; (\varphi \circ \psi)}"'] &D_{\varphi \circ \psi} 
\rar["\pi_1"] 
& D(A)_1 
\end{tikzcd} \qquad \qquad .
\]

\noi The factorization we needed appears on the left of the diagram above and by associativity of composition in $\mD$ we can conclude that the diagram we originally wanted (\ref{dgm weak composition triple factoring}) involving composable triples commmutes. This allows us to compute

\begin{align*}
(\delta_{\varphi ; \psi}^w {,} 1) (\pi_0 \iota_{1_A} {,}\iota_{\varphi ; \psi}^w) (\pi_0 {,} \pi_1 \pi_0 w {,} \pi_1 \pi_1 w) c 
&= \big( \delta_{\varphi ; \psi}^w \iota_{1_A} , \iota_{\varphi ; \psi}^w \pi_0 w , \iota_{\varphi ; \psi}^w \pi_1 w \big) c \\
&= \big( \delta_{\varphi ; \psi}^w \iota_{1_A} , \pi_0 \iota_\varphi^w w , \pi_1 iota_{\psi}^w w \big) c \\
&= \big( \delta_{\varphi ; \psi}^w \iota_{1_A} , \pi_0 w_\varphi \iota_\varphi , \pi_1 w_\psi \iota_{\psi} \big) c \\
&= \big( \delta_{\varphi ; \psi}^w , \pi_0 w_\varphi , \pi_1 w_\psi \big) \iota_{1_A ; \varphi ; \psi} c \\
&= \big( \delta_{\varphi ; \psi}^w , \pi_0 w_\varphi , \pi_1 w_\psi \big) c_{1_A ; \varphi ; \psi} \iota_{\varphi \circ \psi} \\
&= c_{\varphi ; \psi}^w w_{\varphi \circ \psi} \iota_{\varphi \circ \psi} \\
&= c_{\varphi ; \psi}^w \iota_{\varphi \circ \psi}^w w 
\end{align*}

\noi and induce the unique cofiber lift, $\ell_{\varphi ; \psi} : W_{\varphi ; \psi} \to W_\circ$, by the universal property of the pullback, $W_\circ$, that makes the following diagram commute.

\[ \begin{tikzcd}[column sep = large, row sep= large]
W_{\varphi ; \psi} \ar[ddd, bend right = 55, equals] \ar[drr, dotted, "\ell_{\varphi ; \psi}"] \ar[dd, "(\delta_{\varphi ; \psi}^w {,} 1) "'] \rar[rrr, "c_{\varphi ; \psi}^w"] &&& W_{\varphi \circ \psi} \dar["\iota_{\varphi \circ \psi}^w"] \\
&&W_\circ \dar["\pi_0"'] \rar["\pi_1"] 
\arrow[dr, phantom, "\usebox\pullback" , very near start, color=black] 
& W \dar[tail, "w"] \\
D_{1_A} \prescript{}{t}{\times_{s}} (W_{\varphi ; \psi}) \rar[rr,"(\pi_0 \iota_{1_A} {,}\iota_{\varphi ; \psi}^w)"'] 
\ar[d,"\pi_1"']
&& \mD_1 \prescript{}{t}{\times_{ws}} (W \prescript{}{wt}{\times_{ws}} W) 
\rar["(\pi_0 {,} \pi_1 \pi_0 w {,} \pi_1 \pi_1 w) c"'] 
\dar["\pi_{12}"]
& \mD_1\\
W_{\varphi ; \psi} \rar[rr, " \iota_{\varphi l \psi}^w"'] &&
W \prescript{}{wt}{\times_{ws}} W
&
\end{tikzcd}\]

\noi The universal property of coproducts then gives us the desired lift

\[ \begin{tikzcd} 
W \prescript{}{wt}{\times_{ws}} W \ar[rr, bend left, equals] \rar[dotted,"\ell"] & W_\circ \rar["\pi_0 \pi_{12}"]& W \prescript{}{wt}{\times_{ws}} W \\
W_{\varphi ; \psi} \uar["\iota_{\varphi ; \psi}^w"]
 \ar[ur, "\ell_{\varphi ; \psi}"'] \ar[urr, bend right, "\iota_{\varphi ; \psi}^w"'] 
\end{tikzcd}\] 

\noi where we take the identity map 
\[ \begin{tikzcd} W \prescript{}{wt}{\times_{ws}} W \rar[rr, equals, "1_{ W \prescript{}{wt}{\times_{ws}} W} "] && W \prescript{}{wt}{\times_{ws}} W\end{tikzcd} \] 
\noi as our cover. 
\end{proof}

\noi The next thing we need to show is the right Ore condtition. Taking a look at the proof when $\cE = \Set$ will be useful for guiding the reader through the internal version. Start by assuming there exists a cospan in $\mD$ whose right leg is in $W$: 

\[ \begin{tikzcd}[ ]
& \big(C, D(\psi)(b) \big) \dar["\circ" marking, "( \psi {,} 1_{D(\psi)(b)} )"] \\
\big(A,a\big) \rar["( \varphi{,} f)"'] & \big( B,b\big)
\end{tikzcd}\]

\noi Since $\cA$ is filtered, there exists an object $E \in \cA_0$ and two maps $\varphi^* : E \to A$ and $\psi^* : E \to C$ such that the square

\[ \begin{tikzcd}[]
E \rar["\psi^*"] \dar["\varphi^*"'] & C \dar["\psi"] \\
A \rar["\varphi"'] & B
\end{tikzcd}\]

\noi commutes in $\cA$. Now letting $\star$ denote the composition of arrows in the non-indexing compenent of the Grothendieck construction we can consider the commuting diagram: 

\[ \begin{tikzcd}[column sep = huge]
D(\varphi^*)(a) \rar[equals] \dar[ddd,equals] 
\ar[drr, bend right = 20, "1_{D(\varphi^*)(b)} \star f"] 
& D(\varphi^*)(a) \rar["D(\varphi^*)(f)"] 
& D(\varphi^*) \circ D(\varphi)(b) \dar["\delta_{\varphi^* ; \varphi , b}^{-1}"] \\
&
& D(\varphi^* \circ \varphi)(b) \dar[equals] \\
&
& D(\psi^* \circ \psi)(b) \dar["\delta_{\psi^* ; \psi , b}"] \\
D(\varphi^*)(a) \rar["g"'] \ar[urr, bend left =20, "g \star 1_{D(\psi)(b)}"'] 
& D(\psi^*) \circ D(\psi)(b) 
\rar[equals, "D(\psi^*)(1_{D(\psi)(b)}"'] 
& D(\psi^*) \circ D(\psi)(b) 
 \end{tikzcd}\]

\noi where 

\[ g = D(\varphi^*)(f) \delta_{\varphi^* ; \varphi , b}^{-1} \delta_{\psi^* ; \psi , b}. \]

\noi In particular we have 
\[ g \star 1_{D(\psi)(b)} = 1_{D(\varphi^*)(b)} \star f\]

\noi and so the square 

\[ \begin{tikzcd}[ ]
\big( E , D(\varphi^*)(a) \big) 
\dar["\circ" marking, "( \varphi^* {,} 1_{D(\varphi^*)(a)} )"']
\rar["( \psi^* {,} g )"] & \big( C , D(\psi)(b) \big) 
\dar["\circ" marking, "(\psi {,} 1_{D(\psi)(b)} )"] \\
\big(A,a\big) \rar["( \varphi {,} f)"'] & \big(B,b\big)
\end{tikzcd}\]

\noi commutes in $\mD$. Now we show how to internalize this proof when $\cE$ is not necessarily $\Set$. 

\begin{lem}[\textbf{Int.Frc.3}]\label{lem IntFrc3 for internal groth}\

There exists a cover, $\begin{tikzcd}[]
U \rar["/" marking, "u" near start] & \mD_1 \prescript{}{t}{\times_{wt}} W\end{tikzcd}$ and a lift $\begin{tikzcd}[]
U\ar[r, "\ell"] & W_\square 
\end{tikzcd}$ such that the following diagram commutes: 

\[ \begin{tikzcd}
& W_\square \dar["(\pi_0 \pi_1 {,} \pi_1 \pi_1)"] \\
U \rar["/" marking, "u"' near end] \ar[ur, bend left, dotted, "\ell"] & \mD_1 \prescript{}{t}{\times_{wt}} W
\end{tikzcd}\]

\noi where 

\[W_\square = (W \prescript{}{wt}{\times_{s}} \mD_1 ) \prescript{}{c}{\times_c} (\mD_1 \prescript{}{t}{\times_{ws}} W). \]
\end{lem}
\begin{proof}
Recall that $\csp = \mD_1 \prescript{}{t}{\times_{wt}} W$ and let $\csp(A)$ denote all the cospans in $\cA$ for simpler notation. Since $\cE$ is extensive we have the following isomorphisms: 

\begin{align*}
\csp & \cong 
\coprod_{(\varphi, \psi) \in \csp (A)} D_\varphi \prescript{}{t_\varphi}{\times_{w_\psi t_\psi}} W_\psi \\
 \mD_1 \prescript{}{t}{\times_{ws}} W & \cong \coprod_{(\psi^* , \psi) \in \cA_2)} D_{\psi^*} \prescript{}{t_{\psi^*}}{\times_{w_\psi s_\psi}} W_\psi\\
W \prescript{}{wt}{\times_{s}} \mD_1 & \cong \coprod_{(\varphi^* , \varphi) \in \cA_2)} W_{\varphi^*} \prescript{}{w_{\varphi^*}t_{\varphi^*}}{\times_{s_\varphi}} D_\varphi
 \end{align*}

\noi Now we will define two families of maps 

\[ \begin{tikzcd}[column sep = large]
W_{\varphi^*} \prescript{}{w_{\varphi^*}t_{\varphi^*}}{\times_{s_\psi}} D_\psi 
& 
D_\varphi \prescript{}{t_\varphi}{\times_{w_\psi t_\psi}} W_\psi \lar["\ell_{\varphi ; \psi , 0}"'] 
\rar["\ell_{\varphi ; \psi , 1}"]
& 
D_{\psi^*} \prescript{}{t_{\psi^*}}{\times_{w_\psi s_\psi}} W_\psi
\end{tikzcd}\] 

\noi before showing they agree after post-composing them each with the appropriate cofiber compositions. The left-hand side is simpler so we start there. Consider the following commuting pullback diagram: 

\[ \begin{tikzcd}[]
D_\varphi \prescript{}{t_\varphi}{\times_{w_\psi t_\psi}} W_\psi 
\rar["\pi_0"] 
\dar[dd, "\pi_0"'] 
\ar[drr, dotted, "\ell_{\varphi ; \psi , 0}^w"'] 
& W_\psi 
\rar["\pi_1"]
& D(A)_1 
\rar["s"] 
& D(A)_0 
\dar["D(\varphi^*)_0"] \\
& 
&W_{\varphi^*} 
\arrow[dr, phantom, "\usebox\pullback" , very near start, color=black] 
\dar[tail, "w_{\varphi^*}"']
\rar["\pi_{\varphi^*}"] 
&D(E)_0 
\dar[tail, "e"] 
\dar[dd, bend left = 40, equals] \\
D_\varphi 
\ar[dr,bend right = 20,"\pi_1"'] 
&
&D_{\varphi^*} 
\arrow[dr, phantom, "\usebox\pullback" , very near start, color=black] 
\dar[tail, "\pi_0"']
\rar["\pi_1"] 
&D(E)_1 
\dar["t"]\\ 
&
D(A)_1
\rar["s"']
&
D(A)_0 
\rar["D(\varphi^*)_0"'] 
& D(E)_0
\end{tikzcd}\]

\noi The lower left commuting square above then induces the map we need, $\ell_{\varphi ; \psi , 0}$, by the following commuting pullback diagram: 

\[ \begin{tikzcd}[]
D_\varphi \prescript{}{t_\varphi}{\times_{w_\psi t_\psi}} W_\psi 
\ar[dr, dotted, "\ell_{\varphi ; \psi , 0}"] 
\ar[drr, bend left, "\pi_0"] 
\ar[ddr, bend right, "\ell_{\varphi ; \psi , 0}^w"'] 
&
&
& \\
& W_{\varphi^*} \prescript{}{w_{\varphi^*}t_{\varphi^*}}{\times_{s_\varphi}} D_\varphi 
\arrow[dr, phantom, "\usebox\pullback" , very near start, color=black] 
\dar["\pi_0"'] 
\rar["\pi_1"] 
& D_\varphi 
\dar["s_\varphi"'] 
\rar["\pi_1" ]
& D(A)_1 \ar[dl, "s"] \\
& W_{\varphi^*} 
\dar[tail, "w_{\varphi^*}"'] 
\rar["w_{\varphi^*} t_{\varphi^*}"] 
& D(A)_0 
& \\
& D_{\varphi^*} \ar[ur, "\pi_0"'] 
&&
\end{tikzcd}\]\

\noi The map, $\ell_{\varphi ; \psi , 1}$, on the right-hand side is more involving to define, as we saw when $\cE = \Set$, because it requires defining the map `$g$' by composing with several other maps at hand. We first compute 

\begin{align*}
\pi_0 \pi_0 \delta_{\varphi^* ; \varphi}^{-1} s 
&= \pi_0 \pi_0 \delta_{\varphi^* ; \varphi} t \\
&= \pi_0 \pi_0 D(\varphi)_0 D(\varphi^*)_0 \\
&= \pi_0 \pi_1 t D(\varphi^*)_0 \\
&= \pi_0 \pi_1 D(\varphi^*)_1 t
\end{align*}

\noi to get one composable pair and then 

\begin{align*}
\pi_0 \pi_0 \delta_{\varphi^* ; \varphi}^{-1} t 
&= \pi_0 \pi_0 \delta_{\varphi^* ; \varphi} s \\
&= \pi_0 \pi_0 D(\varphi \circ \varphi^*)_0 \\
&= \pi_0 \pi_0 D(\psi \circ \psi^*)_0 \\
&= \pi_0 t_\varphi D(\psi \circ \psi^*)_0 \\
&= \pi_1 w_\psi t_\psi D(\psi \circ \psi^*)_0 \\
&= \pi_1 w_\psi \pi_0 D(\psi \circ \psi^*)_0 \\
&= \pi_1 w_\psi \pi_0 \delta_{\psi^* ; \psi} s 
\end{align*}

\noi to get another with the same map in the middle. This gives a unique map
\[ \begin{tikzcd}[]
D_\varphi \prescript{}{t_\varphi}{\times_{w_\psi t_\psi}} W_\psi 
\rar[rrrrrr, "( \pi_0 \pi_1 D(\varphi^*)_1 {,} \ \pi_0 \pi_0 \delta_{\varphi^* ; \varphi}^{-1} {,} \ \pi_1 w_\psi \pi_0 \delta_{\psi ; \psi^*} )"] 
&&&&&&
D(E)_3
\end{tikzcd}\]

\noi representing compsable triples in $D(E)$ whose composite we denote 

\[ \begin{tikzcd}[]
D_\varphi \prescript{}{t_\varphi}{\times_{w_\psi t_\psi}} W_\psi 
\rar[rrrrrr, "( \pi_0 \pi_1 D(\varphi^*)_1 {,} \ \pi_0 \pi_0 \delta_{\varphi^* ; \varphi}^{-1} {,} \ \pi_1 w_\psi \pi_0 \delta_{\psi^* ; \psi} )"] 
\ar[rrrrrrd, "\tilde{g}"'] 
&&&&&&
D(E)_3 \dar["c"] \\
&&&&&&D(E)_1
\end{tikzcd}.\]

\noi The target of this composite is 

\begin{align*}
\tilde{g} t = \pi_1 w_\psi \pi_0 \delta_{\psi^* ; \psi} t = \pi_1 w_\psi \pi_0 D(\psi)_0 D(\psi^*)_0
\end{align*}

\noi so there exists a unique map, $g$, in the commuting pullback diagram: 

\[ \begin{tikzcd}[column sep = large]
D_\varphi \prescript{}{t_\varphi}{\times_{w_\psi t_\psi}} W_\psi  
\dar["\pi_1"'] 
\ar[ddr, dotted, "g"] 
\ar[ddrr, bend left, "\tilde{g}"] 
& & \\
W_\psi 
\dar[tail, "w_\psi"'] 
\ar[ddr, "\pi_\psi"] 
& & \\
D_\psi \dar["\pi_0"'] 
&D_{\psi^*} 
\arrow[dr, phantom, "\usebox\pullback" , very near start, color=black] 
\rar["\pi_1"] 
\dar["\pi_0"'] 
& D(E)_1 \dar["t"] \\
D(B)_0 \rar["D(\psi)_0"'] 
& D(C)_0 \rar["D(\psi^*)_0"'] 
& D(E)_0
\end{tikzcd}\]

\noi The left side of the diagram above, along with the fact that $s_\psi = \pi_0 s$, allows us to finally define the cofiber lift by the universal property of the pullback: 

\[\begin{tikzcd}[]
D_\varphi \prescript{}{t_\varphi}{\times_{w_\psi t_\psi}} W_\psi  
\ar[ddr, bend right, "g"'] 
\ar[drr, bend left, "\pi_1"] 
\ar[dr, dotted, "\ell_{\varphi ; \psi , 1}"] & & & \\
& D_{\psi^*} \prescript{}{t_{\psi^*}}{\times_{w_\psi s_\psi}} W_\psi 
\rar["\pi_1"] 
\dar["\pi_0"']  
\arrow[dr, phantom, "\usebox\pullback" , very near start, color=black] 
& W_{\psi} 
\dar["\pi_\psi"'] 
\rar[tail , "w_\psi"] 
\ar[dr, "\pi_\psi e "] 
& D_\psi 
\dar["\pi_1"] \\
& D_{\psi^*} \rar["\pi_0"'] 
& D (C)_0 & D(C)_1 \lar["s"] 
\end{tikzcd}\]

\noi It only remains to show that the outside of the diagram, 

\[ \begin{tikzcd}[column sep = large, row sep = large]
&
D_\varphi \prescript{}{t_\varphi}{\times_{w_\psi t_\psi}} W_\psi  
\ar[d, dotted, "\ell_{\varphi ; \psi}"] 
\ar[dd, bend right, ""'] 
\ar[dr, bend left, ""] 
\ar[ddd, bend right = 45, "\ell_{\varphi ; \psi , 0}"'] 
\ar[drr, bend left = 20, "\ell_{\varphi ; \psi , 1}"] 
& 
\\
& W_\square \rar["\pi_1"] \dar["\pi_0"] 
\arrow[dr, phantom, "\usebox\pullback" , very near start, color=black]
& \mD_1 \prescript{}{t}{\times_{ws}} W \dar["c"] 
& D_{\psi^*} \prescript{}{t_{\psi^*}}{\times_{w_\psi s_\psi}} W_\psi 
\dar["1_{D_{\psi^*}} \times w_\psi"] 
\lar[tail, "\iota_{\psi^*} \times \iota_{\psi}^w"'] ''
\\
& W \prescript{}{wt}{\times_{s}} \mD_1 \rar["c"'] 
& \mD_1 
& D_{\psi^* ; \psi} \dar["c_{\psi^* ; \psi}"] 
\\
& W_{\varphi^*} \prescript{}{w_{\varphi^*}t_{\varphi^*}}{\times_{s_\varphi}} D_\varphi 
\uar[tail, "\iota_{\varphi^*}^w \times \iota_\varphi"'] 
\rar["w_{\varphi^*} \times 1_{D_\varphi}"'] 
& D_{\varphi^* ; \varphi} \rar["c_{\varphi^*; \varphi}"'] 
& D_{\varphi^* \circ \varphi } 
\ar[ul, tail, "\iota_{\varphi^* \circ \varphi}"]
\end{tikzcd},\] 

\noi commutes in $\cE$ in order to induce the cofiber lift $\ell_{\varphi ; \psi}$. Then the universal property of the coproduct, $\csp$, will give the lift we need with the cover taken to be the identity on $\csp$. All of the arrows involved have been defined by universal properties of pullbacks so we use pairing map notation to expand and manipulate them. Starting with the bottom composite, first we note, that by the universal properties of the pullbacks in the codomains of the following maps we have

\[ w_\varphi^* \times 1_{D_\varphi} = (\pi_0 w_{\varphi^*} , \pi_1 ), \qquad c_{\varphi^* ; \varphi} = (\pi_1 \pi_0 , c'_{\varphi^* ; \varphi ; \delta^{-1}} c)\]

\noi where similarly, by equation~(\ref{contravar cofiber composition arrow}) at the beginning of this section, 

\[ c'_{\varphi^* ; \varphi ; \delta^{-1}} = \big(\pi_0 \pi_1 , \pi_1 \pi_1 D(\varphi^*)_1 , \pi_1 \pi_0 \delta_{\varphi^* ; \varphi}^{-1} \big). \]

\noi Then in one component of the bottom composite we have 
\begin{align*}
\ell_{\varphi ; \psi , 0} (w_{\varphi^*} \times 1_{D_\varphi}) c_{\varphi^* ; \varphi} \pi_0 
&= \ell_{\varphi ; \psi , 0} (\pi_0 w_{\varphi^*} , \pi_1 ) c_{\varphi^* ; \varphi} \pi_0 \\
&= (\ell_{\varphi ; \psi , 0} \pi_0 w_{\varphi^*}, \ell_{\varphi ; \psi , 0} \pi_1 ) \pi_1 \pi_0 \\
&= \ell_{\varphi ; \psi , 0} \pi_1 \pi_0 \\
&= \pi_0 \pi_0 \\
&= \pi_0 t_\varphi. 
\end{align*}

\noi and in the other component we have

\begin{align*}
\ell_{\varphi ; \psi , 0} (w_{\varphi^*} \times 1_{D_\psi}) c_{\varphi^* ; \varphi} \pi_1 
&= \ell_{\varphi ; \psi , 0} (\pi_0 w_{\varphi^*} , \pi_1 ) c_{\varphi^* ; \varphi} \pi_1 \\
&= (\ell_{\varphi ; \psi , 0} \pi_0 w_{\varphi^*} , \ell_{\varphi ; \psi , 0} \pi_1 ) c_{\varphi^* ; \varphi} \pi_1 \\
&= (\ell_{\varphi ; \psi , 0} \pi_0 w_{\varphi^*}, \ell_{\varphi ; \psi , 0} \pi_1 ) c'_{\varphi^* ; \varphi ; \delta^{-1}} c\\
&= (\ell_{\varphi ; \psi , 0}^w w_{\varphi^*}, \pi_0 ) c'_{\varphi^* ; \varphi ; \delta^{-1}} c \\
&= (\ell_{\varphi ; \psi , 0}^w w_{\varphi^*}, \pi_0 )\big(\pi_0 \pi_1 , \pi_1 \pi_1 D(\varphi^*)_1 , \pi_1 \pi_0 \delta_{\varphi^* ; \varphi}^{-1}\big)c \\
&= \big( \ell_{\varphi ; \psi , 0}^w w_{\varphi^*} \pi_1 , \pi_0 \pi_1 D(\varphi^*)_1 , \pi_0 \pi_0 \delta_{\varphi^* ; \varphi}^{-1}\big)c. 
\end{align*}

\noi Now the first component in the triple of the last line above can be rewritten using the definition of the pullback $W_{\varphi^*}$ along with the definition of $\ell_{\varphi ; \psi , 0}^w$ and functoriality of $D(\varphi^*)$: 

\begin{align*}
\ell_{\varphi ; \psi , 0}^w w_{\varphi^*} \pi_1 
&= \ell_{\varphi ; \psi , 0}^w \pi_{\varphi^*} e \\ 
&= \pi_0 \pi_1 s D(\varphi^*)_0 e \\
&= \pi_0 \pi_1 s e D(\varphi^*)_1. 
\end{align*}

\noi Substituting this side calculation into the last line of the prior calculation and using associativity of composition, functoriality of $D(\varphi^*)$, and the identity law in $D(A)$ allows us to finally see that 

\begin{align*}
\ell_{\varphi ; \psi , 0} (w_{\varphi^*} \times 1_{D_\psi}) c_{\varphi^* ; \varphi} \pi_1 
= ... &= \big( \ell_{\varphi ; \psi , 0}^w w_{\varphi^*} \pi_1 , \pi_0 \pi_1 D(\varphi^*)_1 , \pi_0 \pi_0 \delta_{\varphi^* ; \varphi}^{-1}\big)c \\
&= \big( \pi_0 \pi_1 s e D(\varphi^*)_1 , \pi_0 \pi_1 D(\varphi^*)_1 , \pi_0 \pi_0 \delta_{\varphi^* ; \varphi}^{-1}\big)c \\
&= \big( (\pi_0 \pi_1 s e D(\varphi^*)_1 , \pi_0 \pi_1 D(\varphi^*)_1) c , \ \pi_0 \pi_0 \delta_{\varphi^* ; \varphi}^{-1}\big)c \\
&= \big( (\pi_0 \pi_1 s e , \pi_0 \pi_1 ) c D(\varphi^*)_1 , \ \pi_0 \pi_0 \delta_{\varphi^* ; \varphi}^{-1}\big)c \\
&= \big( \pi_0 \pi_1 (s e , 1_{D(A)_1}) c D(\varphi^*)_1 , \ \pi_0 \pi_0 \delta_{\varphi^* ; \varphi}^{-1}\big)c \\
&= \big( \pi_0 \pi_1 D(\varphi^*)_1 , \ \pi_0 \pi_0 \delta_{\varphi^* ; \varphi}^{-1}\big)c .
\end{align*}

\noi By the universal property of the pullback $D_{\varphi^* \circ \varphi} = D_{\psi^* \circ \psi}$, we can write the bottom composite as the following pairing map: 

\[ \ell_{\varphi ; \psi , 0} (w_{\varphi^*} \times 1_{D_\varphi}) c_{\varphi^* ; \varphi} = \big( \pi_0 t_\varphi , (\pi_0 \pi_1 D(\varphi^*)_1 , \ \pi_0 \pi_0 \delta_{\varphi^* ; \varphi}^{-1}) c \big) \]\

\noi For the top composite, we begin similarly by noting that 

\[1_{D_\psi^*} \times w_\psi = (\pi_0 , \pi_1 w_{\psi}), \qquad c_{\psi^* ; \psi} = (\pi_1 \pi_0 , c'_{\psi^* ; \psi ; \delta^{-1}} c)\]

\noi where similarly, by equation~(\ref{contravar cofiber composition arrow}) at the beginning of this section, 

\[ c'_{\psi^* ; \psi ; \delta^{-1}} = \big(\pi_0 \pi_1 , \pi_1 \pi_1 D(\psi^*)_1 , \pi_1 \pi_0 \delta_{\psi^* ; \psi}^{-1} \big). \]

\noi Then in one component of the top composite we have 
\begin{align*}
\ell_{\varphi ; \psi , 1} (1_{D_\psi^*} \times w_\psi) c_{\psi^* ; \psi} \pi_0 
&= \ell_{\varphi ; \psi , 1}(\pi_0 , \pi_1 w_{\psi}) c_{\psi^* ; \psi} \pi_0 \\
&= (\ell_{\varphi ; \psi , 1} \pi_0 , \ell_{\varphi ; \psi , 1} \pi_1 w_{\psi} ) c_{\psi^* ; \psi} \pi_0 \\
&= (\ell_{\varphi ; \psi , 1} \pi_0 , \ell_{\varphi ; \psi , 1} \pi_1 w_{\psi} ) \pi_1 \pi_0 \\
&= \ell_{\varphi ; \psi , 1} \pi_1 w_\psi \pi_0 \\
&= \pi_1 w_\psi \pi_0 \\
&= \pi_1 w_\psi t_\psi.
\end{align*}

\noi In the other component we get 

\begin{align*}
\ell_{\varphi ; \psi , 1} (1_{D_\psi^*} \times w_\psi) c_{\psi^* ; \psi} \pi_1
&= \ell_{\varphi ; \psi , 1}(\pi_0 , \pi_1 w_{\psi}) c_{\psi^* ; \psi} \pi_1 \\
&= (\ell_{\varphi ; \psi , 1} \pi_0 , \ell_{\varphi ; \psi , 1} \pi_1 w_{\psi} ) c'_{\psi^* ; \psi ; \delta^{-1}} c \\
&= (\ell_{\varphi ; \psi , 1} \pi_0 , \ell_{\varphi ; \psi , 1} \pi_1 w_{\psi} )
\big(\pi_0 \pi_1 , \pi_1 \pi_1 D(\psi^*)_1 , \pi_1 \pi_0 \delta_{\psi^* ; \psi}^{-1} \big) c \\
&= \big( \ell_{\varphi ; \psi , 1} \pi_0 \pi_1 , \ell_{\varphi ; \psi , 1} \pi_1 w_{\psi} \pi_1 D(\psi^*)_1 , \ell_{\varphi ; \psi , 1} \pi_1 w_{\psi} \pi_0 \delta_{\psi^* ; \psi}^{-1} \big) c \\
&= \big( g \pi_1 , \pi_1 \pi_{\psi} e D(\psi^*)_1 , \pi_1 w_{\psi} \pi_0 \delta_{\psi^* ; \psi}^{-1} \big) c \\
&= \big( \tilde{g} , \pi_1 \pi_{\psi} e D(\psi^*)_1 , \pi_1 w_{\psi} \pi_0 \delta_{\psi^* ; \psi}^{-1} \big) c \\
&= \big( \tilde{g} , \ (\pi_1 \pi_{\psi} e D(\psi^*)_1 , \pi_1 w_{\psi} \pi_0 \delta_{\psi^* ; \psi}^{-1}) c \big) c. 
\end{align*}

\noi Now looking at the last line above recall the definition of $\tilde{g}$:

\[ \tilde{g} = ( \pi_0 \pi_1 D(\varphi^*)_1 {,} \ \pi_0 \pi_0 \delta_{\varphi^* ; \varphi}^{-1} {,} \ \pi_1 w_\psi \pi_0 \delta_{\psi^* ; \psi} ) c\] 

\noi By definition of the pullback $W_\psi$, functoriality of $D(\psi^*)$, the definition of the structure isomorphism components $\delta_{\psi^* ; \psi}^{-1}$, and the identity law for internal composition in $D(E)$ we get

\begin{align*}
 (\pi_1 \pi_{\psi} e D(\psi^*)_1 , \pi_1 w_{\psi} \pi_0 \delta_{\psi^* ; \psi}^{-1}) c 
 &= (\pi_1 \pi_{\psi} 1_{D(C)_1} e D(\psi^*)_1 , \pi_1 w_{\psi} \pi_0 \delta_{\psi^* ; \psi}^{-1}) c \\
 &= (\pi_1 \pi_{\psi} 1_{D(C)_1} e D(\psi^*)_1 , \pi_1 w_{\psi} \pi_0 \delta_{\psi^* ; \psi}^{-1}) c \\
 &= (\pi_1 w_{\psi} \pi_0 D(\psi)_0 e D(\psi^*)_1 , \pi_1 w_{\psi} \pi_0 \delta_{\psi^* ; \psi}^{-1}) c \\
 &= (\pi_1 w_{\psi} \pi_0 D(\psi)_0D(\psi^*)_0 e , \pi_1 w_{\psi} \pi_0 \delta_{\psi^* ; \psi}^{-1}) c \\
 &= (\pi_1 w_{\psi} \pi_0 \delta_{\psi^* ; \psi}^{-1} s e , \pi_1 w_{\psi} \pi_0 \delta_{\psi^* ; \psi}^{-1}) c \\
 &= \pi_1 w_{\psi} \pi_0 \delta_{\psi^* ; \psi}^{-1} (s e , 1_{D(E)_1}) c\\
 &= \pi_1 w_{\psi} \pi_0 \delta_{\psi^* ; \psi}^{-1}.
\end{align*}

\noi Taking these side calculations into account and applying associativity of composition; the definition of the structure isomorphism components $\delta_{\psi^* ; \psi}$ and $\delta_{\psi^*; \psi}^{-1}$; the definitions of $t_\psi$, $t_\varphi$, and the pullback $D_\varphi \prescript{}{t_\varphi}{\times_{w_\psi t_\psi}} W_\psi$; the assumption that $\varphi^* \varphi = \psi^* \psi$ in $\cA$ which means $\varphi \circ \varphi^* = \psi \circ \psi^*$ in $\cA^{op}$; and the identity law for internal composition in $D(E)$ gives: 

\begin{align*}
& \ \ \ \ 
\big( \tilde{g} , \ (\pi_1 \pi_{\psi} e D(\psi^*)_1 , \pi_1 w_{\psi} \pi_0 \delta_{\psi^* ; \psi}^{-1}) c \big) c \\
&= \big( \pi_0 \pi_1 D(\varphi^*)_1 {,} \ \pi_0 \pi_0 \delta_{\varphi^* ; \varphi}^{-1} {,} \ \pi_1 w_\psi \pi_0 \delta_{\psi^* ; \psi} , \ \pi_1 w_{\psi} \pi_0 \delta_{\psi^* ; \psi}^{-1} \big) c \\
&= \big( \pi_0 \pi_1 D(\varphi^*)_1 {,} \ \pi_0 \pi_0 \delta_{\varphi^* ; \varphi}^{-1} {,} \ (\pi_1 w_\psi \pi_0 \delta_{\psi^* ; \psi} , \pi_1 w_{\psi} \pi_0 \delta_{\psi^* ; \psi}^{-1}) c \big) c \\
&= \big( \pi_0 \pi_1 D(\varphi^*)_1 {,} \ \pi_0 \pi_0 \delta_{\varphi^* ; \varphi}^{-1} {,} \ \pi_1 w_\psi \pi_0 (\delta_{\psi^* ; \psi} ,  \delta_{\psi^* ; \psi}^{-1}) c \big) c \\
&= \big( \pi_0 \pi_1 D(\varphi^*)_1 {,} \ \pi_0 \pi_0 \delta_{\varphi^* ; \varphi}^{-1} {,} \ \pi_1 w_\psi \pi_0 \delta_{\psi^* ; \psi} s e \big) c \\
&= \big( \pi_0 \pi_1 D(\varphi^*)_1 {,} \ \pi_0 \pi_0 \delta_{\varphi^* ; \varphi}^{-1} {,} \ \pi_1 w_\psi \pi_0 \delta_{\psi^* ; \psi} s e \big) c \\
&= \big( \pi_0 \pi_1 D(\varphi^*)_1 {,} \ \pi_0 \pi_0 \delta_{\varphi^* ; \varphi}^{-1} {,} \ \pi_1 w_\psi \pi_0 D(\psi \circ \psi^*)_0 e \big) c \\
&= \big( \pi_0 \pi_1 D(\varphi^*)_1 {,} \ \pi_0 \pi_0 \delta_{\varphi^* ; \varphi}^{-1} {,} \ \pi_1 w_\psi \pi_0 D(\varphi \circ \varphi^*)_0 e \big) c \\
&= \big( \pi_0 \pi_1 D(\varphi^*)_1 {,} \ \pi_0 \pi_0 \delta_{\varphi^* ; \varphi}^{-1} {,} \ \pi_1 w_\psi t_\psi D(\varphi \circ \varphi^*)_0 e \big) c \\
&= \big( \pi_0 \pi_1 D(\varphi^*)_1 {,} \ \pi_0 \pi_0 \delta_{\varphi^* ; \varphi}^{-1} {,} \ \pi_0 t_\varphi D(\varphi \circ \varphi^*)_0 e \big) c \\
&= \big( \pi_0 \pi_1 D(\varphi^*)_1 {,} \ \pi_0 \pi_0 \delta_{\varphi^* ; \varphi}^{-1} {,} \ \pi_0 \pi_0 D(\varphi \circ \varphi^*)_0 e \big) c \\
&= \big( \pi_0 \pi_1 D(\varphi^*)_1 {,} \ \pi_0 \pi_0 \delta_{\varphi^* ; \varphi}^{-1} {,} \ \pi_0 \pi_0 \delta_{\varphi^* ; \varphi}^{-1} t e \big) c \\
&= \big( \pi_0 \pi_1 D(\varphi^*)_1 {,} \ (\pi_0 \pi_0 \delta_{\varphi^* ; \varphi}^{-1} {,} \ \pi_0 \pi_0 \delta_{\varphi^* ; \varphi}^{-1} t e)c \big) c \\
&= \big( \pi_0 \pi_1 D(\varphi^*)_1 {,} \ \pi_0 \pi_0 \delta_{\varphi^* ; \varphi}^{-1} (1_{D(E)_1} {,} \ t e)c \big) c \\
&= \big( \pi_0 \pi_1 D(\varphi^*)_1 {,} \ \pi_0 \pi_0 \delta_{\varphi^* ; \varphi}^{-1} \big) c.
\end{align*}

\noi Putting all these calculations together along with the universal property of the pullback $D_{\varphi^* \circ \varphi} = D_{\psi^* \circ \psi}$ allows us to write the top composite as the following pairing map: 

\[ \ell_{\varphi ; \psi , 1} (1_{D_\psi^*} \times w_\psi) c_{\psi^* ; \psi} = \big( \pi_1 w_\psi t_\psi , (\pi_0 \pi_1 D(\varphi^*)_1 {,} \ \pi_0 \pi_0 \delta_{\varphi^* ; \varphi}^{-1} )c \big) .\]

\noi By definition of the pullback $D_\varphi \prescript{}{t_\varphi}{\times_{w_\psi t_\psi}} W_\psi$ we know that $\pi_0 t_\varphi = \pi_1 w_\psi t_\psi$ and so both components in the following pairing maps agree: 

\begin{align*}
\ell_{\varphi ; \psi , 0} (w_{\varphi^*} \times 1_{D_\varphi}) c_{\varphi^* ; \varphi} 
&= \big( \pi_0 t_\varphi , (\pi_0 \pi_1 D(\varphi^*)_1 , \ \pi_0 \pi_0 \delta_{\varphi^* ; \varphi}^{-1}) c \big) \\
&= \big( \pi_1 w_\psi t_\psi , (\pi_0 \pi_1 D(\varphi^*)_1 {,} \ \pi_0 \pi_0 \delta_{\varphi^* ; \varphi}^{-1} )c \big) \\
&= \ell_{\varphi ; \psi , 1} (1_{D_\psi^*} \times w_\psi) c_{\psi^* ; \psi}. 
\end{align*}

\noi This finally shows that the outside of the last diagram commutes and induces the cofiber lift 

\[ \begin{tikzcd}[]
D_\varphi \prescript{}{t_\varphi}{\times_{w_\psi t_\psi}} W_\psi 
\rar["\ell_{\varphi ; \psi}"] & W_\square
\end{tikzcd}. \]

\noi The universal property of the coproduct $\csp$ gives a candidate for the lift we need:

\[ \begin{tikzcd}[column sep = large]
\csp \rar[dotted, "\ell"] & W_\square 
\\
D_\varphi \prescript{}{t_\varphi}{\times_{w_\psi t_\psi}} W_\psi 
\ar[ur, "\ell_{\varphi ; \psi}"'] \uar["\iota_\varphi \times \iota_\psi^w"] & 
\end{tikzcd}.\]

\noi To see this is in fact the lift we need we need to see that the diagram 

\[ \begin{tikzcd}[]
W_\square \rar["(\pi_0 \pi_1 {,} \pi_1 \pi_1)"] & \csp 
\\
D_\varphi \prescript{}{t_\varphi}{\times_{w_\psi t_\psi}} W_\psi 
\uar["\ell_{\varphi ; \psi}"] \ar[ur, "\iota_\varphi \times \iota_\psi^w"'] & 
\end{tikzcd}\]

\noi also commutes. This can be done by considering the commuting diagram 

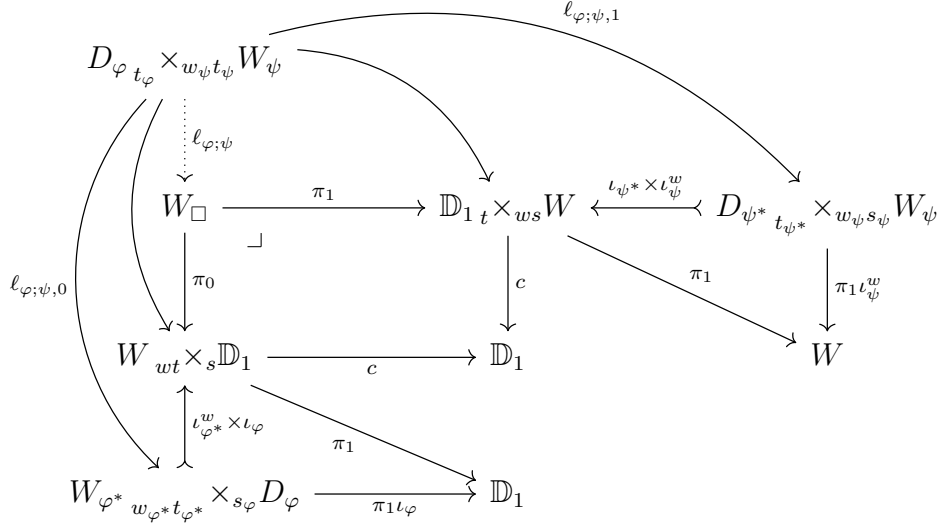
\[ \begin{tikzcd}[column sep = large, row sep = large]
D_\varphi \prescript{}{t_\varphi}{\times_{w_\psi t_\psi}} W_\psi  
\ar[d, dotted, "\ell_{\varphi ; \psi}"] 
\ar[dd, bend right, ""'] 
\ar[dr, bend left, ""] 
\ar[ddd, bend right = 45, "\ell_{\varphi ; \psi , 0}"'] 
\ar[drr, bend left, "\ell_{\varphi ; \psi , 1}"] 
&
& 
\\
 W_\square \rar["\pi_1"] \dar["\pi_0"] 
\arrow[dr, phantom, "\usebox\pullback" , very near start, color=black]
& \mD_1 \prescript{}{t}{\times_{ws}} W \dar["c"] 
\ar[dr, "\pi_1"] 
& D_{\psi^*} \prescript{}{t_{\psi^*}}{\times_{w_\psi s_\psi}} W_\psi 
\dar["\pi_1 \iota_\psi^w"] 
\lar[tail, "\iota_{\psi^*} \times \iota_{\psi}^w"'] 
\\
 W \prescript{}{wt}{\times_{s}} \mD_1 \rar["c"'] 
\ar[dr, "\pi_1"'] 
& \mD_1 
& W
\\
 W_{\varphi^*} \prescript{}{w_{\varphi^*}t_{\varphi^*}}{\times_{s_\varphi}} D_\varphi 
\uar[tail, "\iota_{\varphi^*}^w \times \iota_\varphi"'] 
\rar["\pi_1 \iota_\varphi"'] 
& \mD_1
& 
\end{tikzcd}\]

\noi and recalling that

\[ \ell_{\varphi ; \psi , 0} \pi_0 = \pi_0 \quad \text{ and } \quad \ell_{\varphi ; \psi , 1} \pi_1 = \pi_1. \]

\noi This allows us to see

\begin{align*}
\ell_{\varphi ; \psi} (\pi_0 \pi_1 , \pi_1 \pi_1) 
&= (\ell_{\varphi ; \psi} \pi_0 \pi_1 , \ell_{\varphi ; \psi}\pi_1 \pi_1) \\
&= (\ell_{\varphi ; \psi , 0} \pi_1 \iota_\varphi , \ell_{\varphi ; \psi , 1} \pi_1 \iota_\psi^w) \\
&= (\pi_0 \iota_\varphi ,  \pi_1 \iota_\psi^w) \\
&= \iota_\varphi \times \iota_\psi^w
\end{align*}

\noi and by the universal property of the coproduct $\csp$ we get the commuting diagram: 

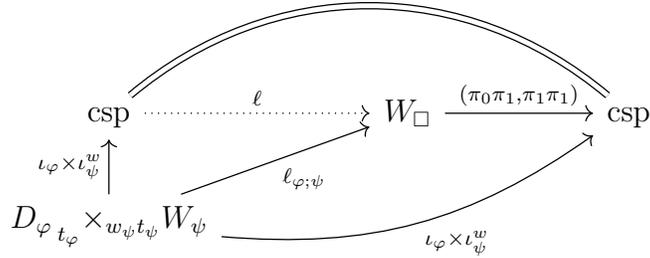
\[ \begin{tikzcd}[column sep = huge]
\csp \rar[rr, bend left = 40, equals] \rar[dotted, "\ell"] & W_\square \rar["(\pi_0 \pi_1 {,} \pi_1 \pi_1)"] & \csp 
\\
D_\varphi \prescript{}{t_\varphi}{\times_{w_\psi t_\psi}} W_\psi 
\ar[ur, "\ell_{\varphi ; \psi}"'] \uar["\iota_\varphi \times \iota_\psi^w"]
\ar[urr, bend right = 20, "\iota_\varphi \times \iota_\psi^w"'] & 
& 
\end{tikzcd}.\]\ 

\noi The top triangle in the previous diagram shows that when taking the identity map $1_{\csp} : \csp \to \csp$ as our cover, the map $\ell : \csp \to W_\square$ is precisely the lift we need. 
\end{proof}

The last condition we need to check is the internal right-cancellation property which we have referred to as `zippering.' The objects in $\cE$ representing diagrams in $\mD$ that are important to recall for this part are those of parallel pairs, $P(\mD)$, parallel pairs that are coequalized by an arrow in $W$, $\cP_{cq}(\mD)$, parallel pairs that are equalized by an arrow in $W$, $\cP_{eq}(\mD)$, and parallel pairs that are simultaneously equalized and coequalized by arrows in $W$ respectively, $\cP(\mD)$. The explicit constructions of these can be reviewed in Section \ref{S context for fractions axioms}.

As is our tradition by now, we first review the usual proof for when $\cE = \Set$ before translating it internally to a more general category $\cE$. Consider the following commuting diagram in $\mD$: 

\[ \begin{tikzcd}[column sep = large]
(A,a) 
\rar[shift left, "( \varphi {,} f)"] 
\rar[shift right, "( \psi{,}g)"'] 
& (B,D(\gamma)(c)) 
\rar["\circ" marking, "(\gamma {,} 1_{D(\gamma)(c)})"] 
& (C,c) 
\end{tikzcd}.\]

\noi By definition of composition in $\mD$, this means $\varphi \circ \gamma = \psi \circ \gamma$ in $\cA^{op}$ and the diagram 

\[ \label{dgm coequalizing parallel pair square} \begin{tikzcd}[]
D(\psi) \circ D(\gamma) (c) 
\ar[d, equals, "D(\psi)(1_{D(\gamma)(c)})"'] 
& a 
\rar["f"] 
\lar["g"'] 
& D(\varphi) \circ D(\gamma)(c) 
\dar[equals, "D(\varphi)(1_{D(\gamma)(c)})"] \\
D(\psi) \circ D(\gamma)(c) 
\dar["\delta_{\psi ; \gamma, c}^{-1} "'] 
& 
&
 D(\varphi) \circ D(\gamma)(c) 
\dar["\delta_{\varphi ; \gamma, c}^{-1}"] \\
D(\psi \circ \gamma)(c)
\rar[rr,, equals]
&
& D (\varphi \circ \gamma) (c) 
\end{tikzcd} \tag{$\star$} \]

\noi commutes in the category $D(A)$. Since $\cA$ is filtered, there exists a map $\mu : E \to A$ such that the square 

\[ \begin{tikzcd}[]
E \rar["\mu"] \dar["\mu"'] & A \dar["\psi"] \\
A \rar["\varphi"'] & B
\end{tikzcd}\]

\noi commutes in $\cA$ and so $\mu \circ \psi = \mu \circ \varphi$ in $\cA^{op}$. There is an obvious candidate equalizing arrow in $W$ for the parallel pair, $(f,\varphi)$ and $(g, \psi)$, seen in the following diagram:

\[ \label{dgm zippered/equalized dgm} \begin{tikzcd}[column sep = large] 
( E , D(\mu)(a) ) \rar["\circ" marking, "( \mu {,} 1_{D(\mu)(a)} )"] 
& ( A,a) 
\rar[shift left, "( \varphi {,} f)"] 
\rar[shift right, "( \psi{,}g)"'] 
& (B, D(\gamma)(c)) 
\end{tikzcd} \tag{$\star \star$} \]

\noi To see this diagram commutes in $\mD$ first notice that the diagram

\[ \label{dgm pasting helper dgm for the zippered/equalized dgm}
\begin{tikzcd}[column sep = small, row sep = large]
D(\mu)\circ D(\psi) \circ D(\gamma) (c) 
\dar["\delta_{\mu ; \psi, D(\gamma)(c)}^{-1}"'] 
\ar[ddr, "D(\mu) (\delta_{\psi ; \gamma , c}^{-1})"' near end] 
& D(\mu) (a) 
\lar["D(\mu)(g)"'] 
\rar["D(\mu)(f)"] 
& D(\mu)\circ D(\varphi) \circ D(\gamma) (c) 
\dar["\delta_{\mu ; \varphi, D(\gamma)(c)}^{-1}"] 
\ar[dl, "D(\mu) (\delta_{\varphi ; \gamma , c}^{-1})"']\\
D(\mu \circ \psi) \circ D(\gamma)(c) 
\dar["\delta_{\mu \circ \psi ; \gamma , c }^{-1}"'] 
& D(\mu) \circ D(\varphi \circ \gamma) (c) 
\dar[equals]
\ar[ddr, "\delta_{\mu ; \varphi \circ \gamma, c}^{-1}" near start]
& D(\mu \circ \varphi) \circ D(\gamma)(c)
\dar["\delta_{\mu \circ \varphi ; \gamma , c }^{-1}"] \\
D( (\mu \circ \psi) \circ \gamma ) 
\ar[d, equals] 
& D(\mu) \circ D(\psi \circ \gamma) (c) 
\ar[dl,"\delta_{\mu ; \psi \circ \gamma, c}^{-1}"]
&D( (\mu \circ \varphi) \circ \gamma ) 
\ar[d, equals] 
\\
D( \mu \circ ( \psi \circ \gamma) ) (c)
\rar[rr, equals] 
& 
& 
D(\mu \circ (\varphi \circ \gamma)) (c)
\end{tikzcd}
\tag{$\star^3$}\]

\noi commutes in $D(E)$. The top square commutes by functoriality of $D(\mu)$ and commutativity of diagram (\ref{dgm coequalizing parallel pair square}) above, the left and right squares commute by coherence of the structure isomorphisms for the pseudofunctor, $D$, and the bottom square commutes trivially because $\psi \circ \gamma = \varphi \circ \gamma$ in $\cA$. Then the outside of the previous diagram commutes and implies that the following diagram commutes as well: 

\[ \label{dgm components for the zippered/equalized candidate agree}
\begin{tikzcd}[column sep = small, row sep = large]
D(\mu)(a) 
\ar[ddd, bend right = 70, "1_{D(\mu)(a)} * g"' near start] 
\ar[d, equals, "1_{D(\mu)(a)}"] 
\rar[rr, equals]
&
& D(\mu)(a) 
\ar[ddd, bend left=70, "1_{D(\mu)(a)} * f"near start] 
\ar[d, equals, "1_{D(\mu)(a)}"' ] 
\\
D(\mu)(a) 
\dar["D(\mu)(g)"]
& 
& D(\mu)(a) 
\dar["D(\mu)(f)"']
\\
D(\mu)\circ D(\psi) \circ D(\gamma) (c) 
\dar["\delta_{\mu ; \psi, D(\gamma)(c)}^{-1}"] 
&
& D(\mu)\circ D(\varphi) \circ D(\gamma) (c) 
\dar["\delta_{\mu ; \varphi, D(\gamma)(c)}^{-1}"'] \\
D(\mu \circ \psi) \circ D(\gamma)(c) 
\rar[rr, equals]
\ar[d, "\delta_{\mu \circ \psi ; \gamma , c }^{-1}"] 
& 
& 
D(\mu \circ \varphi) \circ D(\gamma)(c) 
\ar[d,"\delta_{\mu \circ \psi ; \gamma , c }^{-1}"'] \\
D ( (\mu \circ \psi) \circ\gamma)(c) 
\rar[rr, equals]
& 
& D ( (\mu \circ \varphi) \circ\gamma)(c) 
\end{tikzcd}
\tag{$\star^4 $}\]

\noi This shows that the original diagram (\ref{dgm zippered/equalized dgm}) commutes in $\mD$ and proves the desired property in the case $\cE = \Set$.

\begin{lem}[\textbf{Int.Frc.4}]\label{lem IntFrc4 for internal groth}
There exists a cover $\begin{tikzcd}[]
U \rar["/" marking, "u" near start] & \cP_{cq}(\mD)
\end{tikzcd}$ and a lift 

\[ \begin{tikzcd}[]
& \cP(\mD) \dar["\pi_1"] \\
U \ar[ur, bend left, dotted, "\ell"] \rar["/" marking, "u" near start] & \cP_{cq}(\mD)
\end{tikzcd}.\]
\end{lem}
\begin{proof}

For a more general extensive category $\cE$ with a terminal object, products can be written as pullbacks over the terminal object and equalizers can then be written as pullbacks over products. In particular we have the pullback diagram

\[ \begin{tikzcd}[column sep = large]
\cP_{cq}(\mD) 
\rar[r,tail, "\iota_{cq}"] 
\dar[tail, "\iota_{cq}"'] 
\arrow[dr, phantom, "\usebox\pullback" , very near start, color=black]
& P(\mD) \prescript{}{t}{\times_{ws}} W 
\dar["\rho_1"] \\
 P(\mD) \prescript{}{t}{\times_{ws}} W 
 \rar[r,"\rho_0"'] 
 & \big( P(\mD) \prescript{}{t}{\times_{ws}} W \big) \times \mD_1 
\end{tikzcd},\]

\noi where 

\[ \rho_0 = (1_{P(\mD)} {,} (\pi_0 \pi_0 {,} \pi_1 w)c ) \]
\noi and 
\[ \rho_1 = (1_{P(\mD)} {,} (\pi_0 \pi_1 {,} \pi_1 w)c ) . \]

\noi Note that each object is a pullback of coproducts, and since $\cE$ is extensive each of these pullbacks can be expressed as a coproduct of pullbacks of their corresponding cofibers. The cofibers for the parallel pairs objects are denoted

\[ P(\mD)_{(\varphi , \psi)} = D_\varphi \prescript{}{(s,t)}{\times_{(s,t)}} D_\psi\]

\noi The corresponding pullback diagram of the cofiber corresponding to the maps $\varphi , \psi, $ and $\gamma$ in $\cA$ such that $\varphi \gamma = \psi \gamma$ is 
\[ \label{dgm P_cq(D) cofiber}
\begin{tikzcd}[]
\cP_{cq}(\mD)_{(\varphi ,\psi) ; \gamma} 
\arrow[dr, phantom, "\usebox\pullback" , very near start, color=black]
\dar["\pi_0"'] 
\rar[r,"\pi_1"] 
&
P(\mD)_{(\varphi , \psi)} \prescript{}{t}{\times_{w_\gamma s_\gamma}} W_\gamma
\ar[d,"\rho_{1 , (\varphi ; \psi) ; \gamma}"] 
\\
P(\mD)_{(\varphi , \psi)} \prescript{}{t}{\times_{w_\gamma s_\gamma}} W_\gamma
\ar[r,"\rho_{0 , (\varphi ; \psi) ; \gamma}"'] 
&
\big( P(\mD)_{(\varphi , \psi)} \prescript{}{t}{\times_{w_\gamma s_\gamma}} W_\gamma\big) \times D_{\varphi \circ \gamma}
\end{tikzcd}, \tag{$\star$} \]
\noi where 
\[ \rho_{0, (\varphi ; \psi) ; \gamma} = \big(1_{P(\mD)_{(\varphi , \psi)} \times W_\gamma} {,} (\pi_0 \pi_0 {,} \pi_1 w_\gamma)c_{\varphi ; \gamma} \big)\]
\noi and 
\[ \rho_{1, (\varphi ; \psi) ; \gamma} =\big(1_{P(\mD)_{(\varphi , \psi)} \times W_\gamma} {,} (\pi_0 \pi_1 {,} \pi_1 w_\gamma) c_{\psi ; \gamma} \big). \]

\noi Similarly we have the pullback diagram for the object of parallel pairs that are equalized by an arrow in $W$

\[ \begin{tikzcd}[]
\cP_{eq}(\mD) 
\rar[r,tail, "\iota_{eq}"] 
\dar[tail, "\iota_{eq}"'] 
\arrow[dr, phantom, "\usebox\pullback" , very near start, color=black]
& W \prescript{}{wt}{\times_{s}} P(\mD) 
\dar["\lambda_1"] \\
 W \prescript{}{wt}{\times_{s}} P(\mD) 
 \rar[r,"\lambda_0"'] 
 & \big( W \prescript{}{wt}{\times_{s}} P(\mD) \big) \times \mD_1 
\end{tikzcd},\]

\noi where 
\[ \lambda_0 = (1_{P(\mD)} {,} ( \pi_1 w {,} \pi_0 \pi_0)c ) \] 
\noi and
\[ \lambda_1 = (1_{P(\mD)} {,} (\pi_1 w {,} \pi_0 \pi_1 )c ). \]

\noi The corresponding pullback of a cofiber indexed by a maps $\mu, \varphi , $ and $\psi$ such that $\mu \varphi = \mu \psi$ in $\cA$ is 

\[ \label{dgm P_eq(D) cofiber} \begin{tikzcd}[]
\cP_{eq}(\mD)_{\mu ; (\varphi ,\psi) } 
\arrow[dr, phantom, "\usebox\pullback" , very near start, color=black]
\dar["\pi_0"'] 
\rar[r,"\pi_1"] 
& 
 W_\mu \prescript{}{w_\mu t_\mu}{\times_{s}} (P(\mD)_{(\varphi , \psi)} )
\ar[d,"\lambda_{1 , \mu ; ( \varphi , \psi) }"]
\\
 W_\mu \prescript{}{w_\mu t_\mu}{\times_{s}} (P(\mD)_{(\varphi , \psi)} )
\ar[r,"\lambda_{0 , \mu ; ( \varphi , \psi) } "'] 
&
\big( W_\mu \prescript{}{w_\mu t_\mu}{\times_{s}} P(\mD)_{(\varphi , \psi)} \big) \times D_{\mu \circ \varphi}
\end{tikzcd},\tag{$\star \star$} \]
\noi where 
\[ \lambda_{0 , \mu ; ( \varphi , \psi) } = \big(1_{W_\mu \times P(\mD)_{(\varphi , \psi)} } {,} (\pi_0 w_\mu {, } \pi_1 \pi_0)c_{\mu ; \varphi} \big)\]
\noi and 
\[ \lambda_{1 , \mu ; ( \varphi , \psi) } = \big(1_{ W_\mu \times P(\mD)_{(\varphi , \psi)}} {,} (\pi_0 w_\mu {, } \pi_1 \pi_1 ) c_{\mu ; \psi} \big). \]

\noi We use the cofibers in Diagrams (\ref{dgm P_cq(D) cofiber}) and (\ref{dgm P_eq(D) cofiber}) to translate the usual proof for when $\cE = \Set$ and then the universal property of coproducts will give us the result we want. Since $\cA$ is filtered, there exists a map $\mu : E \to A$ in $\cA$ such that the diagram

\[ \begin{tikzcd}[]
E \rar["\mu"] \dar["\mu"'] & A \dar["\psi"] \\
A \rar["\varphi"'] & B
\end{tikzcd}\]

\noi commutes in $\cA$. Picking out the arrow we need to precompose was done by taking the source of the parallel pair and applying $D(\mu)$ to it. Internally this is done at the level of cofibers by first considering the following commuting diagram, 

\[ \begin{tikzcd}[]
\cP_{cq}(\mD)_{(\varphi ,\psi) ; \gamma} 
\rar["\pi_0 \pi_0"] 
\dar[dd, "\pi_0 \pi_0"'] 
\ar[dddrr, dotted] 
\ar[ddrr, dotted, "\ell_\mu^w"] 
& D_\varphi \prescript{}{(s,t)}{\times_{(s,t)}} D_\psi 
\rar["\pi_0"] 
& D_\varphi \rar["\pi_1"] 
& D(A)_1 \dar["s"]
\\
& 
& 
& D(A)_0 \dar["D(\mu)_0"] 
\\
D_\varphi \prescript{}{(s,t)}{\times_{(s,t)}} D_\psi 
\dar["\pi_0"'] 
&
&W_\mu 
\dar["w_\mu"'] 
\rar["\pi_\mu"] 
\arrow[ddr, phantom, "\usebox\pullback" , very near start, color=black] 
& D(E)_0 
\dar["e"] 
\ar[dd, bend left = 40, equals] 
\\
D_\varphi 
\dar["\pi_1"'] 
& 
& D_\mu 
\arrow[dr, phantom, "\usebox\pullback" , very near start, color=black] 
\dar["\pi_0"'] 
\rar["\pi_1"] 
&D(E)_1 
\dar["t"] \\
D(A)_1 \rar[rr, "s"'] 
&
&D(A)_0 
\rar["D(\mu)_0"'] 
& D(A)_0
\end{tikzcd}\]

\noi The left side of the previous diagram then makes up the outside of the following pullback diagram: 

\[ \begin{tikzcd}[row sep = large]
\cP_{cq}(\mD)_{(\varphi ,\psi) ; \gamma}  
\ar[d, dotted, "\tilde{\ell}_{\mu , (\varphi , \psi)}"] 
\ar[dd, bend right = 80 , "\ell_\mu^w"'] 
\ar[dr, bend left, "\pi_0 \pi_0"] 
&&& \\
 W_\mu \prescript{}{w_\mu t_\mu}{\times_{s}} P(\mD)_{(\varphi , \psi)} 
\dar["\pi_0"'] 
\rar["\pi_1"] 
\arrow[dr, phantom, "\usebox\pullback" , very near start, color=black] 
& D_\varphi \prescript{}{(s,t)}{\times_{(s,t)}} D_\psi 
\rar["\pi_0"] 
\dar["s"'] 
& D_\varphi 
\dar["\pi_1"] \\
W_\mu 
\dar[tail, "w_\mu"'] 
\rar["w_\mu t_\mu"] 
& D(A)_0 
& D(A)_1 \lar["s"] \\
D_\mu \ar[ur, "\pi_0"'] 
& 
& 
\end{tikzcd}\]

\noi There are two ways to compose the arrows in the diagrams being represented by the previous universal map. To show they agree we show the the following diagram commutes 

\[ \begin{tikzcd}[]
\cP_{cq}(\mD)_{(\varphi ,\psi) ; \gamma}  
\rar["\tilde{\ell}_{\mu , (\varphi , \psi)}"] 
\dar["\tilde{\ell}_{\mu , (\varphi , \psi)}"'] 
& W_\mu \prescript{}{w_\mu t_\mu}{\times_{s}} P(\mD)_{(\varphi , \psi)} 
\ar[dr, bend left = 15, "(\pi_0 w_\mu {,} \pi_1 \pi_1)"] 
&
\\
W_\mu \prescript{}{w_\mu t_\mu}{\times_{s}} P(\mD)_{(\varphi , \psi)} 
\ar[dr, bend right = 15, "(\pi_0 w_\mu {,} \pi_1 \pi_0)"'] 
& 
& D_{\mu ; \psi} \dar[" c_{\mu ; \psi}"] 
\\
& D_{\mu ; \varphi}
\rar[" c_{\mu ; \varphi}"' ] 
& D_{\mu \circ \varphi}
\end{tikzcd}. \]

\noi By definition of $\tilde{\ell}_{\mu , (\varphi , \psi)}$ the first two maps on either both sides can be composed to give the top and left arrows in the following square: 

\[ \begin{tikzcd}[column sep = large, row sep = large]
\cP_{cq}(\mD)_{(\varphi ,\psi) ; \gamma}  
\rar["(\ell_{\mu}^w w_\mu {,} \pi_0 \pi_0 \pi_1)"] 
\dar["(\ell_{\mu}^w w_\mu {,} \pi_0 \pi_0 \pi_0)"'] 
& D_{\mu ; \psi} \dar[" c_{\mu ; \psi}"] 
\\
D_{\mu ; \varphi}
\rar[" c_{\mu ; \varphi}"' ] 
& D_{\mu \circ \varphi}
\end{tikzcd}\] 

\noi To see that this square commutes we use the universal property of the pullback $D_{\mu \circ \varphi}$. It will help to recall the definition of cofiber composition in $\mD$. In particular
\[ c_{\mu ; \varphi} = (\pi_1 \pi_0 , c'_{\mu ; \varphi ; \delta^{-1}} c) \quad \text{ and } \quad c_{\mu ; \psi} = (\pi_1 \pi_0 , c'_{\mu ; \psi ; \delta^{-1}} c) \]
\noi where 
\[ c'_{\mu ; \varphi ; \delta^{-1}} = ( \pi_0 \pi_1 , \pi_1 \pi_1 D(\mu)_1 , \pi_1 \pi_0 \delta_{\mu ; \varphi}^{-1} ) \]
\noi and similarly 
\[ c'_{\mu ; \psi ; \delta^{-1}} = ( \pi_0 \pi_1 , \pi_1 \pi_1 D(\mu)_1 , \pi_1 \pi_0 \delta_{\mu ; \psi}^{-1} ). \]

\noi To see both sides agree on the projection $\pi_0 : D_\mu \circ \varphi \to D(E)_0$ we can compute 

\begin{align*}
(\ell_{\mu}^w w_\mu {,} \pi_0 \pi_0 \pi_0) c_{\mu \circ \varphi} \pi_0 
&= (\ell_{\mu}^w w_\mu {,} \pi_0 \pi_0 \pi_0) \pi_1 \pi_0 \\
&= \pi_0 \pi_0 \pi_0 \pi_0
\end{align*}

\noi and 

\begin{align*}
(\ell_{\mu}^w w_\mu {,} \pi_0 \pi_0 \pi_1) c_{\mu \circ \psi} \pi_0 
&= (\ell_{\mu}^w w_\mu {,} \pi_0 \pi_0 \pi_1) \pi_1 \pi_0 \\
&= \pi_0 \pi_0 \pi_1 \pi_0 
\end{align*}

\noi and notice that that the last lines are equal by definition of the pullback $\cP_{cq}(\mD)_{(\varphi , \psi) ; \gamma)}$. To see that the other projection $\pi_1 : D_\mu \circ \varphi \to D(E)_0$ also coequalizes both sides of the square is more involving. First notice that the diagram 

\[ \label{dgm pi_0 pi_0 pi_0 = pi_0 pi_1 pi_0} \begin{tikzcd}[row sep = large]
P(\mD)_{(\varphi , \psi)} \prescript{}{t}{\times_{w_\gamma s_\gamma}} W_\gamma
\ar[d, "\pi_0"] 
\ar[dr, bend left = 20, "\pi_0 \pi_1"] 
\ar[dd, bend right = 60, "\pi_0 \pi_0"'] 
& & \\
D_\varphi \prescript{}{(s,t)}{\times_{(s,t)}} D_\psi 
\dar["\pi_0"'] 
\rar["\pi_1"] 
& D_\psi \dar["(\pi_1 s {,} \pi_0 )"'] 
\ar[dd, bend left = 80 , "\pi_0"]
& \\
 D_\varphi 
\rar["(\pi_1 s {,} \pi_0 )"'] 
\ar[dr, bend right, "\pi_0"'] 
& D(B)_0 \times D(A)_0 \ar[d, "\pi_1"] 
\\
 & D(A)_0
\end{tikzcd} \tag{A}\]

\noi commutes and precomposing with the projection 

\[ \begin{tikzcd} 
\cP_{cq}(\mD)_{(\varphi ,\psi) ; \gamma}
\rar["\pi_0"] 
& 
(D_\varphi \prescript{}{(s,t)}{\times_{(s,t)}} D_\psi ) \prescript{}{t}{\times_{w_\gamma s_\gamma}} W_\gamma
\end{tikzcd}\] 
\noi gives the last line in the following side-calculation:

\begin{align*}
\ell_\mu^w w_\mu \pi_1 
&= \ell_\mu^w \pi_\mu e \\
&= \pi_0 \pi_0 \pi_0 \pi_1 s D(\mu)_0 e \\
&= \pi_0 \pi_0 \pi_0 \pi_1 D(\mu)_1 s e .
\end{align*}

\noi We use this side calculation in the fourth equality of the following calculation: 

\begin{align*}
& \ \ \ \ 
(\ell_{\mu}^w w_\mu {,} \pi_0 \pi_0 \pi_0) c_{\mu \circ \varphi} \pi_1 \\
&= (\ell_{\mu}^w w_\mu {,} \pi_0 \pi_0 \pi_0) c'_{\mu ; \varphi ; \delta^{-1}} c \\
&= (\ell_{\mu}^w w_\mu {,} \pi_0 \pi_0 \pi_0)( \pi_0 \pi_1 , \pi_1 \pi_1 D(\mu)_1 , \pi_1 \pi_0 \delta_{\mu ; \varphi}^{-1} ) c \\
&= \big( \ell_{\mu}^w w_\mu \pi_1 , \ \pi_0 \pi_0 \pi_0 \pi_1 D(\mu)_1 , \ \pi_0 \pi_0 \pi_0 \pi_0 \delta_{\mu ; \varphi}^{-1} \big) c \\
&= \big( \pi_0 \pi_0 \pi_0 \pi_1 D(\mu)_1 s e , \ \pi_0 \pi_0 \pi_0 \pi_1 D(\mu)_1 , \ \pi_0 \pi_0 \pi_0 \pi_0 \delta_{\mu ; \varphi}^{-1} \big) c \\
&= \big( \pi_0 \pi_0 \pi_0 \pi_1 D(\mu)_1 (s e , 1) c , \ \pi_0 \pi_0 \pi_0 \pi_0 \delta_{\mu ; \varphi}^{-1} \big) c \\ 
&= \big( \pi_0 \pi_0 \pi_0 \pi_1 D(\mu)_1 , \ \pi_0 \pi_0 \pi_0 \pi_0 \delta_{\mu ; \varphi}^{-1} \big) c 
\end{align*}

\noi The last calculation we need requires a side calculation along with the coherences for the structure isomorphisms of the pseudo functor $D : \cA^{op} \to \Cat(\cE)$. We start similarly by noticing that the diagram 

\[ \label{dgm pi_0 pi_0 pi_1 s = pi_0 pi_1 pi_1 s} \begin{tikzcd}[row sep = large]
P(\mD)_{(\varphi , \psi)}  \prescript{}{t}{\times_{w_\gamma s_\gamma}} W_\gamma
\ar[d, "\pi_0"] 
\ar[dr, bend left = 20, "\pi_0 \pi_1"] 
\ar[dd, bend right = 60, "\pi_0 \pi_0"'] 
&  \\
D_\varphi \prescript{}{(s,t)}{\times_{(s,t)}} D_\psi 
\dar["\pi_0"'] 
\rar["\pi_1"] 
& D_\psi \dar["(\pi_1 s {,} \pi_0 )"'] 
\ar[dd, bend left = 65, "\pi_1 s "]
\\
D_\varphi 
\rar["(\pi_1 s {,} \pi_0 )"'] 
\ar[dr, bend right = 25, "\pi_1 s"'] 
& D(B)_0 \times D(A)_0 \ar[d, "\pi_0"] 
\\
& D(A)_0
\end{tikzcd} \tag{B} \]

\noi commutes in $\cE$ and noticing that pre-composing with the projection 
\[ \begin{tikzcd} 
\cP_{cq}(\mD)_{(\varphi ,\psi) ; \gamma}
\rar["\pi_0"] 
& 
(D_\varphi \prescript{}{(s,t)}{\times_{(s,t)}} D_\psi ) \prescript{}{t}{\times_{w_\gamma s_\gamma}} W_\gamma
\end{tikzcd}\] 
\noi gives the last line in the following side-calculation: 

\begin{align*}
\ell_\mu^w \pi_\mu e 
&= \pi_0 \pi_0 \pi_0 \pi_1 s D(\mu)_0 e \\
&= \pi_0 \pi_0 \pi_1 \pi_1 s D(\mu)_0 e \\
&= \pi_0 \pi_0 \pi_1 \pi_1 D(\mu)_1 s e . 
\end{align*}

\noi Another side-calculation we will need can be seen in the following commuting diagram:

\[ \label{dgm pi_0 pi_1 pi_0 = pi_1 w_gamma pi_0 D(gamma)_0} \begin{tikzcd}[]
& D(C)_0 \rar["D(\gamma)_0"] 
& D(B)_0 \\
& D_\gamma 
\uar["\pi_0"] 
\rar["\pi_1"] 
\arrow[ur, phantom, "\usebox\dlpullback" , very near start, color=black] 
& D(B)_1 \uar["t"'] \\
P(\mD)_{(\varphi , \psi)}  \prescript{}{t}{\times_{w_\gamma s_\gamma}} W_\gamma 
\arrow[ddrr, phantom, "\usebox\pullback" , very near start, color=black] 
\rar["\pi_1"] 
\dar["\pi_0"'] 
& 
W_\gamma 
\arrow[dr, phantom, "\usebox\pullback" , very near start, color=black] 
\arrow[ur, phantom, "\usebox\dlpullback" , very near start, color=black] 
\uar[tail, "w_\gamma"] 
\rar["\pi_\gamma"'] 
\dar[tail, "w_\gamma"'] 
& D(B)_0 
\ar[uu, bend right=40, equals]
\dar[tail, "e"] 
\uar[tail, "e"] 
\dar[dd, bend left=40, equals] \\
D_\varphi \prescript{}{(s,t)}{\times_{(s,t)}} D_\psi 
\dar["\pi_1"'] 
\ar[dr, "t"'] 
& D_\gamma 
\dar["s_\gamma"'] 
\rar["\pi_1"] 
& D(B)_1 \ar[d, "s"] \\
D_\varphi \rar["\pi_0"'] 
&D(B)_0 \rar[equals] 
& D(E)_0
\end{tikzcd} \tag{C} \]

\noi In particular we will use commutativity of the outside several times which says:
\[ \pi_0 \pi_1 \pi_0 = \pi_1 w_\gamma \pi_0 D(\gamma)_0 \]

\noi Coherence of the structure isomorphisms of the pseudofunctor $D$ says that the diagrams 

\[ 
\begin{tikzcd}[column sep = large, row sep = huge]
D(C)_0 
\dar["(\delta_{\mu ; \psi \circ \gamma} {,} \delta_{\psi ; \gamma} D(\mu)_1)"']
\rar[rr, "(\delta_{\mu \circ \psi ; \gamma} {,} D(\gamma)_0 \delta_{\psi ; \gamma)}"]
&& D(E)_2 \dar["c"] \\
D(C)_2 \rar[rr,"c"'] && D(E)_1
\end{tikzcd} \qquad \qquad \] 
\noi and 
\[
\begin{tikzcd}[column sep = large, row sep = huge]
D(C)_0 
\dar["(\delta_{\mu ; \varphi \circ \gamma} {,} \delta_{\varphi ; \gamma} D(\mu)_1)"']
\rar[rr,"(\delta_{\mu \circ \varphi ; \gamma} {,} D(\gamma)_0 \delta_{\varphi ; \gamma})"]
&& D(E)_2 \dar["c"] \\
D(C)_2 \rar[rr,"c"'] && D(E)_1
\end{tikzcd} \qquad \qquad 
 \]

\noi commute in $\cE$. Using the internal composition in $D(E)$ to pre-compose with the inverse structure isomorphism components $\delta_{\mu \circ \psi; \gamma}^{-1} : D(C)_0 \to D(E)_1$ and $\delta_{\mu \circ \varphi; \gamma}^{-1} : D(C)_0 \to D(E)_1$ respectively and then applying the identity law in $D(E)$ gives new commuting diagrams: 

\[ 
\begin{tikzcd}[column sep = huge, row sep = huge]
D(C)_0 
\dar["(\delta_{\mu \circ \psi ; \gamma}^{-1} {,} \delta_{\mu ; \psi \circ \gamma} {,} \delta_{\psi ; \gamma} D(\mu)_1)"']
\rar["D(\gamma)_0 "]
& D(E)_0 \dar["\delta_{\psi ; \gamma}"] \\
D(C)_3 \rar["c"'] & D(E)_1
\end{tikzcd} \qquad \qquad \] 
\noi and 
\[
\begin{tikzcd}[column sep = huge, row sep = huge]
D(C)_0 
\dar["(\delta_{\mu \circ \varphi ; \gamma}^{-1} {,} \delta_{\mu ; \varphi \circ \gamma} {,} \delta_{\varphi ; \gamma} D(\mu)_1)"']
\rar["D(\gamma)_0 "]
& D(E)_0 \dar["\delta_{\varphi ; \gamma}"] \\
D(C)_3 \rar["c"'] & D(E)_1
\end{tikzcd} \qquad \qquad\]

\noi Taking inverses in $D(E)$ then gives the commuting diagrams,

\[ \begin{tikzcd}[column sep = huge, row sep = huge]
D(C)_0 
\rar["D(\gamma)_0"]
\dar["(\delta_{\psi ; \gamma}^{-1} D(\mu)_1 {,} \delta_{\mu ; \psi \circ \gamma}^{-1} {,} \delta_{\mu \circ \psi ; \gamma} ) "']
& D(B)_0 \dar["\delta_{\mu ; \psi}^{-1}"] \\
D(C)_3 \rar["c"'] & D(E)_1
\end{tikzcd} \qquad \qquad\]
\[
\begin{tikzcd}[column sep = huge, row sep = huge]
D(C)_0 
\rar["D(\gamma)_0"]
\dar["(\delta_{\varphi ; \gamma}^{-1} D(\mu)_1 {,} \delta_{\mu ; \varphi \circ \gamma}^{-1} {,} \delta_{\mu \circ \varphi ; \gamma} ) "']
& D(B)_0 \dar["\delta_{\mu ; \varphi}^{-1}"] \\
D(C)_3 \rar["c"'] & D(E)_1
\end{tikzcd} \qquad \qquad \]

\noi which we will use in the calculation(s) below. The first of the latest side-calculations along with associativity of composition and the identity law for composition in $D(E)$ allows us to see 

\begin{align*}
& \ \ \ \ 
(\ell_{\mu}^w w_\mu {,} \pi_0 \pi_0 \pi_1) c_{\mu \circ \psi} \pi_1 \\
&= (\ell_{\mu}^w w_\mu {,} \pi_0 \pi_0 \pi_1) c'_{\mu ; \psi ; \delta^{-1}} c \\
&= (\ell_{\mu}^w w_\mu {,} \pi_0 \pi_0 \pi_1)( \pi_0 \pi_1 , \pi_1 \pi_1 D(\mu)_1 , \pi_1 \pi_0 \delta_{\mu ; \psi}^{-1} ) c \\
&= \big( \ell_{\mu}^w w_\mu \pi_1 , \pi_0 \pi_0 \pi_1\pi_1 D(\mu)_1 , \pi_0 \pi_0 \pi_1\pi_0 \delta_{\mu ; \psi}^{-1} \big) c \\
&= \big(\pi_0 \pi_0 \pi_1 \pi_1 D(\mu)_1 s e ,  \pi_0 \pi_0 \pi_1\pi_1 D(\mu)_1 , \pi_0 \pi_0 \pi_1\pi_0 \delta_{\mu ; \psi}^{-1} \big) c \\
&= \big( (\pi_0 \pi_0 \pi_1 \pi_1 D(\mu)_1 s e ,  \pi_0 \pi_0 \pi_1\pi_1 D(\mu)_1) c , \pi_0 \pi_0 \pi_1\pi_0 \delta_{\mu ; \psi}^{-1} \big) c \\
&= \big( \pi_0 \pi_0 \pi_1 \pi_1 D(\mu)_1 (s e , 1_{D(E)_1})c , \pi_0 \pi_0 \pi_1\pi_0 \delta_{\mu ; \psi}^{-1} \big) c \\
&= \big( \pi_0 \pi_0 \pi_1 \pi_1 D(\mu)_1 , \pi_0 \pi_0 \pi_1\pi_0 \delta_{\mu ; \psi}^{-1} \big) c 
\end{align*}

\noi The second of the latest side-calculations along with the left square deduced from the coherence diagrams allow us to see

\begin{align*} 
& \ \ \ \ \big( \pi_0 \pi_0 \pi_1 \pi_1 D(\mu)_1 , \pi_0 \pi_0 \pi_1\pi_0 \delta_{\mu ; \psi}^{-1} \big) c \\
&= \big( \pi_0 \pi_0 \pi_1 \pi_1 D(\mu)_1 , \pi_0 \pi_1 w_\gamma \pi_0 D(\gamma)_0 \delta_{\mu ; \psi}^{-1} \big) c \\
&= \big( \pi_0 \pi_0 \pi_1 \pi_1 D(\mu)_1 , \pi_0 \pi_1 w_\gamma \pi_0 (\delta_{\psi ; \gamma} D(\mu)_1 {,} \delta_{\mu ; \psi \circ \gamma}^{-1} {,} \delta_{\mu \circ \psi ; \gamma} )c \big) c
\end{align*}

\noi The next side calculation shows how we can replace the last line above. The first line below comes from the definition of $\cP_{cq}(\mD)_{(\varphi ; \psi) ; \gamma}$. The second and third lines follow from the definition of cofiber composition and the fourth line is a standard computation using the calculus of pairing maps by the universal property of pullbacks. The fifth, sixth, and seventh lines are consequences of the definition of $W_\gamma$ and the eighth line follows by definition of $\delta_{\psi ; \gamma}^{-1}$. The ninth line comes from associativity of composition, and in the tenth line we apply the identity law for composition in $D(A)$. 

\begin{align*}
& \ \ \ \ 
\pi_0 (\pi_0 \pi_0 , \pi_1 w_\gamma) c_{\varphi ; \gamma} \pi_1 D(\mu)_1 \\
&=\pi_1 (\pi_0 \pi_1 , \pi_1 w_\gamma) c_{\psi ; \gamma} \pi_1 D(\mu)_1 \\
&= \pi_1 (\pi_0 \pi_1 , \pi_1 w_\gamma) c'_{\psi ; \gamma; \delta^{-1}} c D(\mu)_1 \\
&= \pi_1 (\pi_0 \pi_1 , \pi_1 w_\gamma) (\pi_0 \pi_1 , \pi_1 \pi_1 D(\psi)_1 , \pi_1 \pi_0 \delta_{\psi ; \gamma}^{-1} ) c D(\mu)_1 \\
&= \big( \pi_1 \pi_0 \pi_1 \pi_1 , \pi_1 \pi_1 w_\gamma \pi_1 D(\psi)_1 , \pi_1 \pi_1 w_\gamma \pi_0 \delta_{\psi ; \gamma}^{-1} \big) c D(\mu)_1 \\
&= \big( \pi_1 \pi_0 \pi_1 \pi_1 , \pi_1 \pi_1 \pi_\gamma e D(\psi)_1 , \pi_1 \pi_1 w_\gamma \pi_0 \delta_{\psi ; \gamma}^{-1} \big) c D(\mu)_1 \\
&= \big( \pi_1 \pi_0 \pi_1 \pi_1 , \pi_1 \pi_1 \pi_\gamma D(\psi)_0 e , \pi_1 \pi_1 w_\gamma \pi0 \delta_{\psi ; \gamma}^{-1} \big) c D(\mu)_1 \\
&= \big( \pi_1 \pi_0 \pi_1 \pi_1 , \pi_1 \pi_1 w_\gamma \pi_0 D(\gamma)_0 D(\psi)_0 e , \pi_1 \pi_1 w_\gamma \pi_0 \delta_{\psi ; \gamma}^{-1} \big) c D(\mu)_1 \\
&= \big( \pi_1 \pi_0 \pi_1 \pi_1 , \pi_1 \pi_1 w_\gamma \pi_0 \delta_{\psi ; \gamma}^{-1} s e , \pi_1 \pi_1 w_\gamma \pi_0 \delta_{\psi ; \gamma}^{-1} \big) c D(\mu)_1 \\
&= \big( \pi_1 \pi_0 \pi_1 \pi_1 , \pi_1 \pi_1 w_\gamma \pi_0 \delta_{\psi ; \gamma}^{-1} (s e ,1_{D(A)_1}) c \big) c D(\mu)_1 \\ 
&= \big( \pi_1 \pi_0 \pi_1 \pi_1 , \pi_1 \pi_1 w_\gamma \pi_0 \delta_{\psi ; \gamma}^{-1} \big) c D(\mu)_1 \\ 
&= \big( \pi_1 \pi_0 \pi_1 \pi_1 D(\mu)_1 , \pi_1 \pi_1 w_\gamma \pi_0 \delta_{\psi ; \gamma}^{-1} D(\mu)_1 \big) c \\ 
\end{align*}

\noi Functoriality of $D(\mu)$ gives the final line in the computation above. Expanding the composition in the last line of the previous calculation and recalling that that the diagram 
\[ \begin{tikzcd}[]
\cP_{cq}(\mD)_{(\varphi , \psi) ; \gamma} \rar[r, shift left, "\pi_0 \pi_0"] \rar[shift right, "\pi_1 \pi_0"'] & \big( P(\mD)_{(\varphi , \psi)} \prescript{}{t}{\times_{w_\gamma s_\gamma}} W_\gamma\big) \times D_{\varphi \circ \gamma}
\end{tikzcd}\]

\noi commutes by definition of the pullback $\cP_{cq}(\mD)_{(\varphi , \psi) ; \gamma} $ allows us to see that up to this point we have: 

\begin{align*} 
(\ell_{\mu}^w w_\mu {,} \pi_0 \pi_0 \pi_1) c_{\mu \circ \psi} \pi_1 
&=\big( \pi_0 \pi_0 \pi_1 \pi_1 D(\mu)_1 , \pi_0 \pi_0 \pi_1\pi_0 \delta_{\mu ; \psi}^{-1} \big) c \\
&= \big(\pi_0 (\pi_0 \pi_0 , \pi_1 w_\gamma) c_{\varphi ; \gamma} \pi_1 D(\mu)_1 , \pi_1 \pi_1 w_\gamma \pi_0 \delta_{\mu ; \psi \circ \gamma}^{-1} \delta_{\mu \circ \psi ; \gamma} \big) c 
\end{align*}

\noi A similar side calculation to the last one, where we can cancel composition with an identity map in the middle, shows

\begin{align*}
& \ \ \ \ 
\pi_0 (\pi_0 \pi_0 , \pi_1 w_\gamma) c_{\varphi ; \gamma} \pi_1 D(\mu)_1\\ 
&= \pi_0 (\pi_0 \pi_0 , \pi_1 w_\gamma) c'_{\varphi ; \gamma; \delta^{-1}} c D(\mu)_1 \\
&= \pi_0 (\pi_0 \pi_0 , \pi_1 w_\gamma)(\pi_0 \pi_1 , \pi_1 \pi_1 D(\varphi)_1 , \pi_1 \pi_0 \delta_{\varphi ; \gamma}^{-1} ) c D(\mu)_1 \\
&= \big( \pi_0 \pi_0 \pi_0 \pi_1 ,  \pi_0 \pi_1 w_\gamma \pi_1 D(\varphi)_1 , \pi_0 \pi_1 w_\gamma \pi_0 \delta_{\varphi ; \gamma}^{-1} \big) c D(\mu)_1 \\
&= \big( \pi_0 \pi_0 \pi_0 \pi_1 ,  \pi_0 \pi_1 \pi_\gamma e D(\varphi)_1 , \pi_0 \pi_1 w_\gamma \pi_0 \delta_{\varphi ; \gamma}^{-1} \big) c D(\mu)_1 \\
&\vdots \\
&= \big( \pi_0 \pi_0 \pi_0 \pi_1 D(\mu)_1 , \pi_0 \pi_1 w_\gamma \pi_0 \delta_{\varphi ; \gamma}^{-1} D(\mu)_1 \big) c. 
\end{align*}

\noi Note that the parallel arrows 

\[ \begin{tikzcd}[column sep = large]
\cP_{cq}(\mD)_{(\varphi , \psi) ; \gamma} \rar[shift left, "\pi_1 \pi_1"] \rar[shift right, "\pi_0 \pi_1"'] 
& W_\gamma
\end{tikzcd}\]

\noi are equal by definition of the pullback $\cP_{cq}(\mD)_{(\varphi , \psi) ; \gamma}$. Since $\psi \circ \gamma = \varphi \circ \gamma$ and $\mu \circ \psi = \mu \circ \varphi$, composition in $D(E)$ is associative, the commuting square(s) deduced from the coherence of the structure isomorphisms for $D$, and by diagrams (\ref{dgm pi_0 pi_1 pi_0 = pi_1 w_gamma pi_0 D(gamma)_0}) and (\ref{dgm pi_0 pi_0 pi_0 = pi_0 pi_1 pi_0}), we have that 

\begin{align*} 
& \ \ \ \ (\ell_{\mu}^w w_\mu {,} \pi_0 \pi_0 \pi_1) c_{\mu \circ \psi} \pi_1  \\
&= \big(\pi_0 \pi_0 \pi_0 \pi_1 D(\mu)_1 , \pi_0 \pi_1 w_\gamma \pi_0 \delta_{\varphi ; \gamma}^{-1} D(\mu)_1 , \pi_1 \pi_1 w_\gamma \pi_0 \delta_{\mu ; \varphi \circ \gamma}^{-1} \delta_{\mu \circ \varphi ; \gamma} \big) c \\
&= \big(\pi_0 \pi_0 \pi_0 \pi_1 D(\mu)_1 , \pi_0 \pi_1 w_\gamma \pi_0 \delta_{\varphi ; \gamma}^{-1} D(\mu)_1 , \pi_0 \pi_1 w_\gamma \pi_0 \delta_{\mu ; \varphi \circ \gamma}^{-1} \delta_{\mu \circ \varphi ; \gamma} \big) c \\
&= \big(\pi_0 \pi_0 \pi_0 \pi_1 D(\mu)_1 , \pi_0 \pi_1 w_\gamma \pi_0 ( \delta_{\varphi ; \gamma}^{-1} D(\mu)_1 , \delta_{\mu ; \varphi \circ \gamma}^{-1} \delta_{\mu \circ \varphi ; \gamma})c \big) c \\
&= \big(\pi_0 \pi_0 \pi_0 \pi_1 D(\mu)_1 , \pi_0 \pi_1 w_\gamma \pi_0 D(\gamma)_0 \delta_{\mu ; \varphi}^{-1} \big) c \\
&= \big(\pi_0 \pi_0 \pi_0 \pi_1 D(\mu)_1 , \pi_0 \pi_0 \pi_1 \pi_0 \delta_{\mu ; \varphi}^{-1} \big) c\\
&= \big(\pi_0 \pi_0 \pi_0 \pi_1 D(\mu)_1 , \pi_0 \pi_0 \pi_0 \pi_0 \delta_{\mu ; \varphi}^{-1} \big) c\\
&= (\ell_{\mu}^w w_\mu {,} \pi_0 \pi_0 \pi_0) c_{\mu \circ \psi} \pi_1  
\end{align*}

\noi It follows that 

\[ (\ell_{\mu}^w w_\mu {,} \pi_0 \pi_0 \pi_0) c_{\mu \circ \psi} =(\ell_{\mu}^w w_\mu {,} \pi_0 \pi_0 \pi_1) c_{\mu \circ \psi} \]

\noi so there exists a unique map $\ell_{\mu ; (\varphi , \psi)} : \cP_{cq}(\mD)_{(\varphi , \psi) ; \gamma} \to \cP_{eq}(\mD)_{\mu ; (\varphi ,\psi) }$ at the cofiber level: 

\[ \begin{tikzcd}[column sep = tiny, row sep = large]
&\cP_{cq}(\mD)_{(\varphi , \psi) ; \gamma} 
\dar[dotted, "\ell_{\mu ; (\varphi ; \psi)} "] 
\ar[ddr, bend left= 15, "\tilde{\ell}_{\mu ; (\varphi , \psi)}"] 
\ar[ddl, bend right=15, "\tilde{\ell}_{\mu ; (\varphi , \psi)}"'] 
& 
\\
&\cP_{eq}(\mD)_{\mu ; (\varphi ,\psi) } 
\arrow[dd, phantom, "\usebox\vpullback" , very near start, color=black]
\ar[dl,"\pi_0"'] 
\ar[dr,"\pi_1"] 
& \\
W_\mu \prescript{}{w_\mu t_\mu}{\times_{s}} P(\mD)_{(\varphi , \psi)} 
\ar[dr,"\big(1_{P(\mD)_{(\varphi , \psi)}} {,} (\pi_0 w_\mu {, } \pi_1 \pi_0)c_{\mu ; \varphi} \big)"'] 
&
&
 W_\mu \prescript{}{w_\mu t_\mu}{\times_{s}} P(\mD)_{(\varphi , \psi)} 
\ar[dl,"\big(1_{P(\mD)_{(\varphi , \psi)}} {,} (\pi_0 w_\mu {, } \pi_1 \pi_1 ) c_{\mu ; \psi} \big)"] \\
&
\big( W_\mu \prescript{}{w_\mu t_\mu}{\times_{s}} P(\mD)_{(\varphi , \psi)} \big) \times D_{\mu \circ \varphi}
&
\end{tikzcd}\]

\noi Since this is true for arbitrary parallel pairs $(\varphi , \psi)$ in $\cA$ and an arbitrary $\gamma$ in $\cA$ which coequalizes them, the universal property of coproducts induces unique maps $\tilde{\ell}_0 : \cP_{cq}(\mD) \to W \prescript{}{wt}{\times_{s}} P(\mD) $ and $\ell_0 : \cP_{cq}(\mD) \to \cP_{eq}(\mD)$ such that the diagram 

\[ \begin{tikzcd}[column sep = large]
\cP_{cq}(\mD) 
\ar[dd, bend right = 60, "\tilde{\ell}_0"'] 
\ar[drr, bend left = 20, "\tilde{\ell}_0"] 
\ar[d, dotted, "\ell_0"] &&& \\
\cP_{eq}(\mD) 
\rar[rr,tail, "\iota_{eq}"] 
\dar[tail, "\iota_{eq}"'] 
\arrow[drr, phantom, "\usebox\pullback" , very near start, color=black]
&
& W \prescript{}{wt}{\times_{s}} P(\mD) 
\dar["(1_{P(\mD)} {,} (\pi_1 w {,} \pi_0 \pi_1 )c )"] \\
 W \prescript{}{wt}{\times_{s}} P(\mD) 
 \rar[rr,"(1_{P(\mD)} {,} ( \pi_1 w {,} \pi_0 \pi_0)c )"'] 
 &
 & \big( W \prescript{}{wt}{\times_{s}} P(\mD) \big) \times \mD_1 
\end{tikzcd}\]

\noi commutes in $\cE$. The lift we need is then given by the following pullback diagram 

\[ \begin{tikzcd}[]
 \cP_{cq}(\mD) 
 \ar[ddr, bend right, "\ell_0"'] 
 \ar[drr, bend left, equals] 
 \ar[dr, dotted, "\ell"] & &\\
& \cP(\mD) 
\dar["\pi_0"'] 
\rar["\pi_1"] 
\arrow[dr, phantom, "\usebox\pullback" , very near start, color=black]
& \cP_{cq}(\mD) \dar["\iota_{cq} \pi_0"] \\
& \cP_{eq}(\mD) \rar["\iota_{eq} \pi_1"'] & P(\mD) 
\end{tikzcd}\]

\noi where the upper triangle shows we can assume the cover to be the identity map $1_{\cP_{cq}(\mD)} : \cP_{cq}(\mD) \to \cP_{cq}(\mD)$. The outside of the pullback diagram commutes because by definition of $\ell_0$ we have

\[ \ell_0 \iota_{eq} \pi_1 = \ell_0^w \pi_1 \]

\noi and then the following diagram 

\[ \begin{tikzcd}[]
 \cP_{cq}(\mD) \rar[dotted, "\tilde{\ell}_0"] 
 \rar[rr, dotted, bend left, "\iota_{cq} \pi_0 "] 
 & W \prescript{}{wt}{\times_{s}} P(\mD) \rar[dotted, "\pi_1"] 
 & P(\mD) \\
 \cP_{cq}(\mD)_{(\varphi , \psi) ; \gamma} 
 \uar["\iota_{(\varphi , \psi) ; \gamma}"] 
 \rar["\tilde{\ell}_{\mu; (\varphi , \psi) }"']
 \rar[rr, bend right, "\pi_0 \pi_0 "'] 
 & W_\mu \prescript{}{w_\mu t_\mu}{\times_{s}} P(\mD)_{(\varphi , \psi)}
 \uar["\iota_\mu^w \times \iota_{(\varphi , \psi)}"']
 \rar["\pi_1"'] 
 & D_\varphi \prescript{}{(s,t)}{\times_{(s,t)}} D_\psi 
 \uar["\iota_{(\varphi , \psi)}"']
\end{tikzcd}\]

\noi commutes in $\cE$. 
\end{proof}

\noi The preceding lemmas in this section come together to show that the object of the convenient cleavage of cartesian arrows we consider for the internal category of elements can be formally inverted to give an internal category of (right) fractions, $\mDW$.

\begin{prop}\label{thm IntGroth satisfies IntFracxioms} 
Let $D : \cA^{op} \to \Cat(\cE)$ be a pseudofunctor such that $\cE$ is a candidate context for internal fractions and admits an internal category of elements, $\mD$, which is a candidate for internal fractions. Let $W = \coprod_{\varphi \in \cA_1} W_\varphi$ be the object of the canonical cleavage of the cartesian arrows we defined at the beginning of this section. Then $(\mD , W) $ admits an internal category of fractions. 
\end{prop}
\begin{proof}
Lemmas \ref{lem IntFrc1 for internal groth}, \ref{lem IntFrc2 for internal groth}, \ref{lem IntFrc3 for internal groth}, and \ref{lem IntFrc4 for internal groth} come together to show that the Internal Fractions Axioms of Definition~\ref{def Internal Fractions Axioms}
are satisfied and the result follows by Definition~\ref{def (mC,W) admits paths of fractions}. 
\end{proof}

\subsection{Pseudocolimits of Small Filtered Diagrams of Internal Categories }\label{SS IntFrac of IntGroth is PseudoColim}

The crux of our main theorem is an observation that in the correspondence between oplax natural transformations $D \implies \Delta \mX$ and internal functors $\mD \to \mX$ for an arbitrary internal category $\mX$ established in Theorem~\ref{thm IntGroth is OpLaxColim}, the components of the natural transformations factor through the family of arrows $w : W \to \mD_1$ that get inverted by the internal localization $L : \mD \to \mD[W^{-1}]$. The oplax natural transformations $D \implies \Delta$ are required here since we are dealing with a contravariant pseudofunctor and constructing a category of right fractions. 

\noi Recall that there is a canonical oplax natural transformation $D \implies \Delta \mD$ whose components are internal functors $\ell_B : D(B) \to \mD$ defined by: 

\[ \begin{tikzcd}[column sep = large]
D(B)_0 \rar[rr, tail, "(\ell_B)_0 = \iota_B"]& &\mD_0
\end{tikzcd}\qquad \qquad \qquad 
 \begin{tikzcd}[column sep = large]
D(B)_1 
\ar[drr, "(\ell_B)_1"'] 
\rar[rr, "\big(t {,} (1_{D(B)_1} {,} t \delta_B^{-1} ) c \big) "] 
&& D_{1_A} \dar[tail, "\iota_{1_A}"] \\
&& \mD_1
\end{tikzcd}.\]

\noi For each $\varphi : A \to B$ in $\cA$, the internal transformation , $\ell_\varphi : D(\varphi) \ell_A \implies \ell_B$ is defined by its components, $\ell_\varphi \iota_\varphi : D(B)_0 \to \mD_1$, which factor through $W_\varphi$ as a consequence of the commuting diagram:

\[ \begin{tikzcd}[]
D(B)_0 
\ar[dddr, bend right, equals] 
\ar[ddr, bend right, dotted, "\ell_\varphi"] 
\ar[dr, dotted, "\ell_\varphi^w"] 
\ar[drr, bend left, "D(\varphi)_0"] 
& & \\
& W_\varphi 
\dar[tail, "w_\varphi"'] 
\rar["\pi_\varphi"] 
\arrow[dr, phantom, "\usebox\pullback" , very near start, color=black] 
& D(A)_0 
\dar[tail, "e"] 
\dar[dd, bend left = 40, equals] \\
& D_\varphi 
\dar["\pi_0"'] 
\rar["\pi_1"] 
\arrow[dr, phantom, "\usebox\pullback" , very near start, color=black] 
& D(A)_1 
\dar["t"] \\
& D(B)_0 
\rar["D(\varphi)_0"'] 
& D(A)_0
\end{tikzcd}\]

\noi More precisely, the components of the natural transformation, $\ell_\varphi$, are picked out by the composite $\ell_\varphi \iota{\varphi} : D(B)_0 \to \mD_1$, which represents the arrows in the component $D_\varphi$ which are given by applying the identity structure map, $e : D(A)_0 \to D(A)_1$, after applying $D(\varphi)_0 : D(B)_0 \to D(A)_0$. 

\begin{defn}\label{def induced internal functor from oplax nat transfmn}
For an arbitrary oplax natural transformation $x : D \implies \Delta \mX$, the induced internal functor $\theta_x : \mD \to \mX$ is defined on components and then induced by the universal property of the coproduct as follows: 

\[ \begin{tikzcd}[column sep = large]
\mD_0 \rar[dotted, "(\theta_x)_0"] & \mX_0 \\
D(B)_0 \uar[tail, "\iota_B"] \ar[ur, "(x_B)_0"'] 
\end{tikzcd} \qquad \qquad \qquad 
\begin{tikzcd}[column sep = large]
\mD_1 \rar[rr, dotted, "(\theta_x)_1"] && \mX_1 \\
D_\varphi 
\uar[tail, "\iota_\varphi"] 
\rar[rr, "\big(\pi_1(x_A)_1 {,} \pi_0 x_\varphi \big)"'] 
&& \mX_2 
\uar["c"'] 
\end{tikzcd}\]
\end{defn}

\noi The subtle difference here from the induced internal functor in Section \ref{SS UP for 1-cells of laxtransfm} is the map 

\[\begin{tikzcd}[column sep = large]
D_\varphi 
\rar[rr, "\big(\pi_1(x_A)_1 {,} \pi_0 x_\varphi \big)"] 
&& \mX_2 
\end{tikzcd}\]

\noi which twists the order of composition in $\mX$ to account for working with a contravariant functor the covariant used in Lemma~\ref{thm IntGroth is OpLaxColim}. Next we review how every internal functor $\mD \to \mX$ induces an oplax natural transformation by whiskering. This is the same as in Section \ref{SS UP of 2-cells for laxtransfm} but we restate it for our reader's convenience and the fact that we're working with a contravariant functor and oplax transformations.

\begin{defn}\label{def induced oplax nat transfmn from internal functor} 
For an arbitrary internal functor $F : \mD \to \mX$, the induced oplax natural transformation $F^*: D \implies \Delta \mX$ has components that are internal functors $F^*_B : D(B) \to \mX$ defined by post-composition 

\[ \begin{tikzcd}[column sep = large ]
D(B) \ar[dr, "F^*"'] \rar["\ell_B"] & \mD \dar["F"] \\
& \mX
\end{tikzcd} \] 

\noi For each $\varphi : A \to B$ in $\cA$, the induced transformation $F^*_\varphi : D(\varphi) F^* \implies F^*$ is defined by whiskering. More precisely, the components are given by post-composing the components of $\ell_\varphi$ with $F_1$: 

\[ \begin{tikzcd}[column sep = large, ]
D(B)_0 
\rar["\ell_\varphi"] 
\ar[dr, "F^*_\varphi"'] 
&\mD_1 \dar["F_1"] \\
& \mX_1 
\end{tikzcd}\]
\end{defn}

\noi A similar proof to the one in Proposition~\ref{Prop 1-cell correspondence} shows the assignments in Definitions \ref{def induced internal functor from oplax nat transfmn} and \ref{def induced oplax nat transfmn from internal functor} are inverses. The following Lemma shows that the induced internal functor in Definition~\ref{def induced internal functor from oplax nat transfmn} inverts the cartesian arrows, $w : W \to \mD_1$.

\begin{lem}\label{Lem induced internal functor inverts object of cartesian arrows}
If $x : D \implies \Delta \mX$ is a natural isomorphism, then the induced internal functor $\theta_x : \mD \to \mX$ inverts the family of cartesian arrows, $w : W \to \mD$, as in Definition~\ref{def internal functor inverts a map}. 
\end{lem}
\begin{proof}
By Definition~\ref{def internal functor inverts a map} we need to show that the composite 

\[ \begin{tikzcd}[]
W \rar["w"] \ar[dr, "\theta_x(W)"'] & \mD_1 \dar["(\theta_x)_1"] \\
& \mX_1
\end{tikzcd}\]

\noi is invertible in $\mX$ as in Definition~\ref{def invertibility of internal arrows}. It suffices to produce a map $\theta_x(w)^{-1} : W \to \mX_1$ such that such that 

\[ ( \theta_x(w)^{-1} , \theta_x(w) ) c = \theta_x(w) t e , \quad (\theta_x(w) , \theta_x(w)^{-1}) c = \theta_x(w) s e. \]

Since $x : D \implies \Delta \mX$ is a natural isomorphism, for each $B \in \cA_0$ the components $x_\varphi : D(B)_0 \to \mX_1$ are invertible in $\mX$. By Definition~\ref{def invertibility of internal arrows} there exists $x_\varphi^{-1} : D(B)_0 \to \mX_1$ such that
\[ x_{\varphi}^{-1} s = x_\varphi t , \quad x_{\varphi}^{-1} t = x_\varphi s \]
\noi and 
\[ ( x_{\varphi}^{-1} , x_\varphi) c = x_\varphi t e , \quad (x_\varphi, x_{\varphi}^{-1}) c = x_\varphi s e .\]

\noi commutes in $\cE$. Define a candidate inverse for $w \theta_x : W \to X_1$ (with respect to composition in $\mX$) to be the universal map induced by the family of composites,

\[ \begin{tikzcd}[]
W_\varphi \rar[tail, "w_\varphi"] 
\ar[drr, "\theta_x(W_\varphi)^{-1}"'] 
& D_\varphi \rar["\pi_0"] 
& D(B)_0 \dar["x_\varphi^{-1}"] \\
& & \mX_1
\end{tikzcd},\]

\noi in $\cE$ for each arrow $\varphi : A \to B$ in $\cA$. Note that by definition of $x_\varphi$ and $\theta_x(W)^{-1}$

\begin{align*}
   \theta_x(W_\varphi)^{-1} s
   &= w_\varphi \pi_0 x_\varphi^{-1} s \\
   &= w_\varphi \pi_0 x_\varphi t \\
   &= w_\varphi ( \pi_1 (x_A)_1 , \pi_0 x_\varphi )c t\\
   &= w_\varphi (\theta_x)_1 t \\
   &= \iota_\varphi^w w (\theta_x)_1 \\
   &= \iota_\varphi^w \theta_x(W).
\end{align*}  

\noi For the rest of our argument we need a nicer characterization of $\theta_x(W) = w (\theta)_1 : W \to $ the following commuting diagram 

\[ \begin{tikzcd}[]
W_\varphi 
\arrow[dr, phantom, "\usebox\pullback" , very near start, color=black] 
\dar[tail, "w_\varphi"'] 
\rar["\pi_\varphi"] 
& D(A)_0 \dar[tail, "e"] \\
D_\varphi 
\arrow[dr, phantom, "\usebox\pullback" , very near start, color=black] 
\dar["\pi_0"'] 
\rar["\pi_1"] 
& D(A)_1 
\dar["t"] 
\rar["(x_A)_1"] 
& \mX_1 \dar["t"] \\
D(B)_0 
\dar["x_\varphi"'] 
\rar["D(\varphi)_0"'] 
&D(A)_0 \rar["(x_A)_0"] 
& \mX_0 \\
\mX_1 \ar[urr, bend right = 20, "s"'] 
\end{tikzcd}\]

\noi in $\cE$, where the pullback squares commute by definition, the bottom right square commutes by functoriality, and the bottom part of the diagram commutes by definition of the natural transformation $x_\varphi : D(\varphi) x_A \implies x_B$. Use the previous diagram's commutativity along with the identity law for internal composition in $\mX$ to compute the composite

\begin{align*}
w_\varphi \big( \pi_1(x_A)_1 {,} \pi_0 x_\varphi) c 
&= \big( w_\varphi \pi_1(x_A)_1 {,} w_\varphi \pi_0 x_\varphi) c \\
&= \big( \pi_\varphi e(x_A)_1 {,} w_\varphi \pi_0 x_\varphi) c \\
&= \big( \pi_\varphi e t e(x_A)_1 {,} w_\varphi \pi_0 x_\varphi) c \\
&= \big( w_\varphi \pi_0 D(\varphi)_0 e(x_A)_1 {,} w_\varphi \pi_0 x_\varphi) c \\
&= \big( w_\varphi \pi_0 D(\varphi)_0 (x_A)_0 e {,} w_\varphi \pi_0 x_\varphi) c \\
&= \big( w_\varphi \pi_0 x_\varphi s e {,} w_\varphi \pi_0 x_\varphi) c \\
&=w_\varphi \pi_0 x_\varphi (s e {,} 1_{\mX_1}) c \\
&= w_\varphi \pi_0 x_\varphi .
\end{align*}

\noi Now we can see the target of $\theta_x(W_\varphi)^{-1}$ is the source of $ \iota_\varphi^w \theta_x(W_\varphi) : W_\varphi \to \mX_0$,

\begin{align*}
  \theta_x(W_\varphi)^{-1} t 
  &= w_\varphi \pi_0 x_\varphi^{-1} t \\
  &= w_\varphi \pi_0 x_\varphi s \\
  &= w_\varphi ( \pi_1 (x_A)_1 , \pi_0 x_\varphi )c s\\
  &= w_\varphi \iota_\varphi (\theta_x)_1 s \\
  &= \iota_\varphi^w w (\theta_x)_1 s \\
  &= \iota_\varphi^w \theta_x(W) s,
\end{align*}

\noi and get a convenient description of the cofibers of the map $w (\theta_x)_1 : W \to \mX_1$ as shown in the commuting diagram: 

\[ \begin{tikzcd}[column sep = huge, row sep = huge]
W \rar[ dotted,"w"] 
& \mD_1
\rar[dotted, "(\theta_x)_1"] 
& \mX_1 \ar[dr, equals] & \\
W_\varphi \rar[tail, "w_\varphi"'] 
\uar[tail, "\iota_\varphi^w"] 
\ar[dr, tail, "w_\varphi"'] 
&D_\varphi 
\uar[tail, "\iota_\varphi"] 
\rar["\big( \pi_1 (x_A)_1 {,} \pi_0 x_\varphi \big)"'] 
& \mX_2 \rar["c"'] 
 \uar["c"'] 
& \mX_1
\\
&D_\varphi \rar["\pi_0"'] 
& D(B)_0 \ar[ur, bend right, "x_\varphi"'] 
\end{tikzcd}. \]

\noi Now we can compute

\begin{align*}
\iota_\varphi^w (\theta_x(W)^{-1} , \theta_x(W))c 
&= ( \iota_\varphi^w \theta_x(W)^{-1} , \iota_\varphi^w \theta_x(W) ) c \\
&= ( w_\varphi \pi_0 x_\varphi^{-1} , \iota_\varphi^w w (\theta_x)_1 ) c \\
&= ( w_\varphi \pi_0 x_\varphi^{-1} , w_\varphi \iota_\varphi (\theta_x)_1 ) c \\
&= ( w_\varphi \pi_0 x_\varphi^{-1} , w_\varphi (\pi_1 (x_A)_1 , \pi_0 x_\varphi)c ) c \\
&= ( w_\varphi \pi_0 x_\varphi^{-1} , w_\varphi (\pi_1 (x_A)_1 , \pi_0 x_\varphi)c ) c \\
&= ( w_\varphi \pi_0 x_\varphi^{-1} , w_\varphi \pi_0 x_\varphi) c \\
&= w_\varphi \pi_0 (x_\varphi^{-1} , x_\varphi) c \\
&= w_\varphi \pi_0 x_\varphi t e \\
&= \theta_x(W_\varphi)^{-1} s e \\
&= \iota_\varphi^w \theta_x(W) t e 
\end{align*}

\noi as well as 

\begin{align*}
\iota_\varphi^w (\theta_x(W) , \theta_x(W)^{-1} )c 
&= ( \iota_\varphi^w \theta_x(W) , \iota_\varphi^w \theta_x(W)^{-1} ) c \\
&= ( \iota_\varphi^w w (\theta_x)_1 , w_\varphi \pi_0 x_\varphi^{-1} ) c \\
&= ( w_\varphi \iota_\varphi (\theta_x)_1 , w_\varphi \pi_0 x_\varphi^{-1}) c \\
&= ( w_\varphi (\pi_1 (x_A)_1 , \pi_0 x_\varphi)c , w_\varphi \pi_0 x_\varphi^{-1} ) c \\
&= ( w_\varphi (\pi_1 (x_A)_1 , \pi_0 x_\varphi)c , w_\varphi \pi_0 x_\varphi^{-1} ) c \\
&= ( w_\varphi \pi_0 x_\varphi , w_\varphi \pi_0 x_\varphi^{-1}) c \\
&= w_\varphi \pi_0 ( x_\varphi , x_\varphi^{-1}) c \\
&= w_\varphi \pi_0 x_\varphi s e \\
&= \theta_x(W_\varphi)^{-1} t e \\
&= \iota_\varphi^w \theta_x(W)^{-1} t e \\
&= \iota_\varphi^w \theta_x(W) s e 
\end{align*}

\noi and by the universal property of the coproduct $W$ we get 

\[ (\theta_x(W)^{-1} , \theta_x(W))c = \theta_x(W) t e \]
\noi and 
\[(\theta_x(W) , \theta_x(W)^{-1} )c = \theta_x(W) s e . \]
\noi It follows that $\theta_x(W)$ has an inverse in $X$ and that $\theta_x$ inverts $w : W \to \mD_1$. 
\end{proof}

\noi The next lemma we will need in our main result will help us establish that every natural transformation induced by an internal functor that inverts $w : W \to \mD_1$ under the equivalence in Theorem~\ref{thm IntGroth is OpLaxColim} is a pseudonatural transformation. This is done by seeing that for each $\varphi : A \to B$ in $\cA$ the internal natural transformation obtained by whiskering $\ell_\varphi$ with $L$, 

\[ \begin{tikzcd}[column sep = huge, row sep = huge]
D(B) \dar["D(\varphi)"'] \ar[dr, bend left, "\ell_B"{name = lb}] & &\\
D(A) \rar["\ell_A"'] 
\arrow[Rightarrow, shorten <=10pt, shorten >=10pt, from = 2-1, to = lb, "\ell_\varphi"'] & \mD \rar["L"'] & \mDW
\end{tikzcd},\]

\noi gives a natural isomorphism. The key observation to make here is that the 2-cells from the canonical oplax natural transformation $\ell : D \implies \mD$ have components that factor through $w : W \to \mD_1$ so that whiskering with the internal localization functor $L : \mD \to \mDW$ inverts them to give a natural isomorphism after whiskering. 

\begin{lem}\label{lem whiskering with localization}
For each $\varphi : A \to B$ in $\cA$, the internal natural transformation 

\[ \ell_\varphi L : D(\varphi) \ell_A L \implies \ell_B L \]

\noi given by whiskering,

\[\begin{tikzcd}[]
D(B) 
\dar["D(\varphi)"'] 
\rar["\ell_B"] 
& \mD \rar["L"] \dar[equals]
& \mDW \dar[equals] \\
D(A) \rar["\ell_A"'] 
\arrow[ur, shorten <=10pt,shorten >=10pt,Rightarrow, " \ell_\varphi"] 
& \mD \rar["L"'] 
\arrow[ur, shorten <=10pt,shorten >=10pt,Rightarrow, " 1_L"] 
& \mDW 
\end{tikzcd}\quad = \quad \begin{tikzcd}[]
D(B) 
\dar["D(\varphi)"'] 
\rar["\ell_B L"] 
& \mDW \dar[equals] \\
D(A) \rar["\ell_A L"'] 
\arrow[ur, shorten <=10pt,shorten >=10pt,Rightarrow, " \ell_\varphi L"] 
& \mDW 
\end{tikzcd},\]

\noi is an isomorphism. 
\end{lem}
\begin{proof}
Recall that the internal localization functor, $L$, inverts $w : W \to \mD_1$. In particular $(wL_1)^{-1} = (1,wse)q : W \to \mDW_1$ is an inverse of $wL1 : W \to \mDW_1$ by Proposition~\ref{prop int loc functor inverts W}. Also recall from Definition~\ref{SS canonical transformation 1-cells} that the components of the natural transformation $\ell_\varphi$ are given by $\ell_\varphi \iota_\varphi : D(B)_0 \to \mD_1$ where $\ell_\varphi = (1_{D(B)_0} , D(\varphi)_0 e)$ is the unique pairing map induced by the universal property of the pullback $D_\varphi$. Notice this map factors through $w : W \to \mD_1$ since $\ell_\varphi : D(B)_0 \to D_\varphi$ factors through $W_\varphi$ via the composite:

\[ \begin{tikzcd}
D(B)_0 \ar[dr, "\ell_\varphi"'] \rar["\ell_\varphi^w"] & W_\varphi \dar[tail, "w_\varphi"] \\
& D_\varphi 
\end{tikzcd}.\]

\noi Then since $w_\varphi \iota_\varphi = \iota_\varphi^w w : W_\varphi \to \mD_1$, we have the following commuting diagram:

\[\label{dgm L inverts ell_varphi L} \begin{tikzcd}[column sep = large, row sep = large]
D(B)_0 \ar[dr, "\ell_\varphi \iota_\varphi"] 
\dar[dotted, "\ell_\varphi^w \iota_\varphi^w "'] 
\rar["\ell_\varphi^w w_\varphi"] 
& D_\varphi \ar[d, "\iota_\varphi"] \\
W 
\rar["w"] 
& \mD_1. 
\end{tikzcd} \tag{$\star$}\]

\noi Abusing notation by reusing the label $\ell_\varphi L$ we can see that the composite

\[ \begin{tikzcd}[column sep = large, row sep = large]
D(B)_0 \rar["\ell_\varphi^w w_\varphi"]
\ar[drr, "\ell_\varphi L"'] & D_\varphi \rar["\iota_\varphi"] & \mD_1 \dar["L_1"] \\
& & \mDW_1 
\end{tikzcd}\]

\noi represents the components of the transformation $\ell_\varphi L : D(\varphi) \ell_A L \implies \ell_B L$. Diagram (\ref{dgm L inverts ell_varphi L}) implies these components are all invertible in $\mDW$ via an inverse given by the composite

\[ \begin{tikzcd}[column sep = large, row sep = large]
D(B)_0 \rar["\ell_\varphi^w"] \ar[drr, "(\ell_\varphi L)^{-1}"']
& W_\varphi \rar[, "\iota_\varphi^w"] 
& W \dar["(wL_1)^{-1}"] \\
& & \mDW_1
\end{tikzcd}\]

\noi in $\cE$. To see this is really an inverse we can use the definitions and diagrams above in the proof of this lemma along with the fact that $(w L_1)^{-1}$ is an inverse of $w L_1$ to see 

\begin{align*}
  ( \ell_\varphi L , (\ell_\varphi L)^{-1} ) c 
  &= (\ell_\varphi^w w_\varphi \iota_\varphi L_1 , \ell_\varphi^w \iota_\varphi^w (w L_1)^{-1})c \\
  &= (\ell_\varphi^w \iota_\varphi^w L_1 , \ell_\varphi^w \iota_\varphi^w (w L_1)^{-1})c \\
  &= \ell_\varphi^w \iota_\varphi^w (L_1 , (w L_1)^{-1})c \\
  &= \ell_\varphi^w \iota_\varphi^w L_1 s e \\
  &= \ell_\varphi^w w_\varphi \iota_\varphi L_1 s e\\
  &= (\ell_\varphi L) s e
\end{align*}
\noi and a similar proof shows
\begin{align*}
  ( (\ell_\varphi L)^{-1} , \ell_\varphi L ) c 
  = (\ell_\varphi L) t e .
\end{align*}

\noi It follows that the whiskered transformation 

\[ \ell_\varphi L : D(\varphi) \ell_A L \implies \ell_B L\]

\noi has an inverse, $(\ell_\varphi L)^{-1}$ and is an internal natural isomorphism between internal functors for each $\varphi : A \to B$ in $\cA$. 
\end{proof}

\begin{lem}\label{lem internal functors inverting W correspond to pseudonatural transoformations via the grothendieck construction equivalence}
Every internal functor $F : \mD \to \mX$ that inverts $w : W \to \mD_1$ corresponds to a pseudonatural transformation $D \implies \Delta \mX$ via the oplax version of the isomorphism of categories in Theorem~\ref{thm IntGroth is OpLaxColim}
\end{lem}
\begin{proof}
Suppose $F : \mD \to \mX$ is an internal functor that inverts $w : W \to \mD_1$. Then by the universal property of the internal localization, in Proposition~\ref{prop UPIntFrac 1-cell corresondence}, there exists a unique $[F] : \mDW \to \mX$ such that $L [F] = F$. The natural transformation corresponding to $F$ under the contravariant version of the isomorphism of categories in Theorem~\ref{thm IntGroth is OpLaxColim} is obtained by whiskering 

\[ \begin{tikzcd}[column sep = huge, row sep = huge]
D(B) \dar["D(\varphi)"'] \ar[dr, bend left, "\ell_B"{name = lb}] & &\\
D(A) \rar["\ell_A"'] 
\arrow[Rightarrow, shorten <=10pt, shorten >=10pt, from = 2-1, to = lb, "\ell_\varphi"'] & \mD \rar["F"] & \mX
\end{tikzcd}\]

\noi for each $\varphi : A \to B$ in $\cA$. Since $F = L [F]$ this whiskering can be done in two steps. Starting with 

\[ \begin{tikzcd}[column sep = huge, row sep = huge]
D(B) \dar["D(\varphi)"'] \ar[dr, bend left, "\ell_B"{name = lb}] & &\\
D(A) \rar["\ell_A"'] 
\arrow[Rightarrow, shorten <=10pt, shorten >=10pt, from = 2-1, to = lb, "\ell_\varphi"'] & \mD \rar["L"] & \mDW \rar["\lbrack F\rbrack"] & \mX
\end{tikzcd},\]

\noi we can use Lemma~\ref{lem whiskering with localization} to get a natural isomomorphism 

\[ \begin{tikzcd}[column sep = huge, row sep = huge]
D(B) \dar["D(\varphi)"'] \ar[dr, bend left, "\ell_B L "{name = lb}] & &\\
D(A) \rar["\ell_A L "'] 
\arrow[Rightarrow, shorten <=10pt, shorten >=10pt, from = 2-1, to = lb, "\ell_\varphi L"'] & \mDW \rar["\lbrack F \rbrack "] & \mX
\end{tikzcd}\]

\noi and then we can whisker once more to get 

\[ \begin{tikzcd}[column sep = large, row sep = large]
D(B) \dar["D(\varphi)"'] \ar[dr, bend left, "\ell_B L \lbrack F \rbrack "{name = lb}] & &\\
D(A) \rar["\ell_A L \lbrack F \rbrack "']
\arrow[Rightarrow, shorten <=10pt, shorten >=10pt, from = 2-1, to = lb, "\ell_\varphi L \lbrack F \rbrack "' near start] & \mX
\end{tikzcd} = \qquad 
\begin{tikzcd}[column sep = large, row sep = large]
D(B) \dar["D(\varphi)"'] \ar[dr, bend left, "\ell_B F "{name = lb}] & &\\
D(A) \rar["\ell_A F "'] 
\arrow[Rightarrow, shorten <=10pt, shorten >=10pt, from = 2-1, to = lb, "\ell_\varphi F "'near start] & \mX
\end{tikzcd}.
\]

\noi Recall that the components of $\ell_\varphi$ are $\ell_\varphi \iota_\varphi : D(B)_0 \to \mD_1$ and notice that the components of $\ell_\varphi F$ are precisely 

\[ \begin{tikzcd}[column sep = large, row sep = large]
D(B)_0 \rar["\ell_\varphi \iota_\varphi"] \rar[rr, bend left, "\ell_\varphi L"] \ar[dr, "\ell_\varphi F"'] & \mD_1 \dar["F_1"] \rar["L_1"] & \mDW_1 \ar[dl,"\lbrack F \rbrack_1"] \\
& \mX_1 & 
\end{tikzcd}.\]

\noi Lemma~\ref{lem whiskering with localization} shows that $\ell_\varphi L : D(B)_0 \to \mDW_1$ is invertible in $\mDW$ and since $[F] : \mDW \to \mX$ is an internal functor, the total composite in the diagram above is invertible in $\mX$ by Lemma~\ref{lem internal functors preserve internal inverses}. 
\end{proof}

\noi The next two lemmas allow us to contextualize the previous two lemmas more precisely in terms of the oplax version of the isomorphism of categories in Theorem~\ref{thm IntGroth is OpLaxColim}. This helps us avoid many explicit but unnecessary details in the proof of the isomorphism of categories in our main result. 

\begin{lem}\label{lem pseudo nat transfms are fully faithful subcat} 
The underlying-structure functor 

\[ \begin{tikzcd}
\lbrack D , \Delta \mX\rbrack _{ps} \rar["U"] & \lbrack D , \Delta \mX\rbrack_{\opl}
\end{tikzcd} \]

\noi is fully faithful. 
\end{lem}
\begin{proof}
If two modifications, $\mu , \nu : \alpha \to \beta$ between pseudonatural transformations $\alpha, \beta : D \implies \Delta \mX$ are equal after forgetting the additional pseudonaturality structure then they are the same modification by definition. This implies $U$ is faithful. It is clearly full because any modification between pseudonatural transformations is what it is. 
\end{proof}

\begin{lem}\label{lem internal functors that invert W are fully faithful subcat} 
The underlying-structure functor 

\[ \begin{tikzcd}
\lbrack \mD , \mX\rbrack^{\cE}_W \rar["U'"] & \lbrack \mD , \mX\rbrack^{\cE}
\end{tikzcd} \]

\noi is fully faithful. 
\end{lem}
\begin{proof}
Any two natural transformations $\alpha , \beta : f \implies g$ between internal functors $f,g : \mD \to \mX$ that invert $w : W \to \mD_1$ which become equal after forgetting that $f$ and $g$ invert $w : W \to \mD_1$ must be the same natural transformations by definition. This implies $U'$ is faithful. Any natural transformation between internal functors $\mD \to \mX$ that invert $w : W \to \mD_1$ is precisely that, so $U'$ is also clearly full. 
\end{proof}

\noi The previous four lemmas come together in the following lemma which does most of the work for the proof of our main theorem which follows immediately after. 

\begin{lem}\label{lem pseudonatural transformations correspond to W-inverting internal functors}
There is an isomorphism of categories 

\[ [D , \Delta \mX]_{ps} \cong [\mD , \mX]^{\cE}_W \]

\noi between the category of pseudonatural transformations $D \implies \Delta \mX$ and their modifications; and internal functors, $\mD\to \mX$, that invert the cartesian arrows, $w: W \to \mD$, and their natural transformations.
\end{lem}
\begin{proof}
By Lemma~\ref{Lem induced internal functor inverts object of cartesian arrows}, every pseudonatural transformation $D \implies \Delta \mX$ induces an internal functor $\mD \to \mX$ that inverts $w : W \to \mD$ by the following composition of functors

\[ \begin{tikzcd}[column sep = large, row sep = large]
\lbrack D , \Delta \mX \rbrack_{ps} \dar["U"'] \ar[dr, dotted] & \\
\lbrack D , \Delta \mX \rbrack_{\mbox{\scriptsize op$\ell$}} \rar["\cong"'] & \lbrack \mD , \mX \rbrack^{\cE}
\end{tikzcd}\]

\noi where the bottom isomorphism of categories is the oplax version of Theorem~\ref{thm IntGroth is OpLaxColim}. By Lemma~\ref{lem whiskering with localization}, the composite 

\[ \begin{tikzcd}[column sep = large, row sep = large]
& \lbrack \mD , \mX\rbrack^{\cE}_W \dar[" U' "] \ar[dl, dotted] \\
\lbrack D , \Delta \mX \rbrack_{\opl} \ar[from = r , " \cong "] & \lbrack \mD , \mX \rbrack^{\cE}
\end{tikzcd} \]

\noi factors through $\lbrack D , \Delta \mX \rbrack_{ps}$. By Lemmas \ref{lem pseudo nat transfms are fully faithful subcat} and \ref{lem internal functors that invert W are fully faithful subcat}, we know $[D , \Delta \mX]_{ps}$ and $[\mD , \mX]^{\cE}_W$ are both fully faithful subcategories of $[D , \Delta \mX]_{\opl}$ and $[\mD , \mX]^{\cE}$ respectively so the isomorphism of categories in Theorem~\ref{thm IntGroth is OpLaxColim} restricts to an isomorphism between these subcategories. 
\end{proof}

We can finally state and prove the main theorem of this paper. 

\begin{thm}\label{thm internal localization is a pseudocolim}
Let $\cA$ be a cofiltered category and let $\cE$ admit an internal Grothendieck constructoin, $\mD$, for the pseudofunctor $D : \cA^{op} \to \Cat(\cE)$. If $(\mD, W)$ admits an internal category of fractions, $\mDW$, then $\mDW$ is the pseudocolimit of $D: \cA^{op} \to \Cat(\cE)$. 
\end{thm}
\begin{proof}
Under the given assumptions, we can apply Lemmas \ref{lem pseudonatural transformations correspond to W-inverting internal functors} and Theorem~\ref{thm 2-dimensional universal property of localization} to get a chain of isomorphisms (of categories) which can be composed to prove the result.

\[ [D , \Delta \mX]_{ps} \cong [\mD , \mX]^{\cE}_W \cong [\mDW , \mX]^{\cE}\]
\end{proof}

\section{Denouement}

Having given contexts for an internal category of elements and an internal category of (right) fractions, we have implicitly described a context for computing (op)lax colimits of small diagrams of internal categories and another for computing pseudocolimits of small filtered diagrams of internal categories. The purpose of doing this was to isolate and better understand the categorical constructions that are used when working in the context of $\Set$, with diagrams of small categories, and to give a new formalism for gluing constructions in categories of internal categories. 

For the internal category of elements of a pseudofunctor we required specific pullbacks along source (or target) maps of our internal categories, and certain disjoint coproducts that commute with these pullbacks. Any extensive category that has these pullbacks will satisfy these conditions, for example $\Set, \Cat,$ and $\Top$ all admit internal category of elementss for small diagrams of their internal categories this way. We state these conditions so carefully in order to include other possible examples of larger categories which may not be extensive overall, or which may not contain all pullbacks, but which have enough coproducts and pullbacks that behave well with one another to allow this construction to take place. For example, small diagrams of Lie groupoids whose source and target maps are surjective submersions. 

The internal category of fractions construction requires a collection of pullbacks and equalizers in order to define the objects involved in the internal description of (a weakened version) of the (right) fractions axioms, as well as the relations and quotient objects which are required to define objects of paths of arrows with an appropriate universal property.

An interesting part of our main result, Theorem~\ref{thm internal localization is a pseudocolim}, is that some of the internal fractions structure from the definition of the Internal Fractions Axioms (Definition~\ref{def Internal Fractions Axioms}) becomes trivial when the internal category being considered is an internal category of elements implying a kind of discreteness. In this case we proved a formal gluing construction for these internal categories by showing that the resulting internal category of fractions is the pseudocolimit of the original diagram. 

Exploring more examples of diagrams of internal categories arising in contexts that satisfy our conditions is ongoing work. We have plenty extensive categories that allow for an internal category of elements, and it would be interesting to find an example where the entire category is not extensive, but the pullbacks and coproducts we have interact nicely for other reasons. Increasing the dimension of the indexing, ambient, and internal categories with the goal of studying higher categorical colimit constructions is the topic of my PhD thesis proposal at Dalhousie.

\appendix
\section{Internal Category of Elements}

This section of the appendix contains technical lemmas used in Chapter \ref{Chaper Internal Grothendieck}. 

\subsection{Associativity of Composition}

This first lemma we need states that the source and target of a composite coincides with the source and target of the first and second map in the composite respectively.

\begin{lem}\label{Lem source and target of component composite}
For any composable pair $(\varphi, \psi) \in \cA(W,X) \times \cA(X, Y)$ in $\cA$ we have that `the source (target) of the composite is the source (target) of the first (second) map (respectively).'

\[ c_{\varphi ; \psi} t_{\varphi \psi} = p_1 t_{\psi} \qquad , \qquad c_{\varphi ; \psi} s_{\varphi \psi} = p_0 s_\varphi . \]
\end{lem}
\begin{proof}
By definition of $t_{\varphi \psi}$, $s_{\varphi \psi}$, $c'_{\delta ; \varphi ; \psi}$, $c'_{\varphi ; \psi}$, and $c_{\varphi ; \psi}$. 

\begin{align*}
c_{\varphi ; \psi} t_{\varphi \psi} 
&= c_{ \varphi ; \psi} \pi_1 t \iota_Y & c_{\varphi ; \psi} s_{\varphi \psi}& =c_{\varphi ; \psi} \pi_0 \iota_W\\
&= c_{ \varphi ; \psi} \pi_1 t \iota_Y &&= p_0 \pi_0 \iota_W  \\
&= c'_{\delta ; \varphi ; \psi} c t \iota_Y &&= p_0 s_\varphi \\
&= c'_{\delta ; \varphi ; \psi} q_{12} q_1 t \iota_Y && \\
&= c'_{\varphi ; \psi} q_1 t \iota_Y &&\\
&= p_1 \pi_1 t \iota_Y && \\
&= p_1 t_\psi &&
\end{align*}
\end{proof}\

\noi This next two lemmas contain calculations that show how to compute cofiber composition of the first and last two maps of a composable triple in the internal category of fractions, $\mD$. These results are used to prove associativity of composition in $\mD$ in Proposition~\ref{Prop composition is associative - appendix}.

\begin{lem}\label{Lem composing first two maps in a triple in Groth Const}
For any $\varphi, \psi, \gamma$ composable in $\cA$

\[ c'_{01} c'_{\delta ; \varphi \psi ;\gamma } = ( p_{01} p_0 \pi_0 \delta_{\varphi \psi ; \gamma} , p_{01} c'_{\delta ; \varphi ; \psi} c D(\gamma)_1 , p_{12} p_1 \pi_1 ) \]

\noi where

\[
c'_{\delta ; \varphi ; \psi} c D(\gamma)_1 = ( p_0 \pi_0 \delta_{\varphi ; \psi} D(\gamma)_1 , p_0 \pi_1 D(\psi)_1 D(\gamma)_1 , p_1 \pi_1 D(\gamma)_1 ) c
\]
\end{lem}

\begin{proof}
By the universal property of the relevant pullback of `composable-triples,' it suffices to check that

\begin{align*} 
c'_{01} c'_{\delta ; \varphi \psi ;\gamma } q_{01} q_0 
&= c'_{01} c'_{\delta ; (\varphi \psi ;\gamma) } q_0 \\
&= c'_{01} p_0 \pi_0 \delta_{\varphi \psi ; \gamma} \\
&= p_{01} c_{\varphi ; \psi} \pi_0 \delta_{\varphi \psi ; \gamma} \\
&= p_{01} p_0 \pi_0 \delta_{\varphi \psi ; \gamma} ,
\end{align*}

\begin{align*}
c'_{01} c'_{\delta ; \varphi \psi ;\gamma } q_{01} q_1 
&= c'_{01} c'_{\delta ; (\varphi \psi ;\gamma) } q_1 \\
&= c'_{01} p_0 \pi_1 D(\gamma)_1 \\
&= p_{01} c_{\varphi ; \psi} \pi_1 D(\gamma)_1 \\
&= p_{01} c'_{\delta ; \varphi ; \psi} c D(\gamma)_1,
\end{align*}

\noi and 

\begin{align*}
c'_{01} c'_{\delta ; \varphi \psi ;\gamma } q_{12} q_1
&= c'_{01} c'_{\delta ; \varphi \psi ;\gamma } q_{12} q_1\\
&= c'_{01} c'_{\varphi \psi ; \gamma} q_1 \\
&= c'_{01} p_1 \pi_1 \\
&= p_{12} p_1 \pi_1 
\end{align*}

\noi respectively. By functoriality of $D(\gamma)$ and associativity of composition the middle component in that triple composite factors

\begin{align*} 
c'_{\delta ; \varphi ; \psi} c D(\gamma)_1 
&= c'_{\delta ; \varphi ; \psi} ( q_{01}q_1 D(\gamma)_1, q_{01} q_1 D(\gamma)_1 , q_{12} q_1 D(\gamma)_1 ) c \\
&= ( p_0 \pi_0 \delta_{\varphi ; \psi} D(\gamma)_1 , p_0 \pi_1 D(\psi)_1 D(\gamma)_1 , p_1 \pi_1 D(\gamma)_1 ) c.
\end{align*}

\end{proof}\

\begin{lem}\label{Lem composing last two maps of a triple in the Groth Const}
For any $\varphi, \psi, \gamma$ composable in $\cA$

\[ c'_{12} c'_{\delta ; \varphi ; \psi \gamma} = ( p_{01} p_0 \pi_0 \delta_{\varphi ; \psi \gamma} , p_{01} p_0\pi_1 D(\psi \gamma)_1 , p_{12} c'_{\delta ; \varphi ; \psi \gamma} c ) .\]
\end{lem}
\begin{proof}
By the universal property of the relvant `composable-triples' pullback, it suffices to check that 

\begin{align*}
c'_{12} c'_{\delta ; \varphi ; \psi \gamma} q_{01} q_0 
&=c'_{12} c'_{\delta ; \varphi ; \psi \gamma} q_{01} q_0 \\
&=c'_{12} c'_{\delta ; (\varphi ; \psi \gamma)} q_0 \\
&= c'_{12} p_0 \pi_0 \delta_{\varphi ; \psi \gamma} \\
&= c'_{12} p_0 \pi_0 \delta_{\varphi ; \psi \gamma} \\
&= p_{01} p_0 \pi_0 \delta_{\varphi ; \psi \gamma} ,
\end{align*}

\begin{align*}
c'_{12} c'_{\delta ; \varphi ; \psi \gamma} q_{01} q_1
&=c'_{12} c'_{\delta ; \varphi ; \psi \gamma} q_{01} q_1 \\
&=c'_{12} c'_{\delta ; (\varphi ; \psi \gamma)} q_1 \\
&= c'_{12} p_0 \pi_1 D(\psi \gamma)_1 \\
&= p_{01} p_0\pi_1 D(\psi \gamma)_1,
\end{align*}

\noi and 

\begin{align*}
c'_{12} c'_{\delta ; \varphi ; \psi \gamma} q_{12} q_1 
&= c'_{12} c'_{\varphi ; \psi \gamma} q_1 \\
&= c'_{12} p_1 \pi_1 \\
&= p_{12} c_{\psi ; \gamma} \pi_1 \\
&= p_{12} c'_{\delta ; \psi ; \gamma} c.
\end{align*}
\end{proof}\

\noi The last lemma provides some calculations using the internal coherence for the composition natural isomorphisms associated to the pseudofunctor along with naturality and functoriality. In the classical Grothendieck construction (when $\cE = \Set$), Lemma~\ref{Lem Coherence + Naturality computation for associativity} internally encodes the intermediate step 

\[ \delta_{\varphi \psi ; \gamma , a} D(\gamma)(\delta_{\varphi ; \psi , a}) D(\gamma)( D(\psi)( D(\varphi)(f) ) ) = \delta_{\varphi; \psi \gamma ,a } D(\psi \gamma) (D(\varphi)(f) ) \delta_{\psi ; \gamma} \]\

\noi for each $a \in D(A)_0$ when proving associativity of composition. 

\begin{lem}\label{Lem Coherence + Naturality computation for associativity}
For any $\varphi, \psi, \gamma$ composable in $\cA$

\[ ( \pi_0 \delta_{\varphi \psi ; \gamma} , \pi_0 \delta_{\varphi ; \psi} D(\gamma)_1 , \pi_1 (D(\psi)_1 D(\gamma)_1)) c = ( \pi_0 \delta_{\varphi ; \psi \gamma} , \pi_1 D(\psi \gamma)_1, \pi_1 t \delta_{\psi ; \gamma} ) c\] 
\end{lem}
\begin{proof}
By coherence of composition isomorphisms for the original pseudofunctor, $D$, we have that 

\[ ( \delta_{\varphi ; \psi \gamma} , D(\varphi)_0 \delta_{\psi ; \gamma} ) c = ( \delta_{\varphi \psi ; \gamma} , \delta_{\varphi ; \psi} D(\gamma)_1 ) c, \]

\noi and by definition of the natural isomorphism $\delta_{\psi ; \gamma} : D(\psi \gamma) \implies D(\psi) D(\gamma)$ 

\[ ( D(\psi \gamma)_1 , t \delta_{\psi ; \gamma} ) c = ( s \delta_{\psi ; \gamma} , D(\psi)_1 D(\gamma)_1 ) c . \]\

\noi Putting coherence and naturality together with associativity we get the following equality of triple composites 

\begin{align*}
 & \ \ \ \ ( \pi_0 \delta_{\varphi \psi ; \gamma} , \pi_0 \delta_{\varphi ; \psi} D(\gamma)_1 , \pi_1 (D(\psi)_1 D(\gamma)_1)) c\\
 &= ( \pi_0 ( \delta_{\varphi \psi ; \gamma} , \delta_{\varphi ; \psi} D(\gamma)_1 ) c , (D(\psi)_1 D(\gamma)_1)) c\\
 &= ( \pi_0 ( \delta_{\varphi ; \psi \gamma} , D(\varphi)_0 \delta_{\psi ; \gamma} ) c , \pi_1 D(\psi)_1 D(\gamma)_1 ) c\\ 
 &= ( \pi_0 \delta_{\varphi ; \psi \gamma} , ( \pi_0 D(\varphi)_0 \delta_{\psi ; \gamma} , \pi_1 D(\psi)_1 D(\gamma)_1 ) c ) c\\ 
  &= ( \pi_0 \delta_{\varphi ; \psi \gamma} , ( \pi_1 s \delta_{\psi ; \gamma} , \pi_1 D(\psi)_1 D(\gamma)_1 ) c ) c\\ 
   &= ( \pi_0 \delta_{\varphi ; \psi \gamma} , \pi_1 ( s \delta_{\psi ; \gamma} , D(\psi)_1 D(\gamma)_1 ) c ) c\\ 
   &= ( \pi_0 \delta_{\varphi ; \psi \gamma} , \pi_1 ( D(\psi \gamma)_1 , t \delta_{\psi ; \gamma} ) c ) c\\ 
    &= ( \pi_0 \delta_{\varphi ; \psi \gamma} , \pi_1 D(\psi \gamma)_1, \pi_1 t \delta_{\psi ; \gamma} ) c\\ 
 \end{align*}
\end{proof}\

\noi We're now ready to prove associativity of composition in $\mD$. 

\begin{prop}\label{Prop composition is associative - appendix}
Composition in $\mD$ is associative. 
\end{prop}
\begin{proof}
The object of composable triples is given by pulling back the pullback projections $\rho_0 , \rho_1 : \mD_2 \to \mD_1$. Denote its canonical maps by $\rho_0'$ and $\rho_1'$ respectively. By Definition~\ref{def E admits an internal category of elements of D} we have


\[ \mD_3 \cong \coprod_{(\varphi, \psi, \gamma) \in \cA_3} D_{\varphi ; \psi; \gamma} \]\

\noi where $D_{\varphi ; \psi ; \gamma}$ is given by pulling back the projections $p_1 : D_{\varphi; \psi} \to D_\psi$ and $p_0 : D_{\psi ; \gamma}\to D_\psi$. More precisely, for any composable triple 

\[ \begin{tikzcd}[]
W \rar["\varphi"] & X \rar["\psi"] & Y \rar["\gamma"]& Z
\end{tikzcd}\]

\noi we have the following commuting diagram where the squares on the front and back are all pullbacks. 

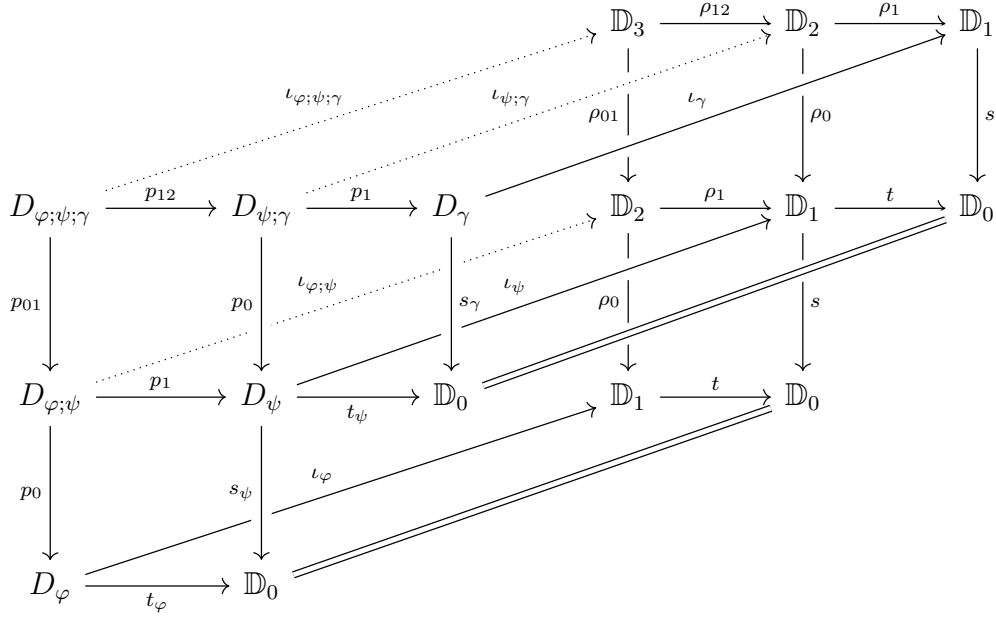
\begin{center}
\begin{tikzcd}[column sep = scriptsize] 
 & & & && &\mD_3 \dar[dd, "\rho_{01}"'] \rar[rr, "\rho_{12}"] & & \mD_2 \dar[dd, "\rho_0"] \rar[rr,"\rho_1"] && \mD_1 \dar[dd, "s"] & & \\
 
 & & && & &   & & && & & \\
 
 D_{\varphi ; \psi; \gamma} \ar[uurrrrrr, dotted, "\iota_{\varphi ; \psi; \gamma}"] \dar[dd, "p_{01}"' ] \rar[rr, "p_{12}"] & & D_{\psi ; \gamma} \ar[uurrrrrr, dotted, crossing over, "\iota_{\psi ; \gamma}"] \rar[rr, "p_1"] && D_\gamma \ar[uurrrrrr, crossing over, "\iota_\gamma"] & &    \mD_2 \dar[dd, "\rho_0"'] \rar[rr, "\rho_1"] & & \mD_1\dar[dd,"s"] \rar[rr, "t"] && \mD_0 & &\\
 
  & & && & & & & && & &     \\
  
 D_{\varphi ; \psi} \ar[uurrrrrr, dotted, "\iota_{\varphi ; \psi}"] \dar[dd, "p_0"'] \rar[rr, "p_1"] & & D_\psi \ar[uurrrrrr, crossing over ,"\iota_\psi"] \ar[from = uu, crossing over, "p_0"']  \rar[rr, "t_\psi"'] && \mD_0 \ar[from = uu, crossing over, "s_\gamma"]\ar[uurrrrrr, crossing over, equals ] & &    \mD_1 \rar[rr,"t"] && \mD_0 &&& &\\
 
  & & &&& &&\\
  
 D_\varphi \rar[rr, "t_\varphi"'] \ar[uurrrrrr, crossing over, "\iota_\varphi"] && \mD_0\ar[from = uu, crossing over, "s_\psi"'] \ar[uurrrrrr, equals ] &&& &   & & & && &
\end{tikzcd}
\end{center}

\noi By the universal property of the coproduct $\mD_3$, we have maps $c_{01}$ and $c_{12}$ which represent composing the first two and last two maps in a composable triple respectively. These are uniquely determined on cofibers by the maps $c'_{01}$ and $c'_{12}$ respectively. The following diagrams are pastings of commuting cubes that show how $c'_{01}$ and $c_{01}$ are related. The coproduct inclusions from left to right are suppressed for readability but are indicated with the bent dotted arrows. 
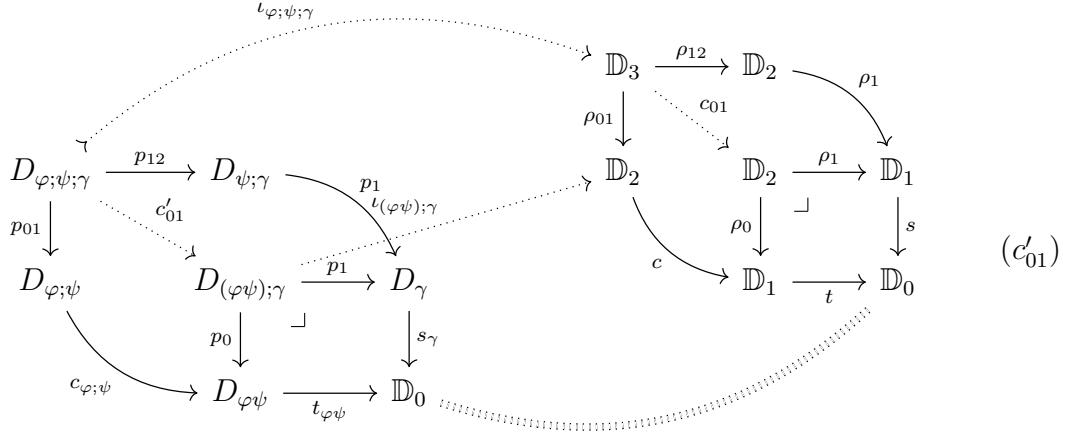
\[\label{dgm int groth cofiber comp associativity c'_01}
\begin{tikzcd}[]
&& && \mD_3 \ar[dr, dotted, "c_{01}"] \dar["\rho_{01}"'] \rar["\rho_{12}"] & \mD_2 \ar[dr, bend left, "\rho_1"]  & \\ 
 D_{\varphi ; \psi ; \gamma} \ar[urrrr, tail, dotted, bend left, "\iota_{\varphi ; \psi ; \gamma}"] \ar[dr, dotted, "c_{01}'"] \dar["p_{01}"'] \rar["p_{12}"] & D_{\psi ; \gamma} \ar[dr, bend left, "p_1"]  & 
 && 
 \mD_2 \ar[dr, bend right, "c"'] & \mD_2 \rar["\rho_1"] \dar["\rho_0"'] \arrow[dr, phantom, "\usebox\pullback" , very near start, color=black] & \mD_1 \dar["s"] \\
D_{\varphi ; \psi} \ar[dr, bend right, "c_{\varphi ; \psi}"'] & D_{(\varphi \psi); \gamma} \ar[urrr, dotted, "\iota_{(\varphi \psi) ;\gamma}"] \rar["p_1"] \dar["p_0"'] \arrow[dr, phantom, "\usebox\pullback" , very near start, color=black] & D_\gamma \dar["s_\gamma"]
 && 
 & \mD_1 \rar["t"'] & \mD_0 \\
& D_{\varphi \psi} \rar["t_{\varphi \psi}"'] & \mD_0 \ar[urrrr, dotted, bend right, equals] && &&\\
\end{tikzcd}
\tag{$c'_{01}$}
\]

\noi A similar diagram shows the relation between $c'_{12}$ and $c_{12}$ and in particular the following squares commute by the universal property of $\mD_2$. 

\begin{center}
\begin{tikzcd}[]
\mD_3 \rar["c_{01}"] & \mD_2 \\
D_{\varphi ; \psi ; \gamma} \uar[tail, "\iota_{\varphi ; \psi ; \gamma}"] \rar["c'_{01}"'] & D_{\varphi \psi ; \gamma} \uar[tail, "\iota_{\varphi \psi; \gamma}"'] 
\end{tikzcd}\qquad \qquad \qquad 
\begin{tikzcd}[]
\mD_3 \rar["c_{12}"] & \mD_2 \\
D_{\varphi ; \psi ; \gamma} \uar[tail, "\iota_{\varphi ; \psi ; \gamma}"] \rar["c'_{12}"'] & D_{\varphi ; \psi \gamma} \uar[tail, "\iota_{\varphi ; \psi \gamma}"'] 
\end{tikzcd}
\end{center}

To show that composition is associative, we need to show that the front of the commuting cube below commutes. 

\[ 
\begin{tikzcd}[]
& & & && \\
& \mD_3 \rar[rr, "c_{12}"] \dar[dd, near end, "c_{01}"'] & & \mD_2 \dar[dd, "c"] & & \\
D_{\varphi ; \psi ; \gamma} \ar[ur, tail, "\iota_{\varphi ; \psi ; \gamma}"] \dar[dd, "c_{01}'"'] \ar[rr, crossing over, near end, "c_{12}'"] && D_{\varphi ; \psi \gamma} \ar[ur, tail, "\iota_{\varphi ; \psi \gamma}"] &&& \\
& \mD_2 \rar[rr, near start, "c"'] & & \mD_1 & & \\
 D_{\varphi \psi ; \gamma} \ar[ur, tail, "\iota_{\varphi \psi ; \gamma}"']\rar[rr, "c_{\varphi \psi ; \gamma}"'] && D_{\varphi \psi \gamma} \ar[ur, tail, "\iota_{\varphi \psi \gamma}"'] \ar[from = uu, crossing over, near start, "c_{\varphi ; \psi \gamma}" ] &&&
\end{tikzcd}
\]

\noi We'll use the universal property of the pullback $D_{\varphi \psi \gamma}$. First notice that 

\[D_{(\varphi \psi) \gamma} = D_{\varphi \psi \gamma} = D_{\varphi (\psi \gamma)}\] 

\noi because of associativity in $\cA$. That is, 

\[(\varphi \psi) \gamma = \varphi (\psi \gamma)\]\ 

\noi so we drop the parentheses and just write $\varphi \psi \gamma$ for the triple composite in $\cA$ without loss of generality. On one hand by Lemma~\ref{Lem source and target of component composite} we have 

\[ p_{01} c \pi_0 \iota_{A} = p_{01} c s_{\varphi \psi} = p_{01} p_0 s_\varphi = p_{01} p_0 \pi_0 \iota_{A} \]\

\noi and since $\iota_A$ is monic, 

\[ \label{eq helper cofiber associativity pi_0 projection} p_{01} c \pi_0 = p_{01} p_0 \pi_0. \tag{*}\]

\noi Now recall by the definition of cofiber composition we have

\[ c_{\varphi \psi ; \gamma} = (p_0 \pi_0 , c'_{\delta ; \varphi \psi; \gamma} c) \qquad \qquad c_{\varphi ; \psi \gamma} = (p_0 \pi_0 , c'_{\delta ; \varphi ; \psi \gamma} c) \]

\noi and so for the $\pi_0$ projection we get: 
\begin{align*}
c'_{01} c_{\varphi \psi ; \gamma} \pi_0 
&= c'_{01} p_0 \pi_0 & \text{Def. } c_{\varphi \psi ; \gamma} \\
&= p_{01} c \pi_0 & \text{Dgm. (\ref{dgm int groth cofiber comp associativity c'_01}) } \\
&= p_{01} p_0 \pi_0 &\text{Eq. } \ref{eq helper cofiber associativity pi_0 projection} \\ 
&= c'_{12} p_0 \pi_0 & \text{Dgm. (\ref{dgm int groth cofiber comp associativity c'_01}) }  \\
&= c'_{12} c_{\varphi ; \psi \gamma} \pi_0 & \text{Def. } c_{\varphi ; \psi \gamma}
\end{align*}

\noi For the $\pi_1$ projection we have the following calculation split up on separate lines for readability. By definition of $c_{\varphi \psi ; \gamma}$:

\[c'_{01} c_{\varphi \psi ; \gamma} \pi_1 
= c'_{01} c'_{\delta ; \varphi \psi ;\gamma } c\]

\noi then by \text{Lemma } \ref{Lem composing first two maps in a triple in Groth Const} the right-hand side is: 
\[ ( p_{01} p_0 \pi_0 \delta_{\varphi \psi ; \gamma} , \ 
p_{01} c'_{\delta ; \varphi ; \psi} c D(\gamma)_1 , \
p_{12} p_1 \pi_1 ) c \] 

\noi The definition of $c'_{\delta ; \varphi ; \psi}$ says this is equal to 

\[( p_{01} p_0 \pi_0 \delta_{\varphi \psi ; \gamma} , 
p_{01} ( p_0 \pi_0 \delta_{\varphi ; \psi} D(\gamma)_1 , \ 
p_0 \pi_1 D(\psi)_1 D(\gamma)_1 , \ 
p_1 \pi_1 D(\gamma)_1 ) c ,
p_{12} p_1 \pi_1 ) c \]

\noi which, by associativity of internal composition (and factoring out a $p_0$ from the pairing map into the object of composable paths of length 4, $\mC_4$) is equal to

\[ ( ( p_{01} p_0 \pi_0 \delta_{\varphi \psi ; \gamma} , 
p_{01} p_0 ( \pi_0 \delta_{\varphi ; \psi} D(\gamma)_1 , \pi_1 D(\psi)_1 D(\gamma)_1 ) c ) c ,
p_{01} p_1 \pi_1 D(\gamma)_1 , \ 
p_{12} p_1 \pi_1  ) c . \]

\noi More associativity of internal composition and factoring $p_{01} p_0$ from the pairing map being post-composing with internal composition gives

\[( p_{01} p_0 ( \pi_0 \delta_{\varphi \psi ; \gamma} ,  ( \pi_0 \delta_{\varphi ; \psi} D(\gamma)_1 , \pi_1 D(\psi)_1 D(\gamma)_1 ) c ) c ,
p_{01} p_1 \pi_1 D(\gamma)_1 ,  p_{12} p_1 \pi_1  ) c  \]

\noi By associativity of internal composition and the definition of $D_{\varphi ; \psi ; \gamma}$ this becomes:

\[ ( p_{01} p_0 ( \pi_0 \delta_{\varphi \psi ; \gamma} ,  \pi_0 \delta_{\varphi ; \psi} D(\gamma)_1 , \pi_1 D(\psi)_1 D(\gamma)_1 ) c , 
p_{12} p_0 \pi_1 D(\gamma)_1 ,  p_{12} p_1 \pi_1  ) c \]

\noi By Lemma~\ref{Lem Coherence + Naturality computation for associativity} this is equal to: 

\[ ( p_{01} p_0 ( \pi_0 \delta_{\varphi ; \psi \gamma} , \pi_1 D(\psi \gamma)_1, \pi_1 t \delta_{\psi ; \gamma} ) c , p_{12} p_1 \pi_1  ) c \]

\noi By associativity of internal composition we get

\[ ( p_{01} p_0 \pi_0 \delta_{\varphi ; \psi \gamma} , \ 
p_{01} p_0 \pi_1 D(\psi \gamma)_1,\ p_{01} p_0 \pi_1 t \delta_{\psi ; \gamma} , 
p_{12} p_0 \pi_1 D(\gamma)_1 , \ 
 p_{12} p_1 \pi_1  ) c  \]
 
\noi and then by more associativity 
 
\[ ( p_{01} p_0 \pi_0 \delta_{\varphi ; \psi \gamma} , p_{01} p_0 \pi_1 D(\psi \gamma)_1, 
 ( p_{01} p_0 \pi_1 t \delta_{\psi ; \gamma} , p_{12} p_0 \pi_1 D(\gamma)_1 ,  p_{12} p_1 \pi_1 ) c  ) c  \]
 
\noi By definition of $D_{\varphi ; \psi}$ this becomes
 
\[ ( p_{01} p_0 \pi_0 \delta_{\varphi ; \psi \gamma} , 
p_{01} p_0 \pi_1 D(\psi \gamma)_1, 
( p_{01} p_1 \pi_0 \delta_{\psi ; \gamma} , p_{12} p_0 \pi_1 D(\gamma)_1 ,  p_{12} p_1 \pi_1 ) c  ) c  
 \]

\noi and by definition of $ D_{\varphi ; \psi ; \gamma}$ we get 
\[ ( p_{01} p_0 \pi_0 \delta_{\varphi ; \psi \gamma} , p_{01} p_0 \pi_1 D(\psi \gamma)_1, 
( p_{12} p_0 \pi_0 \delta_{\psi ; \gamma} , p_{12} p_0 \pi_1 D(\gamma)_1 , p_{12} p_1 \pi_1 ) c  ) c \]

\noi Factoring gives

\[ ( p_{01} p_0 \pi_0 \delta_{\varphi ; \psi \gamma} , p_{01} p_0 \pi_1 D(\psi \gamma)_1,
p_{12} ( p_0 \pi_0 \delta_{\psi ; \gamma} , p_0 \pi_1 D(\gamma)_1 ,  p_1 \pi_1 ) c  ) c p_{12} \]

\noi and the definition of $c'_{\psi ; \gamma} $ says this is equal to

\[ ( p_{01} p_0 \pi_0 \delta_{\varphi ; \psi \gamma} , p_{12} p_0 \pi_1 D(\gamma)_1 , p_{01} p_0 \pi_1 D(\psi \gamma)_1,
p_{12} ( c'_{\delta ; (\psi ; \gamma)} q_0 , c'_{\delta; (\psi ; \gamma} q_1 , c'_{\psi ; \gamma} q_1 ) c ) c \]

\noi The definitions of $ \text{Def. } c'_{\delta ; \psi ; \gamma}$ and $ c'_{\delta ; (\psi ; \gamma)} $ imply the last term is equal to

\[ ( p_{01} p_0 \pi_0 \delta_{\varphi ; \psi \gamma} , p_{01} p_0 \pi_1 D(\psi \gamma)_1, 
p_{12} ( c'_{\delta ; \psi ; \gamma} q_{01} q_0 , c'_{\delta ; \psi ; \gamma} q_{01} q_1 , c'_{\delta ; \psi ; \gamma} q_{12} q_1 ) c  ) c \]

\noi and the definition of $c'_{\delta ;\psi ; \gamma}$ makes it 

\[ ( p_{01} p_0 \pi_0 \delta_{\varphi ; \psi \gamma} , p_{01} p_0 \pi_1 D(\psi \gamma)_1, 
p_{12} c'_{\delta ; \psi ; \gamma} c  ) c  \]

\noi By Lemma~\ref{Lem composing last two maps of a triple in the Groth Const} this is equal to the left-hand side of the final equation
\[ c'_{12} c'_{\delta ; \varphi ; \psi \gamma} c = c'_{12} c_{ \varphi ; \psi \gamma} \pi_1 \]

\noi which follows from the definition of $c_{\varphi ; \psi \gamma}$. Then the universal property of pullbacks says

\[ c'_{01} c_{\varphi \psi ; \gamma} = c'_{12} c_{\varphi ; \psi \gamma} .\]
\noi This shows composition is associative on cofibers/components of the coproduct. Associativity of composition in $\mD$ now follows by the universal property of the coproduct $\mD_3$. 
\end{proof}\

\subsection{Lemmas for 1-cells of the Canonical Lax Transformation}

The following are technical lemmas used in Section \ref{S Int Groth as Oplax Colim} of Chapter \ref{Chaper Internal Grothendieck}. 

\begin{lem}\label{Lem 1-cells preserve composition at cofiber level} 
For any $A \in \cA_0$:

\[( q_0 (\ell_A)'_1 {,} q_1 (\ell_A)'_1 ) c_{1_A; 1_A} = c (\ell_A)'_1\]
\end{lem}
\begin{proof}
First compute the $0$'th projection:

\begin{align*}
( q_0 (\ell_A)'_1 {,} q_1 (\ell_A)'_1 ) c_{1_A; 1_A} \pi_0 
&= ( q_0 (\ell_A)'_1 {,} q_1 (\ell_A)'_1 ) p_0 \pi_0 &\text{Def. } \\
&= q_0 (\ell_A)'_1 \pi_0 &\text{Def. } \\
&= q_0 s &\text{Def. } (\ell_A)'_1\\
&= c s &\text{Def. } c \\
&= c (\ell_A)'_1 \pi_0 &\text{Def. } (\ell_A)'_1
\end{align*}

\noi For the first projection we break up equalities on separate lines and provide justification for each step in between once again for readability. Starting with the equation, 

\[
( q_0 (\ell_A)'_1 {,} q_1 (\ell_A)'_1 ) c_{1_A; 1_A} \pi_1 
= ( q_0 (\ell_A)'_1 {,} q_1 (\ell_A)'_1 ) c'_{\delta ; 1_A ; 1_A} c, \] 

\noi the right-hand side is equal to 
\[ ( q_0 s \delta_{1_A ; 1_A} ,
q_0 ( s \delta_A ,  1_{D(A)_1} ) c D(1_A)_1 \ , 
q_1 ( s \delta_A ,  1_{D(A)_1} ) c ) c \]

\noi by definition of $ c'_{1_A ; 1_A}$. By functoriality of $D(1_A)$ the last term is equal to

\[ ( q_0 s \delta_{1_A ; 1_A} ,
q_0 ( s \delta_A D(1_A)_1 , D(1_A)_1 ) c ,
q_1 ( s \delta_A ,  1_{D(A)_1} ) c ) c \]

\noi which, by associativity of internal composition, is equal to

\[ ( q_0 s \delta_{1_A ; 1_A} , \ q_0 s \delta_A D(1_A)_1 ,
( q_0 D(1_A)_1 , q_1 s \delta_A ) c ,\  q_1 ) c . \]

\noi The definition of $D(A)_2$ makes this equal to 

\[ ( q_0 s \delta_{1_A ; 1_A} , \ q_0 s \delta_A D(1_A)_1 ,
( q_0 D(1_A)_1 , q_0 t \delta_A ) c ,\  q_1 ) c \]

\noi which, by factoring maps with respect to pairing maps, is equal to

\[ ( q_0 s \delta_{1_A ; 1_A} , \ q_0 s \delta_A D(1_A)_1 ,
q_0 ( D(1_A)_1 , t \delta_A ) c ,\  q_1 ) c .\]

\noi Naturality of $\delta_A$ makes this equal to 
\[ ( q_0 s \delta_{1_A ; 1_A} , \ q_0 s \delta_A D(1_A)_1 ,
q_0 ( s \delta_A , 1_{D(A)_1} ) c ,\  q_1 ) c \] 

\noi and by associativity we get 

\[ (  q_0 s \delta_{1_A ; 1_A} , 
( q_0 s \delta_A D(1_A)_1 , q_0 s \delta_A ) c , \ q_0  ,\  q_1 ) c .\] 

\noi Factoring with respect to pairing maps gives 

\[ ( q_0 s \delta_{1_A ; 1_A} , 
q_0 s ( \delta_A D(1_A)_1 , \delta_A ) c , \ q_0 , \  q_1 ) c \]

\noi and associativity then gives

\[ ( ( q_0 s\delta_{1_A ; 1_A}  q_0 s \delta_A D(1_A)_1 ) c , q_0 s \delta_A , \ q_0 ,\  q_1 ) c. \]

\noi By factoring again we get
\[ ( q_0 s ( \delta_{1_A ; 1_A} , 
\delta_A D(1_A)_1 ) c ,
q_0 s \delta_A , \ q_0 ,\  q_1 ) c \]

\noi and by coherence of the structure isomorphisms for the pseudofunctor $D$ this becomes 

\[ ( q_0 s e_A , \ q_0 s \delta_A , \ q_0  ,\  q_1 ) c. \]

\noi By associativity of internal composition we have equality with

\[ ( q_0 s ( e_A , \delta_A) c , \ q_0  ,\  q_1 ) c \]

\noi and by factoring with pairing maps we get equality with 
\[( q_0 s ( 1_{D(A)_0} ,  \delta_A ) ( e_A , 1_{D(A)_1} ) c ,
q_0 ,\  q_1 ) c]\]

\noi The identity law in $D(A)$ makes the last term equal to 

\[ ( q_0 s ( 1_{D(A)_0} ,  \delta_A ) p_1  , \ q_0  ,\  q_1 ) c \]

\noi and by definition of the pullback projections we get 

\[ ( q_0 s \delta_A  , \ q_0 \ ,\  q_1 ) c \]

\noi Definition of internal composition gives equality with 

\[ ( c s \delta_A  , \ q_0 \ ,\  q_1 ) c \]

\noi and associativity gives 

\[ ( c s \delta_A  , \ ( q_0 , q_1 ) c ) c . \]

\noi The universal property of the pullbacks $D(A)_2$ make this equal to 
\[ ( c s \delta_A  , \ 1_{D(A)_2} c ) c \]
\noi which becomes the left-hand side of the final equality:
\[ c ( s \delta_A  , \ 1_{D(A)_1} ) c = c (\ell_A)'_1 \pi_1\]

\noi The result follows by the universal property of the pullback $D_\varphi$. 
\end{proof}\

\subsection{Lemmas for 2-cells of the Canonical Lax Transformation}

\label{SS 2-cells of l}

\begin{lem}\label{Lem sidecalc 1}
\[ ( (\ell_A)'_1 \iota_{1_A} 
\ , \ 
 t ( 1_{D(A)_0} , D(\varphi)_0 e_B ) \iota_\varphi) 
 = ( (\ell_A)'_1 
\ ,\
 t ( 1_{D(A)_0} , D(\varphi)_0 e_B ) ) \iota_{1_A ; \varphi}\] 
\end{lem}
\begin{proof}
By the universal property of $\mD_2$ it suffices to compute 

\begin{align*}
( (\ell_A)'_1 
\ ,\
 t ( 1_{D(A)_0} , D(\varphi)_0 e_B ) ) \iota_{1_A ; \varphi} \rho_0 
 &= ( (\ell_A)'_1 
\ ,\
 t ( 1_{D(A)_0} , D(\varphi)_0 e_B ) ) p_0 \iota_{1_A} \\
 &= (\ell_A)'_1 \iota_{1_A} 
\end{align*}
\noi and 
\begin{align*}
( (\ell_A)'_1 
\ ,\
 t ( 1_{D(A)_0} , D(\varphi)_0 e_B ) ) \iota_{1_A ; \varphi} \rho_1 
 &= ( (\ell_A)'_1 
\ ,\
 t ( 1_{D(A)_0} , D(\varphi)_0 e_B ) ) p_1 \iota_\varphi \\
 &= t ( 1_{D(A)_0} , D(\varphi)_0 e_B ) \iota_\varphi.
\end{align*}
\end{proof}

\noi Notice the following computation contains the first functoriality argument for the naturality proof above in the case $\cE = \Set$. 

\begin{lem}\label{Lem sidecalc 2}
\[ ( (\ell_A)'_1 
, \
t ( 1_{D(A)_0} , D(\varphi)_0 e_B ) 
) c'_{\delta ; 1_A ; \varphi} 
=
( s \delta_{1_A ; \varphi} 
, \ 
( s \delta_A D(\varphi)_1 , D(\varphi)_1 ) c 
, \ 
t e_A D(\varphi)_1 ) 
\]

\end{lem}
\begin{proof}
By the universal property of $D(B)_3$ it suffices to check three equalities. First we have 

\begin{align*}
( (\ell_A)'_1 
, \
t ( 1_{D(A)_0} , D(\varphi)_0 e_B ) 
) c'_{\delta ; 1_A ; \varphi} q_{01} q_0 
&= ( (\ell_A)'_1 
, \
t ( 1_{D(A)_0} , D(\varphi)_0 e_B ) 
) p_0 \pi_0 \delta_{1_A ; \varphi} \\
&= (\ell_A)'_1 \pi_0 \delta_{1_A ; \varphi}  \\
&= s \delta_{1_A ; \varphi} 
\end{align*}

\noi where the first equality is by definition of $c'_{\delta ; 1_A ; \varphi}$, the second line is by definition of the pullback projection, $p_0$, and the third line is by definition of $ (\ell_A)'_1$. Second, 

\begin{align*}
( (\ell_A)'_1 
, \
t ( 1_{D(A)_0} , D(\varphi)_0 e_B ) 
) c'_{\delta ; 1_A ; \varphi} q_{01} q_1 
&= ( (\ell_A)'_1 
, \
t ( 1_{D(A)_0} , D(\varphi)_0 e_B ) 
) p_0 \pi_1 D( \varphi)_1 \\
&= (\ell_A)'_1 \pi_1 D( \varphi)_1 & \text{} \\
&= ( s \delta_A , 1_{D(A)_1}) c D( \varphi)_1 \\
&=( s \delta_A D(\varphi)_1 , D(\varphi)_1) c 
\end{align*}

\noi where the first line is by definition of $c'_{\delta ; 1_A}$, the second line is by definition of the pullback projection $p_0$ and the pairing map it is precomposed with, the third line is be definition of $(\ell_A)'_1$, and the last line is by functoriality of $D(\varphi)$. Finally we can see 

\begin{align*}
( (\ell_A)'_1 
, \
t ( 1_{D(A)_0} , D(\varphi)_0 e_B ) 
) c'_{\delta ; 1_A ; \varphi} q_{12} q_1 
&= ( (\ell_A)'_1 
, \
t ( 1_{D(A)_0} , D(\varphi)_0 e_B ) 
) c'_{1_A ; \varphi} q_1 \\
&= ( (\ell_A)'_1 
, \
t ( 1_{D(A)_0} , D(\varphi)_0 e_B ) 
) p_1 \pi_1 \\
&= t ( 1_{D(A)_0} , D(\varphi)_0 e_B ) \pi_1 \\
&= t D(\varphi)_0 e_B \\
&= t e_A D(\varphi)_1 \\
&= D(\varphi)_1 \\
\end{align*}

\noi where the first line is by definition of $ c'_{\delta ; 1_A ; \varphi}$, the second line is by definition of $ c'_{1_A ; \varphi}$, the third line is by definition of the pullback projection $p_1$, the third line is by definition of the pullback projection $p_1$, the fourth line is by functoriality of $D(\varphi)$ and the last line is by definition of the identity structure map, $e_A$, of $D(A)$. 
\end{proof}

\noi The previous calculation is an intermediate step for the following lemma which we use in our naturality computation at the end of this subsection. 

\begin{lem}\label{Lem sidecalc 3} 
\[ ( (\ell_A)'_1 \ , \ t ( 1_{D(A)_0} , D(\varphi)_0 e_B ) ) c_{1_A ; \varphi} = ( s , \ ( s \delta_{1_A ; \varphi} , s \delta_A D(\varphi)_1 , D(\varphi)_1 ) c ) \] 
\end{lem}
\begin{proof}
By the universal property of $D_\varphi$ it suffices to compute the pullback projections and check that they're equal. First we can see 

\begin{align*}
( (\ell_A)'_1 \ , \ t ( 1_{D(A)_0} , D(\varphi)_0 e_B ) ) c_{1_A ; \varphi} \pi_0 
&= ( (\ell_A)'_1 \ , \ t ( 1_{D(A)_0} , D(\varphi)_0 e_B ) ) p_0 \pi_0 \\
&= (\ell_A)'_1 \pi_0 \\
&= s
\end{align*}

\noi by definition of the pullback projections, $p_0$ and $\pi_0$, and the map $(\ell_A)'_1$. Next we can see 

\begin{align*}
& \ \ \ \ 
( (\ell_A)'_1 \ , \ t ( 1_{D(A)_0} , D(\varphi)_0 e_B ) ) c_{1_A ; \varphi} \pi_1\\
&= ( (\ell_A)'_1 \ , \ t ( 1_{D(A)_0} , D(\varphi)_0 e_B ) ) c'_{\delta ' 1_A ; \varphi} c \\
&= ( s \delta_{1_A ; \varphi} 
, \ 
( s \delta_A D(\varphi)_1 , D(\varphi)_1 ) c 
, \ 
t e_A D(\varphi)_1 ) c \\
&= ( s \delta_{1_A ; \varphi} 
, \ 
 s \delta_A D(\varphi)_1 , \ ( D(\varphi)_1 , t e_A D(\varphi)_1 ) c ) c \\
 &= ( s \delta_{1_A ; \varphi} 
, \ 
 s \delta_A D(\varphi)_1 , \ ( 1_{D(A)_1} , t e_A ) c D(\varphi)_1 ) c \\
 &= ( s \delta_{1_A ; \varphi} 
, \ 
 s \delta_A D(\varphi)_1 , \ ( 1_{D(A)_1} , t ) ( p_0 , p_1 e_A ) c D(\varphi)_1 ) c \\ 
 &= ( s \delta_{1_A ; \varphi} 
, \ 
 s \delta_A D(\varphi)_1 , \ ( 1_{D(A)_1} , t ) p_0 D(\varphi)_1 ) c \\ 
 &= ( s \delta_{1_A ; \varphi} 
, \ 
 s \delta_A D(\varphi)_1 , \ D(\varphi)_1 ) c.
\end{align*}

where the first line is by definition of $c_{1_A; \varphi}$, the second line is by Lemma~\ref{Lem sidecalc 2}, the third line is by associativity of composition, the fourth line is by functoriality of $D(\varphi)$, the fifth line is given by factoring a pairing map, the sixth line is coming from the identity law in $D(A)$, and the last line is by definition of the pullback projection $p_0$.

\end{proof}

\noi The remaining lemmas are side calculations that show different ways of representing internal compositions involving certain pairing maps. We used them to prove results about the 1-cells of the canonical lax natural transformation, $\ell$. 

\begin{lem}\label{Lem sidecalc 4}

The cofiber composition, $D(A)_0 \to D_{\varphi}$, given by the term 
\[ ( 
  s ( 1_{D(A)_0} , e_A D(\varphi)_1 ) 
  , \ 
  ( s D(\varphi)_0 , D(\varphi)_1 (\ell_B)'_1 \pi_1 ) ) c'_{\delta ; \varphi ; 1_B} 
  \]
  \noi is equal to 
 \[ ( s \delta_{\varphi ; 1_B} , s e_A D(\varphi)_1 D(1_B)_1 , D(\varphi)_1 (\ell_B)'_1 \pi_1 ) \]
\end{lem}
\begin{proof}
By the universal property of $D(B)_3$, it suffices to check the three projections $D(B)_3 \to D(B)_1$. First we have 

\begin{align*}
& \quad \ ( 
  s ( 1_{D(A)_0} , e_A D(\varphi)_1 ) 
  , \ 
  ( s D(\varphi)_0 , D(\varphi)_1 (\ell_B)'_1 \pi_1 ) ) c'_{\delta ; \varphi ; 1_B} q_{01} q_0 \\
  &= ( 
  s ( 1_{D(A)_0} , e_A D(\varphi)_1 ) 
  , \ 
  ( s D(\varphi)_0 , D(\varphi)_1 (\ell_B)'_1 \pi_1 ) ) p_0 \pi_0 \delta_{\varphi ; 1_B} \\
  &= 
  s ( 1_{D(A)_0} , e_A D(\varphi)_1 ) 
  \pi_0 \delta_{\varphi ; 1_B} \\
  &= s \delta_{\varphi ; 1_B} & , 
\end{align*}

\noi Second we have 

\begin{align*}
& \quad \ ( 
  s ( 1_{D(A)_0} , e_A D(\varphi)_1 ) 
  , \ 
  ( s D(\varphi)_0 , D(\varphi)_1 (\ell_B)'_1 \pi_1 ) ) c'_{\delta ; \varphi ; 1_B} q_{12} q_0 \\
  &= ( 
  s ( 1_{D(A)_0} , e_A D(\varphi)_1 ) 
  , \ 
  ( s D(\varphi)_0 , D(\varphi)_1 (\ell_B)'_1 \pi_1 ) ) p_0 \pi_1 D(1_B)_1 \\
  &=
  s ( 1_{D(A)_0} , e_A D(\varphi)_1 ) \pi_1 D(1_B)_1 \\
  &= s e_A D(\varphi)_1 D(1_B)_1 &, 
\end{align*}

\noi and finally 

\begin{align*}
& \quad \ ( 
  s ( 1_{D(A)_0} , e_A D(\varphi)_1 ) 
  , \ 
  ( s D(\varphi)_0 , D(\varphi)_1 (\ell_B)'_1 \pi_1 ) ) c'_{\delta ; \varphi ; 1_B} q_{12} q_1 \\
  &= ( 
  s ( 1_{D(A)_0} , e_A D(\varphi)_1 ) 
  , \ 
  ( s D(\varphi)_0 , D(\varphi)_1 (\ell_B)'_1 \pi_1 ) ) p_1 \pi_1 \\
  &= 
  ( s D(\varphi)_0 , D(\varphi)_1 (\ell_B)'_1 \pi_1 )\pi_1 \\
  &= D(\varphi)_1 (\ell_B)'_1 \pi_1 
\end{align*}. 
\end{proof}

\noi 

\begin{lem}\label{Lem sidecalc 5}

The pairing map 
\[ 
( s 
, \
  ( s \delta_{\varphi ; 1_B} , s D(\varphi)_0e_B D(1_B)_1 , s D(\varphi)_0 \delta_B , D(\varphi)_1(\ell_B)'_1 \pi_1 ) c )\]
  \noi is equal to the cofiber composition
  \[ ( 
  s ( 1_{D(A)_0} , e_A D(\varphi)_1 ) 
  , \ 
  ( s D(\varphi)_0 , D(\varphi)_1 (\ell_B)'_1 \pi_1 ) ) c_{\varphi ; 1_B} \]
\end{lem}
\begin{proof}
By the universal property of $D_\varphi$, it suffices to check that 

\begin{align*}
& \ \ \ \ ( 
  s ( 1_{D(A)_0} , e_A D(\varphi)_1 ) 
  , \ 
  ( s D(\varphi)_0 , D(\varphi)_1 (\ell_B)'_1 \pi_1 ) ) c_{\varphi ; 1_B} \pi_0 \\
 &= ( 
  s ( 1_{D(A)_0} , e_A D(\varphi)_1 ) 
  , \ 
  ( s D(\varphi)_0 , D(\varphi)_1 (\ell_B)'_1 \pi_1 ) ) p_0 \pi_0 \\
  &= 
  s ( 1_{D(A)_0} , e_A D(\varphi)_1 ) \pi_0 \\
  &= s 1_{D(A)_0} \\
  &= s
\end{align*}

\noi and by Lemma~\ref{Lem sidecalc 4} and functoriality of $D(\varphi)$ we have

\begin{align*}
& \ \ \ \ 
( 
  s ( 1_{D(A)_0} , e_A D(\varphi)_1 ) 
  , \ 
  ( s D(\varphi)_0 , D(\varphi)_1 (\ell_B)'_1 \pi_1 ) ) c_{\varphi ; 1_B} \pi_1\\
 &= ( 
  s ( 1_{D(A)_0} , e_A D(\varphi)_1 ) 
  , \ 
  ( s D(\varphi)_0 , D(\varphi)_1 (\ell_B)'_1 \pi_1 ) ) c'_{\delta ; \varphi ; 1_B} c\\
 &= ( s \delta_{\varphi ; 1_B} , s e_A D(\varphi)_1 D(1_B)_1 , D(\varphi)_1 (\ell_B)'_1 \pi_1 ) c  \\
  &= ( s \delta_{\varphi ; 1_B} , s D(\varphi)_0 e_B D(1_B)_1 , D(\varphi)_1 (\ell_B)'_1 \pi_1 ) c 
\end{align*}.
\end{proof}

\begin{lem}\label{Lem sidecalc 6} 
\[ ( 
  s ( 1_{D(A)_0} , e_A D(\varphi)_1 ) 
  , \ 
 D(\varphi)_1 (\ell_B)'_1 ) \iota_{\varphi ; 1_B}
= ( 
  s ( 1_{D(A)_0} , e_A D(\varphi)_1 ) \iota_\varphi 
  , \ 
 D(\varphi)_1 (\ell_B)'_1 \iota_{1_B} ) \]
\end{lem}
\begin{proof}
By the universal property of $\mD_2$, it suffices to check 

\begin{align*}
( 
  s ( 1_{D(A)_0} , e_A D(\varphi)_1 ) 
  , \ 
 D(\varphi)_1 (\ell_B)'_1 ) \iota_{\varphi ; 1_B} \rho_0 
 &= ( 
  s ( 1_{D(A)_0} , e_A D(\varphi)_1 ) 
  , \ 
 D(\varphi)_1 (\ell_B)'_1 ) p_0 \iota_\varphi \\
 &= 
  s ( 1_{D(A)_0} , e_A D(\varphi)_1 ) \iota_\varphi
\end{align*}

\noi and 

\begin{align*}
( 
  s ( 1_{D(A)_0} , e_A D(\varphi)_1 ) 
  , \ 
 D(\varphi)_1 (\ell_B)'_1 ) \iota_{\varphi ; 1_B} \rho_1
 &= ( 
  s ( 1_{D(A)_0} , e_A D(\varphi)_1 ) 
  , \ 
 D(\varphi)_1 (\ell_B)'_1 ) p_1 \iota_{1_B} \\
 &= 
  D(\varphi)_1 (\ell_B)'_1 \iota_{1_B}
\end{align*}
\end{proof}

\section{Internal Category of Fractions}\label{App.Ch Cat of Fracs}

\subsection{Defining Span Composition on Representatives}\label{App.S. Defining Span Comp on Reps}

This appendix consists of technical lemmas which are really just computations used in the proof of Lemma~\ref{lem witnessing c' well-defined} in Chapter \ref{Ch Internal Fractions}. We use these to define the composition structure of the internal category of fractions and prove it forms an internal category. They are heavily dependant on their context in that lemma so we restate the beginning of that proof and include the diagrams of covers that define the lifts from the fractions axioms being referred to in the lemmas.

First, pullbacks of $u : U \to \spn \tensor[_{t}]{\times}{_s} \spn$ are taken along $p_0^2$ and $p_1^2$ to get two covers of $\slb \tensor[_{t}]{\times}{_s}\slb$ that witness composition of the sailboat projections, $p_0^2$ and $p_1^2$: 

\[\label{dgm slb proj pb along span comp cover (appendix)}
\begin{tikzcd}[]
\bar{U}_0 \arrow[dr, phantom, "\usebox\pullback" , very near start, color=black] \dar["/" marking, "\bar{u}_0"' near end] \rar["\pi_1"] & U \dar["/" marking, "u" near start] & \bar{U}_1 \dar["/" marking, "\bar{u}_1" near start] \lar["\pi_1"'] \arrow[dl, phantom, "\usebox\urpullback" , very near start, color=black]\\
 \slb \tensor[_{t}]{\times}{_s} \slb \rar["p_0^2"'] & \spn \tensor[_{t}]{\times}{_s} \spn & \slb \tensor[_{t}]{\times}{_s} \slb \lar["p_1^2"]
\end{tikzcd} \tag{1}
\]

\noi A refinement 
\[ \label{dgm slb proj pb along span comp cover refinement (appendix)}
\begin{tikzcd}[]
\bar{U} \ar[d,"\pi_0"'] \ar[r, "\pi_1"] \ar[dr, "/" marking, "\bar{u}" near start] & \bar{U}_1 \dar["/" marking,"\bar{u}_1" near start] \\
\bar{U}_0 \rar["/" marking,"\bar{u}_0"' near end] & \slb \tensor[_{t}]{\times}{_s} \slb
\end{tikzcd} \tag{2}
\]

\noi is given by a pullback of $\bar{u}_0$ and $\bar{u}_1$ and provides us with a common cover domain for the cover. Next we need to describe composition for the intermediate pair of composable spans: 
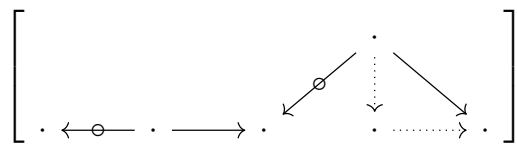
\begin{center}
$\left[
\begin{tikzcd}[]
&& & \cdot \dar[dotted] \ar[dl, "\circ" marking] \ar[dr,] & \\
\cdot & \cdot \lar["\circ" marking] \rar & \cdot & \cdot \rar[dotted] & \cdot 
\end{tikzcd} 
\right]$
\end{center}

\noi The following figure shows the construction of three different composites being constructed. 

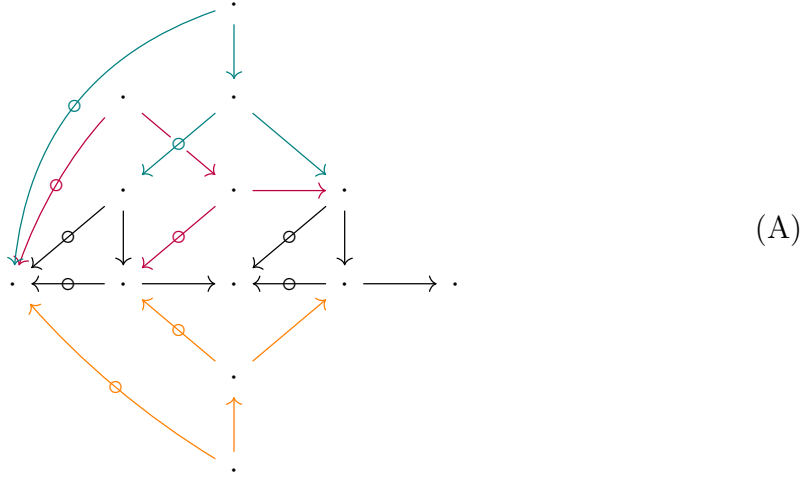
\[ \label{fig intermediate span comp well-defined (appendix)} \begin{tikzcd}[]
	&& \cdot \\
	& \cdot & \cdot \\
	& \cdot & \cdot & \cdot \\
	\cdot & \cdot & \cdot & \cdot & \cdot \\
	&& \cdot \\
	&& \cdot
	\arrow[from=4-2, to=4-1, "\circ" marking ]
	\arrow[from=4-2, to=4-3]
	\arrow[from=3-4, to=4-4]
	\arrow[from=4-4, to=4-3, "\circ" marking ]
	\arrow[from=4-4, to=4-5]
	\arrow[from=3-2, to=4-2]
	\arrow[from=3-2, to=4-1, "\circ" marking ]
	\arrow[from=3-4, to=4-3, "\circ" marking ]
	\arrow[from=2-3, to=3-4, color = teal]
	\arrow[from=3-3, to=4-2, "\circ" marking, color = purple ]
	\arrow[from=3-3, to=3-4, color = purple]
	\arrow[from=5-3, to=4-2, "\circ" marking , color = orange ]
	\arrow[from=5-3, to=4-4, color = orange ]
	\arrow[from=6-3, to=5-3, color = orange ]
	\arrow[from=1-3, to=2-3, color = teal]
	\arrow[from=2-2, to=3-3, color = purple]
	\arrow[from=2-3, to=3-2, crossing over, "\circ" marking, color = teal]
	\arrow[curve={height=6pt}, from=2-2, to=4-1, "\circ" marking, color = purple]
	\arrow[curve={height=30pt}, from=1-3, to=4-1, "\circ" marking , color = teal]
	\arrow[curve={height=-6pt}, from=6-3, to=4-1, "\circ" marking , color = orange ]
\end{tikzcd} \tag{A}\]

\noi We define composition for this intermediate span similarly to how we defined $\sigma_\circ$. This could actually have been done by taking a pullback of the cover, $u: U \to \spn \tensor[_{t}]{\times}{_s} \spn$, witnessing span composition in general and finding a common refinement for this with the previous refinement. The same result holds either way. Denote the comparison pair of composable spans by $\gamma$ and define it by the universal property in the following pullback diagram.

\[ \label{dgm intermediate compblespan (appendix)} 
\begin{tikzcd}[column sep = large, row sep = large]
 \bar{U} \ar[d, "/" marking, "\bar{u}"' near end] \ar[r, "/" marking, "\bar{u}" near end] \ar[dr, dotted, "\gamma"] 
& \slb \tensor[_{t}]{\times}{_s} \slb \rar["p_1^2"] 
& \spn \tensor[_{t}]{\times}{_s} \spn \dar["\pi_1"] 
\\
 \slb \tensor[_{t}]{\times}{_s} \slb \dar[ "p_0^2"'] 
& \spn \tensor[_{t}]{\times}{_s} \spn \dar["\pi_0"'] \rar["\pi_1"] 
& \spn \dar["s"] 
\\
 \spn \tensor[_{t}]{\times}{_s} \spn \rar["\pi_0"'] 
&\spn \rar["t"'] 
& \mC_0 
\end{tikzcd} \tag{3}
\]

\noi The following diagram of covers shows how the intermediate composite span is constructed for $\gamma$. 

\[ \label{dgm diagram of covers to include composite of intermediate comp'ble spans(appendix)}
\begin{tikzcd}[column sep = large ]
 W_\circ \rar[rr, "(\pi_0 \pi_1 {,} \pi_0 \pi_2)"] 
&& W \times_{\mC_0} W 
& \\
\tilde{U} 
\dar[shift left = 1.5, bend left =50, "\sigma_1", color = teal ] 
\dar[color = purple, "\sigma_\gamma"']
\dar[shift right = 1.5, bend right= 50, "\sigma_0"' , color = orange] 
\uar[color = purple, "\omega_\gamma"]
\rar[rr, "/" marking , "\tilde{u}_0" near start]
 &
 & \tilde{U}_0
 \dar[color = purple, "\theta_\gamma"']
 \uar[color = purple, "(\theta_\gamma \pi_0 \pi_0 {, } \tilde{u}_1 \gamma \pi_0 \pi_0) "'] 
 \rar["/" marking , "\tilde{u}_1" near start] 
 & \bar{U} 
 \dar[color = purple, "(\gamma \pi_0 \pi_1 {, }\gamma \pi_1 \pi_0)"] 
 \\
 \spn 
 &
 & W_\square 
 \rar["(\pi_0 \pi_1 {,} \pi_1 \pi_1)"'] 
 & \mC_1 \tensor[_{t}]{\times}{_{wt}} W 
\end{tikzcd} \tag{$\star$}
\]

\noi The left and right curved arrows, $\textcolor{orange}{\sigma_0}$ and $\textcolor{teal}{\sigma_1}$, into $\spn$ in the bottom left corner are defined by applying the composite of spans, $\sigma_\circ$, to the composable spans given by applying $p_0^2$ and $p_1^2$ to the pair of composable sailboats. Since $\sigma_\circ$ is only defined on $U$ we need to pass through the appropriate cover. The colours in the previous diagram and following equations indicate which of the three different span compositions in Figure (\ref{fig intermediate span comp well-defined}) the arrows in the following equations are witnessing. 

\begin{align}\label{def sigma_0 (appendix)} 
\begin{split}
 \textcolor{orange}{\sigma_0} 
 &= \textcolor{orange}{\tilde{u} \pi_0 \pi_1 \sigma_\circ} \\
 &= \textcolor{orange}{\tilde{u} \pi_0 \pi_1 \big( \omega \pi_1 , (\omega \pi_0 \pi_0 {,} u_0 \theta \pi_1 \pi_0 {,} u \pi_1 \pi_1 )c \big)} 
 \end{split}
 \end{align}
 \noi and 
 \begin{align}\label{def sigma_1 (appendix)}
 \begin{split}
 \textcolor{teal}{\sigma_1 }
 &= \textcolor{teal}{\tilde{u} \pi_1 \pi_1 \sigma_\circ} \\
 &= \textcolor{teal}{\tilde{u} \pi_1 \pi_1 \big( \omega \pi_1 , (\omega \pi_0 \pi_0 {,} u_0 \theta \pi_1 \pi_0 {,} u \pi_1 \pi_1 )c \big)}.
 \end{split}
 \end{align}

\noi The arrow into $\spn$ on the bottom left side of the cover diagram is the universal map
\[ \textcolor{purple}{\sigma_\gamma} = \big( \omega_\gamma \pi_1 {,} (\omega_\gamma \pi_0 \pi_0 {,} \tilde{u}' \theta_\gamma \pi_1 \pi_0{,} \tilde{u} \bar{u} p_1^2 \pi_1 \pi_1 )c \big).\]

\noi The data necessary to construct witnessing sailboats for the equivalences between the pairs of spans $\textcolor{orange}{\sigma_0}$, $\textcolor{teal}{\sigma_1}$, and $\textcolor{purple}{ \sigma_\gamma}$ can be obtained by applying the Ore condition, followed by the diagram-extension twice, and then weak composition three times. Internally this corresponds to a chain of six covers and lifts. All of this is color-coded below using \textcolor{olive}{olive} and \textcolor{brown}{brown} for the Ore condition and \textcolor{cyan}{cyan} and \textcolor{violet}{violet} for the zippering and weak composition step(s) that follow. Note that in both cases the first zippering is done to parallel pairs of composites that can be post-composed by the left leg of the bottom left span. The second zipper is done to parallel pairs of composites that can be post-composed with the left leg of the bottom right span in the pair of composale sailboats. Weak composition is then applied three times in to get comparison spans, $\textcolor{cyan}{\sigma_{0, \gamma}}$ and $\textcolor{violet}{\sigma_{1,\gamma}}$, whose left legs are in $W$.

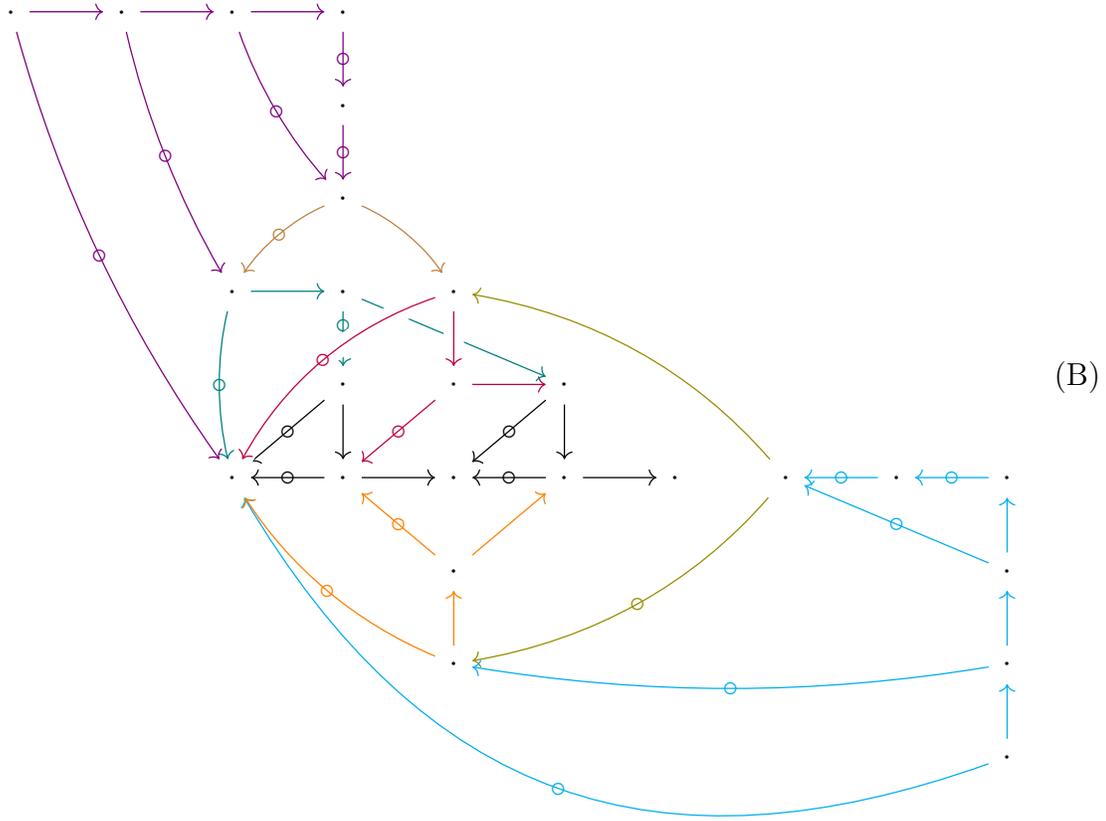
\[ \label{fig Ore + zip + w.c. showing span comp is well def wrt equiv reln (appendix)}
\begin{tikzcd}
	\cdot & \cdot & \cdot & \cdot \\
	&&& \cdot \\
	&&& \cdot \\
	&& \cdot & \cdot & \cdot \\
	&&& \cdot & \cdot & \cdot \\
	{} && \cdot & \cdot & \cdot & \cdot & \cdot & \cdot & \cdot & \cdot \\
	&&&& \cdot &&&&& \cdot \\
	&&&& \cdot &&&&& \cdot \\
	&&&&&&&&& \cdot
	\arrow[from=6-4, to=6-3,"\circ" marking]
	\arrow[from=6-4, to=6-5]
	\arrow[from=6-6, to=6-5, "\circ" marking]
	\arrow[from=6-6, to=6-7]
	\arrow[from=5-6, to=6-6]
	\arrow[from=5-6, to=6-5, "\circ" marking]
	\arrow[from=5-5, to=6-4, purple, "\circ" marking]
	\arrow[from=5-5, to=5-6, purple]
	\arrow[curve={height=-90pt}, from=9-10, to=6-3, cyan, "\circ" marking]
	\arrow[from=9-10, to=8-10, cyan]
	\arrow[from=8-10, to=7-10, cyan]
	\arrow[from=7-10, to=6-10, cyan]
	\arrow[from=6-10, to=6-9, cyan, "\circ" marking]
	\arrow[from=6-9, to=6-8, cyan, "\circ" marking]
	\arrow[from=7-10, to=6-8, cyan, "\circ" marking]
	\arrow[from=7-5, to=6-4, orange, "\circ" marking]
	\arrow[from=7-5, to=6-6, orange]
	\arrow[from=8-5, to=7-5, orange]
	\arrow[curve={height=-12pt}, from=8-5, to=6-3, orange, "\circ" marking]
	\arrow[curve={height=-18pt}, from=6-8, to=8-5, olive, "\circ" marking]
	\arrow[curve={height=-12pt}, from=8-10, to=8-5, cyan, "\circ" marking]
	\arrow[from=1-4, to=2-4, violet, "\circ" marking]
	\arrow[curve={height=6pt}, from=4-3, to=6-3, teal, "\circ" marking]
	\arrow[from=4-4, to=5-6, teal]
	\arrow[from=4-5, to=5-5, purple, crossing over]
	\arrow[from=4-3, to=4-4, teal]
	\arrow[from=5-4, to=6-3, "\circ" marking]
	\arrow[from=5-4, to=6-4]
	\arrow[from=4-4, to=5-4, teal, "\circ" marking, near start]
	\arrow[curve={height=15pt}, from=4-5, to=6-3, purple, "\circ" marking, crossing over]
	\arrow[curve={height=18pt}, from=6-8, to=4-5, olive]
	\arrow[from=2-4, to=3-4, violet, "\circ" marking]
	\arrow[curve={height=6pt}, from=3-4, to=4-3, brown, "\circ" marking]
	\arrow[curve={height=-6pt}, from=3-4, to=4-5, brown]
	\arrow[curve={height=6pt}, from=1-3, to=3-4, violet, "\circ" marking]
	\arrow[curve={height=6pt}, from=1-2, to=4-3, violet, "\circ" marking]
	\arrow[curve={height=12pt}, from=1-1, to=6-3, violet, "\circ" marking]
	\arrow[from=1-1, to=1-2, violet]
	\arrow[from=1-2, to=1-3, violet]
	\arrow[from=1-3, to=1-4, violet]
\end{tikzcd} \tag{B}
\]

\noi The corresponding diagrams of covers and lifts which witness the arrows in the Ore squares and zippering in Diagram \ref{fig Ore + zip + w.c. showing span comp is well def wrt equiv reln} are:

\[\label{cover dgms for showing span comp well def : Ore + zip (appendix)} 
\begin{tikzcd}[column sep = large]
& \cP(\mC) \rar["\pi_1"] 
& \cP_{cq}(\mC)
&
\\
\hat{U}_3
\rar["/" marking, "\hat{u}_3" near start] 
\dar[shift left, "\delta_{\rho_0}", color = cyan] 
\dar[shift right, "\delta_{\rho_1}"', color = violet] 
&\hat{U}_4
\rar["/" marking, "\hat{u}_4" near start]
\uar[shift right, "\delta_{\lambda_0}"', color = cyan] 
\uar[shift left, "\delta_{\lambda_1}", color = violet] 
\dar[shift left, "\rho_0", color = cyan] 
\dar[shift right, "\rho_1"', color = violet] 
&\hat{U}_5
\rar["/" marking, "\hat{u}_5" near start] 
\dar[shift left, "\theta_{\gamma_0}", color = olive]
 \dar[shift right, "\theta_{\gamma_1}"', color = brown] 
 \uar[shift right, "\lambda_0"', color = cyan] 
 \uar[shift left, "\lambda_1", color = violet] 
&\tilde{U} 
\dar[shift right, "(\sigma_1 \pi_0 w {,} \sigma_\gamma \pi_0)"', color = brown]
 \dar[shift left, "(\sigma_0 \pi_0 w {,} \sigma_\gamma \pi_0 )", color = olive] 
\\ 
\cP(\mC) \rar["\pi_1"'] 
& \cP_{cq}(\mC)
& W_\square \rar["(\pi_0 \pi_1 {,} \pi_1 \pi_1)"'] 
& \mC_1 \tensor[_{t}]{\times}{_{wt}} W 
\end{tikzcd} \tag{$\star \star$}
\]

\noi The covers, $\hat{u}_2, \hat{u}_1,$ and $\hat{u}_0$, witness three applications of weak composition in each case as seen in the following continued sequence of covers:

\[\label{cover dgms for showing span comp well def : weak comp x3 (appendix)} 
\begin{tikzcd}[column sep = large]
W_\circ \rar[r, "(\pi_0 \pi_1 {,} \pi_0 \pi_2)"] 
& W \times_{\mC_0} W 
&
W_\circ \rar[r, "(\pi_0 \pi_1 {,} \pi_0 \pi_2)"] 
& W \times_{\mC_0} W 
& \\
\hat{U} 
\rar["/" marking, "\hat{u}_0" near start] 
\uar[shift right, "\omega_{0,0}"', color = cyan] 
\uar[shift left, "\omega_{1,0}", color = violet] 
\dar[shift left = 1.5,""] 
\dar[shift right= 1.5 ,""']
\dar[shift left=.5] 
\dar[shift right= .5] 
& \hat{U}_1
\rar["/" marking,"\hat{u}_1" near start]
\uar[shift right, "\omega_{0,0}'"', color = cyan] 
\uar[shift left, "\omega_{1,0}'", color = violet] 
\dar[shift left, "\omega_{0,1}", color = cyan] 
\dar[shift right, "\omega_{1,1}"', color = violet] 
& \hat{U}_2
\rar["/" marking,"\hat{u}_2" near start]
\uar[shift right, ""', color = cyan, "\omega_{0, 2}"'] 
\uar[shift left, "", color = violet, "\omega_{1, 2}"] 
\dar[shift right, ""', color = violet, "\omega_{1, 1}'"'] 
\dar[shift left, "", color = cyan, "\omega_{0, 1}'"] 
& \hat{U}_3
\uar[shift right, ""', color = cyan, "\omega_{0,2}'"'] 
\uar[shift left, "", color = violet, "\omega_{1,2}'"] 
 \\ 
\slb &
W_\circ \rar[r, "(\pi_0 \pi_1 {,} \pi_0 \pi_2)"'] 
& W \times_{\mC_0} W 
& 
\end{tikzcd} \tag{$\star \star \star$}
\]

\noi The following lemmas refer to the labeled diagrams and equations above.

\begin{lem}\label{appendix lem witness ore squares of span comp well def wrt equiv result (appendix)}
\noi The maps 

\[ \begin{tikzcd}[column sep = large]
\hat{U} \rar[shift right, cyan, "\lambda_0'"'] \rar[shift left, violet, "\lambda_1'"] & P(\mC) \tensor[_t]{\times}{_{ws}} W 
\end{tikzcd}\]

\noi are defined in a similar fashion to $\delta_0'$ in Lemma~\ref{lem defining c'}, namely by descending through the preceeding covers and expanding both sides of the Ore-square equations witnessed. 
\end{lem}

\begin{proof}
To define \textcolor{cyan}{$\lambda_0 '$} we expand both sides of the Ore-square equation 

\[\big( \textcolor{olive}{\theta_{\gamma_0}} \pi_0 \pi_0 w \
{,} \ \textcolor{olive}{\sigma_{\gamma_0}} \pi_0 \pi_1 
\big) c 
= \big( \textcolor{olive}{\theta_{\gamma_0}} \pi_1 \pi_0 w \
{,} \ \textcolor{olive}{\sigma_{\gamma_0}} \pi_1 \pi_1 w
\big) c \\\] 

\noi On the left-hand side we have 

\begin{align}
\begin{split}
& \ \ \ \ \big( \textcolor{olive}{\theta_{\gamma_0}} \pi_0 \pi_0 w \
{,} \ \textcolor{olive}{\sigma_{\gamma_0}} \pi_0 \pi_1 
\big) c \\
&= \big( \textcolor{olive}{\theta_{\gamma_0}} \pi_0 \pi_0 w \
{,} \ \hat{u}_5 \textcolor{orange}{\sigma_0} \pi_0 w
\big) c \\
&= \big( \textcolor{olive}{\theta_{\gamma_0}} \pi_0 \pi_0 w \
{,} \ \hat{u}_5 \tilde{u} \textcolor{orange}{\pi_0 \pi_1} \sigma \pi_0 w
\big) c \\
&=
 \big( \textcolor{olive}{\theta_{\gamma_0}} \pi_0 \pi_0 w \
{,} \ \hat{u}_5 \tilde{u} \textcolor{orange}{\pi_0 \pi_1 \omega} \pi_1 w
\big) c \\
&= 
\big( 
\textcolor{olive}{\theta_{\gamma_0}} \pi_0 \pi_0 w \
{,} \ \hat{u}_5 \tilde{u} \textcolor{orange}{\pi_0 \pi_1 \omega} \pi_1 w
\big) c \\
&= 
\big( 
 \textcolor{olive}{\theta_{\gamma_0}} \pi_0 \pi_0 w \
{,} \ \hat{u}_5 \tilde{u} \textcolor{orange}{\pi_0 \pi_1}(\omega \pi_0 \pi_0 \
{,} \ u_0 \theta \pi_0 \pi_0 w \
{,} \ u \pi_0 \pi_0 w)c 
\big) c \\
&= 
\big( 
 \textcolor{olive}{\theta_{\gamma_0}} \pi_0 \pi_0 w \
{,} \ (\hat{u}_5 \tilde{u} \textcolor{orange}{\pi_0 \pi_1 \omega} \pi_0 \pi_0 \
{,} \ \hat{u}_5 \tilde{u} \textcolor{orange}{\pi_0 \pi_1 u_0 \theta} \pi_0 \pi_0 w \
{,} \ \hat{u}_5 \tilde{u} \pi_0 \pi_1 u \pi_0 \pi_0 w )c 
\big) c\\
&=
\big( 
 \textcolor{olive}{\theta_{\gamma_0}} \pi_0 \pi_0 w \
{,} \ \hat{u}_5 \tilde{u} \textcolor{orange}{\pi_0 \pi_1 \omega} \pi_0 \pi_0 \
{,} \ \hat{u}_5 \tilde{u} \textcolor{orange}{\pi_0 \pi_1 u_0 \theta} \pi_0 \pi_0 w \
{,} \ \hat{u}_5 \tilde{u} \pi_0 \pi_1 u \pi_0 \pi_0 w
\big) c \\
&=
\big( 
 \textcolor{olive}{\theta_{\gamma_0}} \pi_0 \pi_0 w \
{,} \ \hat{u}_5 \tilde{u} \textcolor{orange}{\pi_0 \pi_1 \omega} \pi_0 \pi_0 \
{,} \ \hat{u}_5 \tilde{u} \textcolor{orange}{\pi_0 \pi_1 u_0 \theta} \pi_0 \pi_0 w \
{,} \ \hat{u}_5 \tilde{u} \pi_0 \pi_0 p_0^2 \pi_0 \pi_0 w
\big) c \\
&= \big( 
 (\textcolor{olive}{\theta_{\gamma_0}} \pi_0 \pi_0 w \
{,} \ \hat{u}_5 \tilde{u} \textcolor{orange}{\pi_0 \pi_1 \omega} \pi_0 \pi_0 \
{,} \ \hat{u}_5 \tilde{u} \textcolor{orange}{\pi_0 \pi_1 u_0 \theta} \pi_0 \pi_0 w)c \
{,} \ \hat{u}_5 \tilde{u} \pi_0 \pi_0 p_0^2 \pi_0 \pi_0 w
\big) c 
\end{split}
\end{align}
\noi and on the right we have 
\begin{align}
\begin{split}
 & \ \ \ \  \big( \textcolor{olive}{\theta_{\gamma_0}} \pi_1 \pi_0 w \
{,} \ \textcolor{olive}{\sigma_{\gamma_0}} \pi_1 \pi_1 w
\big) c \\
&= \big( \textcolor{olive}{\theta_{\gamma_0}} \pi_1 \pi_0 w \
{,} \ \hat{u}_5 \textcolor{purple}{\sigma_\gamma} \pi_0 w
\big) c \\
&= \big( \textcolor{olive}{\theta_{\gamma_0}} \pi_1 \pi_0 w \
{,} \ \hat{u}_5 \textcolor{purple}{\omega_\gamma} \pi_1 w 
\big) c \\
&= \big( \textcolor{olive}{\theta_{\gamma_0}} \pi_1 \pi_0 w \
{,} \ \hat{u}_5 ( \textcolor{purple}{\omega_\gamma} \pi_0 \pi_0 \
{,} \ \textcolor{purple}{\omega_\gamma} \pi_0 \pi_1 w\ 
{,} \ \textcolor{purple}{\omega_\gamma} \pi_0 \pi_2 w)c
\big) c \\
&= \big( \textcolor{olive}{\theta_{\gamma_0}} \pi_1 \pi_0 w \
{,} \ \hat{u}_5 ( \textcolor{purple}{\omega_\gamma} \pi_0 \pi_0 \
{,} \ \tilde{u}' \textcolor{purple}{\theta_\gamma} \pi_0 \pi_0 w \ 
{,} \ \tilde{u} \gamma \pi_0 \pi_0w )c
\big) c \\
&= \big( \textcolor{olive}{\theta_{\gamma_0}} \pi_1 \pi_0 w \
{,} \ \hat{u}_5 ( \textcolor{purple}{\omega_\gamma} \pi_0 \pi_0 \
{,} \ \tilde{u}' \textcolor{purple}{\theta_\gamma} \pi_0 \pi_0 w \ 
{,} \ \tilde{u} \pi_0 \pi_1 u \pi_0 \pi_0 w )c
\big) c \\
&= \big( \textcolor{olive}{\theta_{\gamma_0}} \pi_1 \pi_0 w \
{,} \ \hat{u}_5 ( \textcolor{purple}{\omega_\gamma} \pi_0 \pi_0 \
{,} \ \tilde{u}' \textcolor{purple}{\theta_\gamma} \pi_0 \pi_0 w \ 
{,} \ \tilde{u} \pi_0 \pi_0 p_0^2 \pi_0 \pi_0 w )c
\big) c \\
&=  \textcolor{olive}{\theta_{\gamma_0}} \pi_1 \pi_0 w \
{,} \ \hat{u}_5 ( \textcolor{purple}{\omega_\gamma} \pi_0 \pi_0 \
{,} \ \tilde{u}' \textcolor{purple}{\theta_\gamma} \pi_0 \pi_0 w \ 
{,} \ \tilde{u} \pi_0 \pi_0 p_0^2 \pi_0 \pi_0 w )c
\big) c \\
&=  \textcolor{olive}{\theta_{\gamma_0}} \pi_1 \pi_0 w \
{,} \ \hat{u}_5 \textcolor{purple}{\omega_\gamma} \pi_0 \pi_0 \
{,} \ \hat{u}_5 \tilde{u}' \textcolor{purple}{\theta_\gamma} \pi_0 \pi_0 w \ 
{,} \ \hat{u}_5 \tilde{u} \pi_0 \pi_0 p_0^2 \pi_0 \pi_0 w 
\big) c \\
&= \big( (\textcolor{olive}{\theta_{\gamma_0}} \pi_1 \pi_0 w \
{,} \ \hat{u}_5 \textcolor{purple}{\omega_\gamma} \pi_0 \pi_0 \
{,} \ \hat{u}_5 \tilde{u}' \textcolor{purple}{\theta_\gamma} \pi_0 \pi_0 w )c \ 
{,} \ \hat{u}_5 \tilde{u} \pi_0 \pi_0 p_0^2 \pi_0 \pi_0 w 
\big) c 
\end{split}. 
\end{align}

\noi The last lines in equations $(1)$ and $(2)$ uniquely determine $\textcolor{cyan}{\lambda_0'}$ by 

\begin{align*}
  \textcolor{cyan}{\lambda_0'} \pi_1 
  &= \hat{u}_5 \tilde{u} \pi_0 \pi_0 p_0^2 \pi_0 \pi_0 w \\
  \textcolor{cyan}{\lambda_0'} \pi_0 \pi_0 
  &= (\textcolor{olive}{\theta_{\gamma_0}} \pi_0 \pi_0 w \
{,} \ \hat{u}_5 \tilde{u} \textcolor{orange}{\pi_0 \pi_1 \omega} \pi_0 \pi_0 \
{,} \ \hat{u}_5 \tilde{u} \textcolor{orange}{\pi_0 \pi_1 u_0 \theta} \pi_0 \pi_0 w)c \\
  \textcolor{cyan}{\lambda_0'} \pi_0 \pi_1 
  &=(\textcolor{olive}{\theta_{\gamma_0}} \pi_1 \pi_0 w \
{,} \ \hat{u}_5 \textcolor{purple}{\omega_\gamma} \pi_0 \pi_0 \
{,} \ \hat{u}_5 \tilde{u}' \textcolor{purple}{\theta_\gamma} \pi_0 \pi_0 w )c
\end{align*}

\noi The map $\textcolor{violet}{\lambda_1'}$ is similarly determined by expanding both sides of the Ore-square equation:

\[\big( \textcolor{brown}{\theta_{\gamma_1}} \pi_0 \pi_0 w \
{,} \ \textcolor{brown}{\sigma_{\gamma_1}} \pi_0 \pi_1 
\big) c 
= \big( \textcolor{brown}{\theta_{\gamma_1}} \pi_1 \pi_0 w \
{,} \ \textcolor{brown}{\sigma_{\gamma_1}} \pi_1 \pi_1 w
\big) c \] 

\noi On the left-hand side we get 

\begin{align}
\begin{split}
& \ \ \ \  \big( \textcolor{brown}{\theta_{\gamma_1}} \pi_0 \pi_0 w 
{,} \ \textcolor{brown}{\sigma_{\gamma_1}} \pi_0 \pi_1 
\big) c \\
&= \big( \textcolor{brown}{\theta_{\gamma_1}} \pi_0 \pi_0 w 
{,} \ \hat{u}_5 \textcolor{teal}{\sigma_1} \pi_0 w
\big) c \\
&= \big( \textcolor{brown}{\theta_{\gamma_1}} \pi_0 \pi_0 w 
{,} \ \hat{u}_5 \tilde{u} \textcolor{teal}{\pi_1 \pi_1} \sigma \pi_0 w
\big) c \\
&= \big( \textcolor{brown}{\theta_{\gamma_1}} \pi_0 \pi_0 w 
{,} \ \hat{u}5 \tilde{u} \textcolor{teal}{\pi_1 \pi_1 \omega} \pi_1 w
\big) c \\
&= 
\big( 
 \textcolor{brown}{\theta_{\gamma_1}} \pi_0 \pi_0 w 
{,} \ \hat{u}_5 \tilde{u} \textcolor{teal}{\pi_1 \pi_1 \omega} \pi_1 w
\big) c \\
&= 
\big( 
 \textcolor{brown}{\theta_{\gamma_1}} \pi_0 \pi_0 w 
{,} \\
& \ \ \ \ \hat{u}_5 \tilde{u} \textcolor{teal}{\pi_1 \pi_1}(\omega \pi_0 \pi_0 
{,} \ u_0 \theta \pi_0 \pi_0 w 
{,} \ u \pi_0 \pi_0 w)c 
\big) c \\
&= 
\big( 
 \textcolor{brown}{\theta_{\gamma_1}} \pi_0 \pi_0 w 
{,} \\
& \ \ \ \ (\hat{u}_5 \tilde{u} \textcolor{teal}{\pi_1 \pi_1 \omega} \pi_0 \pi_0 
{,} \ \hat{u}_5 \tilde{u} \textcolor{teal}{\pi_1 \pi_1 u_0 \theta} \pi_0 \pi_0 w 
{,} \ \hat{u}_5 \tilde{u} \pi_1 \pi_1 u \pi_0 \pi_0 w )c 
\big) c\\
&=
\big( 
 \textcolor{olive}{\theta_{\gamma_0}} \pi_0 \pi_0 w 
{,} \ \hat{u}_5 \tilde{u} \textcolor{teal}{\pi_1 \pi_1 \omega} \pi_0 \pi_0 
{,} \ \hat{u}_5 \tilde{u} \textcolor{teal}{\pi_1 \pi_1 u_0 \theta} \pi_0 \pi_0 w 
{,} \ \hat{u}_5 \tilde{u} \pi_1 \pi_1 u \pi_0 \pi_0 w
\big) c \\
&=
\big( 
 \textcolor{brown}{\theta_{\gamma_1}} \pi_0 \pi_0 w 
{,} \ \hat{u}_5 \tilde{u} \textcolor{teal}{\pi_1 \pi_1 \omega} \pi_0 \pi_0 
{,} \ \hat{u}_5 \tilde{u} \textcolor{teal}{\pi_1 \pi_1 u_0 \theta} \pi_0 \pi_0 w 
{,} \ \hat{u}_5 \tilde{u} \pi_1 \pi_0 p_1^2 \pi_0 \pi_0 w
\big) c \\
&= \big( 
( \textcolor{brown}{\theta_{\gamma_1}} \pi_0 \pi_0 w 
{,} \ \hat{u}_5 \tilde{u} \textcolor{teal}{\pi_1 \pi_1 \omega} \pi_0 \pi_0 
{,} \ \hat{u}_5 \tilde{u} \textcolor{teal}{\pi_1 \pi_1 u_0 \theta} \pi_0 \pi_0 w)c 
{,} \\
& \ \ \ \ 
\hat{u}_5 \tilde{u} \pi_1 \pi_0 p_1^2 \pi_0 \pi_0 w
\big) c \\
&= \big( 
(\textcolor{brown}{\theta_{\gamma_1}} \pi_0 \pi_0 w 
{,} \ \hat{u}_5 \tilde{u} \textcolor{teal}{\pi_1 \pi_1 \omega} \pi_0 \pi_0 
{,} \ \hat{u}_5 \tilde{u} \textcolor{teal}{\pi_1 \pi_1 u_0 \theta} \pi_0 \pi_0 w)c 
{,} \\
& \ \ \ \ 
\hat{u}_5 \tilde{u} \pi_1 \pi_0 \pi_0 \pi_0 \pi_1 w
\big) c \\
&= \big( 
( \textcolor{brown}{\theta_{\gamma_1}} \pi_0 \pi_0 w 
{,} \ \hat{u}_5 \tilde{u} \textcolor{teal}{\pi_1 \pi_1 \omega} \pi_0 \pi_0 
{,} \ \hat{u}_5 \tilde{u} \textcolor{teal}{\pi_1 \pi_1 u_0 \theta} \pi_0 \pi_0 w)c 
{,} \\ 
& \ \ \ \ 
\hat{u}_5 \tilde{u} \bar{u} \pi_0 \pi_0 \pi_1 w
\big) c \\
&= \big( 
(\textcolor{brown}{\theta_{\gamma_1}} \pi_0 \pi_0 w 
{,} \ \hat{u}_5 \tilde{u} \textcolor{teal}{\pi_1 \pi_1 \omega} \pi_0 \pi_0 
{,} \ \hat{u}_5 \tilde{u} \textcolor{teal}{\pi_1 \pi_1 u_0 \theta} \pi_0 \pi_0 w)c 
{,} \\
& \ \ \ \ 
(\hat{u}_5 \tilde{u} \bar{u} \pi_0 \pi_0 \pi_0 \pi_0  
{,} \ \hat{u}_5 \tilde{u} \bar{u} \pi_0 \pi_0 \pi_0 \pi_1 w)c 
\big) c \\
&=\big( 
(\textcolor{brown}{\theta_{\gamma_1}} \pi_0 \pi_0 w 
{,} \ \hat{u}_5 \tilde{u} \textcolor{teal}{\pi_1 \pi_1 \omega} \pi_0 \pi_0 
{,} \ \hat{u}_5 \tilde{u} \textcolor{teal}{\pi_1 \pi_1 u_0 \theta} \pi_0 \pi_0 w 
{,} \ \hat{u}_5 \tilde{u} \bar{u} \pi_0 \pi_0 \pi_0 \pi_0 )c 
{,} \\
& \ \ \ \ 
\hat{u}_5 \tilde{u} \bar{u} \pi_0 \pi_0 \pi_0 \pi_1 w
\big) c \\
&=\big( 
( \textcolor{brown}{\theta_{\gamma_1}} \pi_0 \pi_0 w 
{,} \ \hat{u}_5 \tilde{u} \textcolor{teal}{\pi_1 \pi_1 \omega} \pi_0 \pi_0 
{,} \ \hat{u}_5 \tilde{u} \textcolor{teal}{\pi_1 \pi_1 u_0 \theta} \pi_0 \pi_0 w 
{,} \ \hat{u}_5 \tilde{u} \bar{u} \pi_0 \pi_0 \pi_0 \pi_0 )c 
{,} \\
& \ \ \ \ 
\hat{u}_5 \tilde{u} \pi_0 \pi_0 \pi_0 \pi_0 \pi_0 \pi_1 w
\big) c \\
&= \big( 
( \textcolor{brown}{\theta_{\gamma_1}} \pi_0 \pi_0 w 
{,} \ \hat{u}_5 \tilde{u} \textcolor{teal}{\pi_1 \pi_1 \omega} \pi_0 \pi_0 
{,} \ \hat{u}_5 \tilde{u} \textcolor{teal}{\pi_1 \pi_1 u_0 \theta} \pi_0 \pi_0 w 
{,} \ \hat{u}_5 \tilde{u} \bar{u} \pi_0 \pi_0 \pi_0 \pi_0 )c 
{,} \\
& \ \ \ \ 
\hat{u}_5 \tilde{u} \pi_0 \pi_0 p_0^2 \pi_0 \pi_0 w
\big) c 
\end{split}
\end{align}

\noi and on the right-hand side we have

\begin{align}
\begin{split}
& \ \ \ \  \big( \textcolor{brown}{\theta_{\gamma_1}} \pi_1 \pi_0 w \
{,} \ \textcolor{brown}{\sigma_{\gamma_1}} \pi_1 \pi_1 w
\big) c \\
&= \big( \textcolor{brown}{\theta_{\gamma_1}} \pi_1 \pi_0 w \
{,} \ \hat{u}_5 \textcolor{purple}{\sigma_\gamma} \pi_0 w
\big) c \\
&= \big( \textcolor{brown}{\theta_{\gamma_1}} \pi_1 \pi_0 w \
{,} \ \hat{u}_5 \textcolor{purple}{\omega_\gamma} \pi_1 w 
\big) c \\
&= \big( \textcolor{brown}{\theta_{\gamma_1}} \pi_1 \pi_0 w \
{,} \ \hat{u}_5 ( \textcolor{purple}{\omega_\gamma} \pi_0 \pi_0 \
{,} \ \textcolor{purple}{\omega_\gamma} \pi_0 \pi_1 w\ 
{,} \ \textcolor{purple}{\omega_\gamma} \pi_0 \pi_2 w)c
\big) c \\
&= \big( \textcolor{brown}{\theta_{\gamma_1}} \pi_1 \pi_0 w \
{,} \ \hat{u}_5 ( \textcolor{purple}{\omega_\gamma} \pi_0 \pi_0 \
{,} \ \tilde{u}' \textcolor{purple}{\theta_\gamma} \pi_0 \pi_0 w \ 
{,} \ \tilde{u} \gamma \pi_0 \pi_0w )c
\big) c \\
&= \big( \textcolor{brown}{\theta_{\gamma_1}} \pi_1 \pi_0 w \
{,} \ \hat{u}_5 ( \textcolor{purple}{\omega_\gamma} \pi_0 \pi_0 \
{,} \ \tilde{u}' \textcolor{purple}{\theta_\gamma} \pi_0 \pi_0 w \ 
{,} \ \tilde{u} \pi_0 \pi_1 u \pi_0 \pi_0 w )c
\big) c \\
&= \big( \textcolor{brown}{\theta_{\gamma_1}} \pi_1 \pi_0 w \
{,} \ \hat{u}_5 ( \textcolor{purple}{\omega_\gamma} \pi_0 \pi_0 \
{,} \ \tilde{u}' \textcolor{purple}{\theta_\gamma} \pi_0 \pi_0 w \ 
{,} \ \tilde{u} \pi_0 \pi_0 p_0^2 \pi_0 \pi_0 w )c
\big) c \\
&= \big( \textcolor{brown}{\theta_{\gamma_1}} \pi_1 \pi_0 w \
{,} \ \hat{u}_5( \textcolor{purple}{\omega_\gamma} \pi_0 \pi_0 \
{,} \ \tilde{u}' \textcolor{purple}{\theta_\gamma} \pi_0 \pi_0 w \ 
{,} \ \tilde{u} \pi_0 \pi_0 p_0^2 \pi_0 \pi_0 w )c
\big) c \\
&= \big( \textcolor{brown}{\theta_{\gamma_1}} \pi_1 \pi_0 w \
{,} \ \hat{u}_5 \textcolor{purple}{\omega_\gamma} \pi_0 \pi_0 \
{,} \ \hat{u}_5 \tilde{u}' \textcolor{purple}{\theta_\gamma} \pi_0 \pi_0 w \ 
{,} \ \hat{u}_5 \tilde{u} \pi_0 \pi_0 p_0^2 \pi_0 \pi_0 w 
\big) c \\ 
&= \big( ( \textcolor{brown}{\theta_{\gamma_1}} \pi_1 \pi_0 w \
{,} \ \hat{u}_5 \textcolor{purple}{\omega_\gamma} \pi_0 \pi_0 \
{,} \ \hat{u}_5 \tilde{u}' \textcolor{purple}{\theta_\gamma} \pi_0 \pi_0 w )c \ 
{,} \ \hat{u}_5 \tilde{u} \pi_0 \pi_0 p_0^2 \pi_0 \pi_0 w 
\big) c 
\end{split}
\end{align}

\noi The last lines of equations (3) and (4) uniquely determine the $\textcolor{violet}{\lambda_1'}$ by

\begin{align*} 
\textcolor{violet}{\lambda_1'} \pi_1 &= \hat{u}_5 \tilde{u} \pi_0 \pi_0 p_0^2 \pi_0 \pi_0 , \\
\textcolor{violet}{\lambda_1'} \pi_0 \pi_0 &= (\textcolor{brown}{\theta_{\gamma_1}} \pi_0 \pi_0 w \
{,} \ \hat{u}_5 \tilde{u} \textcolor{teal}{\pi_1 \pi_1 \omega} \pi_0 \pi_0 \
{,} \ \hat{u}_5 \tilde{u} \textcolor{teal}{\pi_1 \pi_1 u_0 \theta} \pi_0 \pi_0 w \
{,} \ \hat{u}_5 \tilde{u} \bar{u} \pi_0 \pi_0 \pi_0 \pi_0 )c,\\
\textcolor{violet}{\lambda_1'} \pi_0 \pi_1 &=( \textcolor{brown}{\theta_{\gamma_1}} \pi_1 \pi_0 w \
{,} \ \hat{u}_5 \textcolor{purple}{\omega_\gamma} \pi_0 \pi_0 \
{,} \ \hat{u}_5 \tilde{u}' \textcolor{purple}{\theta_\gamma} \pi_0 \pi_0 w )c 
\end{align*}

\end{proof}\

\begin{lem}\label{appendix lem def for rho_0' for span comp well def wrt equiv result }
The equation

\begin{align*}
 (\textcolor{cyan} {\rho_0'} \pi_0 \pi_0 , \textcolor{cyan} {\rho_0'} \pi_1) c 
&= (\textcolor{cyan} {\rho_0'} \pi_0 \pi_1 , \textcolor{cyan} {\rho_0'} \pi_1) c 
\end{align*}\

\noi holds.
\end{lem}
\begin{proof}

This follows from equality between the first and last lines in the following straightforward but tedious calculation. We repeatedly use associativity for internal composition in $\mC$ along with the definitions of the arrows and objects in Diagrams (\ref{dgm diagram of covers to include composite of intermediate comp'ble spans}), (\ref{cover dgms for showing span comp well def : Ore + zip}), and \ref{cover dgms for showing span comp well def : weak comp x3}) of Lemma~\ref{lem witnessing c' well-defined} in this calculation. 

\begin{align*}
& \ \ \ \ 
\big(
(
\textcolor{cyan}{\delta_{\lambda_0}} \pi_0 \iota_{eq} \pi_0w , \
\hat{u}_4 \textcolor{olive}{\theta_{\gamma_0} } \pi_1 \pi_0 , \ 
\hat{u}_{4 ;5} \textcolor{purple}{\omega_\gamma} \pi_0 \pi_0 , \ 
\hat{u}_{4;5} \tilde{u}' \textcolor{purple}{\theta_\gamma} \pi_1 \pi_0 , \ 
\hat{u}_{4;5} \tilde{u} \bar{u} \pi_1 \pi_0 \pi_0 \pi_0 
)c , \\
& \ \ \ \ 
\hat{u}_{4;5} \tilde{u} \bar{u} p_0^2 \pi_1 \pi_0 w 
\big) c \\
&= 
\big(
(
\textcolor{cyan}{\delta_{\lambda_0}} \pi_0 \iota_{eq} \pi_0 w , \
\hat{u}_4 \textcolor{olive}{\theta_{\gamma_0} } \pi_1 \pi_0 , \ 
\hat{u}_{4;5}\textcolor{purple}{\omega_\gamma} \pi_0 \pi_0 , \ 
\hat{u}_{4;5} \tilde{u}' \textcolor{purple}{\theta_\gamma} \pi_1 \pi_0 
)c ,\\
& \ \ \ \ 
(\hat{u}_{4;5} \tilde{u} \bar{u} \pi_1 \pi_0 \pi_0 \pi_0 , \
\hat{u}_{4;5} \tilde{u} \bar{u} p_0^2 \pi_1 \pi_0 w
)c 
\big) c \\
&= 
\big(
(
\textcolor{cyan}{\delta_{\lambda_0}} \pi_0 \iota_{eq} \pi_0 w , \
\hat{u}_4 \textcolor{olive}{\theta_{\gamma_0} } \pi_1 \pi_0 , \ 
\hat{u}_{4;5} \textcolor{purple}{\omega_\gamma} \pi_0 \pi_0 , \ 
\hat{u}_{4;5} \tilde{u}' \textcolor{purple}{\theta_\gamma} \pi_1 \pi_0 
)c ,\\
& \ \ \ \ 
(\hat{u}_{4;5} \tilde{u} \bar{u} \pi_1 \pi_0 \pi_0 \pi_0 , \
\hat{u}_{4;5} \tilde{u} \bar{u} \pi_1 \pi_0 \pi_0 \pi_1 w
)c 
\big) c \\
&= 
\big(
(
\textcolor{cyan}{\delta_{\lambda_0}} \pi_0 \iota_{eq} \pi_0 w , \
\hat{u}_4 \textcolor{olive}{\theta_{\gamma_0} } \pi_1 \pi_0 , \ 
\hat{u}_{4;5} \textcolor{purple}{\omega_\gamma} \pi_0 \pi_0 , \ 
\hat{u}_{4;5} \tilde{u}' \textcolor{purple}{\theta_\gamma} \pi_1 \pi_0 
)c ,\\
& \ \ \ \ 
\hat{u}_{4;5} \tilde{u} \bar{u} \pi_1 \pi_0 \pi_1 w
\big) c \\
&= 
\big(
(
\textcolor{cyan}{\delta_{\lambda_0}} \pi_0 \iota_{eq} \pi_0 w , \
\hat{u}_4 \textcolor{olive}{\theta_{\gamma_0} } \pi_1 \pi_0 , \ 
\hat{u}_{4;5} \textcolor{purple}{\omega_\gamma} \pi_0 \pi_0 
)c , \\
& \ \ \ \ 
(
\hat{u}_{4;5} \tilde{u}' \textcolor{purple}{\theta_\gamma} \pi_1 \pi_0 ,\
\hat{u}_{4;5} \tilde{u} \bar{u} \pi_1 \pi_0 \pi_1 w )c 
\big) c \\
&= 
\big(
(
\textcolor{cyan}{\delta_{\lambda_0}} \pi_0 \iota_{eq} \pi_0 w , \
\hat{u}_4 \textcolor{olive}{\theta_{\gamma_0} } \pi_1 \pi_0 , \ 
\hat{u}_{4;5} \textcolor{purple}{\omega_\gamma} \pi_0 \pi_0 
)c , \\
& \ \ \ \ 
(
\hat{u}_{4;5} \tilde{u}' \textcolor{purple}{\theta_\gamma} \pi_1 \pi_0 ,\
\hat{u}_{4;5} \tilde{u} \bar{u} p_1^2 \pi_1 \pi_0 w )c 
\big) c \\
&= 
\big(
(
\textcolor{cyan}{\delta_{\lambda_0}} \pi_0 \iota_{eq} \pi_0 w  , \
\hat{u}_4 \textcolor{olive}{\theta_{\gamma_0} } \pi_1 \pi_0 , \ 
\hat{u}_{4;5} \textcolor{purple}{\omega_\gamma} \pi_0 \pi_0 
)c , \\
& \ \ \ \ 
(
\hat{u}_{4;5} \tilde{u}' \textcolor{purple}{\theta_\gamma} \pi_0 \pi_0 w ,\
\hat{u}_{4;5} \tilde{u} \bar{u} p_0^2 \pi_0 \pi_1 )c 
\big) c \\
&= 
\big(
(
\textcolor{cyan}{\delta_{\lambda_0}} \pi_0 \iota_{eq} \pi_0 w , \
\hat{u}_4 \textcolor{olive}{\theta_{\gamma_0} } \pi_1 \pi_0 , \ 
\hat{u}_{4;5} \textcolor{purple}{\omega_\gamma} \pi_0 \pi_0 
)c , \\
& \ \ \ \ 
(
\hat{u}_{4;5}\tilde{u}' \textcolor{purple}{\theta_\gamma} \pi_0 \pi_0 w ,\
\hat{u}_{4;5} \tilde{u} \bar{u} p_0^2 \pi_0 \pi_1 )c 
\big) c \\
&= 
\big(
(
\textcolor{cyan}{\delta_{\lambda_0}} \pi_0 \iota_{eq} \pi_0 w , \
\hat{u}_4 \textcolor{olive}{\theta_{\gamma_0} } \pi_1 \pi_0 , \ 
\hat{u}_{4;5} \textcolor{purple}{\omega_\gamma} \pi_0 \pi_0 , \ 
\hat{u}_{4;5} \tilde{u}' \textcolor{purple}{\theta_\gamma} \pi_0 \pi_0 w )c 
 ,\\
& \ \ \ \ 
\hat{u}_{4;5} \tilde{u} \bar{u} p_0^2 \pi_0 \pi_1
\big) c \\
&= 
\big(
(
\textcolor{cyan}{\delta_{\lambda_0}} \pi_0 \iota_{eq} \pi_0 w , \
\hat{u}_4 \textcolor{olive}{\theta_{\gamma_0} } \pi_0 \pi_0 , \ 
\hat{u}_{4 ;5} \tilde{u} \textcolor{orange}{\pi_0 \pi_1 \omega} \pi_0 \pi_0 , \ 
\hat{u}_{4;5} \tilde{u} \textcolor{orange}{\pi_0 \pi_1 u_0 \theta} \pi_0 \pi_0 w
)c ,\\
& \ \ \ \ 
\hat{u}_{4;5} \tilde{u} \bar{u} p_0^2 \pi_0 \pi_1 
\big) c \\
&= 
\big(
(
\textcolor{cyan}{\delta_{\lambda_0}} \pi_0 \iota_{eq} \pi_0 w , \
\hat{u}_4 \textcolor{olive}{\theta_{\gamma_0} } \pi_0 \pi_0 , \ 
\hat{u}_{4;5} \tilde{u} \textcolor{orange}{\pi_0 \pi_1 \omega} \pi_0 \pi_0 )c 
 , \\
& \ \ \ \ 
( \hat{u}_{4;5} \tilde{u} \textcolor{orange}{\pi_0 \pi_1 u_0 \theta} \pi_0 \pi_0 w ,\
\hat{u}_{4;5} \tilde{u} \bar{u} p_0 \pi_0 \pi_1 )c 
\big) c \\
&= 
\big(
(
\textcolor{cyan}{\delta_{\lambda_0}} \pi_0 \iota_{eq} \pi_0 w , \
\hat{u}_4 \textcolor{olive}{\theta_{\gamma_0} } \pi_0 \pi_0 , \ 
\hat{u}_{4;5} \tilde{u} \textcolor{orange}{\pi_0 \pi_1 \omega} \pi_0 \pi_0 )c 
 , \\
& \ \ \ \ 
( \hat{u}_{4;5} \tilde{u} \textcolor{orange}{\pi_0 \pi_1 u_0 \theta} \pi_1 \pi_0  ,\
\hat{u}_{4;5} \tilde{u} \bar{u} p_0^2 \pi_1 \pi_0 w )c 
\big) c \\
&=
\big(
(
\textcolor{cyan}{\delta_{\lambda_0}} \pi_0 \iota_{eq} \pi_0 w , \
\hat{u}_4 \textcolor{olive}{\theta_{\gamma_0} } \pi_0 \pi_0 , \ 
\hat{u}_{4;5} \tilde{u} \textcolor{orange}{\pi_0 \pi_1 \omega} \pi_0 \pi_0 , \ 
\hat{u}_{4;5} \tilde{u} \textcolor{orange}{\pi_0 \pi_1 u_0 \theta} \pi_1 \pi_0 )c 
 ,\\
& \ \ \ \ 
\hat{u}_{4;5} \tilde{u} \bar{u} p_0^2 \pi_1 \pi_0 w 
\big) c. 
\end{align*}

\noi uniquely determine the map $\textcolor{cyan}{\rho_0'}$, for which

\[ \textcolor{cyan}{\rho_0'} \pi_1 = \hat{u}_{4;5} \tilde{u} \bar{u} p_0^2 \pi_1 \pi_0 \]

\noi and $\textcolor{cyan}{\rho_0'} \pi_0 $ is the parallel pair with components 

\[\textcolor{cyan}{\rho_0'} \pi_0 \pi_0 = (
\textcolor{cyan}{\delta_{\lambda_0}} \pi_0 \iota_{eq} \pi_0w , 
\hat{u}_4 \textcolor{olive}{\omega_{\gamma_0} } \pi_1 \pi_0 ,  
\hat{u}_{4;5} \textcolor{purple}{\omega_\gamma} \pi_0 \pi_0 , 
\hat{u}_{4;5} \tilde{u}' \textcolor{purple}{\theta_\gamma} \pi_1 \pi_0 , 
\hat{u}_{4;5} \tilde{u} \bar{u} \pi_1 \pi_0 \pi_0 \pi_0 
)c \]

\noi and 
\[\textcolor{cyan}{\rho_0'} \pi_0 \pi_1 = (
\textcolor{cyan}{\delta_{\lambda_0}} \pi_0 \iota_{eq} \pi_0 w , 
\hat{u}_4 \textcolor{olive}{\omega_{\gamma_0} } \pi_0 \pi_0 , 
\hat{u}_{4;5} \tilde{u} \textcolor{orange}{\pi_0 \pi_1 \omega} \pi_0 \pi_0 ,
\hat{u}_{4;5} \tilde{u} \textcolor{orange}{\pi_0 \pi_1 u_0 \theta} \pi_1 \pi_0 )c . \]

\noi This means that the first and last terms of the calculation above precisely says

\begin{align*}
 (\textcolor{cyan} {\rho_0'} \pi_0 \pi_0 , \textcolor{cyan} {\rho_0'} \pi_1) c 
&= (\textcolor{cyan} {\rho_0'} \pi_0 \pi_1 , \textcolor{cyan} {\rho_0'} \pi_1) c .
\end{align*}
\end{proof}\

\begin{lem}\label{appendix lem def for rho_1' for span comp well def wrt equiv result} 
The equation 
\begin{align*}
 (\textcolor{violet} {\rho_1'} \pi_0 \pi_0 , \textcolor{violet} {\rho_1'} \pi_1) c 
&= (\textcolor{violet} {\rho_1'} \pi_0 \pi_1 , \textcolor{violet} {\rho_1'} \pi_1) c 
\end{align*}\

\noi holds. 
\end{lem}
\begin{proof}
This follows from the first and last lines of the following computation which is saimilar to the one in Lemma~\ref{appendix lem def for rho_0' for span comp well def wrt equiv result }: 

\begin{align*}
&\ \ \ \ 
\big(
(
\textcolor{violet}{\delta_{\lambda_1}} \pi_0 \iota_{eq} \pi_0w , \
\hat{u}_4 \textcolor{brown}{\theta_{\gamma_1} } \pi_1 \pi_0 , \ 
\hat{u}_{4;5} \textcolor{purple}{\omega_\gamma} \pi_0 \pi_0 , \ 
\hat{u}_{4;5} \tilde{u}' \textcolor{purple}{\theta_\gamma} \pi_1 \pi_0 , \ 
\hat{u}_{4;5} \tilde{u} \bar{u} \pi_1 \pi_0 \pi_0 \pi_0 
)c , \\
& \ \ \ \ 
\hat{u}_{4;5} \tilde{u} \bar{u} p_0^2 \pi_1 \pi_0 w 
\big) c \\
&= 
\big(
(
\textcolor{violet}{\delta_{\lambda_1}} \pi_0 \iota_{eq} \pi_0 w , \
\hat{u}_4 \textcolor{brown}{\theta_{\gamma_1} } \pi_1 \pi_0 , \ 
\hat{u}_{4;5}\textcolor{purple}{\omega_\gamma} \pi_0 \pi_0 , \ 
\hat{u}_{4;5} \tilde{u}' \textcolor{purple}{\theta_\gamma} \pi_1 \pi_0 
)c ,\\
& \ \ \ \ 
(\hat{u}_{4;5} \tilde{u} \bar{u} \pi_1 \pi_0 \pi_0 \pi_0 , \
\hat{u}_{4;5} \tilde{u} \bar{u} p_0^2 \pi_1 \pi_0 w
)c 
\big) c \\
&= 
\big(
(
\textcolor{violet}{\delta_{\lambda_1}} \pi_0 \iota_{eq} \pi_0 w , \
\hat{u}_4 \textcolor{brown}{\theta_{\gamma_1} } \pi_1 \pi_0 , \ 
\hat{u}_{4;5} \textcolor{purple}{\omega_\gamma} \pi_0 \pi_0 , \ 
\hat{u}_{4;5} \tilde{u}' \textcolor{purple}{\theta_\gamma} \pi_1 \pi_0 
)c ,\\
& \ \ \ \ 
(\hat{u}_{4;5} \tilde{u} \bar{u} \pi_1 \pi_0 \pi_0 \pi_0 , \
\hat{u}_{4;5} \tilde{u} \bar{u} \pi_1 \pi_0 \pi_0 \pi_1 w
)c 
\big) c \\
&= 
\big(
(
\textcolor{violet}{\delta_{\lambda_1}} \pi_0 \iota_{eq} \pi_0 w , \
\hat{u}_4 \textcolor{brown}{\theta_{\gamma_1} } \pi_1 \pi_0 , \ 
\hat{u}_{4;5} \textcolor{purple}{\omega_\gamma} \pi_0 \pi_0 , \ 
\hat{u}_{4;5} \tilde{u}' \textcolor{purple}{\theta_\gamma} \pi_1 \pi_0 
)c ,\\
& \ \ \ \ 
\hat{u}_{4;5} \tilde{u} \bar{u} \pi_1 \pi_0 \pi_1 w
\big) c \\
&= 
\big(
(
\textcolor{violet}{\delta_{\lambda_1}} \pi_0 \iota_{eq} \pi_0 w , \
\hat{u}_4 \textcolor{brown}{\theta_{\gamma_1} } \pi_1 \pi_0 , \ 
\hat{u}_{4;5} \textcolor{purple}{\omega_\gamma} \pi_0 \pi_0 
)c , \\
& \ \ \ \ 
(
\hat{u}_{4;5} \tilde{u}' \textcolor{purple}{\theta_\gamma} \pi_1 \pi_0 ,\
\hat{u}_{4;5} \tilde{u} \bar{u} \pi_1 \pi_0 \pi_1 w )c 
\big) c \\
&= 
\big(
(
\textcolor{violet}{\delta_{\lambda_1}} \pi_0 \iota_{eq} \pi_0 w , \
\hat{u}_4 \textcolor{brown}{\theta_{\gamma_1} } \pi_1 \pi_0 , \ 
\hat{u}_{4;5} \textcolor{purple}{\omega_\gamma} \pi_0 \pi_0 
)c , \\
& \ \ \ \ 
(
\hat{u}_{4;5} \tilde{u}' \textcolor{purple}{\theta_\gamma} \pi_1 \pi_0 ,\
\hat{u}_{4;5} \tilde{u} \bar{u} p_1^2 \pi_1 \pi_0 w )c 
\big) c \\
&= 
\big(
(
\textcolor{violet}{\delta_{\lambda_1}} \pi_0 \iota_{eq} \pi_0 w  , \
\hat{u}_4 \textcolor{brown}{\theta_{\gamma_1} } \pi_1 \pi_0 , \ 
\hat{u}_{4;5} \textcolor{purple}{\omega_\gamma} \pi_0 \pi_0 
)c , \\
& \ \ \ \ 
(
\hat{u}_{4;5} \tilde{u}' \textcolor{purple}{\theta_\gamma} \pi_0 \pi_0 w ,\
\hat{u}_{4;5} \tilde{u} \bar{u} p_0^2 \pi_0 \pi_1 )c 
\big) c \\
&= 
\big(
(
\textcolor{violet}{\delta_{\lambda_1}} \pi_0 \iota_{eq} \pi_0 w , \
\hat{u}_4 \textcolor{olive}{\theta_{\gamma_1} } \pi_1 \pi_0 , \ 
\hat{u}_{4;5} \textcolor{purple}{\omega_\gamma} \pi_0 \pi_0 
)c , \\
& \ \ \ \ 
(
\hat{u}_{4;5}\tilde{u}' \textcolor{purple}{\theta_\gamma} \pi_0 \pi_0 w ,\
\hat{u}_{4;5} \tilde{u} \bar{u} p_0^2 \pi_0 \pi_1 )c 
\big) c \\
\end{align*}

\noi This calculation continues below, we just had to separate because it wouldn't fit on one page. 
\begin{align*}
& \ \ \ \ \big(
(
\textcolor{violet}{\delta_{\lambda_1}} \pi_0 \iota_{eq} \pi_0 w , \
\hat{u}_4 \textcolor{olive}{\theta_{\gamma_1} } \pi_1 \pi_0 , \ 
\hat{u}_{4;5} \textcolor{purple}{\omega_\gamma} \pi_0 \pi_0 
)c , \\
& \ \ \ \ 
(
\hat{u}_{4;5}\tilde{u}' \textcolor{purple}{\theta_\gamma} \pi_0 \pi_0 w ,\
\hat{u}_{4;5} \tilde{u} \bar{u} p_0^2 \pi_0 \pi_1 )c 
\big) c \\
&= 
\big(
(
\textcolor{violet}{\delta_{\lambda_1}} \pi_0 \iota_{eq} \pi_0 w , \
\hat{u}_4 \textcolor{brown}{\theta_{\gamma_1} } \pi_1 \pi_0 , \ 
\hat{u}_{4;5} \textcolor{purple}{\omega_\gamma} \pi_0 \pi_0 , \ 
\hat{u}_{4;5} \tilde{u}' \textcolor{purple}{\theta_\gamma} \pi_0 \pi_0 w )c 
 ,\\
 & \ \ \ \ 
\hat{u}_{4;5} \tilde{u} \bar{u} p_0^2 \pi_0 \pi_1
\big) c \\
&= 
\big(
(
\textcolor{violet}{\delta_{\lambda_1}} \pi_0 \iota_{eq} \pi_0 w , \
\hat{u}_4 \textcolor{brown}{\theta_{\gamma_1} } \pi_0 \pi_0 , \ 
\hat{u}_{4;5} \tilde{u} \textcolor{teal}{\pi_1 \pi_1 \omega} \pi_0 \pi_0 , \\
&\ \ \ \ 
\hat{u}_{4;5} \tilde{u} \textcolor{teal}{\pi_1 \pi_1 u_0 \theta} \pi_0 \pi_0 w , \ 
\hat{u}_{4;5} \tilde{u} \bar{u} \pi_0 \pi_0 \pi_0 \pi_0 
)c ,\\
& \ \ \ \ 
\hat{u}_{4;5} \tilde{u} \bar{u} p_0^2 \pi_0 \pi_1 
\big) c \\
&= 
\big(
(
\textcolor{violet}{\delta_{\lambda_1}} \pi_0 \iota_{eq} \pi_0 w , \
\hat{u}_4 \textcolor{brown}{\theta_{\gamma_1} } \pi_0 \pi_0 , \ 
\hat{u}_{4;5} \tilde{u} \textcolor{teal}{\pi_1 \pi_1 \omega} \pi_0 \pi_0 , \ 
\hat{u}_{4;5} \tilde{u} \textcolor{teal}{\pi_1 \pi_1 u_0 \theta} \pi_0 \pi_0 w 
)c , \\
& \ \ \ \ 
(\hat{u}_{4;5} \tilde{u} \bar{u} \pi_0 \pi_0 \pi_0 \pi_0  ,
\hat{u}_{4;5} \tilde{u} \bar{u} p_0^2 \pi_0 \pi_1 
)c
\big) c \\
&= 
\big(
(
\textcolor{violet}{\delta_{\lambda_1}} \pi_0 \iota_{eq} \pi_0 w , \
\hat{u}_4 \textcolor{brown}{\theta_{\gamma_1} } \pi_0 \pi_0 , \ 
\hat{u}_{4;5} \tilde{u} \textcolor{teal}{\pi_1 \pi_1 \omega} \pi_0 \pi_0 , \ 
\hat{u}_{4;5} \tilde{u} \textcolor{teal}{\pi_1 \pi_1 u_0 \theta} \pi_0 \pi_0 w 
)c , \\
& \ \ \ \ 
\hat{u}_{4;5} \tilde{u} \bar{u} p_1^2 \pi_0 \pi_1
\big) c \\
&= 
\big(
(
\textcolor{violet}{\delta_{\lambda_1}} \pi_0 \iota_{eq} \pi_0 w , \
\hat{u}_4 \textcolor{brown}{\theta_{\gamma_1} } \pi_0 \pi_0 , \ 
\hat{u}_{4;5} \tilde{u} \textcolor{teal}{\pi_1 \pi_1 \omega} \pi_0 \pi_0 )c 
 , \\
 & \ \ \ \ 
( \hat{u}_{4;5} \tilde{u} \textcolor{teal}{\pi_1 \pi_1 u_0 \theta} \pi_1 \pi_0 ,\
\hat{u}_{4;5} \tilde{u} \bar{u} p_1^2 \pi_1 \pi_0 w )c 
\big) c \\
&=
\big(
(
\textcolor{violet}{\delta_{\lambda_1}} \pi_0 \iota_{eq} \pi_0 w , \
\hat{u}_4 \textcolor{brown}{\theta_{\gamma_1} } \pi_0 \pi_0 , \ 
\hat{u}_{4;5} \tilde{u} \textcolor{teal}{\pi_1 \pi_1 \omega} \pi_0 \pi_0 , \ 
\hat{u}_{4;5} \tilde{u} \textcolor{teal}{\pi_1 \pi_1 u_0 \theta} \pi_1 \pi_0 )c 
 ,\\
 & \ \ \ \ 
\hat{u}_{4;5} \tilde{u} \bar{u} p_1^2 \pi_1 \pi_0 w 
\big) c \\
&=
\big(
(
\textcolor{violet}{\delta_{\lambda_1}} \pi_0 \iota_{eq} \pi_0 w , \
\hat{u}_4 \textcolor{brown}{\theta_{\gamma_1} } \pi_0 \pi_0 , \ 
\hat{u}_{4;5} \tilde{u} \textcolor{teal}{\pi_1 \pi_1 \omega} \pi_0 \pi_0 , \
\hat{u}_{4;5} \tilde{u} \textcolor{teal}{\pi_1 \pi_1 u_0 \theta} \pi_1 \pi_0 )c 
 ,\\
&\ \ \ \ 
(\hat{u}_{4;5} \tilde{u} \bar{u} \pi_1 \pi_0 \pi_0 \pi_0 , 
\hat{u}_{4;5} \tilde{u} \bar{u} \pi_1 \pi_0 \pi_0 \pi_1 w 
)c 
\big) c \\
&=
\big(
(
\textcolor{violet}{\delta_{\lambda_1}} \pi_0 \iota_{eq} \pi_0 w , \
\hat{u}_4 \textcolor{brown}{\theta_{\gamma_1} } \pi_0 \pi_0 , \ 
\hat{u}_{4;5} \tilde{u} \textcolor{teal}{\pi_1 \pi_1 \omega} \pi_0 \pi_0 , \\
& \ \ \ \ 
\hat{u}_{4;5} \tilde{u} \textcolor{teal}{\pi_1 \pi_1 u_0 \theta} \pi_1 \pi_0 ,\
\hat{u}_{4;5} \tilde{u} \bar{u} \pi_1 \pi_0 \pi_0 \pi_0 
)c , \\
& \ \ \ \ 
\hat{u}_{4;5} \tilde{u} \bar{u} \pi_1 \pi_0 \pi_0 \pi_1 w 
\big) c \\
&=
\big(
(
\textcolor{violet}{\delta_{\lambda_1}} \pi_0 \iota_{eq} \pi_0 w , \
\hat{u}_4 \textcolor{brown}{\theta_{\gamma_1} } \pi_0 \pi_0 , \ 
\hat{u}_{4;5} \tilde{u} \textcolor{teal}{\pi_1 \pi_1 \omega} \pi_0 \pi_0 , \\
& \ \ \ \ 
\hat{u}_{4;5} \tilde{u} \textcolor{teal}{\pi_1 \pi_1 u_0 \theta} \pi_1 \pi_0 ,\
\hat{u}_{4;5} \tilde{u} \bar{u} \pi_1 \pi_0 \pi_0 \pi_0 
)c , \\
& \ \ \ \ 
\hat{u}_{4;5} \tilde{u} \bar{u} p_0^2 \pi_1 \pi_0 w 
\big) c. 
\end{align*}

\noi Unsurprisingly we get the same coequalizing arrow in $W$ for \textcolor{violet}{$\rho_1'$} as for \textcolor{cyan}{$\rho_0'$}

\[ \textcolor{violet}{\rho_1'} \pi_1 = \hat{u}_{4;5} \tilde{u} \bar{u} p_0^2 \pi_1 \pi_0 \]

\noi and the parallel pair $\textcolor{violet}{\rho_1'} \pi_0 $ is given by the pair of components 

\[\textcolor{violet}{\rho_1'} \pi_0 \pi_0 = (
\textcolor{violet}{\delta_{\lambda_1}} \pi_0 \iota_{eq} \pi_0w , \
\hat{u}_4 \textcolor{brown}{\omega_{\gamma_1} } \pi_1 \pi_0 , \ 
\hat{u}_{4;5} \textcolor{purple}{\omega_\gamma} \pi_0 \pi_0 , \
\hat{u}_{4;5} \tilde{u}' \textcolor{purple}{\theta_\gamma} \pi_1 \pi_0 ,\ 
\hat{u}_{4;5} \tilde{u} \bar{u} \pi_1 \pi_0 \pi_0 \pi_0 
)c \]

\noi and 
\begin{align*}
  \textcolor{violet}{\rho_0'} \pi_0 \pi_1
  & = (
\textcolor{violet}{\delta_{\lambda_1}} \pi_0 \iota_{eq} \pi_0 w , \
\hat{u}_4 \textcolor{brown}{\theta_{\gamma_1} } \pi_0 \pi_0 , \ 
\hat{u}_{4;5} \tilde{u} \textcolor{teal}{\pi_1 \pi_1 \omega} \pi_0 \pi_0 , \\
& \ \ \ \ 
\hat{u}_{4;5} \tilde{u} \textcolor{teal}{\pi_1 \pi_1 u_0 \theta} \pi_1 \pi_0 ,\
\hat{u}_{4;5} \tilde{u} \bar{u} \pi_1 \pi_0 \pi_0 \pi_0 
)c 
\end{align*}  
\noi The first and last terms of the big equation above being equal then reduces to 

\begin{align*}
 (\textcolor{violet} {\rho_1'} \pi_0 \pi_0 , \textcolor{violet} {\rho_1'} \pi_1) c 
&= (\textcolor{violet} {\rho_1'} \pi_0 \pi_1 , \textcolor{violet} {\rho_1'} \pi_1) c .
\end{align*}\
\end{proof}

\begin{lem}\label{appendix lem def sigma_{0,gamma}}
There is a unique map $\textcolor{cyan}{\sigma_{0, \gamma}} : \hat{U} \to \spn $ determined by

\begin{align*}
   \textcolor{cyan}{\sigma_{0, \gamma}} \pi_0 
   &= \textcolor{cyan}{\omega_{0, 0}} \pi_1 \\
   \textcolor{cyan}{\sigma_{0, \gamma}} \pi_1 
   &= (
\textcolor{cyan}{\omega_0} , \ 
\hat{u}_{0;4} \textcolor{olive}{\theta_{\gamma_0} } \pi_0 \pi_0,\ 
\hat{u} \textcolor{orange}{\sigma_0} \pi_1
)c \\ 
  &= (\textcolor{cyan}{\omega_0} , \ 
\hat{u}_{0;4} \textcolor{olive}{\theta_{\gamma_0} } \pi_1 \pi_0\ ,  \
\hat{u} \textcolor{purple}{\sigma_\gamma} \pi_1
)c
\end{align*}

\noi where $\textcolor{cyan}{\omega_0} : \hat{U} \to W_\circ$ is defined by

\[\textcolor{cyan}{\omega_0} = ( \textcolor{cyan}{\omega_{0,0}} \pi_0 \pi_0 , \
 \hat{u}_0 \textcolor{cyan}{\omega_{0,1}} \pi_0 \pi_0, \
 \hat{u}_{0;1} \textcolor{cyan}{\omega_{0,2}} \pi_1 ) c \]
\end{lem}
\begin{proof}
First, by definition of $W_\circ$ we have

\[ \textcolor{cyan}{\omega_0} s = \textcolor{cyan}{\omega_{0,0}} \pi_0 \pi_0 s = \textcolor{cyan}{\omega_{0,0}} \pi_1 s =\]

\noi showing that $\textcolor{cyan}{\omega_{0 , \gamma}} : \hat{U} \to \spn$ is well-defined. Now let $\textcolor{cyan}{\omega_{0}'}: \hat{U} \to \mC_1$ be defined by 

\[\textcolor{cyan}{\omega_{0}'} = ( \textcolor{cyan}{\omega_{0,0}} \pi_0 \pi_0 ,\
 \hat{u}_0 \textcolor{cyan}{\omega_{0,1}} \pi_0 \pi_0 , \
\hat{u}_{0;1} \textcolor{cyan}{\omega_{0,2}} \pi_0 \pi_0 )c . \]

\noi By definition of $W_\circ$, $\textcolor{cyan}{\omega_0} $, and $\textcolor{cyan}{\omega_0'}$ we have 

\begin{align}\label{appendix eq omega_0 for sigma_{0,gamma}}
\begin{split}
\textcolor{cyan}{\omega_0} 
&= 
( \textcolor{cyan}{\omega_{0,0}} \pi_0 \pi_0 , \
 \hat{u}_0 \textcolor{cyan}{\omega_{0,1}} \pi_0 \pi_0, \
 \hat{u}_{0;1} \textcolor{cyan}{\omega_{0,2}} \pi_1 ) c\\
&= 
( \textcolor{cyan}{\omega_{0,0}} \pi_0 \pi_0 ,\
 \hat{u}_0 \textcolor{cyan}{\omega_{0,1}} \pi_0 \pi_0, \
\hat{u}_{0;1} \textcolor{cyan}{\omega_{0,2}} \pi_0 \pi_0 , \ 
\hat{u}_{0;2} \textcolor{cyan}{\omega_{0,2}'}c ) c \\
&= 
( \textcolor{cyan}{\omega_{0,0}} \pi_0 \pi_0 ,\
 \hat{u}_0 \textcolor{cyan}{\omega_{0,1}} \pi_0 \pi_0, \
\hat{u}_{0;1} \textcolor{cyan}{\omega_{0,2}} \pi_0 \pi_0 , \ 
\hat{u}_{0;2} \textcolor{cyan}{\delta_{\rho_0}} \pi_0 \iota_{eq} \pi_0 , \ 
 \hat{u}_{0;3} \textcolor{cyan}{\delta_{\lambda_0}} \pi_0 \iota_{eq} \pi_0 ) c\\
 &= 
\big( (\textcolor{cyan}{\omega_{0,0}} \pi_0 \pi_0 ,\
 \hat{u}_0 \textcolor{cyan}{\omega_{0,1}} \pi_0 \pi_0, \
\hat{u}_{0;1} \textcolor{cyan}{\omega_{0,2}} \pi_0 \pi_0 , \ 
\hat{u}_{0;2} \textcolor{cyan}{\delta_{\rho_0}} \pi_0 \iota_{eq} \pi_0)c , \ 
 \hat{u}_{0;3} \textcolor{cyan}{\delta_{\lambda_0}} \pi_0 \iota_{eq} \pi_0 \big) c\\
 &= 
\big( \textcolor{cyan}{\omega_{0}'} , \ 
\hat{u}_{0;2} (\textcolor{cyan}{\delta_{\rho_0}} \pi_0 \iota_{eq} \pi_0 , \ 
 \hat{u}_{3} \textcolor{cyan}{\delta_{\lambda_0}} \pi_0 \iota_{eq} \pi_0)c \big) c 
 \end{split}
\end{align}

\noi Notice that by definition of $\textcolor{orange}{\sigma_0}$ and $\textcolor{cyan} {\rho_0'} \pi_0 \pi_1$ and the refinement of covers $\bar{u} : \bar{U} \to \slb \tensor[_t]{\times}{_s} \slb$, we have 
\begin{align}\label{appendix eq sigma_{0,gamma} side-calc sigma_0} 
\begin{split}
& \ \ \ \ (\textcolor{cyan} {\rho_0'} \pi_0 \pi_1 ,  \hat{u}_{4;5} \tilde{u} \bar{u} p_0^2 \pi_1 \pi_1 ) c \\ 
&=
 \big( 
(
\textcolor{cyan}{\delta_{\lambda_0}} \pi_0 \iota_{eq} \pi_0 w , \
\hat{u}_4 \textcolor{olive}{\theta_{\gamma_0} } \pi_0 \pi_0,\ 
\hat{u}_{4;5} \tilde{u} \textcolor{orange}{\pi_0 \pi_1 \omega \pi_0 \pi_0} ,\
\hat{u}_{4 ;5} \tilde{u} \pi_0 \pi_1 u_0 \theta \pi_1 \pi_0
)c ,\\
&\ \ \ \
\hat{u}_{4;5} \tilde{u} \textcolor{orange}{\pi_0 \pi_1 u \pi_1 \pi_1 }
\big)c \\
&= 
\big( 
\textcolor{cyan}{\delta_{\lambda_0}} \pi_0 \iota_{eq} \pi_0 w, \
\hat{u}_4 \textcolor{olive}{\theta_{\gamma_0} } \pi_0 \pi_0,\\
&\ \ \ \ 
(
\hat{u}_{4 ;5} \tilde{u} \textcolor{orange}{\pi_0 \pi_1 \omega \pi_0 \pi_0} ,\
\hat{u}_{4 ;5} \tilde{u} \textcolor{orange}{\pi_0 \pi_1 u_0 \theta \pi_1 \pi_0} ,\
\hat{u}_{4;5} \tilde{u} \pi_0 \pi_1 u \pi_1 \pi_1 
)c
\big)c \\
&= 
\big( 
\textcolor{cyan}{\delta_{\lambda_0}} \pi_0 \iota_{eq} \pi_0 w, \
\hat{u}_4 \textcolor{olive}{\theta_{\gamma_0} } \pi_0 \pi_0,\ 
\hat{u}_{4;5} \textcolor{orange}{\sigma_0} \pi_1 
\big)c
\end{split}
\end{align}\

\noi and similarly by definition of $\textcolor{purple}{\sigma_\gamma}$ and $\textcolor{cyan}{\rho_0'} \pi_0 \pi_0$

\begin{align}\label{appendix eq sigma_{0,gamma} side-calc sigma_gamma}
\begin{split}
& \ \ \ \ (\textcolor{cyan} {\rho_0'} \pi_0 \pi_0 , \hat{u}_{4;5} \tilde{u}' \textcolor{purple}{\theta_\gamma} \pi_1 \pi_0 )c \\ 
&= 
\big( 
(
\textcolor{cyan}{\delta_{\lambda_0}} \pi_0 \iota_{eq} \pi_0w, \
\hat{u}_4 \textcolor{olive}{\theta_{\gamma_0} } \pi_1 \pi_0\ ,  \
\hat{u}_{4;5} \textcolor{purple}{\omega_\gamma} \pi_0 \pi_0, \
\hat{u}_{4;5} \tilde{u}' \textcolor{purple}{\theta_\gamma} \pi_1 \pi_0 , \ 
\hat{u}_{4;5} \tilde{u} \bar{u} \pi_1 \pi_0 \pi_0 \pi_0 
)c , \\
&\ \ \ \ 
\hat{u}_{4;5} \tilde{u} \pi_0 \pi_1 u \pi_1 \pi_1 
\big) c \\
&= 
\big( 
(
\textcolor{cyan}{\delta_{\lambda_0}} \pi_0 \iota_{eq} \pi_0w, \
\hat{u}_4 \textcolor{olive}{\theta_{\gamma_0} } \pi_1 \pi_0\ ,  \
\hat{u}_{4;5} \textcolor{purple}{\omega_\gamma} \pi_0 \pi_0, \
\hat{u}_{4;5} \tilde{u}' \textcolor{purple}{\theta_\gamma} \pi_1 \pi_0
)c , \\
&\ \ \ \ 
(\hat{u}_{4;5} \tilde{u} \bar{u} \pi_1 \pi_0 \pi_0 \pi_0 , \ 
\hat{u}_{4;5} \tilde{u} \pi_0 \pi_1 u \pi_1 \pi_1 )c
\big) c \\
&= 
\big( 
(
\textcolor{cyan}{\delta_{\lambda_0}} \pi_0 \iota_{eq} \pi_0w, \
\hat{u}_4 \textcolor{olive}{\theta_{\gamma_0} } \pi_1 \pi_0\ ,  \
\hat{u}_{4 ;5} \textcolor{purple}{\omega_\gamma} \pi_0 \pi_0, \
\hat{u}_{4 ;5} \tilde{u}' \textcolor{purple}{\theta_\gamma} \pi_1 \pi_0
)c , \\
&\ \ \ \ 
\hat{u}_{4;5} \tilde{u} \bar{u} p_1 \pi_1 \pi_1 
\big) c \\
&= 
\big( 
\textcolor{cyan}{\delta_{\lambda_0}} \pi_0 \iota_{eq} \pi_0w, \
\hat{u}_4 \textcolor{olive}{\theta_{\gamma_0} } \pi_1 \pi_0\ ,  \\
&\ \ \ \
(
\hat{u}_{4 ;5} \textcolor{purple}{\omega_\gamma} \pi_0 \pi_0, \
\hat{u}_{4 ;5} \tilde{u}' \textcolor{purple}{\theta_\gamma} \pi_1 \pi_0
 , \ 
\hat{u}_{4;5} \tilde{u} \bar{u} p_1 \pi_1 \pi_1 
)c
\big) c \\
&= 
\big( 
\textcolor{cyan}{\delta_{\lambda_0}} \pi_0 \iota_{eq} \pi_0w, \
\hat{u}_4 \textcolor{olive}{\theta_{\gamma_0} } \pi_1 \pi_0\ ,  \
\hat{u}_{4 ;5} \textcolor{purple}{\sigma_\gamma} \pi_1 
\big) c . 
\end{split}
\end{align}

\noi Also notice since $\cP(\mC)$ is a pullback of $\cP_{eq}(\mC)$ and $\cP_{cq}(\mC)$ over the object of parallel pairs in $\mC$, $P(\mC)$, we have

\begin{align*}
  \textcolor{cyan} {\delta_{\rho_0}} \pi_0 \iota_{eq} \pi_1 
&= \textcolor{cyan} {\delta_{\rho_0}} \pi_1 \iota_{ceq} \pi_0\\ 
&= \hat{u}_3 \textcolor{cyan} {\rho_0} \iota_{ceq} \pi_0\\
&= \hat{u}_3 \textcolor{cyan} {\rho_0'} \pi_0
\end{align*}

\noi and then by definition of $\cP_{eq}(\mC)$ the composable pairs
\begin{align*}
\textcolor{cyan} {\delta_{\rho_0}} \pi_0 \iota_{eq} (\pi_0 , \pi_1 \pi_0) 
= 
(\textcolor{cyan} {\delta_{\rho_0}} \pi_0 \iota_{eq} \pi_0 , 
\textcolor{cyan} {\delta_{\rho_0}} \pi_0 \iota_{eq} \pi_1 \pi_0 ) 
= 
(\textcolor{cyan} {\delta_{\rho_0}} \pi_0 \iota_{eq} \pi_0, 
\hat{u}_3 \textcolor{cyan} {\rho_0'} \pi_0 \pi_0 )
\end{align*} 
\noi and 
\begin{align*}
\textcolor{cyan} {\delta_{\rho_0}} \pi_0 \iota_{eq} (\pi_0 , \pi_1 \pi_1) 
= 
(\textcolor{cyan} {\delta_{\rho_0}} \pi_0 \iota_{eq} \pi_0 , 
\textcolor{cyan} {\delta_{\rho_0}} \pi_0 \iota_{eq} \pi_1 \pi_1 ) 
= (\textcolor{cyan} {\delta_{\rho_0}} \pi_0 \iota_{eq} \pi_0, 
\hat{u}_3 \textcolor{cyan} {\rho_0'} \pi_0 \pi_1 ) 
\end{align*}

\noi are equal after post-composing with the composition structure map in $\mC$, $c : mC_2 \to \mC_1$: 

\begin{align}\label{appendix eq final helper lemma for sigma_{0,gamma}}
(\textcolor{cyan} {\delta_{\rho_0}} \pi_0 \iota_{eq} \pi_0, 
\hat{u}_3 \textcolor{cyan} {\rho_0'} \pi_0 \pi_0 ) c = (\textcolor{cyan} {\delta_{\rho_0}} \pi_0 \iota_{eq} \pi_0, 
\hat{u}_3 \textcolor{cyan} {\rho_0'} \pi_0 \pi_1 ) c
\end{align}

\noi Associativity of composition and equations ($\ref{appendix eq omega_0 for sigma_{0,gamma}} $); (\ref{appendix eq sigma_{0,gamma} side-calc sigma_0}); (\ref{appendix eq sigma_{0,gamma} side-calc sigma_gamma}); and (\ref{appendix eq final helper lemma for sigma_{0,gamma}}), allow us to see

\begin{align*}
(
\textcolor{cyan}{\omega_0} , \ 
\hat{u}_{0;4} \textcolor{olive}{\theta_{\gamma_0} } \pi_0 \pi_0,\ 
\hat{u} \textcolor{orange}{\sigma_0} \pi_1
)c 
&=
\big( 
\textcolor{cyan}{\omega_0'} , \ 
\hat{u}_{0;2} \textcolor{cyan}{\delta_{\rho_0}} \pi_0 \iota_{eq} \pi_0 , \\
&\ \ \ \
\hat{u}_{0;3} (\textcolor{cyan}{\delta_{\lambda_0}} \pi_0 \iota_{eq} \pi_0 w, \
\hat{u}_4 \textcolor{olive}{\theta_{\gamma_0} } \pi_0 \pi_0,\ 
\hat{u}_{4;5} \textcolor{orange}{\sigma_0} \pi_1 )c \big) c\\
&=
\big(
\textcolor{cyan}{\omega_0'} , \ 
\hat{u}_{0;2} \textcolor{cyan}{\delta_{\rho_0}} \pi_0 \iota_{eq} \pi_0 , \\
&\ \ \ \ 
\hat{u}_{0;3} (\textcolor{cyan} {\rho_0'} \pi_0 \pi_1 ,\ 
\hat{u}_{4;5} \tilde{u} \pi_0 \pi_1 u \pi_1 \pi_1 ) c
)c \\ 
&=
\big( 
\textcolor{cyan}{\omega_0'} , \ 
\hat{u}_{0;2} (\textcolor{cyan}{\delta_{\rho_0}} \pi_0 \iota_{eq} \pi_0 , \ 
\hat{u}_{3} \textcolor{cyan} {\rho_0'} \pi_0 \pi_1)c , \\
&\ \ \ \ 
\hat{u}_{0;5} \tilde{u} \pi_0 \pi_1 u \pi_1 \pi_1 
\big)c \\ 
&=
\big( 
\textcolor{cyan}{\omega_0'} , \ 
\hat{u}_{0;2} (\textcolor{cyan} {\delta_{\rho_0}} \pi_0 \iota_{eq} \pi_0 , \ 
\hat{u}_3 \textcolor{cyan} {\rho_0'} \pi_0 \pi_0 )c , \\
&\ \ \ \ 
\hat{u}_{0;5} \tilde{u} \pi_0 \pi_1 u \pi_1 \pi_1 
\big)c \\ 
&=
\big( 
(\textcolor{cyan}{\omega_0'} , \ 
\hat{u}_{0;2} \textcolor{cyan} {\delta_{\rho_0}} \pi_0 \iota_{eq} \pi_0)c , \\
&\ \ \ \ 
\hat{u}_{0;3} (\textcolor{cyan} {\rho_0'} \pi_0 \pi_0 , \ 
\hat{u}_{4;5} \tilde{u} \pi_0 \pi_1 u \pi_1 \pi_1 )c
\big)c \\ 
&= 
\big( 
(\textcolor{cyan}{\omega_0'} , \ 
\hat{u}_{0;2} \textcolor{cyan}{\delta_{\rho_0}} \pi_0 \iota_{eq} \pi_0) c , \\
&\ \ \ \ 
\hat{u}_{0;3} (\textcolor{cyan}{\delta_{\lambda_0}} \pi_0 \iota_{eq} \pi_0 w, \
\hat{u}_4 \textcolor{olive}{\theta_{\gamma_0} } \pi_1 \pi_0,\ 
\hat{u}_{4;5} \textcolor{purple}{\sigma_\gamma} \pi_1 )c \big) c\\
&=(\textcolor{cyan}{\omega_0} , \ 
\hat{u}_{0;4} \textcolor{olive}{\theta_{\gamma_0} } \pi_1 \pi_0\ ,  \
\hat{u} \textcolor{purple}{\sigma_\gamma} \pi_1
)c
\end{align*}
\end{proof}

\begin{lem}\label{appendix lem def sigma_{1,gamma}}
There is a unique map $\textcolor{violet}{\sigma_{1 , \gamma}} : \hat{U} \to \spn $ determined by
\begin{align*}
   \textcolor{violet}{\sigma_{1 , \gamma}} \pi_0 
   &= \textcolor{violet}{\omega_{1,0}} \pi_1 \\
   \textcolor{violet}{\sigma_{1, \gamma}} \pi_1 
   &= ( \textcolor{violet}{\omega_1} ,\ 
\hat{u}_{0;4} \textcolor{brown}{\omega_{\gamma_1} } \pi_0 \pi_0 ,\ 
\hat{u} \textcolor{teal}{\sigma_1}\pi_1 )c \\ 
   &= 
\big( 
\textcolor{violet}{\omega_{1}} , \
\hat{u}_{0;4} \textcolor{brown}{\omega_{\gamma_1} } \pi_1 \pi_0 , \ 
\hat{u}\textcolor{purple}{\sigma_\gamma} \pi_1
 \big)c
\end{align*}

\noi where $\textcolor{violet}{\omega_1} : \hat{U} \to W_\circ$ is defined by

\[\textcolor{violet}{\omega_1} = ( \textcolor{violet}{\omega_{1,0}} \pi_0 \pi_0 , \
 \hat{u}_0 \textcolor{violet}{\omega_{1,1}} \pi_0 \pi_0, \
 \hat{u}_{0;1} \textcolor{violet}{\omega_{1,2}} \pi_1 ) c \]
\end{lem}
\begin{proof}

\noi Similarly define 

\[ \textcolor{violet}{\omega_{1}'} = (\textcolor{violet}{\omega_{1,0}} \pi_0 \pi_0  , \
\hat{u}_0 \textcolor{violet}{\omega_{1,1}} \pi_0 \pi_0 ,\
\hat{u}_{0;1} \textcolor{violet}{\omega_{1,2}} \pi_0 \pi_0)c \]

\noi and we have:

\begin{align*} 
\textcolor{violet}{\omega_1} 
&= ( \textcolor{violet}{\omega_{1,0}} \pi_0 \pi_0 , \
 \hat{u}_0 \textcolor{violet}{\omega_{1,1}} \pi_0 \pi_0, \
 \hat{u}_{0;1} \textcolor{violet}{\omega_{1,2}} \pi_1 ) c \\
&=
( \textcolor{violet}{\omega_{1,0}} \pi_0 \pi_0  , \
\hat{u}_0 \textcolor{violet}{\omega_{1,1}} \pi_0 \pi_0 ,\
\hat{u}_{0;1} \textcolor{violet}{\omega_{1,2}} \pi_0 \pi_0 ,\ 
\hat{u}_{0;2} \textcolor{violet}{\delta_{\rho_1}} \pi_0 \iota_{eq} \pi_0 ,\ 
 \hat{u}_{0;3} \textcolor{violet}{\delta_{\lambda_1}} \pi_0 \iota_{eq} \pi_0 ) c\\
 &= 
\big( \textcolor{violet}{\omega_{1}'} ,\ 
\hat{u}_{0;2} (\textcolor{violet}{\delta_{\rho_1}} \pi_0 \iota_{eq} \pi_0 ,\ 
 \hat{u}_{3} \textcolor{violet}{\delta_{\lambda_1}} \pi_0 \iota_{eq} \pi_0 ) c \big)c.
\end{align*}

\noi By definition of \textcolor{teal}{$\sigma_1$} and \textcolor{violet}{$\rho_1'$}$\pi_0 \pi_1 $:

\begin{align}\label{appendix eq sigma_{1,gamma} side-calc sigma_1}
\begin{split}
& \ \ \ \ ( \textcolor{violet}{\rho_1'} \pi_0 \pi_1 , 
 \hat{u}_{4;5} \tilde{u} \bar{u} p_0^2 \pi_1 \pi_1)c \\
&=
\big( 
(
\textcolor{violet}{\delta_{\lambda_1}} \pi_0 \iota_{eq} \pi_0 w ,\
\hat{u}_4 \textcolor{brown}{\theta_{\gamma_1} } \pi_0 \pi_0 ,\ 
\hat{u}_{4;5} \tilde{u} \textcolor{teal}{\pi_1 \pi_1 \omega} \pi_0 \pi_0 ,\ 
\hat{u}_{4;5} \tilde{u} \textcolor{teal}{\pi_1 \pi_1 u_0 \theta} \pi_1 \pi_0 ,\
\hat{u}_{4;5} \tilde{u} \bar{u} \pi_1 \pi_0 \pi_0 \pi_0 
)c ,\\
& \ \ \ \ 
 \hat{u}_{4;5} \tilde{u} \bar{u} p_0^2 \pi_1 \pi_1
 \big) c\\
&=
\big( 
\textcolor{violet}{\delta_{\lambda_1}} \pi_0 \iota_{eq} \pi_0 w ,\
\hat{u}_4 \textcolor{brown}{\theta_{\gamma_1} } \pi_0 \pi_0 ,\ 
\hat{u}_{4;5} \tilde{u} \textcolor{teal}{\pi_1 \pi_1 \omega} \pi_0 \pi_0 ,\ 
\hat{u}_{4;5} \tilde{u} \textcolor{teal}{\pi_1 \pi_1 u_0 \theta} \pi_1 \pi_0 ,\\
& \ \ \ \ 
(\hat{u}_{4;5} \tilde{u} \bar{u} \pi_1 \pi_0 \pi_0 \pi_0 ,\ 
 \hat{u}_{4;5} \tilde{u} \bar{u} p_0^2 \pi_1 \pi_1 )c
 \big) c\\
 &=
\big( 
\textcolor{violet}{\delta_{\lambda_1}} \pi_0 \iota_{eq} \pi_0 w ,\
\hat{u}_4 \textcolor{brown}{\theta_{\gamma_1} } \pi_0 \pi_0 ,\ 
\hat{u}_{4;5} \tilde{u} \textcolor{teal}{\pi_1 \pi_1 \omega} \pi_0 \pi_0 ,\ 
\hat{u}_{4;5} \tilde{u} \textcolor{teal}{\pi_1 \pi_1 u_0 \theta} \pi_1 \pi_0 ,\
\hat{u}_{4;5} \tilde{u} \bar{u} p_1^2 \pi_1 \pi_1 
 \big) c \\
 &=
\big( 
\textcolor{violet}{\delta_{\lambda_1}} \pi_0 \iota_{eq} \pi_0 w ,\
\hat{u}_4 \textcolor{brown}{\theta_{\gamma_1} } \pi_0 \pi_0 ,\\
& \ \ \ \  
\hat{u}_{4;5} (\tilde{u} \textcolor{teal}{\pi_1 \pi_1 \omega} \pi_0 \pi_0 ,\ 
 \tilde{u} \textcolor{teal}{\pi_1 \pi_1 u_0 \theta} \pi_1 \pi_0 ,\
 \tilde{u} \pi_1 \pi_0 p_1^2 \pi_1 \pi_1 )c 
 \big) c \\
 &=\big( 
\textcolor{violet}{\delta_{\lambda_1}} \pi_0 \iota_{eq} \pi_0 w ,\
\hat{u}_4 \textcolor{brown}{\theta_{\gamma_1} } \pi_0 \pi_0 ,\\
& \ \ \ \  
\hat{u}_{4;5} (\tilde{u} \textcolor{teal}{\pi_1 \pi_1 \omega} \pi_0 \pi_0 ,\ 
 \tilde{u} \textcolor{teal}{\pi_1 \pi_1 u_0 \theta} \pi_1 \pi_0 ,\
 \tilde{u} \pi_1 \pi_1 u \pi_1 \pi_1 )c 
 \big) c \\
 &=
\big( 
\textcolor{violet}{\delta_{\lambda_1}} \pi_0 \iota_{eq} \pi_0 w ,\
\hat{u}_4 \textcolor{brown}{\theta_{\gamma_1} } \pi_0 \pi_0 ,\ 
\hat{u}_{4;5} \textcolor{teal}{\sigma_1}\pi_1
 \big) c 
 \end{split}
\end{align}

\noi Similarly by definition of $\textcolor{purple}{\sigma_\gamma}$ and $\textcolor{violet}{\rho_1'} \pi_0 \pi_0$:
\begin{align}\label{appendix eq sigma_{1,gamma} side-calc sigma_gamma}
\begin{split}
& \ \ \ \ ( \textcolor{violet}{\rho_1'} \pi_0 \pi_0 , \hat{u}_{4;5} \tilde{u} \bar{u} p_0^2 \pi_1 \pi_1) c \\
&= 
\big( 
(
\textcolor{violet}{\delta_{\lambda_1}} \pi_0 \iota_{eq} \pi_0w , \
\hat{u}_4 \textcolor{brown}{\omega_{\gamma_1} } \pi_1 \pi_0 , \ 
\hat{u}_{4;5} \textcolor{purple}{\omega_\gamma} \pi_0 \pi_0 , \
\hat{u}_{4;5} \tilde{u}' \textcolor{purple}{\theta_\gamma} \pi_1 \pi_0 ,\ 
\hat{u}_{4;5} \tilde{u} \bar{u} \pi_1 \pi_0 \pi_0 \pi_0 
)c ,\\
& \ \ \ \ 
\hat{u}_{4;5} \tilde{u} \bar{u} p_0^2 \pi_1 \pi_1 
\big)c\\
&= 
\big( 
\textcolor{violet}{\delta_{\lambda_1}} \pi_0 \iota_{eq} \pi_0w , \
\hat{u}_4 \textcolor{brown}{\omega_{\gamma_1} } \pi_1 \pi_0 , \ 
\hat{u}_{4;5} \textcolor{purple}{\omega_\gamma} \pi_0 \pi_0 , \
\hat{u}_{4;5} \tilde{u}' \textcolor{purple}{\theta_\gamma} \pi_1 \pi_0 ,\\
& \ \ \ \ 
(\hat{u}_{4;5} \tilde{u} \bar{u} \pi_1 \pi_0 \pi_0 \pi_0 ,\ 
\hat{u}_{4;5} \tilde{u} \bar{u} p_0^2 \pi_1 \pi_1 )c
\big)c \\
&= 
\big( 
\textcolor{violet}{\delta_{\lambda_1}} \pi_0 \iota_{eq} \pi_0w , \
\hat{u}_4 \textcolor{brown}{\omega_{\gamma_1} } \pi_1 \pi_0 , \ 
\hat{u}_{4;5} \textcolor{purple}{\omega_\gamma} \pi_0 \pi_0 , \
\hat{u}_{4;5} \tilde{u}' \textcolor{purple}{\theta_\gamma} \pi_1 \pi_0 ,\ 
\hat{u}_{4;5} \tilde{u} \bar{u} p_1^2 \pi_1 \pi_1 
\big)c\\
&= 
\big( 
\textcolor{violet}{\delta_{\lambda_1}} \pi_0 \iota_{eq} \pi_0w , \
\hat{u}_4 \textcolor{brown}{\omega_{\gamma_1} } \pi_1 \pi_0 , \\
& \ \ \ \  
\hat{u}_{4;5} (\textcolor{purple}{\omega_\gamma} \pi_0 \pi_0 , \
\tilde{u}' \textcolor{purple}{\theta_\gamma} \pi_1 \pi_0 ,\ 
\tilde{u} \bar{u} p_1^2 \pi_1 \pi_1)c
\big)c \\
&= 
\big( 
\textcolor{violet}{\delta_{\lambda_1}} \pi_0 \iota_{eq} \pi_0w , \
\hat{u}_4 \textcolor{brown}{\omega_{\gamma_1} } \pi_1 \pi_0 , \ 
\hat{u}_{4;5} \textcolor{purple}{\sigma_\gamma} \pi_1
\big)c
\end{split}
\end{align}

\noi Since $\cP(\mC)$ is a pullback of $\cP_{eq}(\mC)$ and $\cP_{cq}(\mC)$ over the object of parallel pairs in $\mC$, $P(\mC)$, 

\begin{align*}
  \textcolor{violet}{\delta_{\rho_1}} \pi_0 \iota_{eq} \pi_0 
  &= \textcolor{violet}{\delta_{\rho_1}} \pi_1 \iota_{ceq} \pi_0\\
  &= \hat{u}_3 \textcolor{violet}{\rho_1} \iota_{ceq} \pi_0 \\
  &= \hat{u}_3 \textcolor{violet}{\rho_1'} \pi_0 
\end{align*}

\noi By definition of $\cP_{eq}(\mC)$ the composable pairs
\begin{align*}
\textcolor{violet} {\delta_{\rho_1}} \pi_0 \iota_{eq} (\pi_0 , \pi_1 \pi_0) 
= 
(\textcolor{violet} {\delta_{\rho_1}} \pi_0 \iota_{eq} \pi_0 , 
\textcolor{violet} {\delta_{\rho_1}} \pi_0 \iota_{eq} \pi_1 \pi_0 ) 
= 
(\textcolor{violet} {\delta_{\rho_1}} \pi_0 \iota_{eq} \pi_0, 
\hat{u}_3 \textcolor{violet} {\rho_1'} \pi_0 \pi_0 )
\end{align*}
\noi and 
\begin{align*}
\textcolor{violet} {\delta_{\rho_1}} \pi_0 \iota_{eq} (\pi_0 , \pi_1 \pi_1) 
= 
(\textcolor{violet} {\delta_{\rho_1}} \pi_0 \iota_{eq} \pi_0 , 
\textcolor{violet} {\delta_{\rho_1}} \pi_0 \iota_{eq} \pi_1 \pi_1 ) 
= (\textcolor{violet} {\delta_{\rho_1}} \pi_0 \iota_{eq} \pi_0, 
\hat{u}_3 \textcolor{violet} {\rho_1'} \pi_0 \pi_1 ) .
\end{align*}
\noi are coequalized (in $\cE$) by the composition structure map of $\mC$. This implies 

\begin{align}\label{appendix eq final helper lemma for sigma_{1,gamma}} 
\begin{split}
(\textcolor{violet} {\delta_{\rho_1}} \pi_0 \iota_{eq} \pi_0, 
\hat{u}_3 \textcolor{violet} {\rho_1'} \pi_0 \pi_0 ) c = (\textcolor{violet} {\delta_{\rho_1}} \pi_0 \iota_{eq} \pi_0, 
\hat{u}_3 \textcolor{violet} {\rho_1'} \pi_0 \pi_1 ) c 
\end{split}
\end{align}

\noi Now the span 

\[ \textcolor{violet}{\sigma_{1 , \gamma}} = (\textcolor{violet}{\omega_{0,1}} \pi_1 ,\ \textcolor{violet}{\sigma_{1,\gamma}} \pi_1) \]

\noi is well-defined because 
\[ \textcolor{violet}{\omega_{1,0}} \pi_1 w s = \textcolor{violet}{\omega_{1,0}} \pi_0 \pi_0 s \]
\noi where the right leg, $\textcolor{violet}{\sigma_{1,\gamma}} \pi_1$, is given by the composite 

\begin{align*}
( \textcolor{violet}{\omega_1} ,\ 
\hat{u}_{0;4} \textcolor{brown}{\omega_{\gamma_1} } \pi_0 \pi_0 ,\ 
\hat{u} \textcolor{teal}{\sigma_1}\pi_1 )c 
&= 
\big( 
\textcolor{violet}{\omega_{1}'} ,\ 
\hat{u}_{0;2} \textcolor{violet}{\delta_{\rho_1}} \pi_0 \iota_{eq} \pi_0 ,\\
& \ \ \ \ 
 \hat{u}_{0;3} (\textcolor{violet}{\delta_{\lambda_1}} \pi_0 \iota_{eq} \pi_0 ,\
 \hat{u}_{4} \textcolor{brown}{\omega_{\gamma_1} } \pi_0 \pi_0 ,\ 
\hat{u}_{4;5} \textcolor{teal}{\sigma_1}\pi_1)c \big)c \\
&= 
\big( 
\textcolor{violet}{\omega_{1}'} ,\ 
\hat{u}_{0;2} \textcolor{violet}{\delta_{\rho_1}} \pi_0 \iota_{eq} \pi_0 ,\\
& \ \ \ \ 
 \hat{u}_{0;3} ( \textcolor{violet}{\rho_1'} \pi_0 \pi_1,\
 \hat{u}_{4;5} \tilde{u} \bar{u} p_0^2 \pi_1 \pi_1)c 
 \big)c\\
 &= 
\big( 
\textcolor{violet}{\omega_{1}'} ,\\
& \ \ \ \ 
\hat{u}_{0;2} (\textcolor{violet}{\delta_{\rho_1}} \pi_0 \iota_{eq} \pi_0 ,\ 
 \hat{u}_3 \textcolor{violet}{\rho_1'} \pi_0 \pi_1)c ,\\
& \ \ \ \ 
 \hat{u}_{0;5} \tilde{u} \bar{u} p_0^2 \pi_1 \pi_1 
 \big)c\\
 &= 
\big( 
\textcolor{violet}{\omega_{1}'} ,\\
& \ \ \ \ 
\hat{u}_{0;2} (\textcolor{violet}{\delta_{\rho_1}} \pi_0 \iota_{eq} \pi_0 ,\ 
 \hat{u}_3 \textcolor{violet}{\rho_1'} \pi_0 \pi_0)c ,\\
& \ \ \ \ 
 \hat{u}_{0;5} \tilde{u} \bar{u} p_0^2 \pi_1 \pi_1 
 \big)c\\
 &= 
\big( 
\textcolor{violet}{\omega_{1}'} ,\ 
\hat{u}_{0;2} \textcolor{violet}{\delta_{\rho_1}} \pi_0 \iota_{eq} \pi_0 ,\\
& \ \ \ \ 
 \hat{u}_{0;3} (\textcolor{violet}{\rho_1'} \pi_0 \pi_0,\ 
 \hat{u}_{4;5} \tilde{u} \bar{u} p_0^2 \pi_1 \pi_1 )c
 \big)c\\
 &= 
\big( 
\textcolor{violet}{\omega_{1}'} ,\ 
\hat{u}_{0;2} \textcolor{violet}{\delta_{\rho_1}} \pi_0 \iota_{eq} \pi_0 ,\\
& \ \ \ \ 
 \hat{u}_{0;3} (\textcolor{violet}{\rho_1'} \pi_0 \pi_0,\ 
 \hat{u}_{4;5} \tilde{u} \bar{u} p_0^2 \pi_1 \pi_1 )c
 \big)c\\
 &= 
\big( 
\textcolor{violet}{\omega_{1}'} ,\ 
\hat{u}_{0;2} \textcolor{violet}{\delta_{\rho_1}} \pi_0 \iota_{eq} \pi_0 ,\\
& \ \ \ \ 
 \hat{u}_{0;3} ( 
\textcolor{violet}{\delta_{\lambda_1}} \pi_0 \iota_{eq} \pi_0w , \
\hat{u}_4 \textcolor{brown}{\omega_{\gamma_1} } \pi_1 \pi_0 , \ 
\hat{u}_{4;5} \textcolor{purple}{\sigma_\gamma} \pi_1)c
 \big)c\\
 &= 
\big( 
\textcolor{violet}{\omega_{1}} , \
\hat{u}_{0;4} \textcolor{brown}{\omega_{\gamma_1} } \pi_1 \pi_0 , \ 
\hat{u}\textcolor{purple}{\sigma_\gamma} \pi_1
 \big)c\\
\end{align*}

\end{proof}

\begin{lem}\label{appendix lem defining sailboat varphi_0}
There exists a sailboat $\varphi_0 : \hat{U} \to \slb$, uniquely determined by the pairing map

\begin{align*}
\varphi_0 
= 
\big( 
(
(
\mu_0 , \ 
\hat{u} \textcolor{orange}{\sigma_0} \pi_0
) , \
\textcolor{cyan}{\omega_{0,0}} \pi_1 
) , \
\hat{u} \textcolor{orange}{\sigma_0} \pi_1 
\big)
\end{align*}

\noi where

\[ \mu_0 =(\textcolor{cyan}{\omega_0} ,\
\hat{u}_{0;4} \textcolor{olive}{\theta_{\gamma_0}} \pi_0 \pi_0 
)c\]

\noi such that 

\begin{align*}
   \varphi_0 p_0 &= \hat{u} \textcolor{orange}{\sigma_0} \\
   \varphi_0 p_1 &= \textcolor{cyan}{\sigma_{0,\gamma}} . 
\end{align*}
\end{lem}
\begin{proof}
Recall that $\slb = W_\triangle \tensor[_{\pi_0 \pi_0 \pi_1 s}]{\times}{_s} \mC_1$, where $W_\triangle = (\mC_1 \tensor[_t]{\times}{_{ws}} W) \tensor[_c]{\times}{_w} W$ so to see $\varphi_0 : \hat{U} \to \slb$ is well-defined we need to show that 

\[ (
(
\mu_0 , \ 
\hat{u} \textcolor{orange}{\sigma_0} \pi_0
) , \
\textcolor{cyan}{\omega_{0,0}} \pi_1 
) : \hat{U} \to W_\triangle \]

\noi is well-defined and that 

\[ \hat{u} \textcolor{orange}{\sigma_0} \pi_1 s = \mu_0 t = \hat{u} \textcolor{orange}{\sigma_0} \pi_0 s.\]

By definition of $\mu_0$ and the lift $\textcolor{olive}{\theta_{\gamma_0}} : \hat{U}_5 \to W_\square$ we have 

\begin{align*}
  \mu_0 t
  = \hat{u}_{0;4} \textcolor{olive}{\theta_{\gamma_0}} \pi_0 \pi_0 w t 
  = \hat{u} \textcolor{orange}{\sigma_0} \pi_0 w s
  = \hat{u} \textcolor{orange}{\sigma_0} \pi_1 s 
\end{align*}

\noi showing that $(\mu_0 , \hat{u} \textcolor{orange}{\sigma_0} \pi_0) : \hat{U} \to \mC_1 \tensor[_t]{\times}{_{ws}} W$ are composable with respect to the internal composition structure of $\mC$ (after appliying $w : W \to \mC_1$ in the right-hand component) and that $\hat{u} \textcolor{orange}{\sigma_0} \pi_1$ is well-defined in the right-most component. It remains to see that the $\varphi_0 p_0 : \hat{U} \to W_\triangle$ is well-defined. For this we use the definitions of $\mu_0$ and the lifts in Diagrams (\ref{cover dgms for showing span comp well def : Ore + zip}) and (\ref{cover dgms for showing span comp well def : Ore + zip}) along with associativity of composition in $\mC$ to compute 

\begin{align*}
(
\mu_0 , \ 
\hat{u} \textcolor{orange}{\sigma_0} \pi_0
) , \
\textcolor{cyan}{\omega_{0,0}} \pi_1 
) \pi_1
&= \textcolor{cyan}{\omega_{0,0}} \pi_1 \\
&= (\textcolor{cyan}{\omega_{0,0}} \pi_0 \pi_0 , \textcolor{cyan}{\omega_{0,0}} \pi_0 \pi_1 , \textcolor{cyan}{\omega_{0,0}} \pi_0 \pi_2 ) c \\
&= (\textcolor{cyan}{\omega_{0,0}} \pi_0 \pi_0 , \hat{u}_0 \textcolor{cyan}{\omega_{0,1}} \pi_1 , \hat{u} \textcolor{orange}{\sigma_0} \pi_0) c \\
&= \big(\textcolor{cyan}{\omega_{0,0}} \pi_0 \pi_0 , \\
& \ \ \ \ 
\hat{u}_0 (\textcolor{cyan}{\omega_{0,1}} \pi_0 \pi_0 , 
\textcolor{cyan}{\omega_{0,1}} \pi_0 \pi_1 ,
\textcolor{cyan}{\omega_{0,1}} \pi_0 \pi_2 ) c , \\
& \ \ \ \  \hat{u} \textcolor{orange}{\sigma_0} \pi_0 \big) c \\
&= \big(\textcolor{cyan}{\omega_{0,0}} \pi_0 \pi_0 , \
\hat{u}_0 \textcolor{cyan}{\omega_{0,1}} \pi_0 \pi_0 , \
\hat{u}_0 \textcolor{cyan}{\omega_{0,1}} \pi_0 \pi_1 ,\
\hat{u}_0\textcolor{cyan}{\omega_{0,1}} \pi_0 \pi_2 , \ 
\hat{u} \textcolor{orange}{\sigma_0} \pi_0 \big) c \\
&= \big(\textcolor{cyan}{\omega_{0,0}} \pi_0 \pi_0 , \
\hat{u}_0 \textcolor{cyan}{\omega_{0,1}} \pi_0 \pi_0 , \
\hat{u}_{0;1} \textcolor{cyan}{\omega_{0,2}} \pi_1 ,\
\hat{u}_{0;4} \textcolor{olive}{\theta_{\gamma_0}} \pi_0 \pi_0 , \ 
\hat{u} \textcolor{orange}{\sigma_0} \pi_0 \big) c \\
&= \big(
(\textcolor{cyan}{\omega_{0,0}} \pi_0 \pi_0 , \
\hat{u}_0 \textcolor{cyan}{\omega_{0,1}} \pi_0 \pi_0 , \
\hat{u}_{0;1} \textcolor{cyan}{\omega_{0,2}} \pi_1 
)c ,\\
& \ \ \ \ 
\hat{u}_{0;4} \textcolor{olive}{\theta_{\gamma_0}} \pi_0 \pi_0 , \ 
\hat{u} \textcolor{orange}{\sigma_0} \pi_0 \big) c \\
&= \big( \textcolor{cyan}{\omega_0} ,\
\hat{u}_{0;4} \textcolor{olive}{\theta_{\gamma_0}} \pi_0 \pi_0 , \ 
\hat{u} \textcolor{orange}{\sigma_0} \pi_0 \big) c \\
&= \big( (\textcolor{cyan}{\omega_0} ,\
\hat{u}_{0;4} \textcolor{olive}{\theta_{\gamma_0}} \pi_0 \pi_0 
)c, \ 
\hat{u} \textcolor{orange}{\sigma_0} \pi_0 \big) c \\
&= ( \mu_0 , \hat{u} \textcolor{orange}{\sigma_0} \pi_0)c\\
&= (
\mu_0 , \ 
\hat{u} \textcolor{orange}{\sigma_0} \pi_0
) , \
\textcolor{cyan}{\omega_{0,0}} \pi_1 
) (\pi_0 \pi_0 , \pi_0 \pi_1) c.
\end{align*}

\noi This gives that

\[ \varphi_0 p_0 = (
\mu_0 , \ 
\hat{u} \textcolor{orange}{\sigma_0} \pi_0
) , \
\textcolor{cyan}{\omega_{0,0}} \pi_1 
) : \hat{U} \to W_\triangle \]

\noi is well-defined. Similar techniques allow us to see

\begin{align*}
( \mu_0 , \hat{u} \textcolor{orange}{\sigma_0} \pi_1)c 
&= \big( (\textcolor{cyan}{\omega_0} ,\
\hat{u}_{0;4} \textcolor{olive}{\theta_{\gamma_0}} \pi_0 \pi_0 
)c, \
\hat{u} \textcolor{orange}{\sigma_0} \pi_1 \big)c \\
&= (\textcolor{cyan}{\omega_0} ,\
\hat{u}_{0;4} \textcolor{olive}{\theta_{\gamma_0}} \pi_0 \pi_0 ,\
\hat{u} \textcolor{orange}{\sigma_0} \pi_1)c \\
&= \textcolor{cyan}{\sigma_{0,\gamma}} \pi_1 .
\end{align*}

\noi which shows 

\begin{align*}
  \varphi_0 p_0 
  &= \varphi_0 (\pi_0 \pi_0 \pi_1 , \pi_1)
  & \varphi_0 p_1 
  &= \varphi_0 \big( \pi_0 \pi_1 , ( \pi_0 \pi_0 \pi_0 , \pi_1)c \big) \\
  &= (\hat{u} \textcolor{orange}{\sigma_0} \pi_0 , \hat{u} \textcolor{orange}{\sigma_0} \pi_1) 
  & 
  &= \big( \textcolor{cyan}{\omega_{0,0}} \pi_1 , (\mu_0 , \hat{u} \textcolor{orange}{\sigma_0} \pi_1) c \big) \\
  &= \hat{u} \textcolor{orange}{\sigma_0}
  &
  &= ( \textcolor{cyan}{\omega_{0,0}} \pi_1 , \textcolor{cyan}{\sigma_{0,\gamma}} \pi_1 )\\
  &&
  &= \textcolor{cyan}{\sigma_{0,\gamma}} .
\end{align*}
\end{proof}

\begin{lem}\label{appendix lem defining sailboat varphi_{0,gamma}}
There exists a sailboat $\varphi_{0,\gamma} : \hat{U} \to \slb$, uniquely determined by the pairing map

\begin{align*}
\varphi_{0 , \gamma} = 
\big( 
( 
( \mu_{0, \gamma} , \
\hat{u} \textcolor{purple}{\sigma_\gamma} \pi_0
) , \
\textcolor{cyan}{\omega_{0,0}} \pi_1 ) , \
\hat{u} \textcolor{purple}{\sigma_\gamma} \pi_1 
 \big)
\end{align*}

\noi where

\[ \mu_{0, \gamma} = (\textcolor{cyan}{\omega_0} ,\
\hat{u}_{0;4} \textcolor{olive}{\theta_{\gamma_0}} \pi_1 \pi_0 )c \]

\noi such that 

\begin{align*}
   \varphi_{0,\gamma} p_0 &= \hat{u} \textcolor{purple}{\sigma_\gamma} \\
   \varphi_{0,\gamma} p_1 &= \textcolor{cyan}{\sigma_{0,\gamma}} 
\end{align*}
\end{lem}
\begin{proof}
\noi First note that the components of the pairing map defining $\varphi_{0, \gamma} : \hat{U} \to \slb$ are appropriately composable with respect to the internal composition structure of $\mC$: 

\begin{align*}
  \mu_{0, \gamma} t 
  = \hat{u}_{0;4} \textcolor{olive}{\theta_{\gamma_0}} \pi_1 \pi_0 t
  = \hat{u} \textcolor{purple}{\sigma_\gamma} \pi_0 w s
  = \hat{u} \textcolor{purple}{\sigma_\gamma} \pi_1 s.
\end{align*}

\noi By Definition of $\mu_{0, \gamma}$ and the lifts, $\textcolor{cyan}{\omega_{0,0}} : \hat{U} \to W_\circ$ and $\textcolor{olive}{\theta_{\gamma_0}} : \hat{U}_5 \to W_\square$, we have 

\begin{align*}
\textcolor{cyan}{\omega_{0,0}} \pi_1 
&= \big( \textcolor{cyan}{\omega_0} ,\
\hat{u}_{0;4} \textcolor{olive}{\theta_{\gamma_0}} \pi_0 \pi_0 w , \ 
\hat{u} \textcolor{orange}{\sigma_0} \pi_0 w \big) c \\ 
&= \big( \textcolor{cyan}{\omega_0} ,\
(\hat{u}_{0;4} \textcolor{olive}{\theta_{\gamma_0}} \pi_0 \pi_0 w, \ 
\hat{u} \textcolor{orange}{\sigma_0} \pi_0 w)c \big) c \\
&= \big( \textcolor{cyan}{\omega_0} ,\
(\hat{u}_{0;4} \textcolor{olive}{\theta_{\gamma_0}} \pi_1 \pi_0 , \ 
\hat{u} \textcolor{purple}{\sigma_\gamma} \pi_0 w)c \big) c \\
&= \big( (\textcolor{cyan}{\omega_0} ,\
\hat{u}_{0;4} \textcolor{olive}{\theta_{\gamma_0}} \pi_1 \pi_0 )c , \ 
\hat{u} \textcolor{purple}{\sigma_\gamma} \pi_0 w \big) c \\
&= (\mu_{0, \gamma} , \hat{u} \textcolor{purple}{\sigma_\gamma} \pi_0 w)c .
\end{align*} 

\noi showing that the map 

\[ \varphi_{0,\gamma} \pi_0 = ( 
( \mu_{0, \gamma} , \
\hat{u} \textcolor{purple}{\sigma_\gamma} \pi_0
) , \
\textcolor{cyan}{\omega_{0,0}} \pi_1 ) 
: \hat{U} \to W_\triangle\]

\noi is well-defined. Then by associativity of composition in $\mC$ and the definitions of $\mu_{0,\gamma}$ and $\textcolor{cyan}{\sigma_{0,\gamma}}$ we get

\begin{align*}
( \mu_{0, \gamma} , \hat{u} \textcolor{purple}{\sigma_\gamma} \pi_1 ) c 
&= \big( (\textcolor{cyan}{\omega_0} ,\
\hat{u}_{0;4} \textcolor{olive}{\theta_{\gamma_0}} \pi_1 \pi_0 )c , \
\hat{u} \textcolor{purple}{\sigma_\gamma} \pi_1 \big) c\\
&= \big( \textcolor{cyan}{\omega_0} ,\
\hat{u}_{0;4} \textcolor{olive}{\theta_{\gamma_0}} \pi_1 \pi_0 , \
\hat{u} \textcolor{purple}{\sigma_\gamma} \pi_1 \big) c \\
&= \textcolor{cyan}{\sigma_{0,\gamma}} \pi_1 .
\end{align*}

\noi The previous equation implies the unique pairing, $\hat{U} \to \slb$, given by

\[ \varphi_{0 , \gamma} = 
\big( 
( 
( \mu_{0, \gamma} , \
\hat{u} \textcolor{purple}{\sigma_\gamma} \pi_0
) , \
\textcolor{cyan}{\omega_{0,0}} \pi_1 ) , \
\hat{u} \textcolor{purple}{\sigma_\gamma} \pi_1 
 \big)
\]
 
\noi is well-defined. From here it is straightforward to calculate
 
\begin{align*}
  \varphi_{0,\gamma} p_0 
  &= \varphi_{0 , \gamma} (\pi_0 \pi_0 \pi_1, \pi_1)
  & \varphi_{0,\gamma} p_1
  &=\varphi_{0 , \gamma} \big( \pi_0 \pi_1 , (\pi_0 \pi_0 \pi_0 , \pi_1)c \big) \\
  &= (\hat{u} \textcolor{purple}{\sigma_\gamma} \pi_0 , \hat{u} \textcolor{purple}{\sigma_\gamma} \pi_1 )
  &
  &= \big( \textcolor{cyan}{\omega_{0,0}} \pi_1 , 
  ( \mu_{0, \gamma} , \hat{u} \textcolor{purple}{\sigma_\gamma} \pi_1 ) c \big)\\
  &= \hat{u} \textcolor{purple}{\sigma_\gamma} 
  &
  &=(\textcolor{cyan}{\omega_{0,0}} \pi_1 , \textcolor{cyan}{\sigma_{0,\gamma}} \pi_1 )\\
  &&&= \textcolor{cyan}{\sigma_{0,\gamma}}.
\end{align*} 
\end{proof}

\begin{lem}\label{appendix lem defining sailboat varphi_1}
The sailboat, $\varphi_1 : \hat{U} \to \slb$, defined by 

\begin{align*}
\varphi_1
= 
\big( 
(
(
\mu_1 , \ 
\hat{u} \textcolor{teal}{\sigma_1} \pi_0
) , \
\textcolor{violet}{\omega_{1,0}} \pi_1 
) , \
\hat{u} \textcolor{teal}{\sigma_1} \pi_1 
\big)
\end{align*}

\noi where 
\[ \mu_1 = (\textcolor{violet}{\omega_1} , \hat{u}_{0;4}\textcolor{brown}{\gamma_1} \pi_0 \pi_0) c \] 
\noi is well-defined and relates the spans $\textcolor{teal}{\sigma_1}, \textcolor{violet}{\sigma_{1,\gamma}} : \hat{U} \to \spn$ in the sense that 

\begin{align*}
  \varphi_1 p_0 &= \hat{u}\textcolor{teal}{\sigma_1} 
  & \varphi_1 p_1 &= \textcolor{violet}{\sigma_{1,\gamma}}
\end{align*}
\end{lem}
\begin{proof}

First notice that the components defining $\varphi_1$ are apporpriately composable by checking

\begin{align*}
  \mu_1 t 
  = \hat{u}_{0;4}\textcolor{brown}{\gamma_1} \pi_0 \pi_0 w t 
  =\hat{u} \textcolor{teal}{\sigma_1} \pi_1 s
  = \hat{u} \textcolor{teal}{\sigma_1} \pi_0 w s.
\end{align*}

\noi Now to see that the component 

\[ \varphi_1 p_0 = (
(
\mu_1 , \ 
\hat{u} \textcolor{teal}{\sigma_1} \pi_0
) , \
\textcolor{violet}{\omega_{1,0}} \pi_1 
) : \hat{U} \to W_\triangle\]

\noi is well-defined we use the definitions of the lifts $\textcolor{violet}{\omega_{1,0}} : \hat{U} \to W_\circ$ and $\textcolor{brown}{\theta_{\gamma_1}} : \hat{U}_5 \to W_\square$ to see

\begin{align*}
\textcolor{violet}{\omega_{1,0}} \pi_1 
&= 
(\textcolor{violet}{\omega_{1,0}} \pi_0 \pi_0 , 
\textcolor{violet}{\omega_{1,0}} \pi_0 \pi_1 , 
\textcolor{violet}{\omega_{1,0}} \pi_0 \pi_2 ) c \\
&= 
(\textcolor{violet}{\omega_{1,0}} \pi_0 \pi_0 , 
\hat{u}_0 \textcolor{violet}{\omega_{1,1}} \pi_1 , 
\hat{u} \textcolor{teal}{\sigma_1} \pi_0) c \\
&= \big(\textcolor{violet}{\omega_{1,0}} \pi_0 \pi_0 , \\
& \ \ \ \ 
\hat{u}_0 (\textcolor{violet}{\omega_{1,1}} \pi_0 \pi_0 , 
\textcolor{violet}{\omega_{1,1}} \pi_0 \pi_1 ,
\textcolor{violet}{\omega_{1,1}} \pi_0 \pi_2 ) c, \\
& \ \ \ \  
\hat{u} \textcolor{teal}{\sigma_1} \pi_0 \big) c \\
&= \big(\textcolor{violet}{\omega_{1,0}} \pi_0 \pi_0 , \
\hat{u}_0 \textcolor{violet}{\omega_{1,1}} \pi_0 \pi_0 , \
\hat{u}_0 \textcolor{violet}{\omega_{1,1}} \pi_0 \pi_1 ,\
\hat{u}_0\textcolor{violet}{\omega_{1,1}} \pi_0 \pi_2 , \ 
\hat{u} \textcolor{teal}{\sigma_1} \pi_0 \big) c \\
&= \big(\textcolor{violet}{\omega_{1,0}} \pi_0 \pi_0 , \
\hat{u}_0 \textcolor{violet}{\omega_{1,1}} \pi_0 \pi_0 , \
\hat{u}_{0;1} \textcolor{violet}{\omega_{1,2}} \pi_1 ,\
\hat{u}_{0;4} \textcolor{brown}{\theta_{\gamma_1}} \pi_0 \pi_0 , \ 
\hat{u} \textcolor{teal}{\sigma_1} \pi_0 \big) c \\
&= \big(
(\textcolor{violet}{\omega_{1,0}} \pi_0 \pi_0 , \
\hat{u}_0 \textcolor{violet}{\omega_{1,1}} \pi_0 \pi_0 , \
\hat{u}_{0;1} \textcolor{violet}{\omega_{1,2}} \pi_1 
)c ,\\
& \ \ \ \ 
\hat{u}_{0;4} \textcolor{brown}{\theta_{\gamma_1}} \pi_0 \pi_0 , \ 
\hat{u} \textcolor{teal}{\sigma_1} \pi_0 \big) c \\
&= \big( \textcolor{violet}{\omega_1} ,\
\hat{u}_{0;4} \textcolor{brown}{\theta_{\gamma_1}} \pi_0 \pi_0 , \ 
\hat{u} \textcolor{teal}{\sigma_1} \pi_0 \big) c \\
&= \big( (\textcolor{violet}{\omega_1} ,\
\hat{u}_{0;4} \textcolor{brown}{\theta_{\gamma_1}} \pi_0 \pi_0 
)c, \ 
\hat{u} \textcolor{teal}{\sigma_1} \pi_0 \big) c \\
&= ( \mu_1 , \hat{u} \textcolor{teal}{\sigma_1} \pi_0)c. 
\end{align*}

\noi This shows $\varphi_1 : \hat{U} \to \slb$ is well-defined. By definition of $\mu_1$ and $\textcolor{violet}{\sigma_{1,\gamma}}$ in Lemma~\ref{appendix lem def sigma_{1,gamma}} we have 

\begin{align*}
( \mu_1 , \hat{u} \textcolor{teal}{\sigma_1} \pi_1)c 
&= \big( (\textcolor{violet}{\omega_1} ,\
\hat{u}_{0;4} \textcolor{brown}{\theta_{\gamma_1}} \pi_0 \pi_0 
)c, \
\hat{u} \textcolor{teal}{\sigma_1} \pi_1 \big)c \\
&= (\textcolor{violet}{\omega_1} ,\
\hat{u}_{0;4} \textcolor{brown}{\theta_{\gamma_1}} \pi_0 \pi_0 ,\
\hat{u} \textcolor{teal}{\sigma_1} \pi_1)c \\
&= \textcolor{violet}{\sigma_{1,\gamma}} \pi_1 
\end{align*}

\noi which implies 

\begin{align*}
  \varphi_1 p_0 
  &= \varphi_1 (\pi_0 \pi_0 \pi_1 , \pi_1) 
  & \varphi_1 p_1 
  &= \varphi_1 (\pi_0 \pi_1 , (\pi_0 \pi_0 \pi_0 , \pi_1)c) \\
  &= (\hat{u} \textcolor{teal}{\sigma_1} \pi_0 , \hat{u} \textcolor{teal}{\sigma_1} \pi_1 ) c
  &
  &= ( \mu_1 , \hat{u} \textcolor{teal}{\sigma_1} \pi_1)c \\
  &= \hat{u} \textcolor{teal}{\sigma_1}
  &
  &= ( \textcolor{violet}{\omega_{1,0}} \pi_1 , \textcolor{violet}{\sigma_{1,\gamma}} \pi_1 ) \\
  &
  &
  &= \textcolor{violet}{\sigma_{1,\gamma}}
\end{align*} 
\end{proof}

\begin{lem}\label{appendix lem defining sailboat varphi_{1,gamma}}
The sailboat, $\varphi_{1,\gamma} : \hat{U} \to \slb$, defined by 

\[ \varphi_{1 , \gamma} = 
\big( 
( 
( \mu_{1, \gamma} , \
\hat{u} \textcolor{purple}{\sigma_\gamma} \pi_0
) , \
\textcolor{violet}{\omega_{1,0}} \pi_1 ) , \
\hat{u} \textcolor{purple}{\sigma_\gamma} \pi_1 
 \big)
\]

\noi where

\[ \mu_{1, \gamma} = (\textcolor{violet}{\omega_1} ,\
\hat{u}_{0;4} \textcolor{brown}{\theta_{\gamma_1}} \pi_1 \pi_0 )c \]

\noi relates the spans $\textcolor{purple}{\sigma_\gamma} , \textcolor{violet}{\sigma_{1,\gamma}} : \hat{U} \to \spn$ in the sense that 

\begin{align*}
  \varphi_{1,\gamma} p_0 &= \hat{u} \textcolor{purple}{\sigma_\gamma} 
  & \varphi_{1,\gamma} p_1 &= \textcolor{violet}{\sigma_{1,\gamma}}.
\end{align*}
\end{lem}
\begin{proof}
First use the definition of $\mu_{01}$, the lift $\textcolor{brown}{\theta_{\gamma_1}} : \hat{U}_5 \to W_\square$, and the span $\textcolor{purple}{\sigma_\gamma}$ to see that the components of $\varphi_{1,\gamma}$ are appropriately composable in $\mC$: 

\begin{align*}
  \mu_{1,\gamma} t
  = \hat{u}_{0;4} \textcolor{brown}{\theta_{\gamma_1}} \pi_1 \pi_0 t 
  = \hat{u} \textcolor{purple}{\sigma_\gamma} \pi_0 w s 
  = \hat{u} \textcolor{purple}{\sigma_\gamma} \pi_1 s .
\end{align*}

\noi Now use those definitions to compute 

\begin{align*}
\textcolor{violet}{\omega_{1,0}} \pi_1 
&= \big( \textcolor{violet}{\omega_1} ,\
\hat{u}_{0;4} \textcolor{brown}{\theta_{\gamma_1}} \pi_0 \pi_0 w , \ 
\hat{u} \textcolor{teal}{\sigma_1} \pi_0 w \big) c \\ 
&= \big( \textcolor{violet}{\omega_1} ,\
(\hat{u}_{0;4} \textcolor{brown}{\theta_{\gamma_1}} \pi_0 \pi_0 w, \ 
\hat{u} \textcolor{teal}{\sigma_1} \pi_0 w)c \big) c \\
&= \big( \textcolor{violet}{\omega_1} ,\
(\hat{u}_{0;4} \textcolor{brown}{\theta_{\gamma_1}} \pi_1 \pi_0 , \ 
\hat{u} \textcolor{purple}{\sigma_\gamma} \pi_0 w)c \big) c \\
&= \big( (\textcolor{violet}{\omega_1} ,\
\hat{u}_{0;4} \textcolor{brown}{\theta_{\gamma_1}} \pi_1 \pi_0 )c , \ 
\hat{u} \textcolor{purple}{\sigma_\gamma} \pi_0 w \big) c \\
&= (\mu_{1, \gamma} , \hat{u} \textcolor{purple}{\sigma_\gamma} \pi_0 w)c .
\end{align*} 

\noi This shows that the component 

\[ \varphi_{1, \gamma} \pi_0 = ( 
( \mu_{1, \gamma} , \
\hat{u} \textcolor{purple}{\sigma_\gamma} \pi_0
) , \
\textcolor{violet}{\omega_{1,0}} \pi_1 ) 
: \hat{U} \to W_\triangle\]

\noi is well-defined. Similarly,
\begin{align*}
( \mu_{1, \gamma} , \hat{u} \textcolor{purple}{\sigma_\gamma} \pi_1 ) c 
&= \big( (\textcolor{violet}{\omega_1} ,\
\hat{u}_{0;4} \textcolor{brown}{\theta_{\gamma_1}} \pi_1 \pi_0 )c , \
\hat{u} \textcolor{purple}{\sigma_\gamma} \pi_1 \big) c\\
&= \big( \textcolor{violet}{\omega_1} ,\
\hat{u}_{0;4} \textcolor{brown}{\theta_{\gamma_1}} \pi_1 \pi_0 , \
\hat{u} \textcolor{purple}{\sigma_\gamma} \pi_1 \big) c \\
&= \textcolor{violet}{\sigma_{1,\gamma}} \pi_1 .
\end{align*}

\noi implies

\begin{align*}
  \varphi_{1,\gamma} p_0 
  &= \varphi_{1 , \gamma} (\pi_0 \pi_0 \pi_1, \pi_1) 
  & \varphi_{1,\gamma} p_1 
  &= \varphi_{1 , \gamma} \big( \pi_0 \pi_1 , (\pi_0 \pi_0 \pi_0 , \pi_1)c \big) \\
  &= (\hat{u} \textcolor{purple}{\sigma_\gamma} \pi_0 , \hat{u} \textcolor{purple}{\sigma_\gamma} \pi_1 ) 
  &
  &= \big(\textcolor{violet}{\omega_{1,0}} \pi_1 , ( \mu_{1, \gamma} , \hat{u} \textcolor{purple}{\sigma_\gamma} \pi_1 ) c  \big) \\
  & = \hat{u} \textcolor{purple}{\sigma_\gamma} 
  &
  &= (\textcolor{violet}{\omega_{1,0}} \pi_1 , \textcolor{violet}{\sigma_{1,\gamma}} \pi_1 ) \\
  &&
  &= \textcolor{violet}{\sigma_{1,\gamma}}.
\end{align*} 

\end{proof}

\bibliographystyle{plain}
\bibliography{thesis}

\end{document}